\documentclass[11pt]{article}
\usepackage{amssymb,amsbsy,amsthm,amsmath,latexsym,bm}
\usepackage{graphicx} 
\usepackage{subfigure}
\usepackage[english]{babel}
\usepackage{float}
\usepackage{color}
\usepackage[T1]{fontenc}
\usepackage{multicol}
\usepackage{csquotes}
\usepackage{enumerate}
\usepackage{setspace}
\usepackage[margin=1in,text={6.5in,9in}]{geometry}
\usepackage{nonfloat}
\usepackage{comment}
\usepackage{booktabs, makecell, tabularx}
\usepackage{rotating}
\usepackage[export]{adjustbox}
\usepackage{mathtools}
  
\usepackage[
	colorlinks=true,
	citecolor=red,
	linkcolor=blue,
	urlcolor=red,
	pdfpagetransition=Blinds,
	pdftoolbar=false,
	pdfmenubar=false
]{hyperref}

\usepackage{cleveref}

\usepackage{cite}

\def\w        {\textit{\textbf{w}}}
\def\u        {\textit{\textbf{u}}}
\def\v        {{\textit{\textbf{v}}}}

\def\x  {\boldsymbol x}
\def \xR {\mathbb{R}}
\def\R {\xR}
\def\xLtwo{{L}^2}
\def\xHone{{H}^1}
\def\xCtwo{{C}^2}

\def\U        {\textit{\textbf{U}}}

\def\Om       {\Omega}

\def\n        {{\boldsymbol n}}

\def\eps        {\varepsilon}
\def \dt{\Delta t}
\def \pt{\delta_t}
\def \boldphi{\boldsymbol{\phi}}
\def \boldmu{\boldsymbol{\mu}}
\def \ODTwo{\mbox{\scriptsize{OD2}}}
\def \LT{\mbox{\scriptsize{LT}}}

\def \hueco{\noalign{\medskip}}
\def \beq{\begin{equation}}
\def \eeq{\end{equation}}
\def \ba{\begin{array}}
\def \ea{\end{array}}
\def \dis{\displaystyle}

\newcolumntype{Y}{>{\centering\arraybackslash}X} 

\newtheorem{thrm}{Theorem}[section] 
\newtheorem{lmm}[thrm]{Lemma} 
 
\newtheorem{prpstn}[thrm]{Proposition} 
\newtheorem{bsrvtn}[thrm]{Observation}
\newtheorem{dfntn}[thrm]{Definition}

\begin{document}

\title{Efficient linear schemes for a penalized ternary Cahn-Hilliard system} 

%

\author{Justin Swain and Giordano Tierra\\
\texttt{JustinSwain2@my.unt.edu}, \texttt{gtierra@unt.edu} \\
Department of Mathematics, University of North Texas, Denton TX (USA)
}

\date{}

%


\maketitle
%

\section{Introduction}\label{sec:intro}
The Cahn-Hilliard equation \cite{Allen1979,Cahn1961,Cahn1958} is a physical model for phase separation with applications in physics, chemistry, biology, and engineering such as the spinodal decomposition, tumor growth \cite{Xu2020}, multi-phase fluid flows \cite{Kim2012}, vesicle membrane deformation \cite{Guillen2018}, and diblock copolymers \cite{Cheng2017}. Diffuse interface models such as this consider phases separated by a small interface of positive thickness, so that each phase changes continuously across the interface. Much work has been devoted to investigate the binary phase system using a free energy containing a double well potential. Additionally, many authors have coupled the Cahn-Hilliard equation with Navier-Stokes to include hydrodynamic effects with phase separation \cite{Hohenberg1977}. 

The biphasic Cahn-Hilliard system can be extended to a ternary phase model whose dynamics simulate the separation process of three different phases \cite{Boyer2006}. Additionally, this model has also been coupled with Navier-Stokes for a hydrodynamic ternary phase separation model \cite{Brannick2015,Dong2014}. The study of this ternary Cahn-Hilliard model is relatively new and poses added challenges compared with the two phase model. The general problem is a system of fourth order partial differential equations that consists of three unknowns coupled through several nonlinear terms. This model is thermodynamically consistent, in the sense that the dynamics of the system are determined by the dissipation of a free energy. This free energy consists of a sum of a capillary term given by the gradients of the individual component's scalar order parameters, and an extension of the nonlinear double well potential from the biphasic free energy. 
{The phase separation of alloys with two or more components is studied in \cite{Eyre93}, with special attention given to the differences between multicomponent and binary alloys, finding that for the ternary case intermediate products with  both separated and metastable phases can be found, which is not seen in binary materials.}

We should note some previous works designing numerical schemes for the ternary Cahn-Hilliard model. One major difficulty for designing numerical schemes for this model is the non-convex nature of the potential in the total free energy. For this reason, some authors consider a convex splitting method \cite{Boyer2011,Chen2020, Lee2024}, which may introduce large numerical dissipation into the system \cite{Tierra2014, Xu2019}. A similar idea has been applied to the hydrodynamic model in \cite{Kim2004}. Energy stable numerical schemes based on a Lagrange multiplier approach were presented in \cite{Yang2020, Yang2021}, which rewrites the unknowns to remove the nonlinearity but adds an additional equation to the system. Recently, an energy quadratization (IEQ) approach \cite{GongIEQ} has become popular for energy based models such as this, and in \cite{Yang2017} the authors use IEQ for the ternary Cahn-Hilliard model. Another approach based on a scalar auxiliary variable (SAV) \cite{GongSAV} has been applied to this problem in \cite{Wu2023,Zhang2020}. It is common for authors to design numerical schemes which rewrite the ternary model with only two unknowns using a total volume relationship. This has the advantage of reducing the complexity of the model by eliminating two of the six partial differential equations in the system, however, the reformulation introduces additional nonlinearities into an already highly nonlinear problem, which adds to the challenge of developing efficient numerical schemes. Another idea is to use a Lagrange multiplier to enforce the total volume constraint \cite{Boyer2006,Boyer2011,Boyer2014} but this adds additional coupling terms to the system. {Additionally, there have been many works focusing on studying theoretically and numerically Cahn-Hilliard systems for phase separation of multi-component mixtures considering non-smooth free energies and degenerate mobility matrices, for some of the most relevant works in this direction check \cite{BarretBlowey99,BarrettBlowey99b,BarrettBloweyGarcke01,Nurnberg09} 
and the references therein. 
}

In this work we propose (1) a new way of writing the model which penalizes a total volume constraint and (2) new accurate and efficient numerical schemes for this new model. {Our formulation distinguishes from previous results in instead on enforcing the total volume constraint exactly, we propose to handle the restriction using a penalization approach. Then, we combine ideas previously exploited for deriving accurate and efficient numerical schemes in other energy-based systems to develop numerical problems for this new formulation of the problem.}
In particular, we propose three different conservative and linear numerical schemes for a penalized ternary Cahn-Hilliard model which is thermodynamically consistent. The first scheme uses a truncation of the potential function so that we can add enough numerical dissipation to guarantee the decreasing energy property at the discrete level. Following this first scheme, we present another first order accurate, linear, conservative, and decoupled scheme identical to the first one except without the truncated potential function. This second scheme promises added computational efficiency but at the cost of not being provably unconditionally energy stable. Finally, we will present a second order accurate, linear, and conservative numerical scheme with coupled unknowns. This is done using an adapted second order optimal dissipation algorithm \cite{Tierra2014} so that the added numerical dissipation is kept low allowing for the discrete dynamics to most closely resemble the dynamics of the continuous solution.

The contents of this work are organized as follows. In Section~\ref{sec:model} we summarize the ternary Cahn-Hilliard model, and present a modified model with energetic penalization in Section~\ref{sec:modifiedModel}. Next, in Section~\ref{sec:numericalSchemes} we introduce our new numerical schemes and their main properties. Afterwards, in Section~\ref{sec:simulations}, we present results of several numerical simulations to showcase the efficiency and accuracy of each scheme. Finally, we discuss the conclusions from this work in Section~\ref{sec:conclusion}.

\section{The Model}\label{sec:model}

In this section we start by reviewing the main ideas that have been considered in the literature to model mixtures of three components using the phase field approach by using Cahn-Hilliard type systems. Then we propose a new formulation that includes a penalization of a total volume constraint in the energy. Finally, we will show how this model can be easily extended to $N$-components.

\subsection{The Ternary Cahn-Hilliard Model}
Let $\Omega\subset\xR^d$ (with $d=1,2,3$) be a {bounded, convex polygonal spatial
domain with Lipschitz boundary $\partial\Omega$ }and $[0,T]$ a finite time interval. The state vector $\boldphi=(\phi_1,\phi_2,\phi_3)$ represents the volume fraction of the three components of the mixture so that $\phi_i=1$ when in the $i$-th phase and $\phi_i=0$ in the other two phases. A thin, smooth interfacial layer of thickness associated with the parameter $\eps>0$ connects the phases so that each $\phi_i$ varies smoothly between $0$ and $1$. 
The total free energy of the system is defined as
\beq\label{eq:totalFreeEnergy}
\mathcal{E}(\boldphi)
\,\vcentcolon=\,
\int_\Omega 
\Big(
{\mathrm{G}}(\boldphi)
+ {\mathrm{F}}(\boldphi)
\Big)d\x\,,
\eeq
where  ${\mathrm{G}}(\boldphi)$ represents the \textit{phillic} part of the energy, which includes the cappillary effects such that
\beq
{\mathrm{G}}(\boldphi)
\,\vcentcolon=\,
   \frac{3\eps}{4}\Sigma_1 G(\phi_1)
+ \frac{3\eps}{4}\Sigma_2 G(\phi_2)
+ \frac{3\eps}{4}\Sigma_3 G(\phi_3) \, ,
\eeq
with
\beq\label{eq:functionsG}
G(\phi_i)\vcentcolon=\frac12|\nabla\phi_i|^2 
\quad\quad(i=1,2,3)\,.
\eeq
It is also possible to consider cross-capillary terms of the form $\nabla\phi_i \cdot \nabla\phi_j$ in this part of the energy as seen in \cite{Wu2017}. Here, the coefficients $\boldsymbol{\Sigma} = \left(\Sigma_1 , \Sigma_2 , \Sigma_3 \right)$ are the spreading coefficients which describe the interaction of phase $i$ at the interface of phases $j$ and $k$. These coefficients are related to the pairwise surface tension parameters $\sigma_{ij}\geq0$ $( i,j=1,2,3 )$ by
\beq\label{eq:spreadingCoefficients}
\Sigma_i = \sigma_{ij} + \sigma_{ik} - \sigma_{jk} \, .
\eeq
Since $\Sigma_i + \Sigma_j = 2\sigma_{ij} \geq 0$, only one of the spreading coefficients can be negative at a time and this situation is referred to the case of {\it total spreading}. Otherwise, we have {\it partial spreading}, i.e. when $\Sigma_i\geq 0$ for all $i=1,2,3$.

The potential ${\mathrm{F}}(\boldphi)$ represents the \textit{phobic} effects such that it models the tendency of the components to be in the pure states by using a non-convex function.

System dynamics are represented by an $H^{-1}$-gradient flow of the free energy functional given by a system of three fourth order PDEs: Find $\boldphi(\x , t) = \left( \phi_1 (\x , t) , \phi_2 (\x , t) , \phi_3 (\x , t)\right)$ such that
\beq\label{eq:gradientFlow}
\boldphi_t - \nabla \cdot \left[ \mathbf{M} (\boldphi) \nabla \left( \frac{\delta \mathcal{E}}{\delta \boldphi} \right) \right] = \bold{0}, \quad \quad \mbox{ for } (\x, t) \in \Om \times (0,T) \, .
\eeq
Here, $\frac{\delta \mathcal{E}}{\delta \boldphi} = \left(\frac{\delta \mathcal{E}}{\delta \phi_1} ,  \frac{\delta \mathcal{E}}{\delta \phi_2} , \frac{\delta \mathcal{E}}{\delta \phi_3} \right)$ denotes the Riesz identification in $H^{-1}(\Omega)$  of the variational derivative of the functional $\mathcal{E}(\boldphi)$ with respect to $\boldphi$ and {$\mathbf{M}(\boldphi) = \left( M_{ij}(\boldphi) \right)$ is the mobility matrix, which has to be symmetric and semi-positive definite \cite{ElliottLuckhaus91}. Moreover, the volume constraint will imply that the mobility matrix has to satisfy $\mathbf{M}(\boldphi)\mathbf{e} = \mathbf{0}$, with $\mathbf{e}=(1,1,\dots,1)^T$. } 
Additionally the system is complemented with the following boundary and initial conditions:
\beq\label{eq:CHbc}
\partial_\n \boldphi|_{\partial\Omega} \, = \, \bold{0} \, , 
\quad 
\partial_\n \left(\frac{\delta \mathcal{E}}{\delta \boldphi}\right)\Big|_{\partial\Omega}\,=\,\bold{0}\,,
\quad
0\leq{\phi_i}(\x,0)\leq 1 \,\quad (i=1,2,3)
\quad
\mbox{ and }
\quad
\sum_{i=1}^3 \phi_i(\x,0) = 1\,.
\eeq

The model presented in \cite{Boyer2006} has been used as the standard for these types of systems. In that work the authors study which choices of the potential $F(\boldphi)$ are admissible in order to be compatible with enforcing a constraint of the total volume by introducing a Lagrange multiplier to impose 
\beq\label{eq:constaintsum1}
\sum_{i=1}^{3}\phi_i(\x,t)
\,=\,
1\,\quad \forall\, (\x,t)\in\Omega\times(0,T]\,.
\eeq
Interestingly, the authors are able to rewrite the system eliminating the Lagrange multiplier from the equations, which reduces the complexity of the system. But in order to achieve this, the authors need to {consider mobilities $M_i(\boldphi)$} such that there exists a function $M_0(\boldphi)$ satisfying
\beq\label{eq:sigmaAndM}
M_0(\boldphi)
\,=\,
\Sigma_1M_1(\boldphi)
\,=\,
\Sigma_2M_2(\boldphi)
\,=\,
\Sigma_3M_3(\boldphi)
\,.
\eeq
{
In fact, the same constraint arises when the model is written in terms of conservative diffusion fluxes \cite{Boyer2014}, where the total diffusion of each component  is written as a sum (this arises specifically due to the constraint $\Sigma_{i} \phi_i(x,t)=1$). Although this requirement is mathematically reasonable, it can be interpreted as completely different materials needing to have their mobility coefficients values linked, which seems a bit unnatural.}
\\
For the purpose of developing numerical schemes it is usual to remove the fourth order derivative by calling the functional derivative of the energy the {\it chemical potential} and denoting the components by $\mu_i$ for $i=1,2,3$. Thus each subsystem can be written as
\beq\label{eq:tCH3}
\left\{
\ba{l} \dis
(\phi_i)_t - \nabla \cdot \big[ M_i(\boldphi) \nabla \mu_i \big] = 0\, , \\ \hueco \dis
\mu_i = \frac{\delta \mathcal{E}}{\delta \phi_i} = - \frac{3\eps}{4}\Sigma_i\Delta\phi_i
+\frac{4\Sigma_T}{\varepsilon} 
\sum_{i\neq j}\left(
\frac1{\Sigma_j}(\partial_i {\mathrm{F}}(\boldphi) - \partial_j {\mathrm{F}}(\boldphi) )
\right)
\, ,
\ea
\right. 
\eeq
where $3\Sigma_T^{-1}\vcentcolon=\Sigma_1^{-1} +\Sigma_2^{-1} +\Sigma_3^{-1}$ . In particular, now one can take advantage of \eqref{eq:constaintsum1} to eliminate the computation of one of the unknowns, solving only two systems of equations instead of three.
\\
This problem is both conservative, i.e.
$$
\frac{d}{dt} \left(\int_\Omega  \phi_i d\x\right)= 0 \quad \quad \mbox{ for } i=1,2,3 \, ,
$$
and satisfies a dissipative energy law.
\begin{lmm}
System \eqref{eq:tCH3} satisfies the following dissipative energy law
$$
\frac{d}{dt}\mathcal{E}(\boldphi)
+ \|\sqrt{M_1(\boldphi)}\nabla\mu_1\|^2_{\xLtwo}
+ \|\sqrt{M_2(\boldphi)}\nabla\mu_2\|^2_{\xLtwo}
+ \|\sqrt{M_3(\boldphi)}\nabla\mu_3\|^2_{\xLtwo}
\,=\,
0\,.
$$
\end{lmm}
\begin{proof}
Testing \eqref{eq:tCH3} by $(\mu_1,(\phi_1)_t,\mu_2,(\phi_2)_t,\mu_3,(\phi_3)_t)$.
\end{proof}

In \cite{Boyer2006} the authors provide the conditions for the existence and uniqueness of solutions to problem \eqref{eq:CH3}. We summarize these results now. The following needs to be true for the phobic term ${\mathrm{F}}(\boldsymbol\phi)\in \xCtwo$ for all $\boldsymbol\phi$ satisfying $\sum\phi_i = 1$:
\beq\label{eq:condF}
\ba{l} \dis
{\mathrm{F}}(\boldsymbol\phi) \geq 0\, , \quad \quad 
|{\mathrm{F}}(\boldsymbol\phi) | \leq B_1 |\boldsymbol\phi|^p + B_2 \, ,  \quad \quad 
|\nabla_{\boldsymbol\phi} {\mathrm{F}}(\boldsymbol\phi)| \leq B_1 |\boldsymbol\phi|^{p-1} + B_2 \, , \\ \hueco \dis
|\nabla^2_{\boldsymbol\phi} {\mathrm{F}}(\boldsymbol\phi)| \leq B_1 |\boldsymbol\phi|^{p-2} + B_2 \, , \quad \quad 
(\nabla^2_{\boldsymbol\phi} {\mathrm{F}}(\boldsymbol\phi)\xi, \xi) \geq -D_1(1+|\boldsymbol\phi|^q)|\xi|^2 \, , \quad \forall\, \xi\in\R^3\, ,
\ea
\eeq
for some constants $B_1,B_2>0$, and $D_1\geq 0$ where $p=6$ and $0\leq q < 4$ if $d=3$, and $2\leq p < \infty$ and $0\leq q < \infty$ if $d=2$.
\begin{thrm}\label{th:Wellposedness}\cite{Boyer2006}
Let $\Om$ be a bounded smooth domain in $\R^d$ with $d=2,3$. Assume constant mobility $\mathbf{M}(\boldphi) = (M_1, M_2, M_3)$, with $M_i\geq 0$, $(\Sigma_1 , \Sigma_2 , \Sigma_3)$ satisfying 
$$
\Sigma_1 \Sigma_2 + \Sigma_1 \Sigma_3 + \Sigma_2 \Sigma_3 > 0 \, ,
$$
and ${\mathrm{F}}(\boldsymbol\phi)$ satisfying \eqref{eq:condF}. Then for any $\boldsymbol\phi_0\in(\xHone(\Om))^3$ satisfying \eqref{eq:CHbc} there exists a unique weak solution $(\boldsymbol\phi,\boldsymbol\mu)$ of \eqref{eq:CH3}.
\end{thrm}

\subsubsection{Particular choice of potential ${\mathrm{F}}(\phi)$}

The authors in \cite{Boyer2006} study several possible choices for ${\mathrm{F}}(\boldphi)$ and they arrive to one that has been considered in many other works \cite{Boyer2011,Chen2020,Kim2012,Lee2024,Wu2017,Yang2017,Yang2018,Zhao2017}: 
\beq\label{eq:potentialF}
{\mathrm{F}}(\boldphi)
\,\vcentcolon=\,
    \frac{24}{\eps}\Sigma_1 F(\phi_1)
+  \frac{24}{\eps}\Sigma_2 F(\phi_2)
+  \frac{24}{\eps}\Sigma_3 F(\phi_3)
+  \frac{24} \eps\Lambda F_{123}(\boldphi)\,,
\eeq
with
\beq\label{eq:potentialFunctions}
F(\phi_i)
\,\vcentcolon=\,\dis
 \frac14 \phi_i^2(1 - \phi_i)^2
\,=\,
\frac14\phi_i^4 - \frac12\phi_i^3 + \frac14\phi_i^2
\quad\quad(i=1,2,3)
\eeq
and
$$
F_{123}(\boldphi)
\,\vcentcolon=\,
\frac12 \phi_1^2\phi_2^2\phi_3^2\,,
$$
where the parameter $\Lambda\geq 0$ is chosen large enough so that $F(\boldphi)$ is bounded from below \cite{Boyer2006,Boyer2011}. 
In particular we also have
\beq\label{eq:functionsf}
\ba{rcccl}
f(\phi_i)
&\vcentcolon=&\dis
\frac{\partial F}{\partial \phi_i}
&=&\dis
\left(\phi_i - \frac12\right)^3 - \frac14\left(\phi_i - \frac12\right)
\quad\quad(i=1,2,3)\,,
\\ \hueco
(f_{123})_i(\boldphi)
&\vcentcolon=&\dis
\frac{\partial F_{123}}{\partial \phi_i}
&=&\dis
\phi_i\phi_j^2\phi_k^2 \hfill\quad(i,j,k=1,2,3\mbox{ and } i\neq j,k)\, .
\ea
\eeq
\begin{bsrvtn}
In \cite{Yuan2021} the authors propose a modification of this approach, where instead of approximating system \eqref{eq:tCH3} with potential term ${\mathrm{F}}(\boldphi)$, they directly rewrite the total energy of the system in terms of only two of the unknowns (implicitly assuming the constraint \eqref{eq:constaintsum1}) and then they derive the equations of the system using Cahn-Hilliard dynamics. Although it is not clear that this model and the one in equation \eqref{eq:tCH3} are equivalent, the presented numerical results seem to validate their approach.
\end{bsrvtn}
\subsection{A Penalized Ternary Cahn-Hilliard Model}\label{sec:modifiedModel}
Our goal is to present a new formulation of the ternary Cahn-Hilliard model such that it avoids the restriction \eqref{eq:sigmaAndM} and at the same time is easy to extend to N-component models. Although mathematically it makes sense to restrict to \eqref{eq:sigmaAndM}, from the physical point of view there is not a clear justification of why the mobilities should be linked, 
{and depending on how you interpret the mobility, restriction \eqref{eq:sigmaAndM} might mean that total spreading implies a negative value of one of the mobility terms.}


In this section we introduce a modified ternary Cahn-Hilliard model which is based on adding a penalization term to the energy to ensure that the total sum of the volume fraction components is conserved. In the continuous problem the phase components satisfy a total volume relationship 
$$
\phi_1 + \phi_2 + \phi_3 = 1 \, ,
$$
and therefore many authors consider solving the system with only two unknowns by letting $\phi_3 = 1 - \phi_1 - \phi_2$. Thus, the problem is reduced to a system of two coupled fourth order PDEs with two unknowns, which by introducing two chemical potentials leads to a system of four second order PDEs with four unknowns. 
In this approach, by using the total volume relationship to eliminate one unknown, {the coupling nonlinearities in the energy change from being quadratic on each of the three unknowns ($\phi_1^2\phi_2^2\phi_3^2$) to be quartic on each of the two unknowns ($\phi_1^2\phi_2^2(1 - \phi_1 - \phi_2)^2$), resulting in a system of PDEs with nonlinear terms that are more challenging to derive energy-stable numerical schemes compared with the three components case. In our approach, by keeping the three components we are able to derive schemes that decouple the computation of each unknown in a linear way, leading to very efficient numerical schemes.}

The main idea is to consider a modified energy functional which considers all three components of the volume fraction given by
$$
E(\boldphi)
\,:=\,
\int_\Omega 
\Big(
{\mathrm{G}}(\boldphi)
+ {\mathrm{F}}(\boldphi)
+ P(\boldphi)
\Big)d\x\,,
$$
where $P(\boldphi)$ represents a penalization of the constraint $\dis\sum_{i=1}^3 \phi_i=1$ such that
$$
P(\boldphi)
\,=\,
\frac1{2{\lambda}}(\phi_1 + \phi_2 + \phi_3 - 1)^2 \,,
$$
{with $\lambda>0$ denoting a penalization parameter.} As before, the dynamics of the modified system are written as a gradient flow of the free energy functional $E(\boldsymbol\phi)$. Therefore, the penalized ternary Cahn-Hilliard system is written as a system of three fourth order PDEs
\beq\label{eq:CH}
\dis
(\phi_i)_t - \nabla \cdot\left[
{M_i(\boldphi)}
\nabla\left(\frac{\delta E}{\delta \phi_i}\right)\right] 
=0\, ,
\quad\quad\quad \mbox{ for }  i=1,2,3\,.
\eeq
{From now on we focus on the case with constant mobilities, that is $M_i(\boldphi)=M_i>0$ ($i=1,2,3$), which is the closest case to the well-posed case studied in Theorem~\ref{th:Wellposedness}. }
Our goal is to develop linear numerical schemes to solve the problem \eqref{eq:CH} subject to \eqref{eq:CHbc} by way of a continuous Finite Element approximation. We rewrite the system using the chemical potentials $\mu_i$ as follows:
$$
\mu_i
\,:=\,
\dis
\frac{\delta E}{\delta \phi_i}
\,=\,
-\frac{3\eps}{4}\Sigma_i\Delta\phi_i 
+  \frac{24}{\eps}\Sigma_if(\phi_i)
+ \frac{24} \eps\Lambda (f_{123})_i(\boldphi)
+p(\boldphi)\,,
$$
where
$$
p(\boldphi) 
\,:=\,
\frac{\partial P}{\partial \phi_i}
\,=\,
\dis
\frac1{{\lambda}}(\phi_1 + \phi_2 + \phi_3 - 1) \, ,
\hfill\quad \forall\, i=1,2,3 \,.
$$
In summary, the system that we consider is a nonlinear coupled system of six second order PDEs:
\beq\label{eq:CH3}
\left\{\ba{rcl}\dis
(\phi_1)_t - M_1\Delta\mu_1
&=&0\, ,
\\ \hueco\dis
-\frac{3\eps}{4}\Sigma_1\Delta\phi_1
+  \frac{24}{\eps}\Sigma_1 f(\phi_1)
+ \frac{24} \eps\Lambda \phi_1\phi_2^2\phi_3^2
+ p(\boldphi)
&=&\mu_1\, ,
\\ \hueco\dis
(\phi_2)_t - M_2\Delta\mu_2
&=&0\, ,
\\ \hueco\dis
-\frac{3\eps}{4}\Sigma_2\Delta\phi_2
+  \frac{24}{\eps}\Sigma_2 f(\phi_2)
+ \frac{24} \eps\Lambda  \phi_1^2\phi_2\phi_3^2
+ p(\boldphi)
&=&\mu_2\, ,
\\ \hueco\dis
(\phi_3)_t -  M_3\Delta\mu_3
&=&0\, ,
\\ \hueco\dis
-\frac{3\eps}{4}\Sigma_3\Delta\phi_3
+  \frac{24}{\eps}\Sigma_3 f(\phi_3)
+ \frac{24} \eps\Lambda  \phi_1^2\phi_2^2\phi_3
+ p(\boldphi)
&=&\mu_3\, .
\ea\right.
\eeq
Similar to \eqref{eq:tCH3}, system \eqref{eq:CH3} is conservative, and satisfies a dissipative energy law with respect to the energy $E(\boldphi)$.
\begin{prpstn}\label{prop:ch3ModEnergyLaw}
System \eqref{eq:CH3} satisfies the following dissipative energy law
$$
\frac{d}{dt}E(\boldphi)
+ M_1\|\nabla\mu_1\|^2_{\xLtwo}
+ M_2\|\nabla\mu_2\|^2_{\xLtwo}
+ M_3\|\nabla\mu_3\|^2_{\xLtwo}
\,=\,
0\,.
$$
\end{prpstn}
\begin{proof}
{
Testing equations 
\eqref{eq:CH3}$_1$ by $\mu_1$, 
\eqref{eq:CH3}$_2$ by $(\phi_1)_t$,
\eqref{eq:CH3}$_3$ by $\mu_2$,
\eqref{eq:CH3}$_4$ by $(\phi_2)_t$,
\eqref{eq:CH3}$_5$ by $\mu_3$,
\eqref{eq:CH3}$_6$ by $(\phi_3)_t$
and adding all together.
}
\end{proof}

\subsubsection{{Extension to a Ternary Navier-Stokes-Cahn-Hilliard Model}}\label{sec:NSCH3}
{
System \eqref{eq:CH3} can be extended to represent mixtures of three newtonian and incompressible fluids with constant and equal densities and with different viscosities ($\nu_1$, $\nu_2$ and $\nu_3$), arriving to the following system:
\beq\label{eq:NSCH3}
\left\{\ba{rcl}\dis
\u_t 
+ (\u\cdot\nabla)\u
- \nabla\cdot\big(\nu(\phi_1,\phi_2,\phi_3)\nabla\u\big)
+\nabla p
+ \sum_{i=1}^{3}\phi_i\nabla\mu_i
&=&\bold{0}\,,
\\ \hueco\dis
\nabla\cdot\u
&=&0\,,
\\ \hueco\dis
(\phi_1)_t 
+ \nabla\cdot(\phi_1\u)
 - M_1\Delta\mu_1
&=&0\, ,
\\ \hueco\dis
-\frac{3\eps}{4}\Sigma_1\Delta\phi_1
+  \frac{24}{\eps}\Sigma_1 f(\phi_1)
+ \frac{24} \eps\Lambda \phi_1\phi_2^2\phi_3^2
+ p(\boldphi)
&=&\mu_1\, ,
\\ \hueco\dis
(\phi_2)_t 
+ \nabla\cdot(\phi_2\u)
- M_2\Delta\mu_2
&=&0\, ,
\\ \hueco\dis
-\frac{3\eps}{4}\Sigma_2\Delta\phi_2
+  \frac{24}{\eps}\Sigma_2 f(\phi_2)
+ \frac{24} \eps\Lambda  \phi_1^2\phi_2\phi_3^2
+ p(\boldphi)
&=&\mu_2\, ,
\\ \hueco\dis
(\phi_3)_t 
+ \nabla\cdot(\phi_3\u)
-  M_3\Delta\mu_3
&=&0\, ,
\\ \hueco\dis
-\frac{3\eps}{4}\Sigma_3\Delta\phi_3
+  \frac{24}{\eps}\Sigma_3 f(\phi_3)
+ \frac{24} \eps\Lambda  \phi_1^2\phi_2^2\phi_3
+ p(\boldphi)
&=&\mu_3\,,
\ea\right.
\eeq
where $(\u,p)$ denotes the velocity and pressure of the fluid,  and $\nu(\phi_1,\phi_2,\phi_3)>0$ is a viscosity function such that $\nu(\phi_1,\phi_2,\phi_3)=\nu_i$ in regions where $\phi_i=1$ ($i=1,2,3$). 
Similar to \eqref{eq:CH3}, system \eqref{eq:NSCH3} is conservative, and satisfies a dissipative energy law with respect to the total energy $E_{Tot}(\u,\boldphi)$, which is defined as the addtion between $E(\boldphi)$ and the kinetic energy, that is:
$$
E_{Tot}(\u,\boldphi)
\,=\,
E_{Kin}(\u)
+
E(\boldphi)
\,=\,
\int_\Omega\frac12|\u|^2 d\x 
+ 
E(\boldphi)\,.
$$
\begin{prpstn}\label{prop:nsch3ModEnergyLaw}
System \eqref{eq:NSCH3} satisfies the following dissipative energy law
\beq\label{eq:energylawNSCH3}
\frac{d}{dt}E(\u,\boldphi)
+ \|\sqrt{\nu(\phi_1,\phi_2,\phi_3)}\nabla\u\|^2_{\xLtwo}
+ M_1\|\nabla\mu_1\|^2_{\xLtwo}
+ M_2\|\nabla\mu_2\|^2_{\xLtwo}
+ M_3\|\nabla\mu_3\|^2_{\xLtwo}
\,=\,
0\,.
\eeq
\end{prpstn}
\begin{proof}
Testing equations 
\eqref{eq:NSCH3}$_1$ by $\u$, 
\eqref{eq:NSCH3}$_2$ by $p$,
\eqref{eq:NSCH3}$_3$ by $\mu_1$,
\eqref{eq:NSCH3}$_4$ by $(\phi_1)_t$,
\eqref{eq:NSCH3}$_5$ by $\mu_2$,
\eqref{eq:NSCH3}$_6$ by $(\phi_2)_t$
\eqref{eq:NSCH3}$_7$ by $\mu_3$,
\eqref{eq:NSCH3}$_8$ by $(\phi_3)_t$
and adding all together.
\end{proof}
}
\subsubsection{{Potential} extension to $N$-Component Versions of the Cahn-Hilliard Model}\label{sec:nComponentModel}
{There are several works in the literature focusing on extensions of the Cahn-Hilliard equations to $N$-components versions using different ideas (check \cite{dong15,Wu2017,Yang2018} and references therein). In this section we don't work directly with any of them but we want to emphasize that all have something in common, the approaches are based on introducing two types of terms in the total energy of the form: (a) terms of the form $\nabla\phi_i\cdot\nabla\phi_j$, which lead to new linear terms on the resulting PDE system, and (b) 
terms with products of squares of the phases ($\phi_i^2\phi_j^2$, $\phi_i^2\phi_j^2\phi_k^2$...up to products including all the phases) which lead to new nonlinear and coupling terms in the PDE system.
In this section we outline how our model can be easily extended to systems where terms of the type mentioned in (b) appears. We just focus on the most complicated of these products, that is, the one with all the components. We call this model an $N$-component version of the Cahn-Hilliard model \eqref{eq:CH3}, although we want to clarify that we don't claim any physical relevance of this model, we just want to illustrate later on that the proposed numerical schemes can handle any of these type of terms. 
}
Considering the energy 
\beq
E_N(\boldphi) 
\, \vcentcolon= \,
\int_\Om \left( {\mathrm{G}}_N(\boldphi) + {\mathrm{F}}_N(\boldphi) + P_N(\boldphi)  \right) d\x\,,
\eeq
where
$$
{\mathrm{G}}_N(\boldphi) 
\,\vcentcolon=\, \dis
\sum_{i=1}^N \frac{3\eps}{4}\Sigma_i G(\phi_i) \, ,\quad\quad 
{\mathrm{F}}_N(\boldphi) \,\vcentcolon=\, \dis
\frac{12}{\eps}\Lambda \phi_1^2 \phi_2^2 \dots \phi_N^2 + \sum_{i=1}^N \frac{24}{\eps} F(\phi_i)
$$
and
$$
P_N(\boldphi)
\,\vcentcolon=\, \dis
\frac{1}{2{\lambda}} (\phi_1 + \phi_2 + \dots + \phi_N - 1)^2 \, ,
$$
with $G(\phi_i)$ and $F(\phi_i)$ defined as in \eqref{eq:functionsG} and \eqref{eq:potentialFunctions}, respectively. As before, by considering a gradient flow of the energy and rewriting using chemical potentials we arrive at the full $N$-component PDE system
\beq\label{eq:CHN}
\left\{\ba{rcl}\dis
(\phi_1)_t - M_1\Delta\mu_1
&=&0\, ,
\\ \hueco\dis
-\frac{3\eps}{4}\Sigma_1\Delta\phi_1
+  \frac{24}{\eps}\Sigma_1 f(\phi_1)
+ \frac{24} \eps\Lambda \phi_1\phi_2^2\dots\phi_N^2
+ p(\boldphi)
&=&\mu_1\, ,  \\ \hueco
(\phi_2)_t - M_2\Delta\mu_2
&=&0\, ,
\\ \hueco\dis
-\frac{3\eps}{4}\Sigma_2\Delta\phi_2
+  \frac{24}{\eps}\Sigma_2 f(\phi_2)
+ \frac{24} \eps\Lambda \phi_1^2\phi_2\dots\phi_N^2
+ p(\boldphi)
&=&\mu_2\, , \\ \hueco
&\vdots& \\ \hueco
(\phi_N)_t - M_N\Delta\mu_N
&=&0\, ,
\\ \hueco\dis
-\frac{3\eps}{4}\Sigma_N\Delta\phi_N
+  \frac{24}{\eps}\Sigma_N f(\phi_N)
+ \frac{24} \eps\Lambda \phi_1^2\phi_2^2\dots\phi_N
+ p(\boldphi)
&=&\mu_N\, , 
\ea\right. 
\eeq
where $f(\phi_i)$ is given by \eqref{eq:functionsf}, and
$$
p(\boldphi)
\,=\,
\frac{\partial P}{\partial \phi_i}
\,= \,
\frac{1}{{\lambda}}( \phi_1 + \dots + \phi_N - 1) \, .
$$
It is easy to see how this scheme follows an energy law that is an extension of Proposition~\ref{prop:ch3ModEnergyLaw}.
\begin{prpstn}
System \eqref{eq:CHN} satisfies the following dissipative energy law
$$
\frac{d}{dt}E_N(\boldphi)
+ \sum_{i=1}^NM_i\|\nabla\mu_i\|^2_{\xLtwo}
\,=\,
0\,.
$$
\end{prpstn}

\section{Numerical Schemes}\label{sec:numericalSchemes}
For designing numerical schemes our primary objectives will be threefold: (I) to have schemes that satisfy a dissipative energy law at the discrete level similar to the one of the continuous problem, (II) to have the discrete energy dissipation be close to the continuous problem so that the dynamics of the solution to the scheme are as close as possible to the solution of the continuous problem, and (III) to be as computationally inexpensive as possible. In doing so we face the challenge of how to deal with the nonlinear terms which couple the unknowns.

In this section we will introduce (\textit{i}) a linear, decoupled, and unconditionally energy stable Finite Element numerical scheme for a penalized Cahn-Hilliard system, (\textit{ii}) a similar scheme which is not unconditionally energy stable but promises greater computational efficiency, and (\textit{iii}) a provable second order in time linear scheme with coupled unknowns.

\subsection{Finite Element Space Discretization}
The numerical schemes proposed in this section are designed as approximations of the corresponding weak formulation of system \eqref{eq:CH3}. Hereafter $(\cdot,\cdot)$ denotes the $\xLtwo(\Omega)$-scalar product. Our generic Finite Element method is as follows. Let $P_h, M_h \subset \xHone(\Omega)$, with $P_h \subseteq M_h$, be two conformed Finite Element spaces for a
triangulation $\mathcal{T}_h$ of a bounded domain $\Om$ 
with {Lipschitz} polyhedric boundary $\partial\Om$, where $h$ denotes the size of the mesh. {We assume the triangulation to be regular and  
quasi-uniform in the stricter sense (i.e. is shape regular and there exists a uniform lower bound on the mesh size)\cite{Brenner,ErnGuermond}}.
In order to simplify notation we will omit the superscript $h$ when denoting functions that are discrete in space. 

The Finite Element approximation of our problem is as follows: for $i=1,2,3$, find $\phi_i(\mathbf{x} , t ) \in P_h$ and $\mu_i(\x, t) \in M_h$ such that 
$$
\phi_i(\mathbf{x} , 0 ) = \mathcal{I}_{P_h} \bm{\phi}_0\,,
$$
where $\mathcal{I}_{P_h}$ is the $\xLtwo$-projection operator onto $P_h$, and
\beq\label{eq:FEMCH3}
\left\{\ba{rcl}\dis
\left( (\phi_1)_t , \overline{\mu}_1 \right)
+ M_1\left( \nabla \mu_1 , \nabla \overline{\mu}_1 \right)
&=&0\, ,
\\ \hueco\dis
\frac{3\eps}{4}\Sigma_1 \left( \nabla \phi_1 , \nabla \overline{\phi}_1 \right)
+ \frac{24}{\eps}\Sigma_1 \left( f(\phi_1) , \overline{\phi}_1 \right)
+ \frac{24} \eps\Lambda \left( \phi_1\phi_2^2\phi_3^2 , \overline{\phi}_1 \right)
+ \left( p(\boldphi) , \overline{\phi}_1 \right)
&=& \left( \mu_1 , \overline{\phi}_1 \right) \, ,
\\ \hueco\dis
\left( (\phi_2)_t , \overline{\mu}_2 \right)
+ M_2\left( \nabla \mu_2 , \nabla \overline{\mu}_2 \right)
&=&0\, ,
\\ \hueco\dis
\frac{3\eps}{4}\Sigma_2 \left( \nabla \phi_2 , \nabla \overline{\phi}_2 \right)
+ \frac{24}{\eps}\Sigma_2 \left( f(\phi_2) , \overline{\phi}_2 \right)
+ \frac{24} \eps\Lambda \left( \phi_1^2\phi_2\phi_3^2 , \overline{\phi}_2 \right)
+ \left( p(\boldphi) , \overline{\phi}_2 \right)
&=& \left( \mu_2 , \overline{\phi}_2 \right)\, ,
\\ \hueco\dis
\left( (\phi_3)_t , \overline{\mu}_3 \right)
+ M_3\left( \nabla \mu_3 , \nabla \overline{\mu}_3 \right)
&=&0\, ,
\\ \hueco\dis
\frac{3\eps}{4}\Sigma_3 \left( \nabla \phi_3 , \nabla \overline{\phi}_3 \right)
+ \frac{24}{\eps}\Sigma_3 \left( f(\phi_3) , \overline{\phi}_3 \right)
+ \frac{24} \eps\Lambda \left( \phi_1^2\phi_2^2\phi_3 , \overline{\phi}_3 \right)
+ \left( p(\boldphi) , \overline{\phi}_3 \right)
&=& \left( \mu_3 , \overline{\phi}_3 \right)\, ,
\ea\right.
\eeq
for any $\left(\overline{\phi}_1, \overline{\phi}_2, \overline{\phi}_3\right) \in P^3_h$, and $\left(\overline{\mu}_1, \overline{\mu}_2, \overline{\mu}_3\right) \in M^3_h$.

\subsection{Time Discretization}
Consider a regular partition of a finite time interval $[0,T]$ into $N$ subintervals, with time step $\dt = T/N$. Some notation for the following sections are $t^n = n\dt$, $u^n$ denotes the approximation of $u(\x , t^n)$, $u^0 = \mathcal{I}_{P_h} u(\x , t)$, 
$$
\pt u^{n+1}\,:=\,\frac{u^{n+1} - u^n}{\Delta t}\, \quad \quad \mbox{ and } \quad \quad
u^{n+\frac12} \,:=\, \frac{u^{n+1} + u^n}{2} \, .
$$

\subsection{A Truncated, Decoupled, and First Order Scheme (TD1)}\label{sec:TD1}
The idea for this numerical scheme is to use a modified energy where the key point is to truncate the potentials $F(\phi_i)$ outside of the physically meaningful range $0 \leq \phi_i \leq 1$. Although a maximum principle is not expected for this model with constant mobilities, the assumption of modifying the growth of a double well potential has been widely considered before from the analytical and numerical point of view in different applications \cite{Cabrales2015,Caffarelli1995,ElHaddad2022,Nochetto2011,Shen2010}, because it is known that $\boldphi$ {(the solution of the continuous problem)} will remain bounded at least in {some particular} binary cases { under very specific conditions} \cite{Caffarelli1995}. The proposed modification of the potential $F(\phi)$ is defined as

$$
\widetilde{F}(\phi)
=\left\{
\ba{lr}
\dis\frac14(\phi - 1)^2
& \mbox{ if }
\phi>1\,,
\\ \hueco
\dis\frac14\phi^2(1 - \phi)^2
& \mbox{ if }
\phi\in[0,1]\,,
\\ \hueco
\dis\frac14\phi^2
& \mbox{ if }
\phi<0\,,
\ea
\right.
$$
with
$$
\widetilde{f}(\phi)
=\left\{
\ba{lr}
\dis\frac12(\phi - 1)
& \mbox{ if }
\phi>1\,,
\\ \hueco
\dis\phi^3 - \frac32\phi^2 + \frac12\phi
& \mbox{ if }
\phi\in[0,1]\,,
\\ \hueco
\dis\frac12\phi
& \mbox{ if }
\phi<0\,,
\ea
\right.
\quad\quad
\mbox{and} \quad \quad
\widetilde{f}'(\phi)
=\left\{
\ba{lr}
\dis\frac12
& \mbox{ if }
\phi>1\,,
\\ \hueco
\dis3\phi^2 - 3\phi + \frac12
& \mbox{ if }
\phi\in[0,1]\,,
\\ \hueco
\dis\frac12
& \mbox{ if }
\phi<0\, .
\ea
\right .
$$
Therefore the modified system that we are going to consider is:
\beq\label{eq:CH3mod}
\left\{\ba{rcl}\dis
(\phi_1)_t - M_1\Delta\mu_1
&=&0\, ,
\\ \hueco\dis
-\frac{3\eps}{4}\Sigma_1\Delta\phi_1
+  \frac{24}{\eps}\Sigma_1 \widetilde{f}(\phi_1)
+ \frac{24} \eps\Lambda \phi_1\phi_2^2\phi_3^2
+ p(\boldphi)
&=&\mu_1\, ,
\\ \hueco\dis
(\phi_2)_t - M_2\Delta\mu_2
&=&0\, ,
\\ \hueco\dis
-\frac{3\eps}{4}\Sigma_2\Delta\phi_2
+  \frac{24}{\eps}\Sigma_2 \widetilde{f}(\phi_2)
+ \frac{24} \eps\Lambda  \phi_1^2\phi_2\phi_3^2
+ p(\boldphi)
&=&\mu_2\, ,
\\ \hueco\dis
(\phi_3)_t -  M_3\Delta\mu_3
&=&0\, ,
\\ \hueco\dis
-\frac{3\eps}{4}\Sigma_3\Delta\phi_3
+  \frac{24}{\eps}\Sigma_3 \widetilde{f}(\phi_3)
+ \frac{24} \eps\Lambda  \phi_1^2\phi_2^2\phi_3
+ p(\boldphi)
&=&\mu_3\, .
\ea\right.
\eeq

\begin{lmm}
System \eqref{eq:CH3mod} satisfies the following dissipative energy law
\beq\label{eq:CH3modEnergyLaw}
\frac{d}{dt}\widetilde{E}(\boldphi)
+ M_1\|\nabla\mu_1\|^2_{\xLtwo}
+ M_2\|\nabla\mu_2\|^2_{\xLtwo}
+ M_3\|\nabla\mu_3\|^2_{\xLtwo}
\,=\,
0\,,
\eeq
where 
$$
\widetilde{E}(\boldphi)
\,:=\,
\int_\Omega 
\Big(
{\mathrm{G}}(\boldphi)
+ {\widetilde{\mathrm{F}}}(\boldphi)
+ P(\boldphi)
\Big)d\x\,,
$$
with
$$
{\widetilde{\mathrm{F}}}(\boldphi)
\,=\,
    \frac{24}{\eps}\Sigma_1 \widetilde{F}(\phi_1)
+  \frac{24}{\eps}\Sigma_2 \widetilde{F}(\phi_2)
+  \frac{24}{\eps}\Sigma_3 \widetilde{F}(\phi_3)
+  \frac{24} \eps\Lambda F_{123}(\boldphi)\,.
$$
\end{lmm}
\begin{proof}
{
Testing equations 
\eqref{eq:CH3mod}$_1$ by $\mu_1$, 
\eqref{eq:CH3mod}$_2$ by $(\phi_1)_t$,
\eqref{eq:CH3mod}$_3$ by $\mu_2$,
\eqref{eq:CH3mod}$_4$ by $(\phi_2)_t$,
\eqref{eq:CH3mod}$_5$ by $\mu_3$,
\eqref{eq:CH3mod}$_6$ by $(\phi_3)_t$
and adding all together.
}
\end{proof}
We now propose a conservative, linear, decoupled and energy stable numerical scheme for this modified model. The idea is to use an implicit-explicit scheme by first applying a second order optimal dissipation method (OD2) \cite{Tierra2014} to the function $\widetilde{f}(\phi)$, and to the penalization term $p(\phi_1, \phi_2 , \phi_3)$. The OD2 approximation $\widetilde{f}^{\mbox{\ODTwo}}$ for each function is given by
\beq\label{eq:od2}
\widetilde{f}^{\ODTwo}(\phi^{n+1}_i , \phi^n_i) \, = \,
\widetilde{f}(\phi^n_i) + \frac{1}{2} \widetilde{f}'_i (\phi^n_i) (\phi^{n+1}_i - \phi^n_i ).
\eeq
Next, we apply the same OD2 approximation to the penalization term $p(\boldphi)$:
\beq
\ba{rcl}
p^{\ODTwo}(\boldphi^{n+1} , \boldphi^n )
&=& \dis
\nabla_{\boldphi} P(\boldphi^n) + \frac{1}{2} H_P({\boldphi}) (\boldphi^{n+1} - \boldphi^n)\,, 
\ea
\eeq
where
$$
\nabla_{\boldphi}P(\boldphi)
\,=\,
p(\boldphi)
\begin{pmatrix}
1 \\ 1 \\ 1
\end{pmatrix}
\quad\quad
\mbox{ and }
\quad\quad
H_P({\boldphi})
\,=\,
\frac1{{\lambda}}
\left(\ba{ccc}
1& 1 & 1 
\\
1& 1 & 1 
\\
1& 1 & 1 
\ea\right).
$$
However, the second term in this approximation has the effect of coupling the unknowns. Therefore, we will replace the Hessian matrix with a lower triangular version of it, such that
\beq
p^{\LT}(\boldphi^{n+1} , \boldphi^n )
\, = \, \dis
p(\boldphi^n)
\begin{pmatrix}
1 \\ 1 \\ 1
\end{pmatrix}
+ \frac1{2{\lambda}} \begin{pmatrix}
1 & 0 & 0 \\
2 & 1 & 0 \\
2 & 2 & 1 
\end{pmatrix} (\boldphi^{n+1} - \boldphi^n ) \, ,
\eeq
{where $p^{\LT}(\boldphi^{n+1} , \boldphi^n )$ is an approximation of $p^{\ODTwo} (\boldphi^{n+1} , \boldphi^n )$ of order $\mathcal{O}(\Delta t)$ (at least heuristically), but it introduces the same numerical dissipation because it }satisfies the following relation
\beq 
(\boldphi^{n+1} - \boldphi^n)^T p^{\ODTwo} (\boldphi^{n+1} , \boldphi^n ) 
\, = \,
(\boldphi^{n+1} - \boldphi^n)^T p^{\LT} (\boldphi^{n+1} , \boldphi^n )  .
\eeq
A similar idea was developed in \cite{Swain2024} in the context of nematic liquid crystals to decouple the different unknowns.

We propose the following algorithm, where the computation of the unknowns in this scheme is decoupled and some additional numerical dissipation is introduced following the ideas presented in \cite{Wu2014}. Given $(\phi_1^n , \phi_2^n , \phi_3^n) \in P^3_h$:
\begin{itemize}
\item[\textbf{Step 1}] Compute $(\phi_1^{n+1},\mu_1^{n+1}) \in P_h\times M_h$ such that for all $(\bar\phi_1,\bar\mu_1) \in P_h\times M_h$ we have

\beq\label{eq:CH3_1}
\left\{\ba{rcl}\dis
\left(\frac{\phi_1^{n+1}- \phi_1^n}{\Delta t} ,\bar\mu_1\right) 
+M_1(\nabla\mu_1^{n+1},\nabla\bar\mu_1)
&=&0\, ,
\\ \hueco\dis
\frac{3\eps}{4}\Sigma_1(\nabla\phi_1^{n+\frac12},\nabla\bar\phi_1)
+\tau_1\Delta t\Big(\nabla(\phi_1^{n+1} - \phi_1^{n}),\nabla\bar\phi_1\Big) 
+ \frac{24} \eps\Lambda   \Big(\phi_1^{n+\frac12} (\phi_2^n)^2(\phi_3^n)^2,\bar\phi_1\Big)
\\ \hueco\dis
+  \frac{24}{\eps}\Sigma_1 (\widetilde f(\phi_1^n),\bar\phi_1)
\dis+  \frac{12}{\eps}\Sigma_1 \Big(\widetilde f'(\phi_1^n)(\phi_1^{n+1} - \phi_1^n),\bar\phi_1\Big)
\\ \hueco\dis
+ \frac1{{\lambda}} \Big(\phi^n_1 + \phi^n_2 + \phi^n_3 - 1 ,\bar\phi_1\Big)
+ \frac1{2{\lambda}}(\phi_1^{n+1} - \phi_1^n,\bar\phi_1)
&=&(\mu_1^{n+1},\bar\phi_1)\, .
\ea\right.
\eeq

\item[\textbf{Step 2}] Compute $(\phi_2^{n+1},\mu_2^{n+1}) \in P_h\times M_h$ such that for all $(\bar\phi_2,\bar\mu_2) \in P_h\times M_h$ we have

\beq\label{eq:CH3_2}
\left\{\ba{rcl}\dis
\left(\frac{\phi_2^{n+1}- \phi_2^n}{\Delta t} ,\bar\mu_2\right) 
+M_2(\nabla\mu_2^{n+1},\nabla\bar\mu_2)
&=&0\, ,
\\ \hueco\dis
\frac{3\eps}{4}\Sigma_2(\nabla\phi_2^{n+\frac12},\nabla\bar\phi_2)
+\tau_2\Delta t\Big(\nabla(\phi_2^{n+1} - \phi_2^{n}),\nabla\bar\phi_2\Big)
+ \frac{24} \eps\Lambda\Big((\phi_1^{n+1})^2\phi_2^{n+\frac12} (\phi_3^n)^2,\bar\phi_2\Big)
&&
\\ \hueco\dis+  \frac{24}{\eps}\Sigma_2 (\widetilde f(\phi_2^n),\bar\phi_2)
+  \frac{12}{\eps}\Sigma_2 \Big(\widetilde f'(\phi_2^n)(\phi_2^{n+1} - \phi_2^{n}),\bar\phi_2\Big)
\\ \hueco\dis
+  \frac1{{\lambda}}\Big(\phi^n_1 + \phi^n_2 + \phi^n_3 - 1 ,\bar\phi_2\Big)
+ \frac1{2{\lambda}} \Big[2(\phi_1^{n+1} - \phi_1^n,\bar\phi_2)
+ (\phi_2^{n+1} - \phi_2^n,\bar\phi_2) \Big] &=&(\mu_2^{n+1},\bar\phi_2)\, .
\ea\right.
\eeq

\item[\textbf{Step 3}] Compute $(\phi_3^{n+1},\mu_3^{n+1}) \in P_h\times M_h$ such that for all $(\bar\phi_3,\bar\mu_3) \in P_h\times M_h$ we have

\beq\label{eq:CH3_3}
\left\{\ba{rcl}\dis
\left(\frac{\phi_3^{n+1}- \phi_3^n}{\Delta t} ,\bar\mu_3\right) 
+M_3(\nabla\mu_3^{n+1},\nabla\bar\mu_3)
&=&0\, ,
\\ \hueco\dis
\frac{3\eps}{4}\Sigma_3(\nabla\phi_3^{n+\frac12},\nabla\bar\phi_3)
+\tau_3\Delta t\Big(\nabla(\phi_3^{n+1} - \phi_3^{n}),\nabla\bar\phi_3\Big)
+ \frac{24} \eps\Lambda  \Big((\phi_1^{n+1})^2(\phi_2^{n+1})^2\phi_3^{n+\frac12} ,\bar\phi_3\Big)
&&
\\ \hueco\dis
+  \frac{24}{\eps}\Sigma_3 (\widetilde f(\phi_3^n),\bar\phi_3)
+  \frac{12}{\eps}\Sigma_3 \Big(\widetilde f'(\phi_3^n)(\phi_3^{n+1} - \phi_3^{n}),\bar\phi_3\Big)
+  \frac1{{\lambda}} \Big(\phi^n_1 + \phi^n_2 + \phi^n_3 - 1 ,\bar\phi_3\Big)
\\\hueco \dis
+ \frac1{2{\lambda}}\Big[ 2(\phi_1^{n+1} - \phi_1^n,\bar\phi_3)
+ 2(\phi_2^{n+1} - \phi_2^n,\bar\phi_3) 
+ (\phi_3^{n+1} - \phi_3^n,\bar\phi_3) \Big] &=&(\mu_3^{n+1},\bar\phi_3)\, .
\ea\right.
\eeq
\end{itemize}
Here, the parameters $\tau_i$ for $i=1,2,3$ are chosen to introduce enough numerical dissipation to guarantee energy stability for any choice of $\dt$.
\begin{dfntn}
A numerical scheme is called energy stable with respect to $\widetilde{E}(\boldphi)$ if and only if 
\beq\label{eq:decreasingEnergy}
\widetilde{E}(\phi^{n+1}) \leq \widetilde{E}(\phi^n) \, ,
\eeq
for all $n\geq 0$. \\
If \eqref{eq:decreasingEnergy} is satisfied for all $\dt > 0$, then the scheme is called unconditionally energy stable.
\end{dfntn}
\begin{prpstn}\label{prop:energyStable}
Scheme \eqref{eq:CH3_1}-\eqref{eq:CH3_3} is conservative, linear and decoupled. Moreover, taking the stabilization terms $\tau_1\geq \frac{72M_1}{\eps^2}\Sigma^2_1$, $\tau_2\geq \frac{72M_2}{\eps^2}\Sigma^2_2$ and $\tau_3\geq \frac{72M_3}{\eps^2}\Sigma^2_3$, the scheme satisfies a discrete version of the energy law \eqref{eq:CH3modEnergyLaw}:
\beq\label{eq:discreteenergylaw}
\delta_t \widetilde{E}(\phi_1^{n+1},\phi_2^{n+1},\phi_3^{n+1})
+  \frac{M_1}{2}\|\nabla\mu_1^{n+1}\|^2_{\xLtwo}
+  \frac{M_2}{2}\|\nabla\mu_2^{n+1}\|^2_{\xLtwo}
+  \frac{M_3}{2}\|\nabla\mu_3^{n+1}\|^2_{\xLtwo}
\,\leq\,
0\,.
\eeq
In particular, this implies unconditional energy stability with respect to $\widetilde{E}(\boldphi)$.
\end{prpstn}

\begin{proof}
Testing \eqref{eq:CH3_1} by $(\bar\mu_1,\bar\phi_1)=\left(\mu_1^{n+1},\frac1{\Delta t}(\phi_1^{n+1} - \phi_1^n)\right)$ we obtain
\beq\label{eq:proofenergy1}
\ba{rl}\dis
\delta_t\left(
\frac{3\eps}{4}\Sigma_1 G(\phi_1^{n+1})
+ \frac{24}{\eps}\Sigma_1 \widetilde{F}(\phi_1^{n+1})
\right)
+\dis  \frac{12} \eps\Lambda \int_\Omega \delta_t\big((\phi_1^{n+1})^2\big)(\phi_2^n)^2(\phi_3^n)^2 d\x
+ M_1\|\nabla\mu^{n+1}\|^2_{\xLtwo}
&
\\ \hueco \dis
+\Big(p(\phi_1^n,\phi_2^n,\phi_3^n) ,\delta_t\phi_1^{n+1}\Big)
+ \frac1{2{\lambda}\dt} \| \phi^{n+1}_1 - \phi^n_1 \|^2
+ \tau_1\|\nabla(\phi_1^{n+1} - \phi_1^{n})\|^2_{\xLtwo}
+ \textbf{ND}(\phi_1^{n+1},\phi_1^{n}) 
&\, = \, 0 \, ,
\ea
\eeq
where $\textbf{ND}(\phi^{n+1}_1, \phi^n_1)$ indicates the numerical dissipation introduced by the scheme in \eqref{eq:proofenergy1}:
\beq\label{eq:numericalDissipation}
\textbf{ND}(\phi_1^{n+1},\phi_1^{n})
\,:=\,\dis
\frac1{\Delta t}\frac{24}{\eps}\Sigma_1\int_\Omega\left[
\left(\widetilde f(\phi_1^n) + \frac12\widetilde f'(\phi_1^n)(\phi_1^{n+1} - \phi_1^n)\right)(\phi_1^{n+1} - \phi_1^n)
-
\left(\widetilde F(\phi_1^{n+1}) - \widetilde F(\phi_1^{n})\right)
\right]d\x.
\eeq

Now we use the second order Taylor's expansion of $\widetilde{F}$, and the fact that $\|\widetilde{f}'\|_{L^\infty}=\frac12$ to obtain an upper bound on the numerical dissipation. We can write
$$
\ba{rcl}
\Big| \textbf{ND}(\phi_1^{n+1},\phi_1^{n}) \Big|
&\leq&\dis
\frac1{\Delta t}\frac{24}{\eps} |\Sigma_1 | \int_\Omega
\left|\frac12\widetilde f'(\phi_1^n)(\phi_1^{n+1} - \phi_1^n)^2 
+ \left( \widetilde f(\phi_1^n)(\phi_1^{n+1} - \phi_1^n) -  \left( \widetilde F(\phi_1^{n+1}) - \widetilde F(\phi_1^{n}) \right) \right)
\right|
d\x
\\ \hueco
&=&\dis
\frac1{\Delta t}\frac{24}{\eps}| \Sigma_1 | \int_\Omega
\left|\frac12\widetilde f'(\phi_1^n)(\phi_1^{n+1} - \phi_1^n)^2 - \frac12\widetilde f'(\phi_1^\xi)(\phi_1^{n+1} - \phi_1^n)^2
\right|
d\x
\\ \hueco
&=&\dis
\frac1{\Delta t}\frac{12}{\eps}| \Sigma_1 |\int_\Omega
\left|\widetilde f'(\phi_1^n) - \widetilde f'(\phi_1^\xi)\right|\left|\phi_1^{n+1} - \phi_1^n\right|^2 
d\x
\\ \hueco
&\leq&\dis
\frac1{\Delta t}\frac{12}{\eps}| \Sigma_1 |\|\phi_1^{n+1} - \phi_1^n\|_{\xLtwo}^2 \,.
\ea
$$
Testing \eqref{eq:CH3_1}$_1$ by $\bar\mu_1=(\phi_1^{n+1} - \phi_1^n)$ we obtain 
$$
\frac1{\Delta t}\|\phi_1^{n+1}- \phi_1^n\|^2_{\xLtwo}
\,=\,
-M_1 \int_\Omega \nabla\mu_1^{n+1}\cdot\nabla(\phi_1^{n+1} - \phi_1^n) d\x \, ,
$$
then by using Young's inequality
$$
\frac1{\Delta t}\frac{12}{\eps}| \Sigma_1 |\|\phi_1^{n+1}- \phi_1^n\|^2_{\xLtwo}
\leq\dis
\frac{M_1}2 \|\nabla\mu_1^{n+1}\|^2_{\xLtwo} 
+ \frac{72M_1}{\eps^2}\Sigma_1^2\|\nabla(\phi_1^{n+1} - \phi_1^n)\|^2_{\xLtwo}\,. 
$$
Thus equation \eqref{eq:proofenergy1} leads to
$$
\ba{rl}\dis
\delta_t\left(
\frac{3\eps}{4}\Sigma_1 G(\phi_1^{n+1})
+ \frac{24}{\eps}\Sigma_1 F(\phi_1^{n+1})
\right)
+\dis   \frac{12} \eps\Lambda \int_\Omega \delta_t\big((\phi_1^{n+1})^2\big)(\phi_2^n)^2(\phi_3^n)^2 d\x
+ \frac{M_1}2\|\nabla\mu_1^{n+1}\|^2_{\xLtwo}
&
\\ \hueco
\dis
+ \Big(p(\phi_1^n,\phi_2^n,\phi_3^n) ,\delta_t\phi_1^{n+1}\Big)
+ \frac1{2{\lambda}\Delta t}\|\phi_1^{n+1} - \phi_1^n\|^2_{\xLtwo}
+ \left(\tau_1 - \frac{72M_1}{\eps^2}\Sigma_1^2\right)
\|\nabla(\phi_1^{n+1} - \phi_1^{n})\|^2_{\xLtwo}
&\leq 0\,.
\ea
$$
Therefore, taking $\tau_1\geq \frac{72M_1}{\eps^2}\Sigma_1^2$, we obtain 
\beq\label{eq:profener1}
\ba{rl}\dis
\delta_t\left(
\frac{3\eps}{4}\Sigma_1 G(\phi_1^{n+1})
+ \frac{24}{\eps}\Sigma_1 F(\phi_1^{n+1})
\right)
+ \frac{M_1}2\|\nabla\mu_1^{n+1}\|^2_{\xLtwo}
&
\\ \hueco\dis
+ \Big(p(\phi_1^n,\phi_2^n,\phi_3^n) ,\delta_t\phi_1^{n+1}\Big)
+ \frac1{2{\lambda}\Delta t}\|\phi_1^{n+1} - \phi_1^n\|^2_{\xLtwo}
+\dis  \frac{12} \eps\Lambda  \int_\Omega \delta_t\big((\phi_1^{n+1})^2\big)(\phi_2^n)^2(\phi_3^n)^2 d\x
&\leq 0\,.
\ea
\eeq
Testing system \eqref{eq:CH3_2} by $(\bar\mu_2,\bar\phi_2)=\left(\mu_2^{n+1},\frac1{\Delta t}(\phi_2^{n+1} - \phi_2^n)\right)$, considering $\tau_2\geq \frac{72M_2}{\eps^2}\Sigma_2^2$ and working in a similar way we obtain
\beq\label{eq:profener2}
\ba{rl}\dis
\delta_t\left(
\frac{3\eps}{4}\Sigma_2 G(\phi_2^{n+1})
+ \frac{24}{\eps}\Sigma_2 F(\phi_2^{n+1})
\right)
+ \frac{M_2}2\|\nabla\mu_2^{n+1}\|^2_{\xLtwo}
&
\\ \hueco\dis
+ \Big(p(\phi_1^n,\phi_2^n,\phi_3^n) ,\delta_t\phi_2^{n+1}\Big)
+ \frac2{{\lambda}\Delta t}(\phi_1^{n+1} - \phi_1^n,\phi_2^{n+1} - \phi_2^n)
+ \frac1{{\lambda}\Delta t}\|\phi_2^{n+1} - \phi_2^n\|^2_{\xLtwo}
&
\\ \hueco\dis
+ \dis   \frac{12} \eps\Lambda \int_\Omega (\phi_1^{n+1})^2\delta_t\big((\phi_2^{n+1})^2\big)(\phi_3^n)^2 d\x
&\leq 0\, .
\ea
\eeq
Moreover, considering  $\tau_3\geq \frac{72M_3}{\eps^2}\Sigma_3^2$  and testing \eqref{eq:CH3_3} by $(\bar\mu_3,\bar\phi_3)=\left(\mu_3^{n+1},\frac1{\Delta t}(\phi_3^{n+1} - \phi_3^n)\right)$ we can deduce that
\beq\label{eq:profener3}
\ba{rl}\dis
\delta_t\left(
\frac{3\eps}{4}\Sigma_3 G(\phi_3^{n+1})
+ \frac{24}{\eps}\Sigma_3 F(\phi_3^{n+1})
\right)
+ \frac{M_3}2\|\nabla\mu_3^{n+1}\|^2_{\xLtwo}
&
\\ \hueco\dis
+  \Big(p(\phi_1^n,\phi_2^n,\phi_3^n) ,\delta_t\phi_3^{n+1}\Big)
+ \frac2{{\lambda} \Delta t}(\phi_1^{n+1} - \phi_1^n,\phi_3^{n+1} - \phi_3^n)
+ \frac2{{\lambda} \Delta t}(\phi_2^{n+1} - \phi_2^n,\phi_3^{n+1} - \phi_3^n) 
&
\\ \hueco\dis
+ \frac1{{\lambda} \Delta t}\|\phi_3^{n+1} - \phi_3^n\|^2_{\xLtwo}
+\dis   \frac{12} \eps\Lambda \int_\Omega (\phi_1^{n+1})^2 (\phi_2^{n+1})^2 \delta_t\big((\phi_3^{n+1})^2\big)d\x
&\leq0\,.
\ea
\eeq
Using the relation 
\beq\label{eq:equation1}
\delta_t(a^{n+1}) b^n c^n
+ a^{n+1} \delta_t (b^{n+1}) c^n
+  a^{n+1} b^{n+1} \delta_t(c^{n+1}) 
\,=\,
\delta_t(a^{n+1}b^{n+1}c^{n+1})\,,
\eeq
we can deduce
\beq\label{eq:relforF123}
\ba{r}\dis
\dis   \frac{12} \eps\Lambda \int_\Omega \Big(
\delta_t\big((\phi_1^{n+1})^2\big)(\phi_2^n)^2(\phi_3^n)^2 
+
(\phi_1^{n+1})^2\delta_t\big((\phi_2^{n+1})^2\big)(\phi_3^n)^2
+
(\phi_1^{n+1})^2 (\phi_2^{n+1})^2\delta_t\big((\phi_3^{n+1})^2\big)
\Big)d\x
\\ \hueco\dis
\,=\,
 \frac{24} \eps\Lambda 
 \int_\Omega\delta_t F_{123}(\phi_1^{n+1},\phi_2^{n+1},\phi_3^{n+1}) d\x\, .
\ea
\eeq
Moreover, using a second order Taylor expansion of $P(\boldphi^{n+1})$ we obtain
$$
P(\boldphi^{n+1})
\,=\,
P(\boldphi^n)
\,+\,
\nabla_{\boldphi}P(\boldphi^{n})(\boldphi^{n+1} - \boldphi^n)
\,+\,
\frac12 (\boldphi^{n+1} - \boldphi^n)^T H_P({\boldphi}^\xi)
(\boldphi^{n+1} - \boldphi^n)\,,
$$
Then, using that
$$
\ba{rcl}
\dis\int_\Omega
(P(\boldphi^{n+1})
\,-\,
P(\boldphi^n))
d\x
&=&\dis
\int_\Omega
\nabla_{\boldphi}P(\boldphi^{n})(\boldphi^{n+1} - \boldphi^n) d\x\,+\,
\frac12\int_\Omega (\boldphi^{n+1} - \boldphi^n)^T H_P({\boldphi}^\xi)(\boldphi^{n+1} - \boldphi^n) d\x
\ea
$$ 
by expanding products we can write
\beq\label{eq:gradPpart}
\Big(p(\phi_1^n,\phi_2^n,\phi_3^n) ,\delta_t\phi_1^{n+1}\Big)
+ \Big(p(\phi_1^n,\phi_2^n,\phi_3^n) ,\delta_t\phi_2^{n+1}\Big)
+ \Big(p(\phi_1^n,\phi_2^n,\phi_3^n) ,\delta_t\phi_3^{n+1}\Big)
\,=\,
\frac1{\Delta t}\int_\Omega
\nabla_{\boldphi} P(\boldphi^{n})(\boldphi^{n+1} - \boldphi^n)d\x
\eeq
and 
\beq\label{eq:hessPpart}
\ba{l} \dis
\frac1{\Delta t}\int_\Omega
(\boldphi^{n+1} - \boldphi^n)^T H_P({\boldphi}^\xi)(\boldphi^{n+1} - \boldphi^n)
d\x \\ \hueco \dis
\quad \quad \,=\,\dis
\frac1{\eps^2\Delta t}\|\phi_1^{n+1} - \phi_1^n\|^2_{\xLtwo}
\\ \hueco
\quad \quad \, + \, \dis\frac2{{\lambda} \dt}(\phi_1^{n+1} - \phi_1^n,\phi_2^{n+1} - \phi_2^n)
+ \frac1{{\lambda} \dt}\|\phi_2^{n+1} - \phi_2^n\|^2_{\xLtwo}
\\ \hueco\dis
\quad \quad \, + \,  \dis\frac2{{\lambda} \Delta t}(\phi_1^{n+1} - \phi_1^n,\phi_3^{n+1} - \phi_3^n)
+ \frac2{{\lambda} \dt}(\phi_2^{n+1} - \phi_2^n,\phi_3^{n+1} - \phi_3^n) 
+ \frac1{{\lambda} \dt}\|\phi_3^{n+1} - \phi_3^n\|^2_{\xLtwo} \, .
\ea
\eeq
Therefore, by adding relations \eqref{eq:profener1}-\eqref{eq:profener3} and taking into account equations \eqref{eq:relforF123}, \eqref{eq:gradPpart} and \eqref{eq:hessPpart}
we obtain the desired relation \eqref{eq:discreteenergylaw}.
\end{proof}
\begin{bsrvtn}\label{obs:TD1ND}
We can show {(at least heuristically)} that the numerical dissipation introduced by the scheme is second order in time. Indeed, using the first and second order Taylor expansions we can write
$$
\ba{rcl}
\widetilde{F}(\phi^{n+1})
&=&\dis
\widetilde{F}(\phi^{n})
+ \widetilde{f}(\phi^n)(\phi^{n+1} - \phi^n)
+ \frac12\widetilde{f}'(\phi^n)(\phi^{n+1} - \phi^n)^2
+ \frac16\widetilde{f}''(\phi^\chi)(\phi^{n+1} - \phi^n)^3\,,
\\ \hueco
\widetilde{F}(\phi^{n+1})
&=&\dis
\widetilde{F}(\phi^{n})
+ \widetilde{f}(\phi^n)(\phi^{n+1} - \phi^n)
+ \frac12\widetilde{f}'(\phi^\xi)(\phi^{n+1} - \phi^n)^2 \, .
\ea
$$
By taking the difference and dividing by $\dt$ we obtain
$$
\frac1{2\Delta t}\big(\widetilde{f}'(\phi^n) - \widetilde{f}'(\phi^\xi)\big)(\phi^{n+1} - \phi^n)^2
\,=\,
-\frac1{6\Delta t} \widetilde{f}''(\phi^\chi)(\phi^{n+1} - \phi^n)^3
\,=\,
-\frac{(\Delta t)^2}6 \widetilde{f}''(\phi^\chi)(\delta_t(\phi^{n+1}))^3
\,\sim\,\mathcal{O}\big((\Delta t)^2\big)\,.
$$
Moreover, 
$$
\tau\|\nabla(\phi^{n+1} - \phi^{n})\|^2_{\xLtwo}
\,=\,
\tau(\Delta t)^2\|\nabla\delta_t(\phi^{n+1})\|^2_{\xLtwo}
\,\sim\,
\mathcal{O}\big((\Delta t)^2\big)\,.
$$
Then, the total numerical dissipation introduced by the system satisfy
$$
\mathbf{TND}^{n+1}
\,=\,
  \sum_{j=1}^{3} \mathbf{ND}(\phi_j^{n+1},\phi_j^{n})
+   \sum_{j=1}^{3} \tau_j\|\nabla(\phi_j^{n+1} - \phi_j^{n})\|^2_{\xLtwo}  
\,\sim\,\mathcal{O}((\Delta t)^2)\,.
$$
\end{bsrvtn}

\begin{bsrvtn}\label{obs:td1Expensive}
The unconditional energy stability of this scheme relies on replacing the function ${{\mathrm{F}}}(\boldphi)$ with a truncated version ${\widetilde{\mathrm{F}}}(\boldphi)$, which is computationally expensive {due to the fact that we need to check the value of ${\widetilde{{f}}}$ in each node of the mesh, in each time step}.
\end{bsrvtn}

\begin{prpstn}[Partial Spreading] \label{prop:solvablePartialTD1}
Suppose $\Sigma_i \geq 0$, $i=1,2,3$, and $\tau_i \geq \frac{72 M_i}{\eps^2}\Sigma_i^2$. Then scheme TD1 is uniquely solvable for all $\dt > 0$.
\end{prpstn}
\begin{proof}
We will show the uniqueness of the solution since this implies existence. 
Given $\boldphi^n$, we consider two solutions of \eqref{eq:CH3_1},  $(\phi_1, \mu_1)$ and $(\phi_1^\star, \mu_1^\star)$, and we define $(\widehat{\phi}_1 , \widehat{\mu}_1) = (\phi_1 - \phi_1^\star, \mu_1 - \mu_1^\star)$ which satisfies
\beq\label{eq:hatSystem1}
	\left\{
	\ba{rcl} \dis
		\frac{1}{\dt} (\widehat{\phi}_1 , \overline{\mu}_1 )
		+ M_1 (\nabla \widehat\mu_1 , \nabla \overline{\mu}_1 ) 
		&=& 0 \, , \\ \hueco \dis
		\left( \frac{12}{\eps}\Sigma_1 \widetilde{f'}(\phi^n_1)  \widehat{\phi}_1 , \overline{\phi}_1 \right)
		+ \left( \frac{3\eps}{4}\Sigma_1 + \tau_1 \dt \right) (\nabla \widehat{\phi}_1 , \nabla \overline{\phi}_1 ) 
		& & \\ \hueco \dis
		+ \frac{1}{2{\lambda}} \left( \widehat{\phi}_1 , \overline{\phi}_1 \right)
		+ \frac{12}{\eps} \Lambda \left( \widehat{\phi}_1 (\phi_2^n)^2 (\phi_3^n)^2 , \overline{\phi}_1\right)
		&=& \dis
		(\widehat{\mu}_1 , \overline{\phi}_1) \,,
	\ea
\right.
\eeq
for all $(\overline{\mu}_1,\overline{\phi}_1)\in M_h\times P_h$.
First, testing \eqref{eq:hatSystem1} by $(\overline{\mu}_1,\overline{\phi}_1) = (\dt\widehat{\mu}_1 , \widehat{\phi}_1)$
gives
\beq\label{eq:td1SolvabilityEq1}
\ba{rcl}\dis
\frac{12}{\eps}\Sigma_1  \int_\Om \widetilde{f'}(\phi^n_1) |\widehat{\phi}_1|^2d\x
+ \left( \frac{3\eps}{4}\Sigma_1 + \tau_1 \dt \right) \| \nabla \widehat{\phi}_1 \|^2_{\xLtwo}
+ \dt  M_1\| \nabla \widehat{\mu}_1 \|^2_{\xLtwo} 
& & \\ \hueco \dis
+ \frac{1}{2{\lambda}} \| \widehat{\phi}_1 \|^2_{\xLtwo}
+ \frac{12}{\eps} \Lambda \| \widehat{\phi}_1 \phi_2^n \phi_3^n\|^2 _{\xLtwo}
& = & 0\,.
\ea
\eeq
Now testing \eqref{eq:hatSystem1} by $\overline{\mu}_1 = \dt \widehat\phi_1$ to obtain
\beq\label{eq:hatEq}
\| \widehat{\phi}_1 \|^2_{\xLtwo} 
+ M_1 \dt \left( \nabla \widehat{\mu}_1 , \nabla \widehat{\phi}_1 \right) = 0\, .
\eeq
Multiplying by $( 12 / \eps ) |\Sigma_1| \|\widetilde{f'} (\phi_1^n)\|_{L^\infty}$ and using Young's inequality we get
\beq \label{eq:td1SolvabilityInequality1} 
	\dis \frac{12}{\eps} |\Sigma_1| \|\widetilde{f'} (\phi_1^n)\|_{L^\infty} \| \widehat{\phi}_1 \|^2_{\xLtwo} 
	\, \leq \,
	 \frac{M_1 \dt }{ 2 } \| \nabla \widehat{\mu}_1 \|^2 _{\xLtwo}
	 + \frac{M_1 \dt }{ 2 } \left( \frac{12 \Sigma_1}{\eps} \right)^2 
	\| \widetilde{f'} (\phi_1^n) \|_{L^\infty}^2 \| \nabla \widehat{\phi}_1 \|^2_{\xLtwo} \, .
\eeq 
Looking at \eqref{eq:td1SolvabilityEq1} and applying \eqref{eq:td1SolvabilityInequality1} gives
\begin{equation}\label{eq:td1SolvabilityEq21}
	\ba{rcl}\dis 
		\left( 
			\frac{3\eps}{4} \Sigma_1 
			+ \tau_1 \dt - \frac{M_1 \dt}{2} \left( \frac{12}{\eps}\Sigma_1 \right)^2 \| \widetilde{f}' (\phi_1^n) \|_{L^\infty}^2
		\right) \| \nabla \widehat{\phi}_1 \|^2_{\xLtwo}
		+ \frac{M_1 \dt}{2}\| \nabla\widehat{\mu}_1 \|^2 _{\xLtwo}
		& & \\ \hueco \dis
		+ \frac1{2{\lambda}} \| \widehat{\phi}_1 \|^2_{\xLtwo}
		+ \frac{12}{\eps} \Lambda \| \widehat{\phi}_1 \phi_2^n \phi_3^n\|^2_{\xLtwo} 
& \leq & 0 \, .
	\ea
\end{equation}
Now each term in \eqref{eq:td1SolvabilityEq21} clearly have positive coefficients except for the first one. In order to show 
\beq\label{eq:td1SolvabilityEq31}
	\frac{3\eps}{4} \Sigma_1 
	+ \tau_1 \dt - \frac{M_1 \dt}{2} \left( \frac{12}{\eps}\Sigma_1 \right)^2 \| \widetilde{f}' (\phi_1^n) \|_{L^\infty}^2
	\, \geq \, 0 \, ,
\eeq
we use the fact that $\tau_1 \geq 72 M_1 \Sigma_1^2 /\eps^2$ and the bound $\widetilde{f}' (\phi_1^n) \leq 1/2$  to see that
$$
	\tau_1 \dt - \frac{M_1 \dt}{2} \left( \frac{12}{\eps}\Sigma_1 \right)^2 \| \widetilde{f}' (\phi_1^n) \|_{L^\infty}^2
	\, \geq \,
	\frac{36 M_1 \dt \Sigma_1^2}{\eps^2} \geq 0 \, ,
$$
for any $\dt \geq 0$.  Thus \eqref{eq:td1SolvabilityEq1} gives $\widehat{\phi}_1 = 0$, $\nabla \widehat{\phi}_1=\bold{0}$, and $\nabla \widehat{\mu}_1=\bold{0}$. With these results, and considering $\eqref{eq:hatSystem1}_{2}$ we can also conclude that $\widehat{\mu}_1=0$. Therefore the solution to problem \eqref{eq:CH3_1} is unique and we denote it by $(\phi_1^{n+1},\mu^{n+1}_1)$. Then, given $\boldphi^n$ and $\phi_1^{n+1}$ we consider two solutions of \eqref{eq:CH3_2},  $(\phi_2, \mu_2)$ and $(\phi_2^\star, \mu_2^\star)$, and we define $(\widehat{\phi}_2, \widehat{\mu}_2) = (\phi_2 - \phi_2^\star, \mu_2 - \mu_2^\star)$ which satisfy
\beq\label{eq:hatSystem2}
	\left\{
	\ba{rcl} \dis
\dis		\frac{1}{\dt} (\widehat{\phi}_2 , \overline{\mu}_2 )
		+ M_2 (\nabla \widehat\mu_2 , \nabla \overline{\mu}_2 ) 
		&=& 0 \, , \\ \hueco \dis
		\left( \frac{12}{\eps}\Sigma_2 \widetilde{f'}(\phi^n_2)  \widehat{\phi}_2 , \overline{\phi}_2 \right)
		+ \left( \frac{3\eps}{4}\Sigma_2 + \tau_2 \dt \right) (\nabla \widehat{\phi}_2 , \nabla \overline{\phi}_2 ) 
		& & \\ \hueco \dis
		+ \frac{1}{2{\lambda}} \left(\widehat{\phi}_2 , \overline{\phi}_2 \right)
		+ \frac{12}{\eps} \Lambda \left( ({\phi}^{n+1}_1)^2 \widehat{\phi}_2 (\phi_3^n)^2 , \overline{\phi}_2\right)
		&=& \dis
		(\widehat{\mu}_2 , \overline{\phi}_2) \,,
	\ea
\right.
\eeq
for all $(\overline{\mu}_2,\overline{\phi}_2)\in M_h\times P_h$. Now testing \eqref{eq:hatSystem2} by $(\overline{\mu}_2,\overline{\phi}_2) = (\dt\widehat{\mu}_2 , \widehat{\phi}_2)$ and following the same arguments as before, we can easily conclude that the solution to problem \eqref{eq:CH3_2} is unique and we denote it by $(\phi_2^{n+1},\mu^{n+1}_2)$. Finally, given $\boldphi^n$, $\phi_1^{n+1}$ and $\phi_2^{n+1}$ we consider two solutions of \eqref{eq:CH3_3},  $(\phi_3, \mu_3)$ and $(\phi_3^\star, \mu_3^\star)$, and we define $(\widehat{\phi}_3, \widehat{\mu}_3) = (\phi_3 - \phi_3^\star, \mu_3 - \mu_3^\star)$ which satisfy
\beq\label{eq:hatSystem3}
	\left\{
	\ba{rcl} \dis
		\dis
		\frac{1}{\dt} (\widehat{\phi}_3 , \overline{\mu}_3 )
		+ M_3 (\nabla \widehat\mu_3 , \nabla \overline{\mu}_3 ) 
		&=& 0 \, , \\ \hueco \dis
		\left( \frac{12}{\eps}\Sigma_3 \widetilde{f'}(\phi^n_3)  \widehat{\phi}_3 , \overline{\phi}_3 \right)
		+ \left( \frac{3\eps}{4}\Sigma_3 + \tau_3 \dt \right) (\nabla \widehat{\phi}_3 , \nabla \overline{\phi}_3 ) 
		& & \\ \hueco \dis
		+ \frac{1}{2{\lambda}} \left( \widehat{\phi}_3, \overline{\phi}_3 \right)
		+ \frac{12}{\eps} \Lambda \left( ({\phi}_1^{n+1})^2 ({\phi}^{n+1}_2)^2 \widehat{\phi}_3 , \overline{\phi}_3\right)
		&=& \dis
		(\widehat{\mu}_3 , \overline{\phi}_3) \, ,
	\ea
\right.
\eeq
for all $(\overline{\mu}_3,\overline{\phi}_3)\in M_h\times P_h$. Now testing \eqref{eq:hatSystem3} by $(\overline{\mu}_3,\overline{\phi}_3) = (\dt\widehat{\mu}_3 , \widehat{\phi}_3)$ and following the same arguments as before, we conclude that the solution to problem \eqref{eq:CH3_3} is unique. Therefore the solution to scheme TD1 is unique.

\end{proof}

\begin{prpstn}[Total Spreading] \label{prop:solvableTotalTD1}
Suppose $\Sigma_1 < 0$, and $\Sigma_2,\Sigma_3 \geq 0$ (reorder otherwise). Let $\tau_i$ be chosen as in \eqref{prop:energyStable}, and $h>0$ being the spatial mesh size. Then scheme TD1 is uniquely solvable {if any of the two following assumptions holds:}
\begin{itemize}
\item[(a)]  {Fixed  $h$, $\lambda$ and $\varepsilon$ with $\eps < \frac{4}{3|\Sigma_1|}$, then scheme TD1 is uniquely solvable }for all $\dt$ satisfying
\beq\label{eq:solvableTotalTD1}
\dt \leq C\frac{{h^4}(4 - 3\eps|\Sigma_1|)}{M_1}\,,
\eeq
with $C>0$ being a constant independent of the physical parameters, $\Delta t$ and $h$.
\item[(b)]  {Fixed  $h$, $\Delta t$ and $\varepsilon$, then scheme TD1 is uniquely solvable for all $\lambda$ satisfying
\beq\label{eq:solvableTotalTD1lambda}
\lambda
\,\leq\,
 C\frac{h^2} { \eps | \Sigma_1 |}\,.
\eeq
with $C>0$ being a constant independent of the physical parameters, $\Delta t$ and $h$.
}

\end{itemize}
\end{prpstn}
\begin{proof}
(a)
Looking at \eqref{eq:td1SolvabilityEq21} we see that the case of $\Sigma_1 < 0$ introduces another potentially negative term into this equation. Therefore, we begin this proof by considering \eqref{eq:td1SolvabilityEq1} in the previous proof for Proposition~\ref{prop:solvablePartialTD1}. First, we write the following inverse inequality \cite{Brenner,ErnGuermond} for $\widehat{\phi}_1$:
\beq\label{eq:inverseineq}
\| \nabla \widehat{\phi}_1 \|^2_{\xLtwo} \leq \frac{c}{{h^2}} \| \widehat{\phi}_1 \|^2_{\xLtwo} \, ,
\eeq
for some $c>0$ which depends on the Finite Element space $P_h$. Now multiply \eqref{eq:hatEq} by $ch^{-1}$ and use Young's Inequality to get
$$
\| \nabla \widehat{\phi}_1 \|^2_{\xLtwo} \leq \frac{c}{{h^2}} \| \widehat{\phi}_1 \|^2_{\xLtwo} 
\leq 
\frac{M_1 \dt c}{2\alpha {h^2}}\| \nabla \widehat{\mu}_1 \|^2_{\xLtwo}
+ \frac{M_1 \dt c \alpha}{2{h^2}} \| \nabla \widehat{\phi}_1 \|^2_{\xLtwo} \, ,
$$
hence
\beq\label{eq:solvableTD1TotalEq}
\left(1 -  \frac{M_1 \dt c \alpha}{2{h^2}} \right) \| \nabla \widehat{\phi}_1 \|^2_{\xLtwo}
\leq
\frac{M_1 \dt c}{2\alpha {h^2}}\| \nabla \widehat{\mu}_1 \|^2_{\xLtwo} \, ,
\eeq
where $\alpha$ is chosen so that
$$
1 - \frac{M_1 \dt \alpha c}{2 {h^2}} 
\,=\,
\frac{3 \eps | \Sigma_1 |}{4} \, ,
$$
which implies
$$
\alpha = \frac{(4 - 3\eps | \Sigma_1 |) {h^2}}{2 M_1 \dt c} \, .
$$
Now $\alpha > 0$ when $\eps < \frac{4}{3 | \Sigma_1|}$, and in this case we can see that
$$
1 - \frac{M_1 \dt \alpha c}{2 {h^2}} \geq 0 
\quad\quad\mbox{ if, and only if } \quad\quad 
\dt \leq \frac{2{h^2}}{M_1 \alpha c} \, .
$$
Thus by applying \eqref{eq:solvableTD1TotalEq} to \eqref{eq:td1SolvabilityEq21} gives
$$
\ba{rcl}\dis 
\left( 
\tau_1 \dt - \frac{M_1 \dt}{2} \left( \frac{12}{\eps}\Sigma_1 \right)^2 \| \widetilde{f}' (\phi^n) \|_{L^\infty}^2
\right) \| \nabla \widehat{\phi}_1 \|^2_{\xLtwo}
+ \left(\frac{M_1\dt}2  - \frac{M_1 \dt c}{2 \alpha {h^2}} \right)\| \nabla\widehat{\mu}_1 \|^2_{\xLtwo} 
& & \\ \hueco \dis
+ \frac1{2{\lambda}} \| \widehat{\phi}_1 \|^2_{\xLtwo}
+ \frac{12}{\eps} \Lambda \| \widehat{\phi}_1 \phi_2^n \phi_3^n\|^2_{\xLtwo} 
 & \leq & 0 \, .
\ea
$$
The coefficient for $\| \nabla \widehat{\phi}_1 \|^2_{\xLtwo}$ is positive by the same argument in Proposition~\ref{prop:solvablePartialTD1}. Hence, to be sure that the coefficient of $\|\nabla\widehat\mu_1\|_{\xLtwo}^2$ is non-negative we need that 
$$
\frac{M_1 \dt}2\left(1-\frac{c}{\alpha {h^2}}\right)\geq 0 \, ,
$$
which is satisfied if, and only if
$$
\dt \leq \frac{(4 - 3\eps|\Sigma_1|){h^4}}{2M_1 c^2}\, .
$$
Therefore, by setting $C = \frac1{2c^2}$ we obtain \eqref{eq:solvableTotalTD1}.

\noindent
{(b) From relation \eqref{eq:inverseineq} we can deduce
$$
\frac{3 \eps | \Sigma_1 |}{4} \| \nabla \widehat{\phi}_1 \|^2_{\xLtwo} 
\leq 
 \| \widehat{\phi}_1 \|^2_{\xLtwo} \, ,
$$
and using this relation in \eqref{eq:td1SolvabilityEq21} gives
$$
\ba{rcl}\dis 
\left( 
\tau_1 \dt - \frac{M_1 \dt}{2} \left( \frac{12}{\eps}\Sigma_1 \right)^2 \| \widetilde{f}' (\phi_1^n) \|_{L^\infty}^2 \right) \| \nabla \widehat{\phi}_1 \|^2_{\xLtwo}
+ \frac{M_1 \dt}{2}\| \nabla\widehat{\mu}_1 \|^2 _{\xLtwo}
& & 
\\ \hueco \dis
+ \left(\frac1{2\lambda} - \frac{3 \eps | \Sigma_1 |c}{4h^2}\right)  \| \widehat{\phi}_1 \|^2_{\xLtwo}
+ \frac{12}{\eps} \Lambda \| \widehat{\phi}_1 \phi_2^n \phi_3^n\|^2_{\xLtwo} 
& \leq & 0 \, .
\ea
$$
Finally, to be sure that the coefficient of $\|\widehat{\phi}_1\|_{\xLtwo}^2$ is non-negative we need that
$$
\frac1{2\lambda}
\, \geq\,
\frac{3 \eps | \Sigma_1 |c}{4h^2} 
\quad\quad
\Longleftrightarrow
\quad\quad
\lambda
\,\leq\,
\frac{2h^2} {3 \eps | \Sigma_1 |c}
\,.
$$ 
}
\end{proof}
{
\begin{bsrvtn}
The constraint presented in \eqref{eq:solvableTotalTD1} is quite restrictive on the size of the time step but the one in \eqref{eq:solvableTotalTD1lambda} is much softer. In practice, we assumed always values of $\lambda$ such that \eqref{eq:solvableTotalTD1lambda} holds, as can be observed in the results presented in Section~\ref{sec:simulations}.
\end{bsrvtn}}
\begin{bsrvtn}
We can extend this scheme to develop an unconditionally energy stable linear decoupled scheme for the $N$-components case presented in Section~\ref{sec:nComponentModel} by considering an extension of \eqref{eq:equation1}:
\beq\label{eq:equation1Extension}
\delta_t(\phi_1^{n+1}\phi_2^{n+1}\dots\phi_N^{n+1}) 
\,=\,
\sum_{i=1}^N \delta_t (\phi_i^{n+1}) \prod_{j< i}\phi_j^{n+1} \prod_{k>i} \phi_k^n \, .
\eeq
\end{bsrvtn}

\subsubsection{{Extension of the scheme for the Ternary Navier-Stokes-Cahn-Hilliard Model}}\label{sec:extensiontoNSCH3}
{
In this section we show how it is possible to extend the ideas in scheme \eqref{eq:CH3_1}-\eqref{eq:CH3_3} to approximate system \eqref{eq:NSCH3} in a linear and decoupling way following the arguments introduced in \cite{Minjeaud13}. Then, the first three steps of the scheme are like \eqref{eq:CH3_1}-\eqref{eq:CH3_3} but with the difference that we modify the first equation of each of those steps with
\beq
\left(\frac{\phi_i^{n+1}- \phi_i^n}{\Delta t} ,\bar\mu_i\right) 
- (\u_i^*\phi_i^n,\nabla\bar\mu_i)
+M_i(\nabla\mu_i^{n+1},\nabla\bar\mu_i)
\,=\,0
\quad
\mbox{ with }
\quad
\u_i^*=\u^n - \phi_i^n\nabla\mu_i
\quad
\mbox{ for }
\quad
i=1,2,3\,.
\eeq
Then for the fourth step we need to consider a pair of Finite Element spaces $\U_h\times \Pi_h$ such that they satisfy a discrete version of the inf-sup condition \cite{GiraultRaviart}. Then the step reads:  Compute $(\u^{n+1},p^{n+1}) \in \U_h\times \Pi_h$  such that for all $(\bar\u,\bar p) \in \U_h\times \Pi_h$ we have
\beq\label{eq:Fluidpartschemes}
\left\{\ba{rcl}\dis
\left(\frac{\u^{n+1} - \sum_{i=1}^3\u_i^*}{\Delta t},\bar \u\right)
+c(\u^n,\u^{n+1},\bar{\u})
+ \big(\nu(\phi_1^n,\phi^n_2,\phi^n_3)\nabla\u^{n+1},\nabla\bar\u\big)
+ (\nabla p^{n+1},\bar\u)
&=&0\,,
\\ \hueco\dis
(\nabla\cdot\u^{n+1},\bar p)
&=&0\,,
\ea\right.
\eeq
where 
$$
c(\u,\v,\w)=\Big((\u \cdot\nabla)\v,\bar \w\Big)
+ \frac12\Big((\nabla\cdot \u)\v,\bar\w\Big)\,.
$$
\begin{prpstn}\label{prop:energyStableNSCH3}
Scheme presented in this section is conservative, linear and decoupled. Moreover, taking the stabilization terms $\tau_1\geq \frac{72M_1}{\eps^2}\Sigma^2_1$, $\tau_2\geq \frac{72M_2}{\eps^2}\Sigma^2_2$ and $\tau_3\geq \frac{72M_3}{\eps^2}\Sigma^2_3$, the scheme satisfies a discrete version of the energy law \eqref{eq:energylawNSCH3}:
\beq\label{eq:discreteenergylaw}
\ba{rcl}\dis
\delta_t \widetilde{E}_{Tot}(\u^{n+1},\phi_1^{n+1},\phi_2^{n+1},\phi_3^{n+1})
+ \left\|\sqrt{\nu(\phi^n_1,\phi^n_2,\phi^n_3)}\nabla\u^{n+1}\right\|^2_{\xLtwo}&&
\\ \hueco\dis
+  \frac{M_1}{2}\|\nabla\mu_1^{n+1}\|^2_{\xLtwo}
+  \frac{M_2}{2}\|\nabla\mu_2^{n+1}\|^2_{\xLtwo}
+  \frac{M_3}{2}\|\nabla\mu_3^{n+1}\|^2_{\xLtwo}
&\leq&
0\,.
\ea
\eeq
In particular, this implies unconditional energy stability with respect to $\widetilde{E}_{Tot}(\u,\boldphi)$.
\end{prpstn}
\begin{proof}
Working as in the proof of Proposition~\ref{prop:energyStable} and taking $(\bar \u,\bar p)=(\u^{n+1},p^{n+1})$ in \eqref{eq:Fluidpartschemes}.
\end{proof}
}
\subsection{A Nontruncated, Decoupled, and First Order Scheme (NTD1)}\label{sec:scheme2}
By truncating the potential functions, $F(\phi_i)$, scheme TD1 was able to guarantee a discrete version of the dissipative energy law which is a desirable property of a numerical scheme. However, this truncation procedure is computationally expensive since it requires checking the value of $\phi^n_i$ in each element for each $i=1,2,3$, and each time step $n=1,\dots,N$. Therefore, to compute solutions with relatively fine spatial mesh, or in three dimensions we pay the price of computing time to get energy stability. 

{
To circumvent this issue we propose a new model called NTD1, where we avoid truncating the double well potentials $F(\phi_i)$.}
When presenting the modified model we stated that the truncation is imposed outside of the physically meaningful range $0\leq \phi_i \leq 1$, since no maximum principle exists for the continuous problem. However, we still expect that the solution to the discrete problem should exist close to the interval $[0,1]$, so the amount of numerical dissipation for both schemes should be similar {and of order $(\Delta t)^2$ (at least heuristically)}.
\\
{
In similar fashion to TD1, scheme NTD1 is solved in three sequential substeps as in \eqref{eq:CH3_1}-\eqref{eq:CH3_3}, just replacing $\widetilde{f}(\phi_i)$ by ${f}(\phi_i)$ and $\widetilde{f}'(\phi_i)$ by ${f}'(\phi_i)$.
}
Scheme NTD1 has the nice properties of TD1 such as being linear, decoupled, but by forgoing the truncation at each discrete time step we significantly improve computational efficiency.

\begin{prpstn}\label{prop:energyStableNTD1}
Scheme NTD1 is conservative, linear and decoupled. Moreover, given $\boldphi^n$, taking the stabilization terms $\tau^n_1\geq \frac{72M_1}{\eps^2}\Sigma^2_1  \| f' (\phi_1^n) \|_{L^\infty}^2$, $\tau^n_2\geq \frac{72M_2}{\eps^2}\Sigma^2_2  \| f' (\phi_2^n) \|_{L^\infty}^2$ and $\tau^n_3\geq \frac{72M_3}{\eps^2}\Sigma^2_3  \| f' (\phi_3^n) \|_{L^\infty}^2$, the scheme satisfies a discrete version of the energy law \eqref{eq:CH3modEnergyLaw}:
\beq\label{eq:ntd1EnergyLaw}
\delta_t E(\phi_1^{n+1},\phi_2^{n+1},\phi_3^{n+1})
+  \frac{M_1}{2}\|\nabla\mu_1^{n+1}\|^2_{\xLtwo}
+  \frac{M_2}{2}\|\nabla\mu_2^{n+1}\|^2_{\xLtwo}
+  \frac{M_3}{2}\|\nabla\mu_3^{n+1}\|^2_{\xLtwo}
\,\leq\,
0\,.
\eeq
\end{prpstn}


\begin{prpstn}[Partial Spreading] \label{prop:solvablePartialNTD1}
Suppose $\Sigma_i \geq 0$, $i=1,2,3$, and $\tau^n_i \geq \frac{72 M_i}{\eps^2}\Sigma_i^2  \| f' (\phi_i^n) \|_{L^\infty}^2$. Then scheme NTD1 is uniquely solvable for all $\dt > 0$.
\end{prpstn}
\begin{proof}
The proof is nearly identical to that of Proposition~\ref{prop:solvablePartialNTD1} except that instead of \eqref{eq:td1SolvabilityEq31} we will obtain 
$$
\frac{3\eps}{4}\Sigma_1
+ \tau^n_1 \dt
- \frac{M_1 \dt}{2} \left( \frac{12}{\eps}\Sigma_1 \right)^2 \| f' (\phi_1^n) \|_{L^\infty}^2
\geq 0 \, .
$$
This inequality is positive given the conditions on $\Sigma_1$ and $\tau^n_1$. Moreover, the same idea can be considered for $\tau^n_2$ and $\tau^n_3$.
\end{proof}

\begin{prpstn}[Total Spreading] \label{prop:solvableTotalNTD1}
Suppose $\Sigma_1 < 0$, and $\Sigma_i \geq 0$, $i=2,3$. Let $\tau^n_i$ be chosen as in Proposition~\ref{prop:solvablePartialNTD1} and $h>0$ being the spatial mesh size. 
Then scheme NTD1 is uniquely solvable {if any of the two following assumptions holds:
\begin{itemize}
\item[(a)]  Fixed  $h$, $\lambda$ and $\varepsilon$ with $\eps < \frac{4}{3|\Sigma_1|}$, then scheme TD1 is uniquely solvable for all $\dt$ satisfying
$$
\dt \leq C\frac{h^4(4 - 3\eps|\Sigma_1|)}{M_1}\,,
$$
with $C>0$ being a constant independent of the physical parameters, $\Delta t$ and $h$.
\item[(b)] Fixed  $h$, $\Delta t$ and $\varepsilon$, then scheme NTD1 is uniquely solvable for all $\lambda$ satisfying
$$
\lambda
\,\leq\,
 C\frac{h^2} { \eps | \Sigma_1 |}\,.
$$
with $C>0$ being a constant independent of the physical parameters, $\Delta t$ and $h$.
\end{itemize}
}
\end{prpstn}
\begin{proof}
The proof follows the same arguments of Proposition~\ref{prop:solvableTotalTD1} using Proposition~\ref{prop:solvablePartialNTD1}.
\end{proof}

\begin{bsrvtn}
{If the stabilization terms $\tau^n_i$ are chosen as in Proposition~\ref{prop:energyStableNTD1}, the numerical dissipation introduced in the system will be of order $\mathcal{O}({(\Delta t)^2}/{\varepsilon^2})$ which seems restrictive from the point of view of parameter $\varepsilon$. From now when we refer to NTD1 as the scheme where apart of no truncating the potentials we neglect the stabilization terms (i.e., $\tau_1=\tau_2=\tau_3=0$) to obtain a scheme that is not provable energy stable because we can't control the sign of the numerical dissipation, but the numerical dissipation will be small, of order $\mathcal{O}({(\Delta t)^2)}$. This type of ideas have proved themselves useful to derive efficient and accurate numerical schemes for the two components Cahn-Hilliard equation \cite{guillentierra2013,Tierra2014}.
}
\end{bsrvtn}


%

\begin{bsrvtn}
Scheme NTD1 can be extended to $N$-components by considering \eqref{eq:equation1Extension}.
\end{bsrvtn}

\subsection{A Nontruncated, Coupled, and Second Order Scheme (NTC2)}\label{sec:schemeNTC2}
In this section, we present a second order accurate in time, conservative, linear numerical scheme. To achieve higher order accuracy, we pay the price of coupling the six equations in the system into one large problem. We are presenting this second order scheme to provide a cost-benefit comparison with our energy stable scheme.  This is done by applying the OD2 approximation to the nonlinear term $\phi_1^2\phi^2_2\phi_3^2$. Since this will couple the unknowns we will forego the lower triangular replacement in the OD2 approximation for $p(\boldphi^{n+1}, \boldphi^n)$. 

Scheme NTC2 is as follows: Given $\left(\phi_1^n , \phi_2^n , \phi_3^n\right) \in P_h^3$ and $\left(\mu_1^n , \mu_2^n , \mu_3^n\right) \in  M_h^3$, find $\left(\phi_1^{n+1} , \phi_2^{n+1} , \phi_3^{n+1}\right) \in P_h^3$ and $\left(\mu_1^{n+1} , \mu_2^{n+1} , \mu_3^{n+1}\right) \in M_h^3$ satisfying
\beq\label{eq:schemeNTC2}
\left\{\ba{rcl}\dis
\left( \pt \phi^{n+1}_1 , \overline{\mu}_1 \right)
+ M_1\left( \nabla \mu^{n+\frac12}_1 , \nabla \overline{\mu}_1 \right)
&=&0\, ,
\\ \hueco\dis
\frac{3\eps}{4}\Sigma_1 \left( \nabla \phi^{n+\frac12}_1 , \nabla \overline{\phi}_1 \right)
+\tau^n_1\Delta t\Big(\nabla(\phi_1^{n+1} - \phi_1^{n}),\nabla\bar\phi_1\Big)&& \\ \hueco \dis
+ \frac{24}\eps \Lambda \left( \phi^n_1(\phi^n_2)^2(\phi^n_3)^2 , \overline{\phi}_1 \right) 
+ \frac{12}\eps \Lambda \left( H_{F_{123}}(\boldphi^n) (\boldphi^{n+1} - \boldphi^n ) ,\overline{\phi}_1 \right)&& \\ \hueco \dis
+ \frac{24}{\eps}\Sigma_1 \left( f(\phi^n_1) , \overline{\phi}_1 \right) 
+ \frac{12}{\eps} \Sigma_1 \left( f'(\phi^n_1)(\phi^{n+1}_1 - \phi^n_1) , \overline{\phi}_1 \right) 
&& \\ \hueco \dis
+ \frac{1}{{\lambda}} \left( \phi^n_1 + \phi^n_2 + \phi^n_3 - 1  , \overline{\phi}_1 \right)
+ \frac{1}{2{\lambda}} \left( (\phi^{n+1}_1 - \phi^n_1) + ( \phi^{n+1}_2 - \phi^n_2 ) + ( \phi^{n+1}_3 - \phi^n_3) , \overline{\phi}_1 \right)
&=& \left( \mu^{n+\frac12}_1 , \overline{\phi}_1 \right) \, ,
\\ \hueco\dis
\left( \pt \phi^{n+1}_2 , \overline{\mu}_2 \right)
+ M_2\left( \nabla \mu^{n+\frac12}_2 , \nabla \overline{\mu}_2 \right)
&=&0\, ,
\\ \hueco\dis
\frac{3\eps}{4}\Sigma_2 \left( \nabla \phi^{n+\frac12}_2 , \nabla \overline{\phi}_2 \right)
+\tau^n_2\Delta t\Big(\nabla(\phi_2^{n+1} - \phi_2^{n}),\nabla\bar\phi_2\Big)&& \\ \hueco \dis
+ \frac{24}\eps \Lambda \left( (\phi^n_1)^2\phi^n_2(\phi^n_3)^2 , \overline{\phi}_2 \right) 
+ \frac{12}\eps \Lambda \left( H_{F_{123}}(\boldphi^n) (\boldphi^{n+1} - \boldphi^n ),\overline{\phi}_2 \right)&& \\ \hueco \dis
+ \frac{24}{\eps}\Sigma_2 \left( f(\phi^n_2) , \overline{\phi}_2 \right) 
+ \frac{12}{\eps} \Sigma_2 \left( f'(\phi^n_2)(\phi^{n+1}_2 - \phi^n_2) , \overline{\phi}_2 \right) && \\ \hueco \dis
+ \frac{1}{{\lambda}} \left( \phi^n_1 + \phi^n_2 + \phi^n_3 - 1  , \overline{\phi}_2 \right)
+ \frac{1}{2{\lambda}} \left( (\phi^{n+1}_1 - \phi^n_1) + ( \phi^{n+1}_2 - \phi^n_2 ) + ( \phi^{n+1}_3 - \phi^n_3) , \overline{\phi}_2 \right)
&=& \left( \mu^{n+\frac12}_2 , \overline{\phi}_2 \right) \, ,
\\ \hueco\dis
\left( \pt \phi^{n+1}_3 , \overline{\mu}_3 \right)
+ M_3\left( \nabla \mu^{n+\frac12}_3 , \nabla \overline{\mu}_3 \right)
&=&0\, ,
\\ \hueco\dis
\frac{3\eps}{4}\Sigma_3 \left( \nabla \phi^{n+\frac12}_3 , \nabla \overline{\phi}_3 \right)
+\tau^n_3\Delta t\Big(\nabla(\phi_3^{n+1} - \phi_3^{n}),\nabla\bar\phi_3\Big)&& \\ \hueco \dis
+ \frac{24}\eps \Lambda \left( (\phi^n_1)^2(\phi^n_2)^2\phi^n_3 , \overline{\phi}_3 \right) 
+ \frac{12}\eps \Lambda \left( H_{F_{123}}(\boldphi^n) (\boldphi^{n+1} - \boldphi^n ),\overline{\phi}_2 \right)&& \\ \hueco \dis
+ \frac{24}{\eps}\Sigma_3 \left( f(\phi^n_3) , \overline{\phi}_3 \right) 
+ \frac{12}{\eps} \Sigma_3 \left( f'(\phi^n_3)(\phi^{n+1}_3 - \phi^n_3) , \overline{\phi}_3 \right) && \\ \hueco \dis
+ \frac{1}{{\lambda}} \left( \phi^n_1 + \phi^n_2 + \phi^n_3 - 1  , \overline{\phi}_3 \right)
+ \frac{1}{2{\lambda}} \left( (\phi^{n+1}_1 - \phi^n_1) + ( \phi^{n+1}_2 - \phi^n_2 ) + ( \phi^{n+1}_3 - \phi^n_3) , \overline{\phi}_3 \right)
&=& \left( \mu^{n+\frac12}_3 , \overline{\phi}_3 \right) \, ,
\ea\right.
\eeq
for all $\left(\overline{\phi}_1, \overline{\mu}_1,\overline{\phi}_2, \overline{\mu}_2,\overline{\phi}_3, \overline{\mu}_3\right) \in (P_h \times M_h)^3$, where 
$$
H_{F_{123}}(\boldphi) := 
\frac{\partial \left(\nabla_{\boldphi} F_{123}\right)}{\partial \boldphi} = 
\begin{bmatrix}
\phi_2^2\phi_3^2 & 2\phi_1 \phi_2\phi_3^2 & 2\phi_1 \phi_2^2\phi_3 \\ \hueco 
2\phi_1 \phi_2\phi_3^2 & \phi_1^2\phi_3^2 & 2\phi_1^2\phi_2\phi_3 \\ \hueco
2\phi_1 \phi_2^2\phi_3 & 2\phi_1^2\phi_2\phi_3 & \phi_2^2\phi_3^2
\end{bmatrix} \, .
$$

\begin{prpstn}\label{prop:energyStableNTC2}
Scheme \eqref{eq:schemeNTC2} is conservative and linear. Moreover, taking the stabilization terms $\tau^n_1\geq \frac{72M_1}{\eps^2}\Sigma^2_1  \| f ' (\phi_1^n) \|_{L^\infty}^2$, $\tau^n_2\geq \frac{72M_2}{\eps^2}\Sigma^2_2  \| f ' (\phi_2^n) \|_{L^\infty}^2$ and $\tau^n_3\geq \frac{72M_3}{\eps^2}\Sigma^2_3  \| f ' (\phi_3^n) \|_{L^\infty}^2$ the scheme satisfies the following discrete version of the energy law \eqref{eq:CH3modEnergyLaw}:
\beq\label{eq:ntc2EnergyLaw}
\delta_t E(\phi_1^{n+1},\phi_2^{n+1},\phi_3^{n+1})
+  M_1\|\nabla\mu_1^{n+\frac12}\|^2_{\xLtwo}
+  M_2\|\nabla\mu_2^{n+\frac12}\|^2_{\xLtwo}
+  M_3\|\nabla\mu_3^{n+\frac12}\|^2_{\xLtwo}
+ \mathbf{TND}^{n+1}
\, = \,
0\, ,
\eeq
where
\beq
\mathbf{TND}^{n+1} = \sum_{i=1}^3 \left( \mathbf{ND}^{n+1}(\phi_i^{n+1} , \phi_i^n)   + \tau_i \| \nabla( \phi_i^{n+1} - \phi_i^n) \|^2_{\xLtwo} \right) + \mathbf{ND}^{n+1}_{F_{123}}(\boldphi^{n+1} , \boldphi^n)
\eeq
and
\beq
\begin{array}{rcl}
\mathbf{ND}^{n+1}_{F_{123}}(\boldphi^{n+1} , \boldphi^n) 
&=& \dis
\int_\Om \left[\left( f_{123}(\boldphi^n) - \frac{1}{2}H_{F_{123}}(\boldphi^n) (\boldphi^{n+1} - \boldphi^n) \right)\cdot(\boldphi^{n+1} - \boldphi^n)  \right.\\ \hueco
& &\dis \quad\quad \left.\phantom{\frac12}- \left( F_{123}(\boldphi^{n+1}) - F_{123}(\boldphi^n) \right) \right]d\x \, .
\end{array}
\eeq
\end{prpstn}
\begin{proof}
Test \eqref{eq:schemeNTC2} by $\left( \overline{\mu}_1, \overline{\phi}_1, \overline{\mu}_2, \overline{\phi}_2, \overline{\mu}_3, \overline{\phi}_3\right) = \left( \mu_1^{n+\frac12} , \delta_t \phi_1^{n+1} , \mu_2^{n+\frac12} , \delta_t \phi_2^{n+1} , \mu_3^{n+\frac12} , \delta_t \phi_3^{n+1}\right)$ to get
\beq
\ba{rl}\dis
\sum_{i=1}^3\delta_t\left(
\frac{3\eps}{4}\Sigma_i G(\phi_i^{n+1})
+ \frac{24}{\eps}\Sigma_i F(\phi_i^{n+1})
+ \frac{24}{\eps}\Lambda F_{123}(\phi_1^n , \phi_2^n , \phi_3^n)
+ \frac{1}{\eps^2} P(\phi_1^n , \phi_2^n , \phi_3^n)
\right)
\\ \hueco \dis
+ \sum_{i=1}^3 M_i\|\nabla\mu_i^{n+\frac12}\|^2_{\xLtwo}
+ \textbf{TND}^{n+1}
&\, = \, 0 \, .
\ea
\eeq
From Observation~\ref{obs:TD1ND} we see that $\mathbf{TND}^{n+1} \sim \mathcal{O}(\dt ^2)$. Hence
$$
\delta_t E(\phi_1^{n+1},\phi_2^{n+1},\phi_3^{n+1})
+ \sum_{i=1}^3 M_i\|\nabla\mu_i^{n+\frac12}\|^2_{\xLtwo}
+ \mathcal{O}(\dt ^2)
\,= \,
0 \, .
$$
\end{proof}

\begin{bsrvtn}
Scheme \eqref{eq:schemeNTC2} is linear, second order in time and introduces numerical dissipation of order $(\Delta t)^2$, but we can't control its sign, hence the scheme is not provably energy-stable. However, since the numerical dissipation is rather small, in practice the scheme always satisfies the decreasing energy in time property.
\end{bsrvtn}

\section{Simulations}\label{sec:simulations}
In this section we present several numerical studies to explain the behavior of the proposed model and the numerical schemes. First, we show simulations to study the order of convergence in time of the three proposed schemes. Secondly we study how the different schemes approximate the liquid lens experiment, which is a common benchmark for this problem. After that we investigate additional non-trivial examples such as the interactions of droplets of two components immersed in a third one and the spinodal decomposition process in both $2D$ and $3D$. Finally, we present results extending our approach to more complicated problems, such as systems coupled with fluid dynamics effects and the potential extension to systems with more than three components. {Note that in all of the simulations in this section we will use constant and identical mobility $M_i = M$, $(i=1,2,3)$ and will detail the values of each of the rest of the parameters for each numerical experiment.}
Implementation of the schemes is done in FreeFEM$++$ \cite{freeFEM} {In all the simulations we consider structured meshes of size $h$ and the following choice of the discrete spaces:
$$
\Phi_h\times M_h 
\,=\,
\mathbb{P}_1\times\mathbb{P}_1\,,
\quad
\mbox{ and }
\quad
\U_h\times \Pi_h 
\,=\,
\mathbb{P}_1\mbox{-}bubble\times\mathbb{P}_1
$$
where $\mathbb{P}_k$ denotes the subspace of $H^1(\Omega)$ consisting of elementwise polynomials of degree at most $k \geq 0$ and the pair $\mathbb{P}_1\mbox{-}bubble\times\mathbb{P}_1$ is the mini-element, that is known to be a stable pair for Navier-Stokes \cite{GiraultRaviart}. For visualization, data is processed with MATLAB \cite{MATLAB} and ParaView \cite{Paraview}. Finally, unless otherwise specified we will present the dynamics of the system by plotting the quantity $\frac12\phi_3 + \phi_1$, which will lead to $\phi_1$ being represented by red color, $\phi_2$ by blue color and $\phi_3$ by green color.
}
\subsection{Experimental Order of Convergence {and Comparison of Computational Cost}}\label{sec:convergenceRate}

In this section we perform an experimental order of convergence test in two scenarios: a partial spreading case using $(\Sigma_1 , \Sigma_2 , \Sigma_3 ) = (1,1,1)$, and a total spreading case with $(\Sigma_1 , \Sigma_2 , \Sigma_3 ) = (-0.1,3,3)$, {using two different spatial meshes, to be sure that the spatial discretization errors are not dominating the results.} In all the cases, and for each scheme, the solution at time $T$ is computed using a sequence of time steps $\dt = \,$1e-5, 5e-6, 3.33e-6, 2.5e-6, and 2e-6. Since an exact solution is not known, to calculate the experimental order of convergence (EOC) we will use a reference solution $\left(\boldphi^{\star}, \boldmu^{\star} \right)$ which is obtained from each scheme using a relatively fine time step of $\dt = \,$1e-7 {for each of the spatial meshes.}. Then we solve the system and compute the relative errors compared with the reference solution at time $T$ using both the $\xLtwo(\Om)$ and $\xHone(\Om)$ discrete norms as follows:
\beq\label{eq:error} \dis
e_2 (\u) : = \frac{ \| \u^{\star} - \u \|_{\xLtwo} }{ \| \u^{\star} \|_{\xLtwo} } \quad \quad \mbox{ and } \quad \quad 
e_1 (\u) : = \frac{ \| \u^{\star} - \u \|_{\xHone} }{ \| \u^{\star} \|_{\xHone} }\, .
\eeq
Then the EOC is computed using adjacent time steps $\dt$ and $\widetilde{\dt}$ by the formula
\beq\label{eq:ROC} \dis
r_k := \log\left( \frac{e_k}{\widetilde{e_k}} \right) \Big/ \log\left( \frac{\dt}{\widetilde{\dt}} \right) \, .
\eeq
{In this example we set the domain to $\Omega=[0,1]^2$, the final time to $T=1\mbox{e-}4$, the mobility to $M=1\mbox{e-}4$, the penalization term to $\lambda=1\mbox{e-}4$ and $\Lambda=7$}. Moreover, we use the following initial conditions:
$$
\phi^0_1 =\dis 0.3\sin(\pi x)\cos(\pi(y - 0.5) )\,,
\quad
\phi^0_2 =\dis 0.15 + 0.15\sin(2\pi x)\cos(\pi(2y - 0.5) )
\quad
\mbox{ and }
\quad
\phi^0_3 = 1 - \phi_1 - \phi_2\,.
$$

Tables~\ref{tab:EOCPartial100} and \ref{tab:EOCTotal100} show the results of the convergence test when the size of the mesh is set to $h=1/100$. We see that the two decoupled schemes offer first order convergence, and the errors associated with each scheme are comparable. Scheme NTC2 gives second order convergence, and errors one order of magnitude less than those from the other two schemes for similar time steps.

\begin{table}[h]
\begin{center}
\resizebox{\textwidth}{!}{%
{
\begin{tabular}{c|cccc|cccc|cccc}
  & \multicolumn{4}{c|}{TD1} &  \multicolumn{4}{c|}{NTD1} & \multicolumn{4}{c}{NTC2}\\
$\dt$	 	
&	$e_2(\boldphi)$	&	$r_2$	&	$e_1(\boldphi)$ &	$r_1$	
&	$e_2(\boldphi)$	&	$r_2$	&	$e_1(\boldphi)$	&	$r_1$	
&	$e_2(\boldphi)$ &	$r_2$	&	$e_1(\boldphi)$	&	$r_1$ \\
\hline
1.00e-5	&	
1.680e-3 & - & 6.710e-2 & - &
1.884e-3 & - & 8.149e-2 & - &
8.352e-6 & - & 4.869e-4 & - 
\\
5.00e-6	&	
7.461e-4 & 1.171 & 2.534e-2 & 1.404 &
7.435e-4 & 1.341 & 2.558e-2 & 1.670 &
2.085e-6 & 2.001 & 1.216e-4 & 2.000 	
\\		
3.33e-6	&	
4.730e-4 & 1.124 & 1.553e-2 & 1.207 &
4.697e-4 & 1.132 & 1.540e-2 & 1.250 &
9.257e-6 & 2.003 & 5.404e-5 & 2.001 	
\\		
2.50e-6	&	
3.474e-4 & 1.072 & 1.128e-2 & 1.111 &
3.455e-4 & 1.067 & 1.120e-2 & 1.107 &
5.198e-7 & 2.005 & 3.037e-5 & 2.003 	
\\
2.00e-6	&	
2.736e-4 & 1.069 & 8.841e-3 & 1.092 &
2.723e-4 & 1.065 & 8.790e-3 & 1.086 &
3.320e-7 & 2.008 & 1.941e-5 & 2.005 
\\ \hueco
$\dt$	 	
&	$e_2(\boldmu)$	&	$r_2$	&	$e_1(\boldmu)$&	$r_1$	
	&	$e_2(\boldmu)$	&	$r_2$	&	$e_1(\boldmu)$	&	$r_1$	
&	$e_2(\boldmu)$&	$r_2$	&	$e_1(\boldmu)$	&	$r_1$ 
\\
\hline 
1.00e-5	&	
1.205e-2 & - & 6.468e-2 & - &
1.207e-2 & - & 7.867e-2 & - &
7.633e-5 & - & 4.340e-3 & - 
\\
5.00e-6	&	
5.915e-3 & 1.027 & 2.831e-2 & 1.161 &
5.917e-3 & 1.029 & 2.958e-2 & 1.382 &
1.933e-5 & 1.981 & 1.100e-3 & 1.980 	
\\		
3.33e-6	&	
3.891e-3 & 1.033 & 1.747e-2 & 1.182 &
3.891e-3 & 1.033 & 1.750e-2 & 1.282 &
8.611e-6 & 1.994 & 4.899e-4 & 1.994 	
\\		
2.50e-6	&	
2.884e-3 & 1.041 & 1.272e-2 & 1.100 &
2.884e-3 & 1.041 & 1.272e-2 & 1.106 &
4.846e-6 & 1.998 & 2.756e-4 & 1.999 	
\\
2.00e-6	&	
2.281e-3 & 1.050 & 9.984e-3 & 1.084 &
2.281e-3 & 1.050 & 9.981e-3 & 1.086 &
3.102e-6 & 1.999 & 1.763e-4 & 2.001 
\\
\end{tabular} }
}
\caption{\label{tab:EOCPartial100} Experimental order of convergence for $\boldphi$ ({\it top}) and $\boldmu$ ({\it bottom}) for the case of partial spreading $(\Sigma_1,\Sigma_2,\Sigma_3)=(0.4,1.6,1.2)$ with $h=1/100$.}
\end{center}
\end{table}

\begin{table}[h]
\begin{center}
\resizebox{\textwidth}{!}{%
{
\begin{tabular}{c|cccc|cccc|cccc}
  & \multicolumn{4}{c|}{TD1} &  \multicolumn{4}{c|}{NTD1} & \multicolumn{4}{c}{NTC2}\\
$\dt$	 	
&	$e_2(\boldphi)$	&	$r_2$	&	$e_1(\boldphi)$ &	$r_1$	
&	$e_2(\boldphi)$	&	$r_2$	&	$e_1(\boldphi)$	&	$r_1$	
&	$e_2(\boldphi)$ &	$r_2$	&	$e_1(\boldphi)$	&	$r_1$ \\
\hline 
1.00e-5	&	
1.873e-3 & - & 6.358e-2 & - &
2.257e-3 & - & 8.102e-2 & - &
2.293e-5 & - & 1.306e-3 & - 
\\
5.00e-6	&	
9.370e-4 & 0.999 & 2.757e-2 & 1.204 &
9.330e-4 & 1.274 & 2.759e-2 & 1.553 &
5.715e-6 & 2.004 & 3.243e-4 & 2.010 	
\\		
3.33e-6	&	
6.111e-4 & 1.054 & 1.770e-2 & 1.093 &
6.087e-4 & 1.053 & 1.764e-2 & 1.102 &
2.537e-6 & 2.002 & 1.439e-4 & 2.004 	
\\		
2.50e-6	&	
4.515e-4 & 1.051 & 1.302e-2 & 1.067 &
4.500e-4 & 1.050 & 8.592e-2 & 1.063 &
1.425e-6 & 2.003 & 8.086e-5 & 2.003 	
\\ 
2.00e-6	&	
3.566e-4 & 1.057 & 1.026e-2 & 1.065 &
3.556e-4 & 1.054 & 1.299e-2 & 1.062 &
9.116e-7 & 2.004 & 5.169e-5 & 2.004 
\\ \hueco
$\dt$	 	
&	$e_2(\boldmu)$	&	$r_2$	&	$e_1(\boldmu)$&	$r_1$	
	&	$e_2(\boldmu)$	&	$r_2$	&	$e_1(\boldmu)$	&	$r_1$	
&	$e_2(\boldmu)$&	$r_2$	&	$e_1(\boldmu)$	&	$r_1$ 
\\
\hline
1.00e-5	&	
6.462e-2 & - & 1.565e-1 & - &
6.471e-2 & - & 1.922e-1 & - &
3.624e-4 & - & 2.110e-2 & - 
\\
5.00e-6	&	
3.169e-2 & 1.027 & 7.207e-2 &  0.991 &
3.170e-2 & 1.029 & 7.388e-2 & 1.257 &
9.551e-5 & 1.923 & 5.559e-3 & 1.924 	
\\		
3.33e-6	&	
2.085e-2 & 1.032 & 4.728e-2 & 1.038 &
2.085e-2 & 1.033 & 4.774e-2 & 1.070 &
4.290e-5 & 1.973 & 2.495e-3 & 1.975 	
\\		
2.50e-6	&	
1.545e-2 & 1.041 & 3.502e-2 & 1.042 &
1.545e-2 & 1.041 & 3.524e-2 & 1.053 &
2.423e-5 & 1.985 & 1.408e-3 & 1.989 	
\\
2.00e-6	&	
1.222e-2 & 1.051 & 2.770e-2 & 1.051 &
1.222e-2 & 1.051 & 2.782e-2 & 1.057 &
1.554e-5 & 1.989 & 9.021e-4 & 1.995 
\\
\end{tabular} }
}
\caption{\label{tab:EOCTotal100} Experimental order of convergence for $\boldphi$ ({\it top}) and $\boldmu$ ({\it bottom}) for the case of total spreading $(\Sigma_1,\Sigma_2,\Sigma_3)=(-0.1,3,3)$ with $h=1/100$.}
\end{center}
\end{table}

{
The results of the convergence test when the size of the mesh is set to $h=1/200$ are detailed in Tables~\ref{tab:EOCPartial200} and \ref{tab:EOCTotal200}. Again we observe the same behavior, the two decoupled schemes achieve first order convergence with comparable error size. Also in this case scheme NTC2 achieves second order convergence, with errors one order of magnitude less than those from the decoupled schemes.  These results makes us conclude that the spatial discretization errors are not dominating the results and we are really observing the temporal order of the error.
}

\begin{table}[h]
\begin{center}
\resizebox{\textwidth}{!}{%
\begin{tabular}{c|cccc|cccc|cccc}
  & \multicolumn{4}{c|}{TD1} &  \multicolumn{4}{c|}{NTD1} & \multicolumn{4}{c}{NTC2}\\
$\dt$	 	
&	$e_2(\boldphi)$	&	$r_2$	&	$e_1(\boldphi)$ &	$r_1$	
&	$e_2(\boldphi)$	&	$r_2$	&	$e_1(\boldphi)$	&	$r_1$	
&	$e_2(\boldphi)$ &	$r_2$	&	$e_1(\boldphi)$	&	$r_1$ \\
\hline
1.00e-5
&	1.447e-3	&	-		&	6.031e-2	&	-		
	&	1.447e-3	&	-		&	6.031e-2	&	-		
		&	4.402e-6	&	-		&	3.918e-4	&	-	\\
		
5.00e-6
&	6.505e-4	&	1.154	&	2.358e-2	&	1.355	
	&	6.505e-4	&	1.154	&	2.358e-2	&	1.355	
		&	1.807e-6	&	2.018	&	9.601e-5	&	2.029	\\
		
3.33e-6
&	4.192e-4	&	1.084	&	1.460e-2	&	1.183
	&	4.192e-4	&	1.084	&	1.460e-2	&	1.183
		&	4.820e-7	&	2.005	&	4.261e-5	&	2.004	\\
		
2.50e-6
&	3.086e-5	&	1.064	&	1.063e-2	&	1.103
	&	3.086e-5	&	1.064	&	1.063e-2	&	1.103
		&	2.706e-7	&	2.007	&	2.394e-5	&	2.004	\\
		
2.00e-6
&	2.433e-5	&	1.065	&	8.340e-3	&	1.087
	&	2.433e-5	&	1.065	&	8.340e-3	&	1.087
		&	1.728e-7	&	2.010	&	1.530e-5	&	2.006	\\
\hueco
$\dt$	 	
&	$e_2(\boldmu)$	&	$r_2$	&	$e_1(\boldmu)$&	$r_1$	
&	$e_2(\boldmu)$	&	$r_2$	&	$e_1(\boldmu)$	&	$r_1$	
&	$e_2(\boldmu)$&	$r_2$	&	$e_1(\boldmu)$	&	$r_1$ \\
\hline
1.00e-5
&	1.205e-2	&	-		&	6.066e-2	&	-		
	&	1.205e-2	&	-		&	6.066e-2	&	-		
		&	1.055e-3	&	-		&	1.183e-1	&	-	\\
		
5.00e-6
&	5.912e-3	&	1.027	&	2.777e-2	&	1.127	
	&	5.912e-3	&	1.027	&	2.777e-2	&	1.127	
		&	3.115e-4	&	1.761	&	3.532e-2	&	1.744	\\
		
3.33e-6
&	3.889e-3	&	1.032	&	1.740e-2	&	1.153
	&	3.889e-3	&	1.032	&	1.740e-2	&	1.153
		&	1.435e-4	&	1.911	&	1.632e-2	&	1.904	\\
		
2.50e-6
&	2.883e-3	&	1.041	&	1.271e-2	&	1.091
	&	2.883e-3	&	1.041	&	1.271e-2	&	1.091
		&	8.176e-5	&	1.956	&	9.306e-3	&	1.953	\\
		
2.00e-6
&	2.280e-3	&	1.051	&	9.992e-3	&	1.080
	&	2.280e-3	&	1.051	&	9.992e-3	&	1.080
		&	5.261e-5	&	1.976	&	5.990e-3	&	1.974	\\
\end{tabular} }
\caption{\label{tab:EOCPartial200} Experimental order of convergence for $\boldphi$ ({\it top}) and $\boldmu$ ({\it bottom}) for the case of partial spreading $(\Sigma_1,\Sigma_2,\Sigma_3)=(1,1,1)$ {with $h=1/200$.}}
\end{center}
\end{table}

\begin{table}[h]
\begin{center}
\resizebox{\textwidth}{!}{%
\begin{tabular}{c|cccc|cccc|cccc}
  & \multicolumn{4}{c|}{TD1} &  \multicolumn{4}{c|}{NTD1} & \multicolumn{4}{c}{NTC2}\\
$\dt$	 	
&	$e_2(\boldphi)$	&	$r_2$	&	$e_1(\boldphi)$ &	$r_1$	
&	$e_2(\boldphi)$	&	$r_2$	&	$e_1(\boldphi)$	&	$r_1$	
&	$e_2(\boldphi)$ &	$r_2$	&	$e_1(\boldphi)$	&	$r_1$ \\
\hline
1.00e-5
&	1.741e-3	&	-		&	5.921e-2	&	-		
	&	1.741e-3	&	-		&	5.921e-2	&	-		
		&	1.412e-5	&	-		&	1.140e-3	&	-	\\
		
5.00e-6
&	8.442e-4	&	1.044	&	2.606e-2	&	1.184	
	&	8.442e-4	&	1.044	&	2.606e-2	&	1.184	
		&	3.497e-6	&	2.014	&	2.816e-4	&	2.017	\\
		
3.33e-6
&	5.518e-4	&	1.049	&	1.676e-2	&	1.089
	&	5.518e-4	&	1.049	&	1.676e-2	&	1.089
		&	1.551e-6	&	2.005	&	1.249e-4	&	2.006	\\
		
2.50e-6
&	4.080e-4	&	1.049	&	1.233e-2	&	1.066
	&	4.080e-4	&	1.049	&	1.233e-2	&	1.066
		&	8.709e-7	&	2.006	&	7.014e-5	&	2.005	\\
		
2.00e-6
&	3.224e-4	&	1.055	&	9.724e-3	&	1.064
	&	3.224e-4	&	1.055	&	9.724e-3	&	1.064
		&	5.564e-7	&	2.007	&	4.484e-5	&	2.006	\\
\hueco
$\dt$	 	
&	$e_2(\boldmu)$	&	$r_2$	&	$e_1(\boldmu)$&	$r_1$	
	&	$e_2(\boldmu)$	&	$r_2$	&	$e_1(\boldmu)$	&	$r_1$	
&	$e_2(\boldmu)$&	$r_2$	&	$e_1(\boldmu)$	&	$r_1$ \\
\hline
1.00e-5
&	6.459e-2	&	-		&	1.547e-1	&	-		
	&	6.459e-2	&	-		&	1.547e-1	&	-			
		&	3.058e-3	&	-		&	3.555e-1	&	-	\\
		
5.00e-6
&	3.169e-2	&	1.028	&	7.764e-2	&	0.994	
	&	3.169e-2	&	1.028	&	7.764e-2	&	0.994	
		&	1.279e-3	&	1.257	&	1.492e-1	&	1.252	\\
		
3.33e-6
&	2.084e-2	&	1.033	&	5.097e-2	&	1.038
	&	2.084e-2	&	1.033	&	5.097e-2	&	1.038
		&	6.536e-4	&	1.657	&	7.629e-2	&	1.655	\\
		
2.50e-6
&	1.545e-2	&	1.041	&	3.776e-2	&	1.043
	&	1.545e-2	&	1.041	&	3.776e-2	&	1.043
		&	3.880e-4	&	1.813	&	4.530e-2	&	1.812	\\
		
2.00e-6
&	1.221e-2	&	1.051	&	2.986e-3	&	1.052
	&	1.221e-2	&	1.051	&	2.986e-3	&	1.052
		&	2.548e-4	&	1.885	&	2.975e-2	&	1.885	\\
\end{tabular} }
\caption{\label{tab:EOCTotal200} Experimental order of convergence for $\boldphi$ ({\it top}) and $\boldmu$ ({\it bottom}) for the case of total spreading $(\Sigma_1,\Sigma_2,\Sigma_3)=(-0.1,3,3)$ {with $h=1/200$.}}
\end{center}
\end{table}

{
\subsubsection{Comparison of the computational cost}
Since accuracy needs to be balanced with efficiency, in Table~\ref{tab:ROCcomputationaltime} we present a comparison of the computational time needed to compute 100 iterations with each of the schemes using mesh $h=1/200$ under the same conditions. In particular, for this comparison we use the same desktop computer with sixteen core cpu with base clock 3.2 GHz, and 512 GB of RAM when no other software is running for each of the computations and we use direct solvers for each of the computations. As expected, scheme NTD1 that is decoupled and does not truncate the potential functions in each iteration is the fastest and offers the same convergence rate and errors of comparable size as the ones from scheme TD1, which is slow due to the need of computing the $\mathbb{P}_3$ truncated function $\widetilde{f}(\phi^n)$ for computing the scheme and the $\mathbb{P}_4$ truncated function $\widetilde{F}(\phi^{n+1})$ for computing the energy. The coupled scheme NTC2 is also slow in comparison with NTD1 due to the need of solving a three times larger linear system, but we remark that it achieves higher order convergence than scheme TD1/NTD1 but costing more than three times the computational time as compared with scheme NTD1. 
\begin{table}[h]
{
\begin{tabular}{c|c|c|c}
Scheme			& TD1	& NTD1	& NTC2 \\
\hline
Computational Time (sec) 	& 3334	& 861 		& 3244
\end{tabular}
}
\caption{\label{tab:ROCcomputationaltime} Time to compute 100 iterations for each scheme using a mesh $h=1/200$.}
\end{table}
}
\subsection{Lens Between Stratified Phases}
In this section we recreate a three component lens between two stratified phases experiment {that is usually considered as a benchmark for three component phase field systems} \cite{Boyer2006}. In all the examples in this section
 we consider the experimental parameters given in Table~\ref{tab:lensParameters} and the initial condition described in \eqref{eq:lensInitial} and presented in Figure~\ref{fig:intialLens}.
\begin{table}[H]
{
\begin{tabular}{c|c|c|c|c|c|c|c}
$\Omega$ 			& $h$ 		& $[0,T]$ 		& $\dt$ 	& $\eps$ & $\lambda$ 	& $M$ 	& $\Lambda$ \\
\hline
$[-0.25, 0.25]\times[-0.1,0.15]$ 	& 1/300	& $[0,2.5]$ 	& 1e-4 	& 1e-2  & 1e-4 	&  1e-3 	& 7 \\
\end{tabular}
}
\caption{\label{tab:lensParameters} Parameters for the stratified lens experiments.}
\end{table}
\beq\label{eq:lensInitial}
\left\{
\ba{rcl}
\phi^0_1(x,y) &=&\dis \frac12 \left[1 + \tanh\left( \frac{2}{\eps} \min\left\{  \sqrt{ x^2 + y^2 } - 0.1 , y \right\} \right) \right] \, , \\ \hueco
\phi^0_2(x,y) &=&\dis \frac12 \left[1 - \tanh\left(  \frac{2}{\eps} \max\left\{-\sqrt{ x^2 + y^2 } + 0.1 , y \right\} \right) \right] \, ,\\ \hueco
\phi^0_3(x,y) &=&\dis 1 - \phi_1 - \phi_2 \, .
\ea
\right.
\eeq
\begin{figure}[h]
\begin{center}
\includegraphics[scale=0.15]{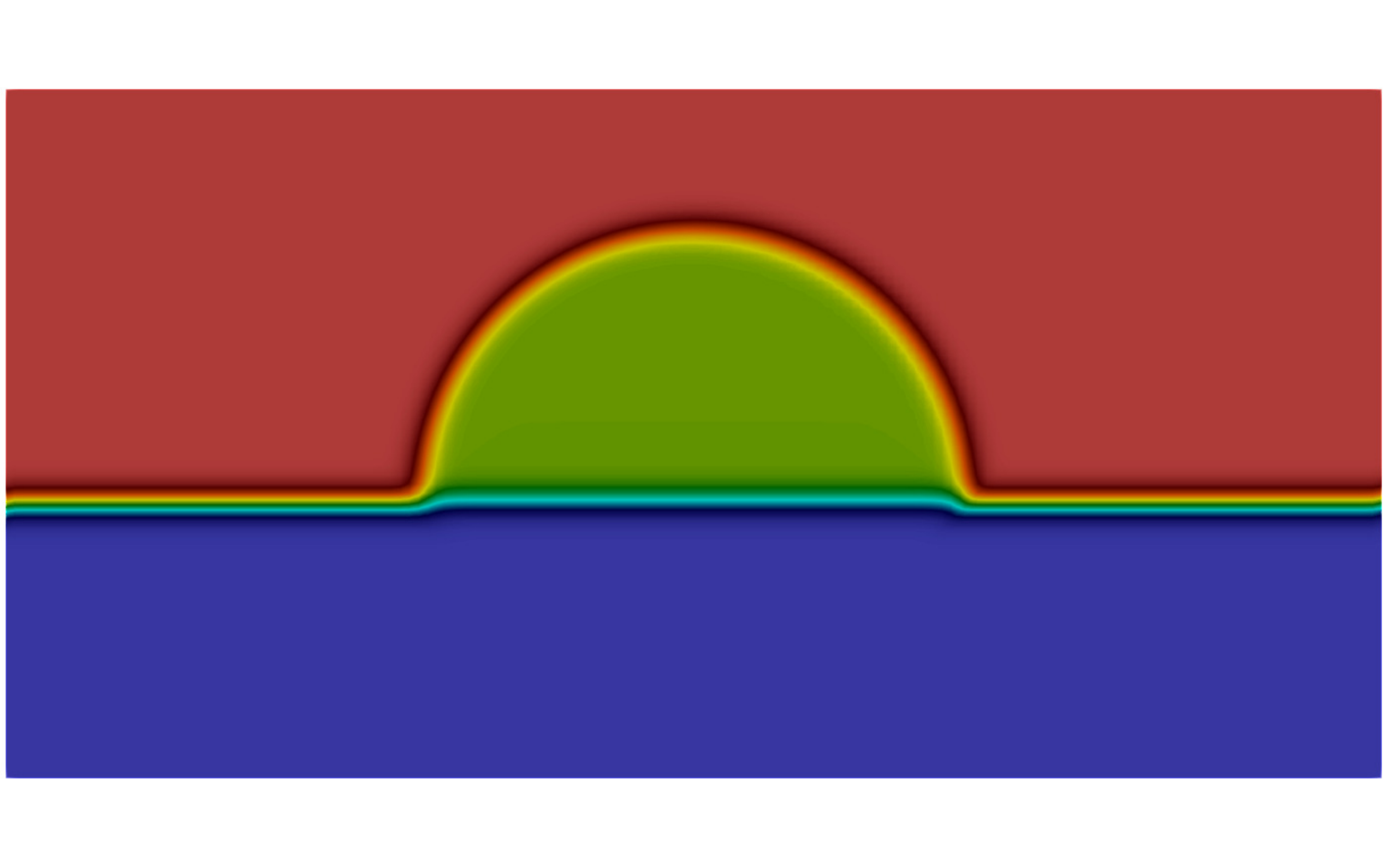}
\includegraphics[scale=0.15]{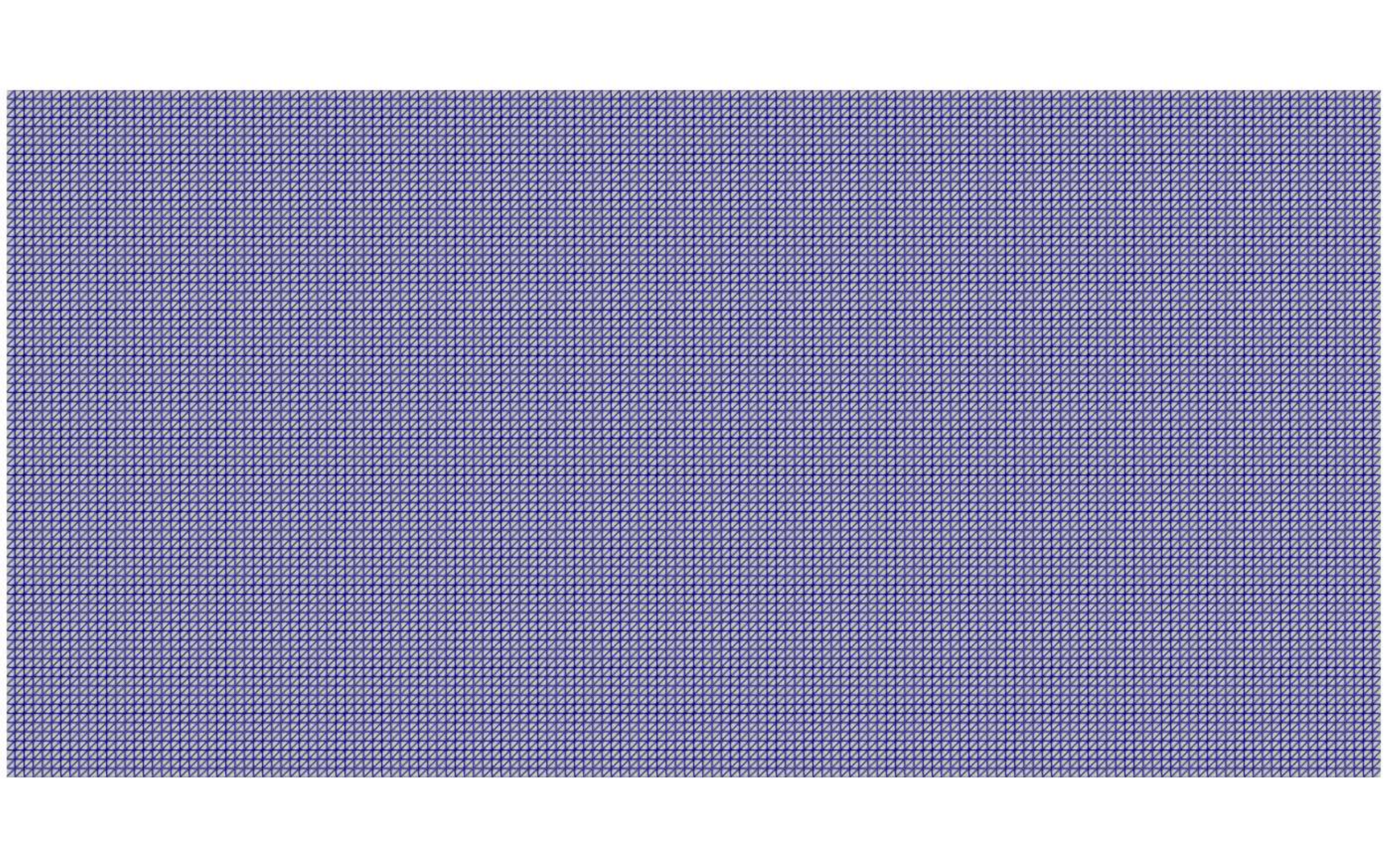}
\end{center}
\caption{Initial condition for the lens between stratified phases experiments and the considered mesh. Red color represents phase $\phi_1$, blue color represents phase $\phi_2$ and green color represents $\phi_3$.}\label{fig:intialLens}
\end{figure}

\subsubsection{Partial Spreading}

{
First we focus on two cases of positive spreading coefficients $(\Sigma_1, \Sigma_2 , \Sigma_3) = (1,1,1)$, and $(\Sigma_1, \Sigma_2 , \Sigma_3) = (0.4,1.6,1.2)$. In Figures~\ref{fig:lensPartialDynamics1} and \ref{fig:lensPartialDynamics2} we present the dynamics obtained using the three schemes presented in this work, and we can observe how all the schemes produce the same dynamics and achieve the same equilibrium configurations. Moreover, the evolution in time of the energy, the volume, the numerical dissipation and the $L^2$ and $L^\infty$ norms of the restriction $\Sigma_{i=1}^3\phi_i - 1$ are presented in Figures~\ref{fig:lensPartialPlots1} and \ref{fig:lensPartialPlots2}. In all cases the energy decays throughout the entire simulation and although the dynamics are the same for the three schemes, it is clear that the energy is decaying slower for schemes TD1 and NTC2, which is produced because the schemes are introducing two much numerical dissipation. The numerical dissipation in TD1 is positive as expected but the one associated with NTC2 is negative (it is not an energy stable scheme) and much larger in magnitude. The approximation of the restriction in $L^2$ norm seems under control for all schemes, being always of the order of $10^{-4}$ although the $L^\infty$ norm shows that there are some point in the domain where the restriction is not as well approximated (points where three components interact), but this does not seem to prevent the systems to achieve the correct dynamics. Moreover, the volume is observed to be constant for each scheme which confirms that the numerical schemes are conservative.
In this case of partial spreading, an analytical solution in the limit $\eps \rightarrow 0$ exists, and we note that our results are in agreement with those solutions discussed in other works \cite{Boyer2006,Boyer2011}.
}
\begin{figure}[h]
\begin{center}
\includegraphics[scale=0.07]{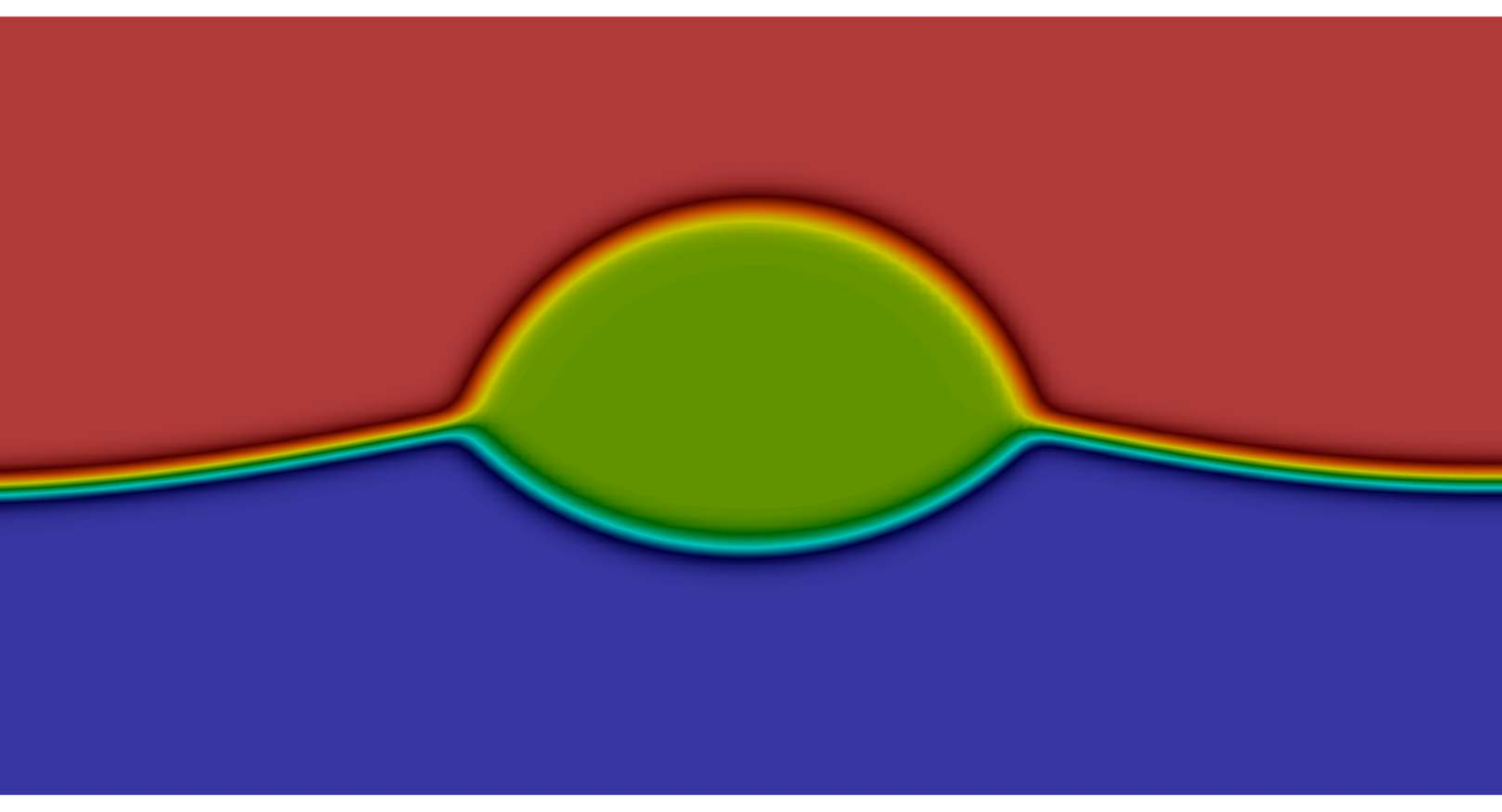}
\includegraphics[scale=0.07]{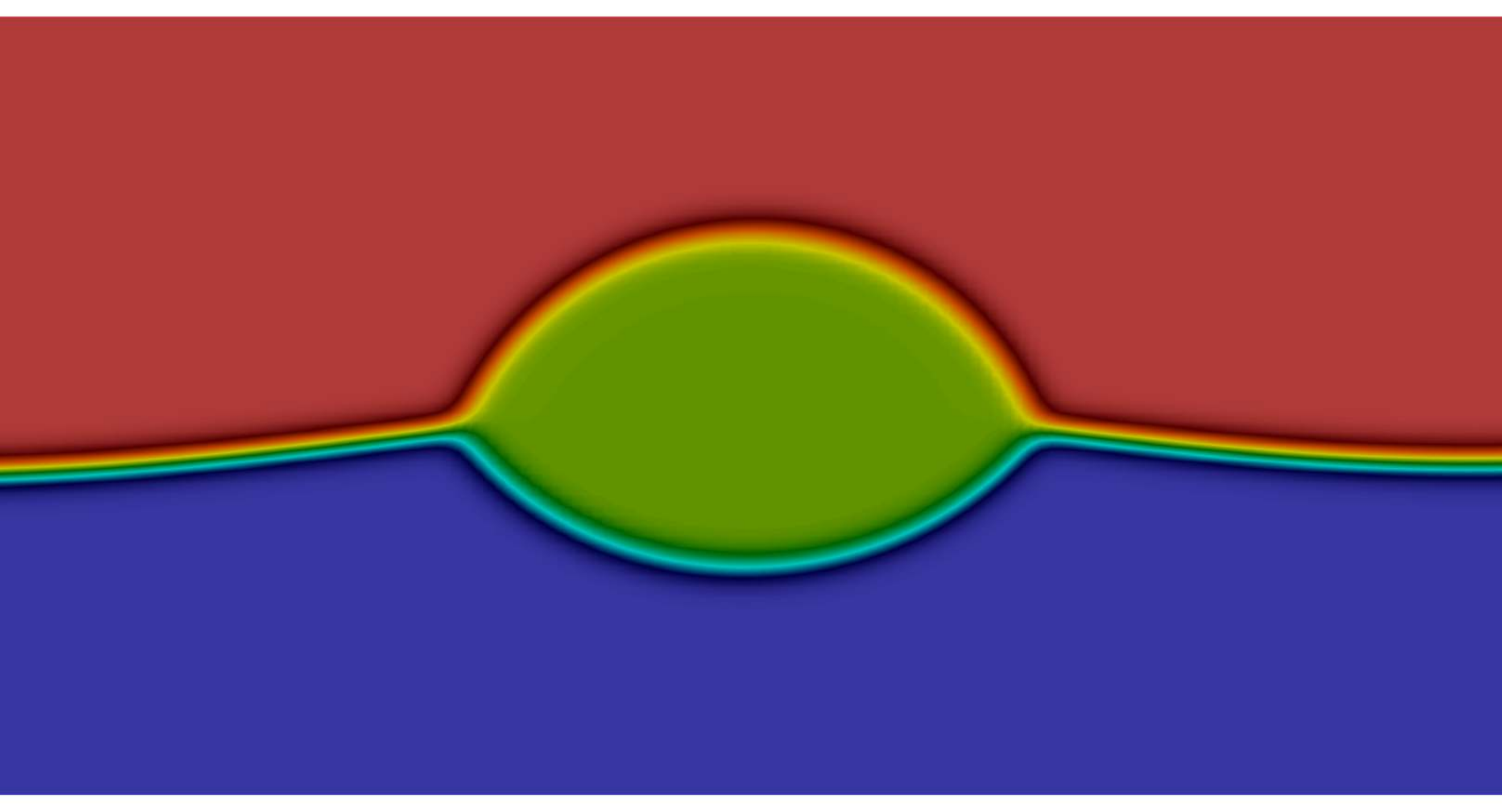}
\includegraphics[scale=0.07]{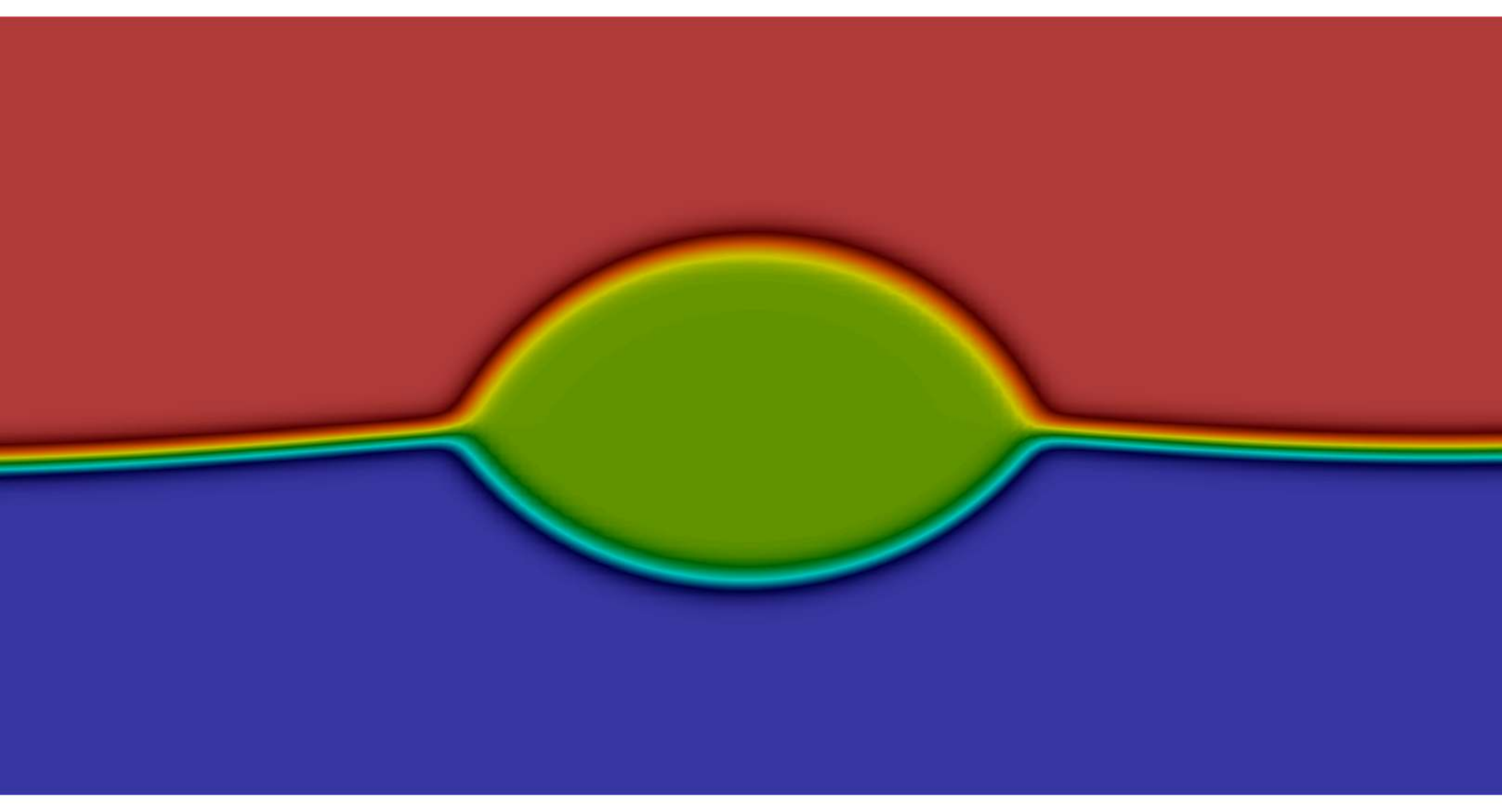}
\includegraphics[scale=0.07]{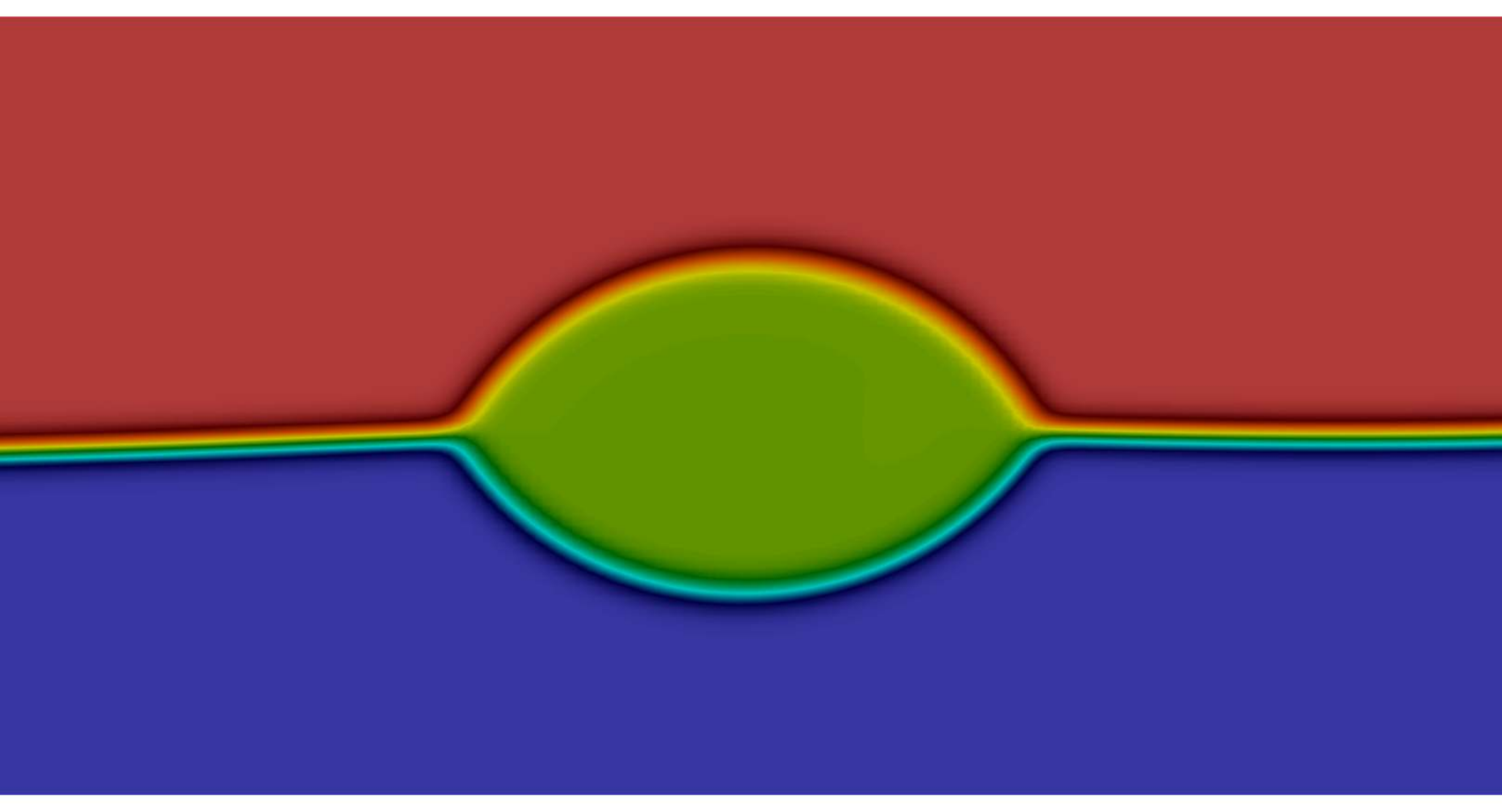}
\\
\includegraphics[scale=0.07]{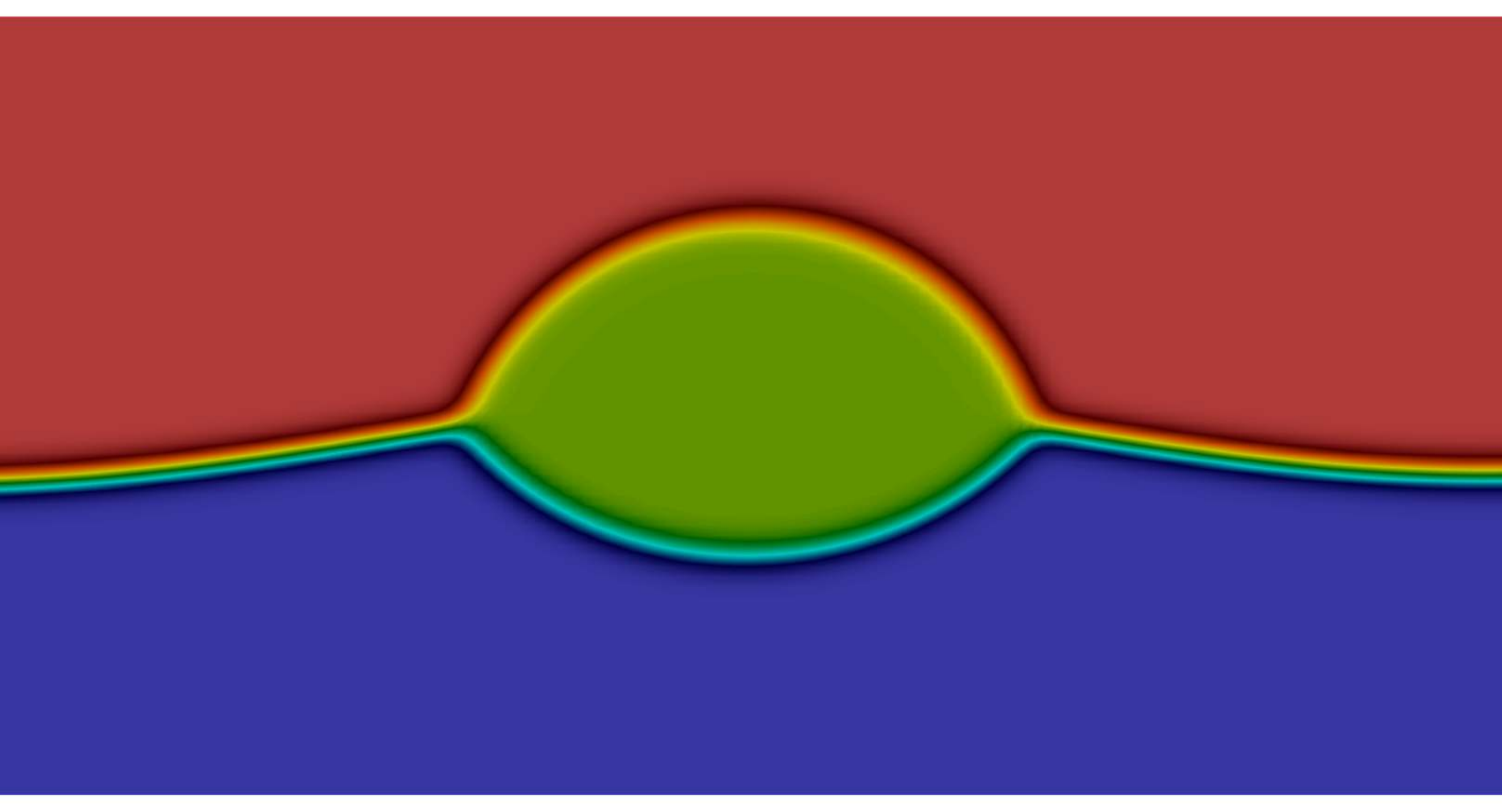}
\includegraphics[scale=0.07]{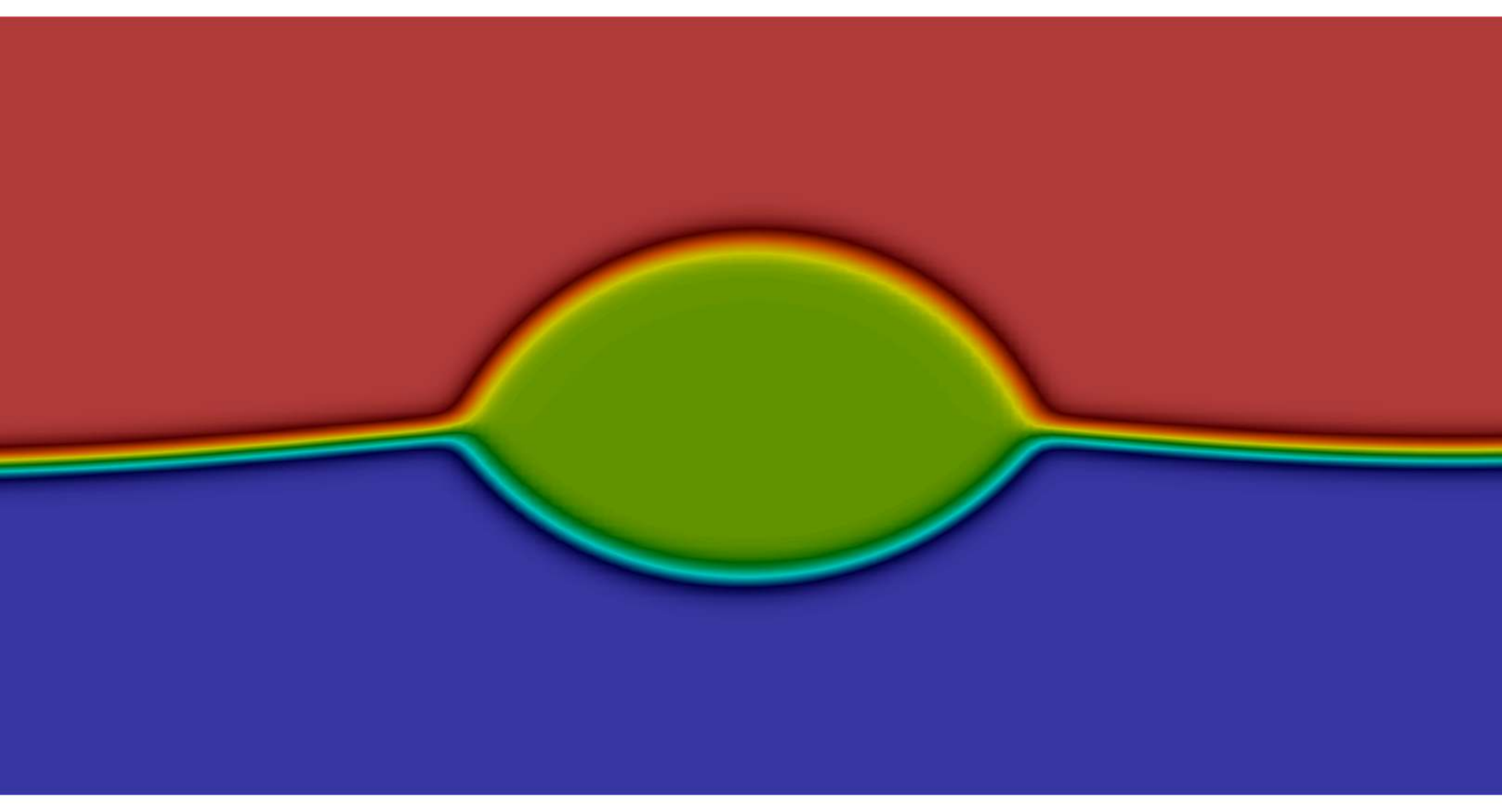}
\includegraphics[scale=0.07]{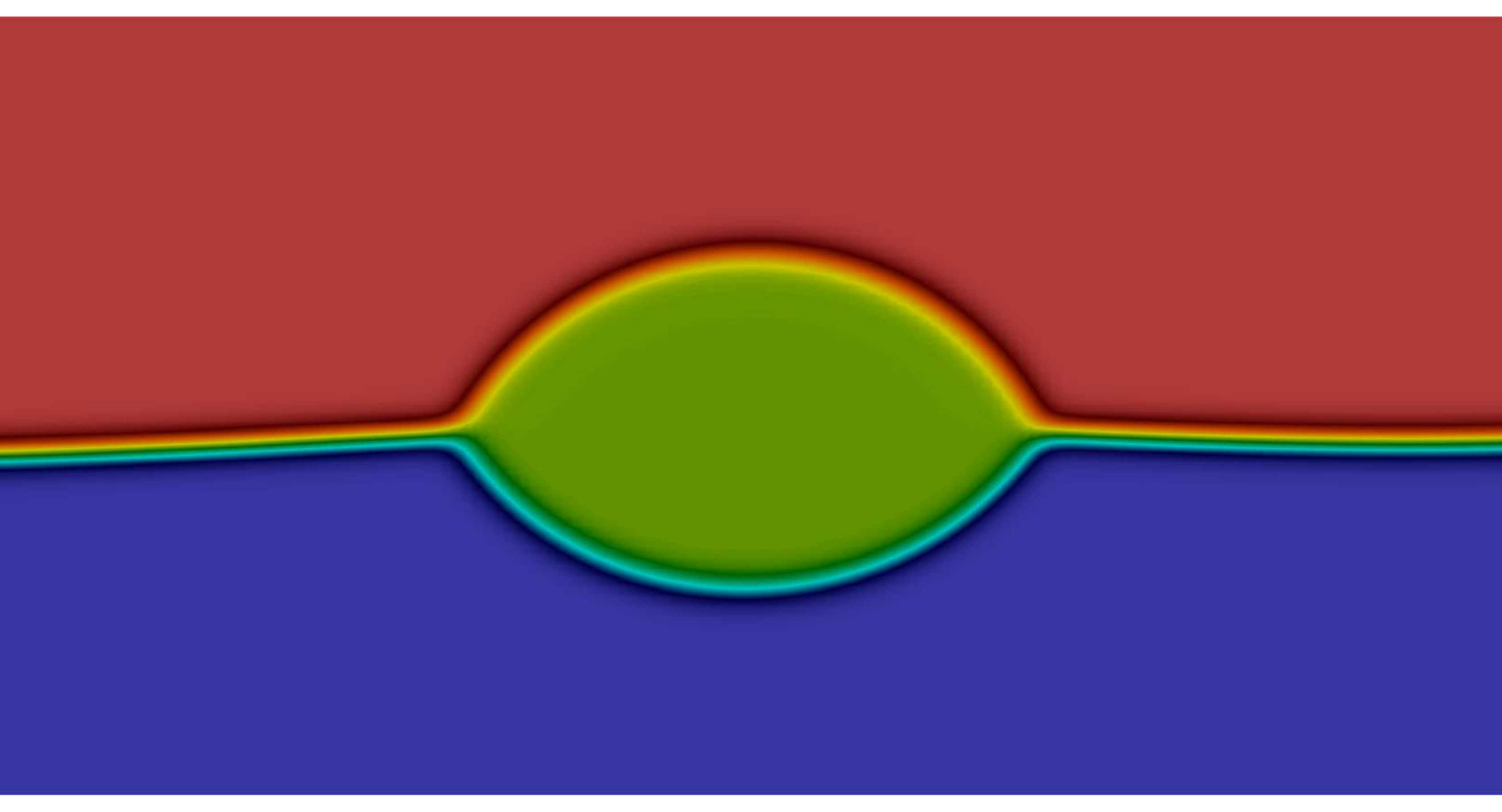}
\includegraphics[scale=0.07]{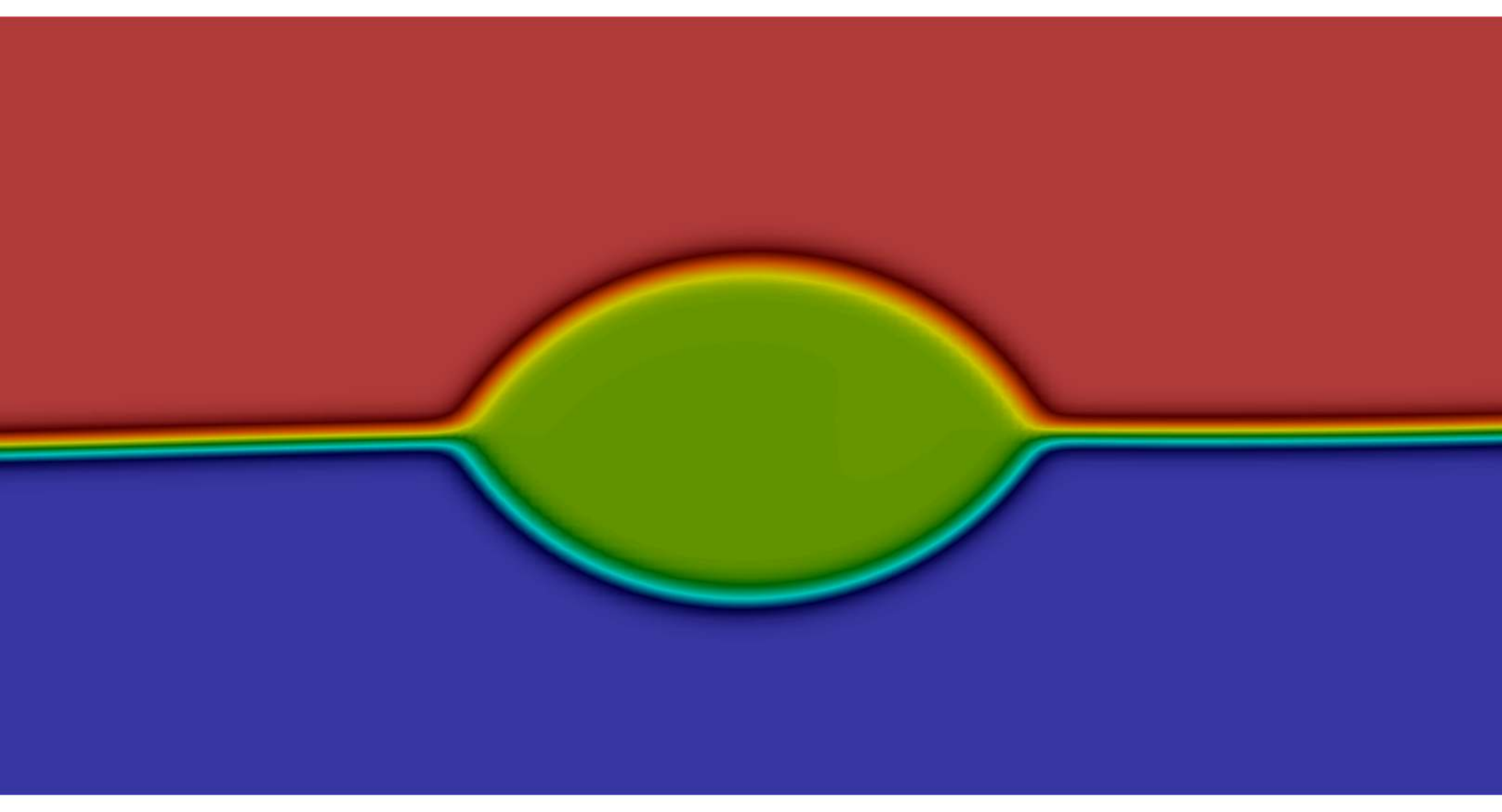}
\\
\includegraphics[scale=0.07]{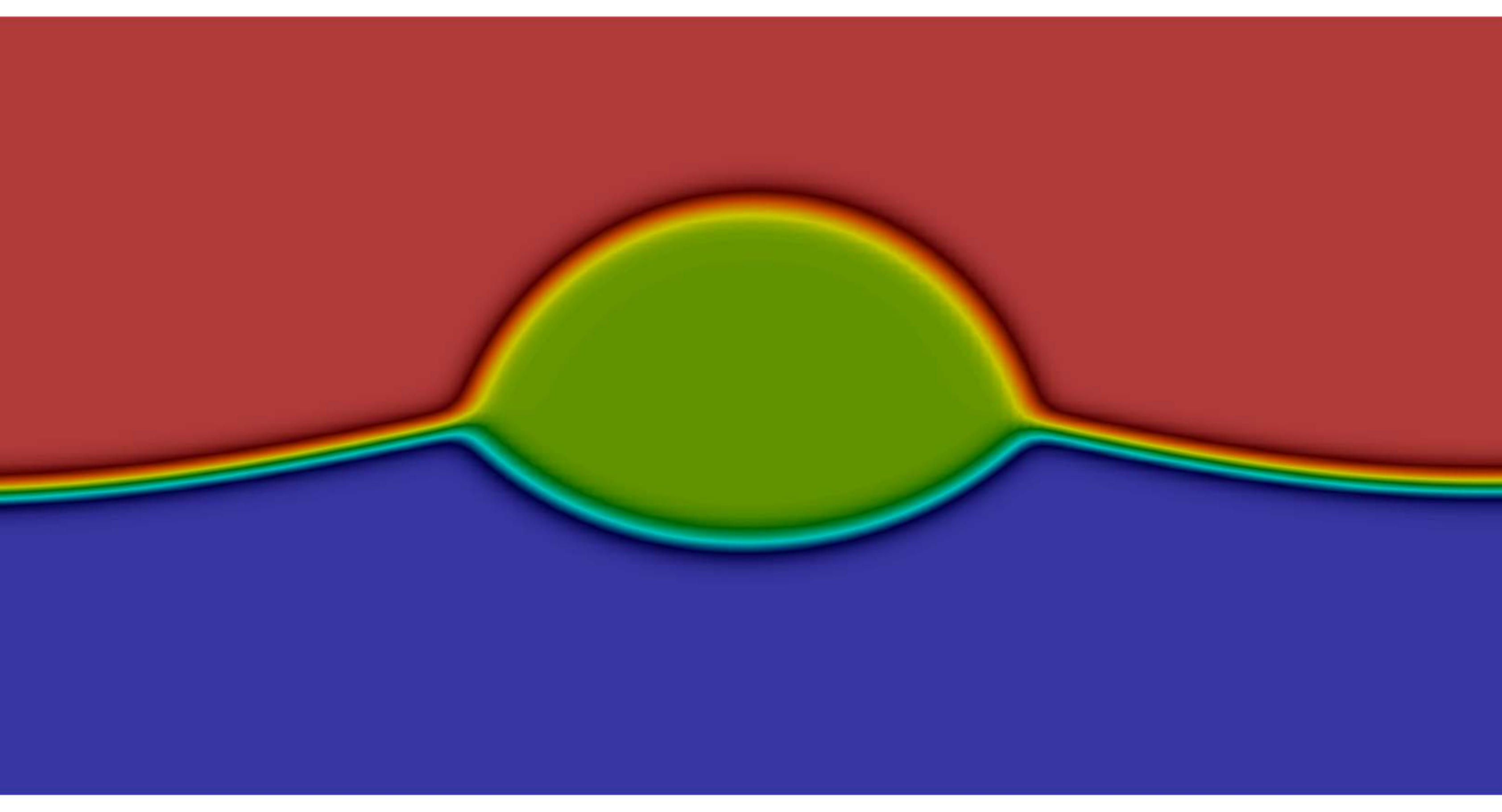}
\includegraphics[scale=0.07]{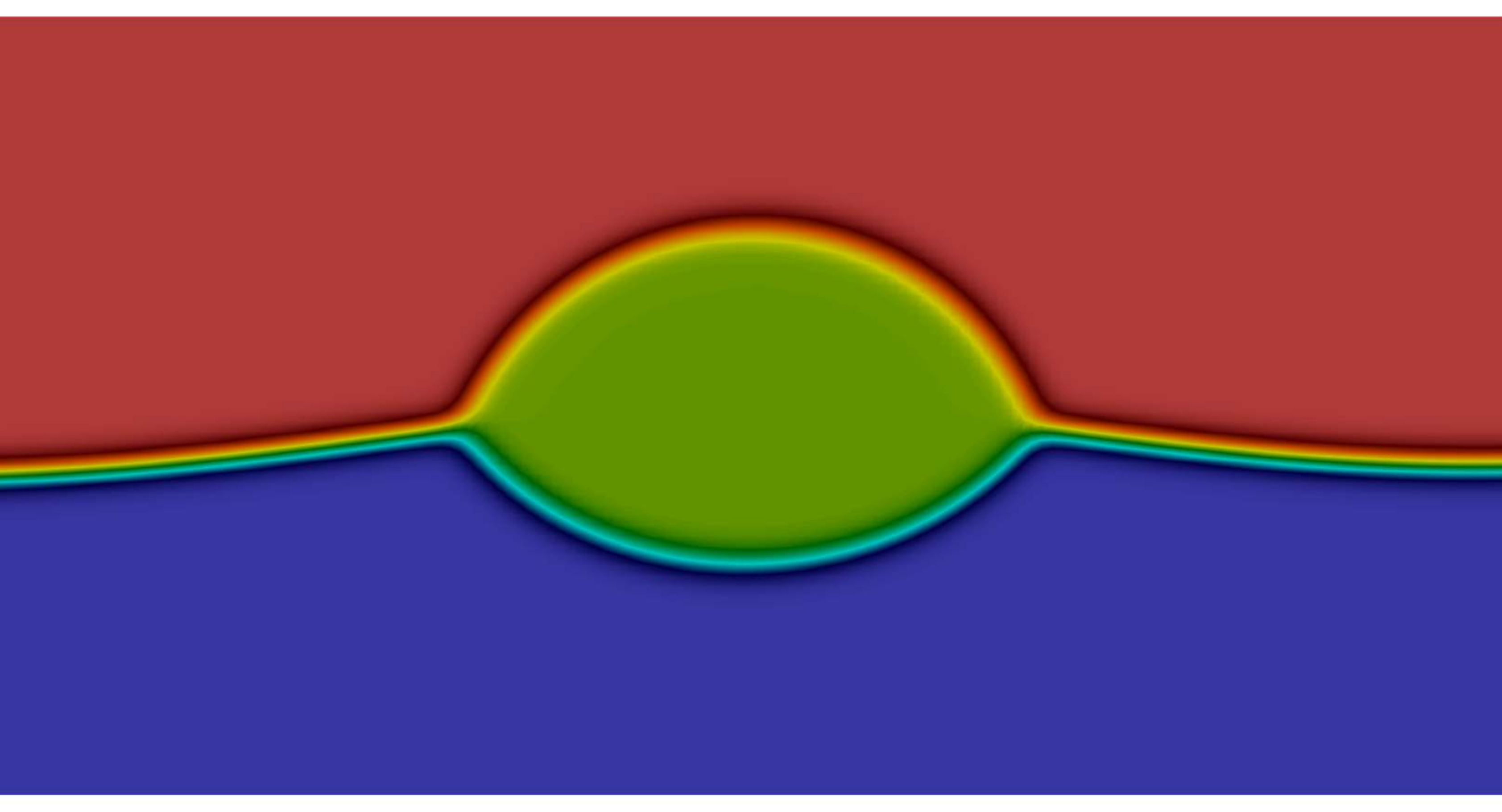}
\includegraphics[scale=0.07]{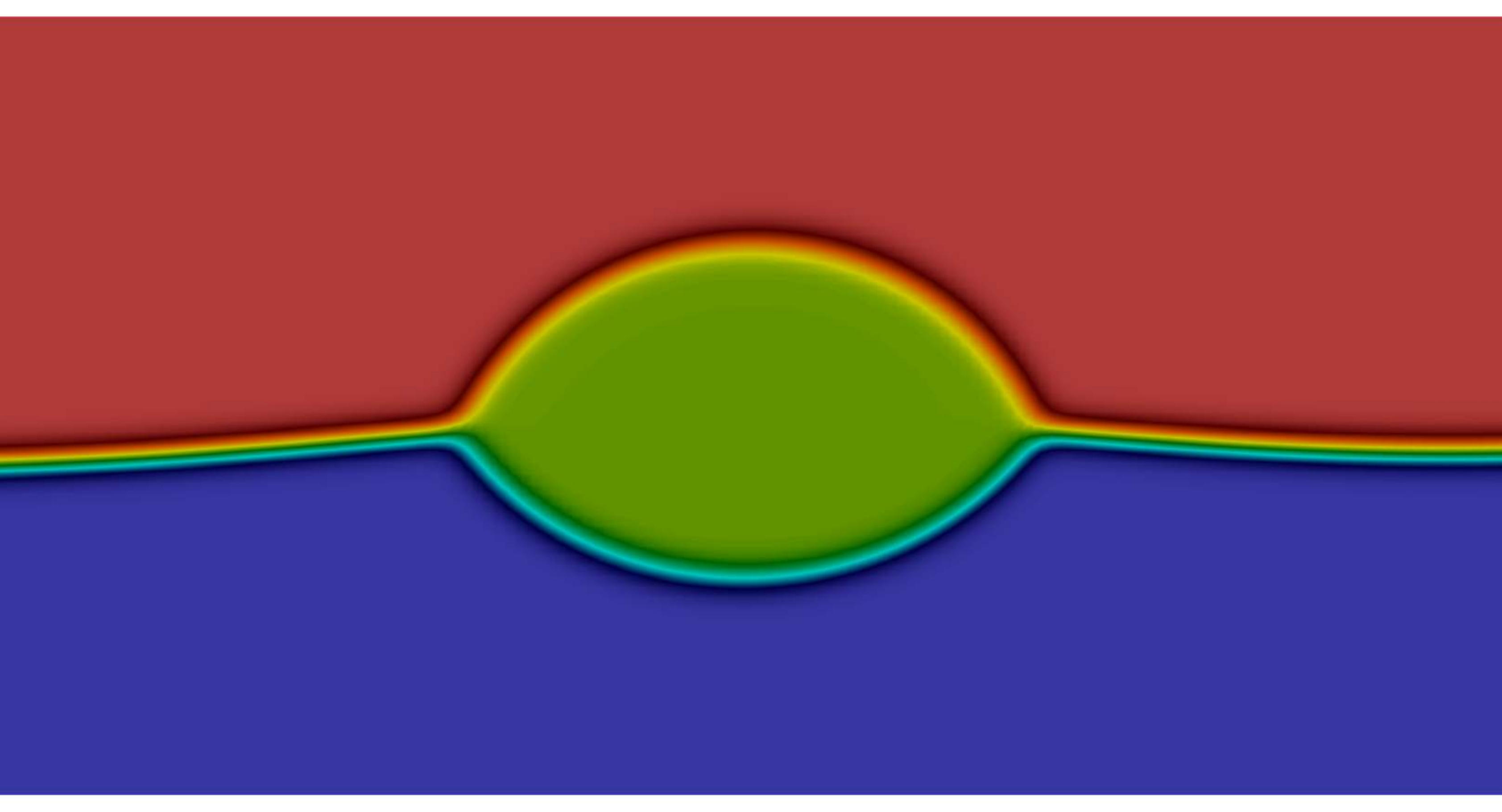}
\includegraphics[scale=0.07]{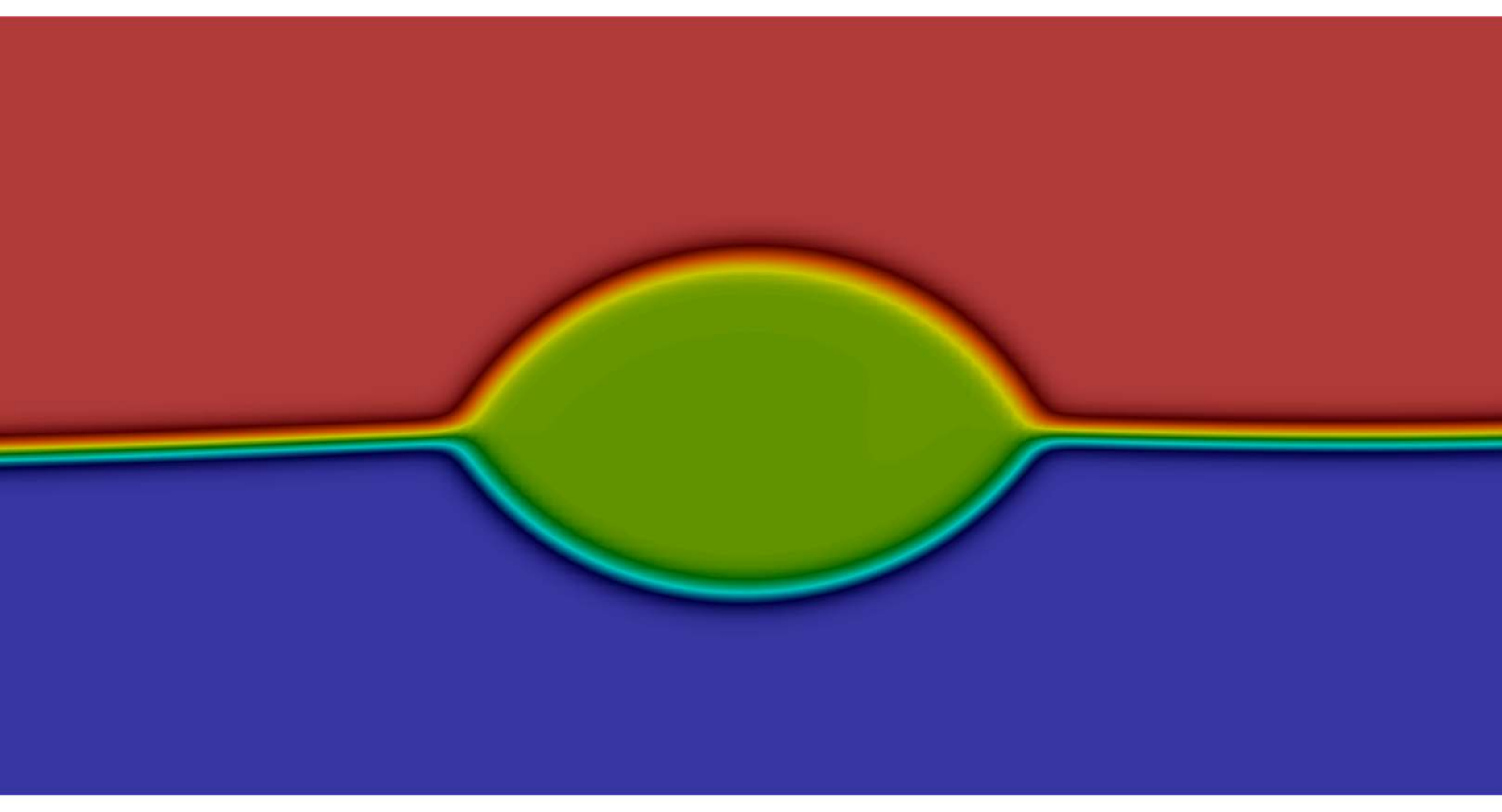}
\end{center}
\caption{Dynamics of schemes TD1 (top row), NTD1 (center row) and NTC2 (bottom row) at times $t=0.5, 1, 1.5$ and $2.5$ (from left to right) with spreading coefficients $(\Sigma_1, \Sigma_2 , \Sigma_3) = (1,1,1)$.}\label{fig:lensPartialDynamics1}
\end{figure}

\begin{figure}[h]
\begin{center}
\includegraphics[scale=0.11]{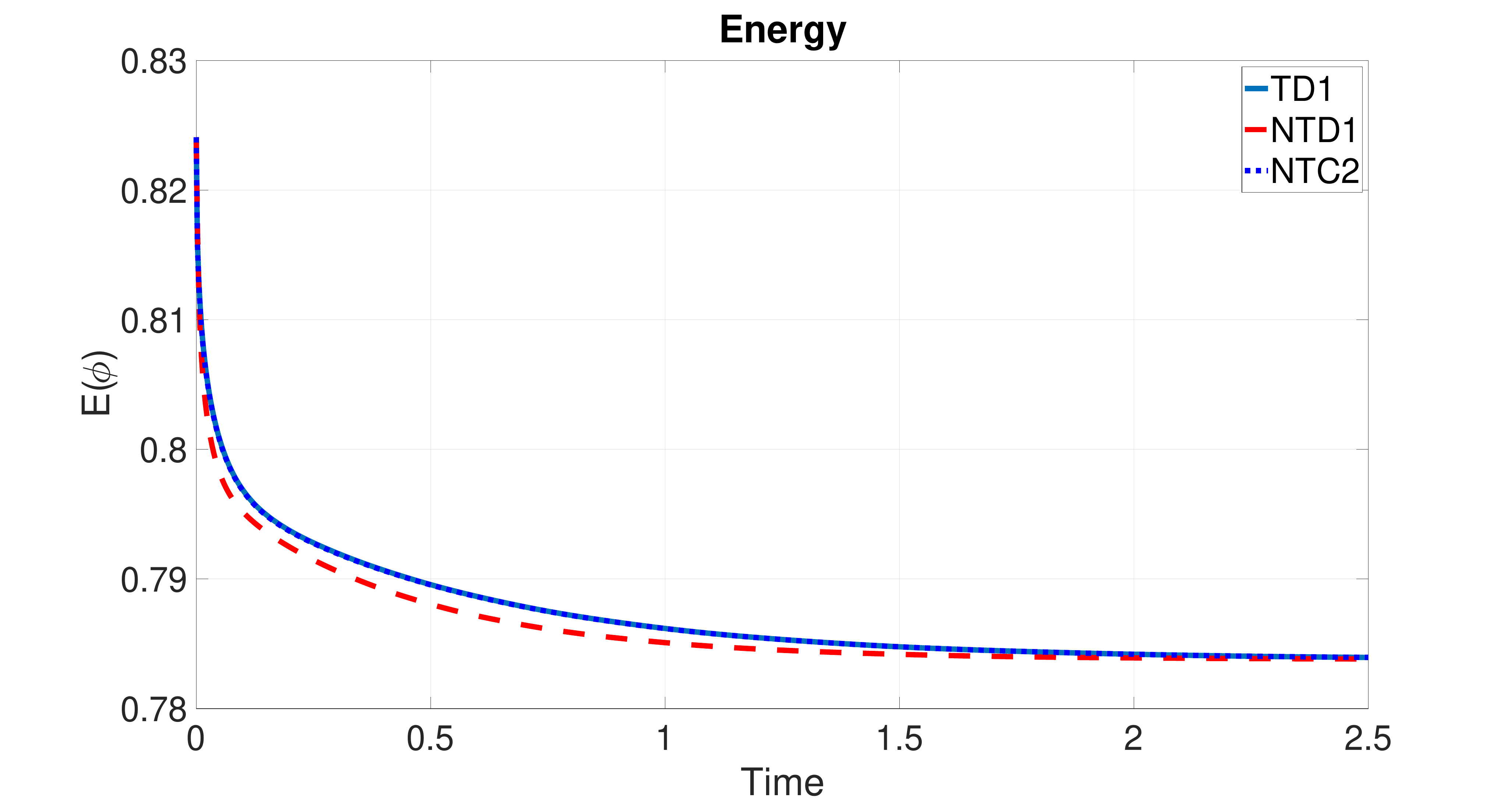}
\includegraphics[scale=0.11]{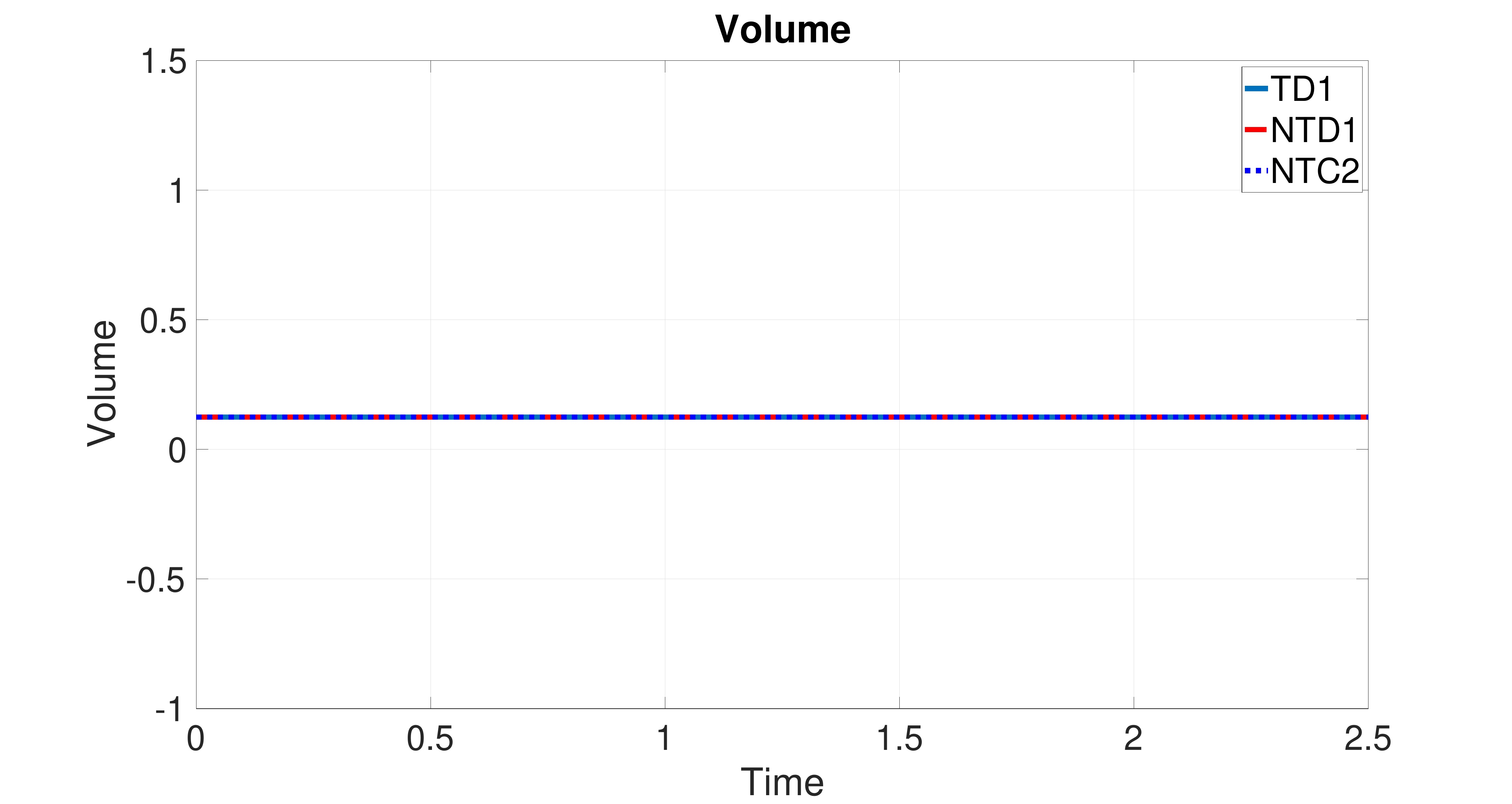}
\\ [1ex]
\includegraphics[scale=0.11]{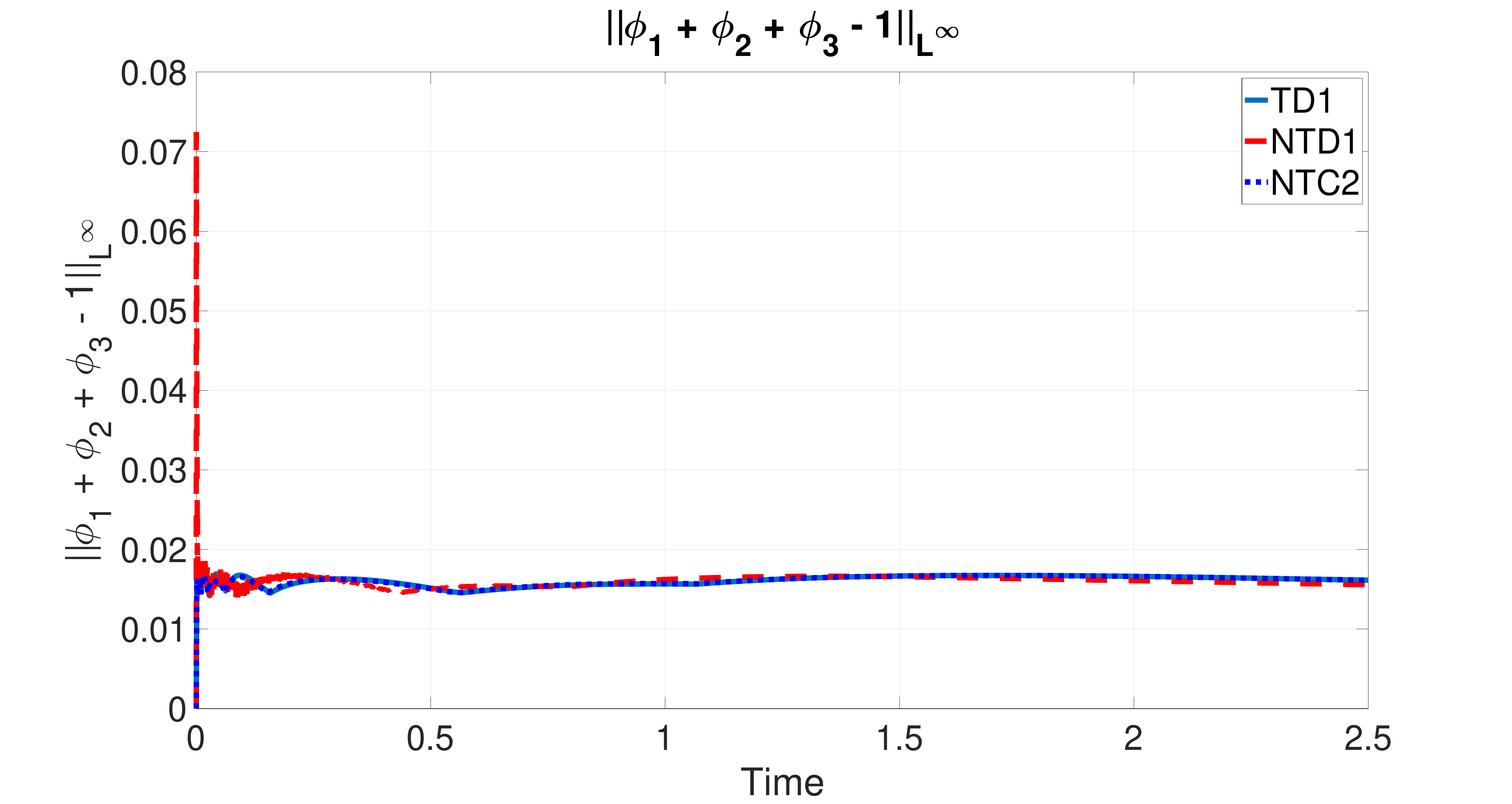}
\includegraphics[scale=0.11]{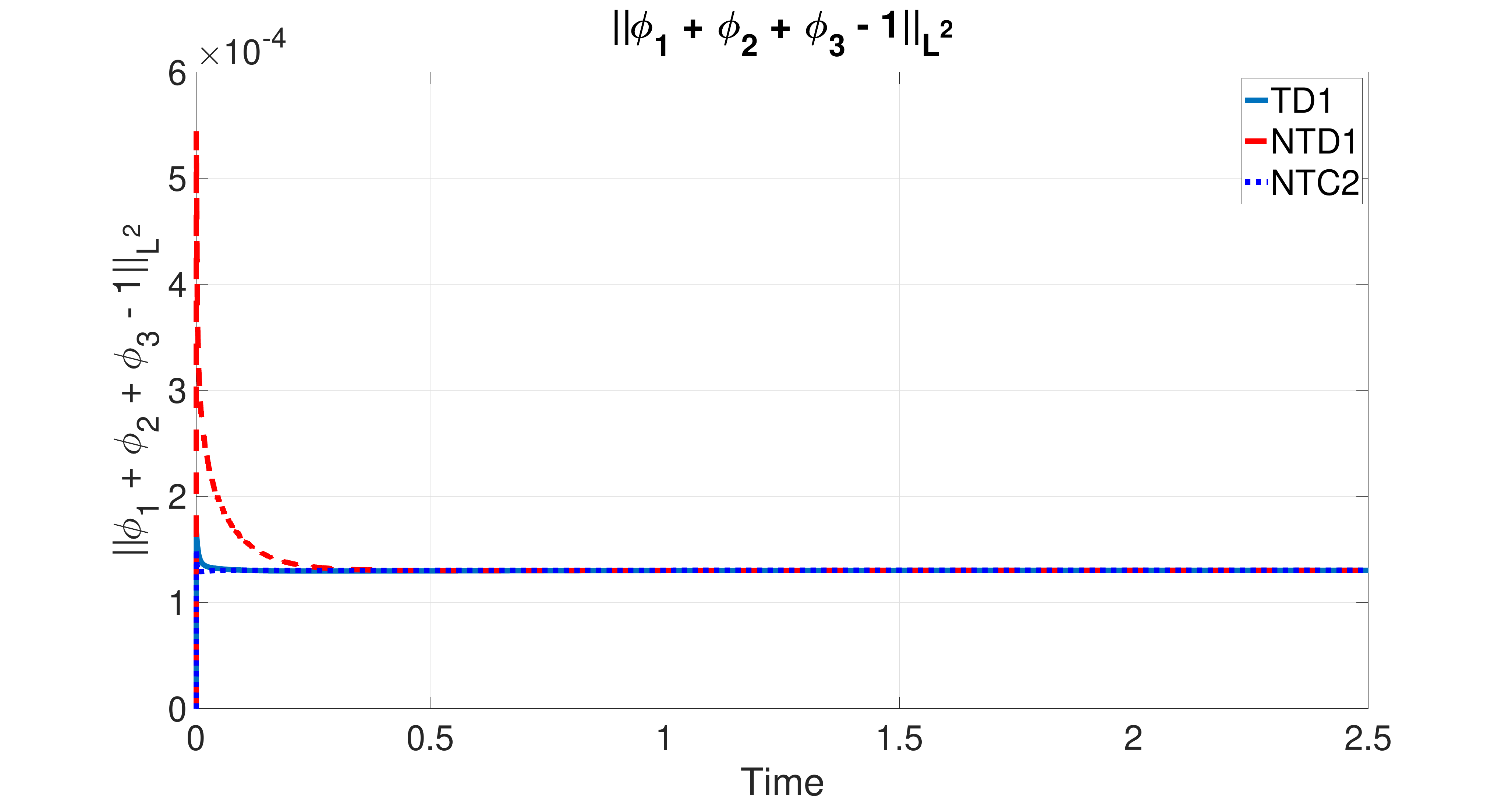}
\\ [1ex]
\includegraphics[scale=0.11]{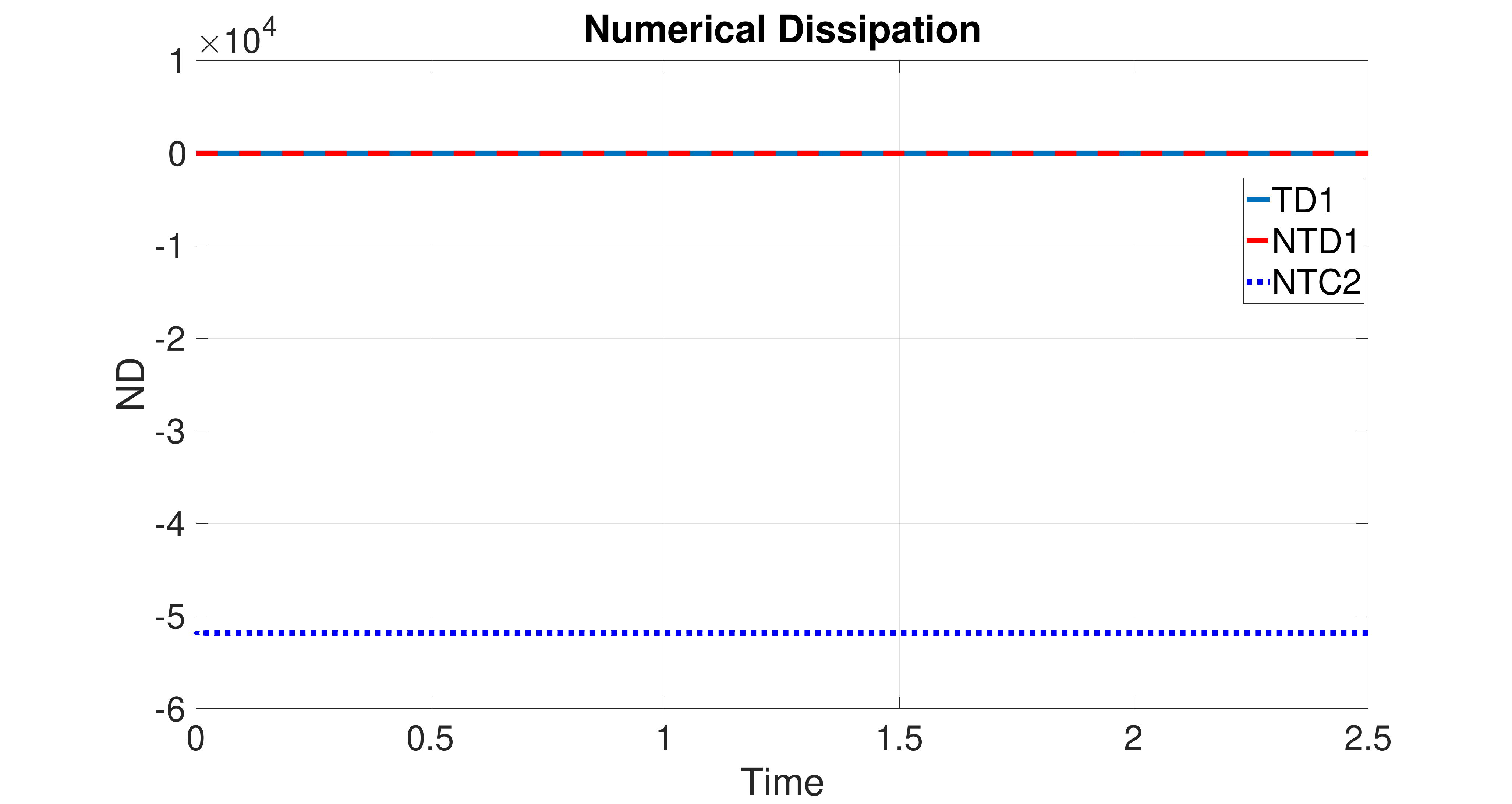}
\includegraphics[scale=0.11]{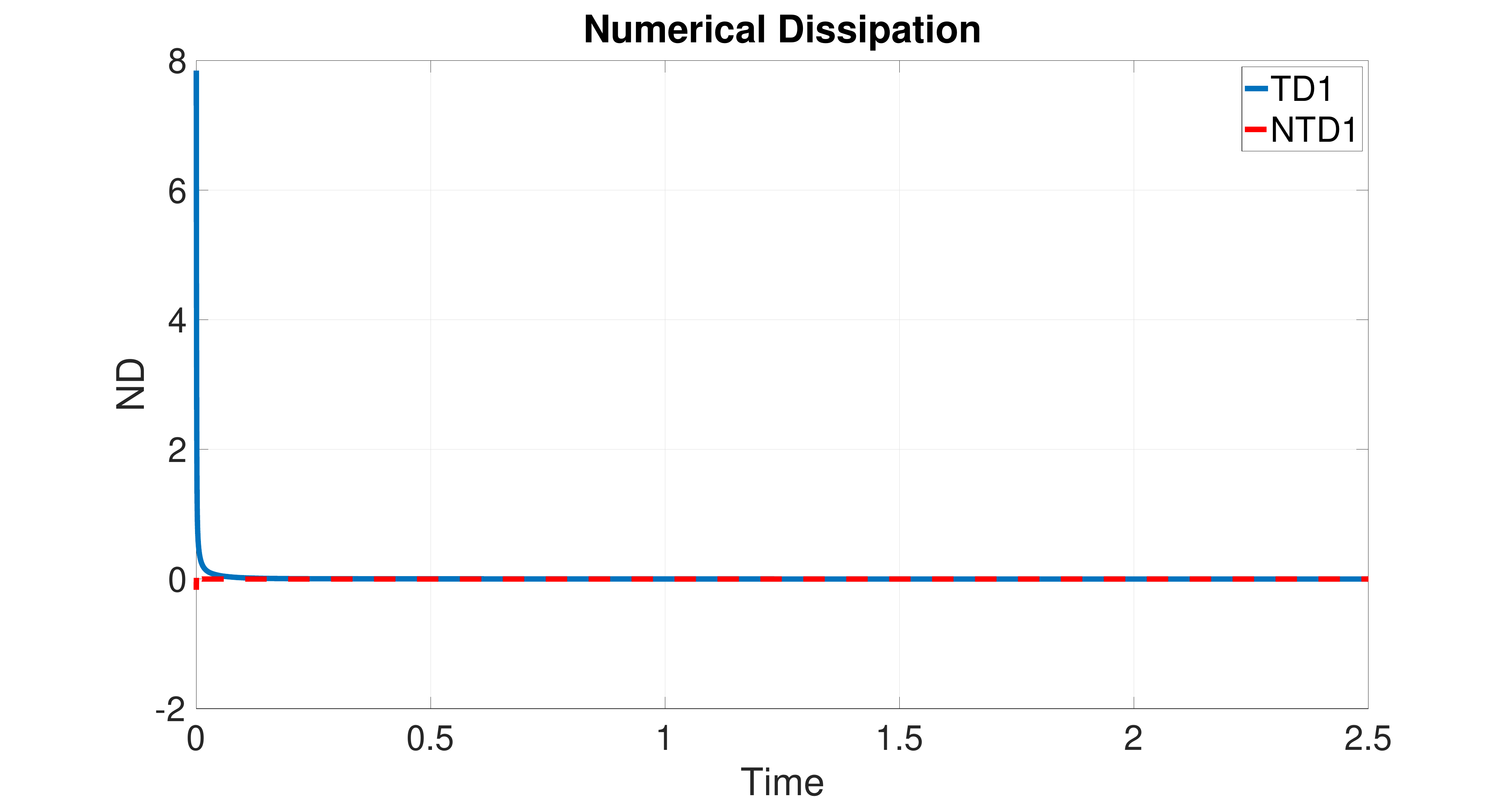}
\end{center}
\caption{Evolution in time of the energies (top left), the volume $\int_\Omega(\phi_1+\phi_2+\phi_3)$ (top right), $\|\phi_1 + \phi_2 + \phi_3 -1\|_{L^\infty}$ (center left), $\|\phi_1 + \phi_2 + \phi_3 -1\|_{L^2}$ (center right) and the evolution of the numerical dissipation (bottom row) with spreading coefficients $(\Sigma_1, \Sigma_2 , \Sigma_3) = (1,1,1)$.}
\label{fig:lensPartialPlots1}
\end{figure}

\begin{figure}[h]
\begin{center}
\includegraphics[scale=0.07]{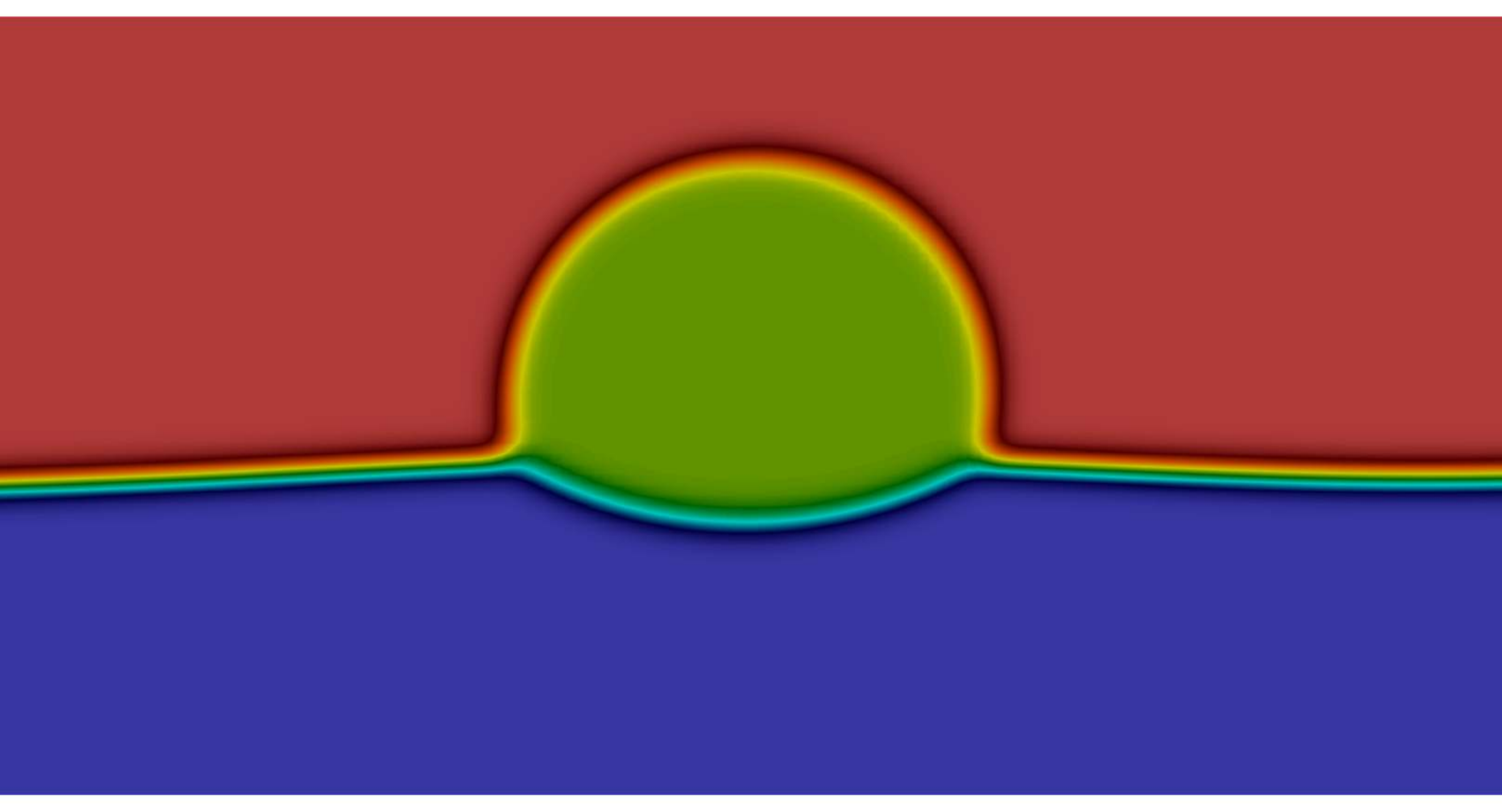}
\includegraphics[scale=0.07]{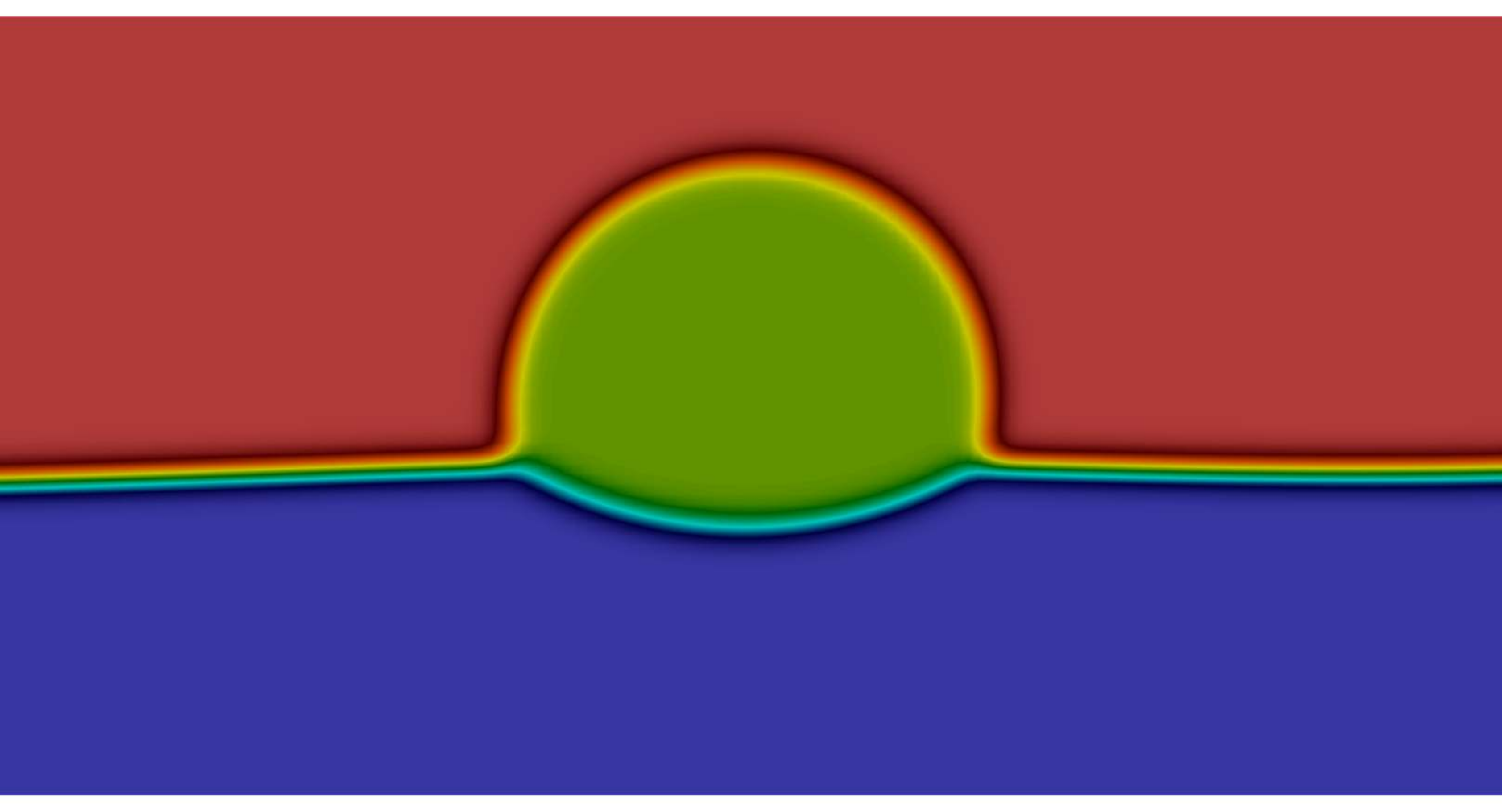}
\includegraphics[scale=0.07]{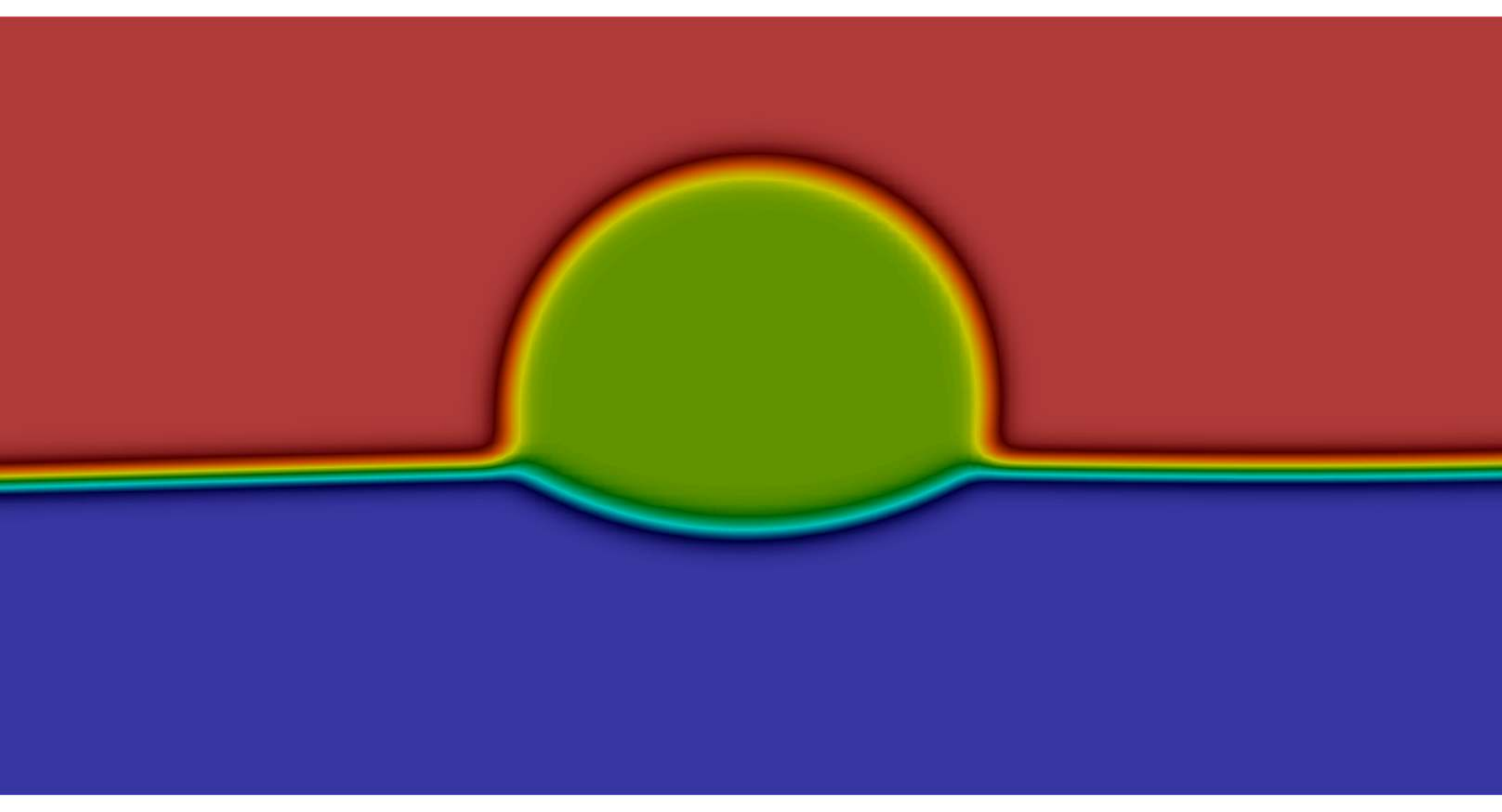}
\includegraphics[scale=0.07]{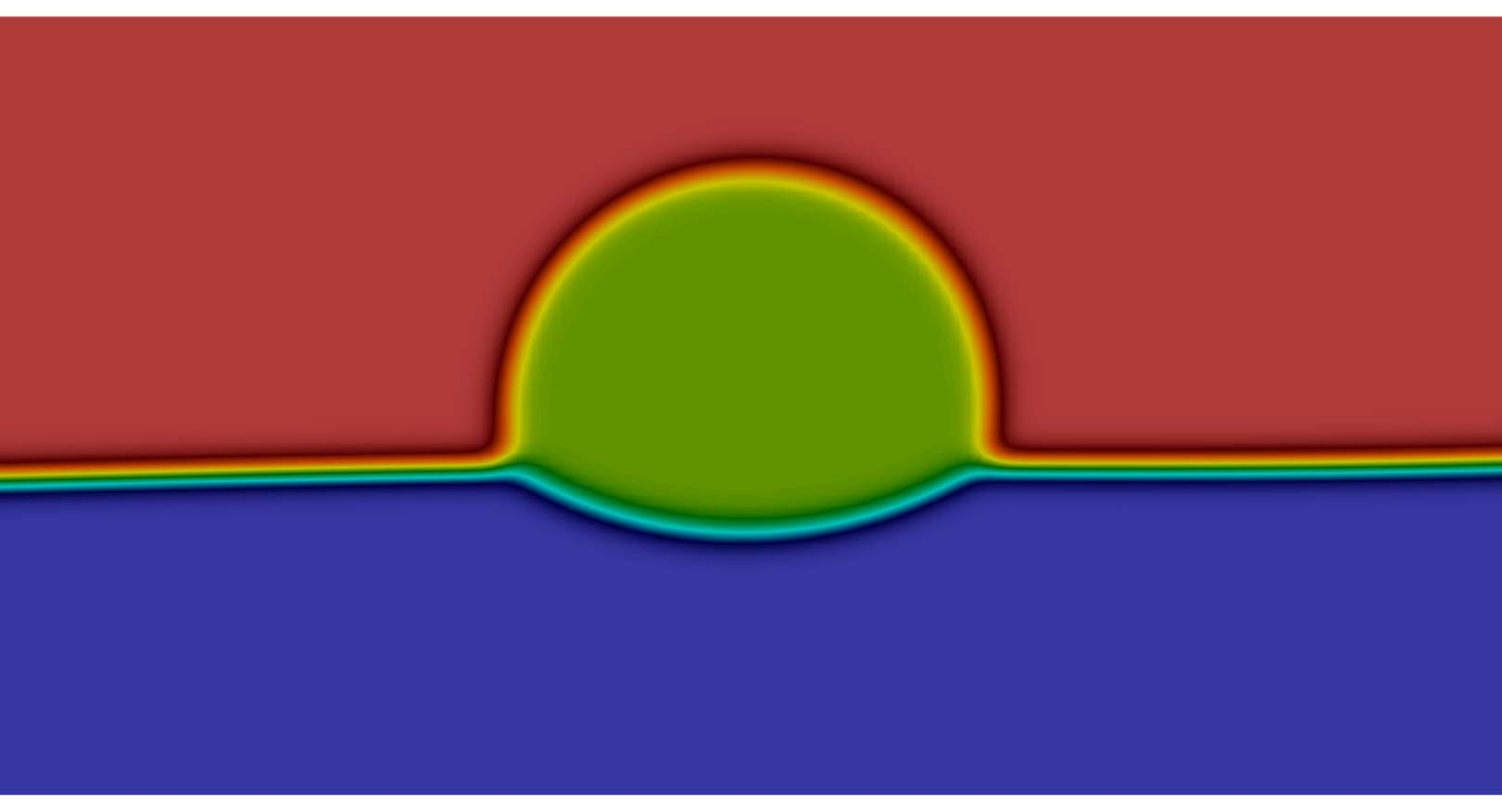}
\\
\includegraphics[scale=0.07]{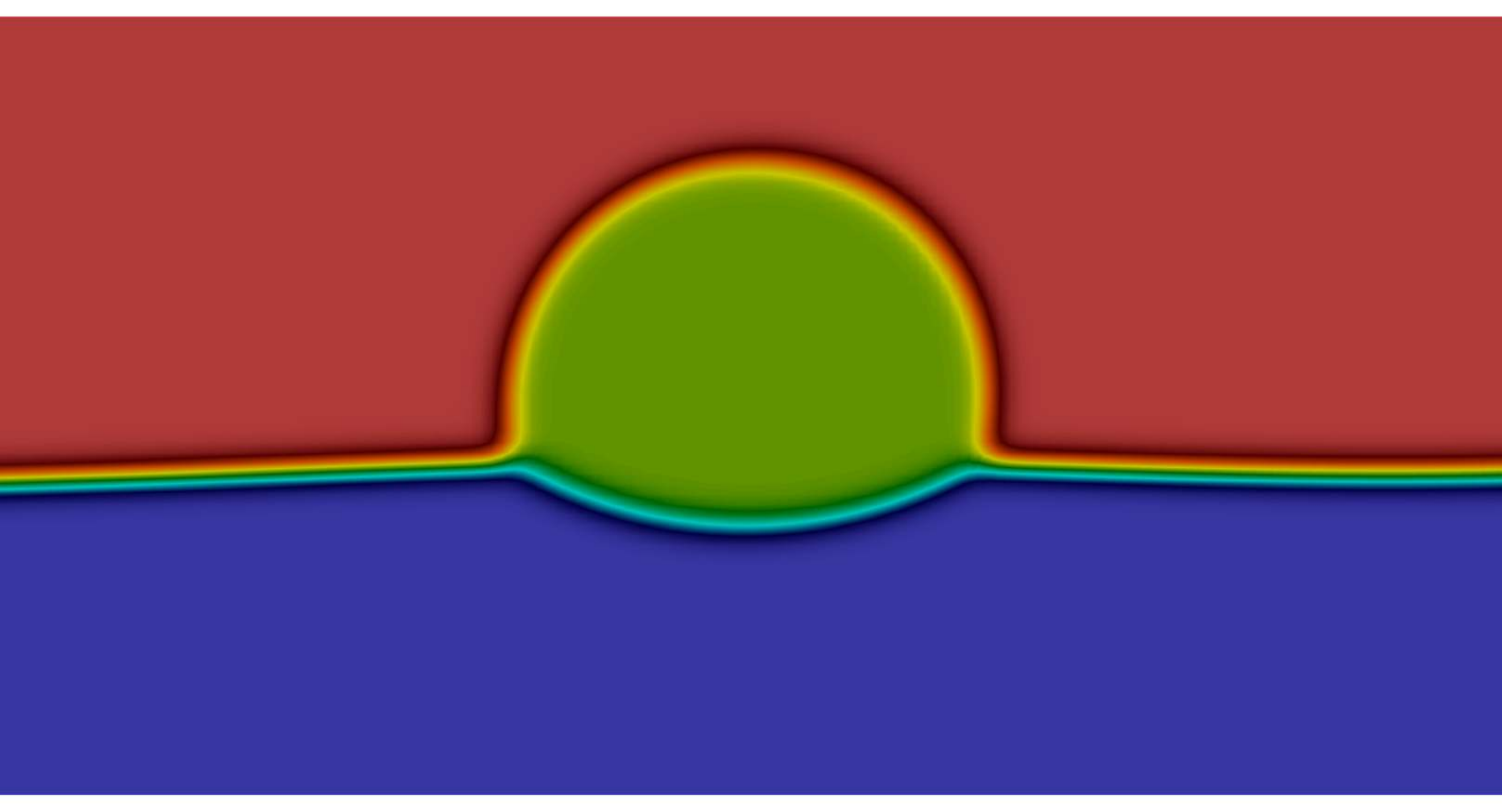}
\includegraphics[scale=0.07]{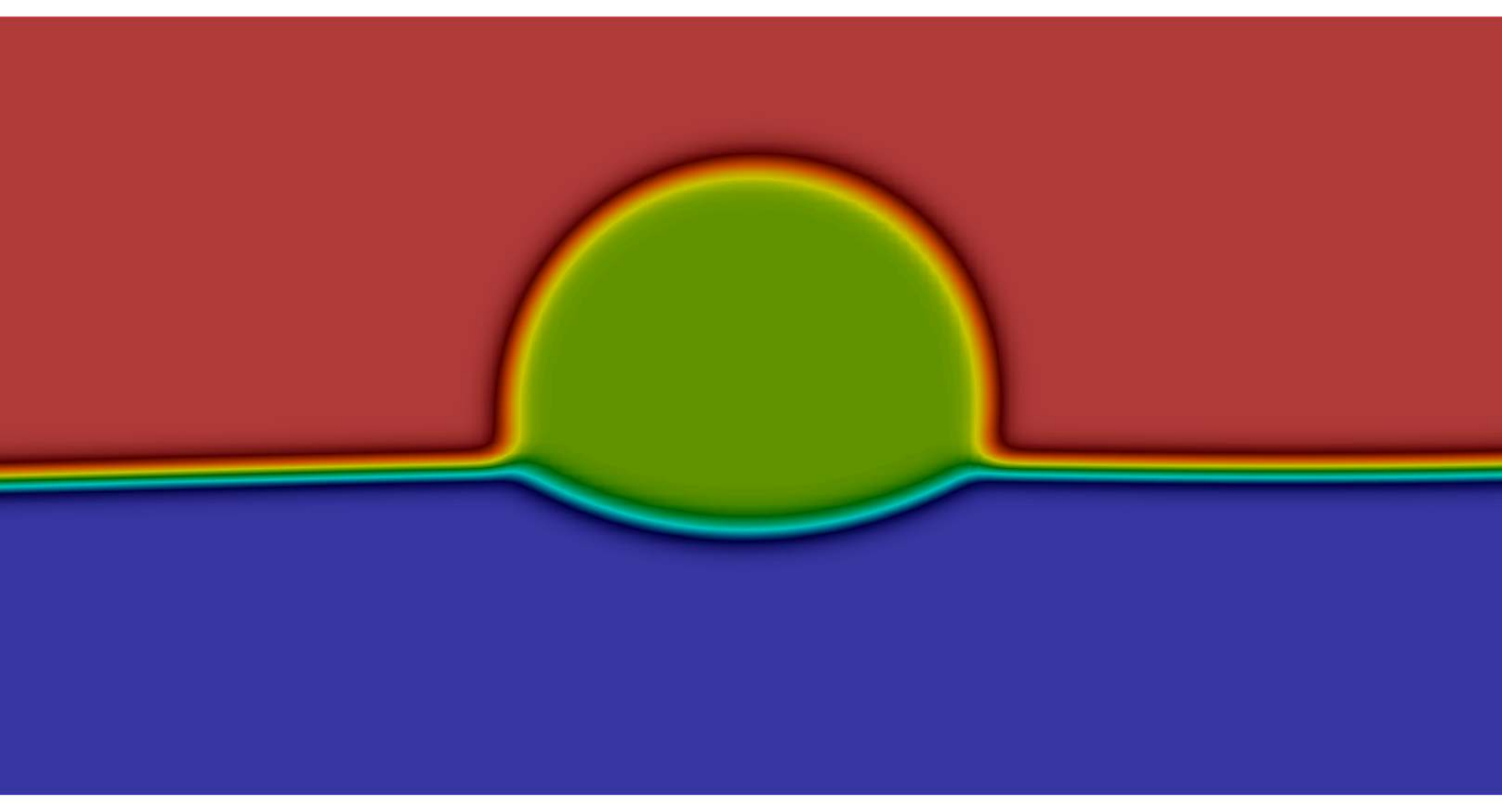}
\includegraphics[scale=0.07]{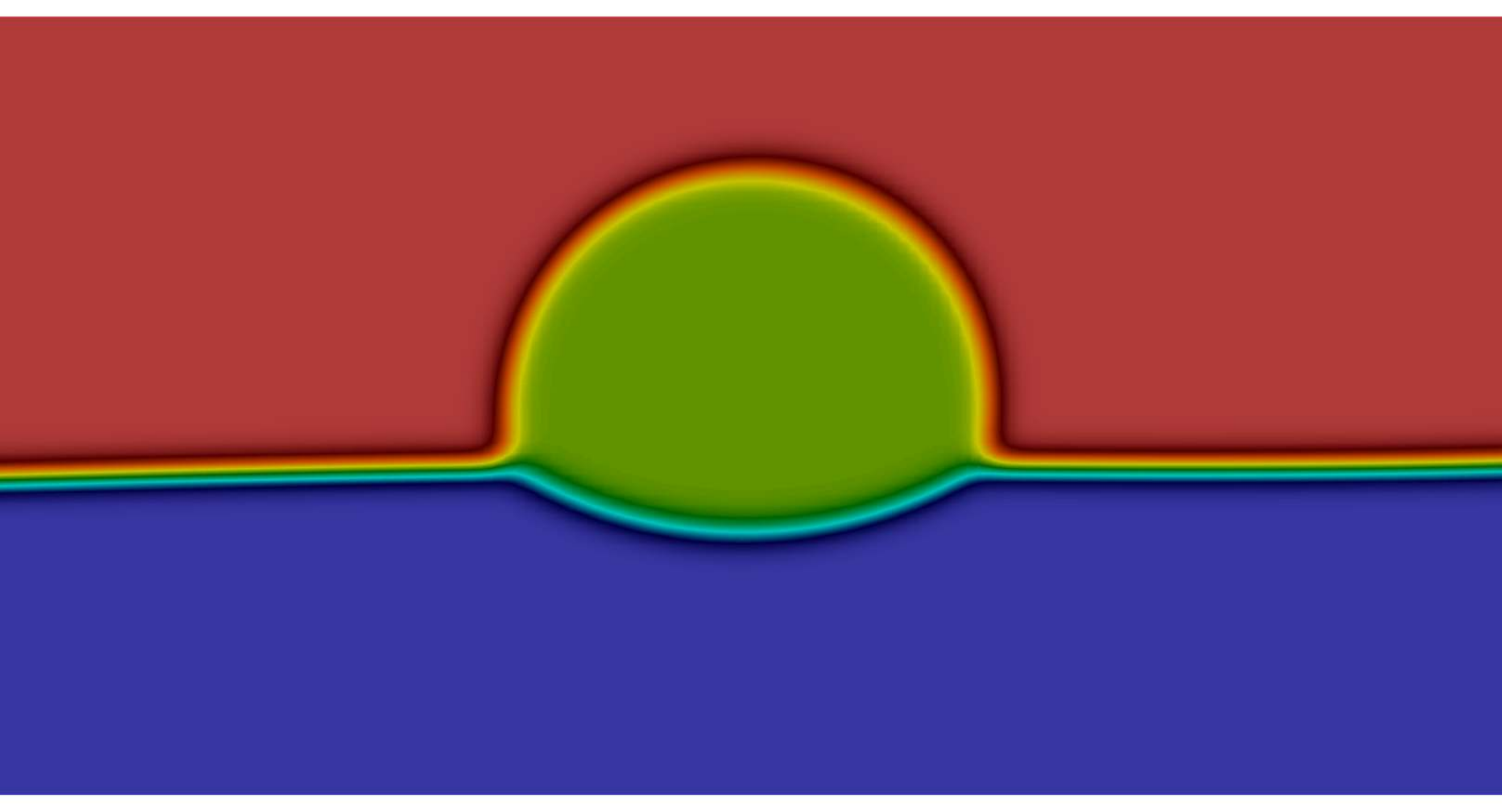}
\includegraphics[scale=0.07]{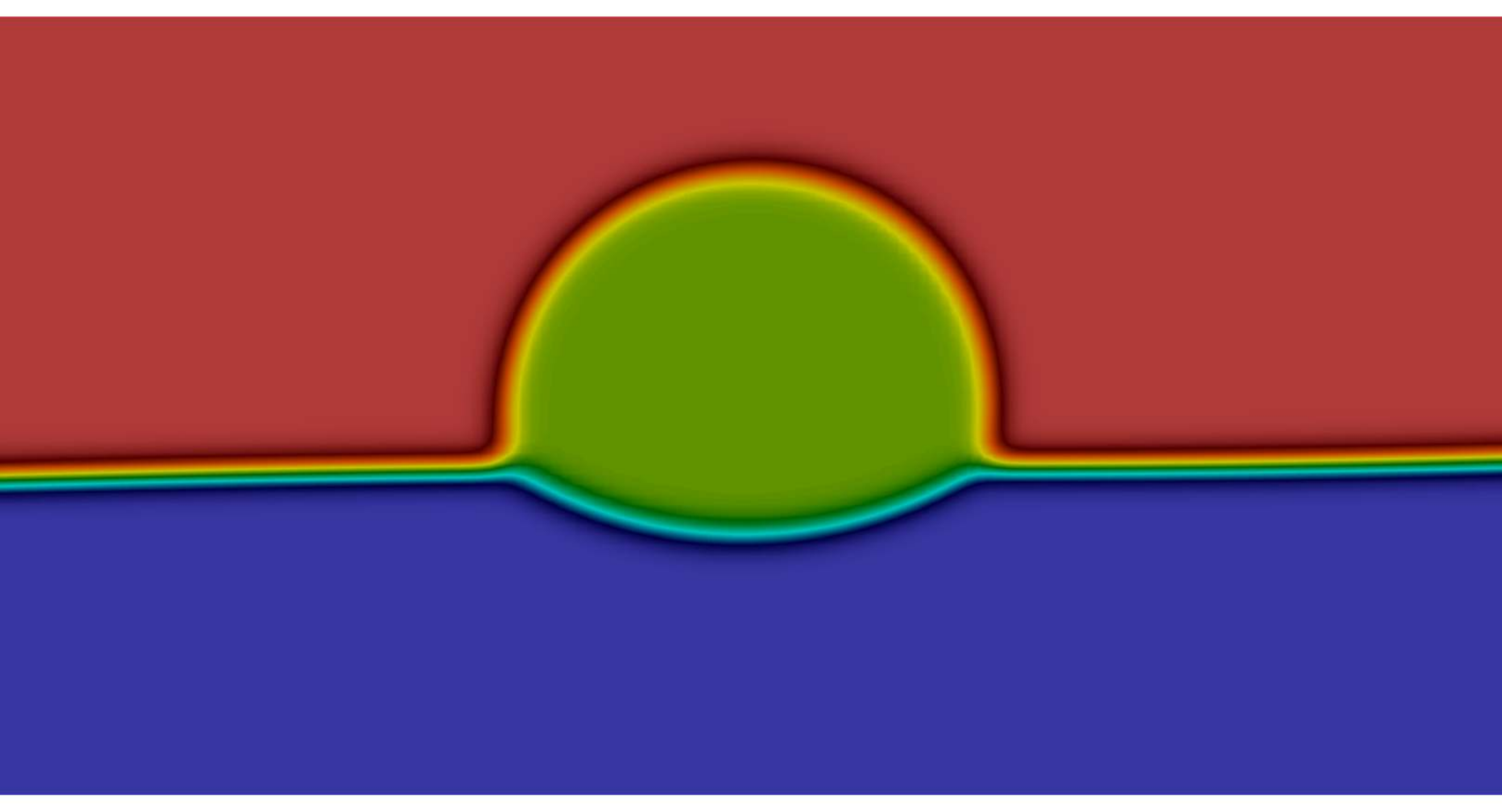}
\\
\includegraphics[scale=0.07]{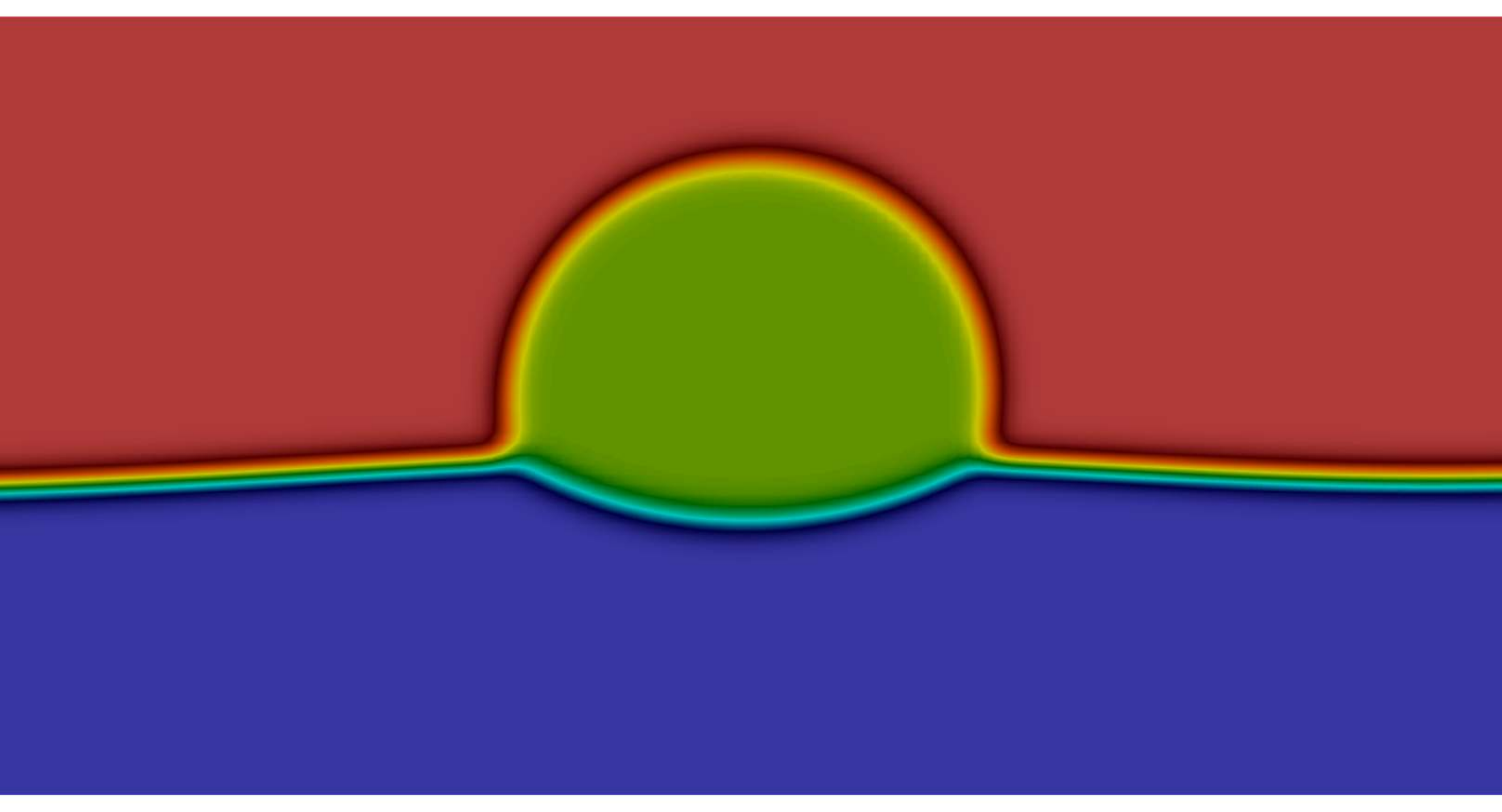}
\includegraphics[scale=0.07]{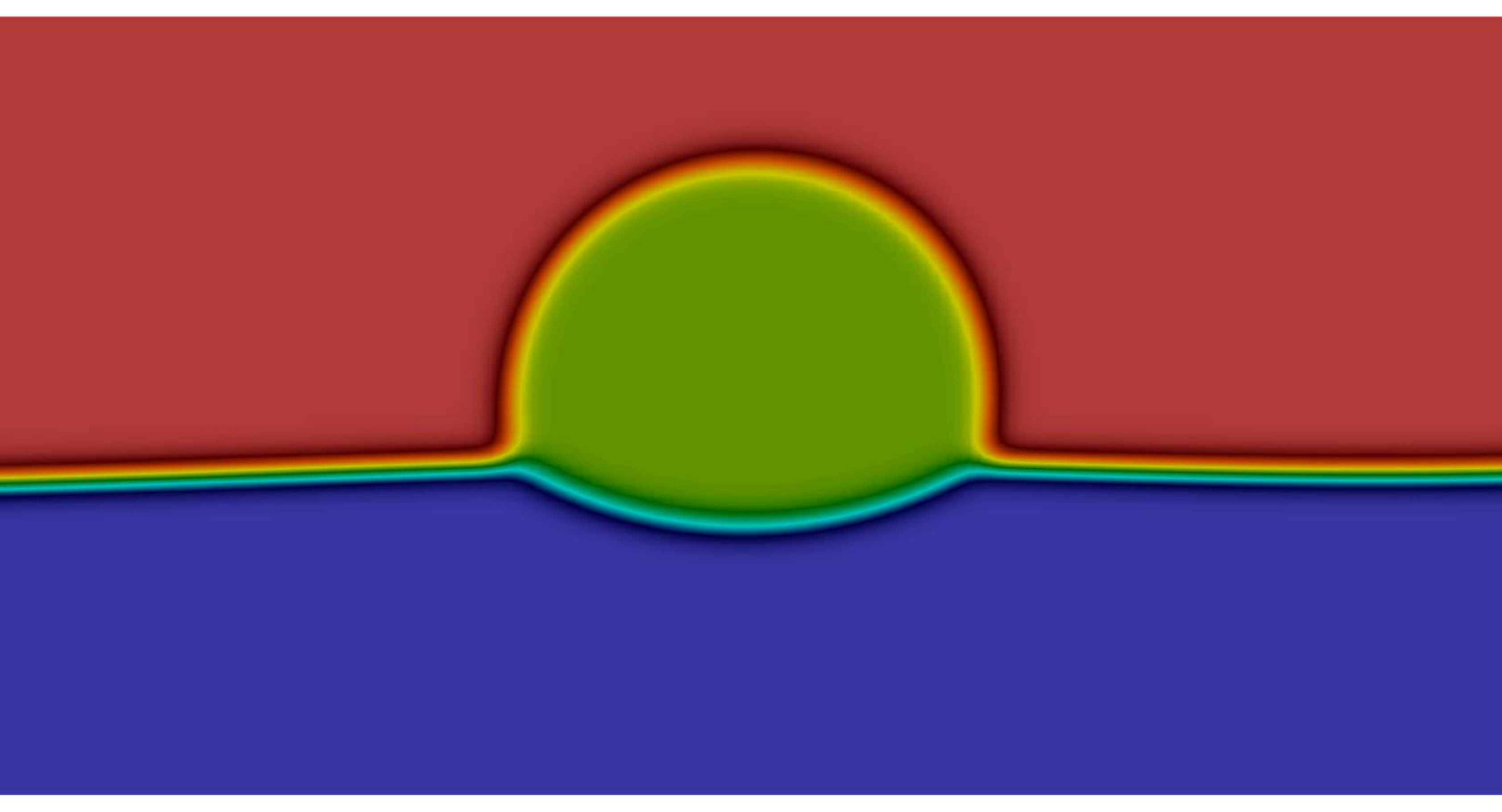}
\includegraphics[scale=0.07]{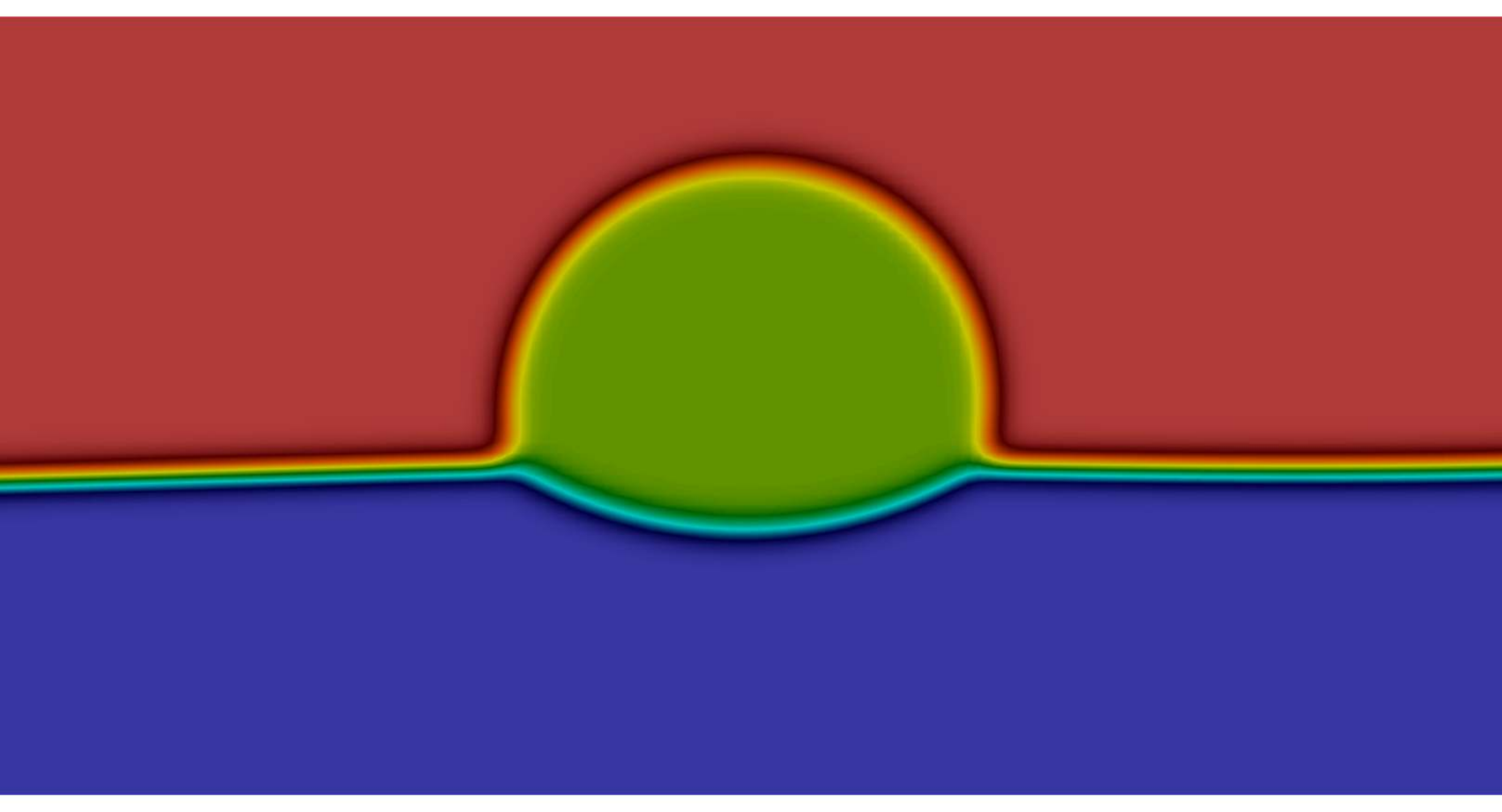}
\includegraphics[scale=0.07]{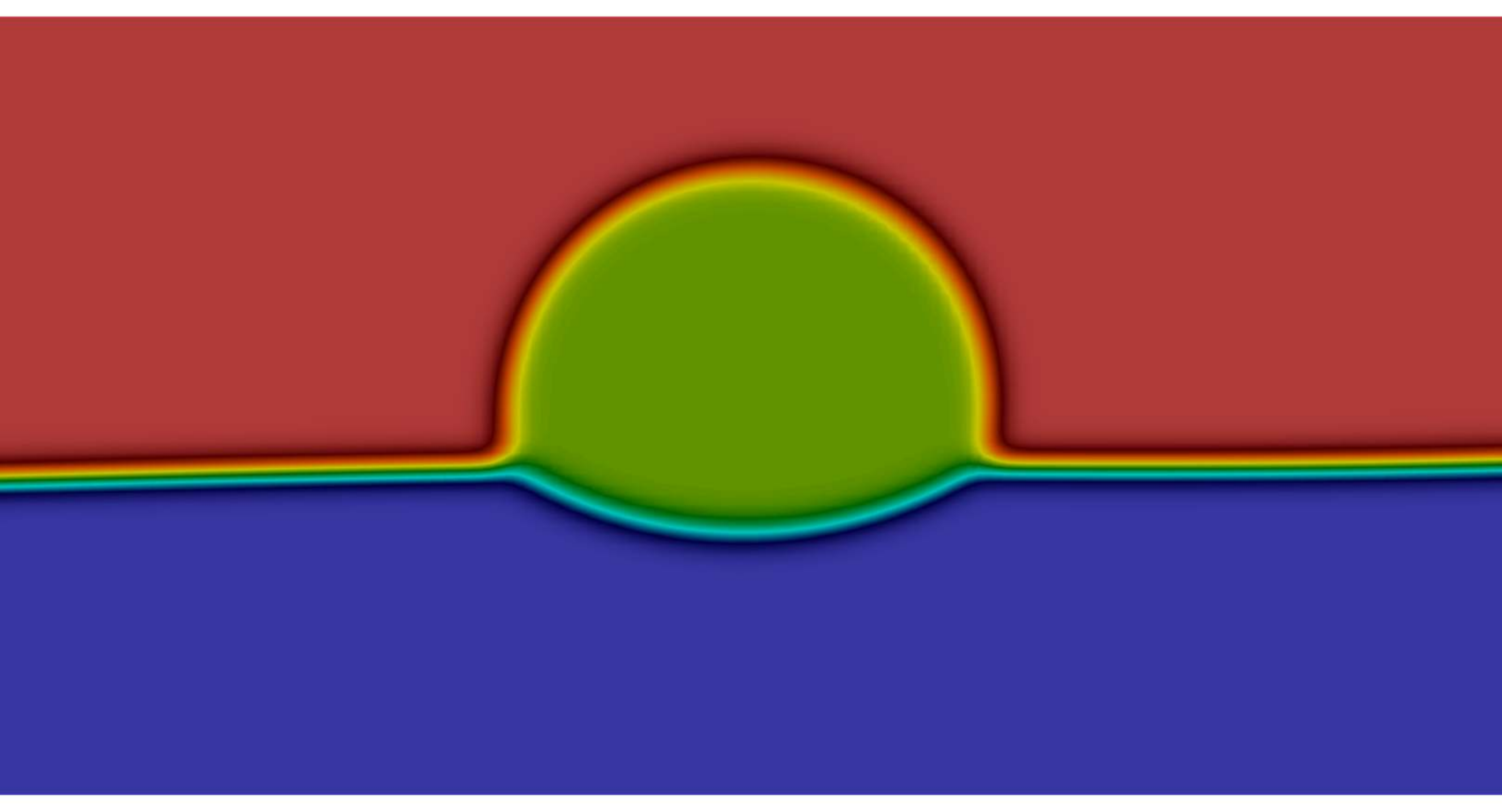}
\end{center}
\caption{Dynamics of schemes TD1 (top row), NTD1 (center row) and NTC2 (bottom row) at times $t=0.5, 1, 1.5$ and $2.5$ (from left to right) with spreading coefficients $(\Sigma_1, \Sigma_2 , \Sigma_3) = (0.4,1.6,1.2)$}
\label{fig:lensPartialDynamics2}
\end{figure}

\begin{figure}[h]
\begin{center}
\includegraphics[scale=0.11]{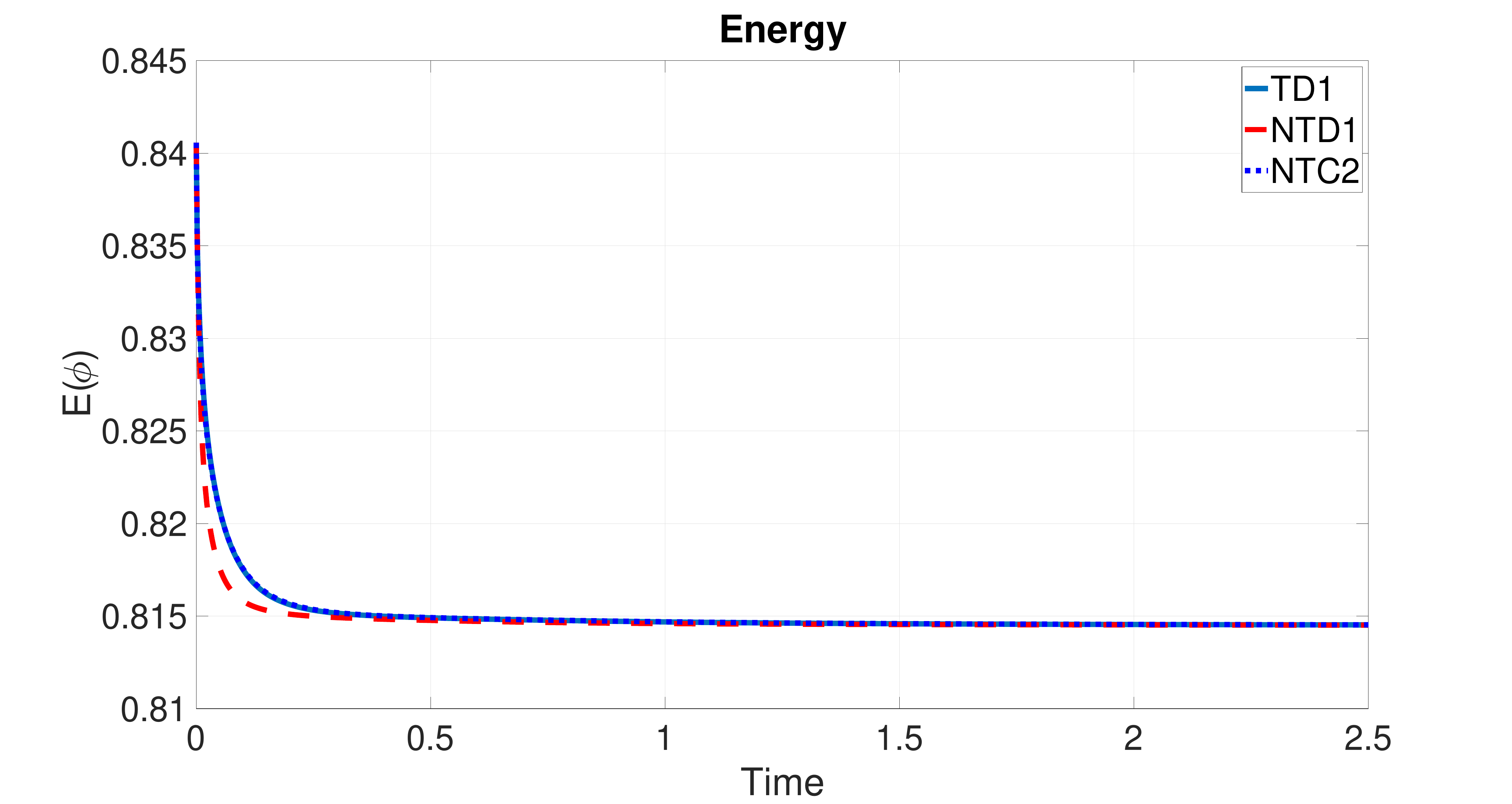}
\includegraphics[scale=0.11]{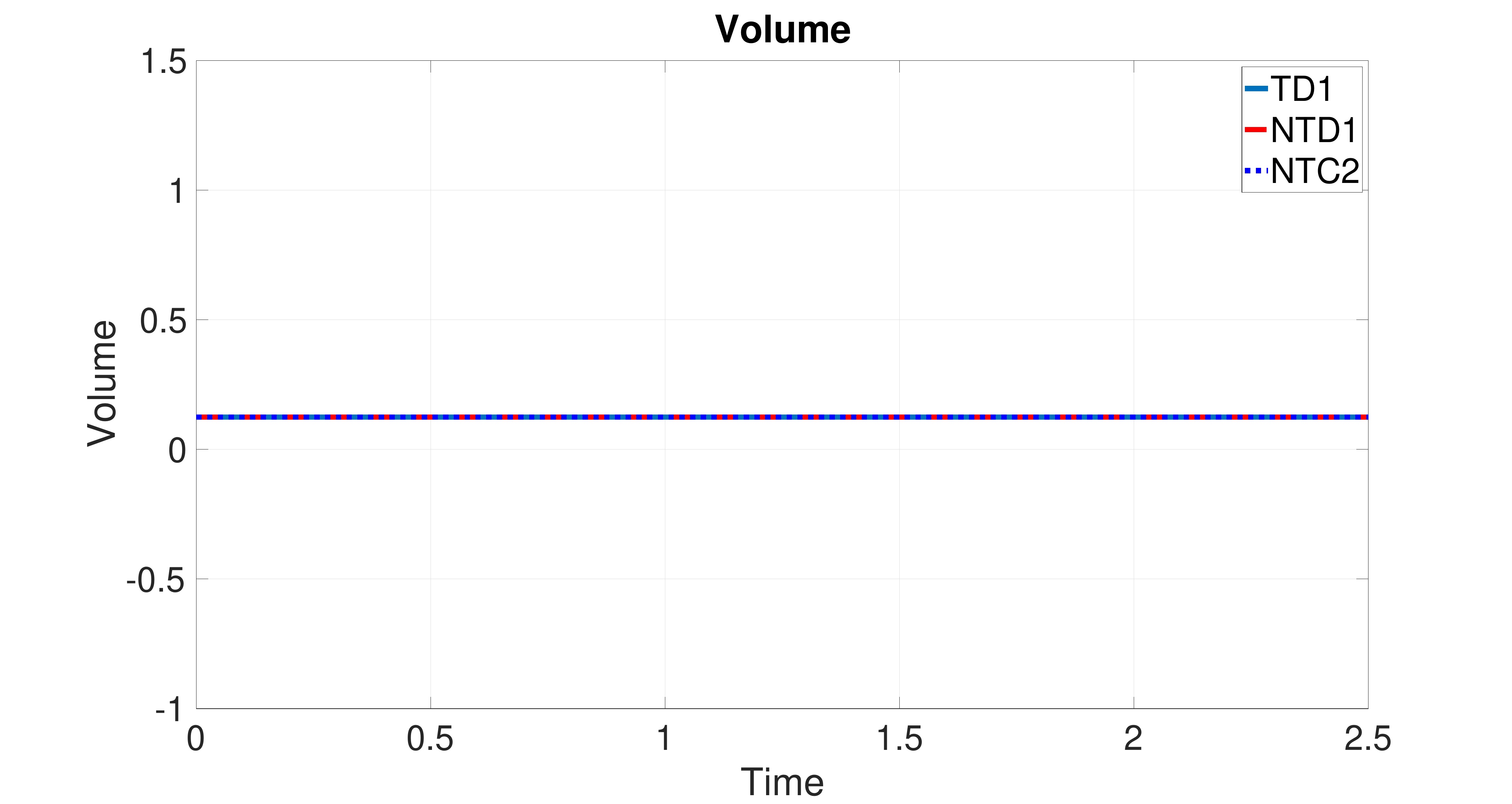}
\\ [1ex]
\includegraphics[scale=0.11]{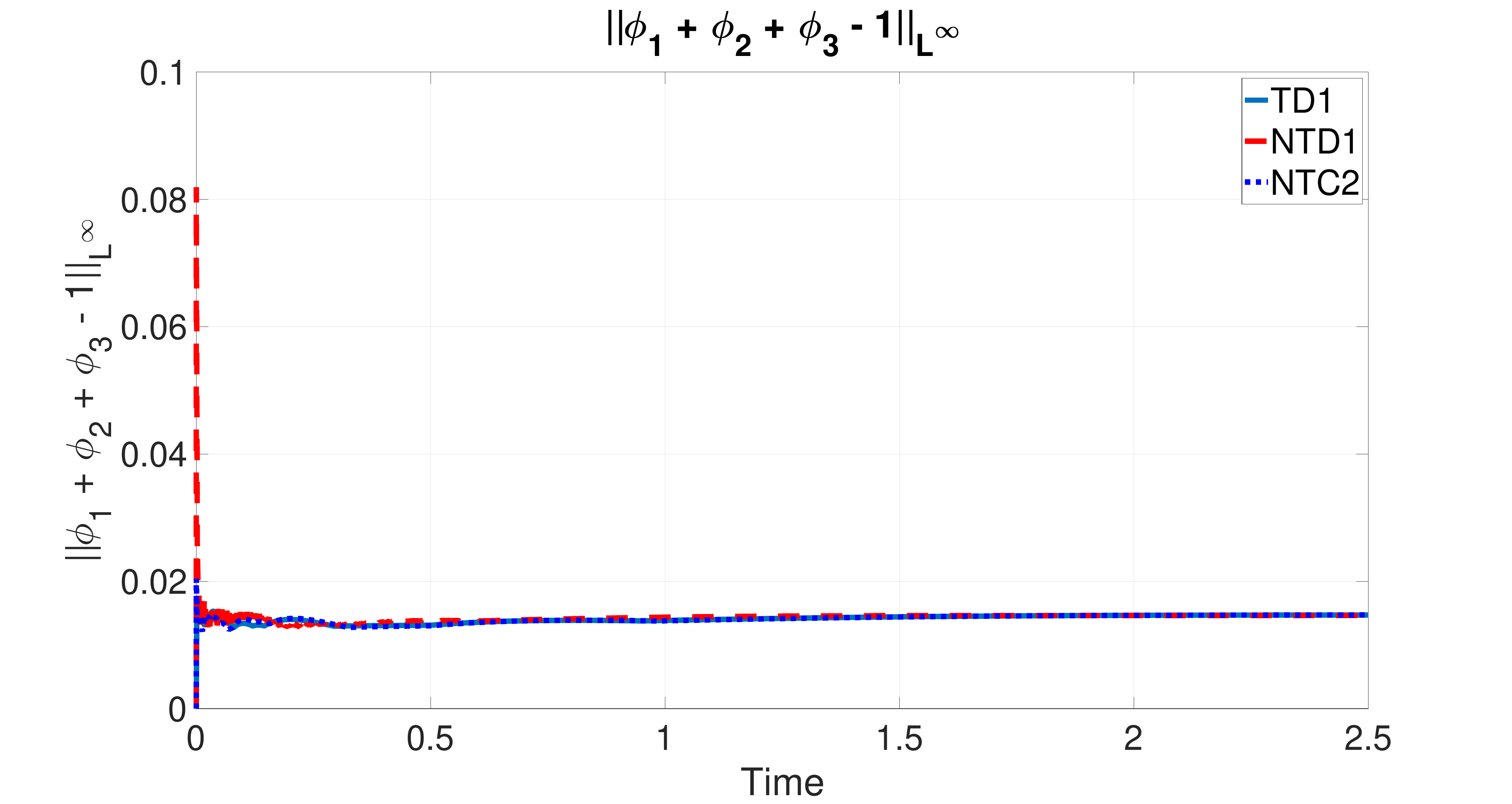}
\includegraphics[scale=0.11]{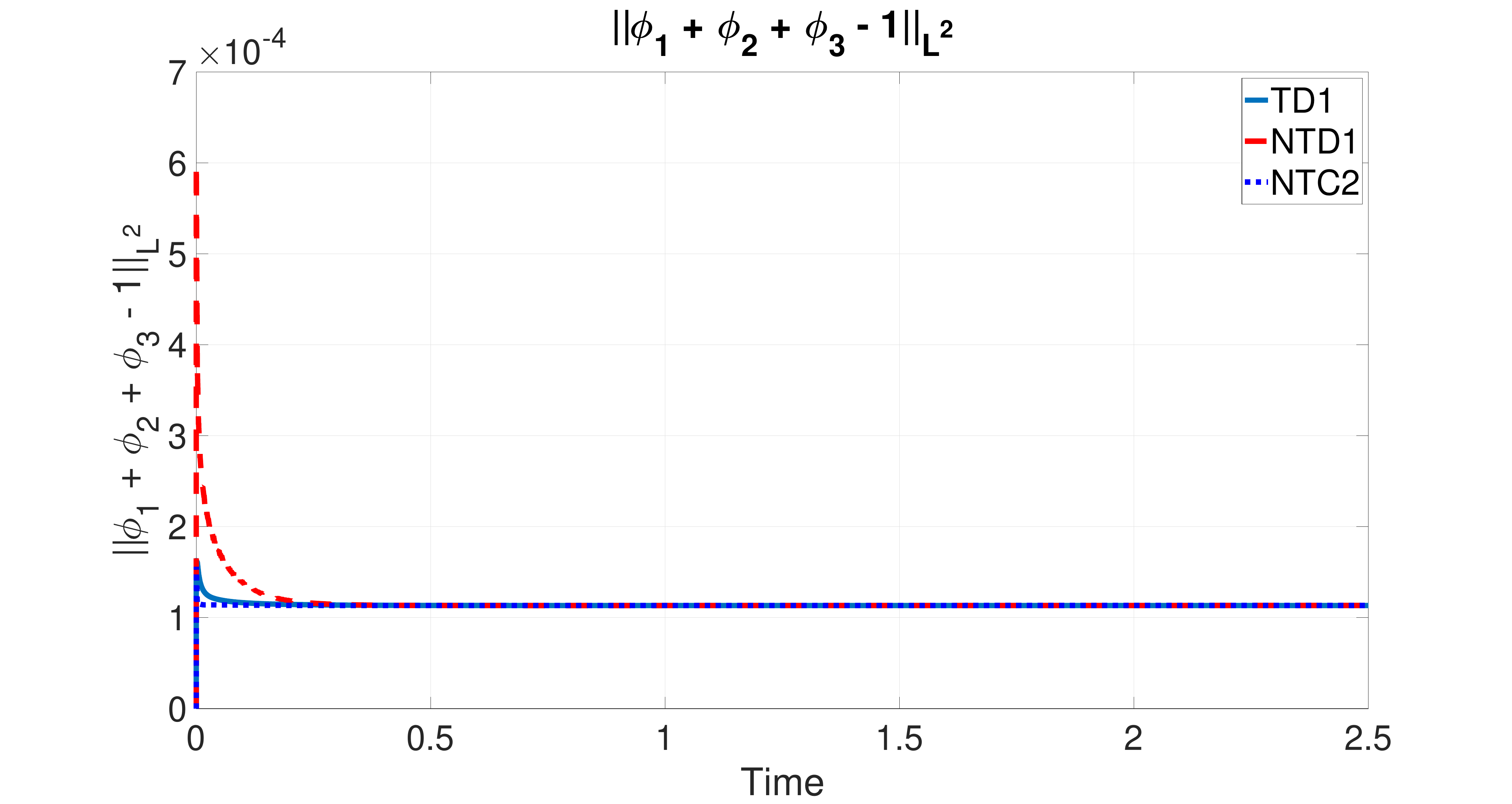}
\\ [1ex]
\includegraphics[scale=0.11]{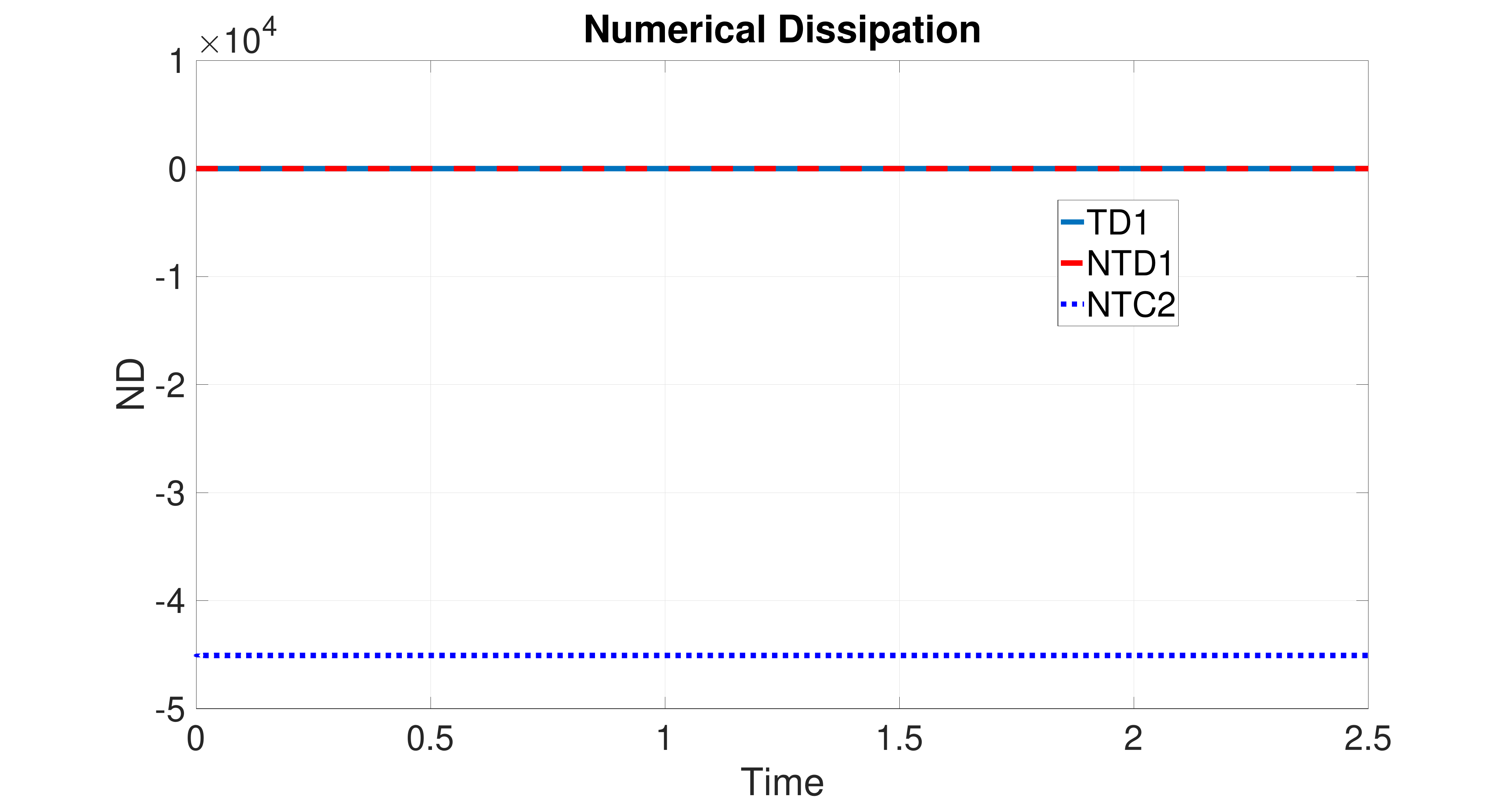}
\includegraphics[scale=0.11]{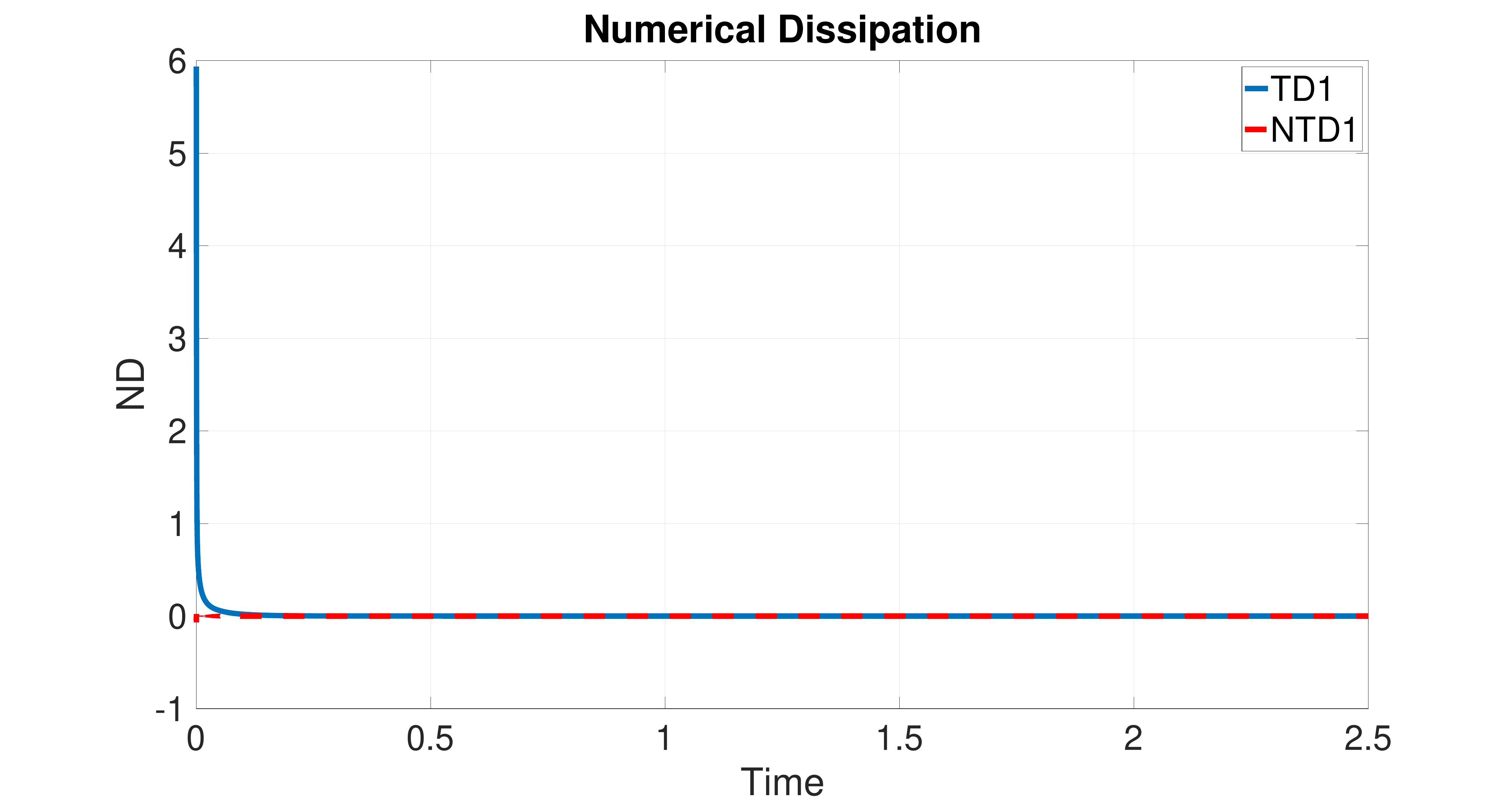}
\end{center}
\caption{Evolution in time of the energies (top left), the volume $\int_\Omega(\phi_1+\phi_2+\phi_3)$ (top right), $\|\phi_1 + \phi_2 + \phi_3 -1\|_{L^\infty}$ (center left), $\|\phi_1 + \phi_2 + \phi_3 -1\|_{L^2}$ (center right) and the evolution of the numerical dissipation (bottom row) with spreading coefficients $(\Sigma_1, \Sigma_2 , \Sigma_3) = (0.4,1.6,1.2)$.}
\label{fig:lensPartialPlots2}
\end{figure}

\subsubsection{Total Spreading}
{
Now we focus on two examples of total spreading $(\Sigma_1, \Sigma_2 , \Sigma_3) = (3,3,-0.1)$, and $(\Sigma_1, \Sigma_2 , \Sigma_3) = (-0.1,3,3)$, where in the second case we have changed the domain to $\Omega=[-0.25,0.25]\times[-0.1,0.15]$ to be sure that the interface doesn't touch the boundary. The dynamics can be seen in Figures~\ref{fig:lensTotalDynamics1} and \ref{fig:lensTotalDynamics2}, respectively. In these experiments the effect of negative spreading coefficients can be seen where the phase has the tendency to spread between the other two.  In Figure~\ref{fig:lensTotalDynamics1}, we assign $\phi_3$ (green phase) the negative spreading coefficient and we observe that the lens spreads out between the top and bottom layers, producing two flat interfaces. On the other hand, in Figure~\ref{fig:lensTotalDynamics2}, $\phi_1$ (red phase) spreads between the green and blue phase which causes the lens to drift upwards. In these two examples it is clear that all schemes are capturing the same dynamics but both schemes TD1 and NTC2 are doing it slower. The explanation can be deduced from Figures~\ref{fig:lensTotalPlots1} and \ref{fig:lensTotalPlots2}, where we see that the numerical dissipation is not as close to zero as for NTD1, producing that the decay of the energy occurs in a much slower way. This effect is exactly the same one that is studied in \cite{Xu2019} for the two-components model, large amount of numerical dissipation might help stabilizing the energy, but it ends up slowing the dynamics of the system. In fact, it is surprising that NTC2 produces good results with such negative numerical dissipation, but it seems it works due to the stabilization terms $\tau_i$. In fact, running NTC2 in any of the presented situations with $\tau_i=0\, (i=1,2,3)$ will make the system crash in few iterations. Then the a priori advantage of using NTC2 (due to its second order in time) or TD1 (due to be energy stable) with respect to NTD1 is now not so clear, because these two schemes are computationally more expensive and they need smaller time steps to achieve the correct dynamics. For this reason we consider scheme NTD1 the most efficient one and the remaining simulations will be computed using this scheme.\\
Finally, the approximation of the restriction in $L^2$ and $L^\infty$ norms seems to follow similar behavior as before, that is, while there are points where the three phases interact the $L^\infty$ norm is of order $10^{-2}$, but as soon as the phases completely separates, the approximation of the restriction becomes completely satisfied in the whole domain. Moreover, the volume is observed to be constant for each scheme which confirms again that the presented numerical schemes are conservative.
}

\begin{figure}[h]
\begin{center}
\includegraphics[scale=0.07]{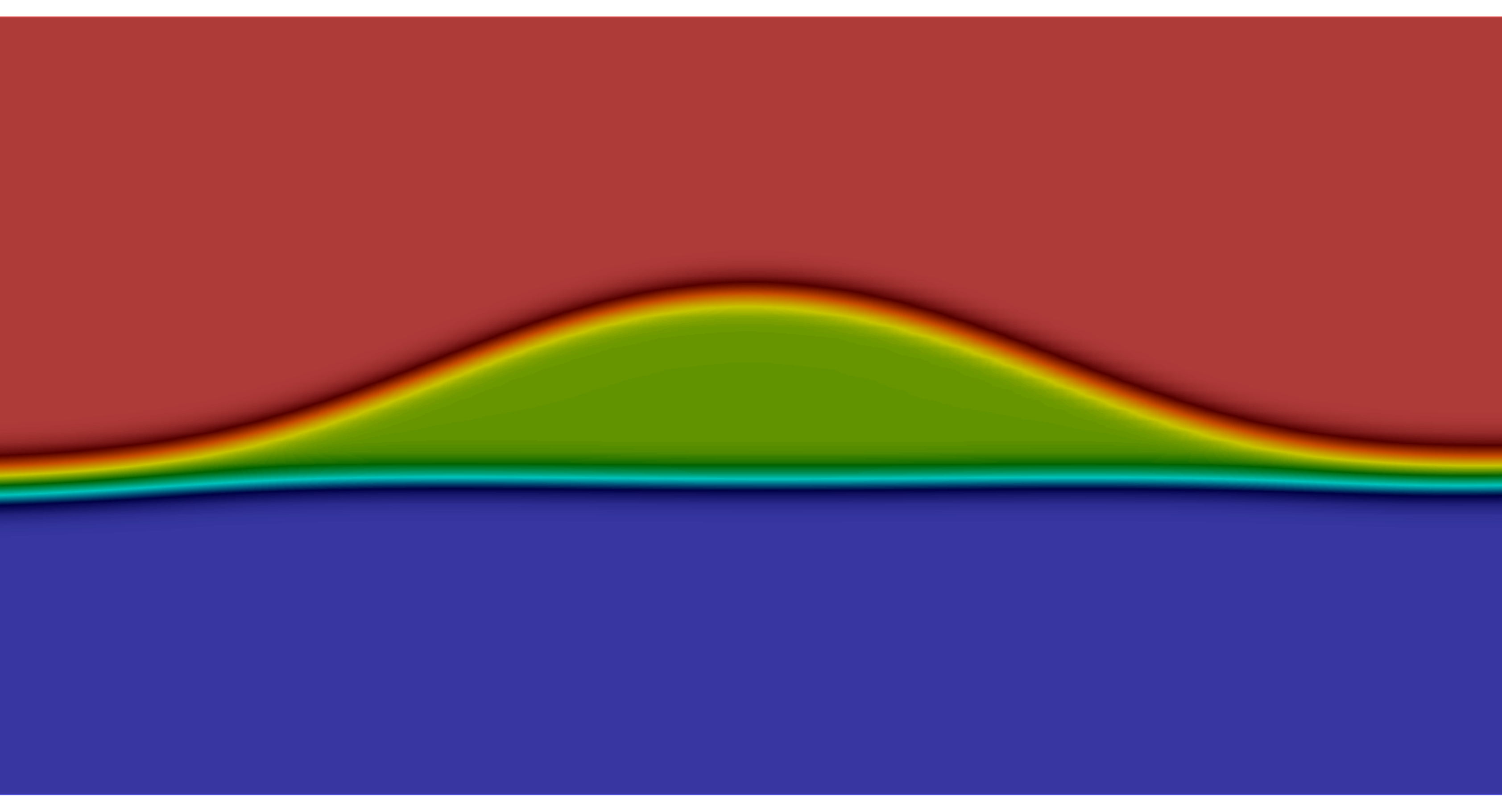}
\includegraphics[scale=0.07]{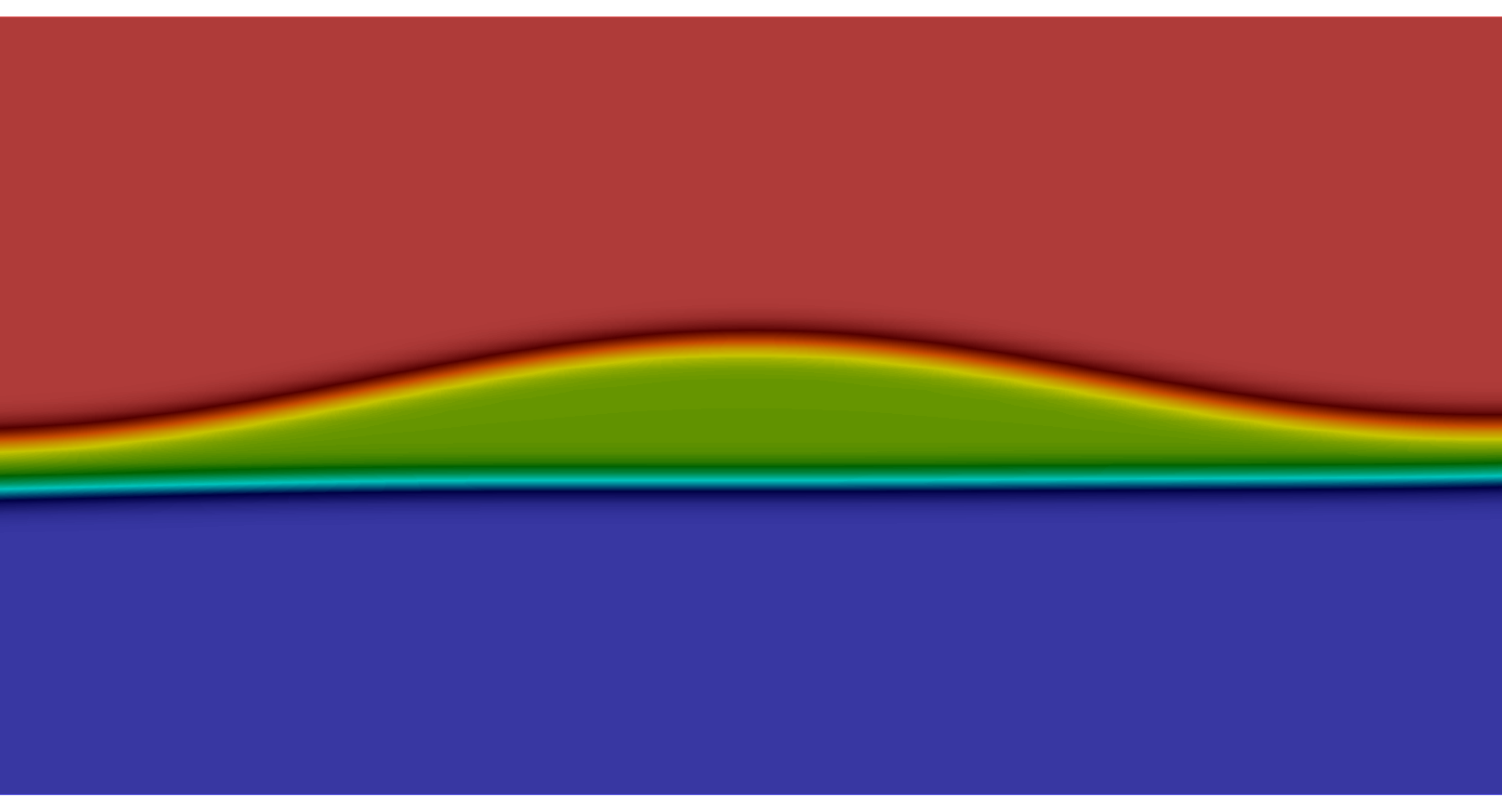}
\includegraphics[scale=0.07]{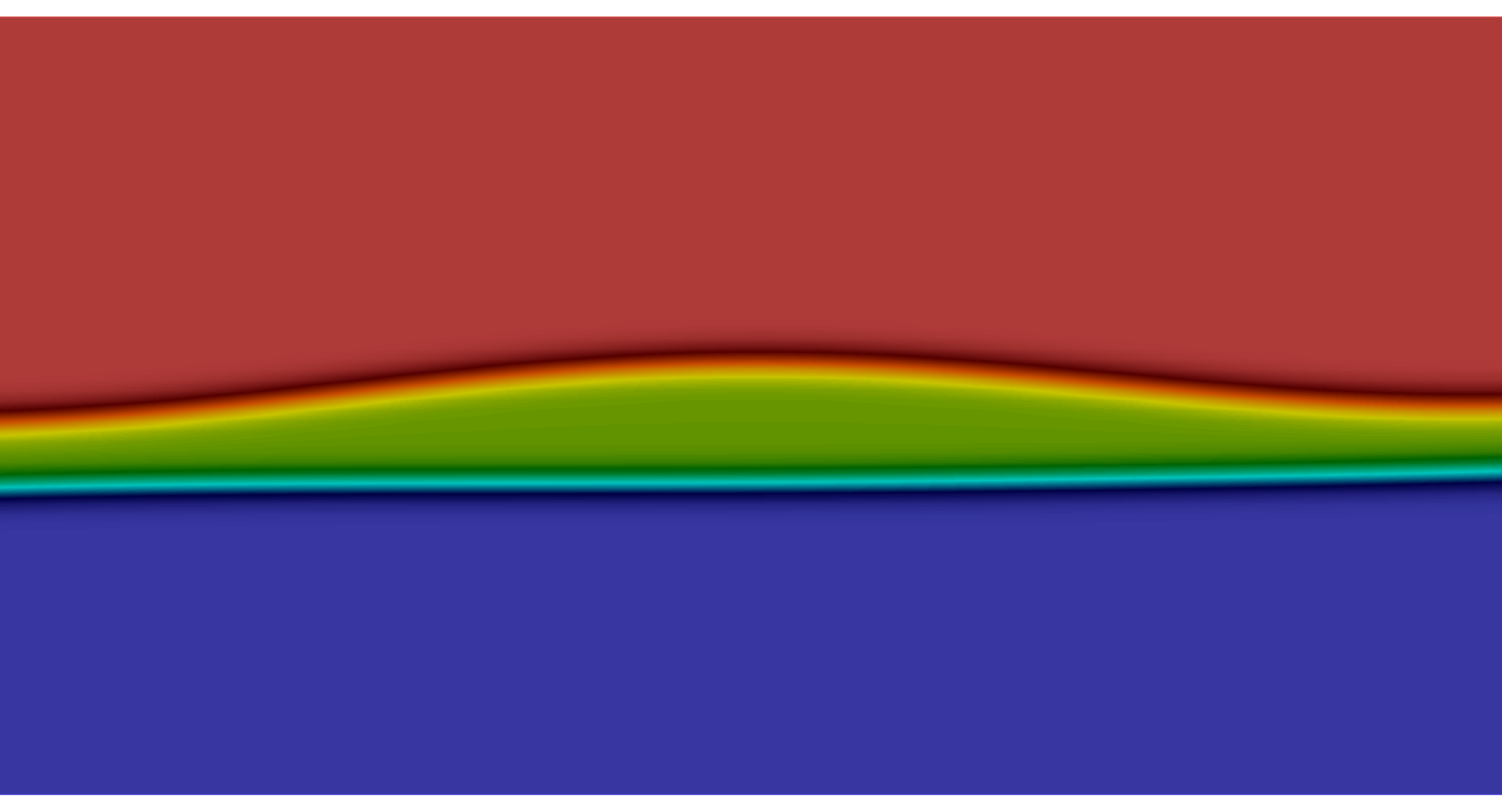}
\includegraphics[scale=0.07]{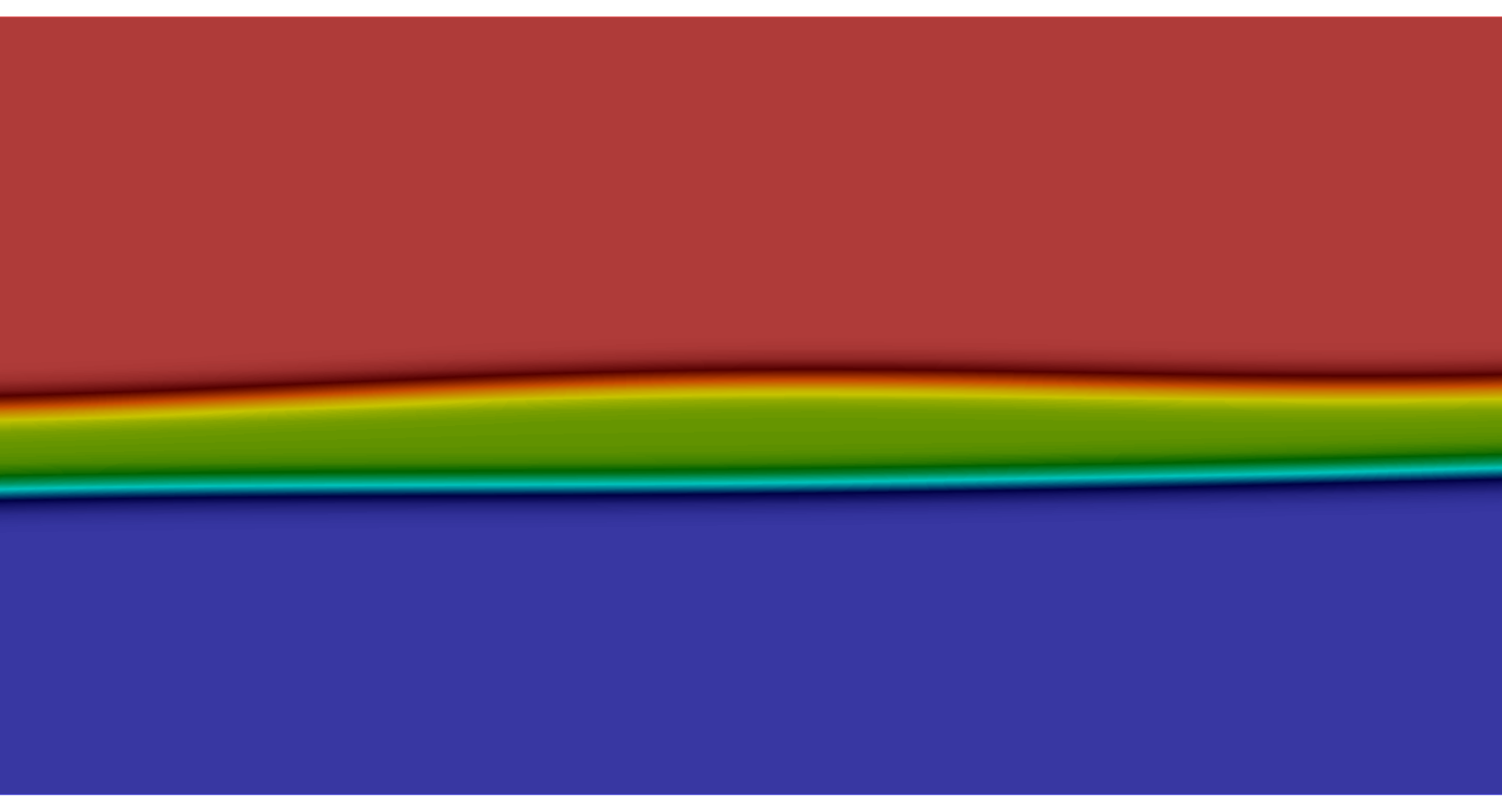}
\\
\includegraphics[scale=0.07]{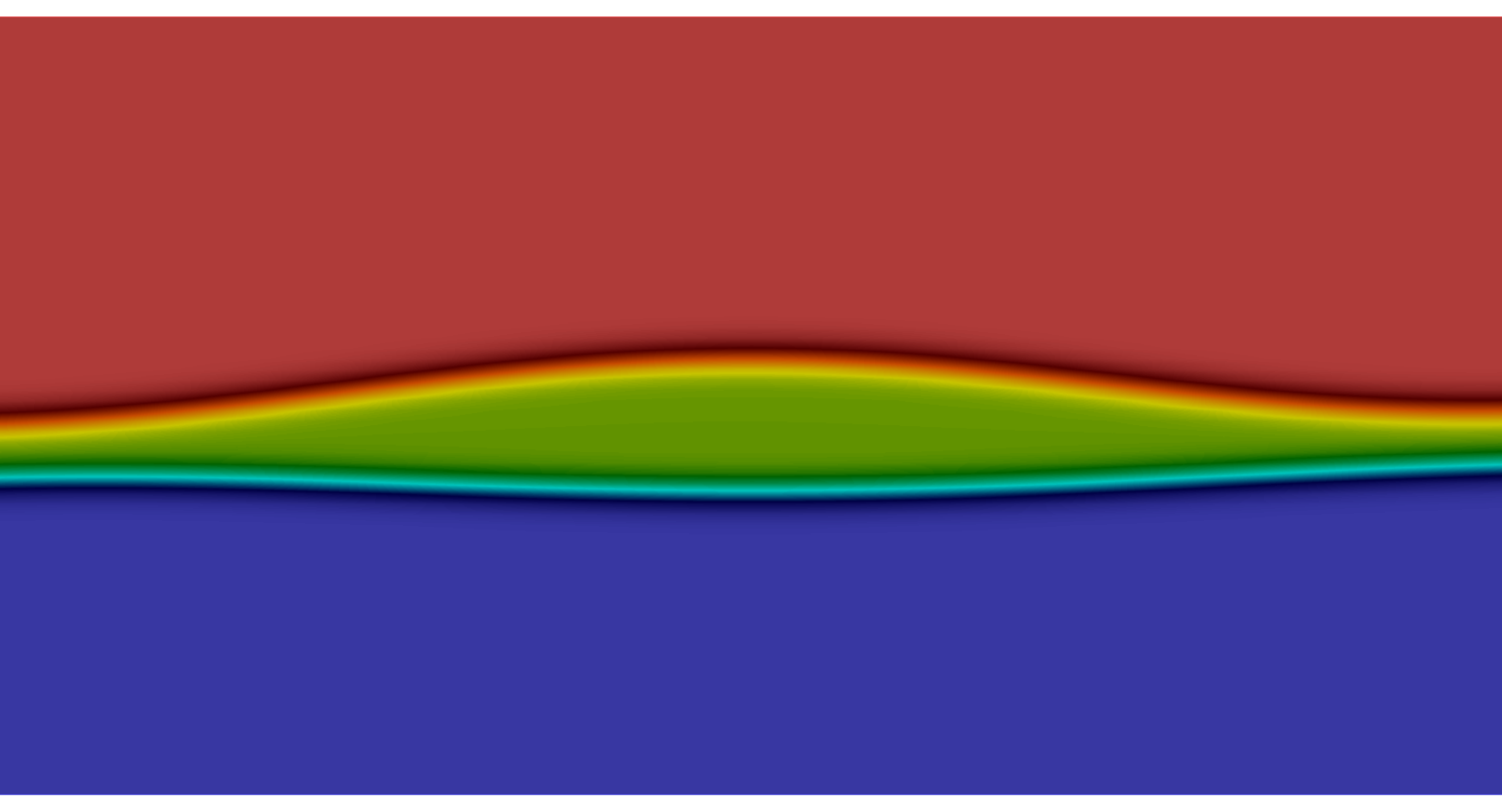}
\includegraphics[scale=0.07]{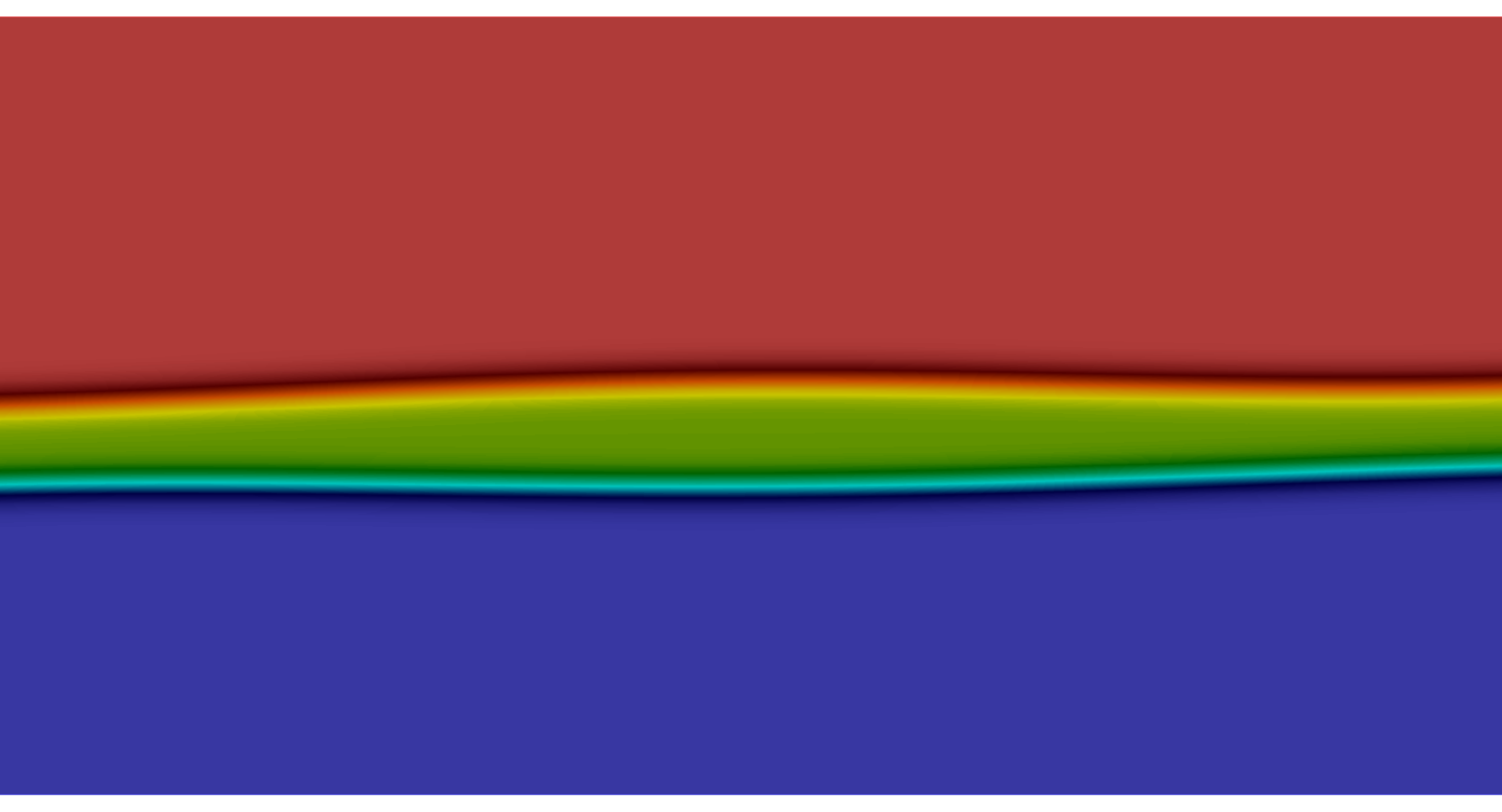}
\includegraphics[scale=0.07]{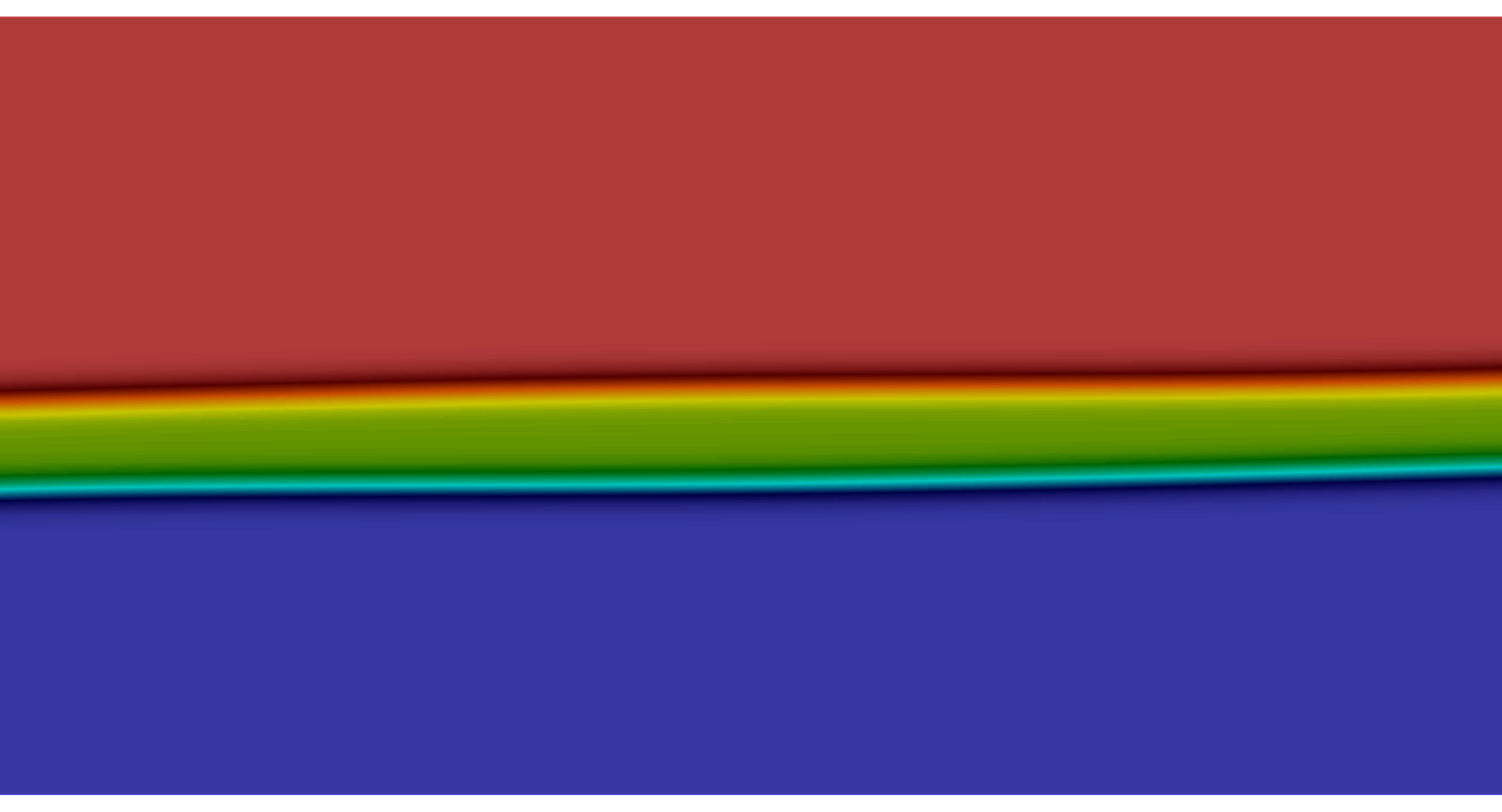}
\includegraphics[scale=0.07]{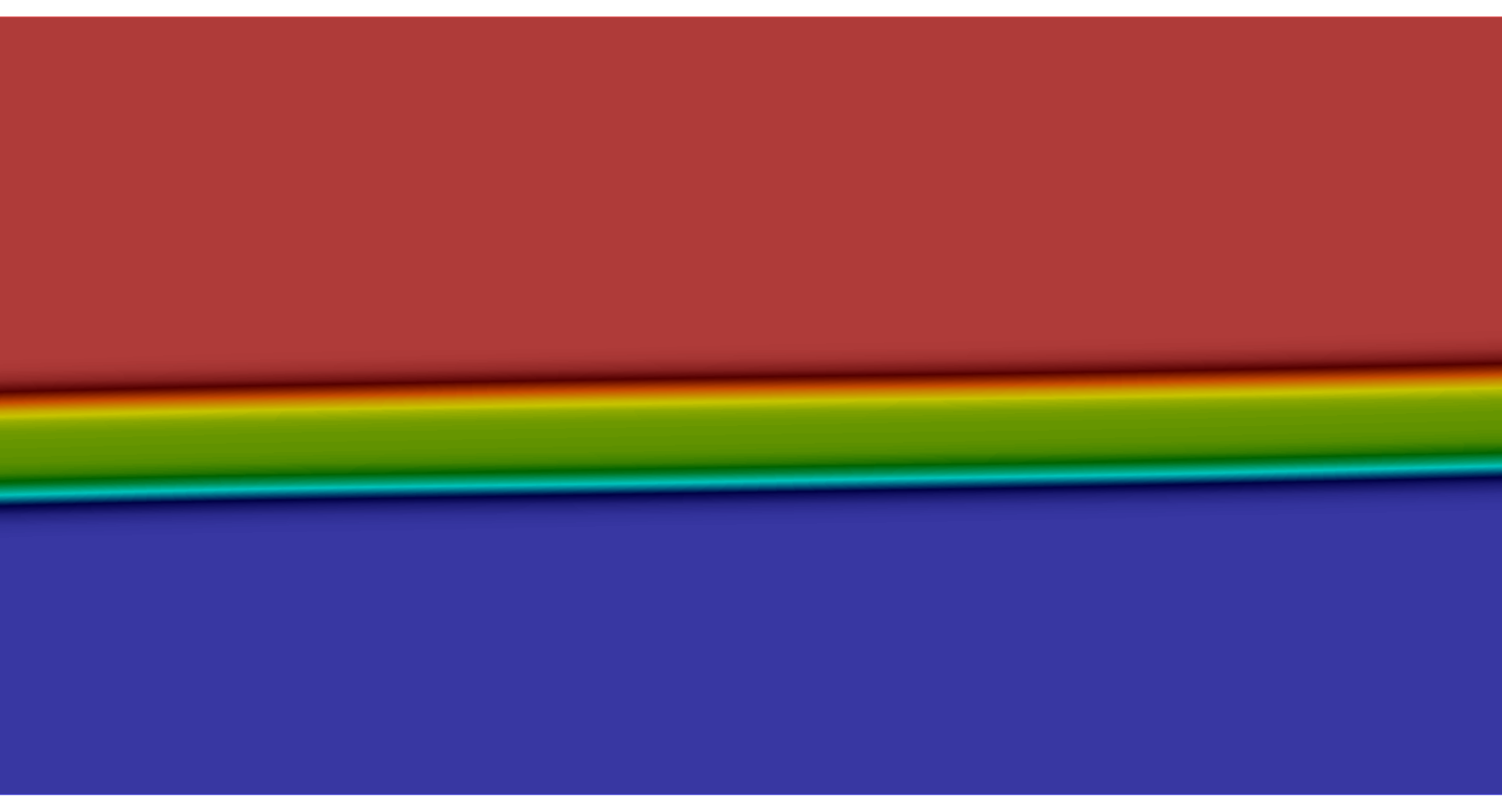}
\\
\includegraphics[scale=0.07]{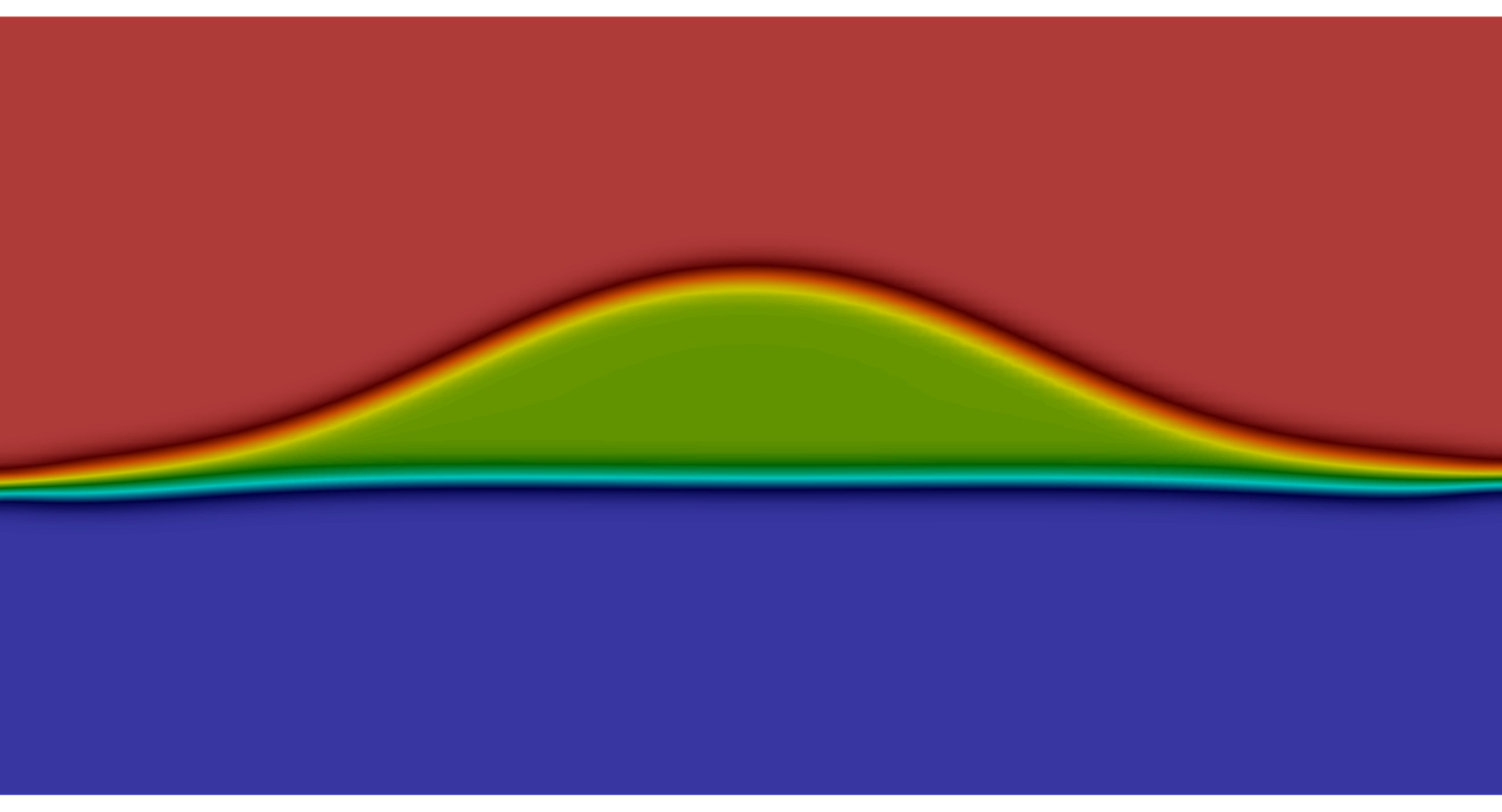}
\includegraphics[scale=0.07]{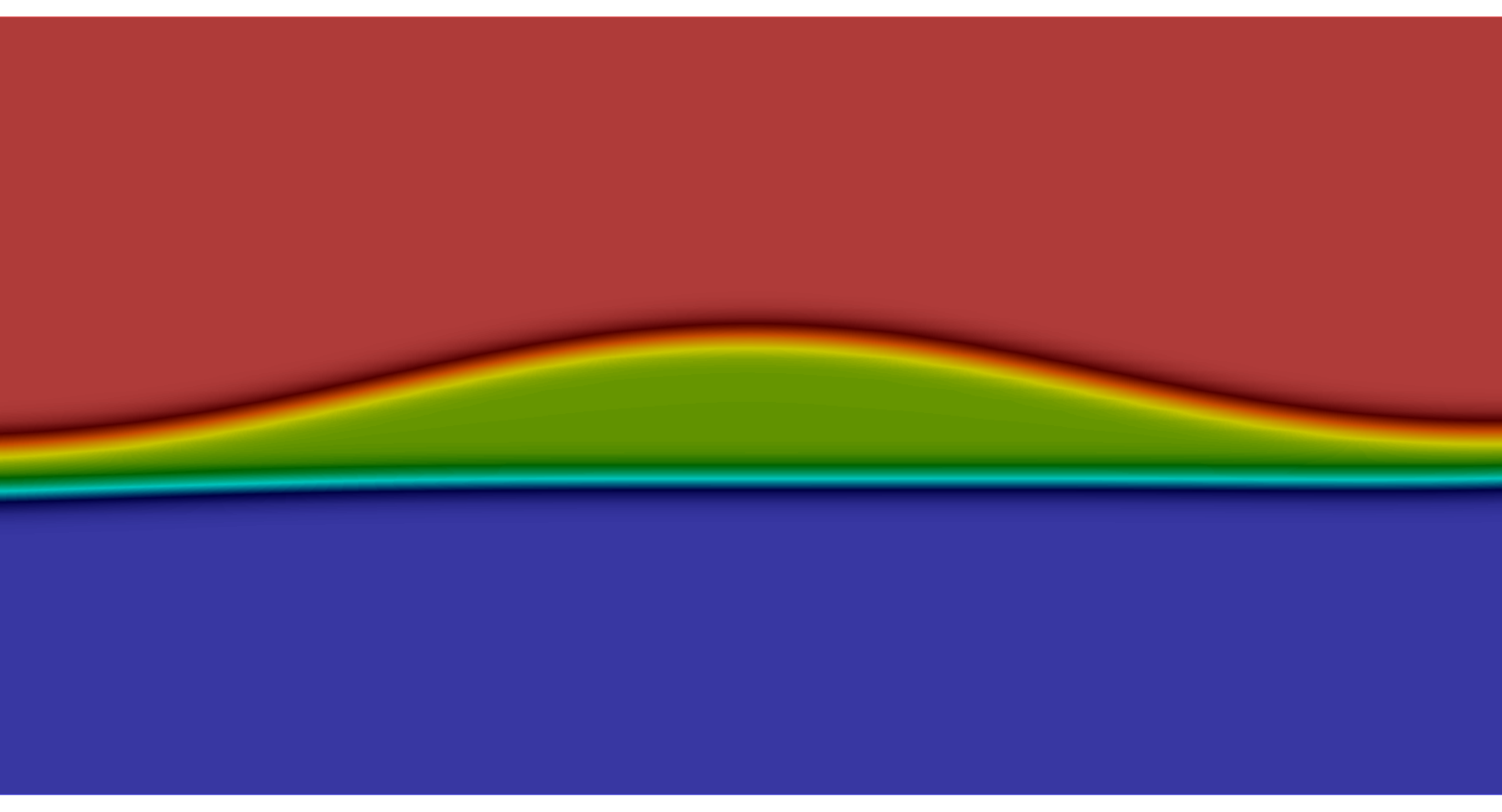}
\includegraphics[scale=0.07]{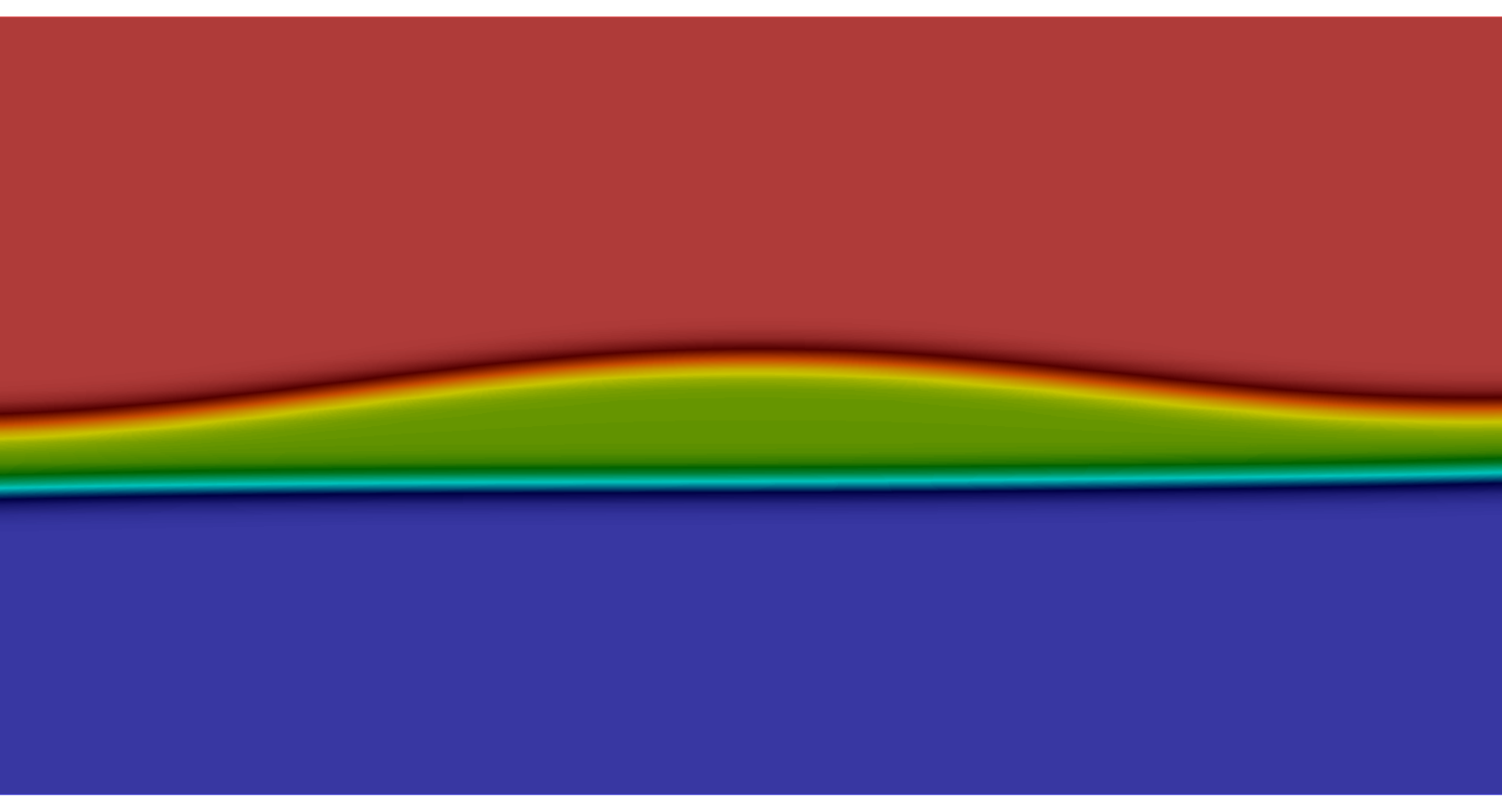}
\includegraphics[scale=0.07]{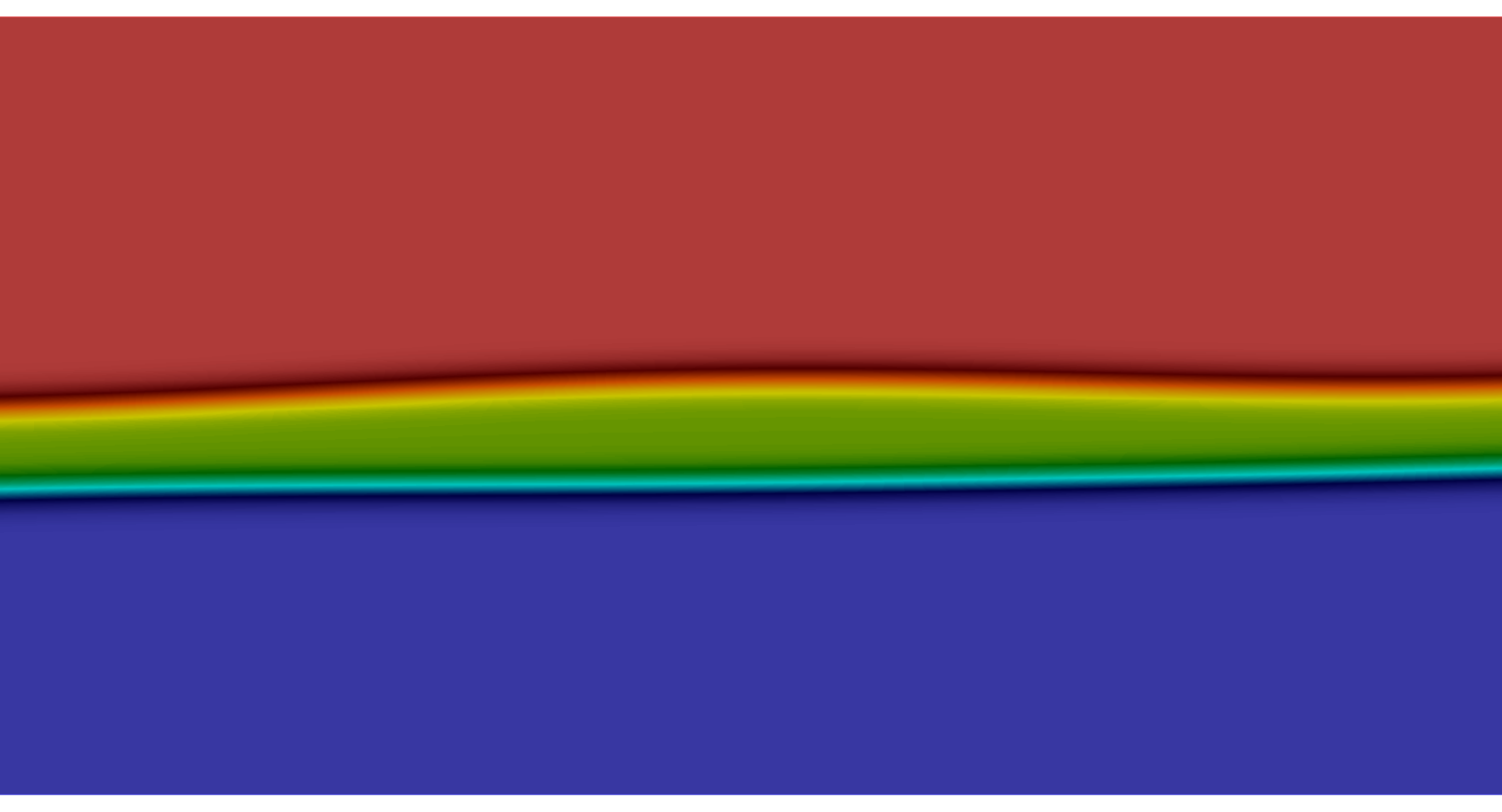}
\end{center}
\caption{Dynamics of schemes TD1 (top row), NTD1 (center row) and NTC2 (bottom row) at times $t=0.5, 1, 1.5$ and $2.5$ (from left to right) with spreading coefficients $(\Sigma_1, \Sigma_2 , \Sigma_3) = (3,3,-0.1)$.}
\label{fig:lensTotalDynamics1}
\end{figure}

\begin{figure}[h]
\begin{center}
\includegraphics[scale=0.11]{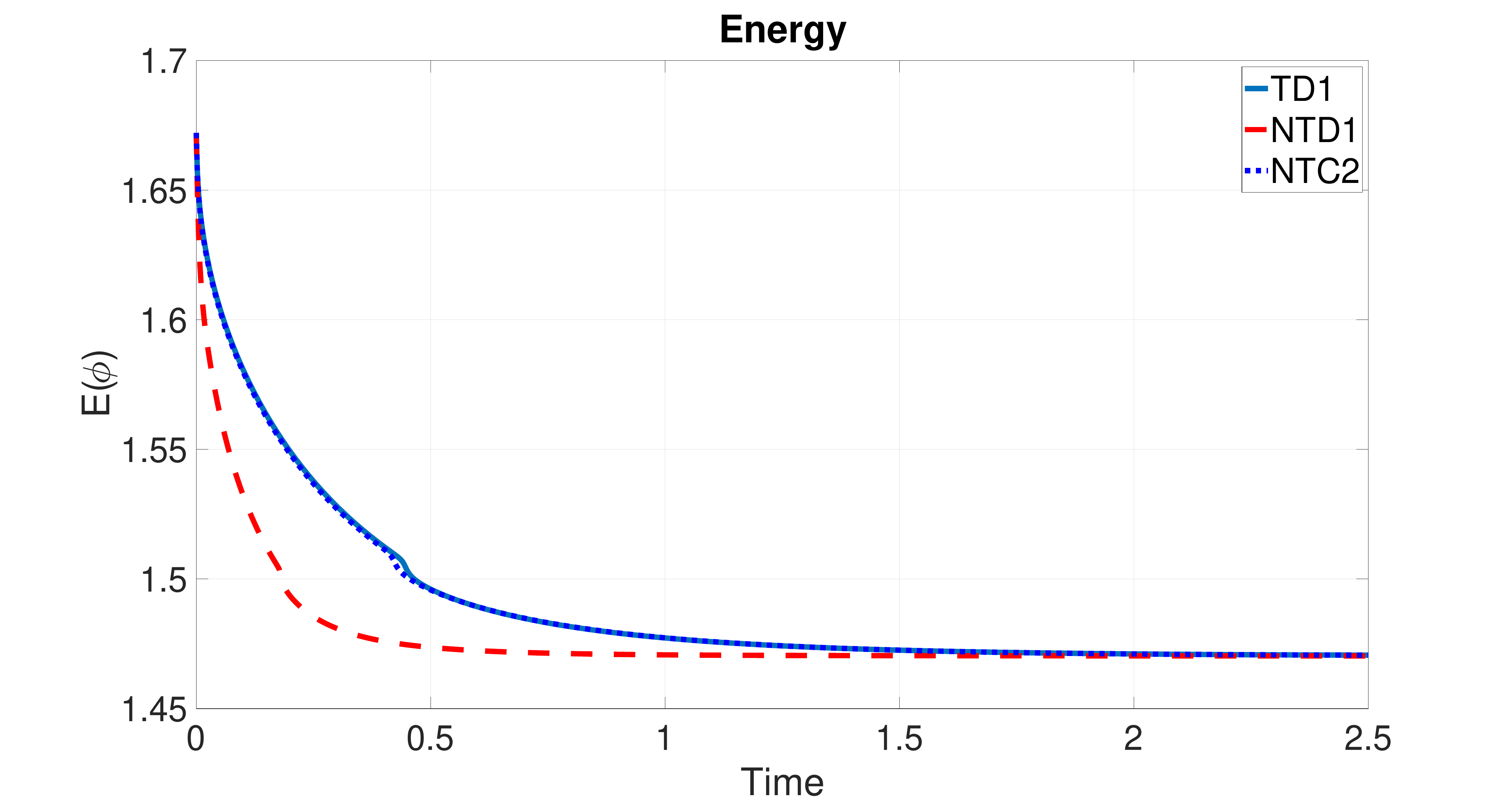}
\includegraphics[scale=0.11]{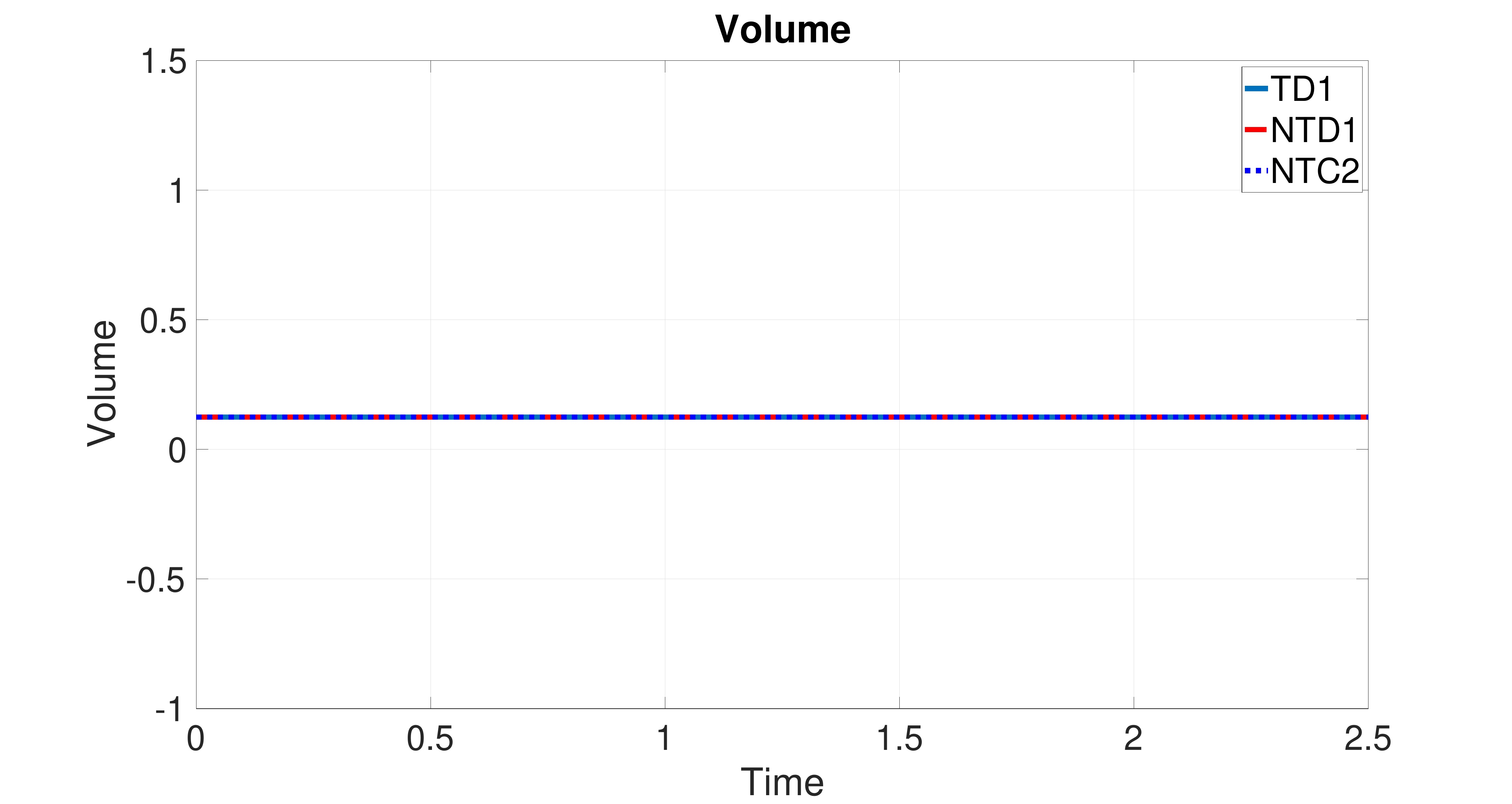}
\\ [1ex]
\includegraphics[scale=0.11]{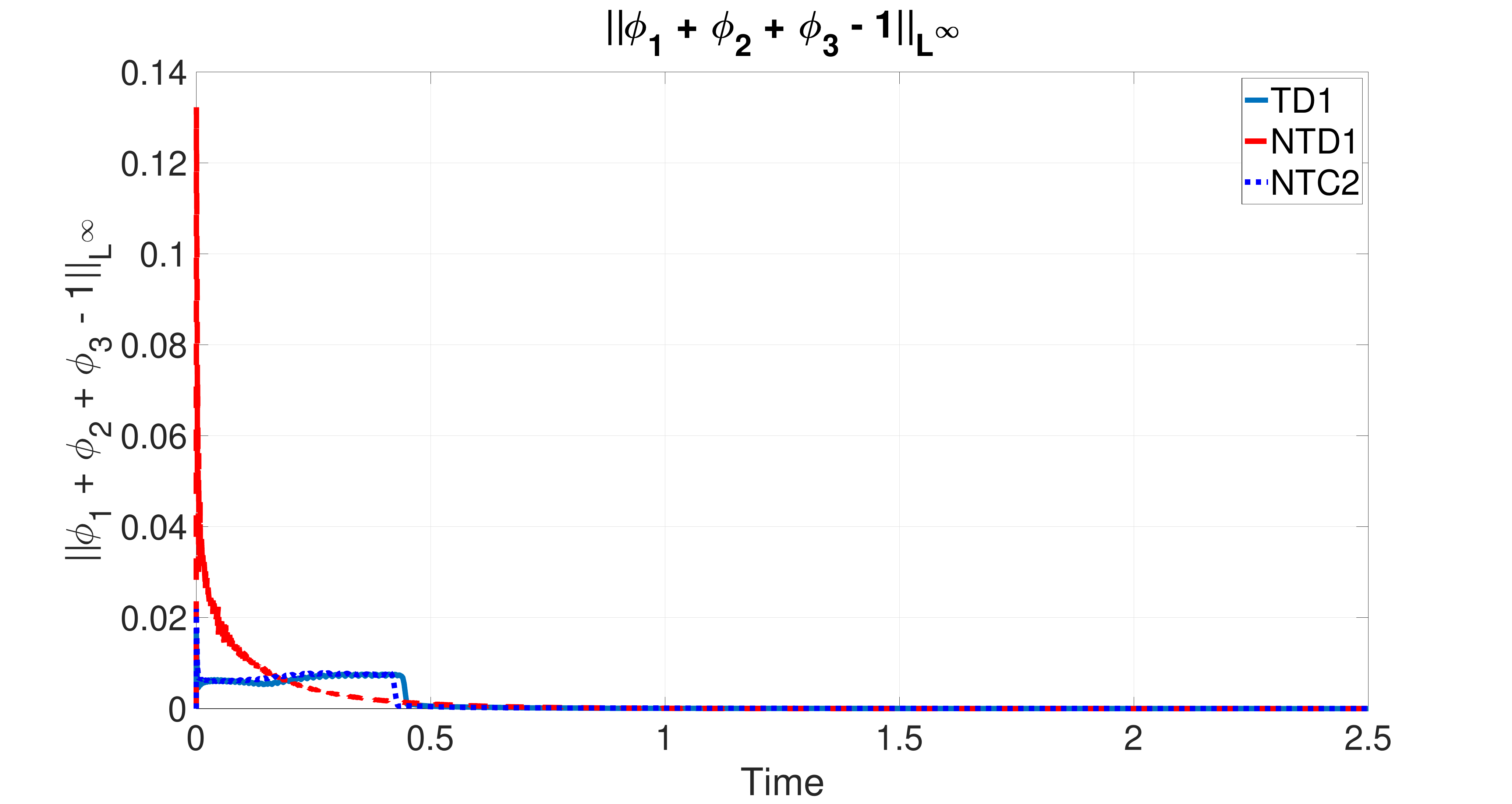}
\includegraphics[scale=0.11]{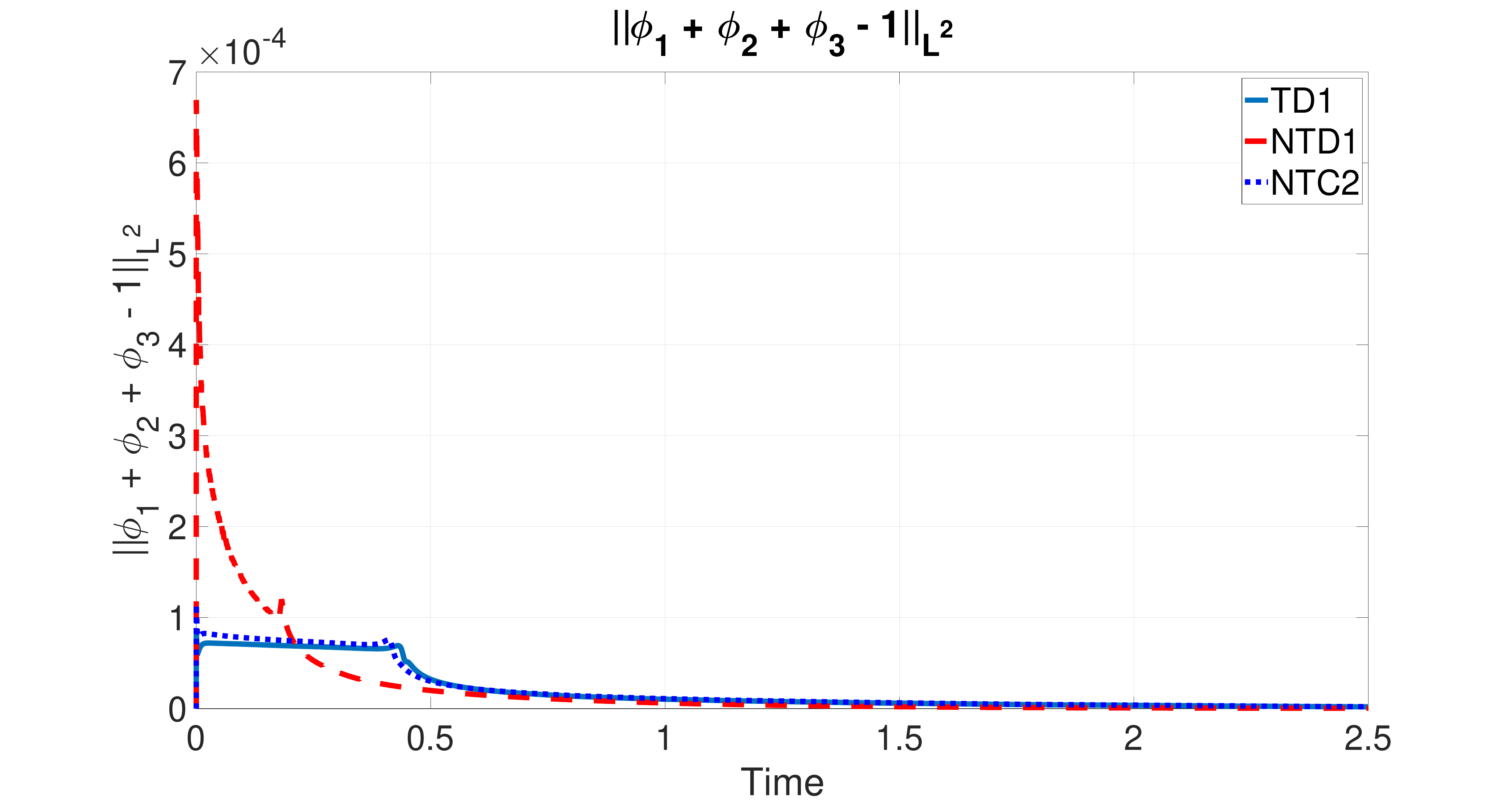}
\\ [1ex]
\includegraphics[scale=0.11]{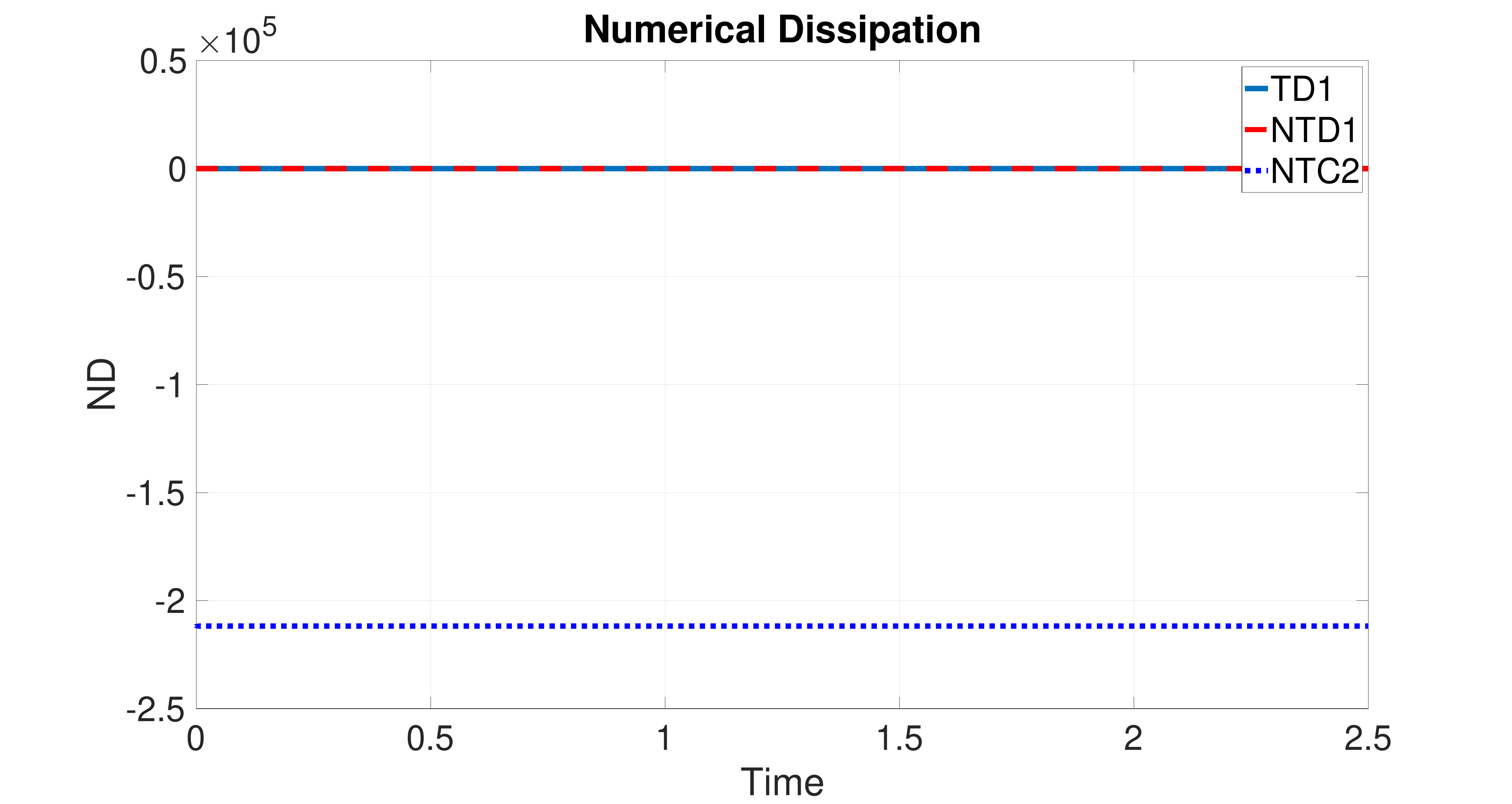}
\includegraphics[scale=0.11]{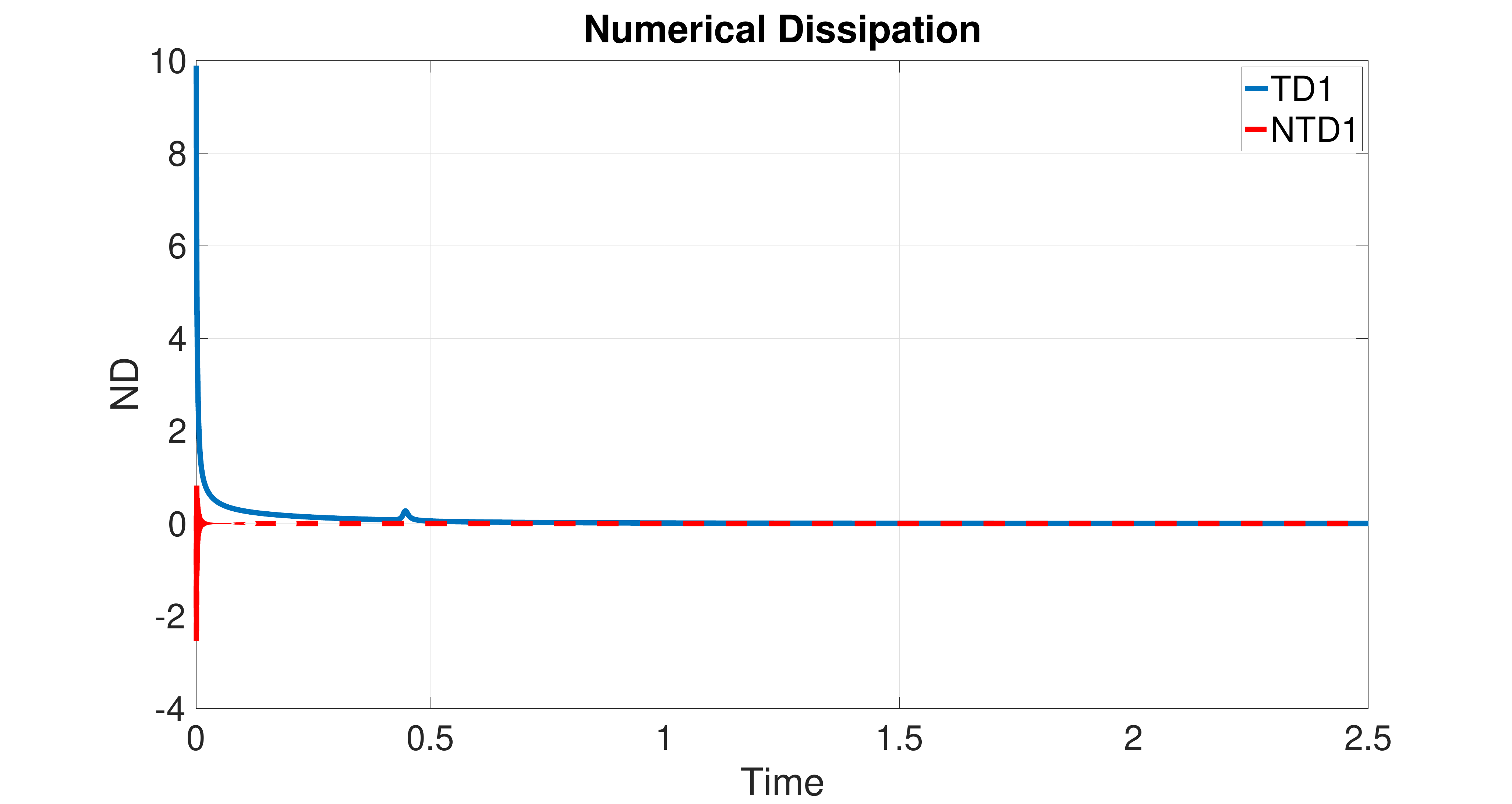}
\end{center}
\caption{Evolution in time of the energies (top left), the volume $\int_\Omega(\phi_1+\phi_2+\phi_3)$ (top right), $\|\phi_1 + \phi_2 + \phi_3 -1\|_{L^\infty}$ (center left), $\|\phi_1 + \phi_2 + \phi_3 -1\|_{L^2}$ (center right) and the evolution of the numerical dissipation (bottom row) with spreading coefficients $(\Sigma_1, \Sigma_2 , \Sigma_3) = (3,3,-0.1)$.}
\label{fig:lensTotalPlots1}
\end{figure}

\begin{figure}[h]
\begin{center}
\includegraphics[scale=0.07]{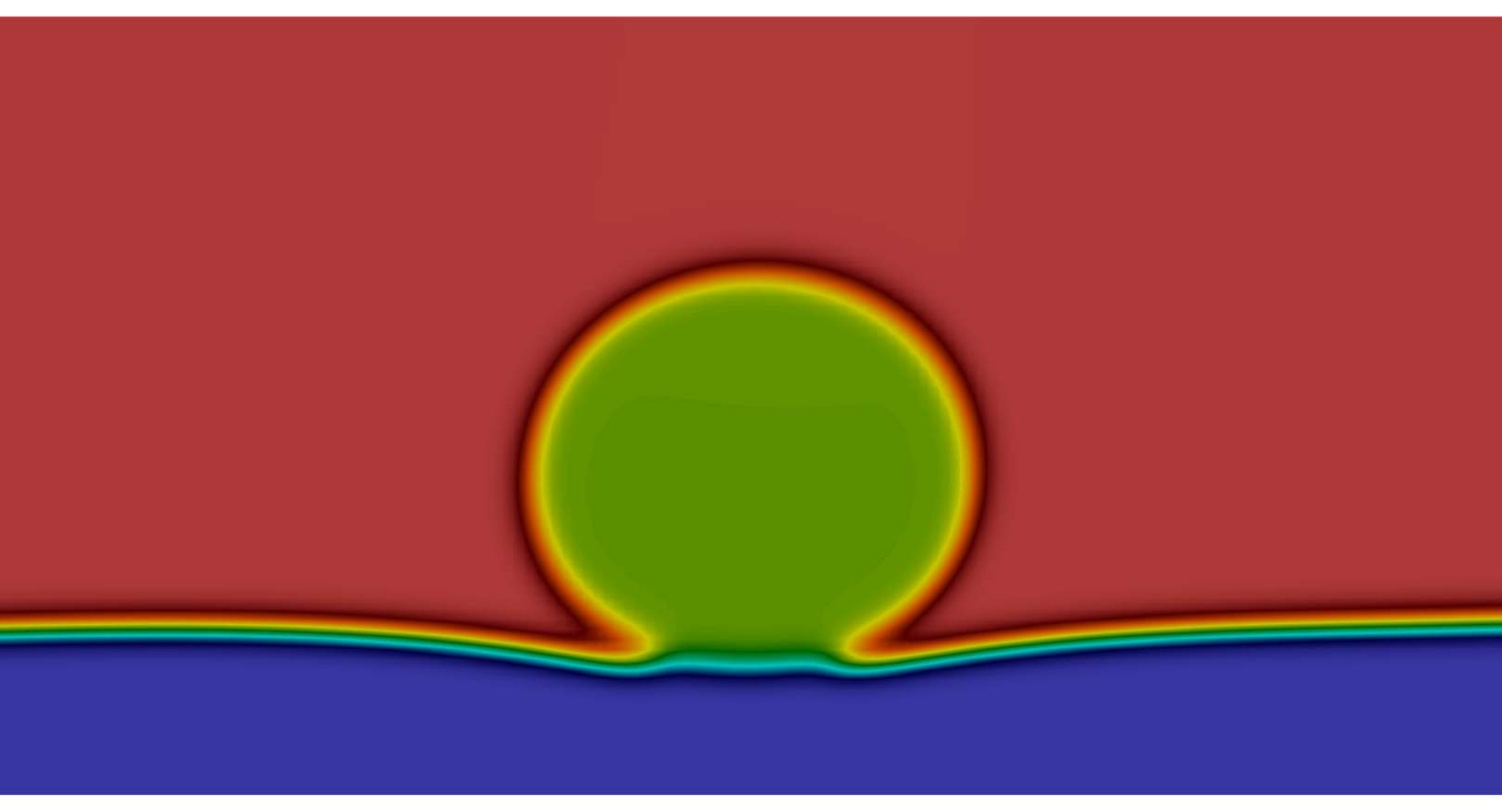}
\includegraphics[scale=0.07]{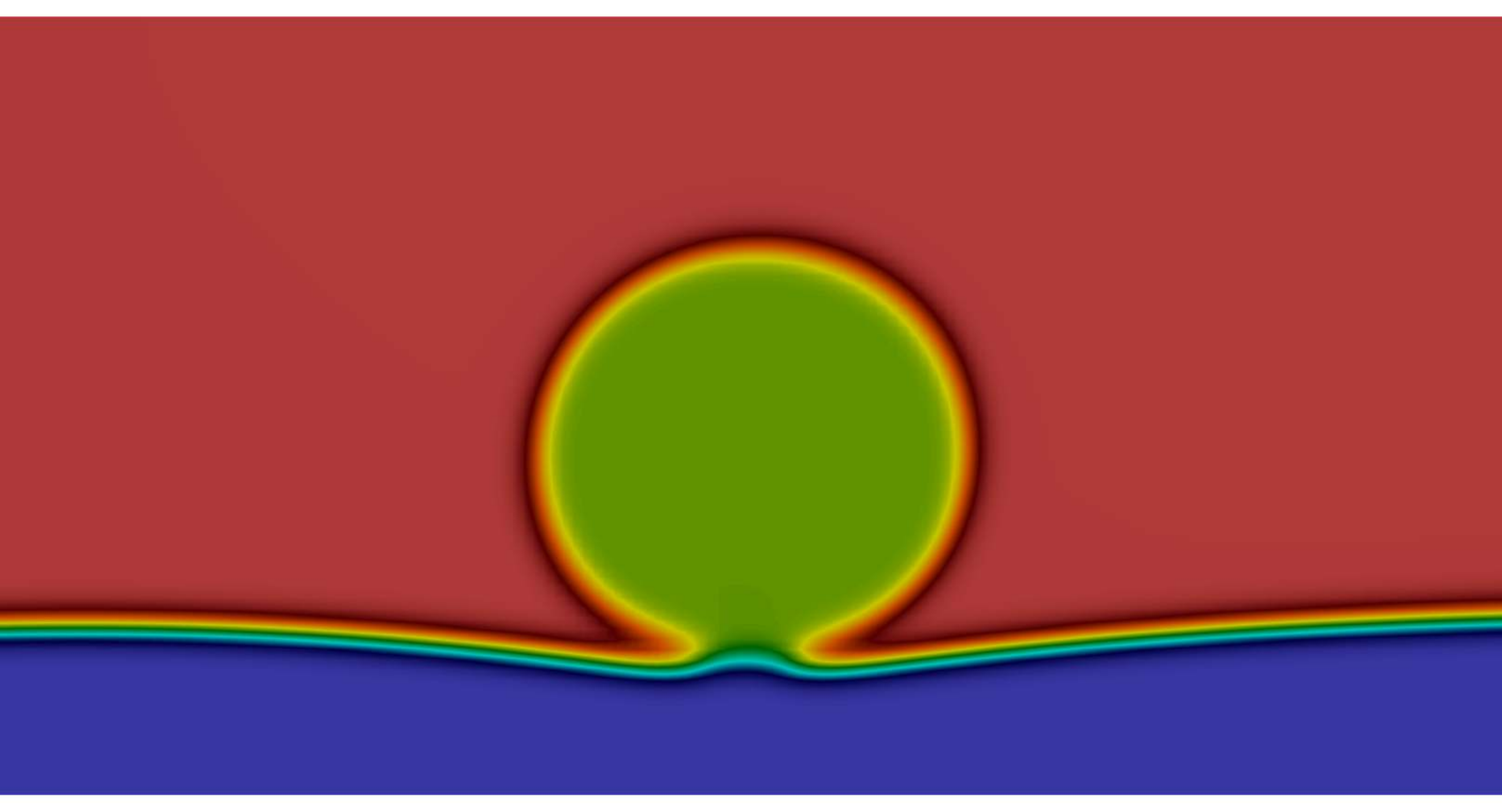}
\includegraphics[scale=0.07]{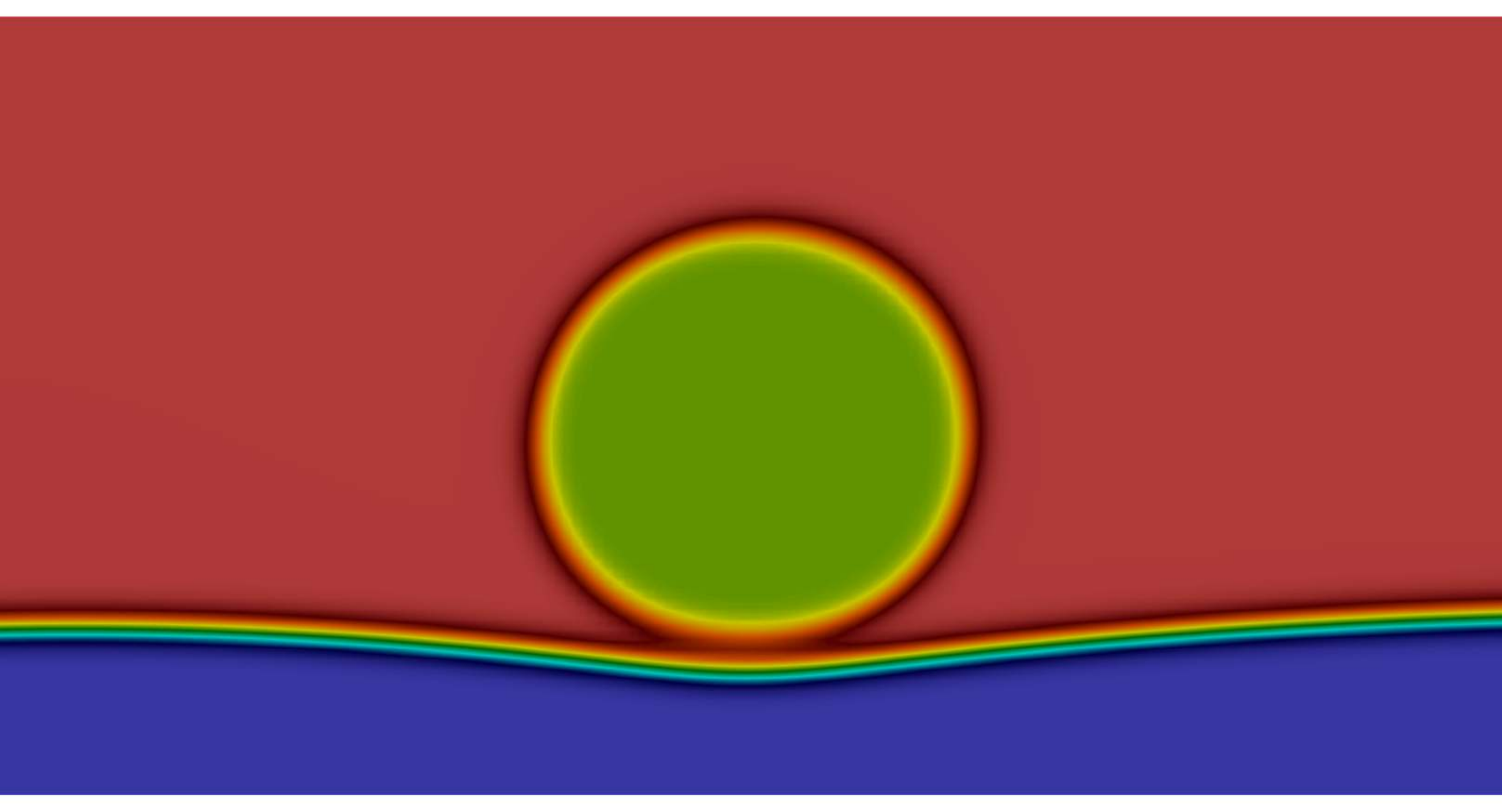}
\includegraphics[scale=0.07]{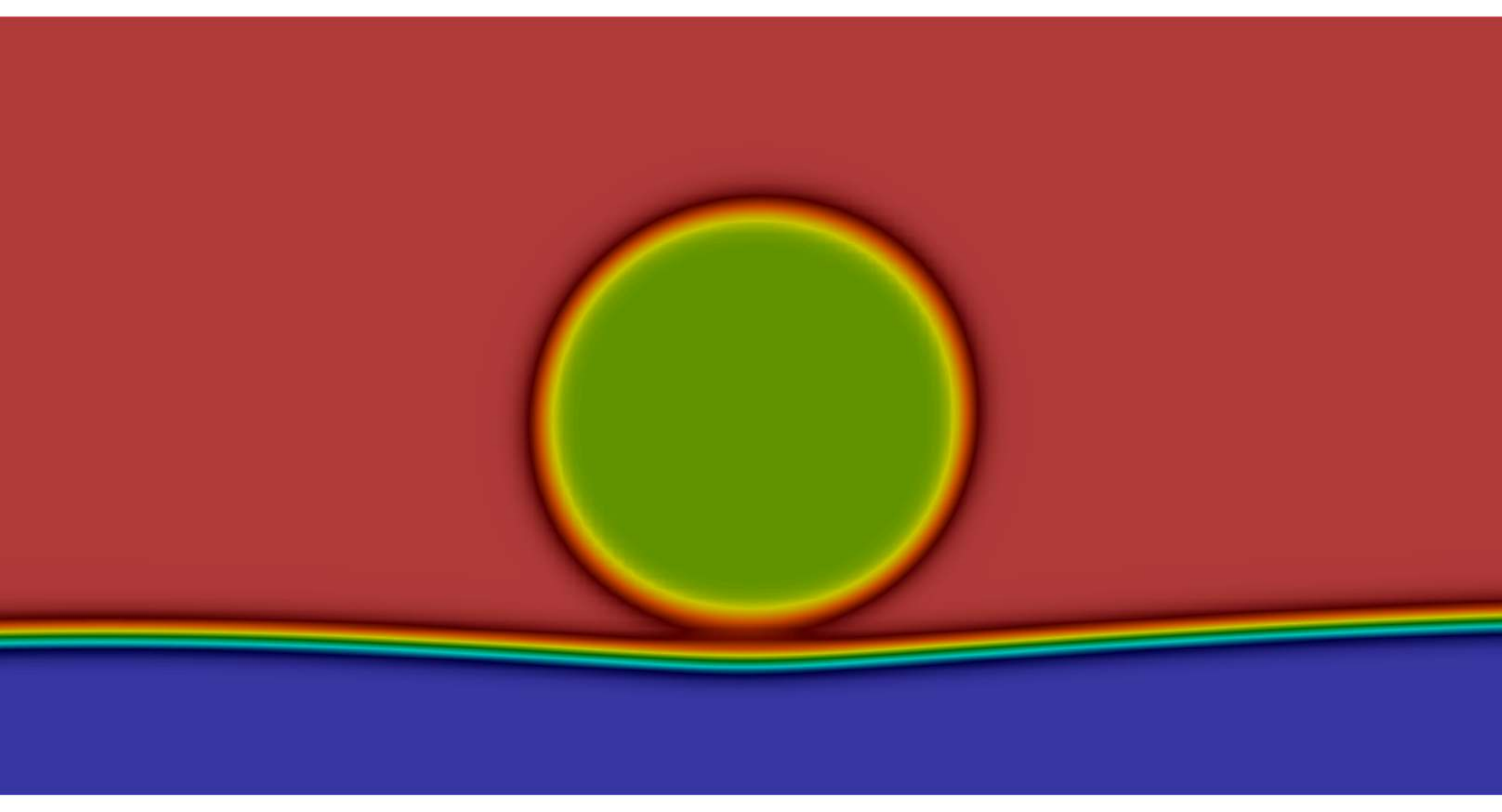}
\\
\includegraphics[scale=0.07]{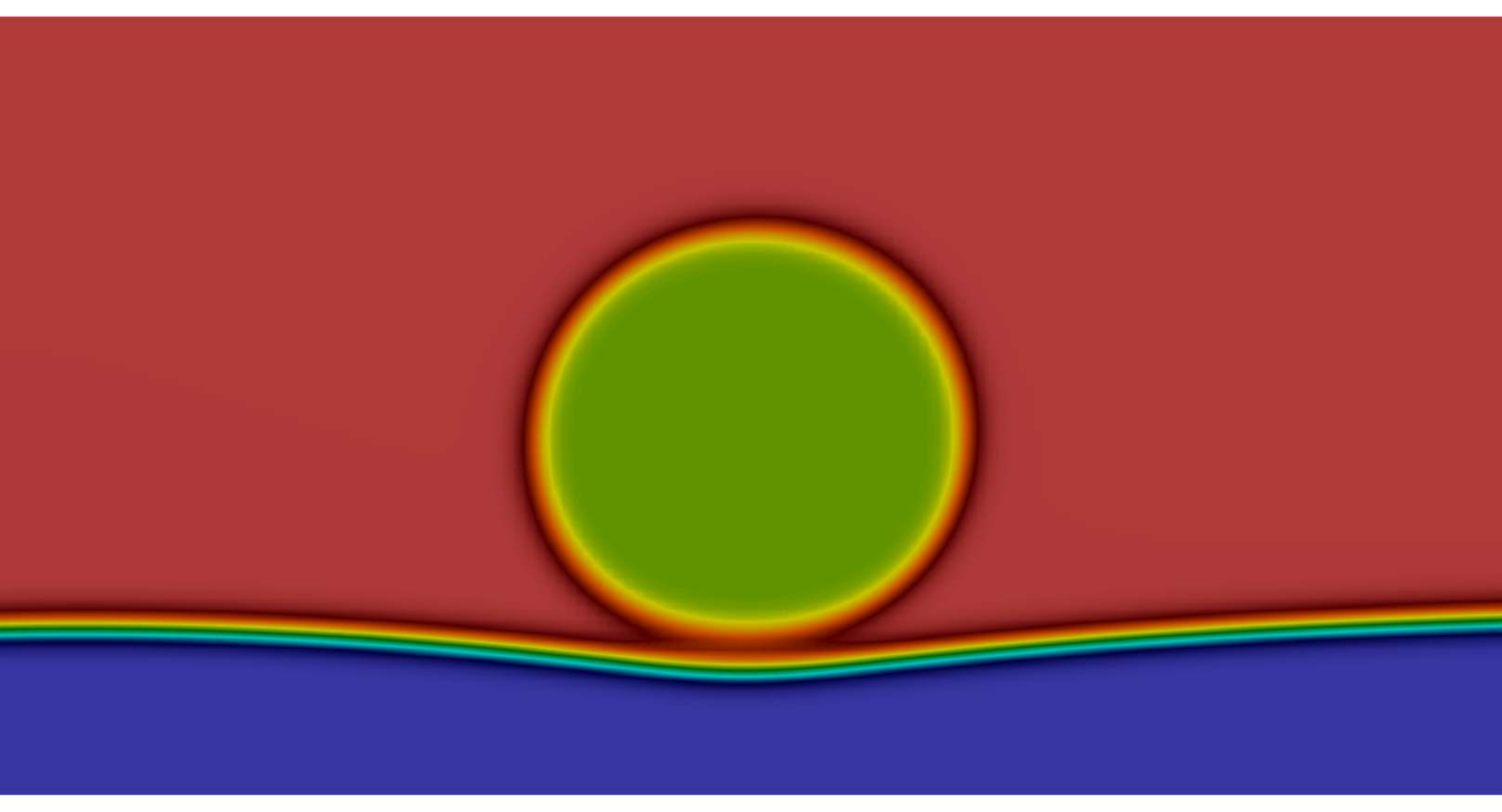}
\includegraphics[scale=0.07]{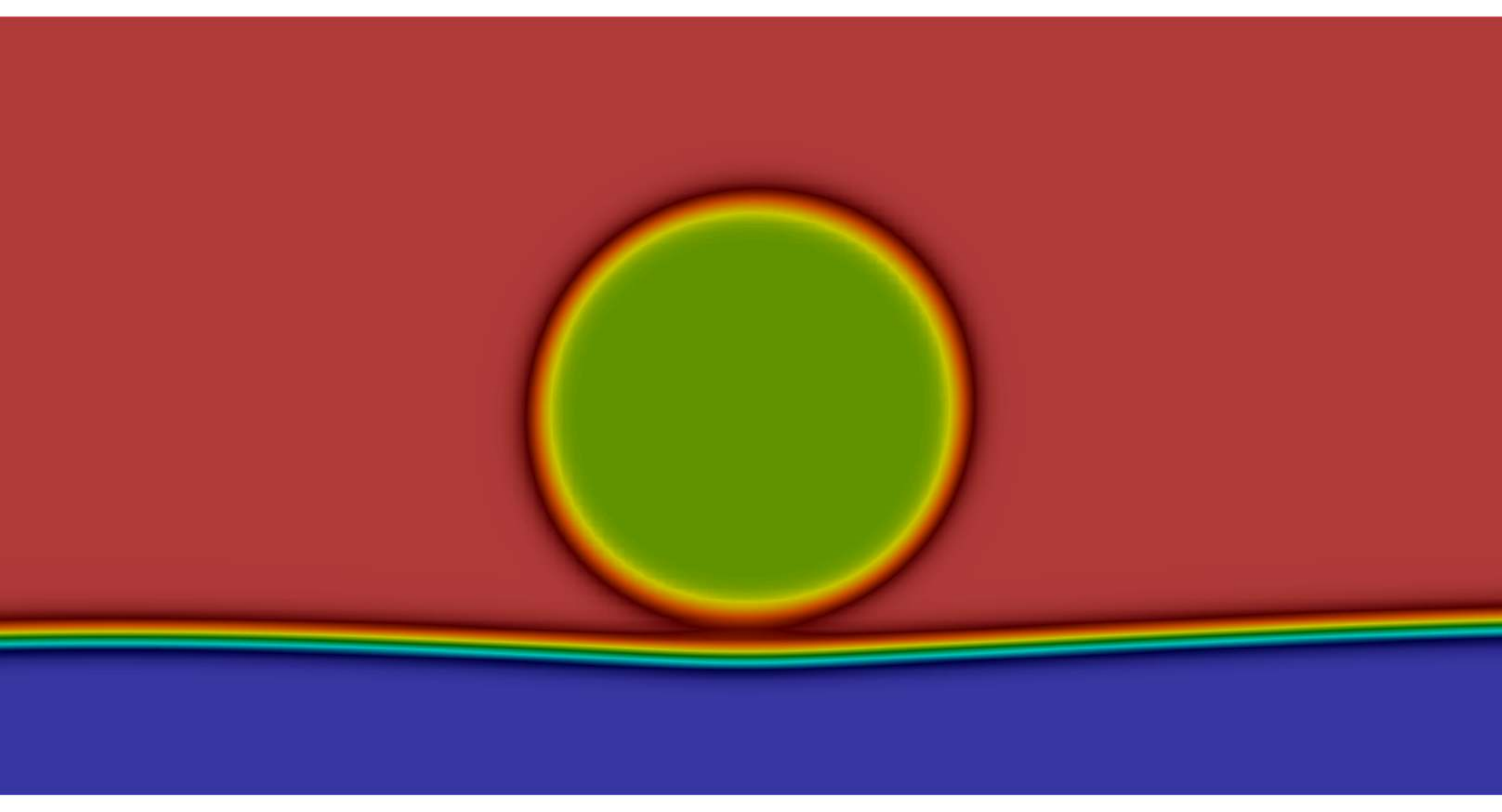}
\includegraphics[scale=0.07]{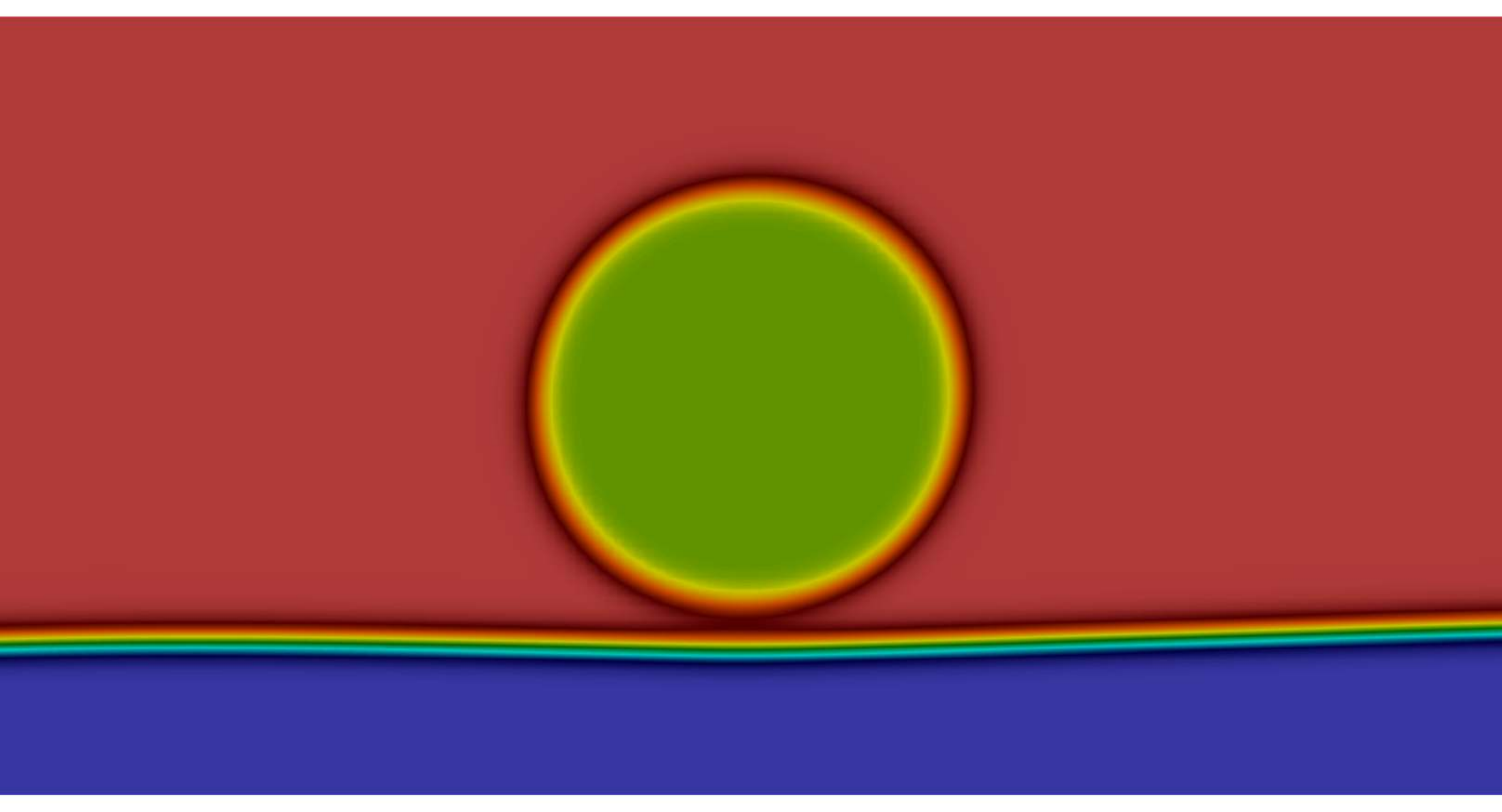}
\includegraphics[scale=0.07]{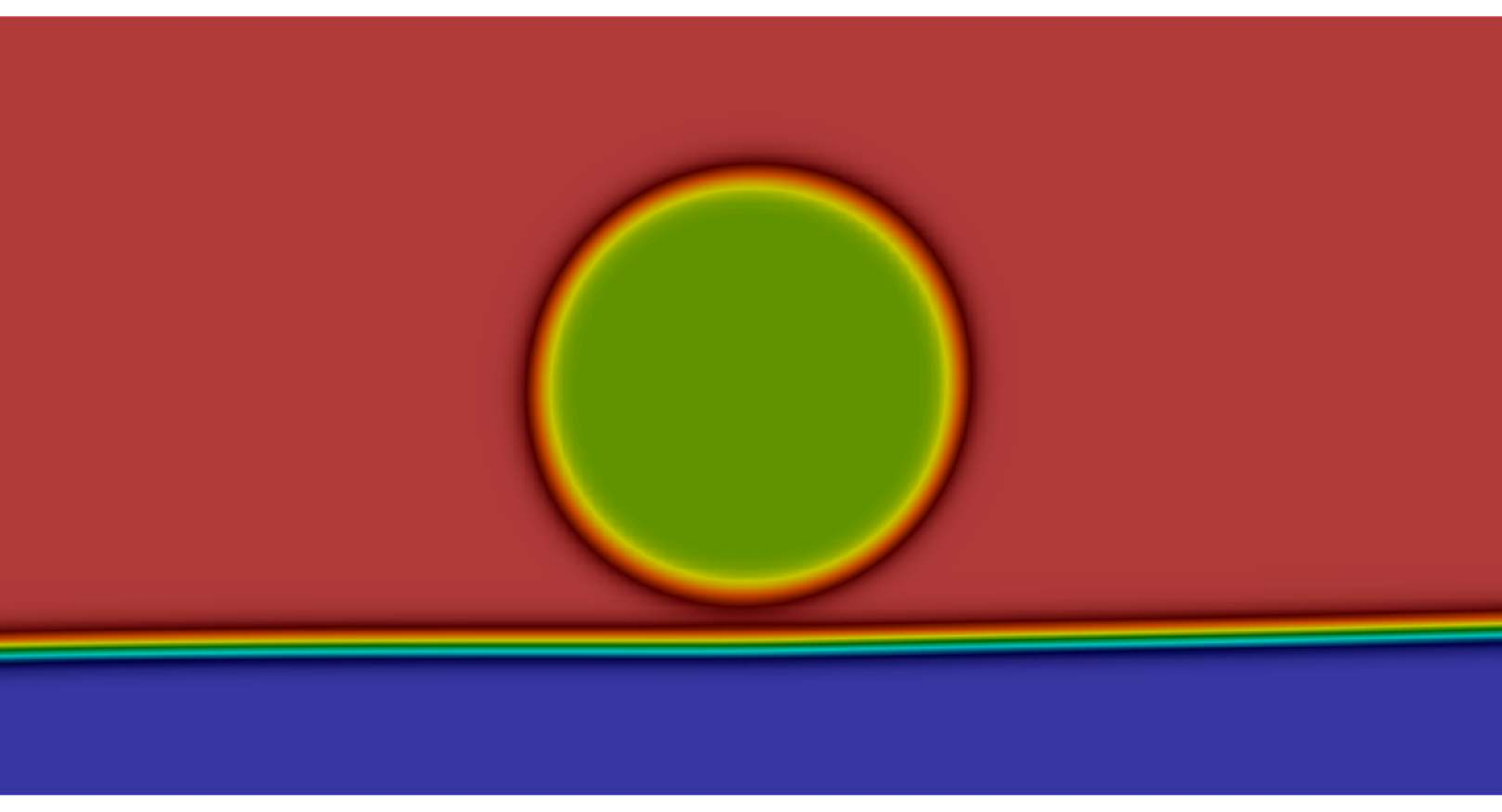}
\\
\includegraphics[scale=0.07]{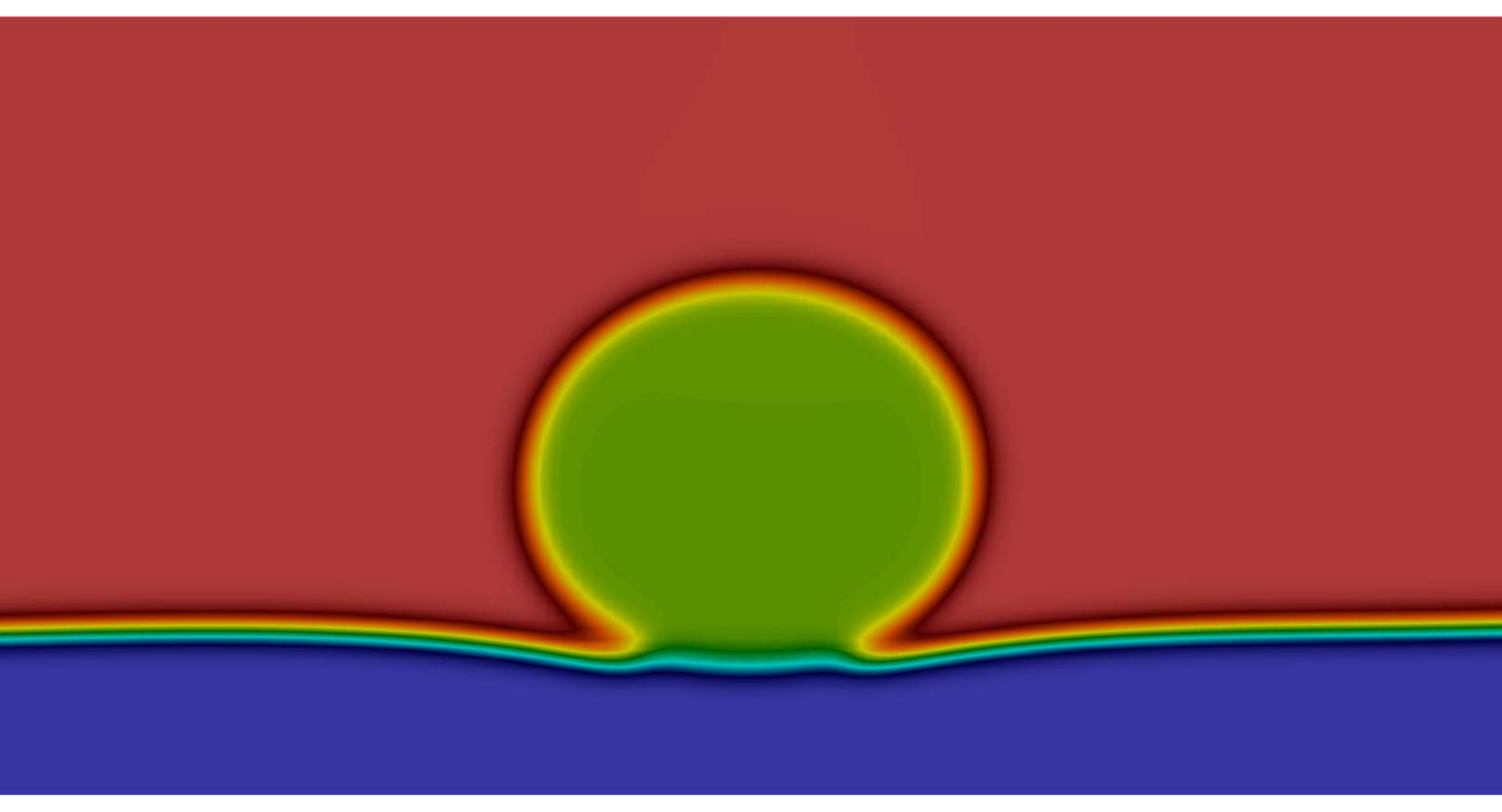}
\includegraphics[scale=0.07]{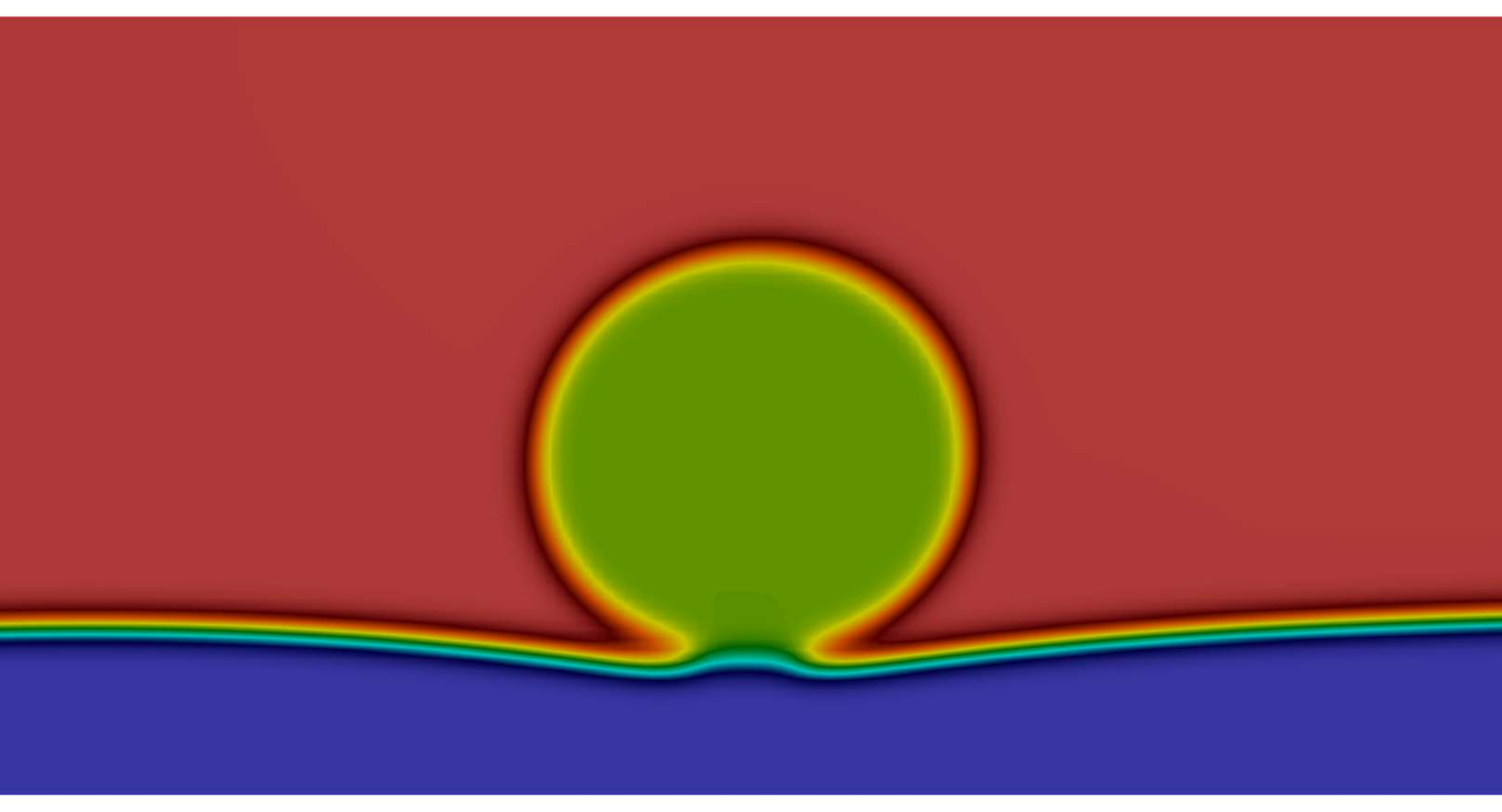}
\includegraphics[scale=0.07]{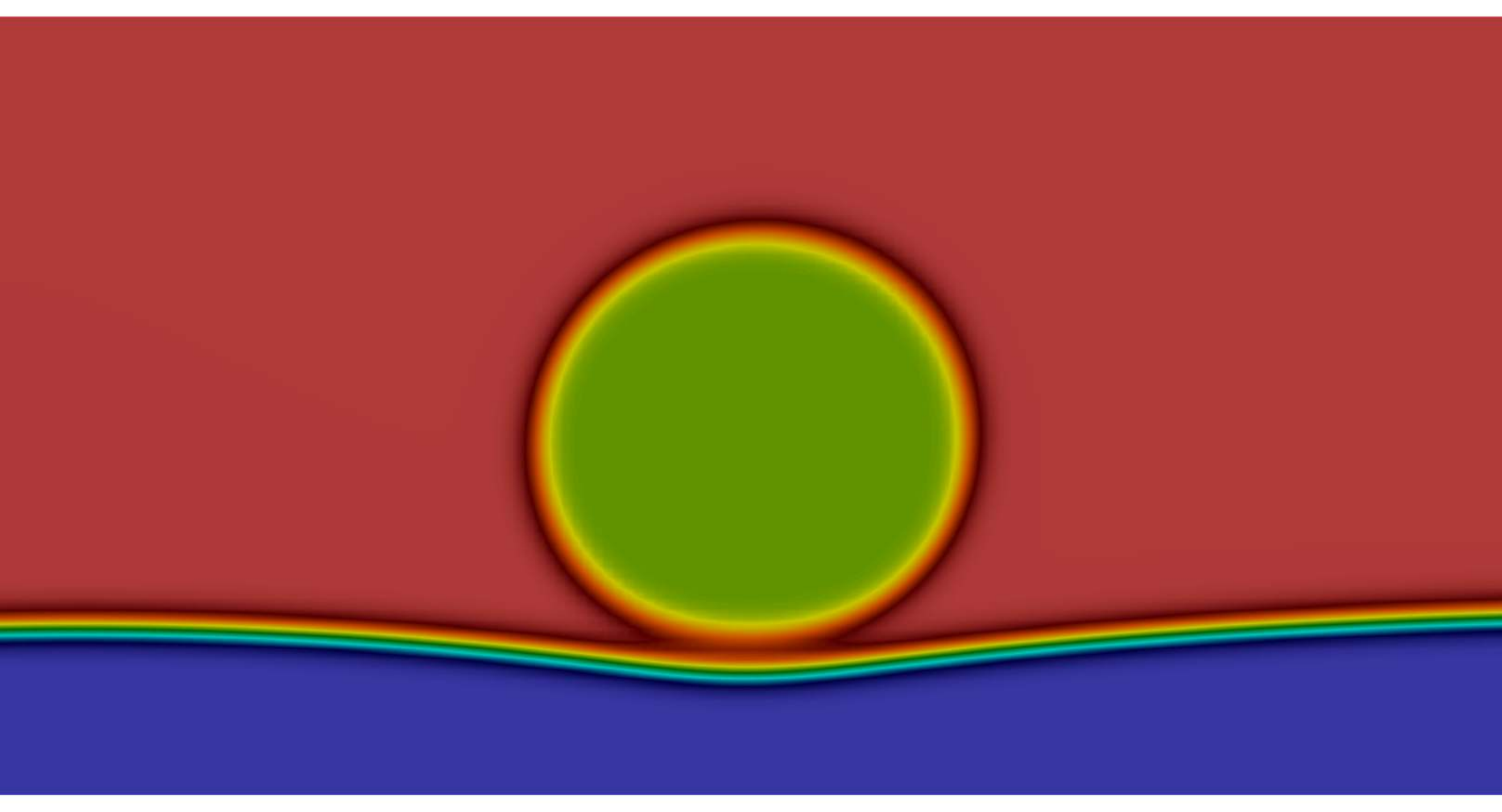}
\includegraphics[scale=0.07]{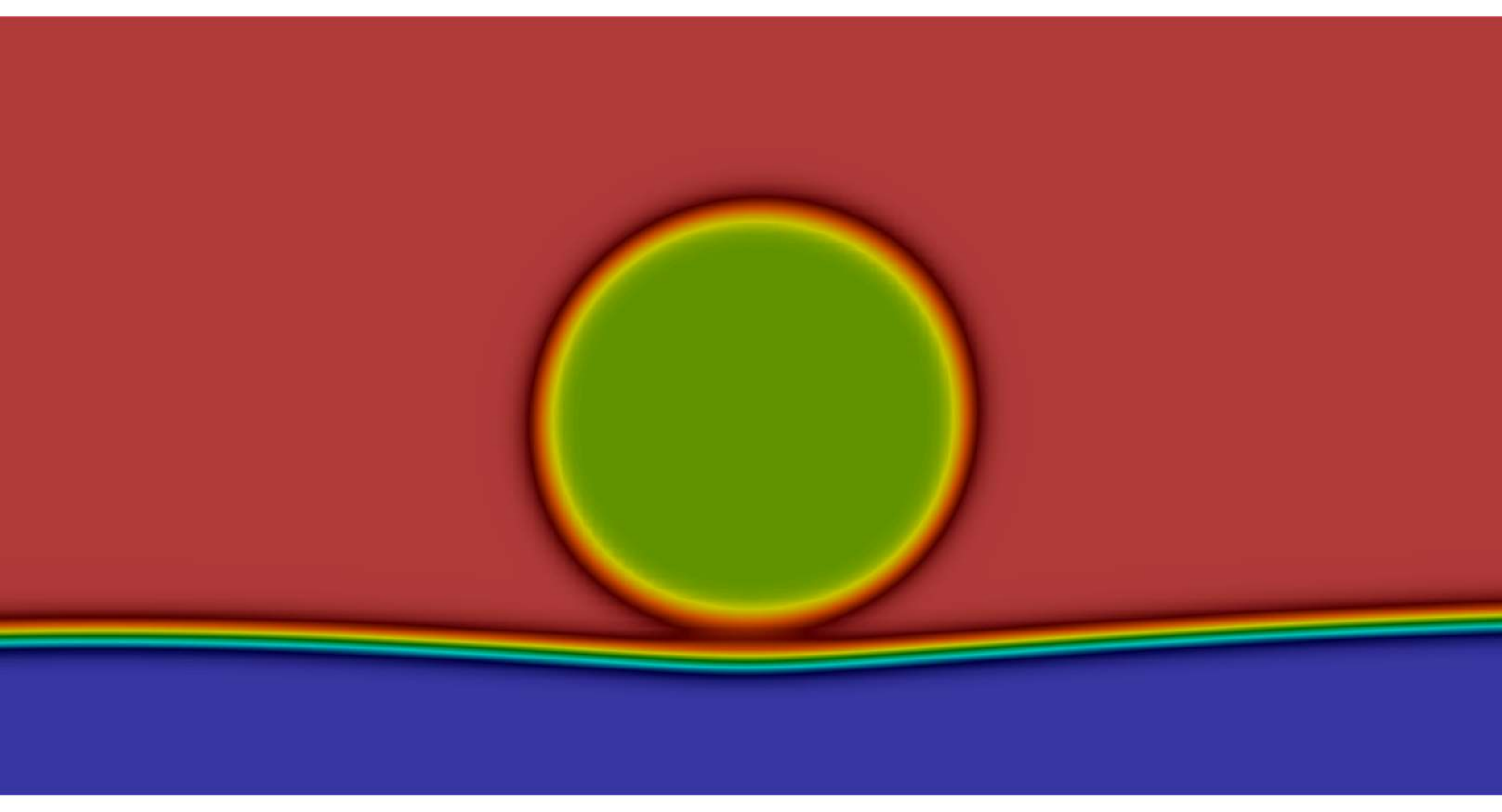}
\end{center}
\caption{Dynamics of schemes TD1 (top row), NTD1 (center row) and NTC2 (bottom row) at times $t=0.5, 1, 1.5$ and $2.5$ (from left to right) with spreading coefficients $(\Sigma_1, \Sigma_2 , \Sigma_3) = (-0.1,3,3)$.}\label{fig:lensTotalDynamics2}
\end{figure}

\begin{figure}[h]
\begin{center}
\includegraphics[scale=0.11]{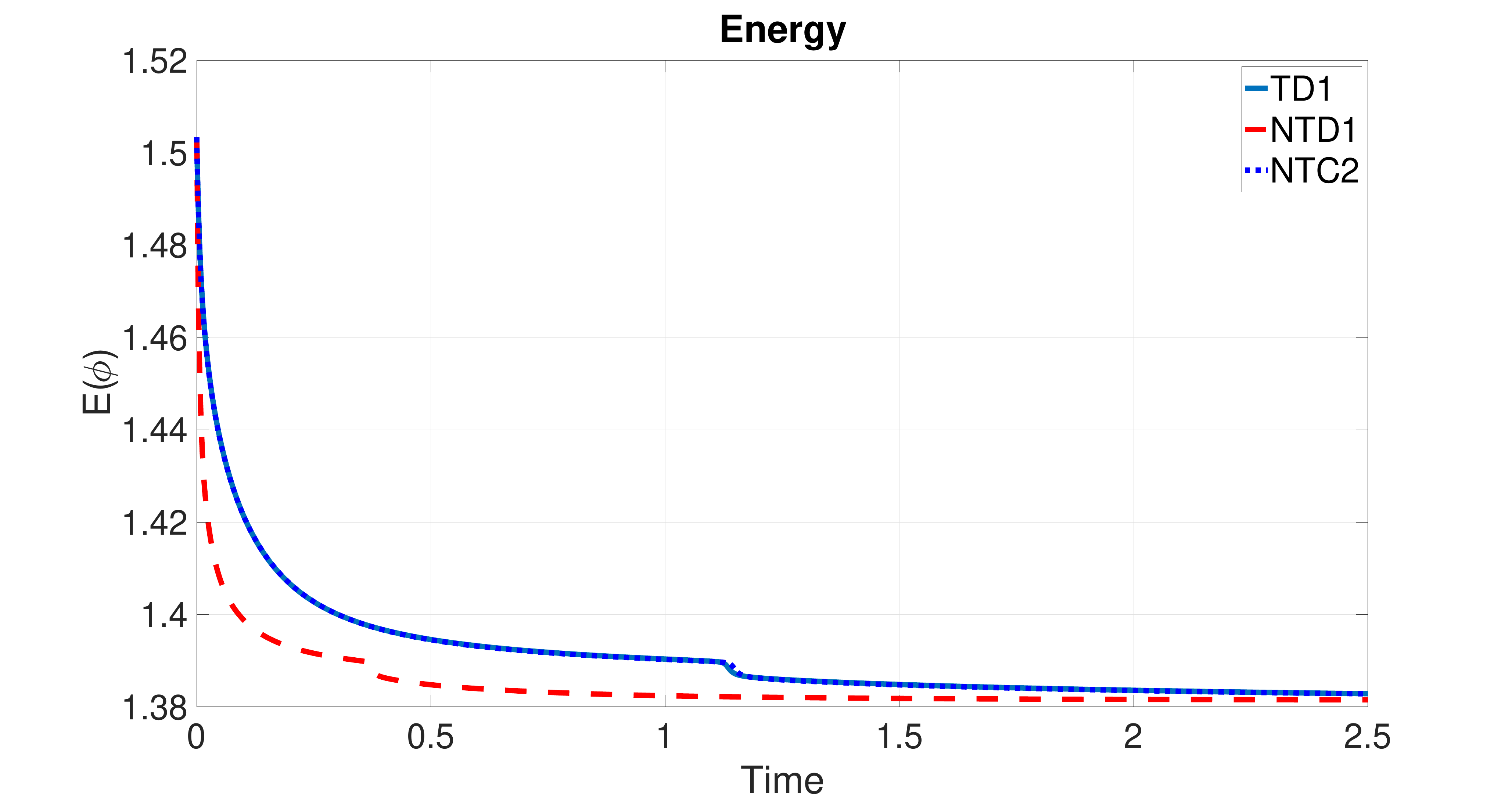}
\includegraphics[scale=0.11]{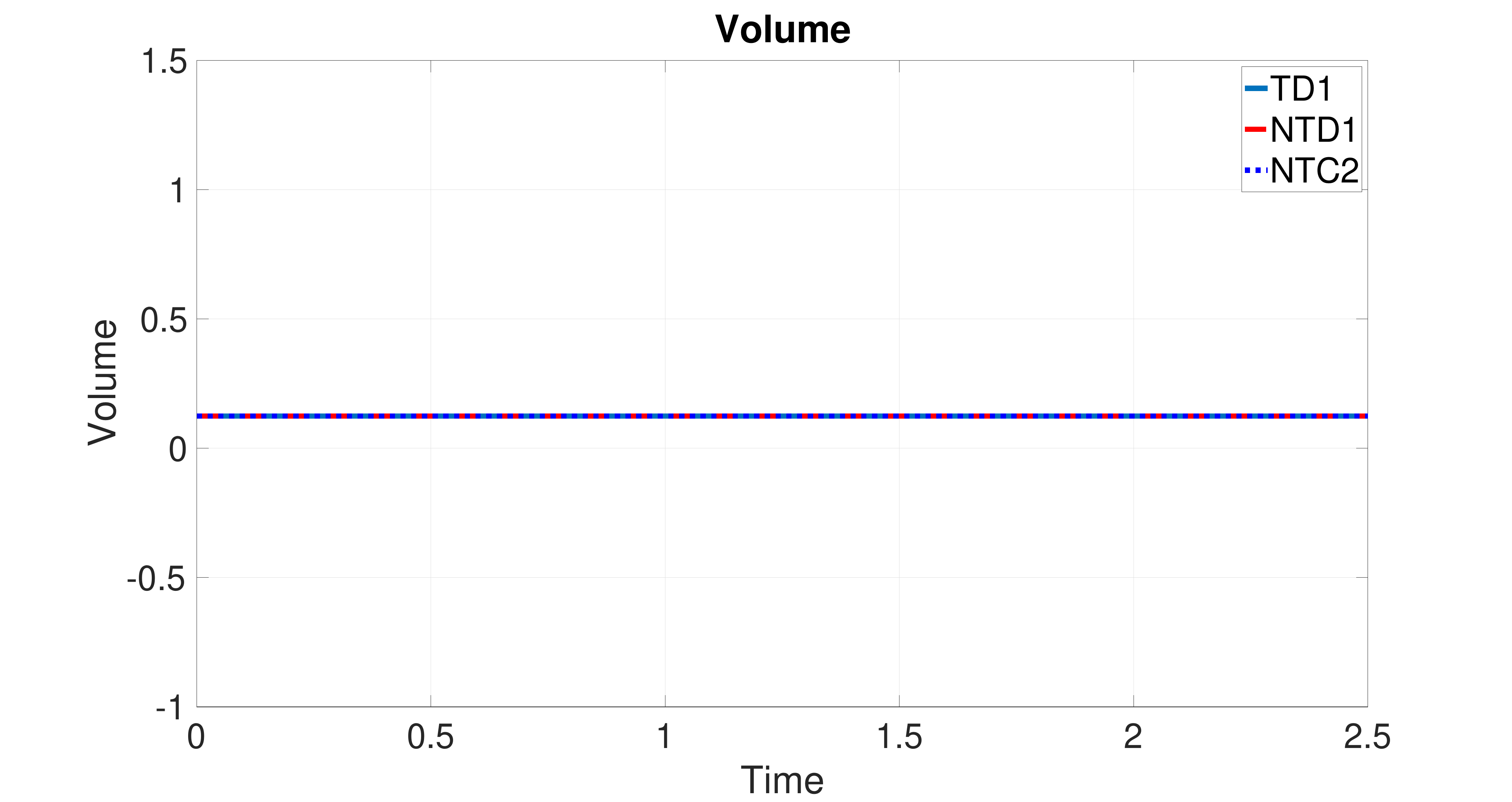}
\\ [1ex]
\includegraphics[scale=0.11]{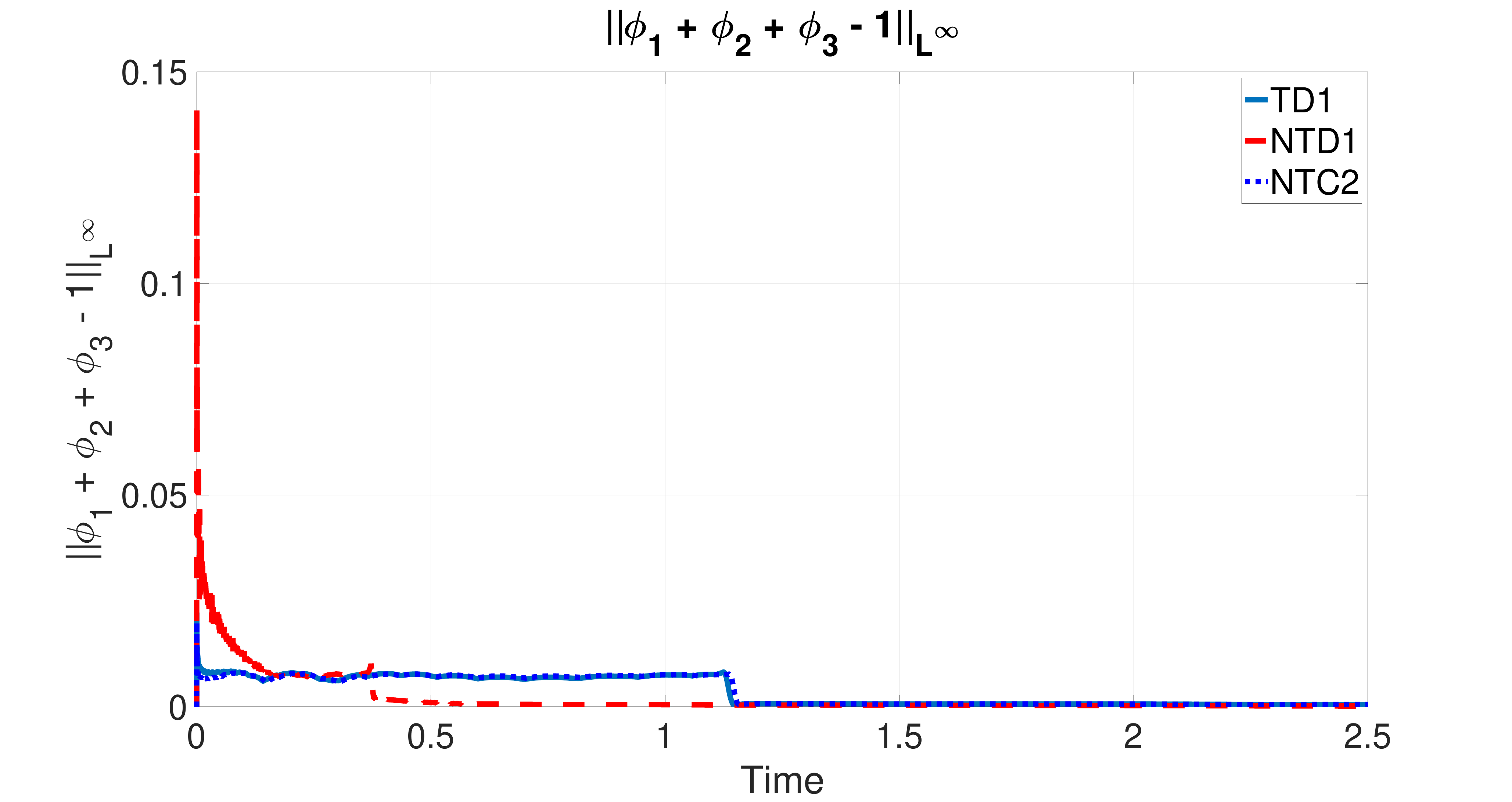}
\includegraphics[scale=0.11]{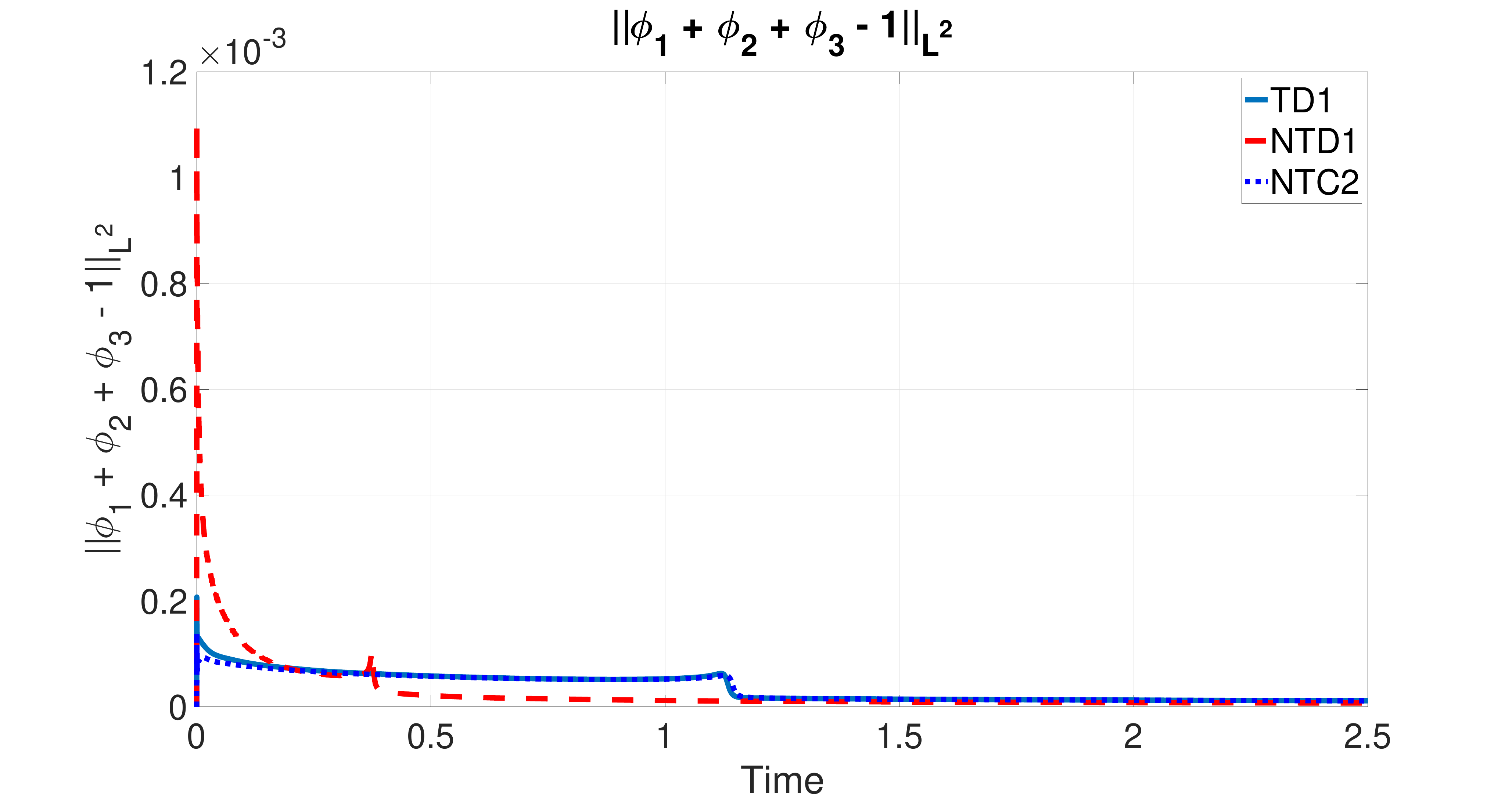}
\\ [1ex]
\includegraphics[scale=0.11]{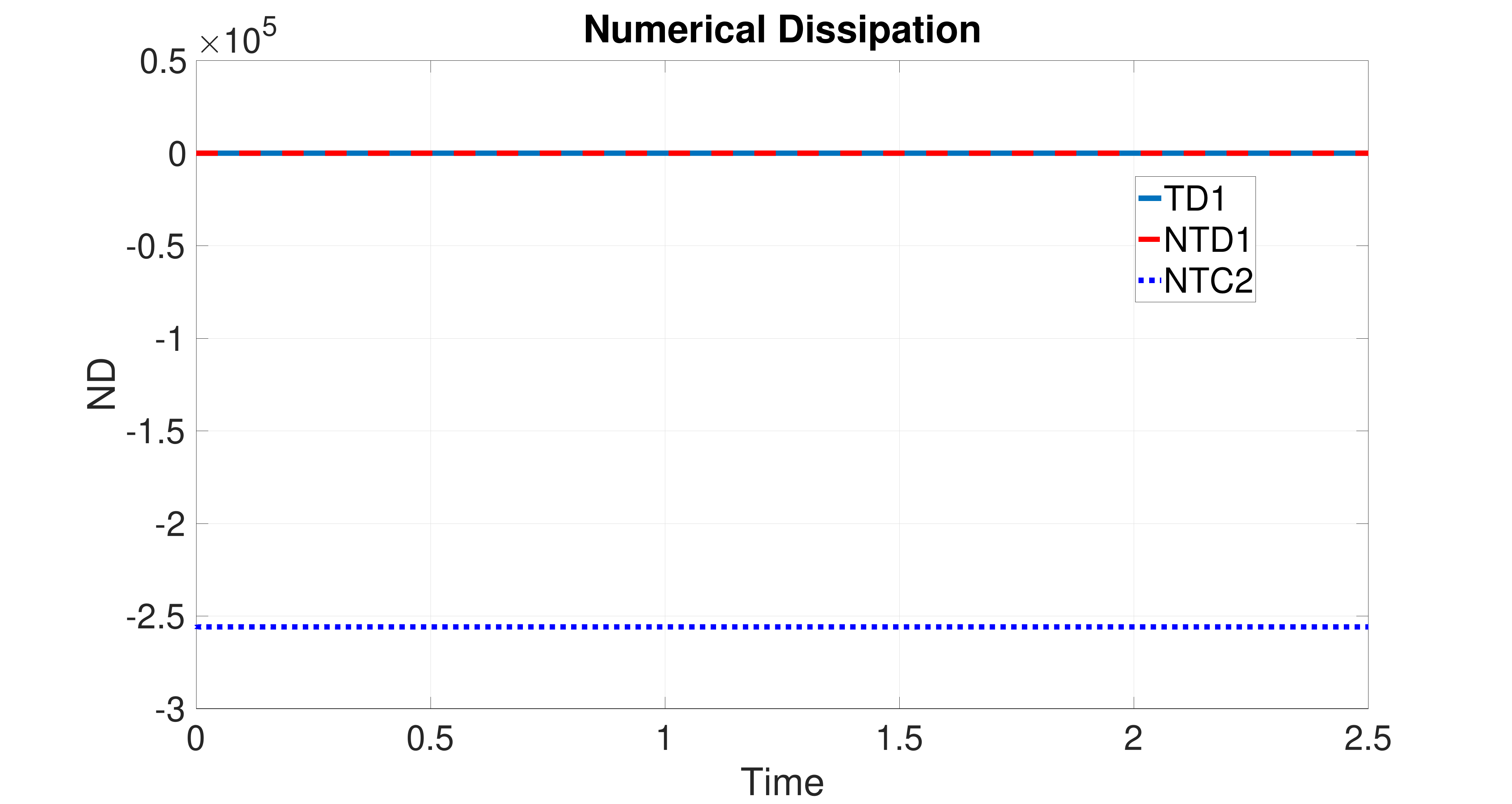}
\includegraphics[scale=0.11]{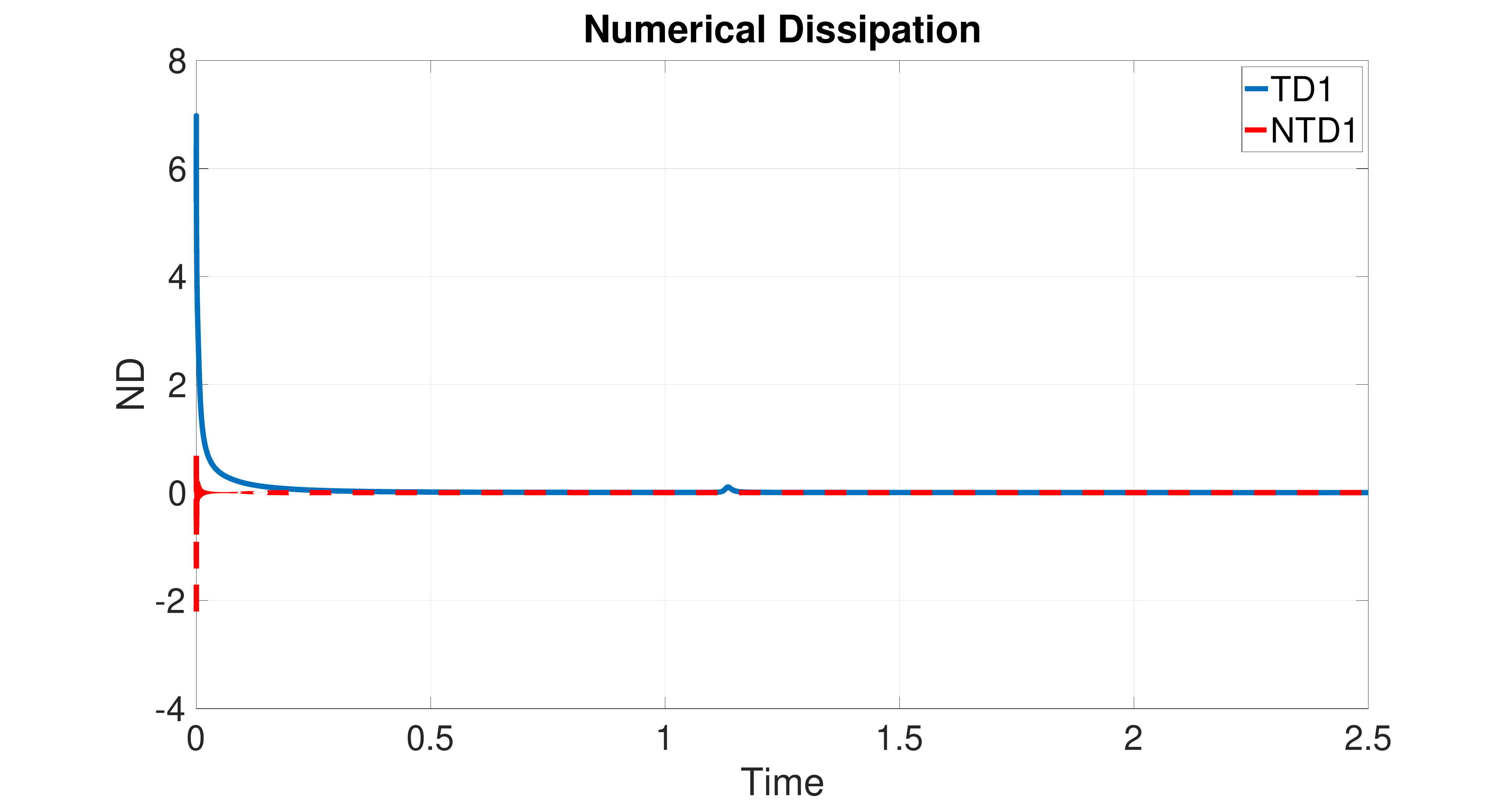}
\end{center}
\caption{Evolution in time of the energies (top left), the volume (top right), $\|\phi_1 + \phi_2 + \phi_3 -1\|_{L^\infty}$ (center left), $\|\phi_1 + \phi_2 + \phi_3 -1\|_{L^2}$ (center right) and the evolution of the numerical dissipation (bottom row) with spreading coefficients $(\Sigma_1, \Sigma_2 , \Sigma_3) = (-0.1,3,3)$.}
\label{fig:lensTotalPlots2}
\end{figure}

\subsubsection{Particular case of considering only two components $(\phi_2=0)$}
{
In this example we study how consistent is the model and the proposed numerical schemes with the two components systems. To this end we modify the initial condition presented in \eqref{eq:lensInitial} taking $\phi_2=0$. In Figure~\ref{fig:lensCase0Dyn} we can see how the three schemes produce the same dynamics. But plotting function $\phi_2$ (see Figure~\ref{fig:lensCase0phi2}) illustrates that some spurious creation of phase $\phi_2$  happens in the interface between the two phases when using schemes TD1 and NTD1 but not with scheme NTC2. In Figures~\ref{fig:lensCase0phi2} and \ref{fig:lensCase0NTD1phi2maxmin} we show the evolution in time of $\phi_2$ and its maximum and minimum, respectively for scheme NTD1. It seems that after the first iteration (where the system is adapting to an initial condition that is not exactly a solution of the PDE), the spurious phase $\phi_2$ is disappearing and not affecting at all the dynamics of the system. These results support the idea that the proposed model itself is consistent with the two components systems, but the decoupled schemes produce a small creation of the missing phase in the interface between the components, but is so small (and eventually dissapears) that is not changing the dynamics of the system. 
}
\begin{figure}[h]
\begin{center}
\includegraphics[scale=0.07]{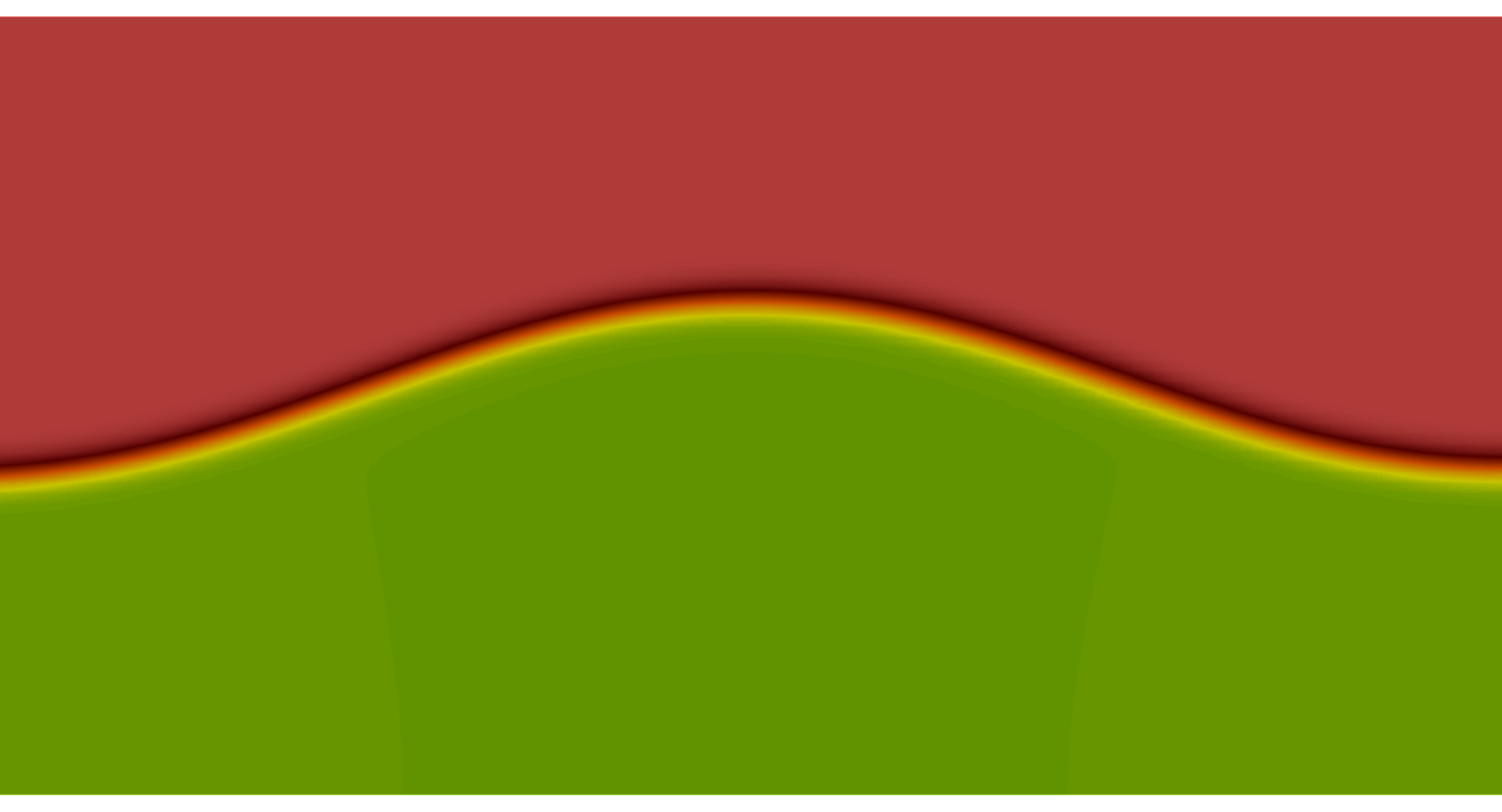}
\includegraphics[scale=0.07]{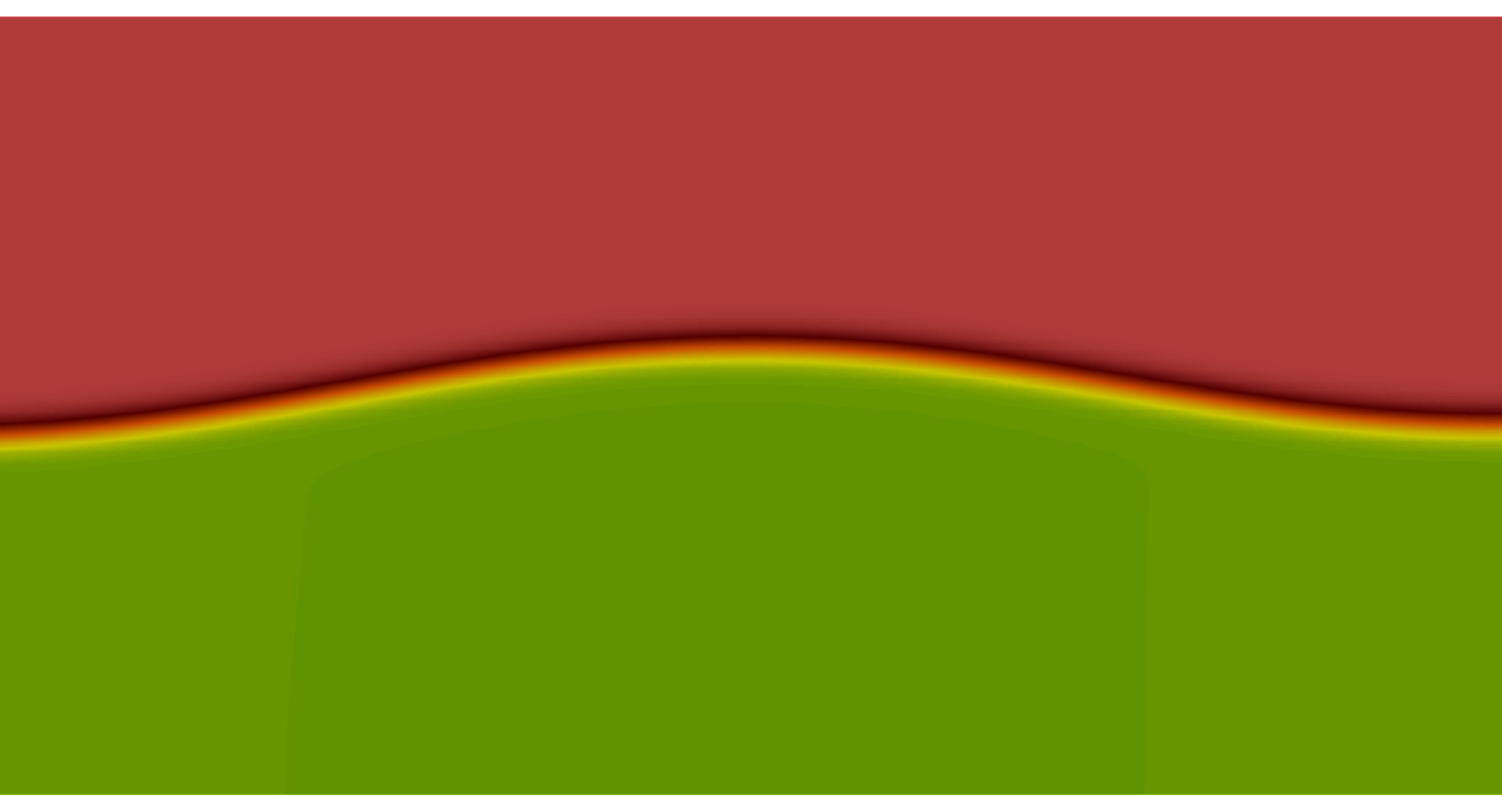}
\includegraphics[scale=0.07]{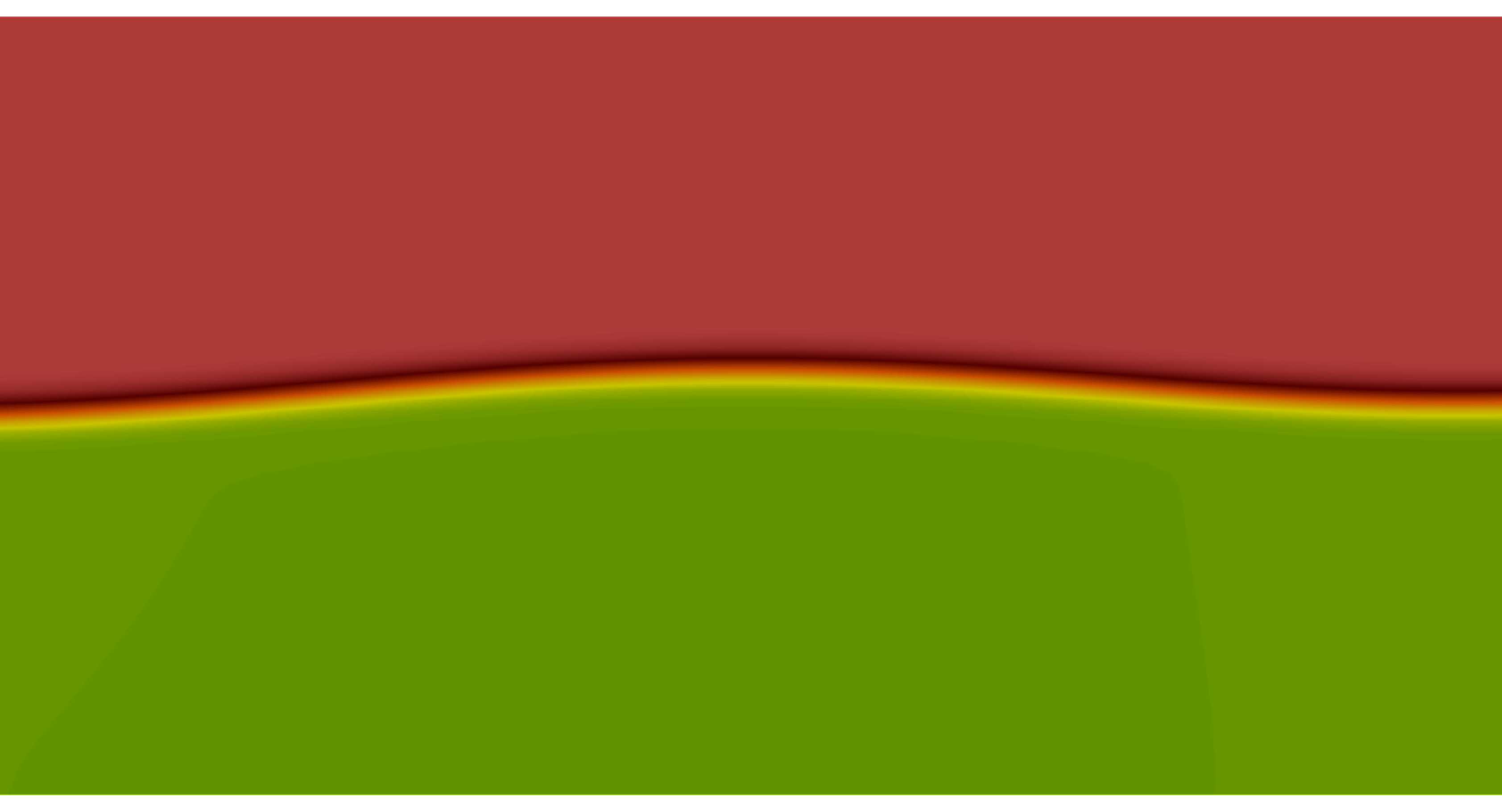}
\includegraphics[scale=0.07]{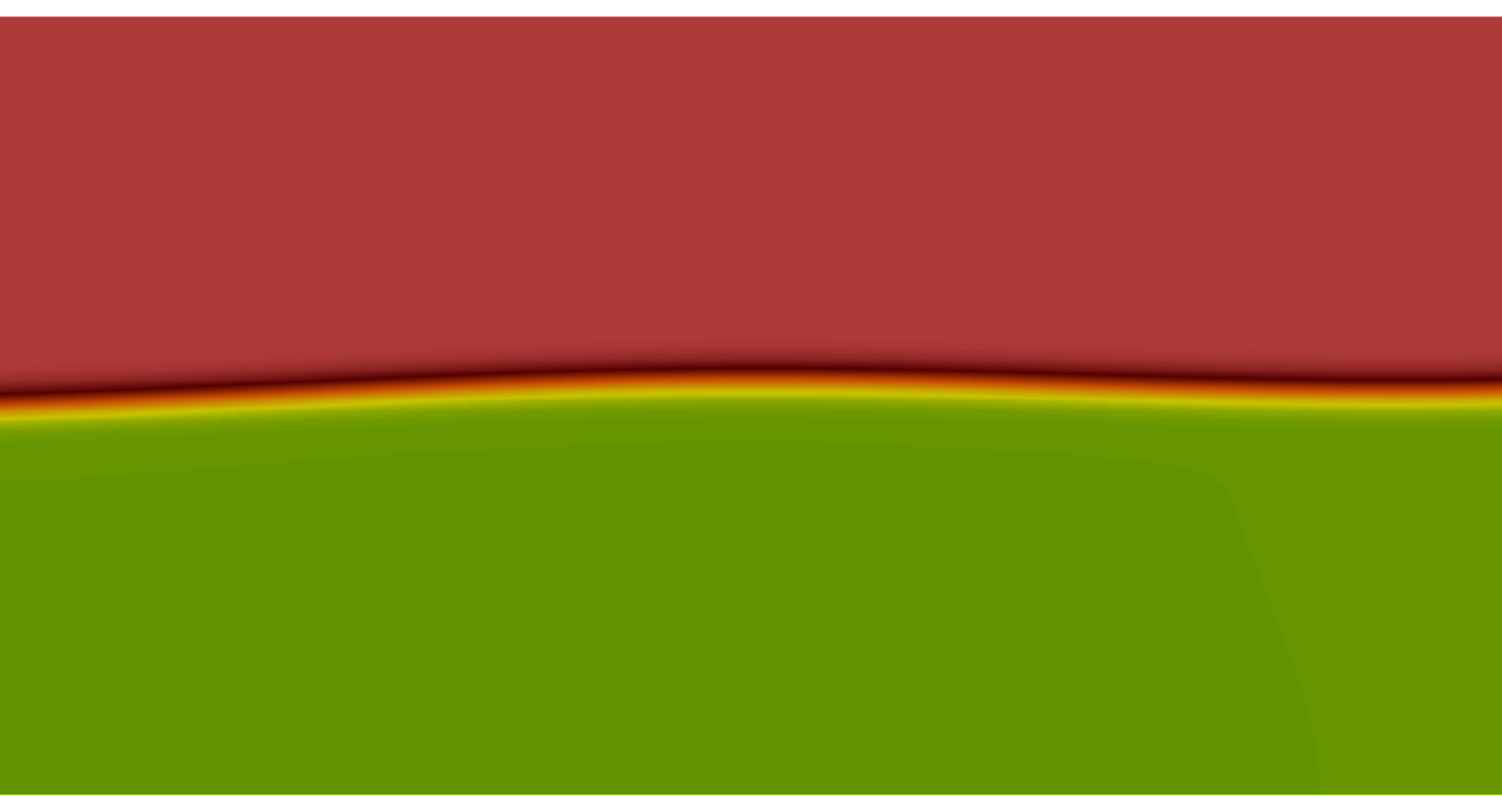}
\\
\includegraphics[scale=0.07]{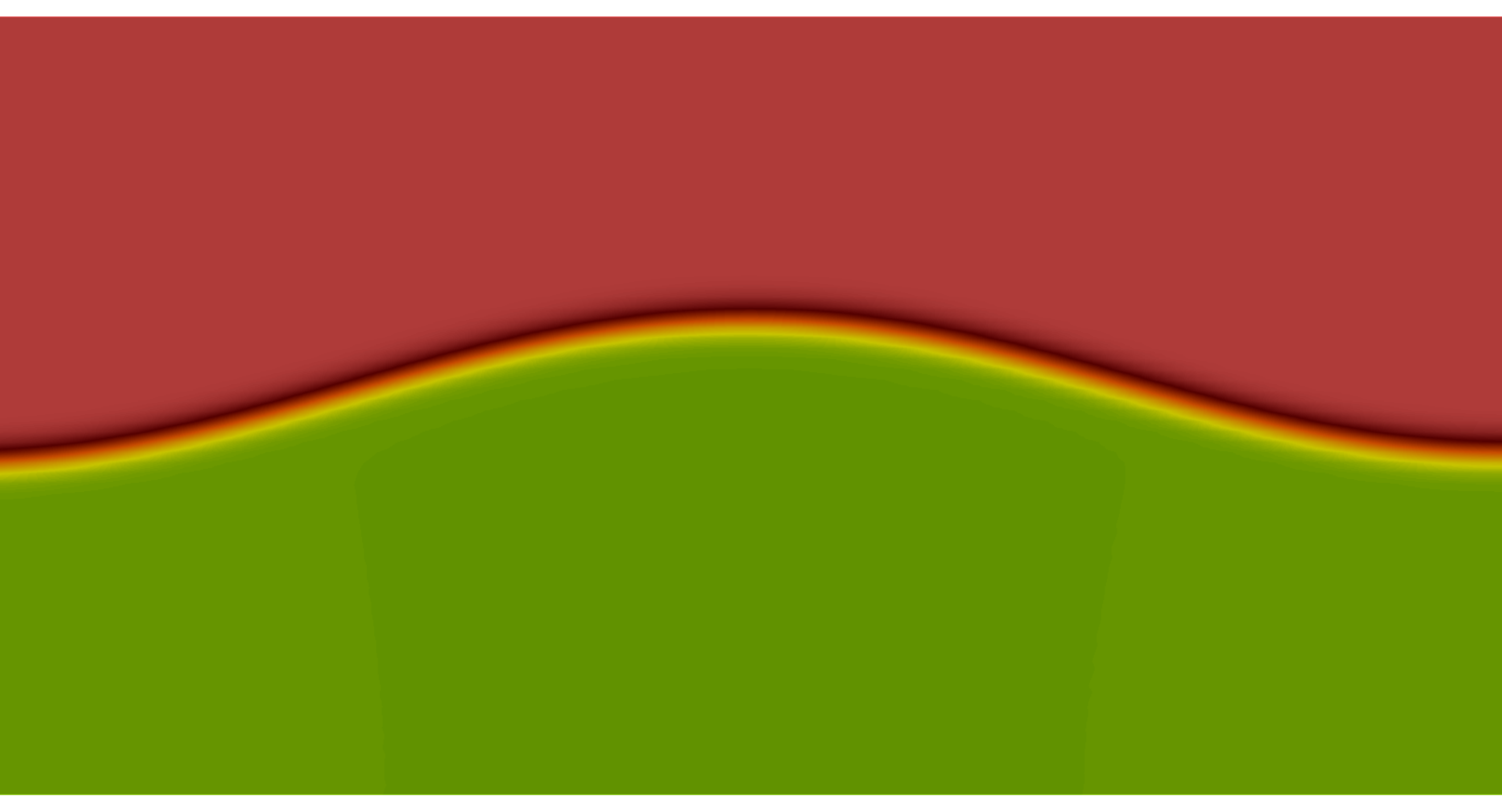}
\includegraphics[scale=0.07]{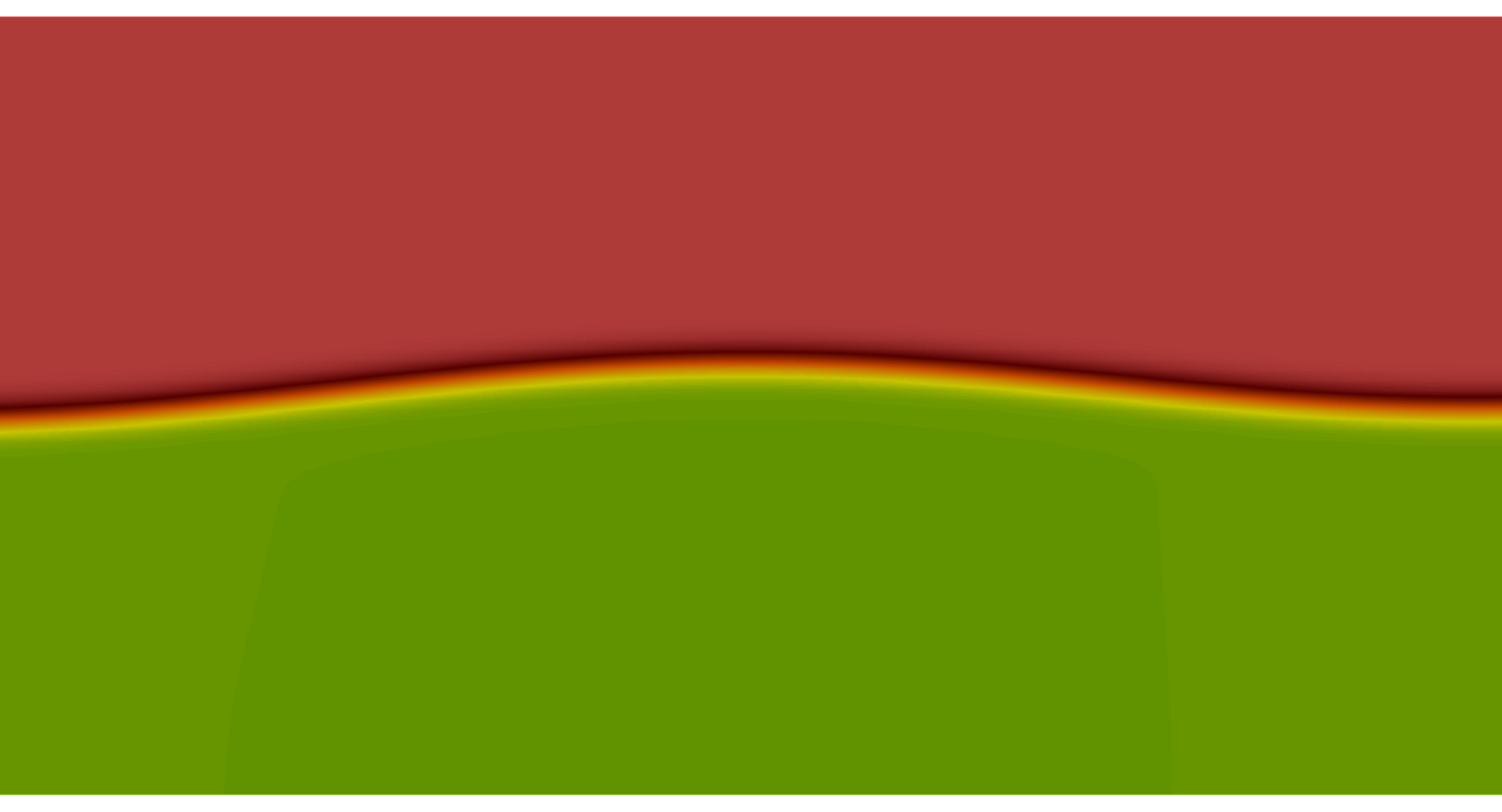}
\includegraphics[scale=0.07]{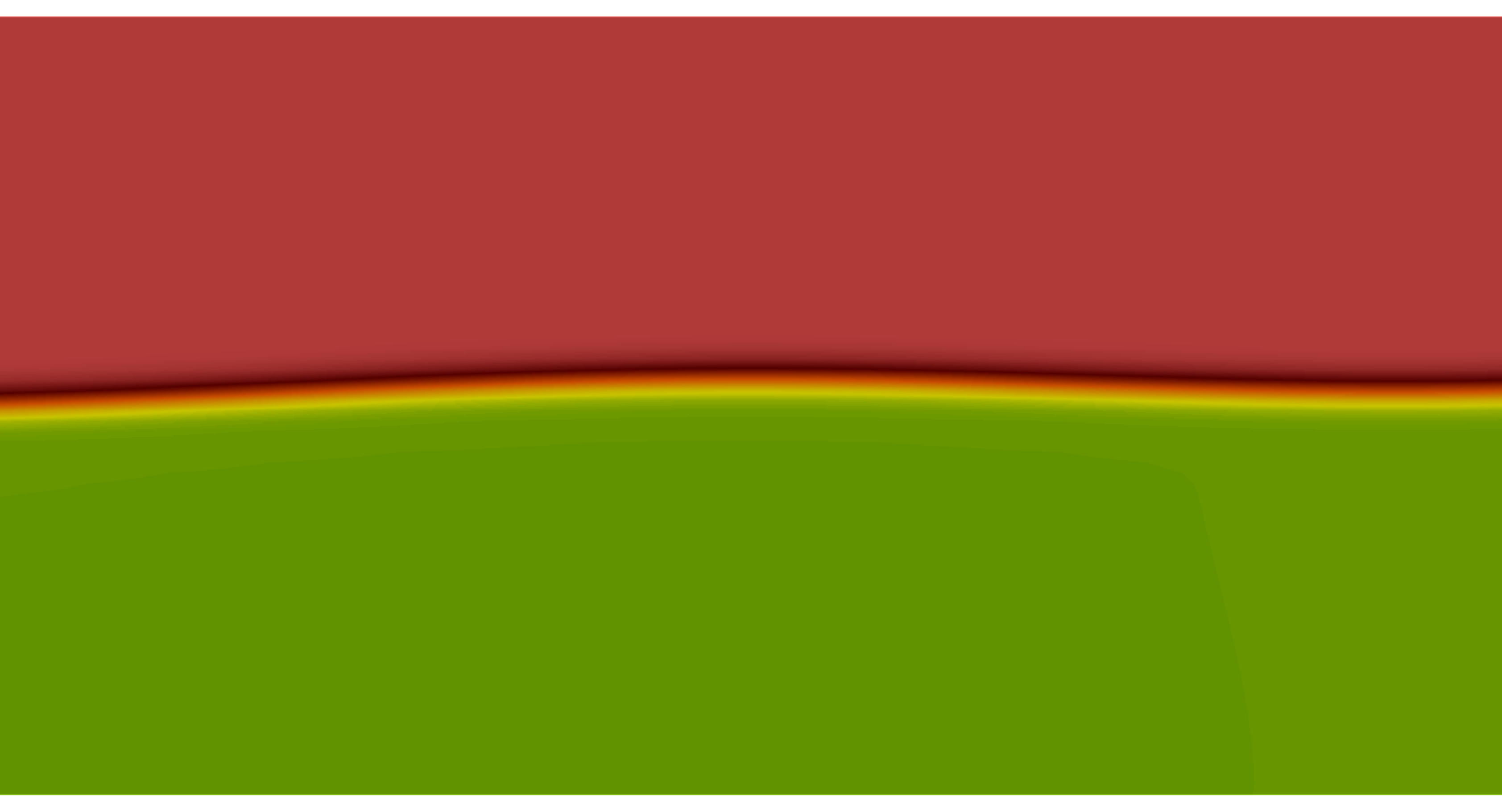}
\includegraphics[scale=0.07]{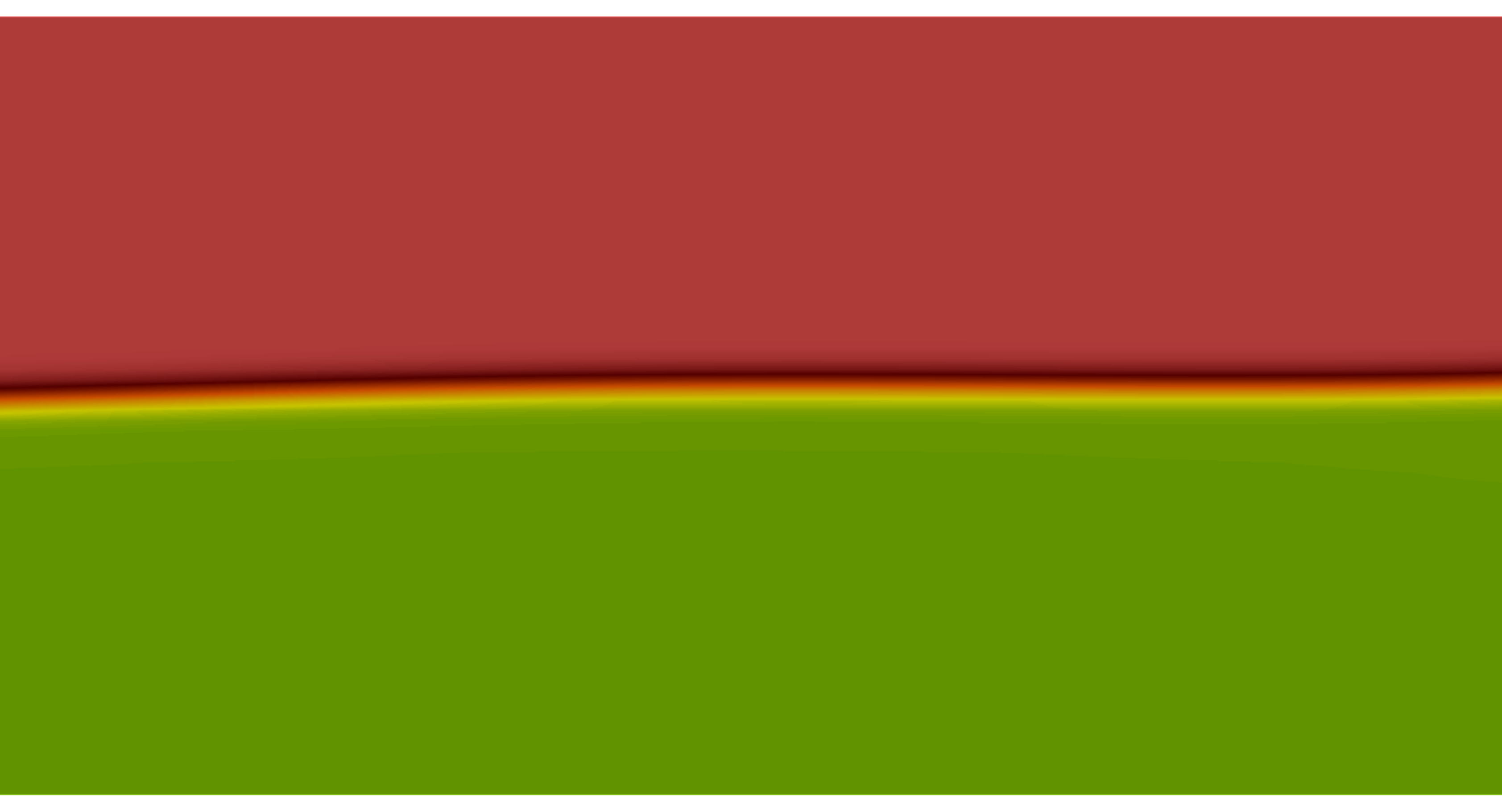}
\\
\includegraphics[scale=0.07]{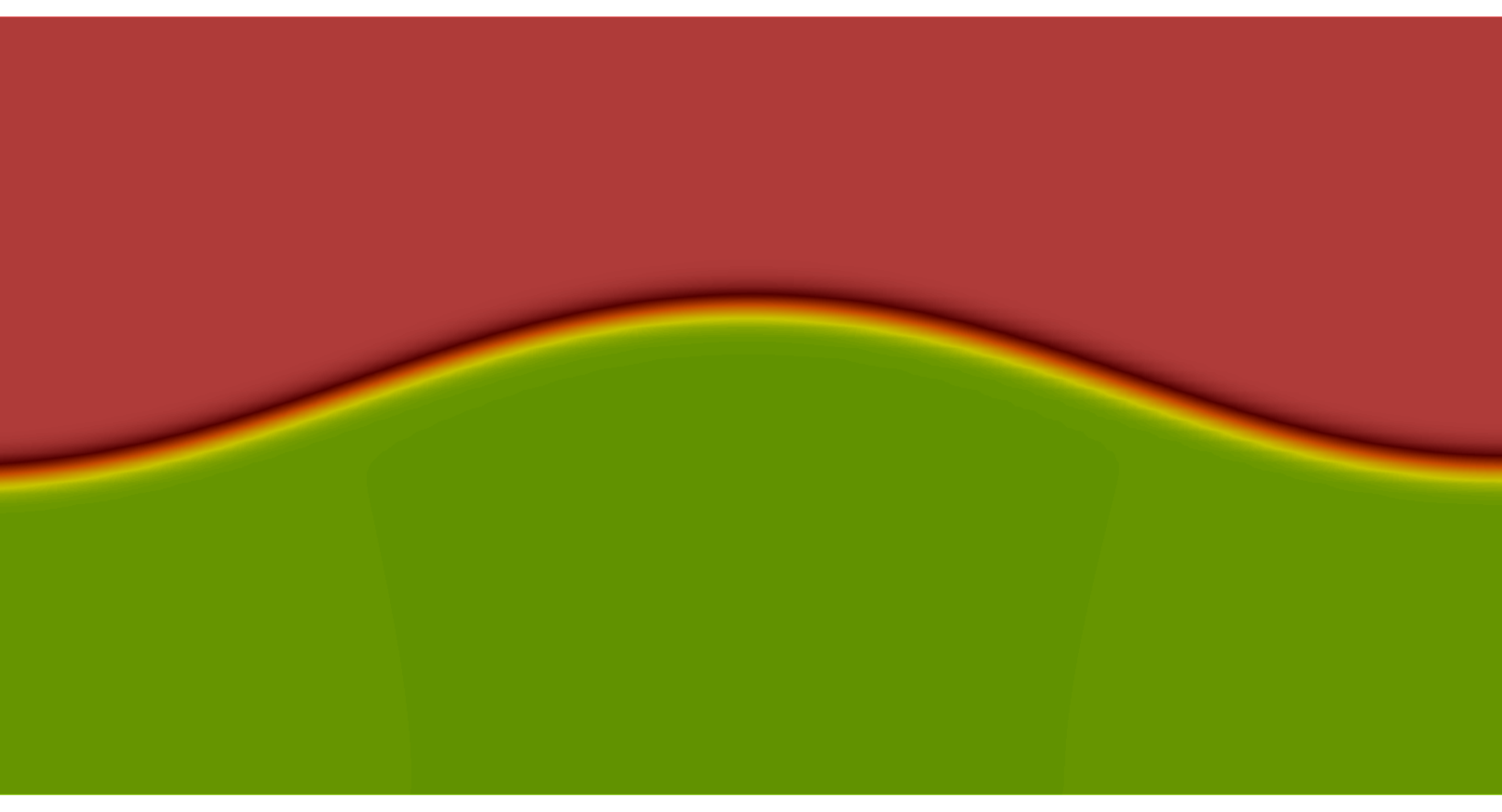}
\includegraphics[scale=0.07]{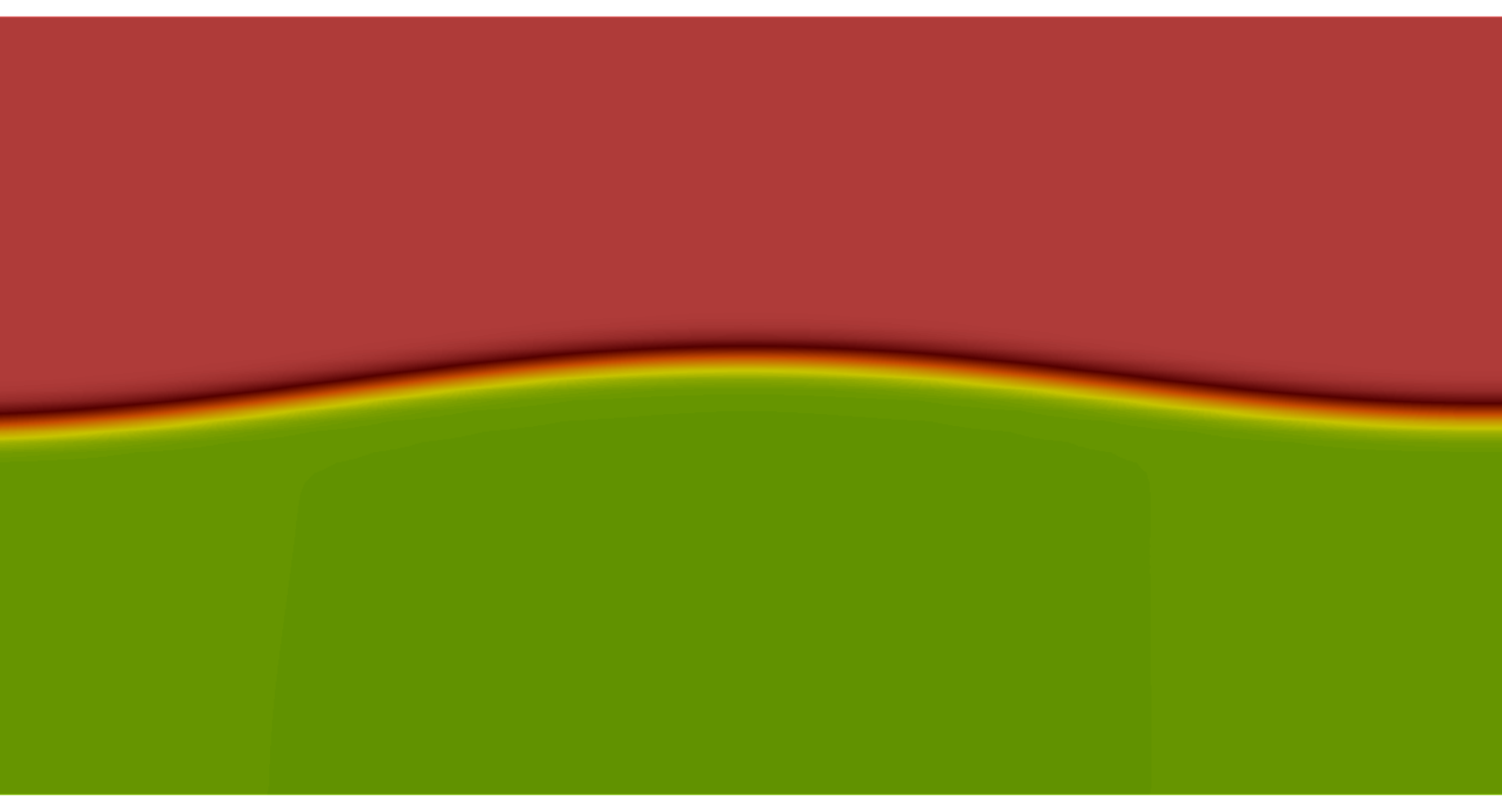}
\includegraphics[scale=0.07]{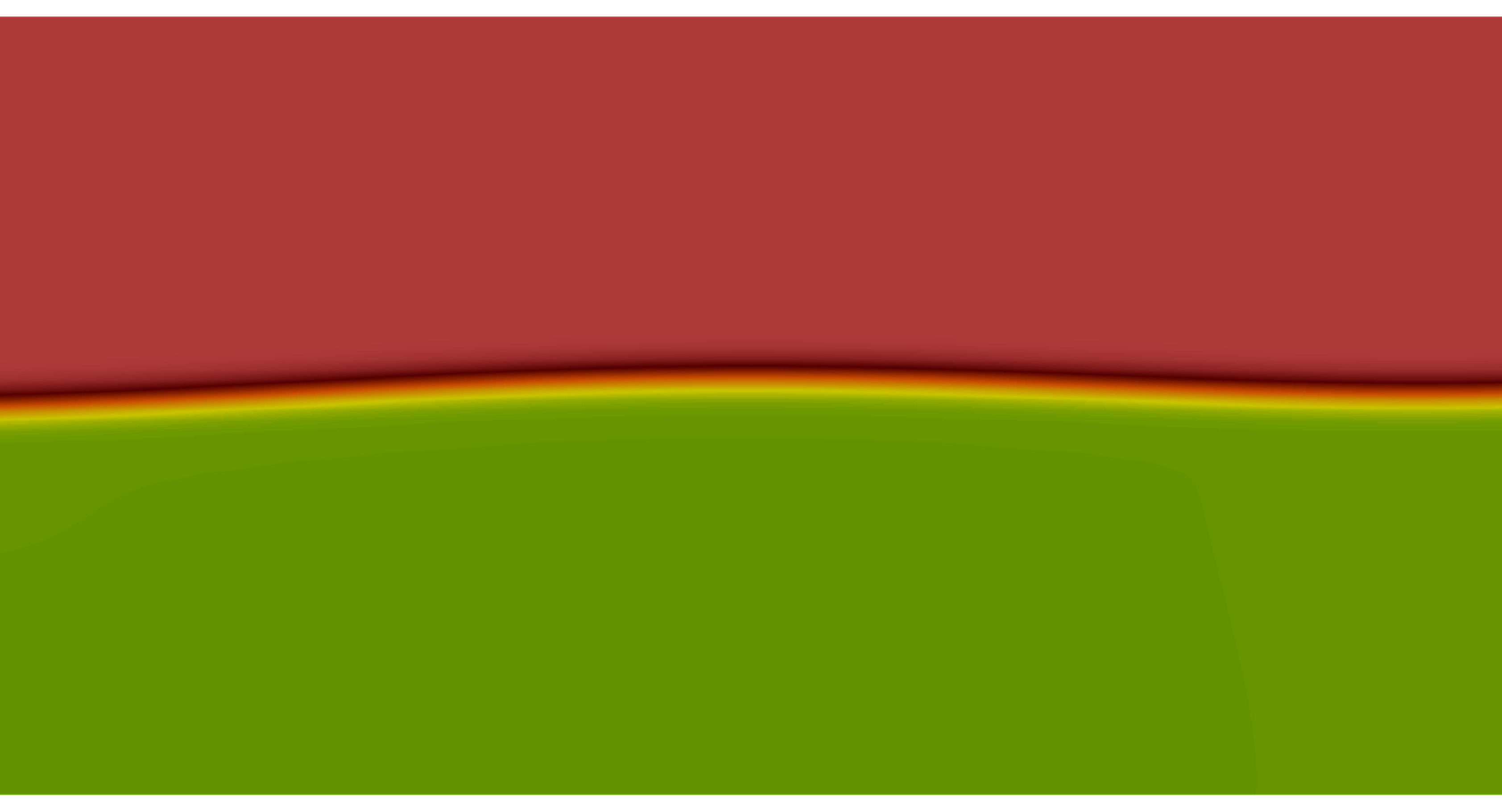}
\includegraphics[scale=0.07]{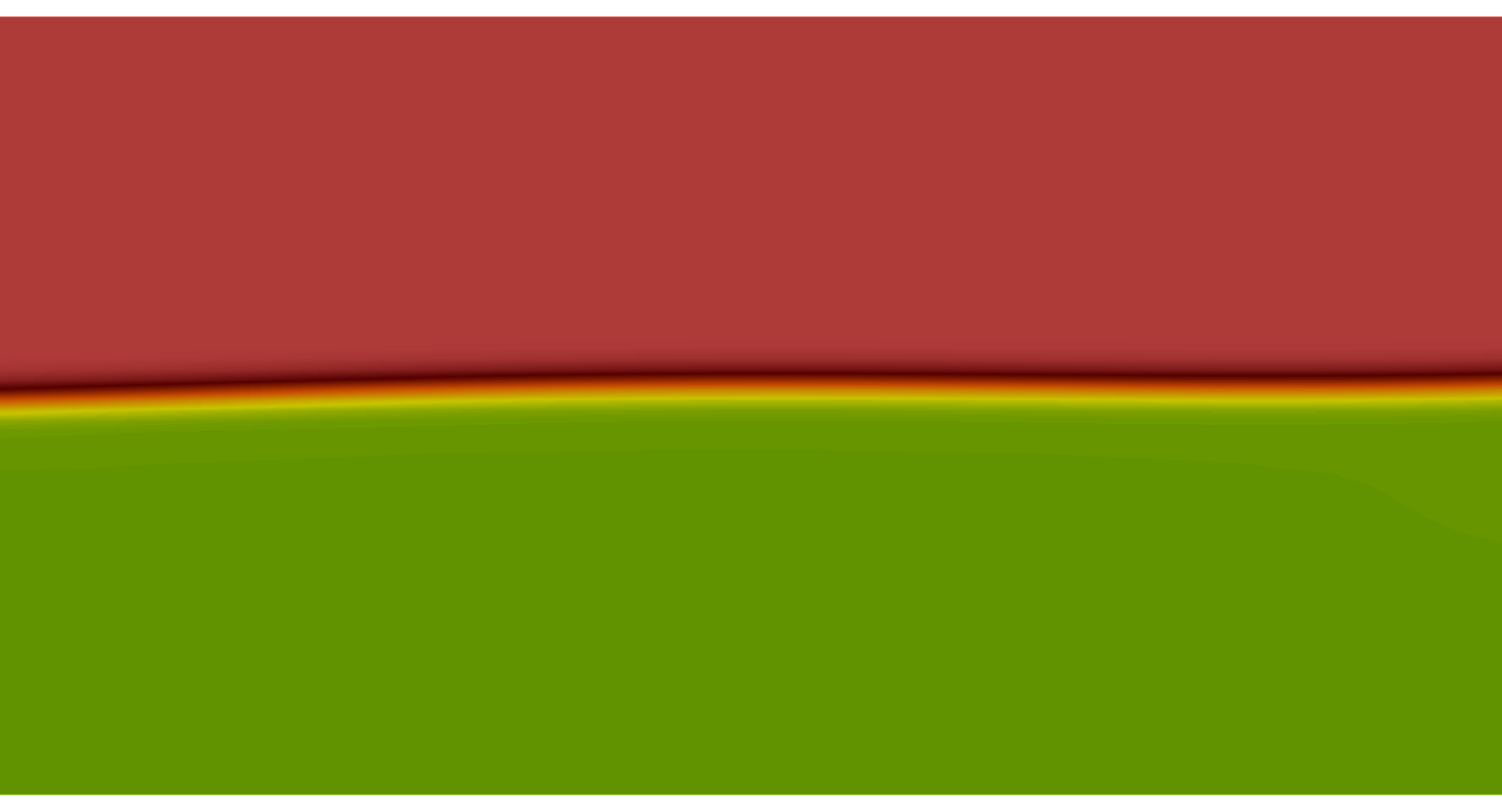}
\end{center}
\caption{Dynamics of schemes TD1 (top row), NTD1 (center row) and NTC2 (bottom row) at times $t=0.5, 1, 1.5$ and $2$ (from left to right) with spreading coefficients $(\Sigma_1, \Sigma_2 , \Sigma_3) = (1,1,1)$.}
\label{fig:lensCase0Dyn}
\end{figure}

\begin{figure}[h]
\begin{center}
\includegraphics[scale=0.095]{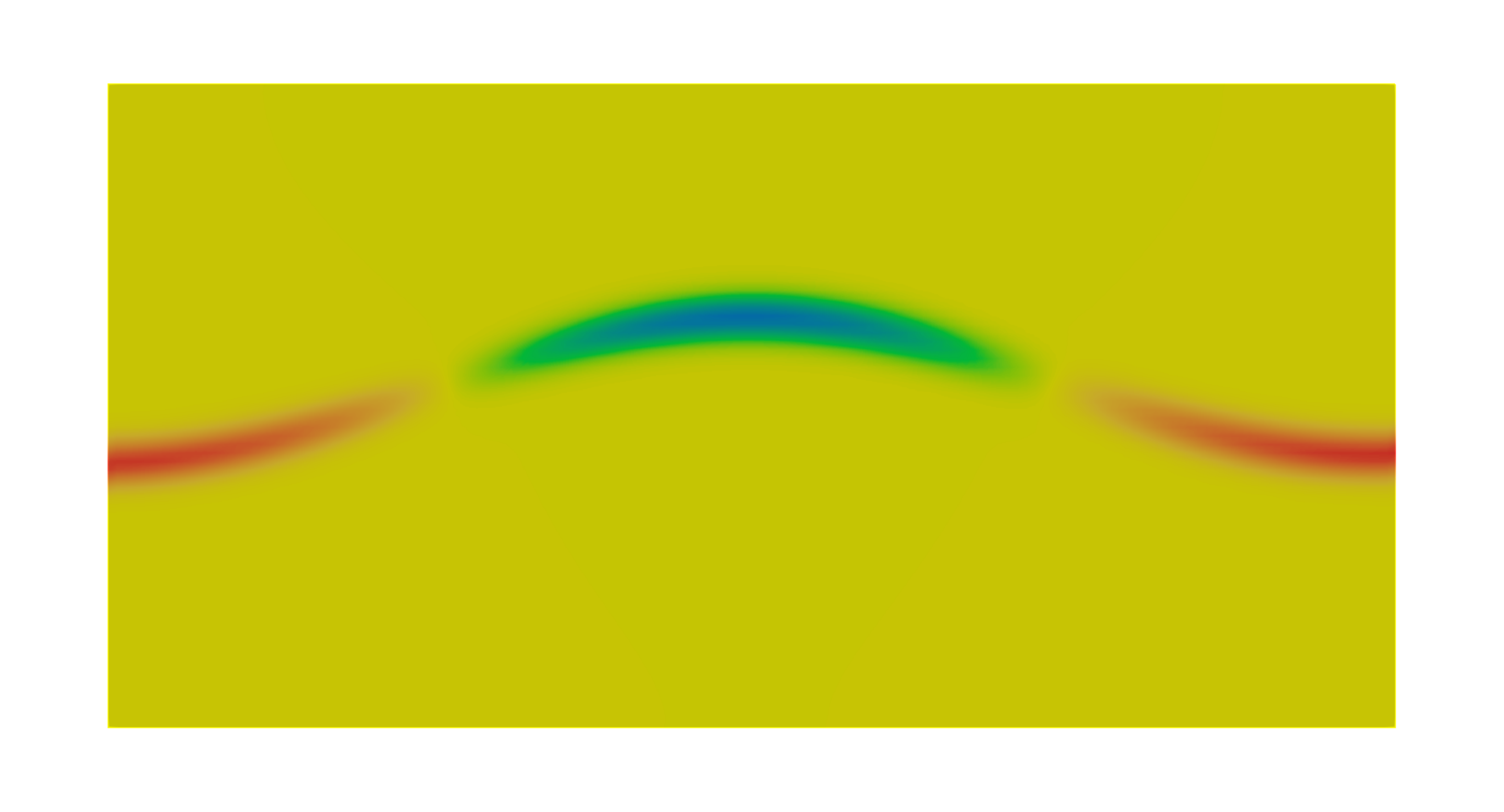}
\includegraphics[scale=0.095]{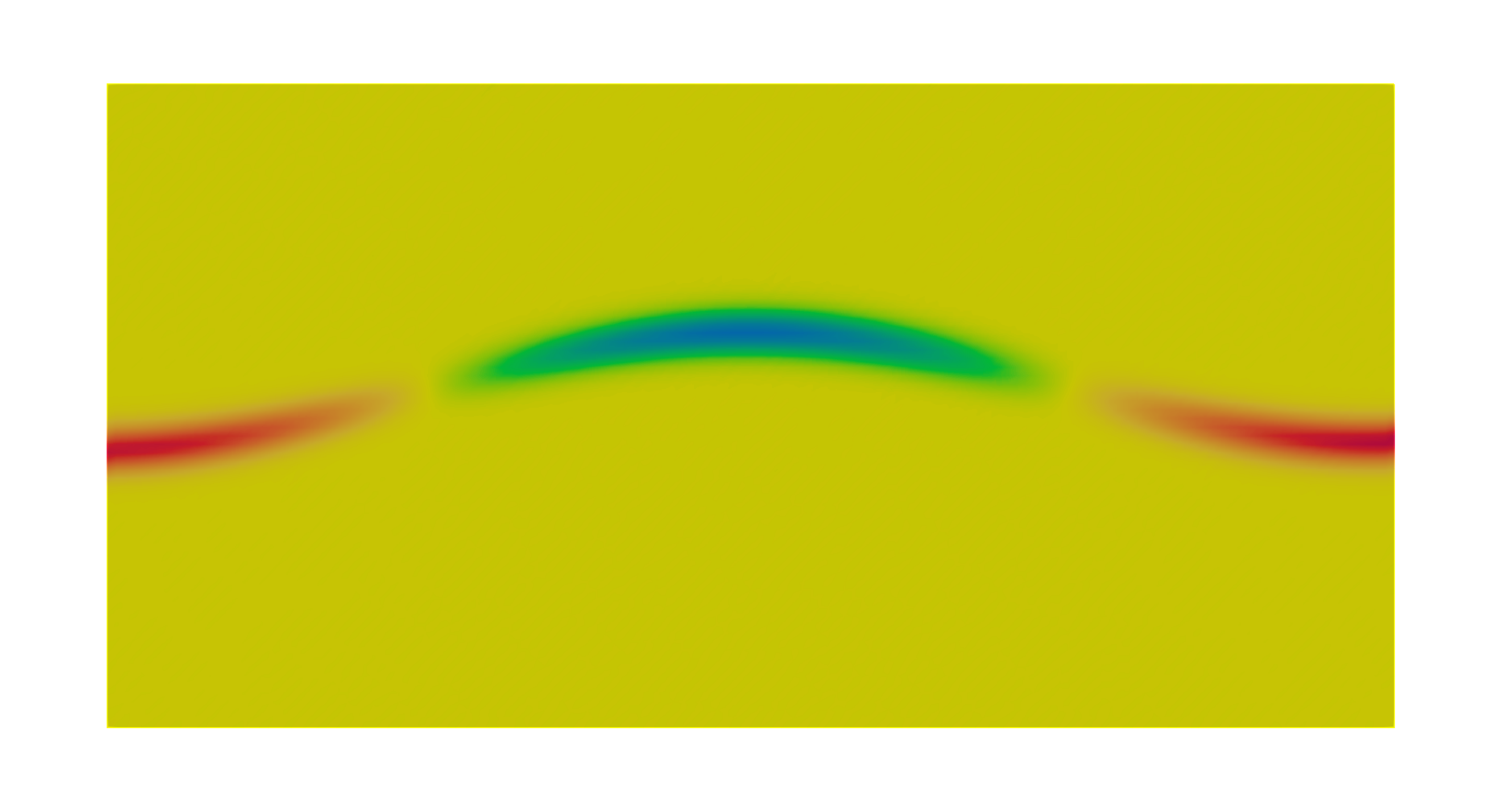}
\includegraphics[scale=0.095]{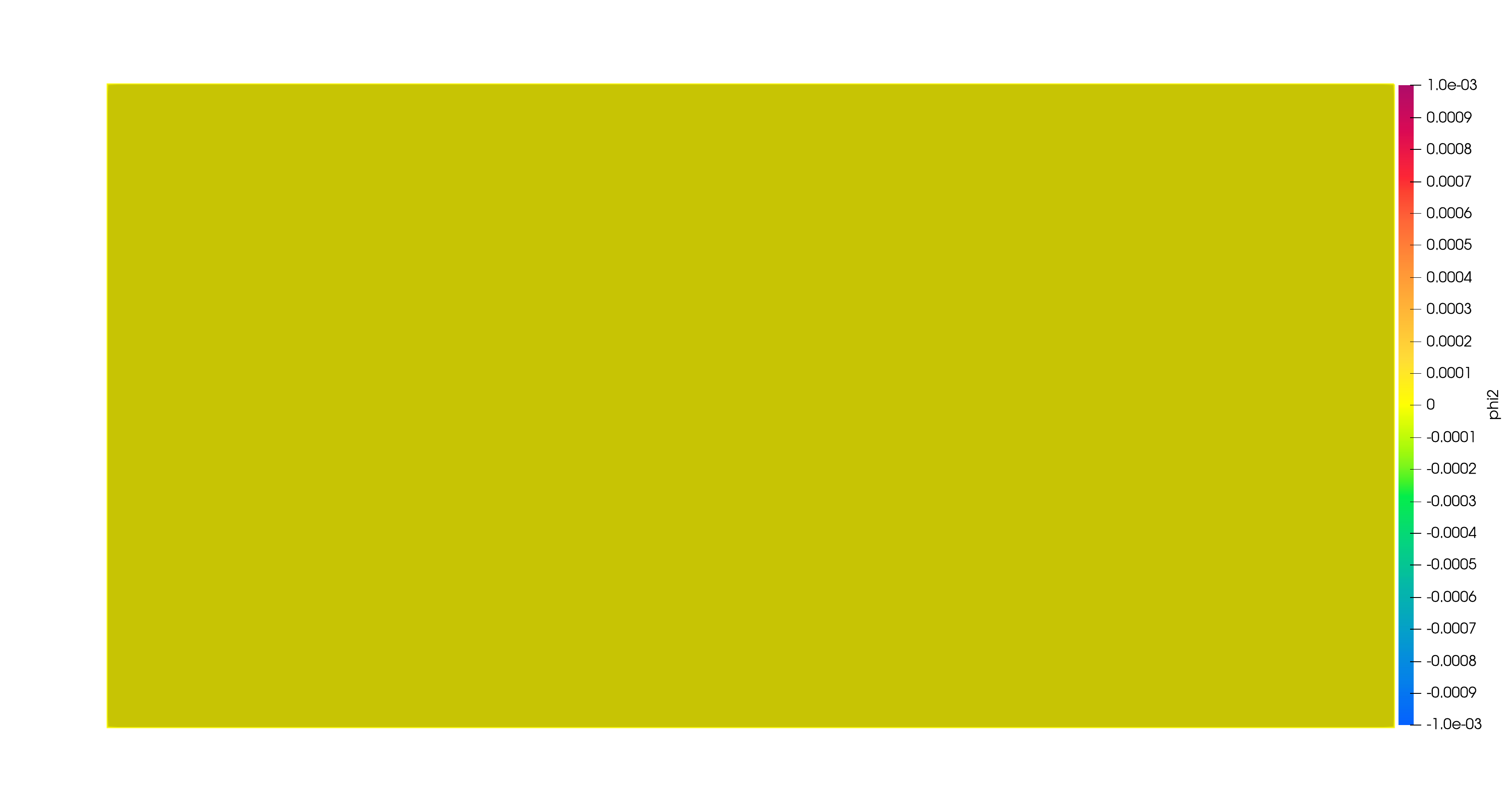}
\end{center}
\caption{Comparison of $\phi_2$ at time $t=0.5$ for schemes TD1 (left), NTD1 (center) and NTC2 (right)}\label{fig:lensCase0phi2}
\end{figure}

\begin{figure}[h]
\begin{center}
\includegraphics[scale=0.07]{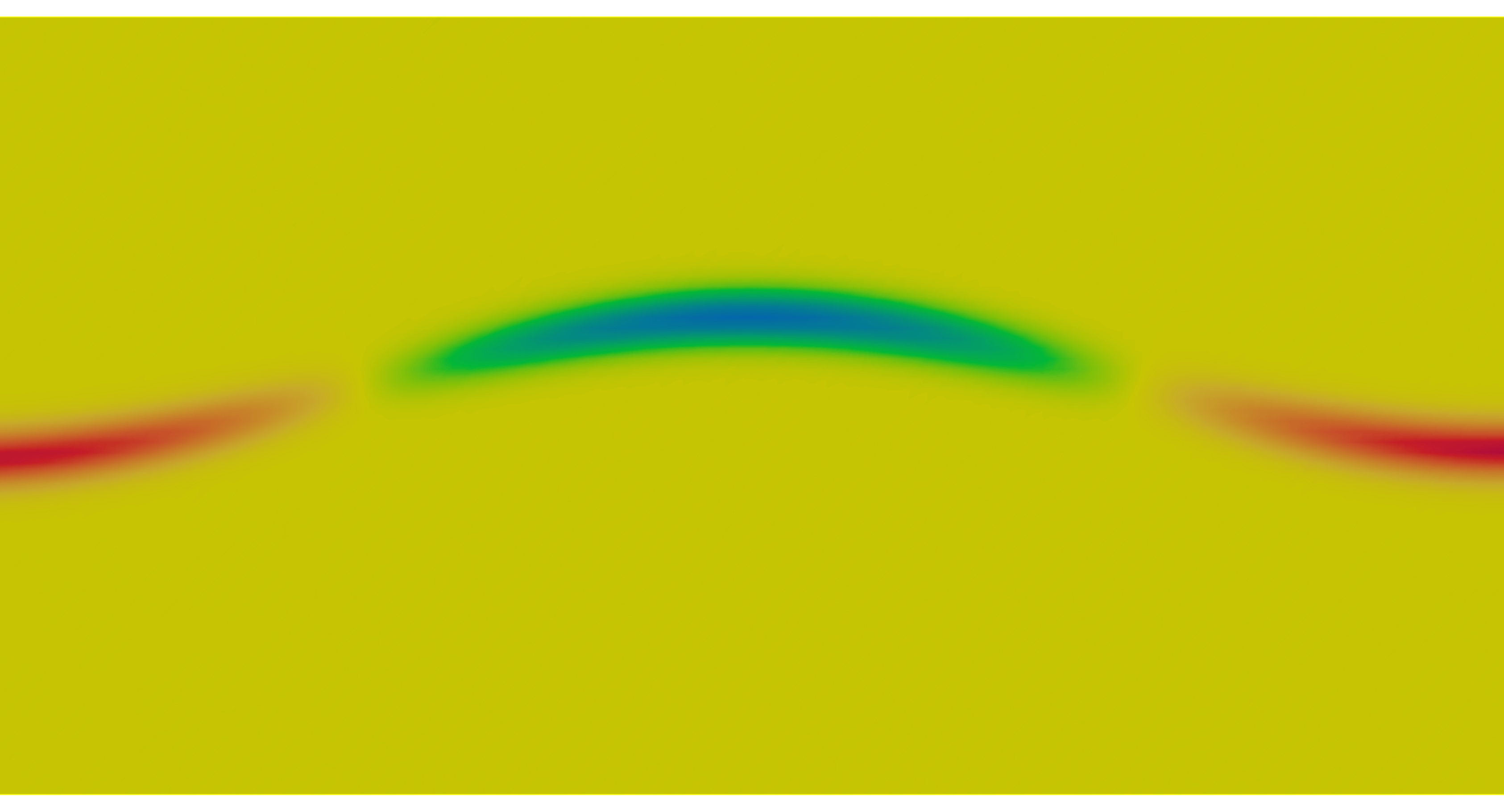}
\includegraphics[scale=0.07]{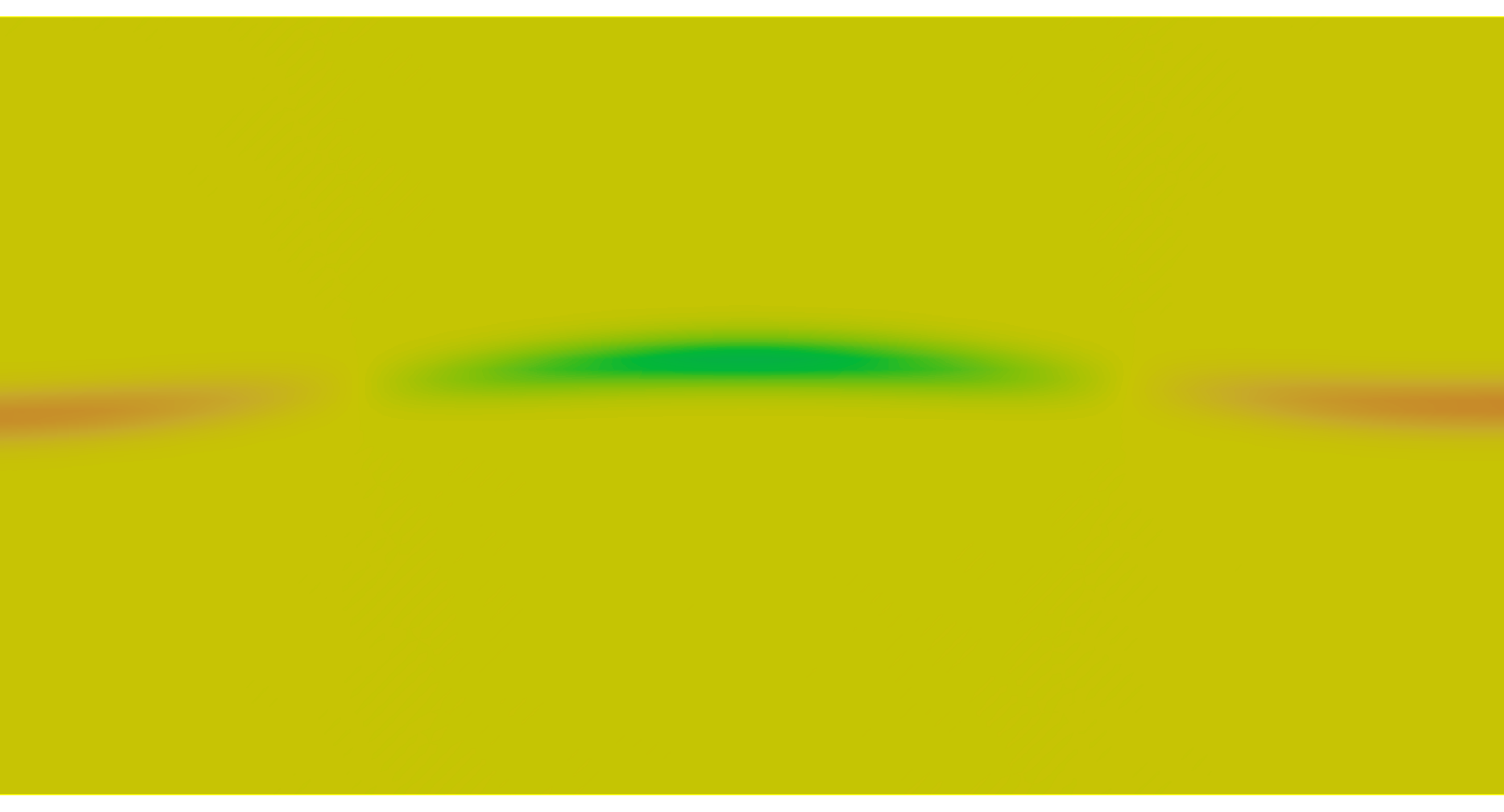}
\includegraphics[scale=0.07]{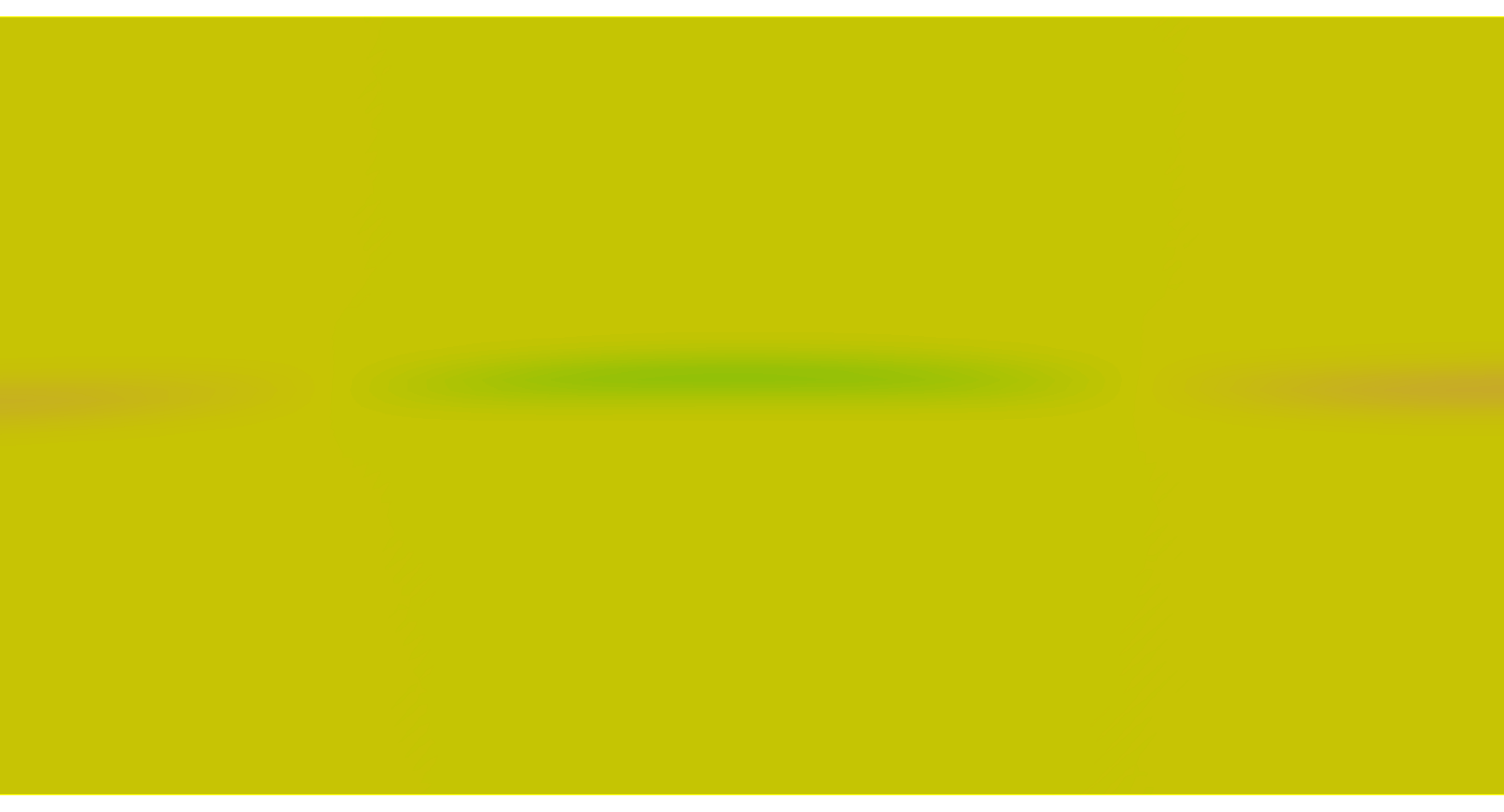}
\includegraphics[scale=0.07]{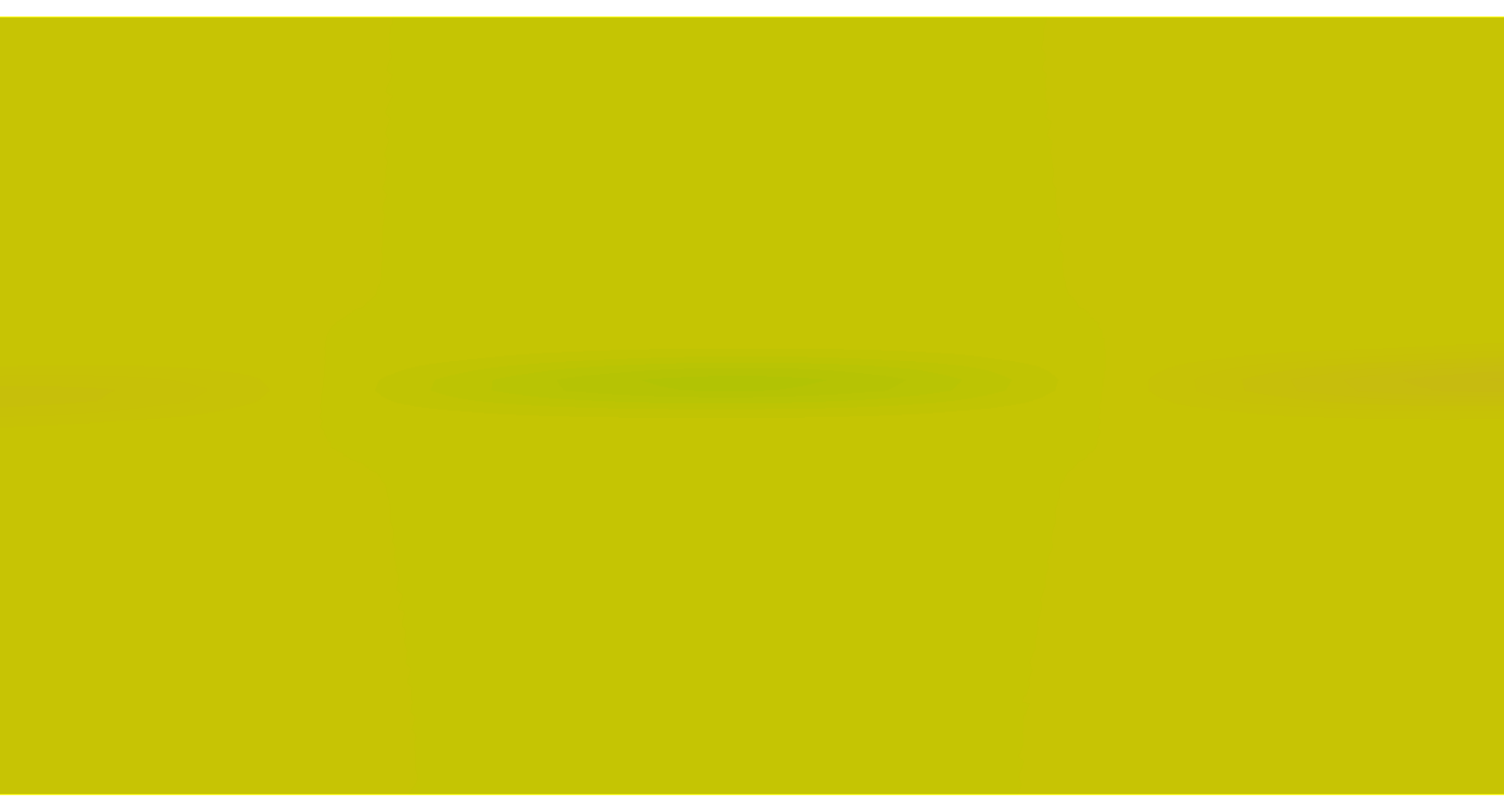}
\end{center}
\caption{Dynamics of $\phi_2$ for scheme NTD1 at times $t=0.5, 1, 1.5$ and $2$ (from left to right) with spreading coefficients $(\Sigma_1, \Sigma_2 , \Sigma_3) = (1,1,1)$.}
\label{fig:lensCase0NTD1phi2}
\end{figure}

\begin{figure}[h]
\begin{center}
\includegraphics[scale=0.11]{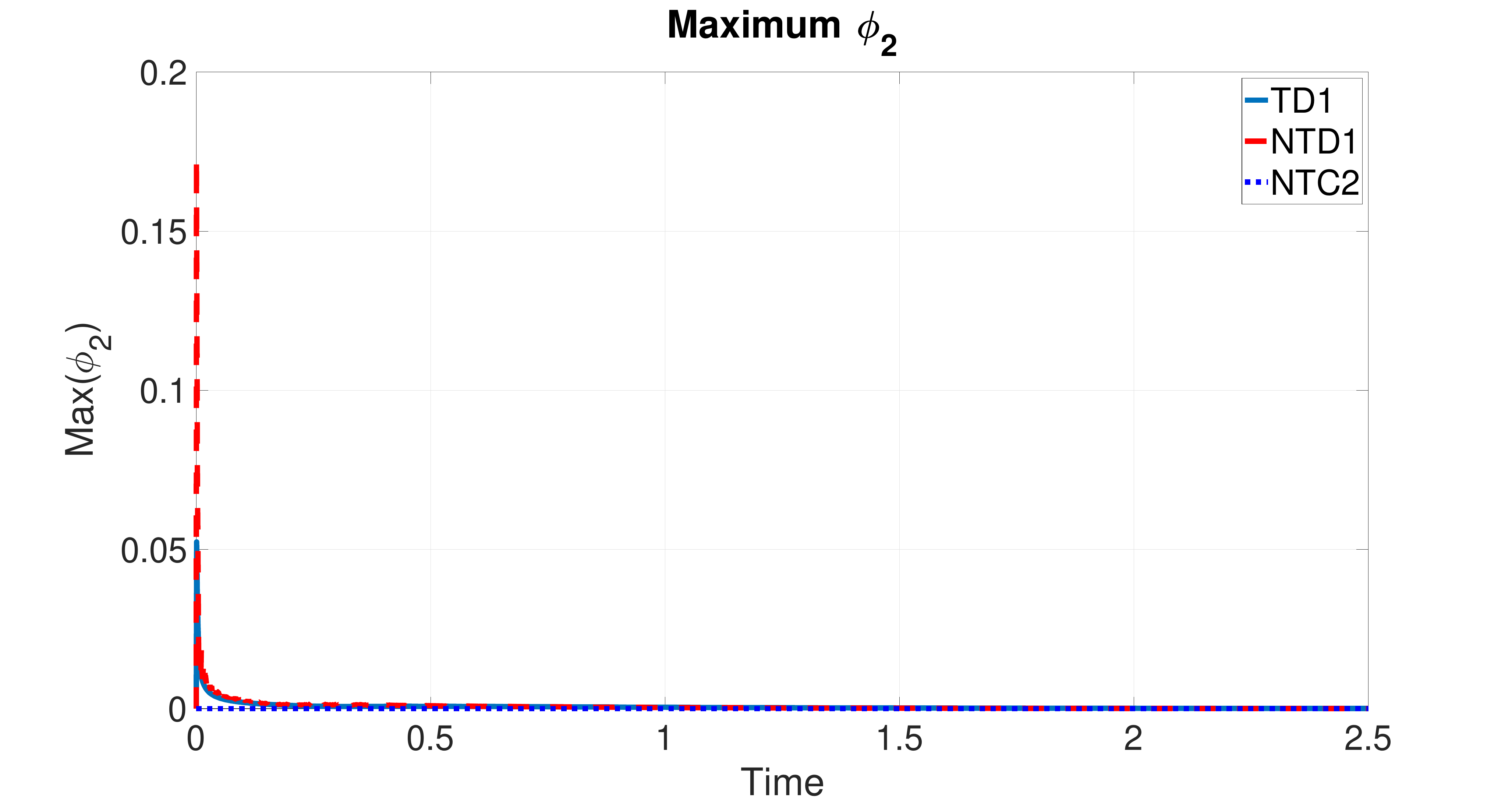}
\includegraphics[scale=0.11]{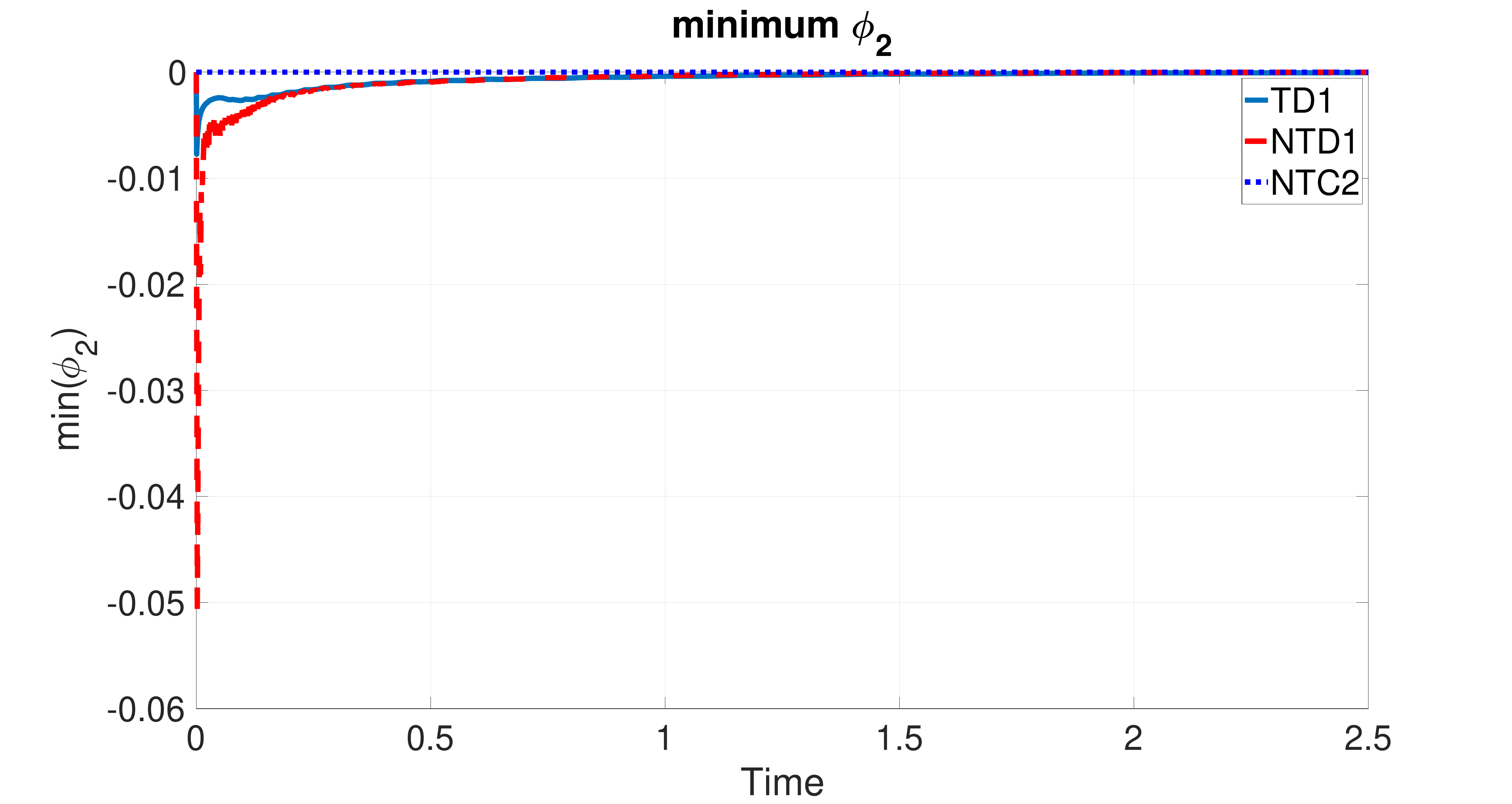}
\end{center}
\caption{Evolution in time of $\max(\phi_2)$ and $\min(\phi_2)$ for NTD1. }
\label{fig:lensCase0NTD1phi2maxmin}
\end{figure}

\subsection{Two Bubbles Suspended in a Third Phase}\label{sec:2bubbles}

In the experiments in this section we consider as initial condition two circular concentrated phase regions suspended in a third phase and observe the effect of surface tension on the dynamics, comparing two cases of partial spreading and two cases of total spreading. Moreover, we focus only on using scheme NTD1 since as has been seen in previous sections, it is the most computationally efficient. We consider the experimental parameters given in Table~\ref{tab:ballsParameters} and the initial condition described in \eqref{eq:twoBubblesInitial} and presented in Figure~\ref{fig:intialBalls}.
\begin{table}[h]
{
\begin{tabular}{c|c|c|c|c|c|c|c}
$\Omega$ 			& $h$ 		& $[0,T]$ 		& $\dt$ 	& $\eps$ & $\lambda$ 	& $M$ 	& $\Lambda$ \\
\hline
$[-0.125, 0.125]\times[-0.125,0.125]$ 	& 1/300	& $[0,0.5]$ 	& 1e-4 	& 1e-2  & 1e-4 	&  1e-3 	& 7 \\
\end{tabular}
}
\caption{\label{tab:ballsParameters} Parameters for the two bubbles suspended in a third phase experiments.}
\end{table}

\beq\label{eq:twoBubblesInitial}
\left\{
\ba{rcl}
\phi^0_1(x,y) &=&\dis \frac12 
-\frac12\tanh\left(\frac2\eps \sqrt{ (x - 0.035)^2 + y^2 } - 0.035 \right)\,, 
\\ \hueco
\phi^0_2(x,y) &=&\dis \frac12 
-\frac12\tanh\left(\frac2\eps \sqrt{ (x + 0.035)^2 + y^2 } - 0.035 \right)\,, 
\\ \hueco
\phi^0_3(x,y) &=&\dis 1 - \phi_1 - \phi_2\,.
\ea
\right. 
\eeq

\begin{figure}[h]
\begin{center}
\includegraphics[scale=0.125]{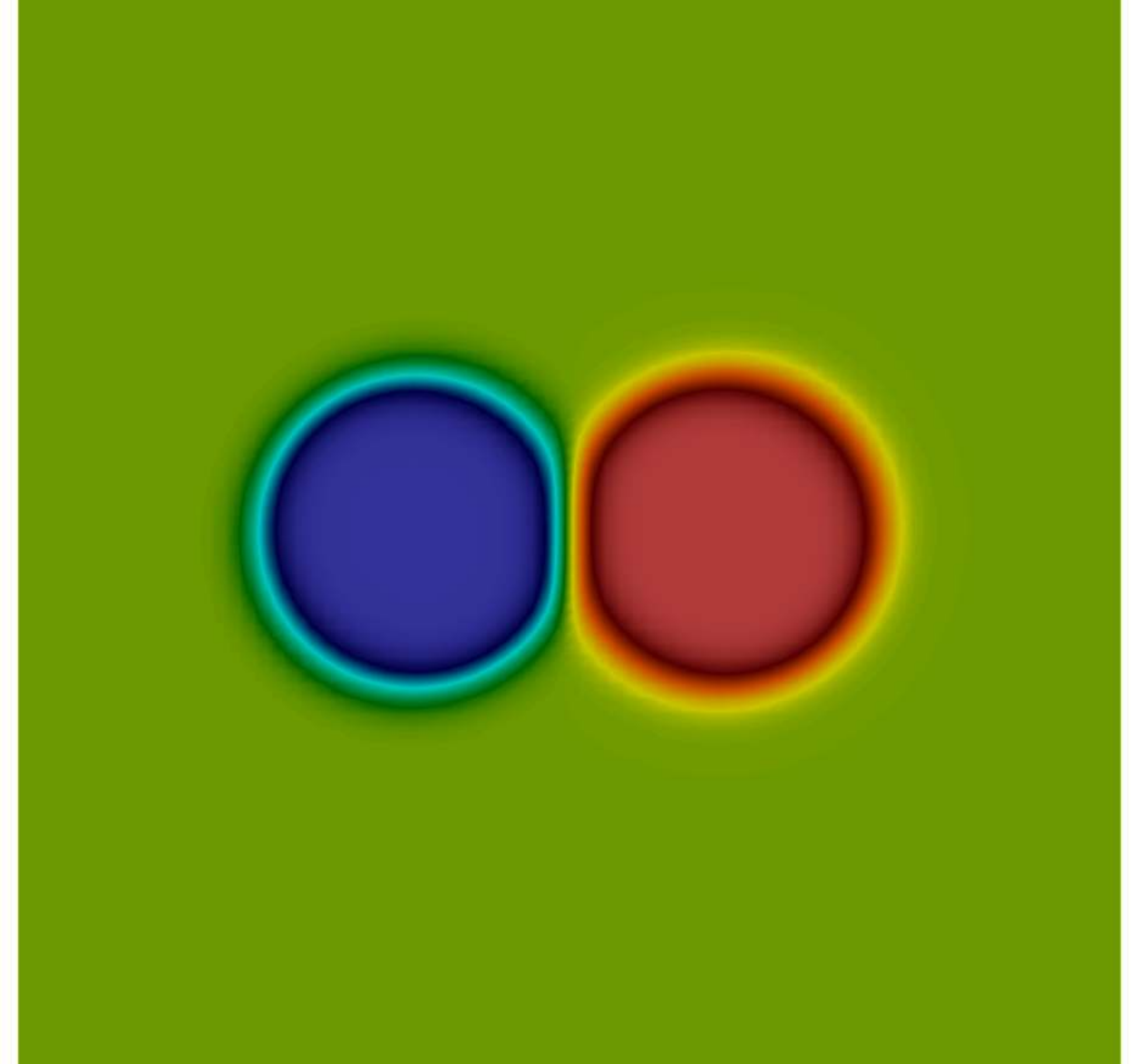}
\includegraphics[scale=0.125]{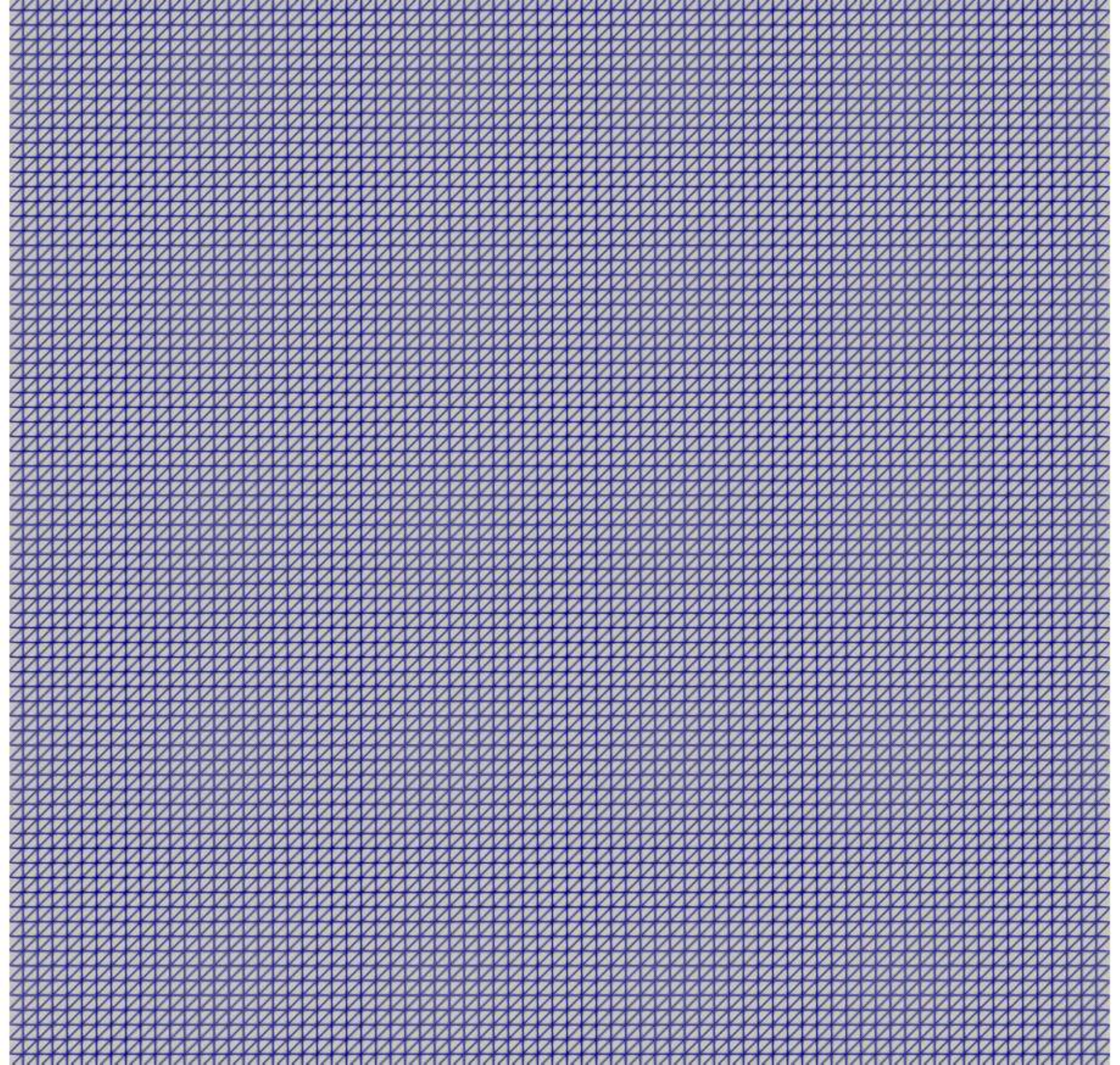}
\end{center}
\caption{Initial condition for the two bubbles suspended in a third phase experiments and the considered mesh. Red color represents phase $\phi_1$, blue color represents phase $\phi_2$ and green color represents $\phi_3$.}\label{fig:intialBalls}
\end{figure}

{
The dynamics of the four cases are presented in Figure~\ref{fig:BallsDynamics}.  
In the top row we consider the case $(\Sigma_1, \Sigma_2 , \Sigma_3) = (1,1,1)$ and we observe how the equal values of the parameters $\Sigma_i$ produces that the boundary between the red phase ($\phi_1$) and blue phase ($\phi_2$) is flat. The choice $(\Sigma_1, \Sigma_2 , \Sigma_3) = (0.4, 1.6, 1.2)$ is presented  in the second row where the unequal surface tensions force the system to develop an asymmetric interface. In the third row we consider the total spreading case with $(\Sigma_1, \Sigma_2 , \Sigma_3) = (3,3,-0.1)$ and because $\Sigma_3<0$, the green phase ($\phi_3$) wants to be between the other two phase, making its way to separate the bubbles. Finally, in the bottom row we present the interesting dynamics obtained under the choice $(\Sigma_1, \Sigma_2 , \Sigma_3) = (-0.1,3,3)$, where the negative spreading coefficient is given to $\phi_1$ (red phase) and therefore it spreads between the other two phases by engulfing the blue bubble.
\\
In Figure~\ref{fig:BallsPlots} we present the evolution in time of the energy, the volume and the $L^2$ and $L^\infty$ norms of the restriction $\Sigma_{i=1}^3\phi_i - 1$. In all the cases considered the energy decreases until the system reaches an equilibrium state (i.e., constant energy in time) and the volume is conserved as expected. As in previous examples, the $L^2$ norm of the restriction is of order $10^{-3}$ and the $L^\infty$ norm is of order $10^{-2}$ while there are points of the domain where the three components interact (except initially, when the system is adapting to the initial condition and the approximation of the restriction is a bit worse), but in all cases the obtained dynamics seems reasonable and not affected by the non-exact conservation of the restriction.
}

\begin{figure}[h]
\begin{center}
\includegraphics[scale=0.09]{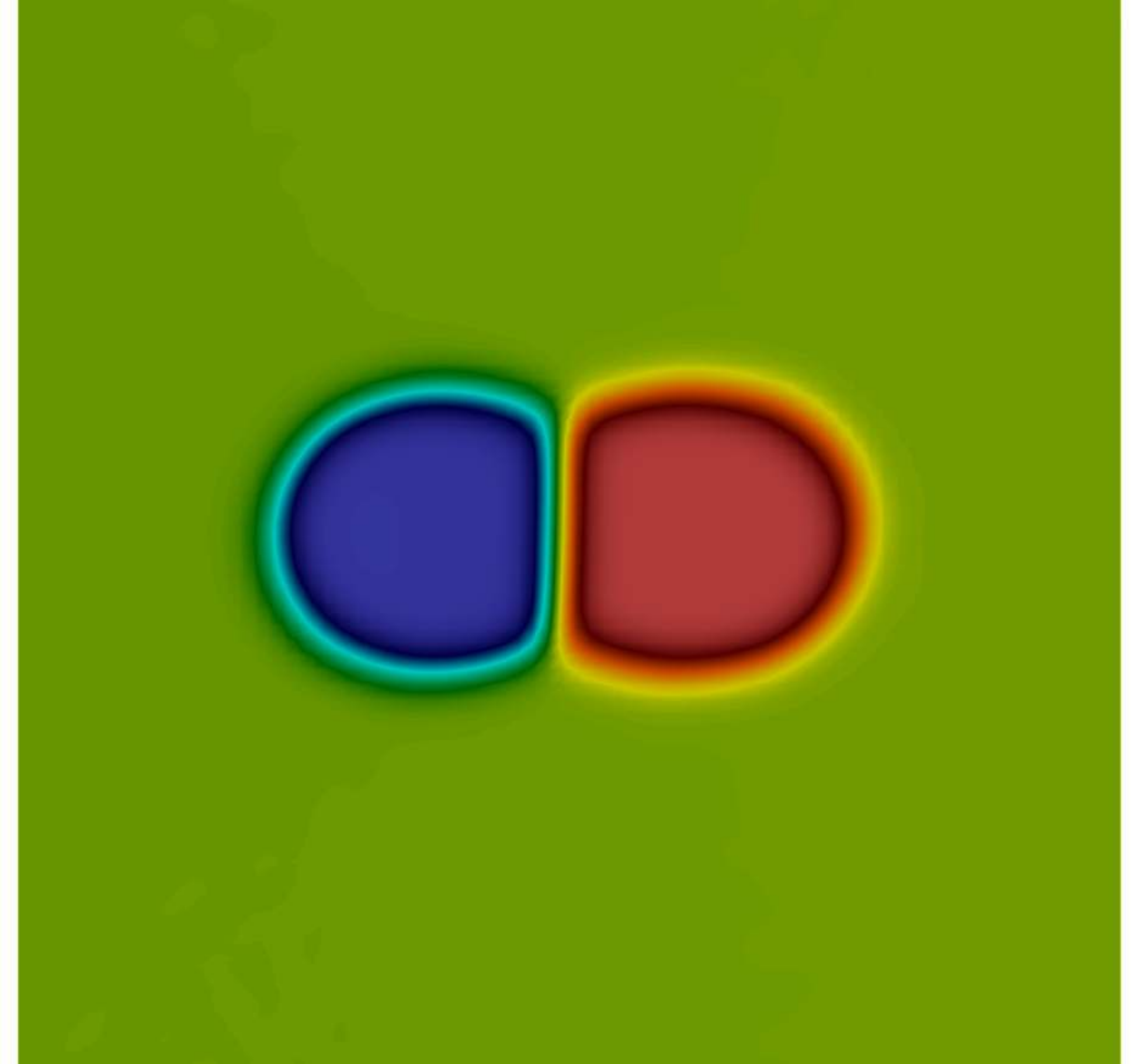}
\includegraphics[scale=0.09]{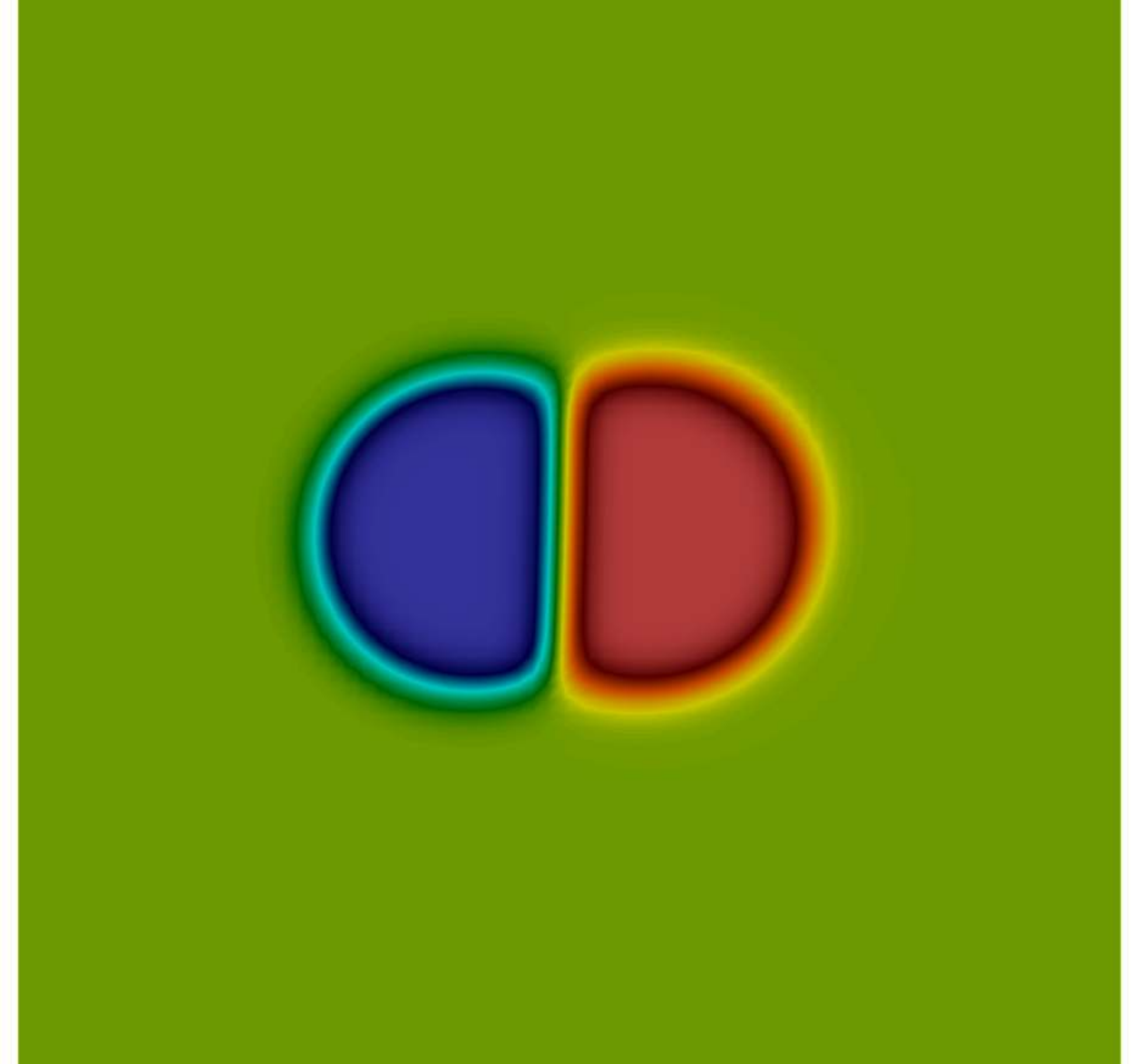}
\includegraphics[scale=0.09]{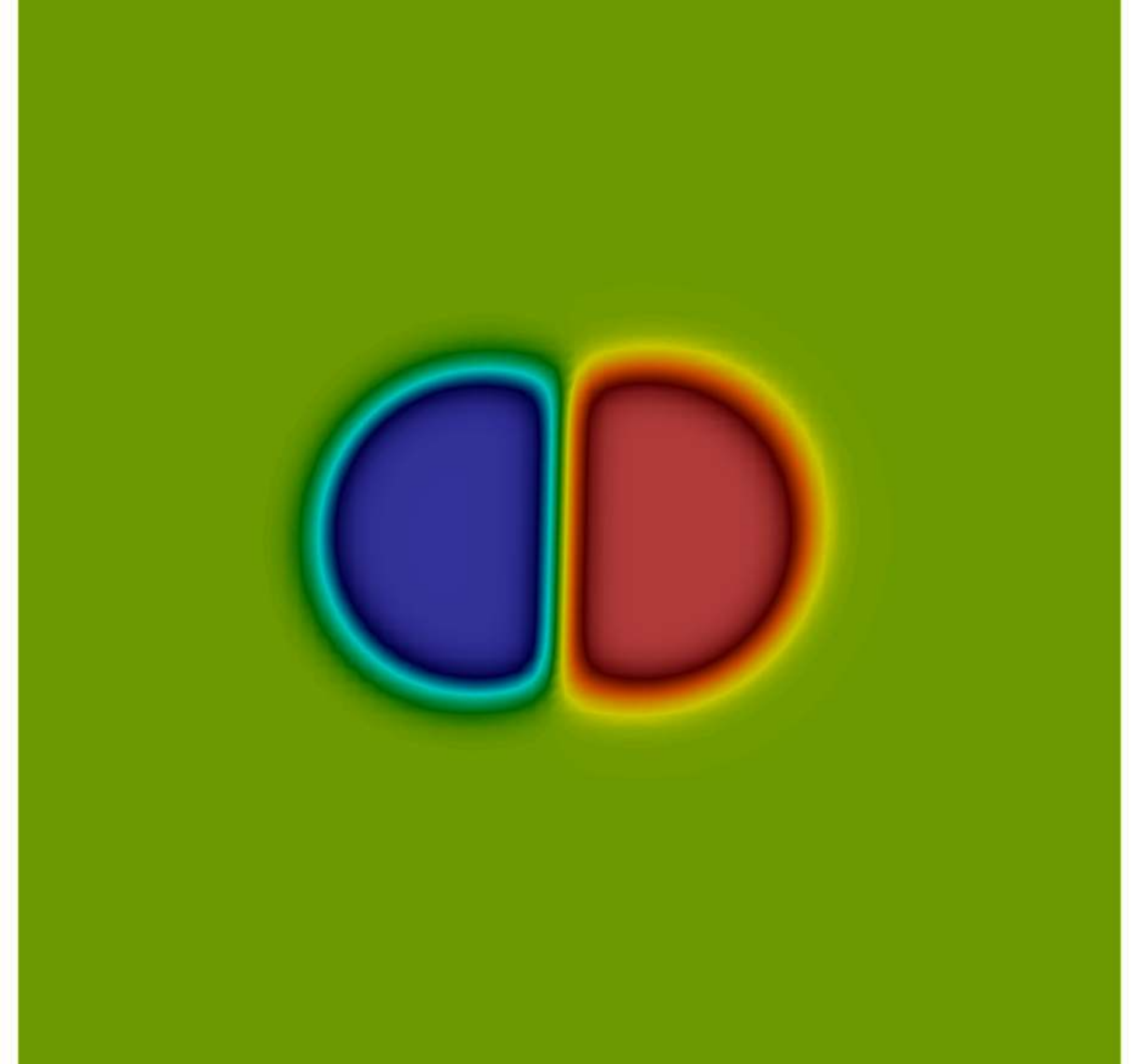}
\includegraphics[scale=0.09]{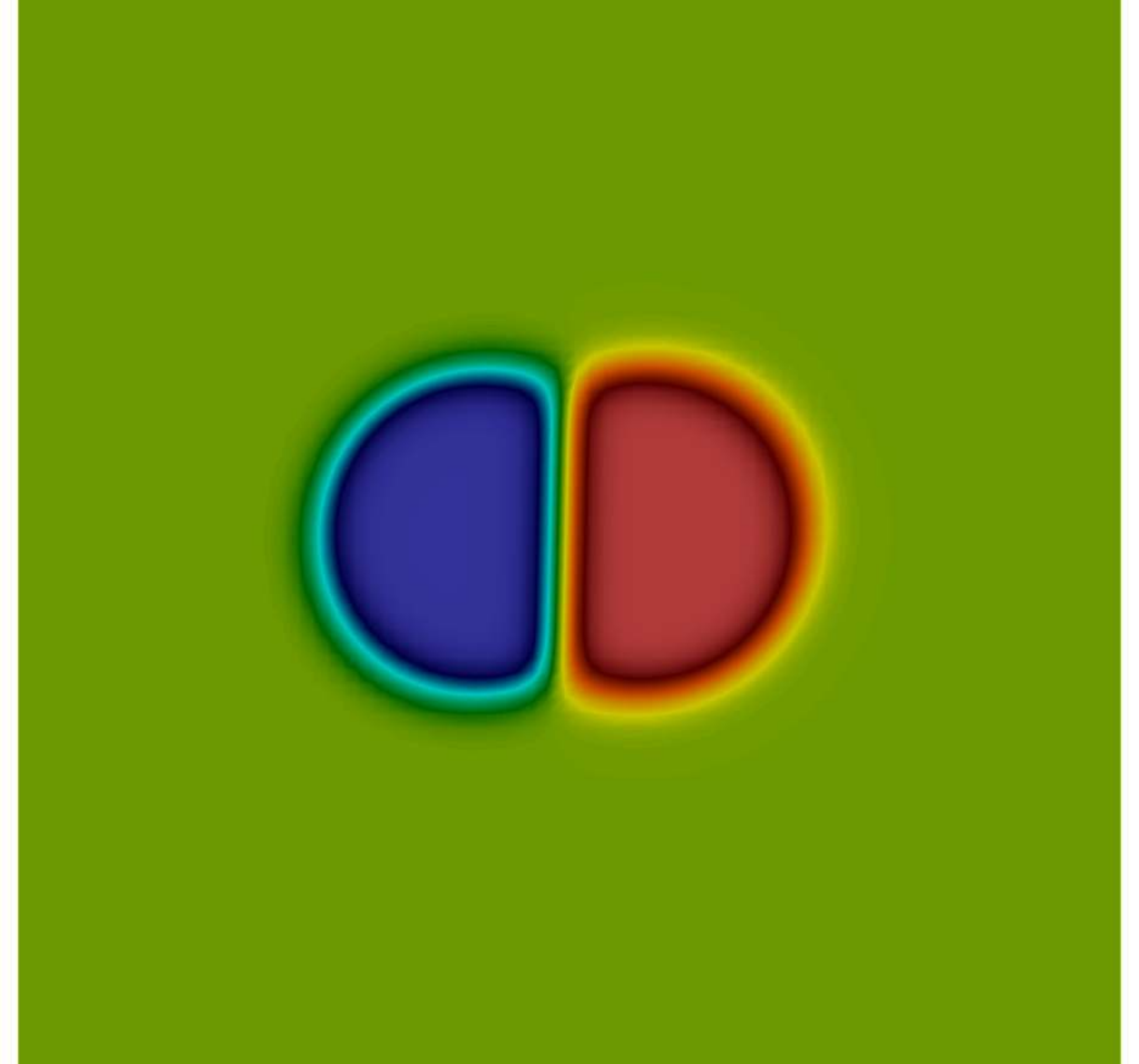}
\includegraphics[scale=0.09]{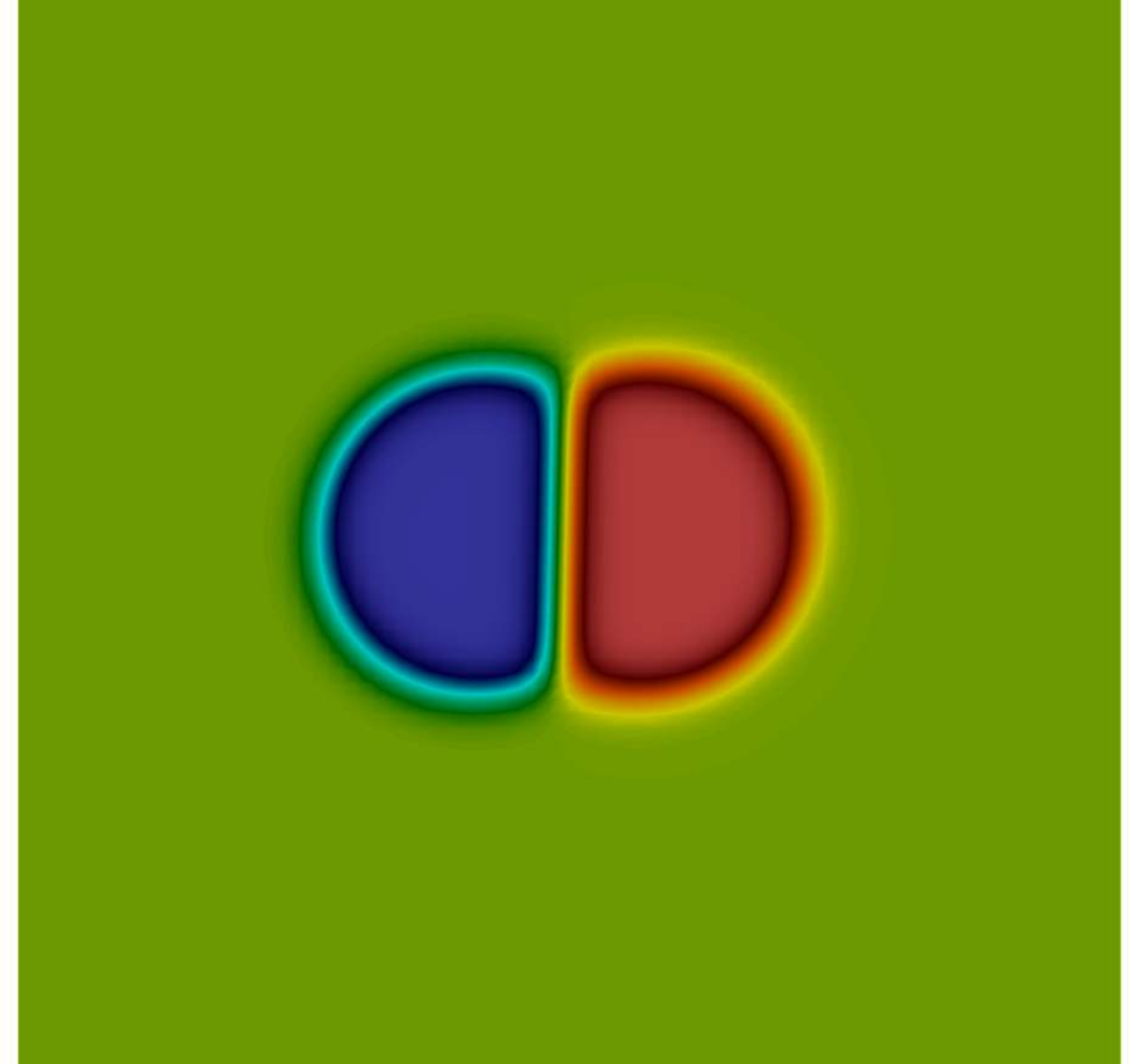}
\\ [1ex]
\includegraphics[scale=0.09]{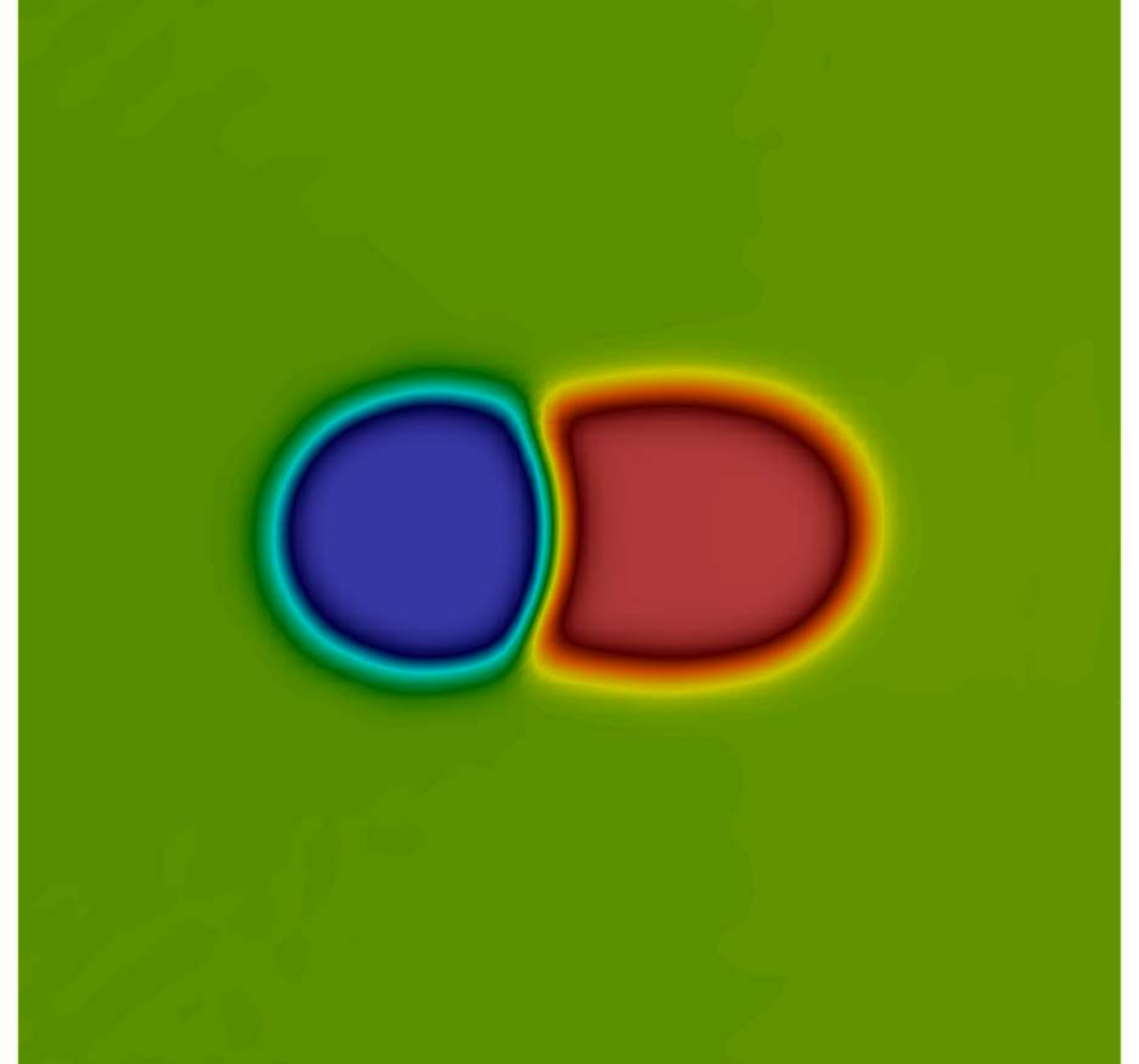}
\includegraphics[scale=0.09]{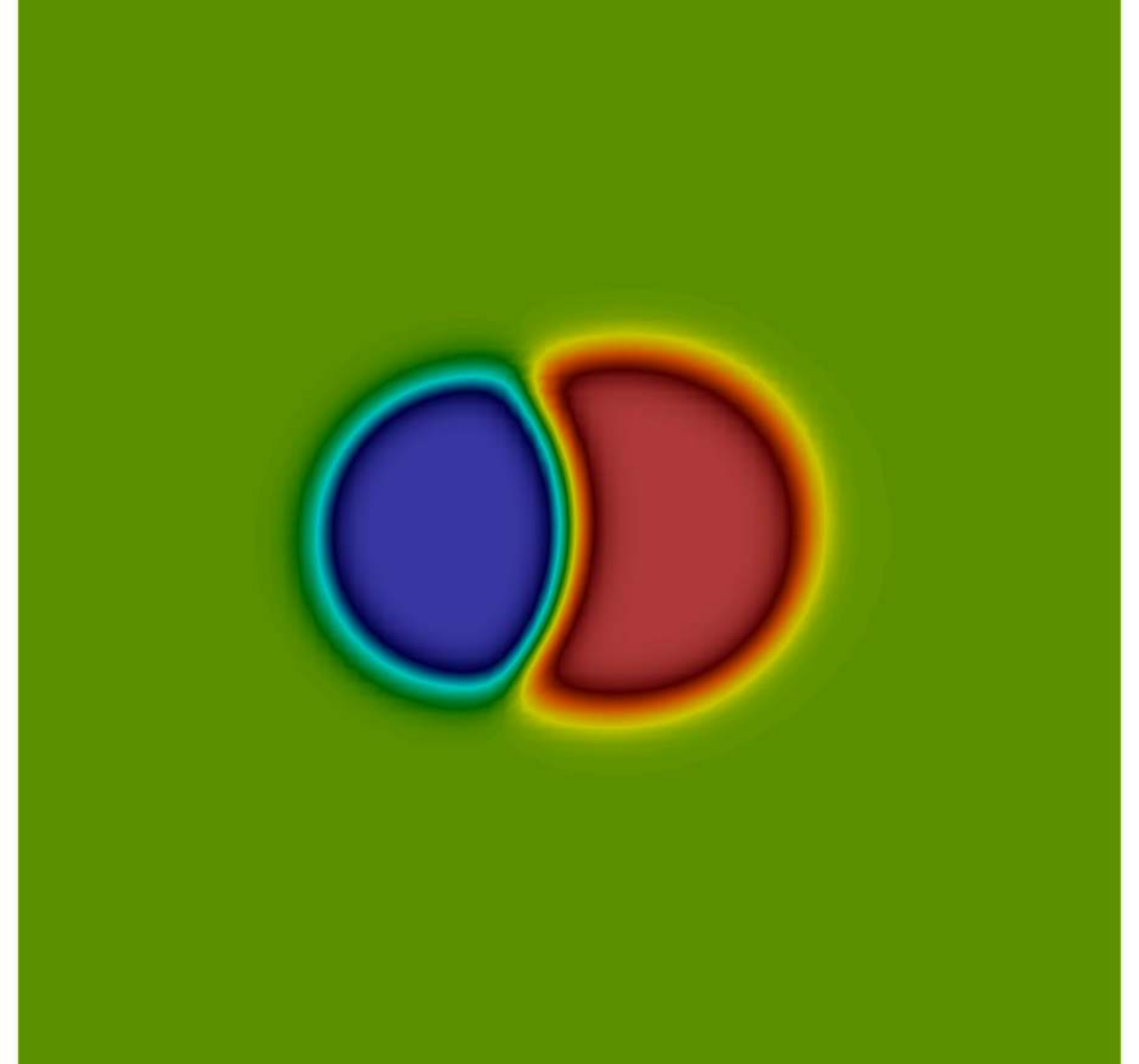}
\includegraphics[scale=0.09]{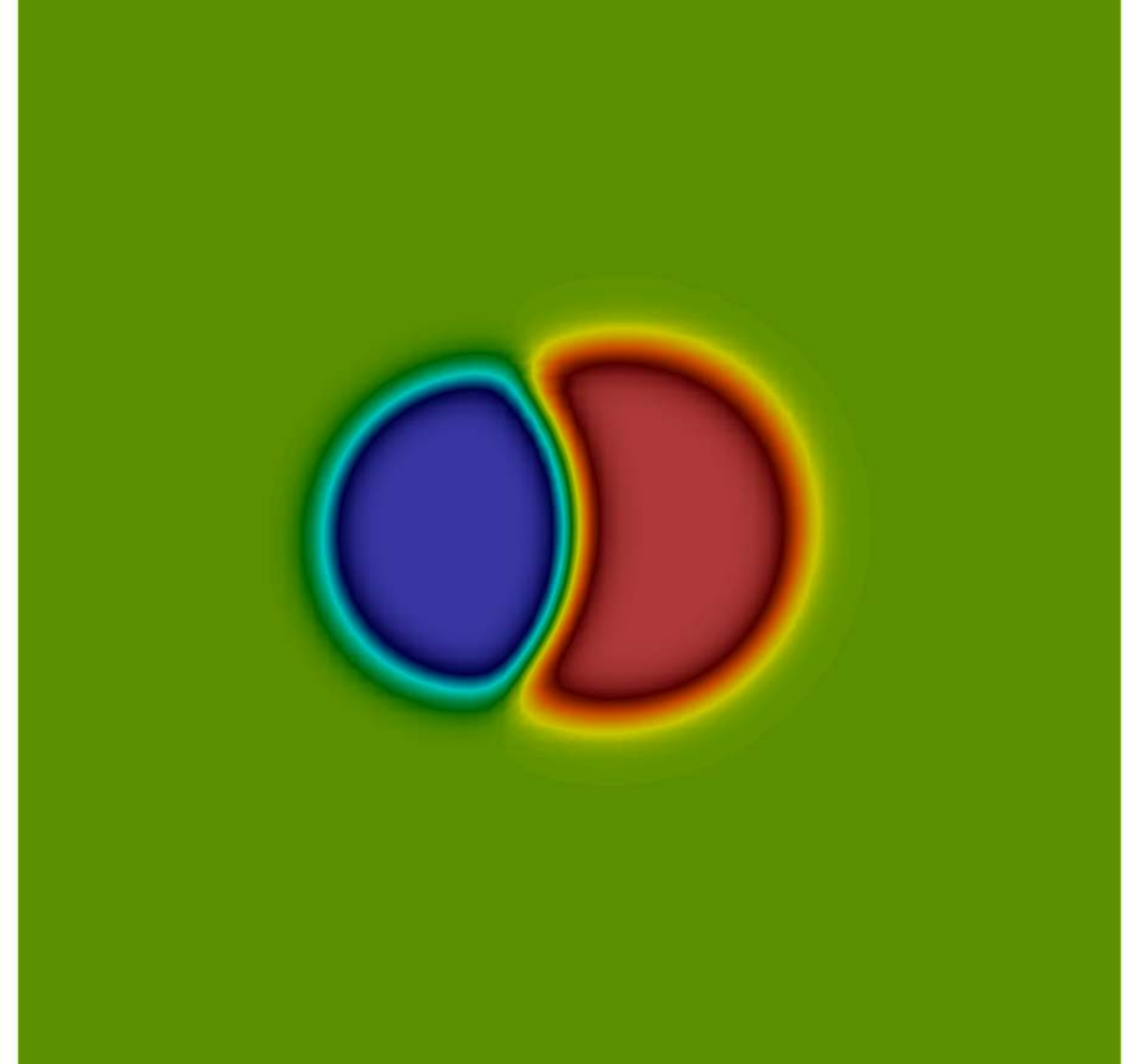}
\includegraphics[scale=0.09]{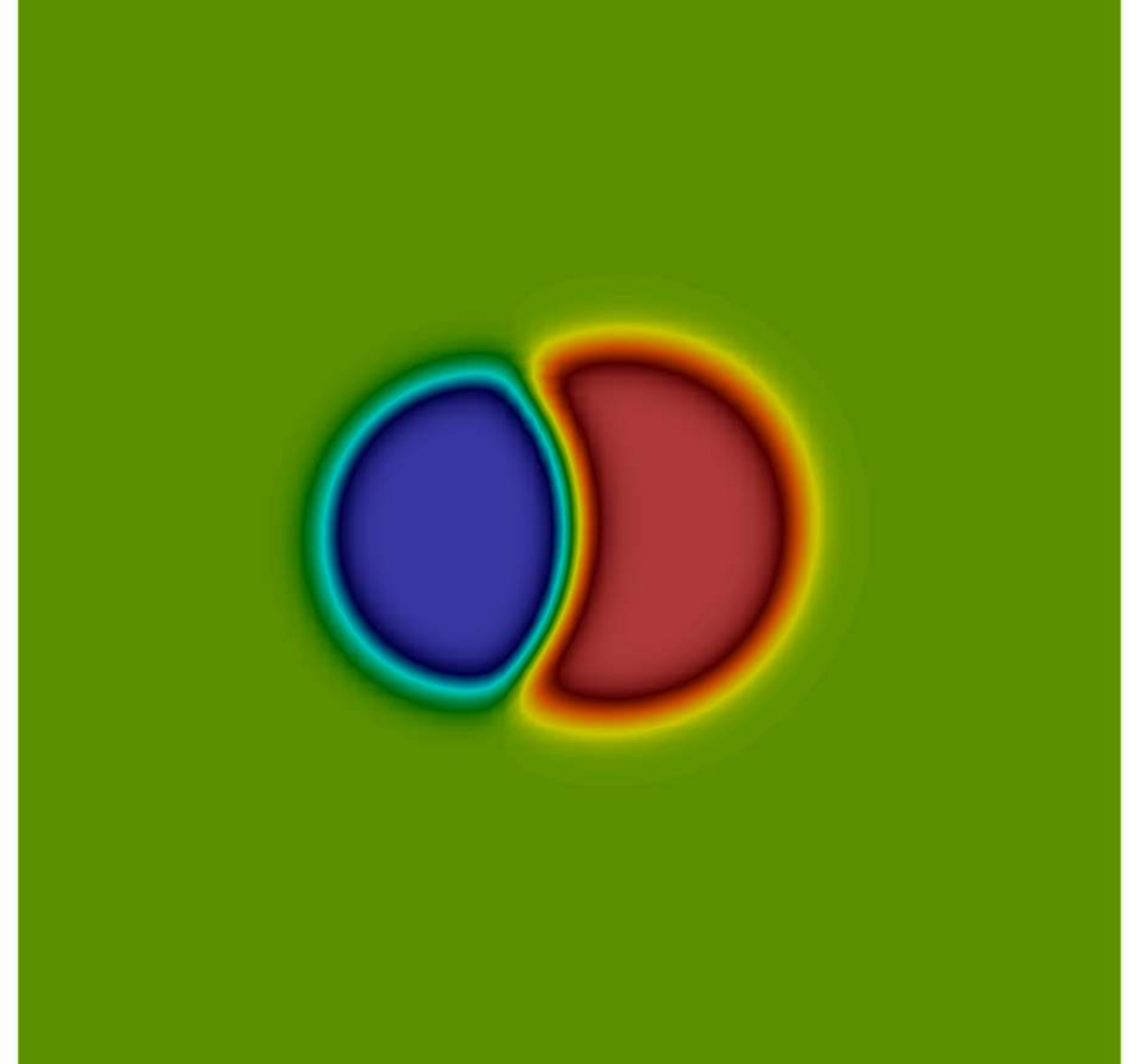}
\includegraphics[scale=0.09]{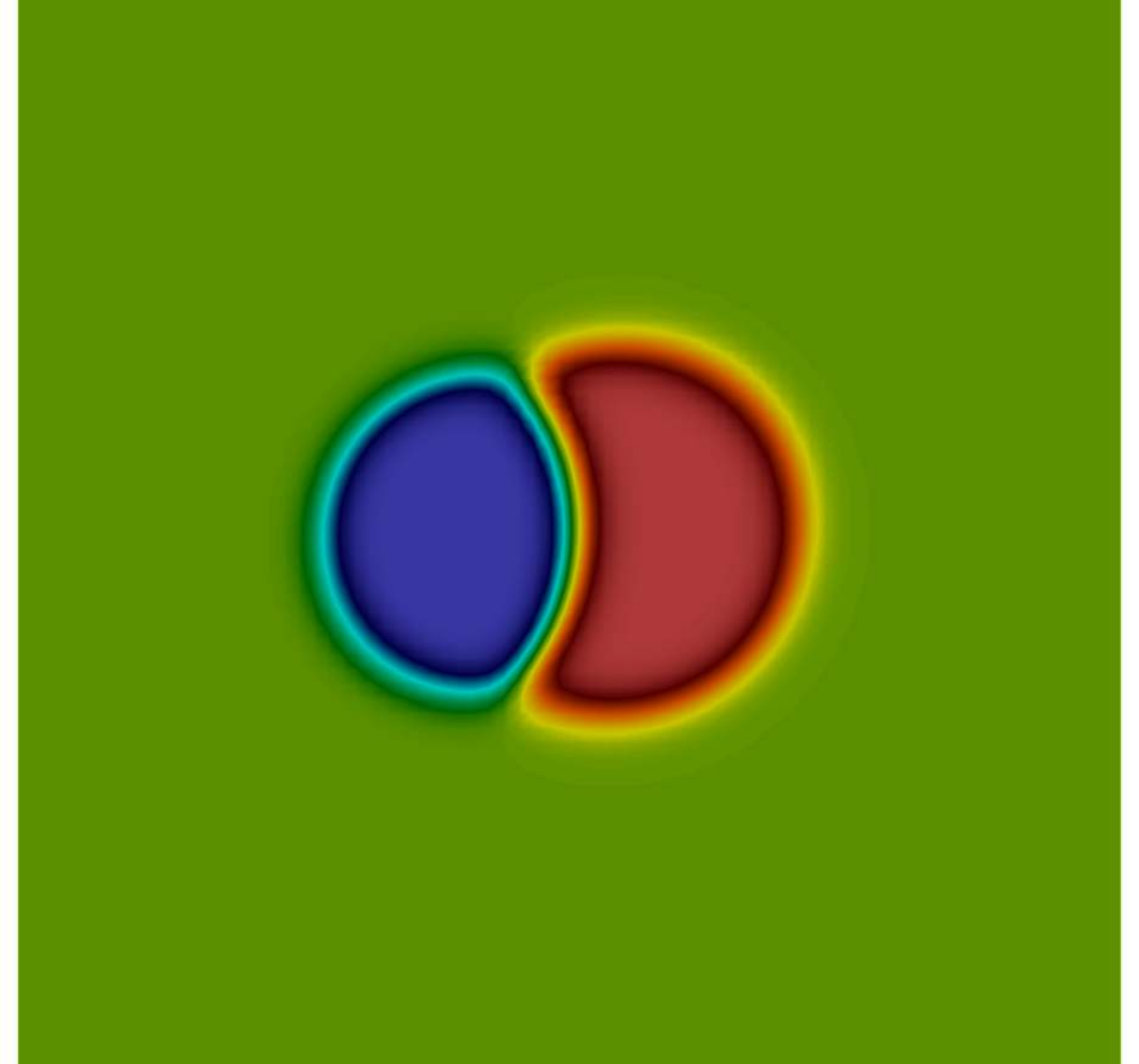}
\\ [1ex]
\includegraphics[scale=0.09]{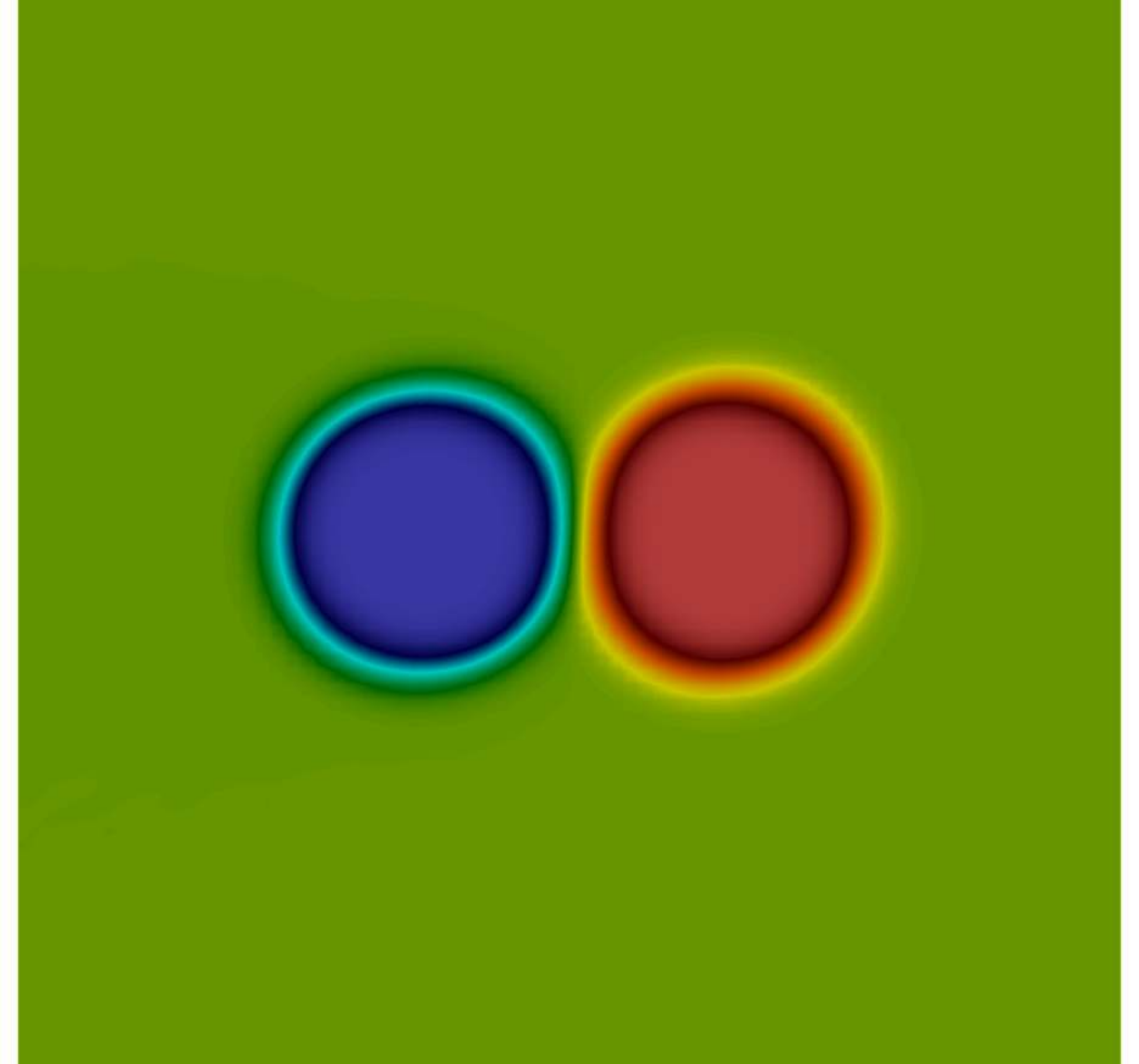}
\includegraphics[scale=0.09]{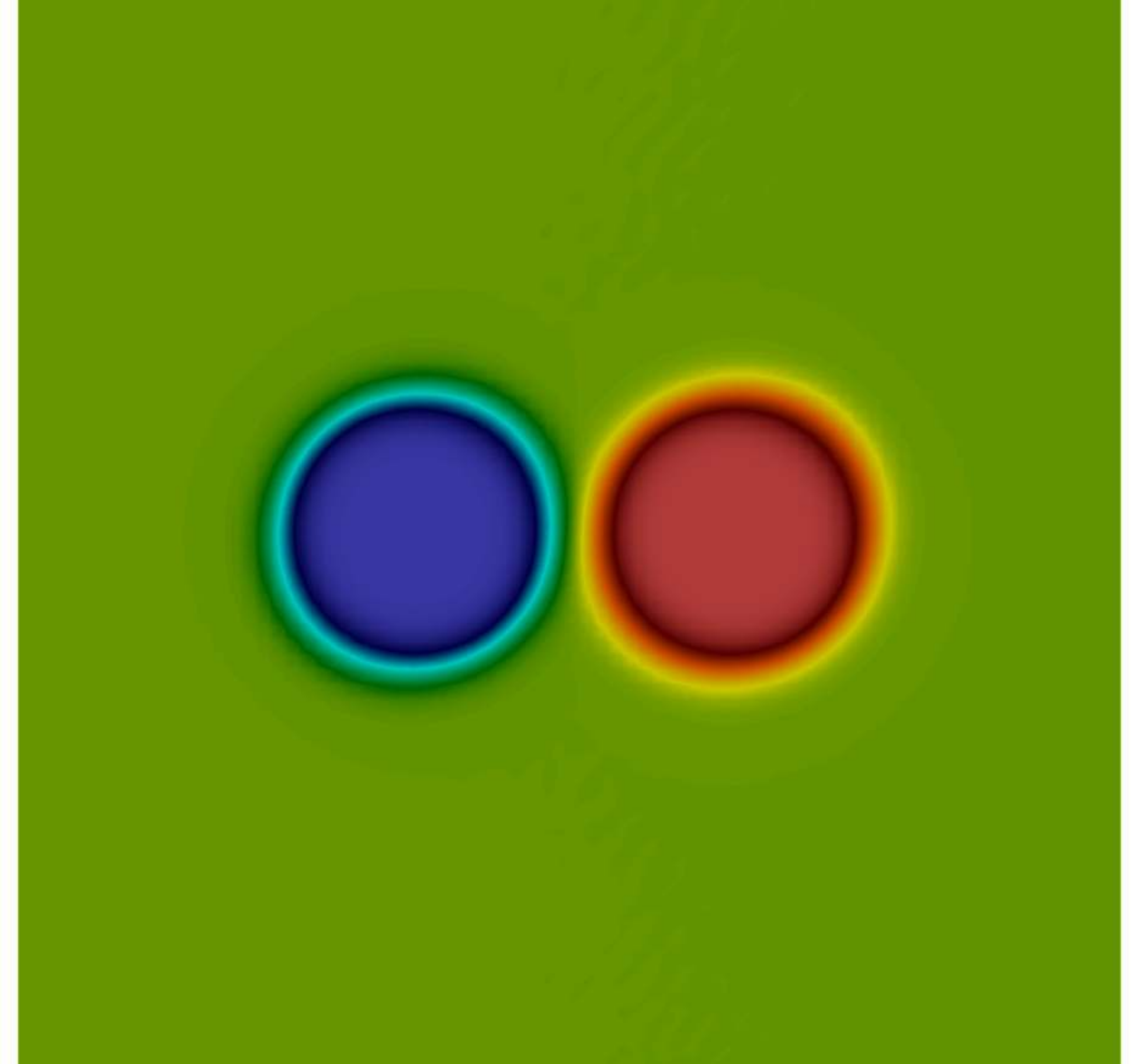}
\includegraphics[scale=0.09]{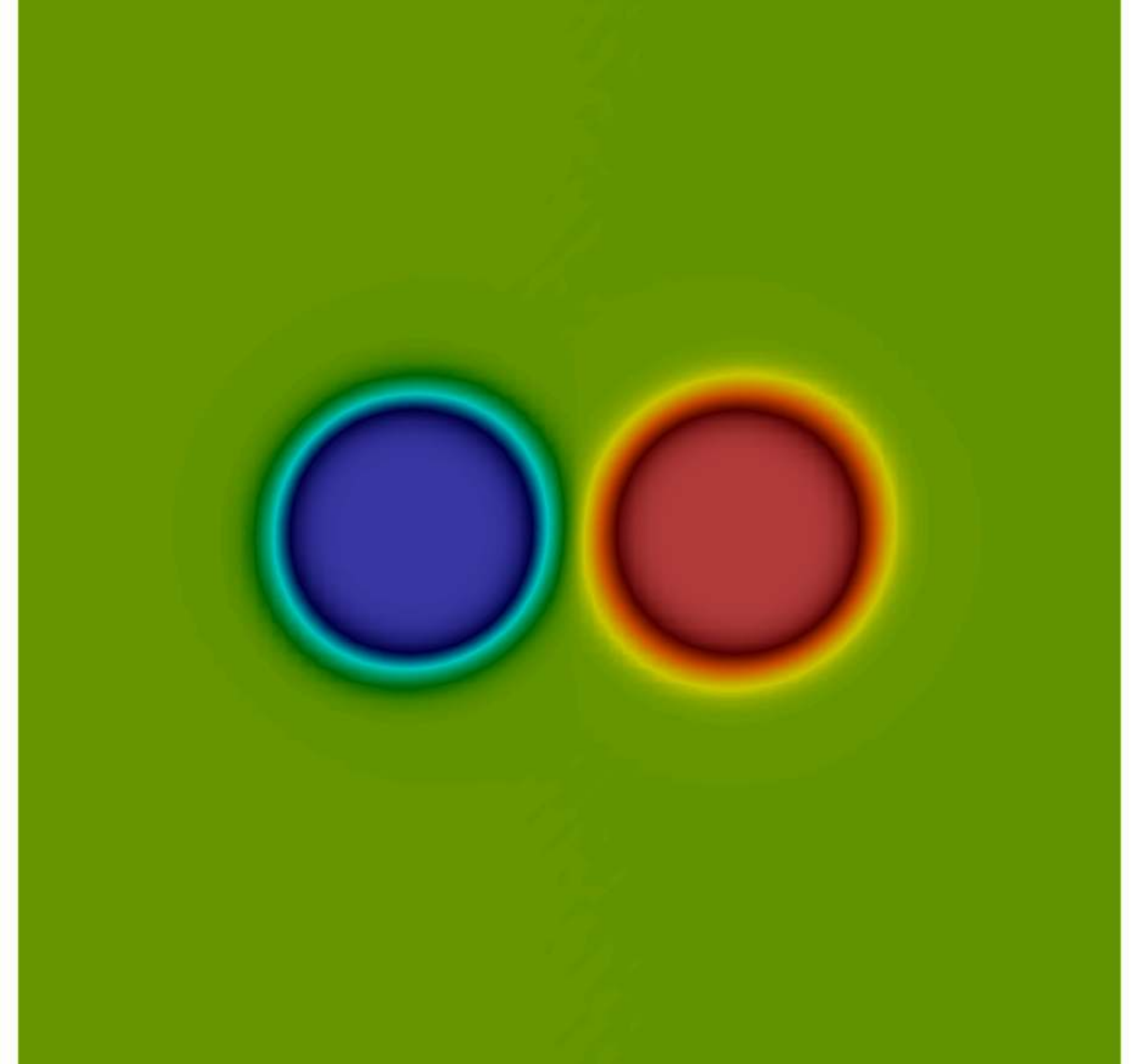}
\includegraphics[scale=0.09]{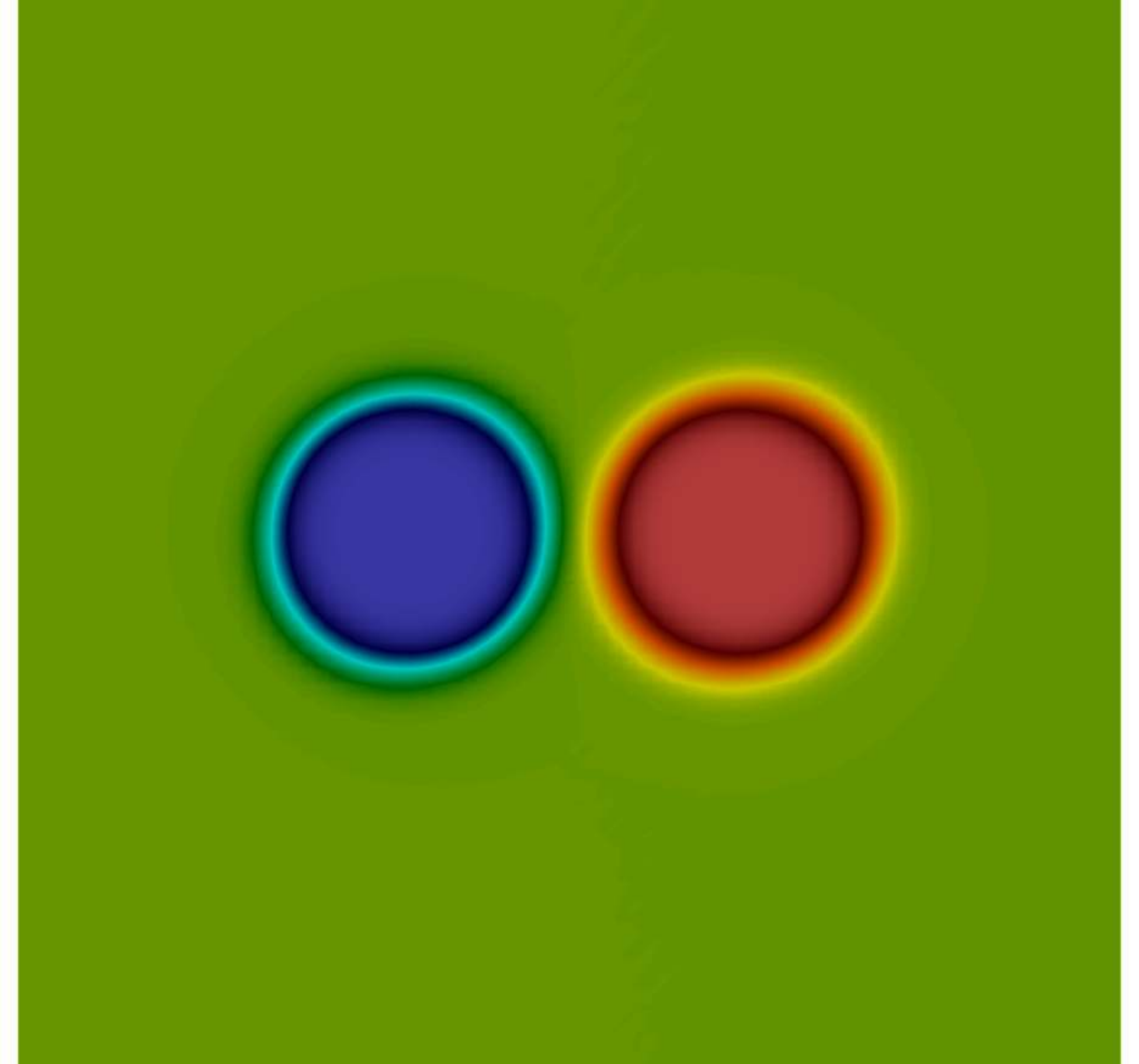}
\includegraphics[scale=0.09]{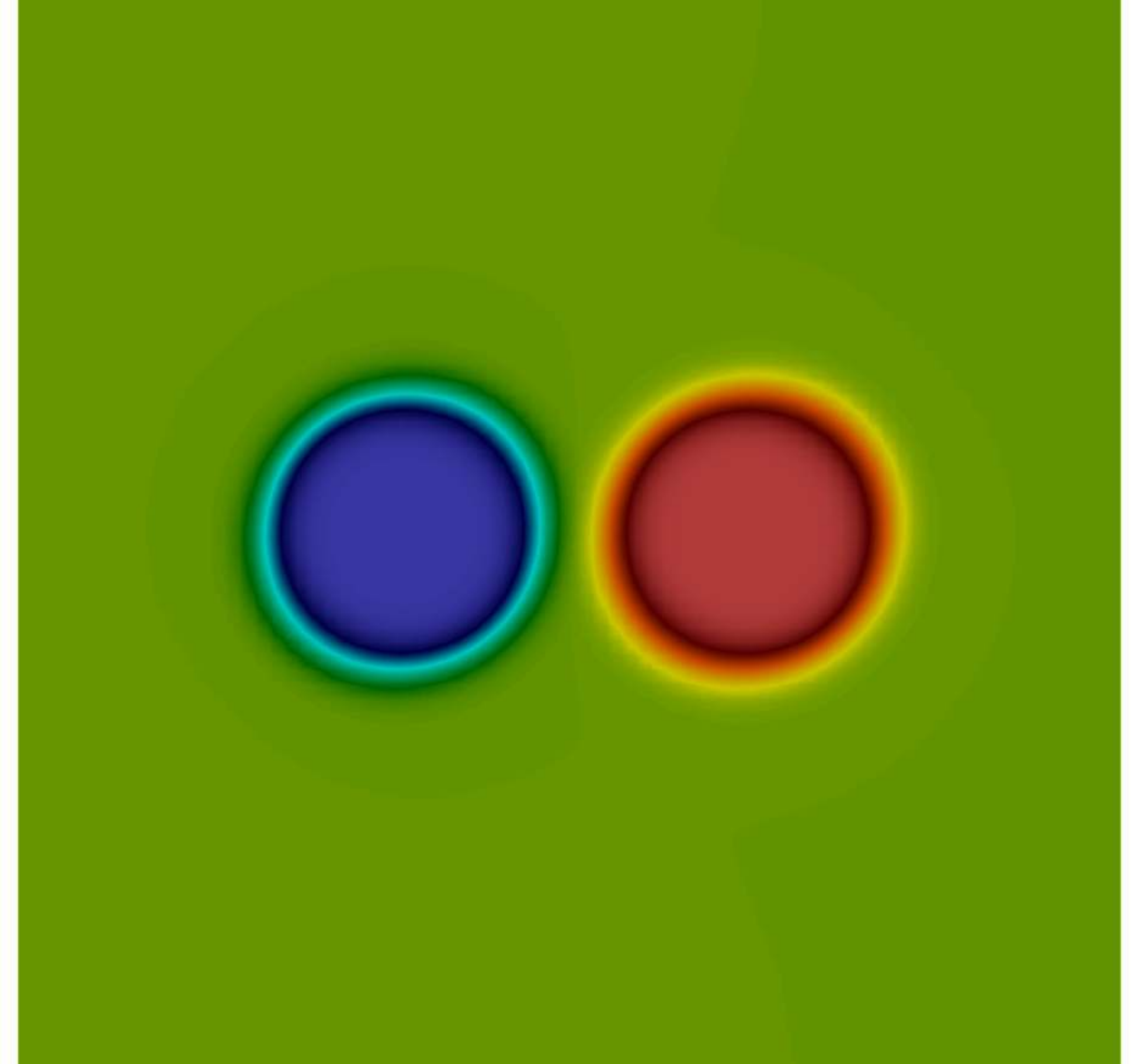}
\\ [1ex]
\includegraphics[scale=0.09]{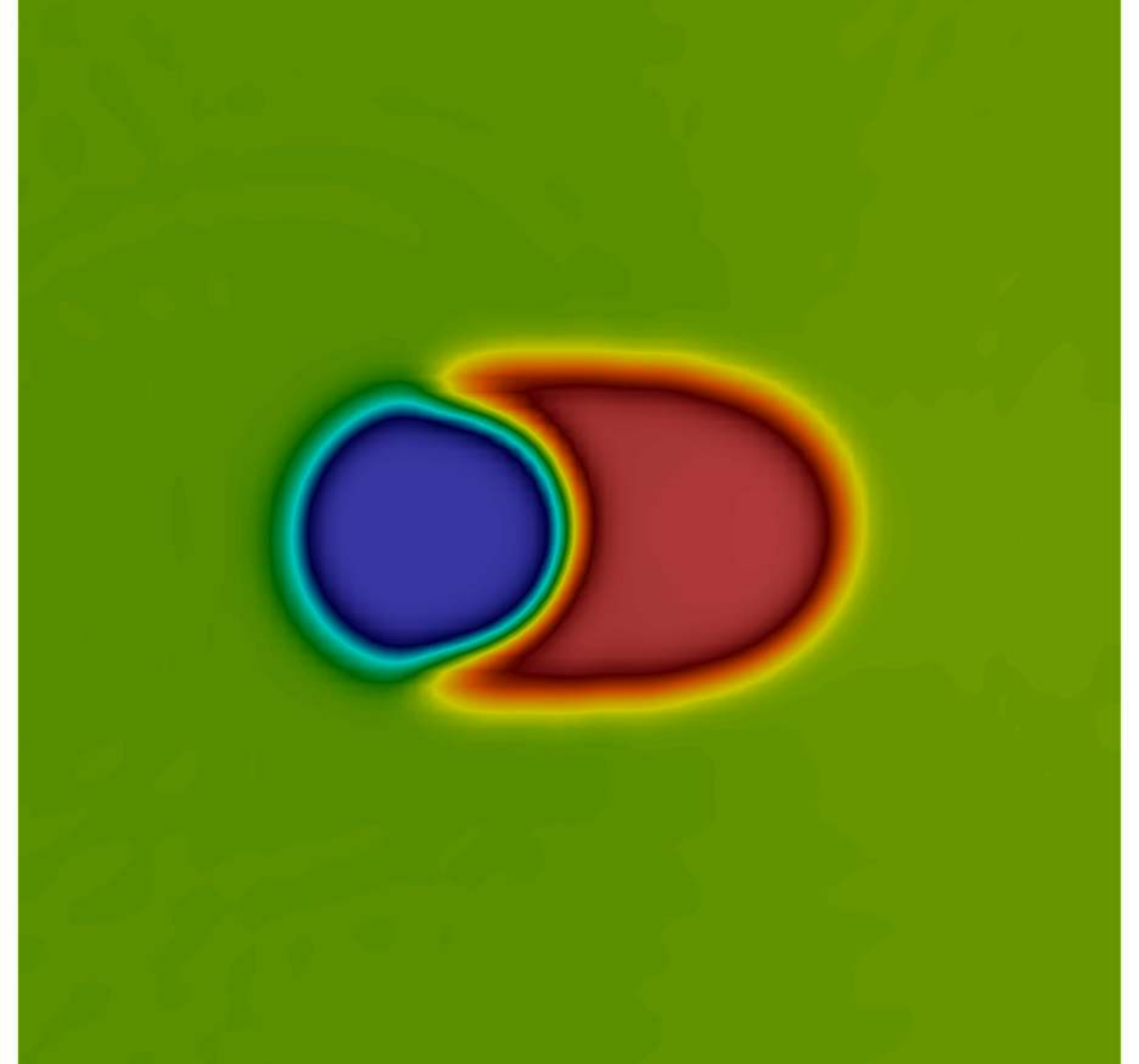}
\includegraphics[scale=0.09]{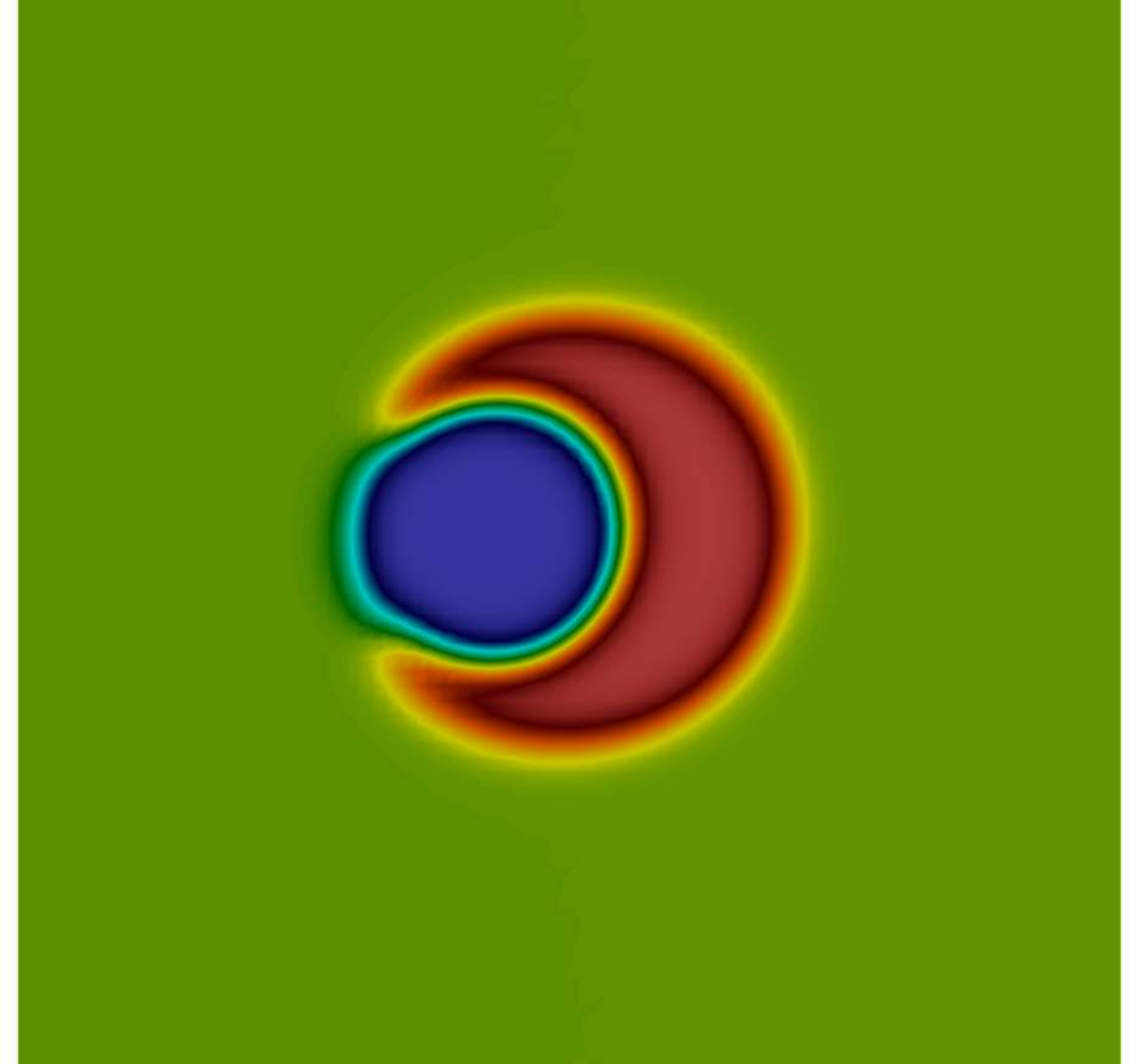}
\includegraphics[scale=0.09]{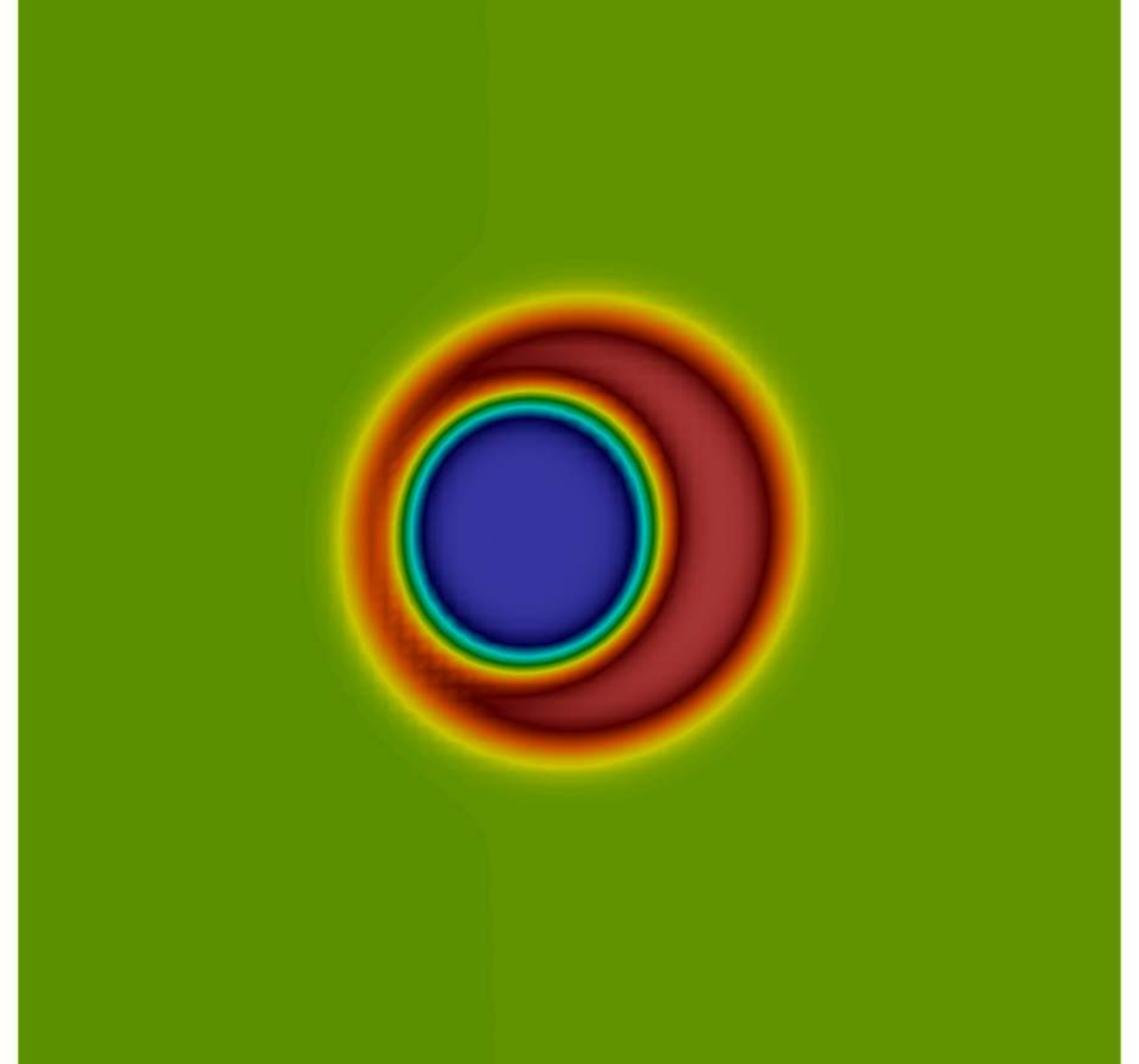}
\includegraphics[scale=0.09]{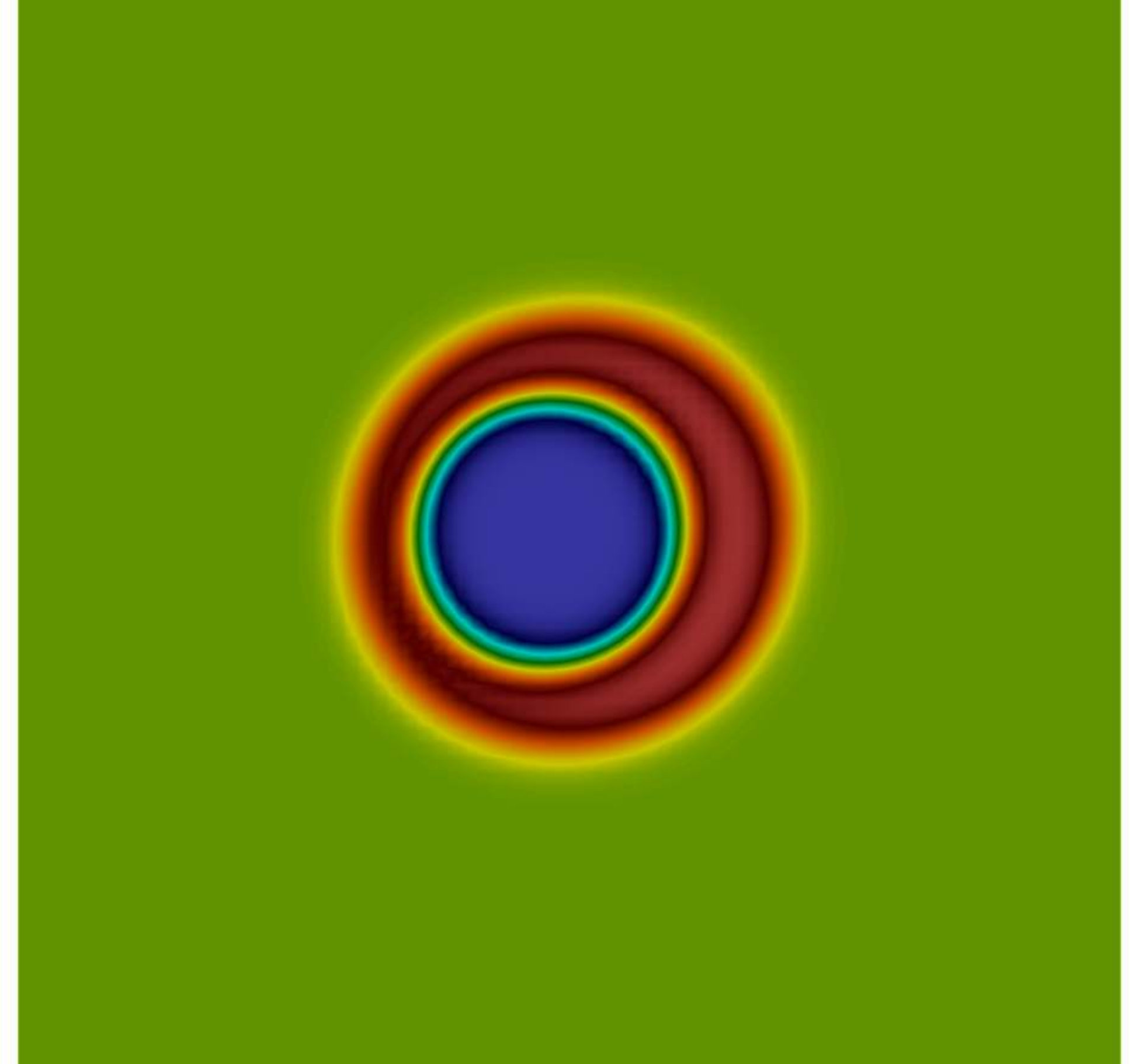}
\includegraphics[scale=0.09]{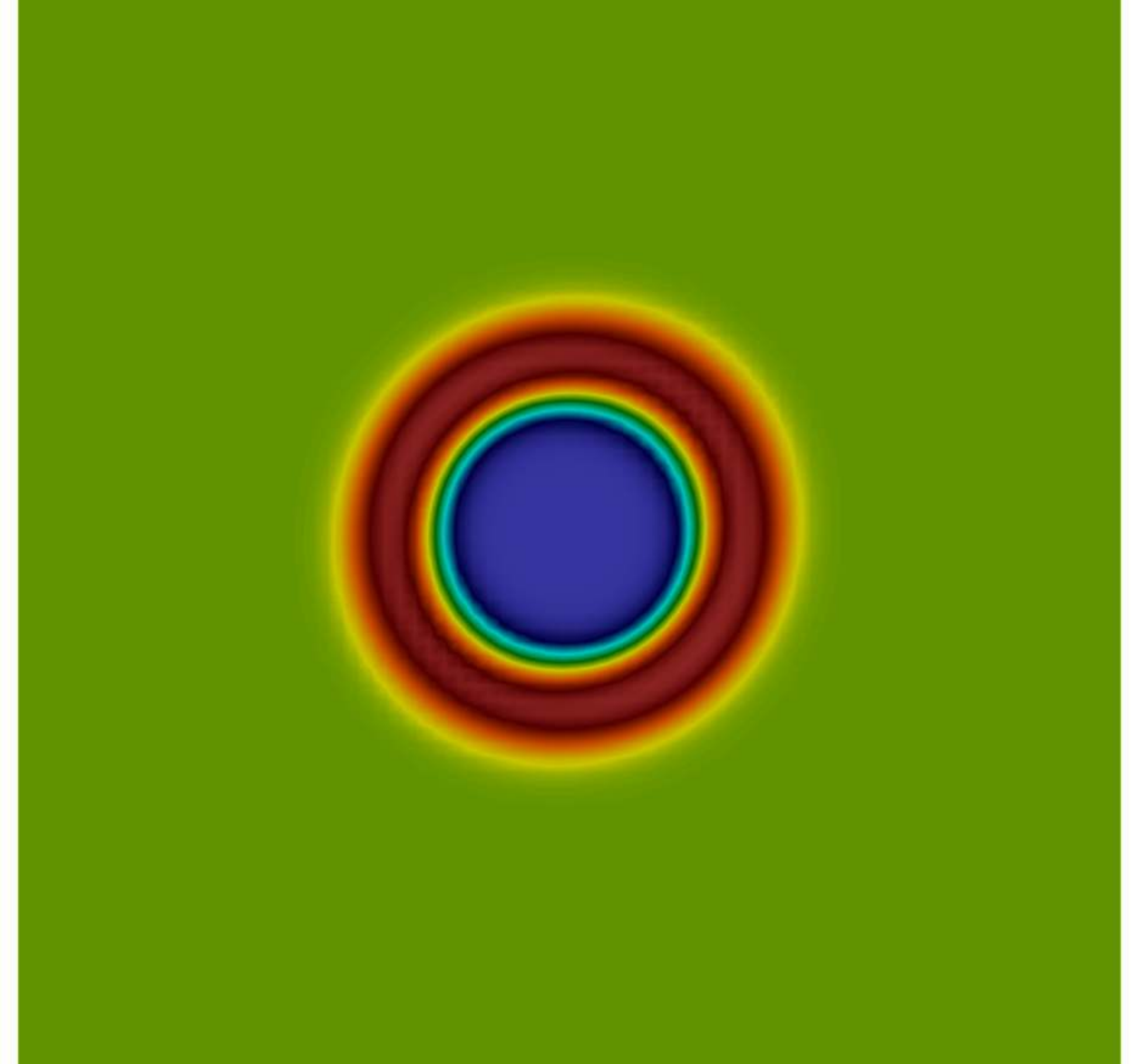}
\end{center}
\caption{Dynamics of scheme NTD1 at times $t=0.01, 0.05, 0.1, 0.15$ and $0.5$ (from left to right) with spreading coefficients 
$(\Sigma_1, \Sigma_2 , \Sigma_3) = (1,1,1)$ (top row)
$(\Sigma_1, \Sigma_2 , \Sigma_3) = (0.4, 1.6, 1.2)$ (second row)
$(\Sigma_1, \Sigma_2 , \Sigma_3) = (3,3,-0.1)$ (third row)
$(\Sigma_1, \Sigma_2 , \Sigma_3) = (-0.1,3,3)$ (bottom row).}\label{fig:BallsDynamics}
\end{figure}

\begin{figure}[h]
\begin{center}
\includegraphics[scale=0.11]{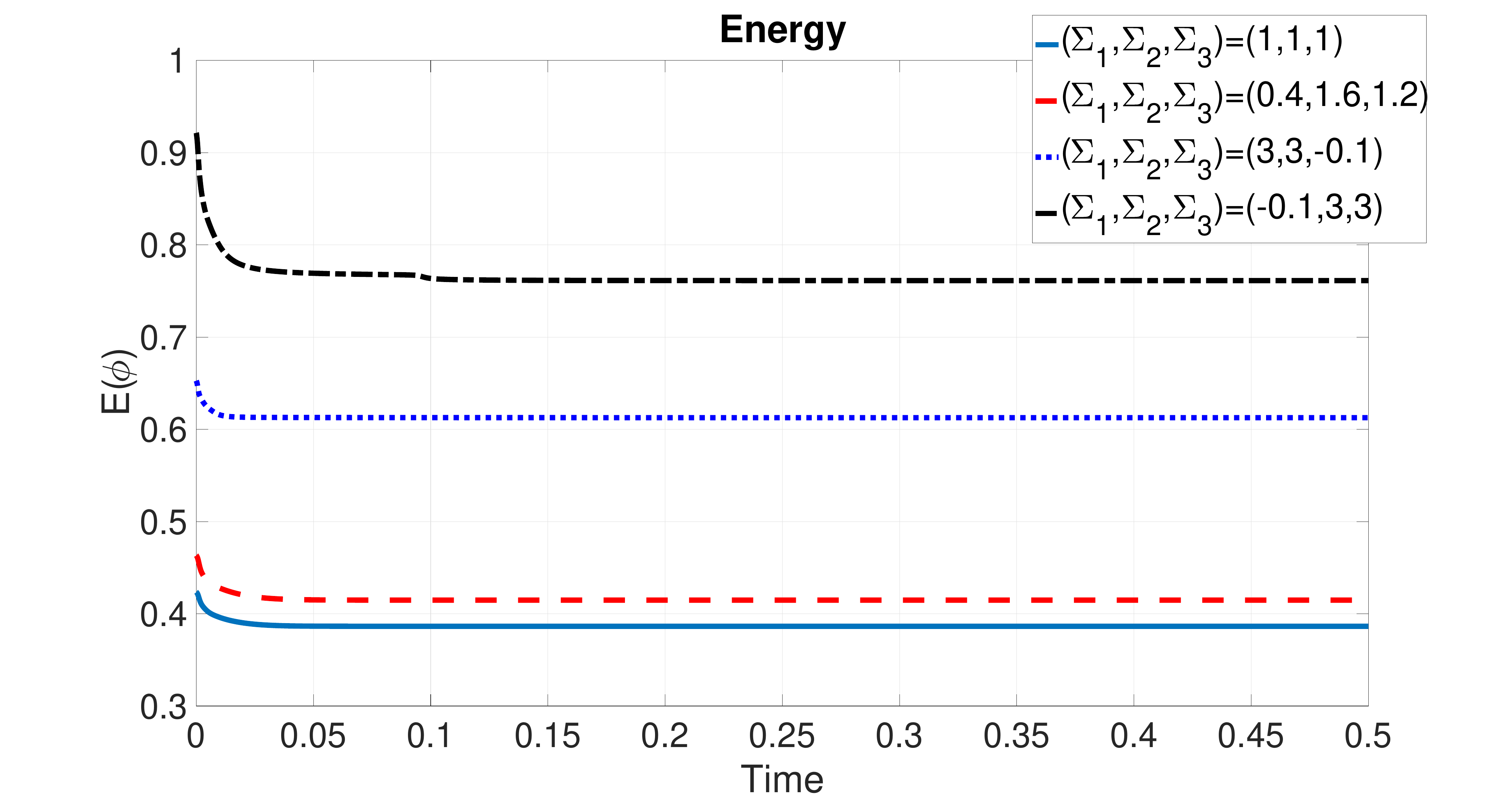}
\includegraphics[scale=0.11]{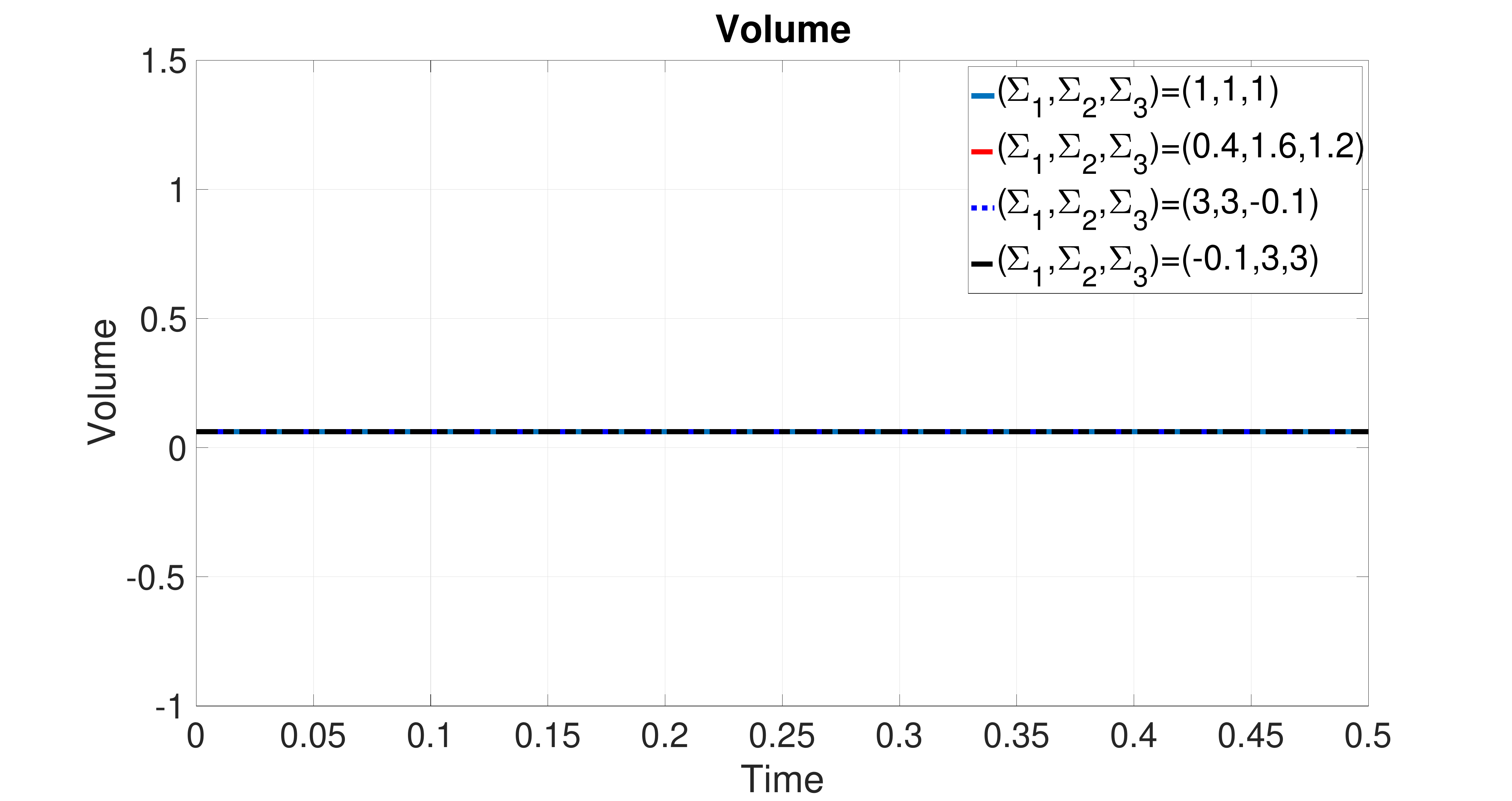}
\\ [1ex]
\includegraphics[scale=0.11]{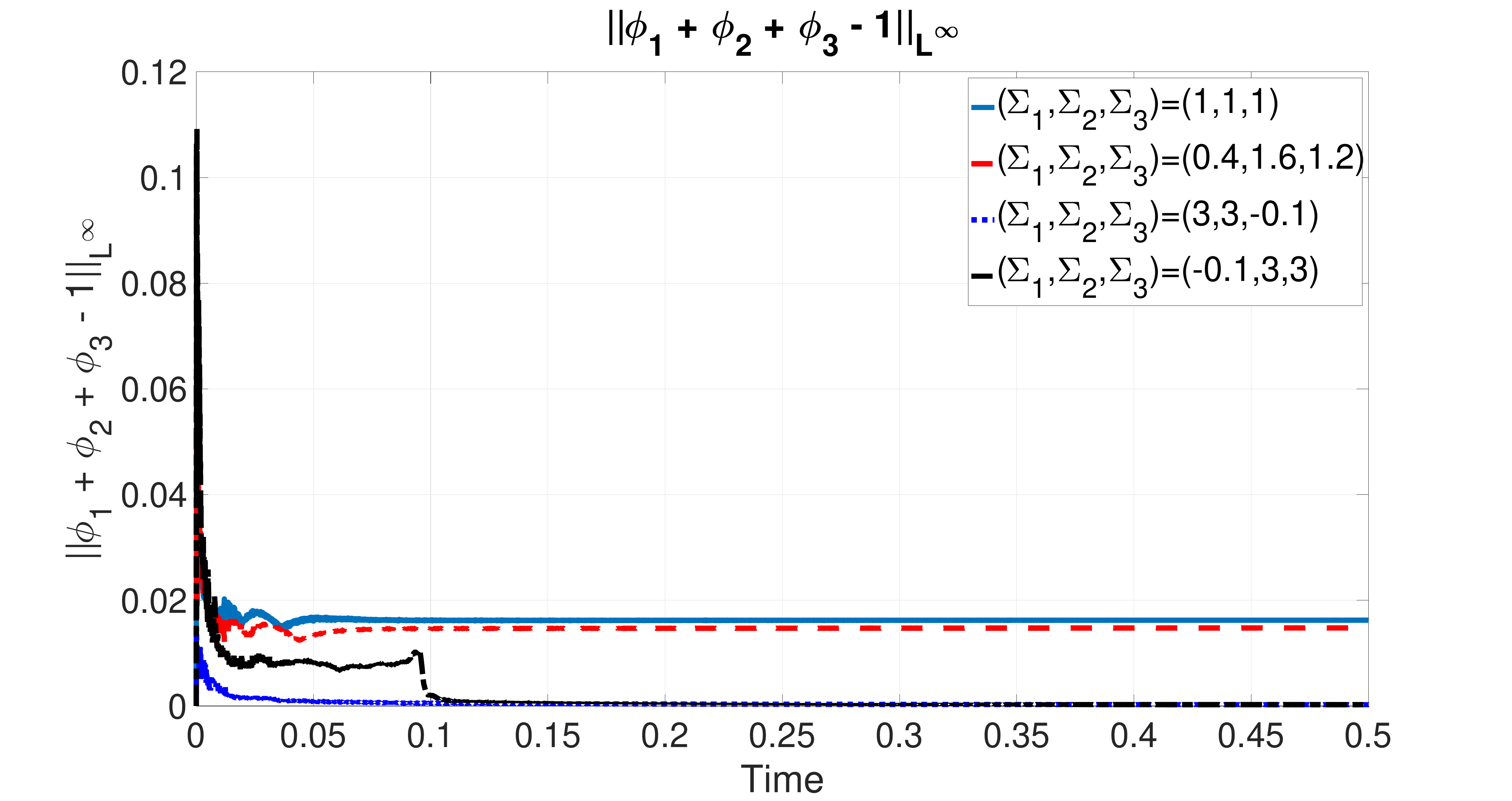}
\includegraphics[scale=0.11]{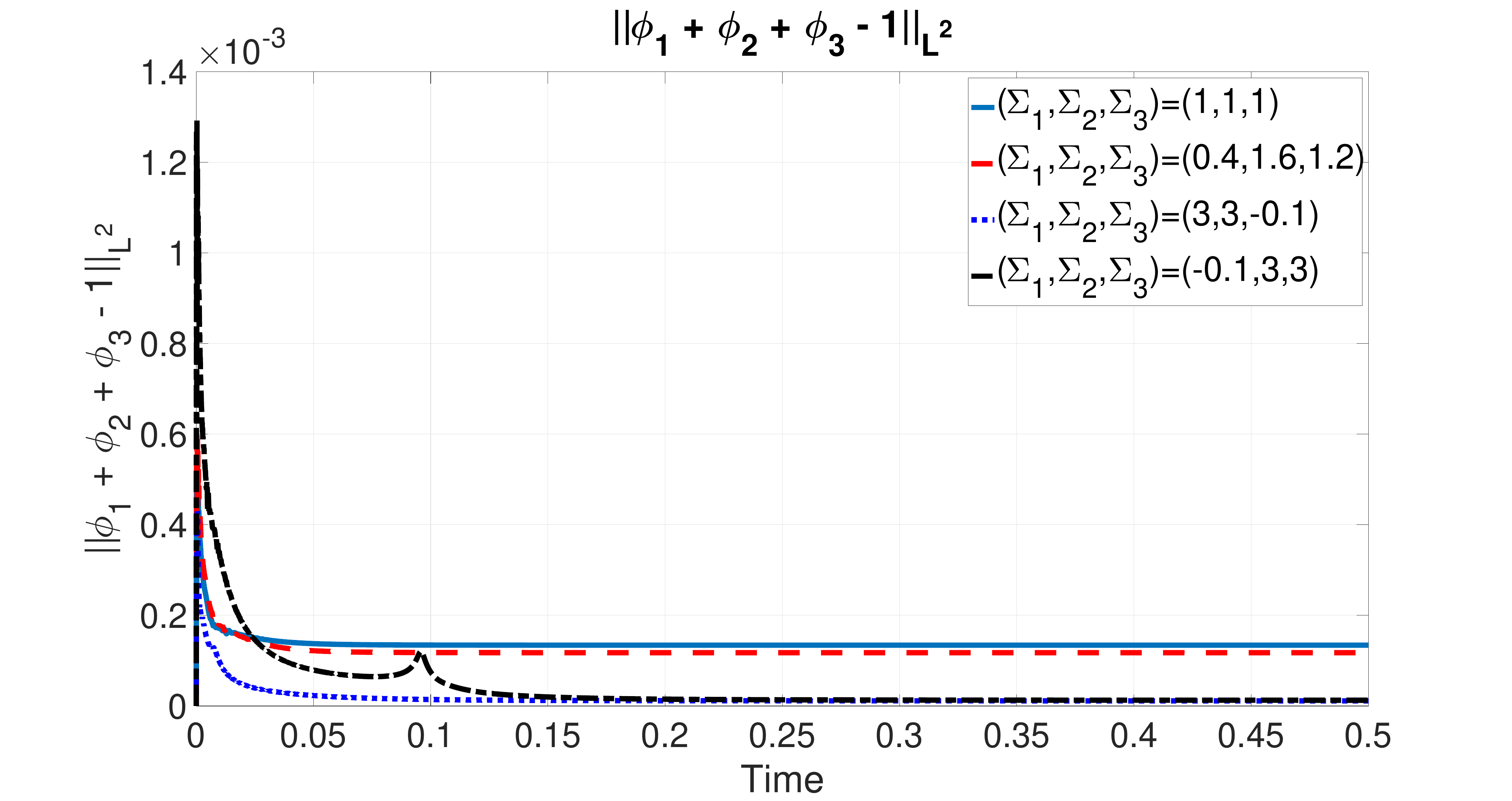}
\end{center}
\caption{Evolution in time of the energies (top left), the volume (top right), $\|\phi_1 + \phi_2 + \phi_3 -1\|_{L^\infty}$ (bottom left), $\|\phi_1 + \phi_2 + \phi_3 -1\|_{L^2}$ (bottom right) for the results presented in Figure~\ref{fig:BallsDynamics}.}
\label{fig:BallsPlots}
\end{figure}

\subsubsection{Study on the effect of $\lambda$ for $(\Sigma_1, \Sigma_2 , \Sigma_3) = (3,3,-0.1)$}
{
In this section we perform an study on the influence of the parameter $\lambda$ in the dynamics of the system and in the approximation of the restriction $\Sigma_{i=1}^3\phi_i - 1$. First we plot in Figure~\ref{fig:BallsDynamicsLambda0} the dynamics obtained when no penalization of the restriction is imposed and we can observe how the three phases rearrange freely without taking into account the shape of the other two, creating regions where obviously the restriction is not going to be well approximated (this is confirmed in the bottom row of plots in Figure~\ref{fig:BallsPlotsLambda}). 
}

\begin{figure}[h]
\begin{center}
\includegraphics[scale=0.09]{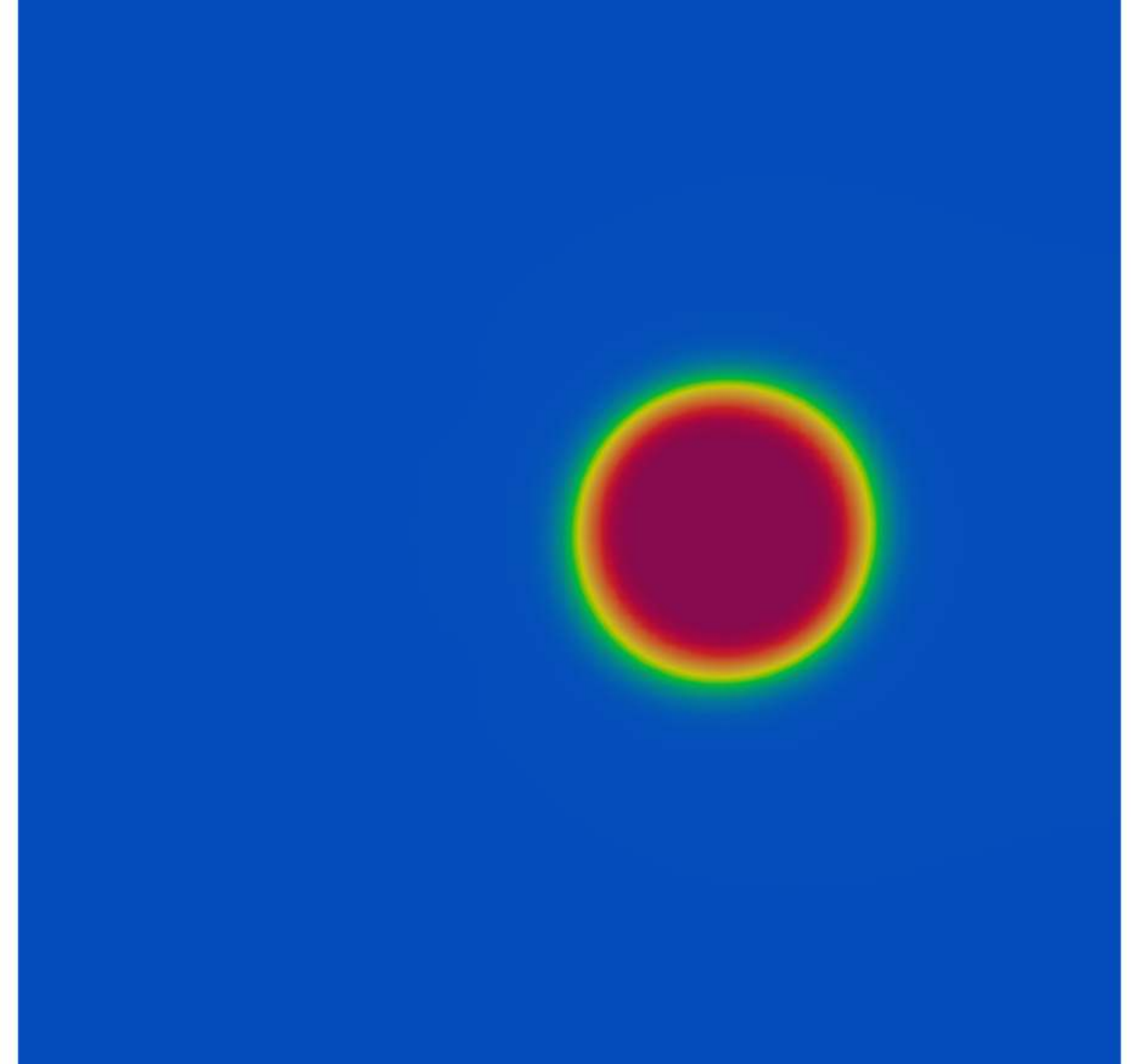}
\includegraphics[scale=0.09]{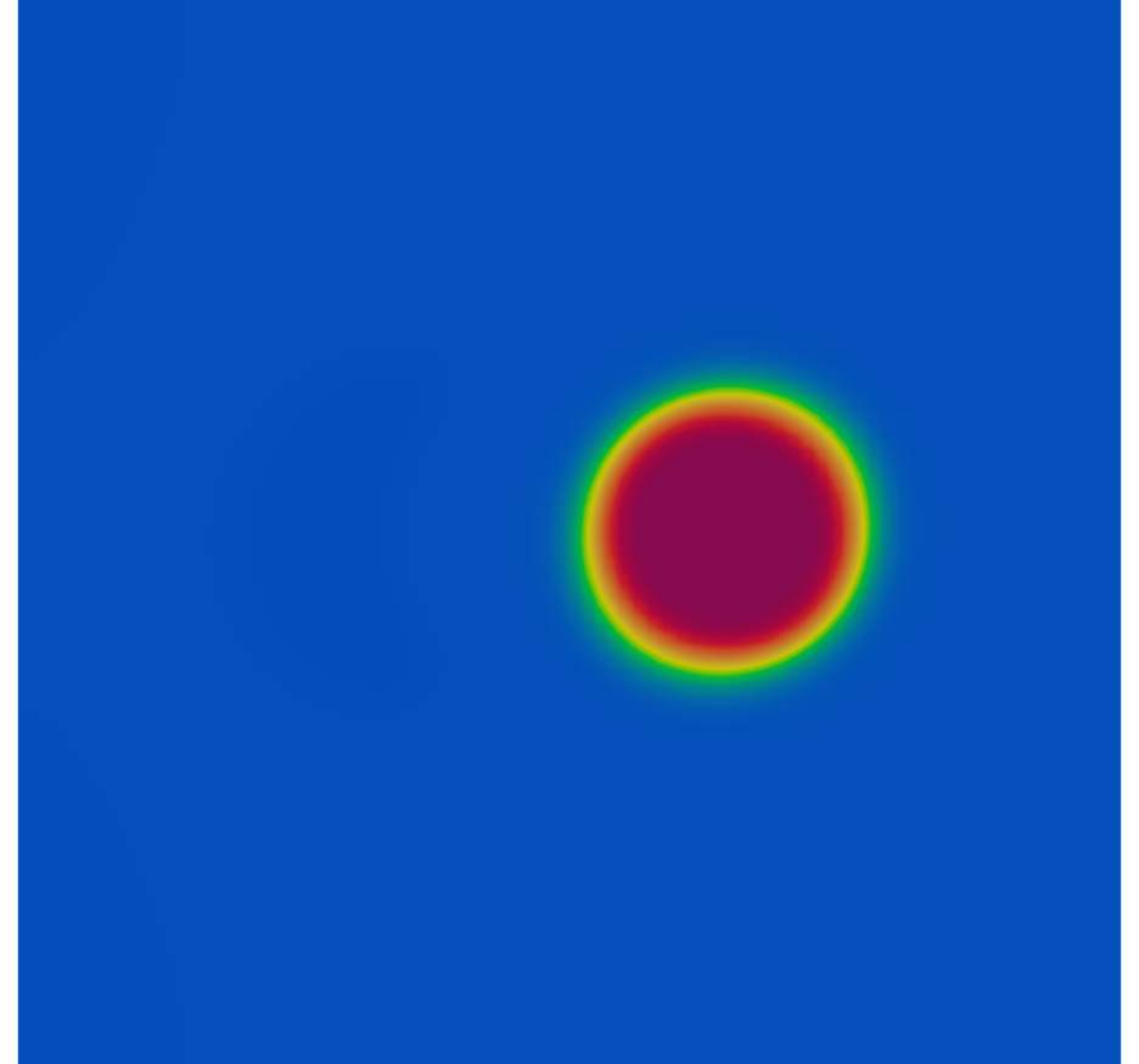}
\includegraphics[scale=0.09]{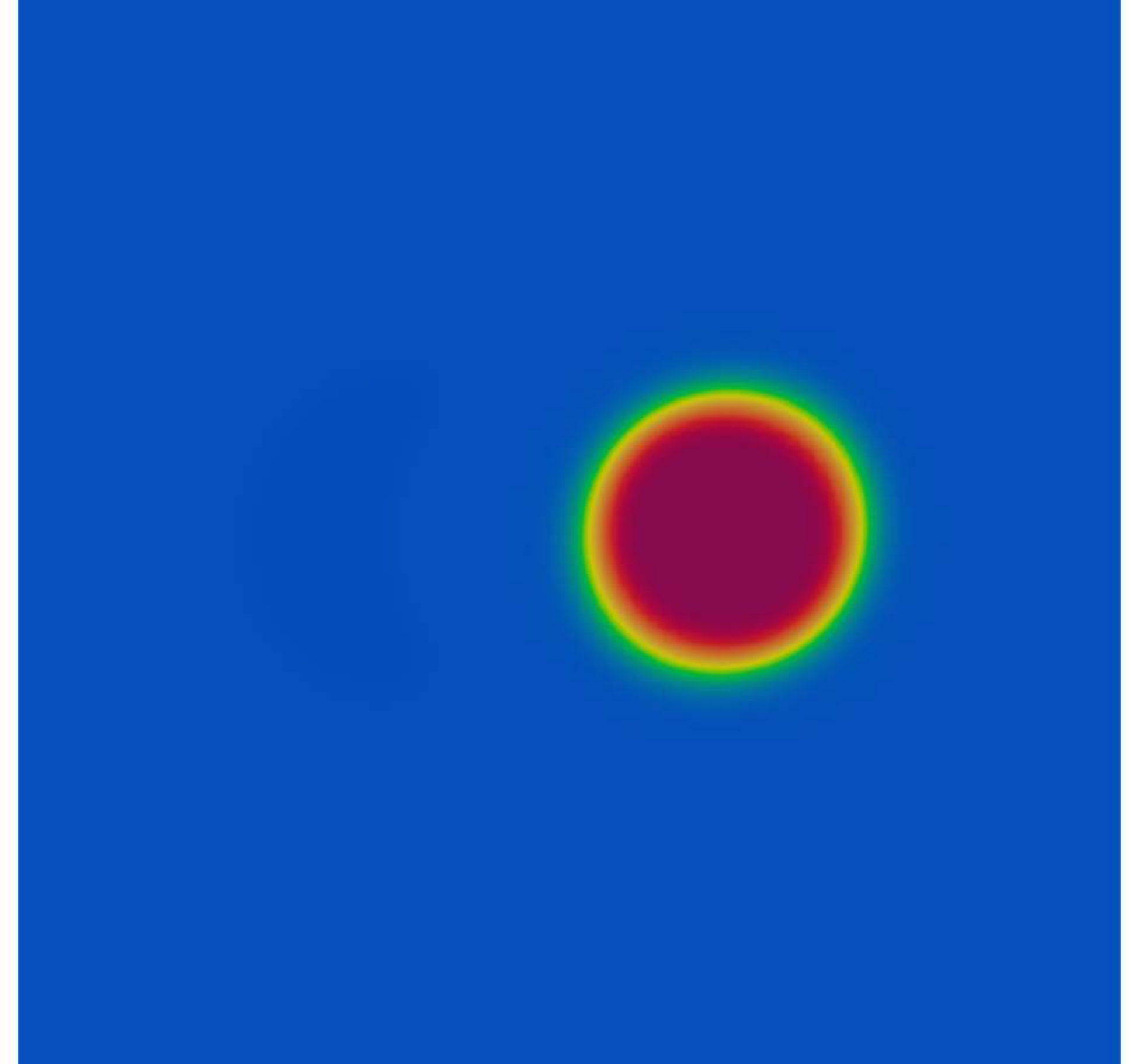}
\includegraphics[scale=0.09]{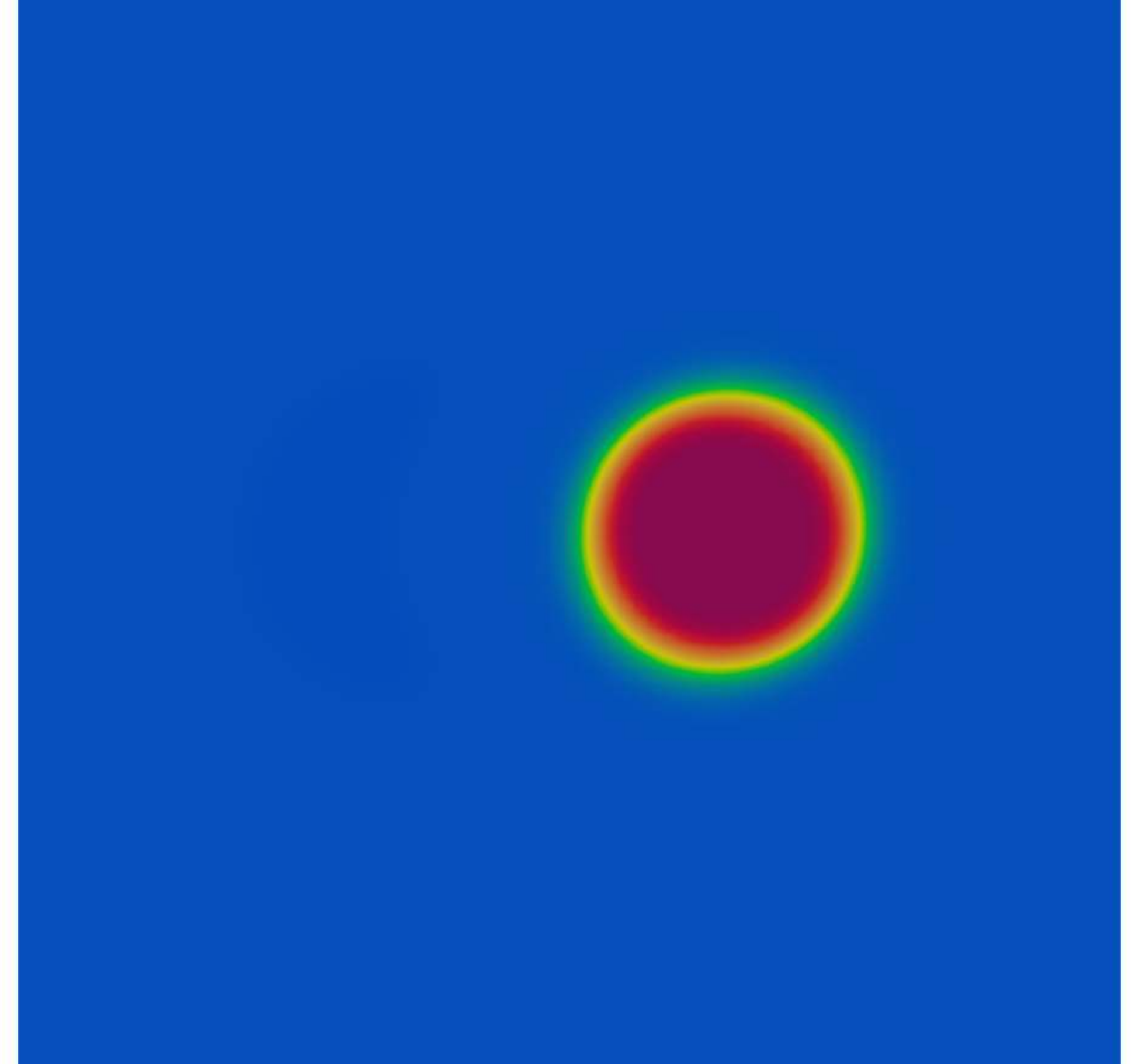}
\includegraphics[scale=0.09]{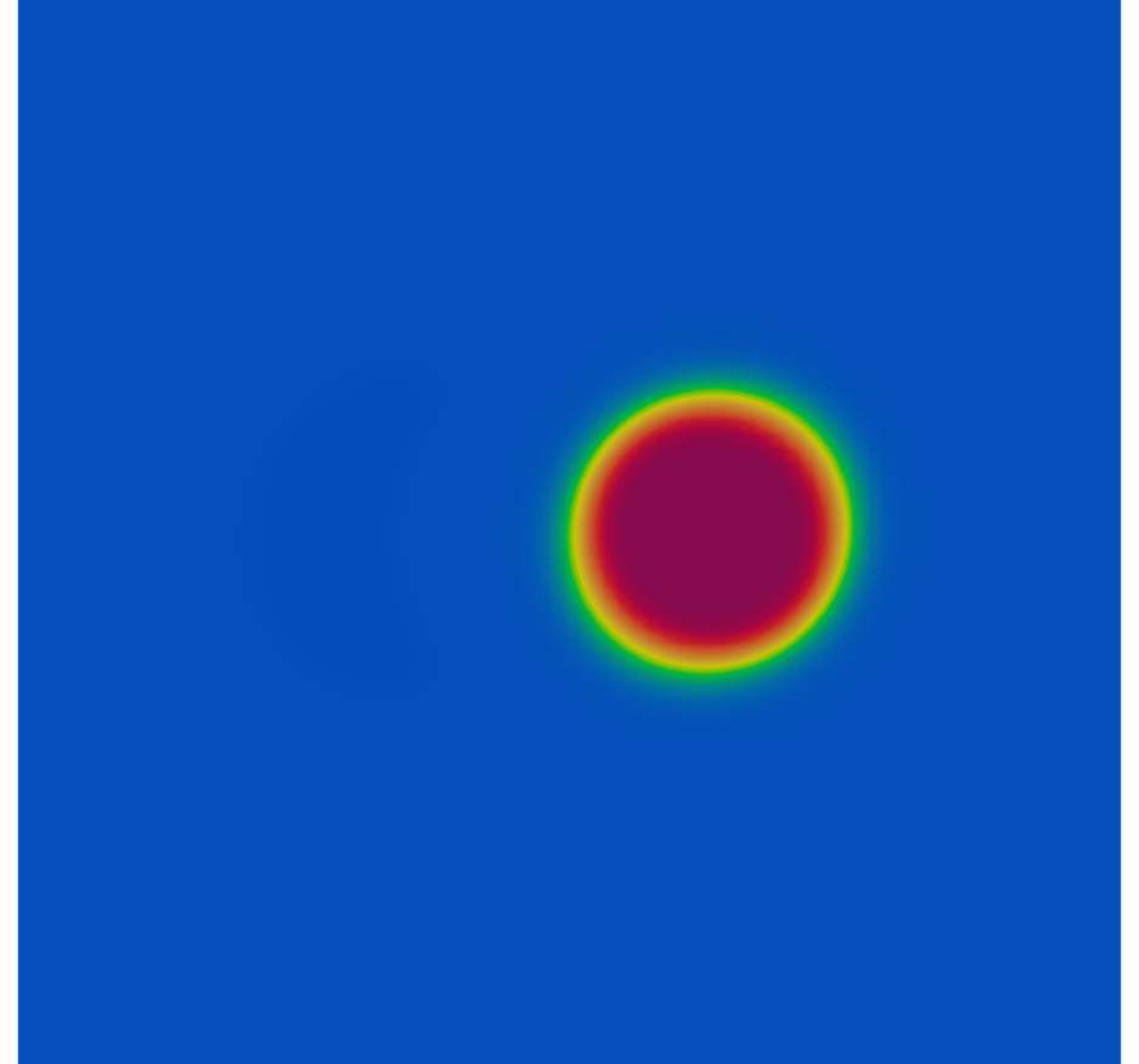}
\\ [1ex]
\includegraphics[scale=0.09]{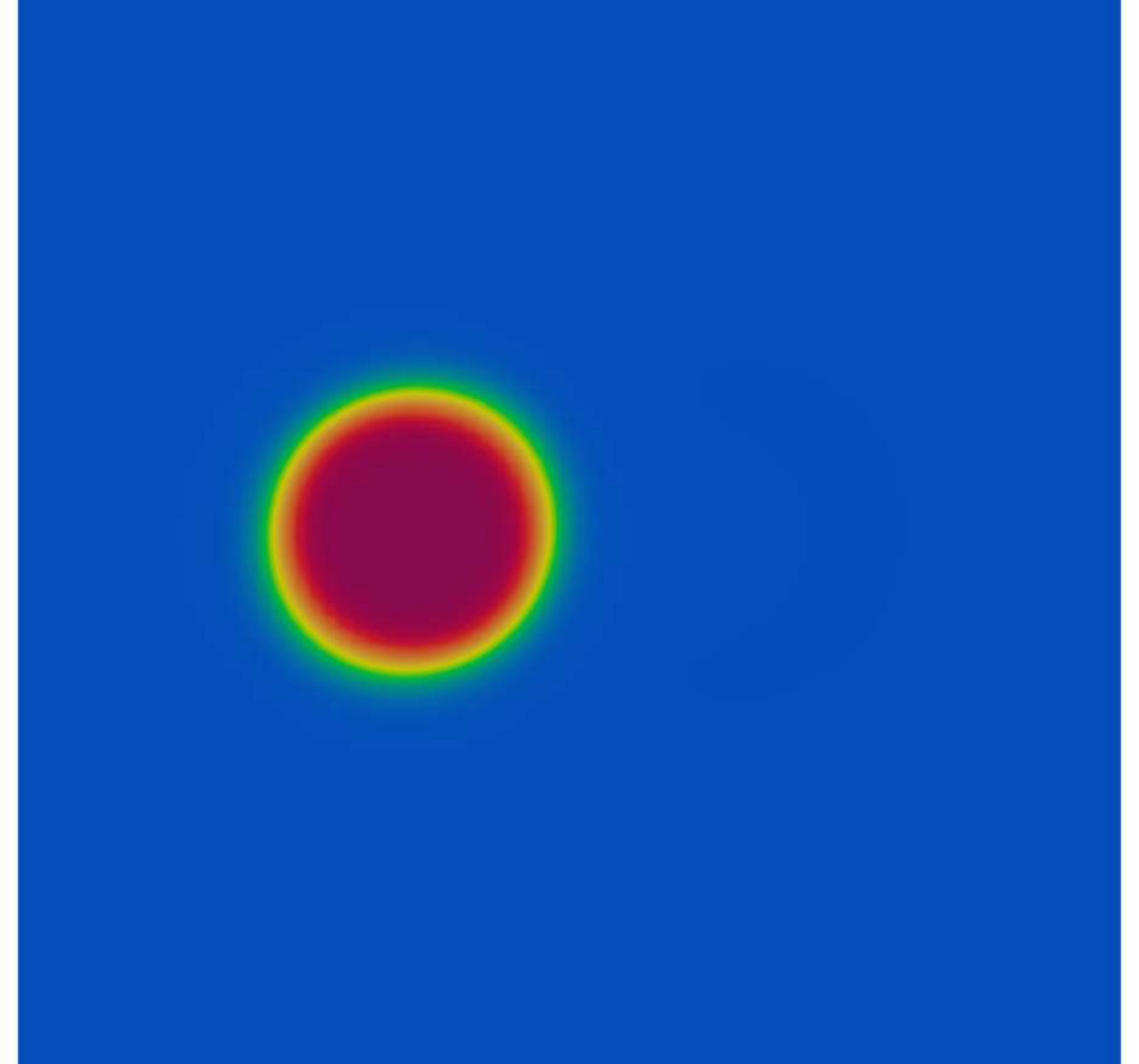}
\includegraphics[scale=0.09]{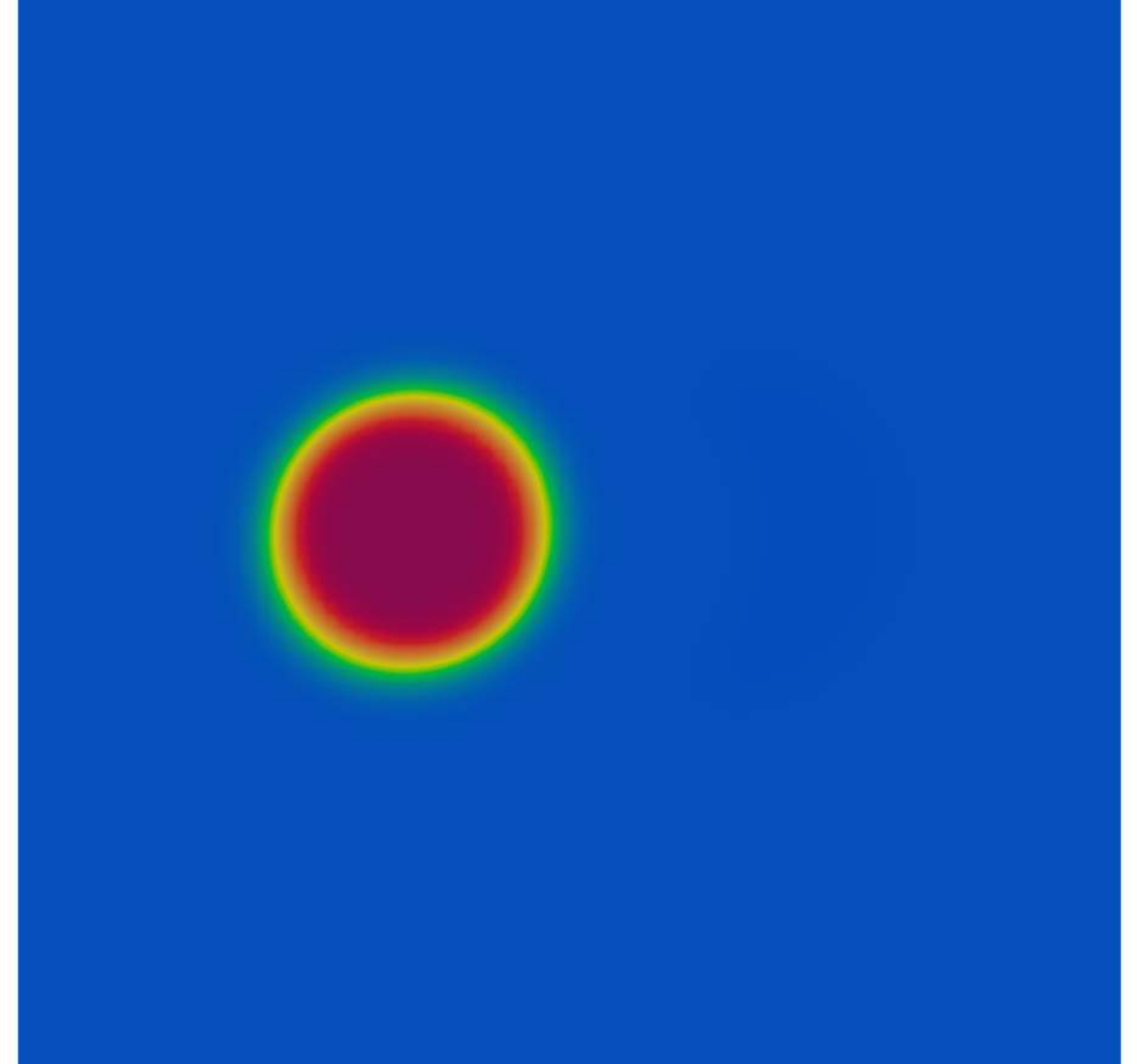}
\includegraphics[scale=0.09]{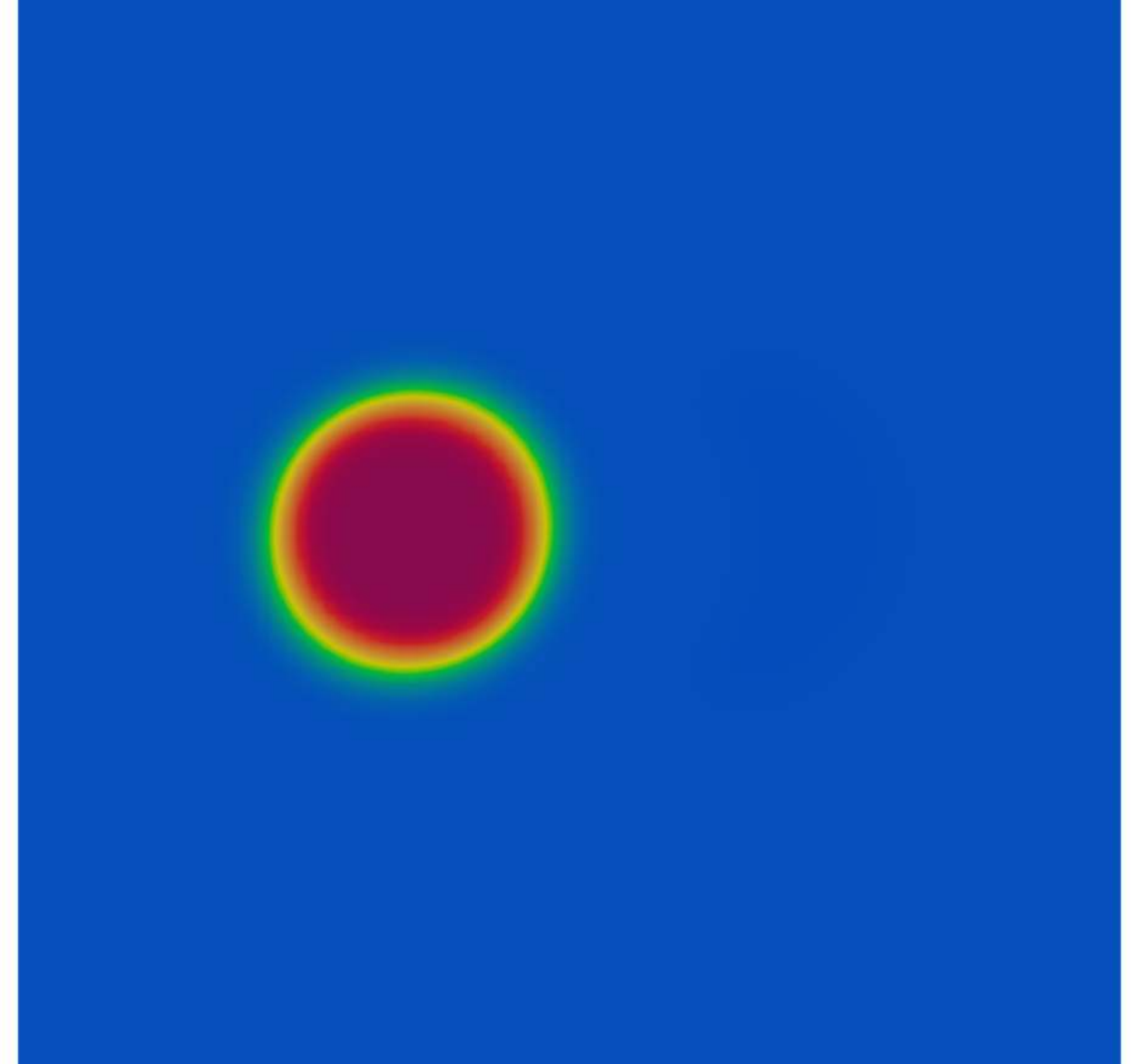}
\includegraphics[scale=0.09]{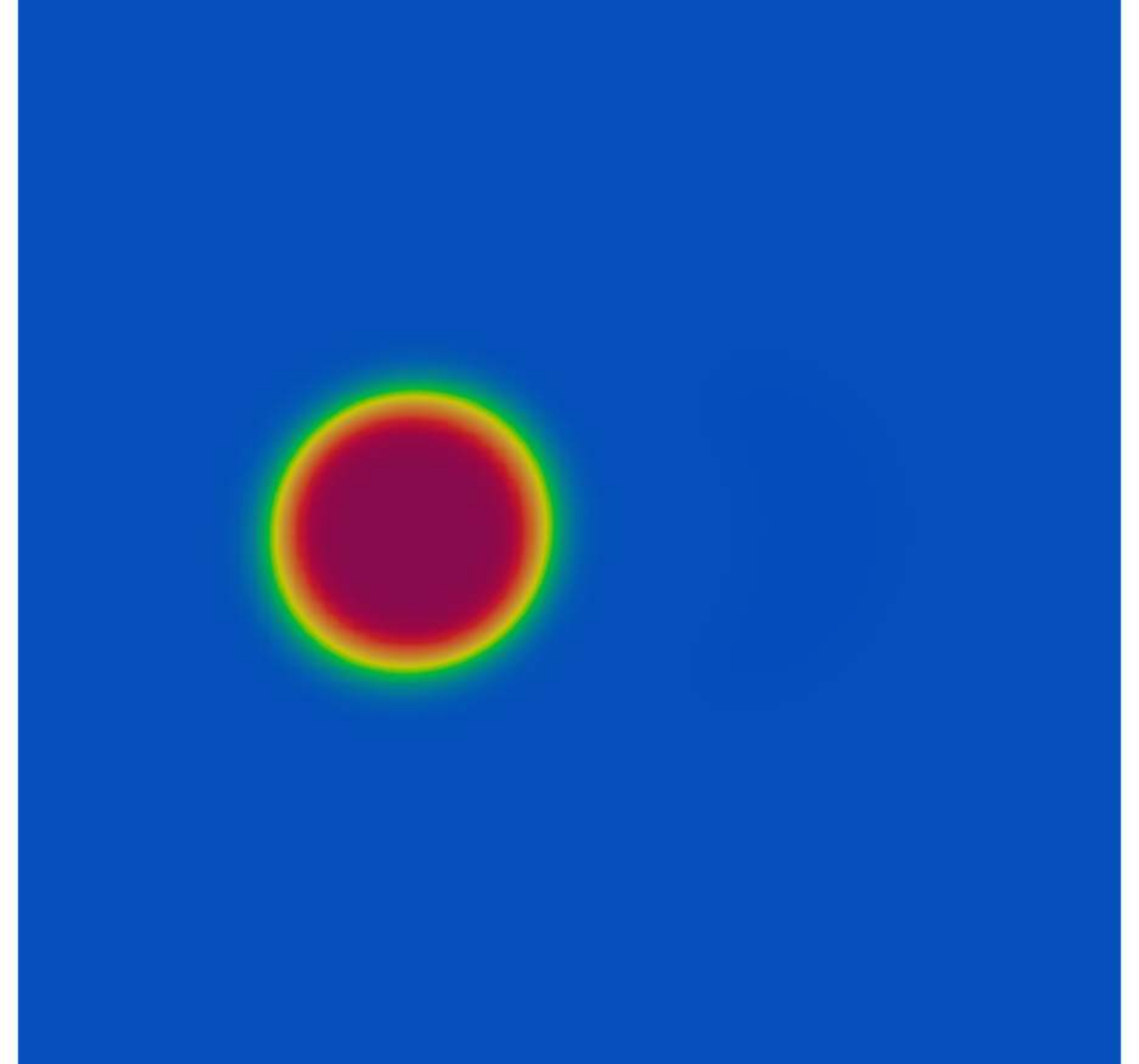}
\includegraphics[scale=0.09]{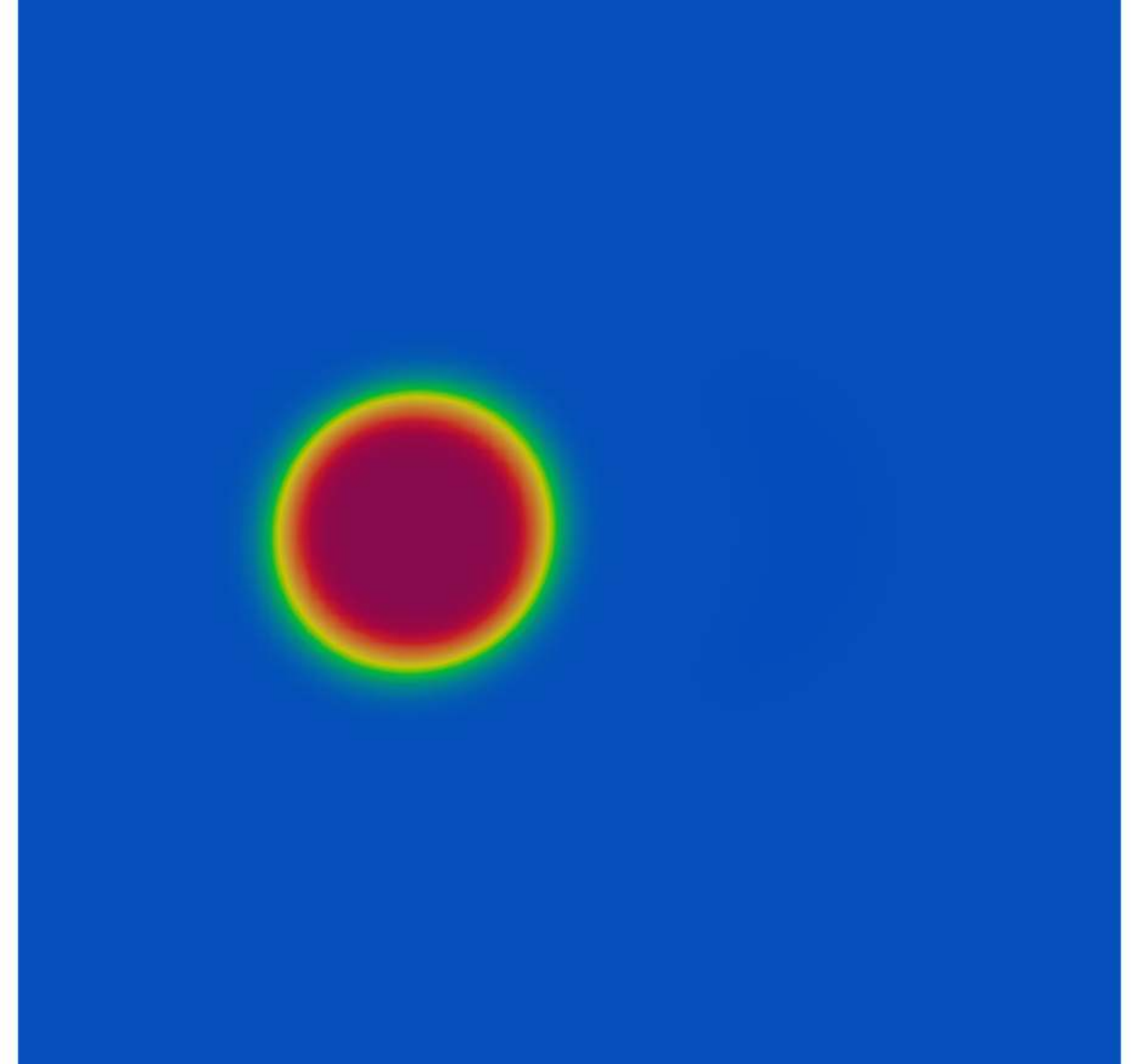}
\\ [1ex]
\includegraphics[scale=0.09]{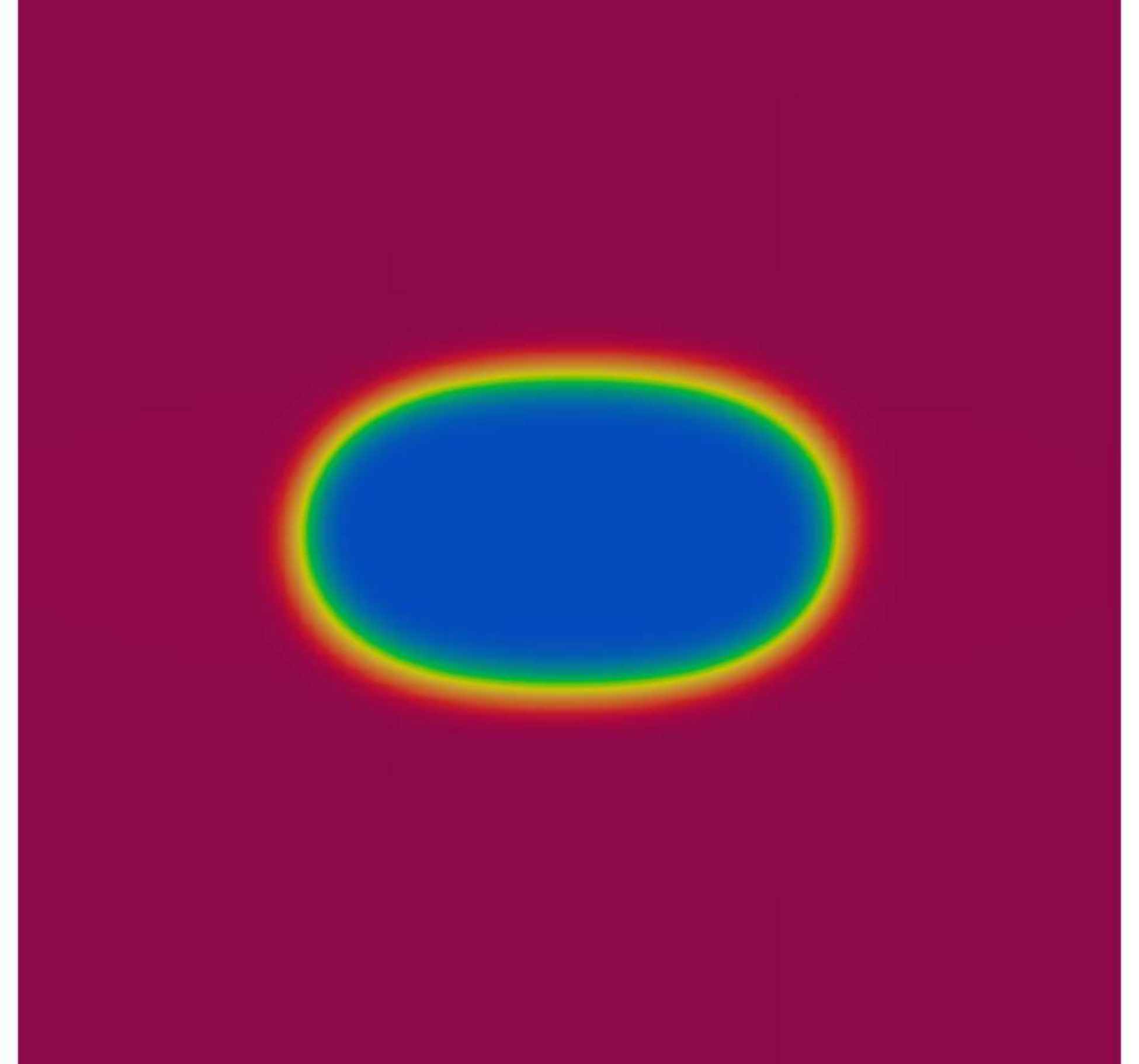}
\includegraphics[scale=0.09]{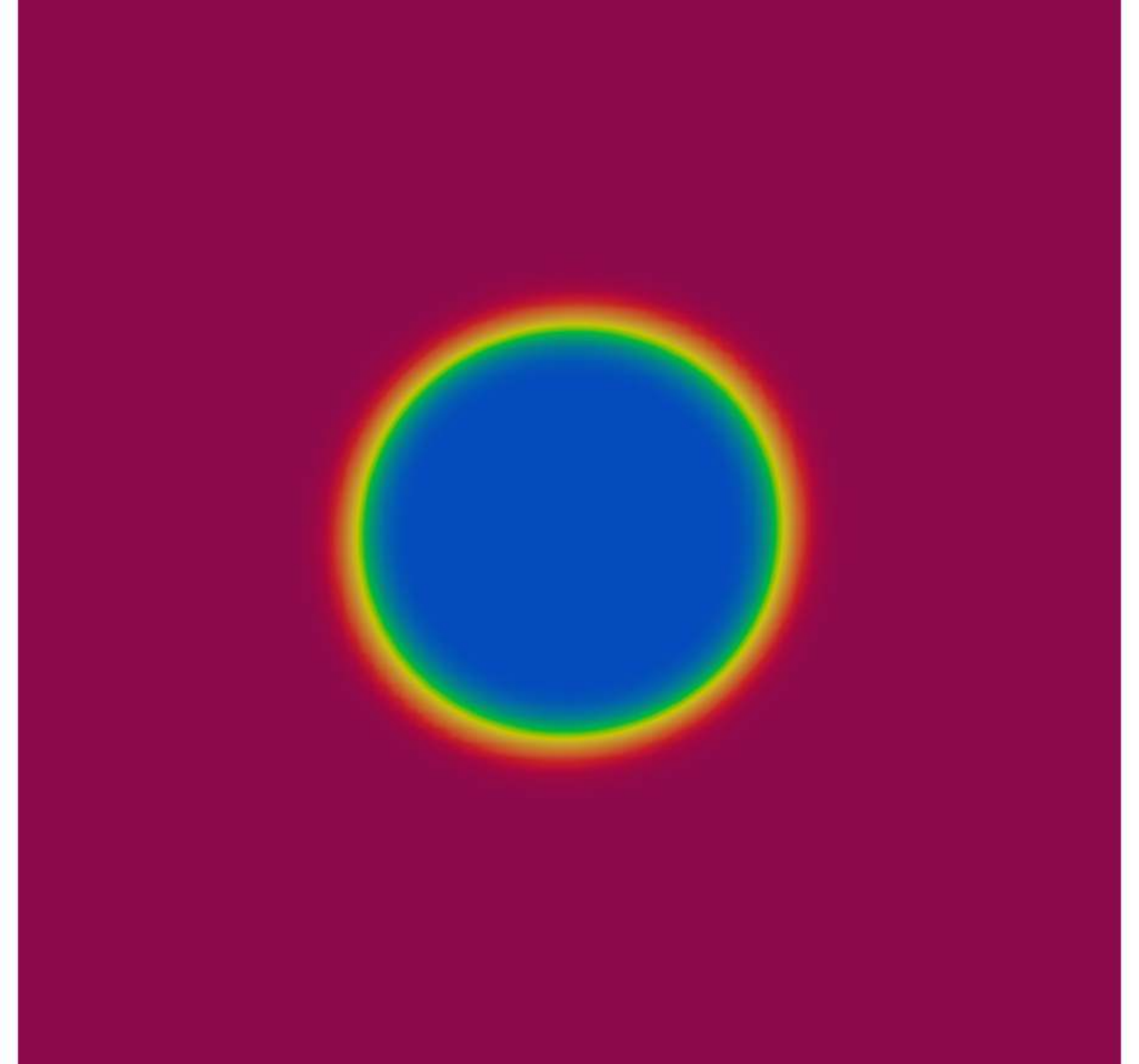}
\includegraphics[scale=0.09]{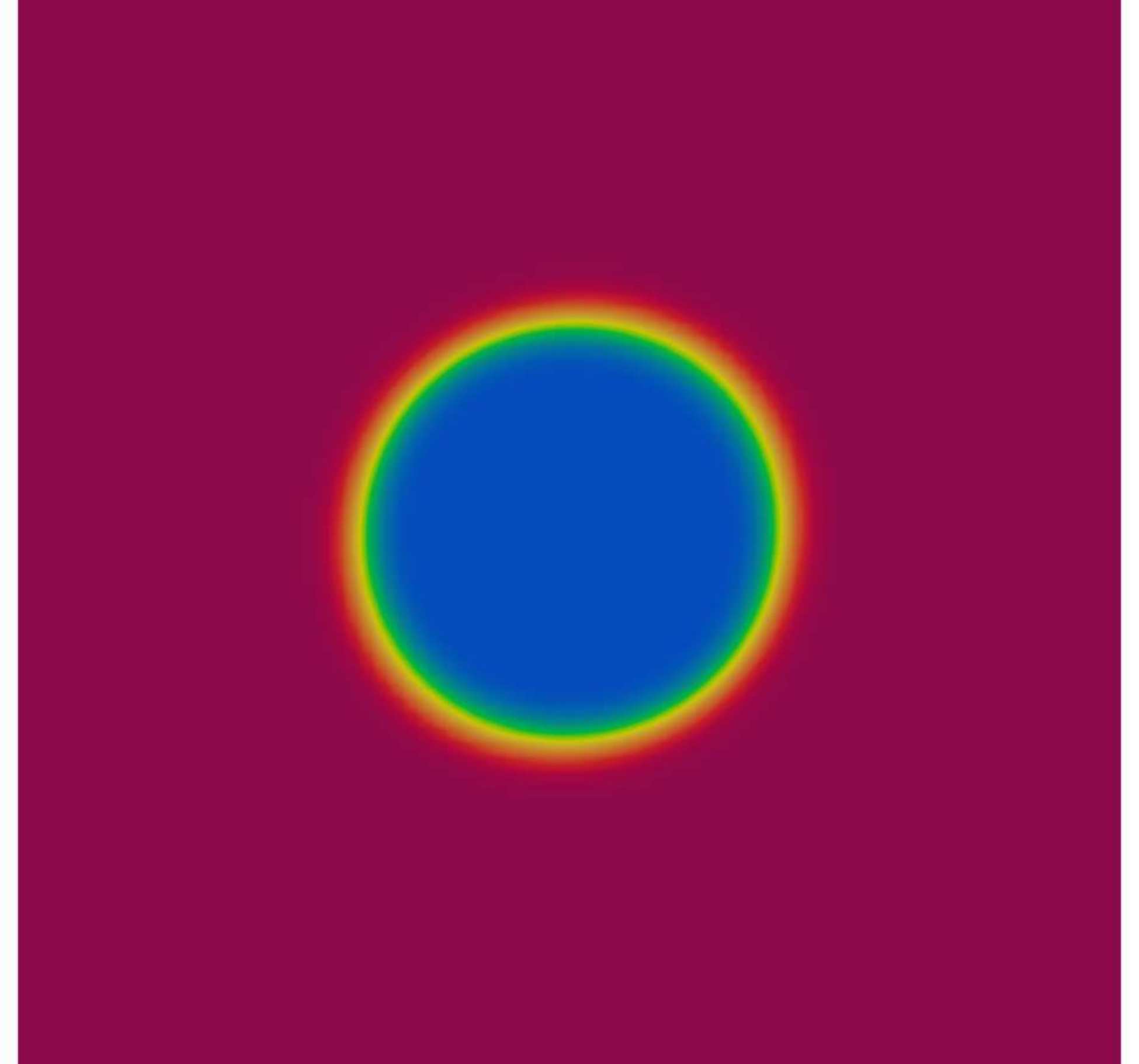}
\includegraphics[scale=0.09]{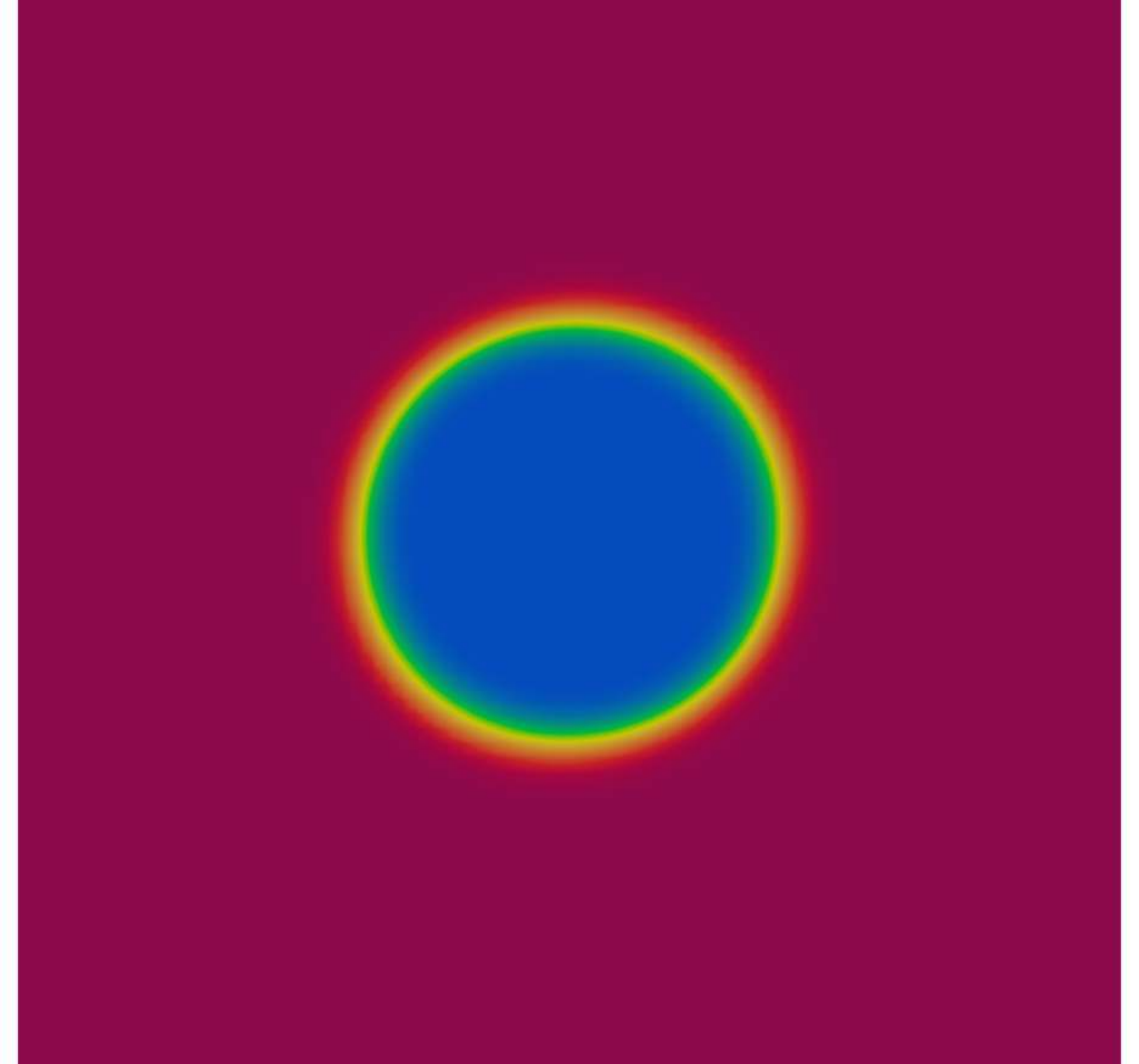}
\includegraphics[scale=0.09]{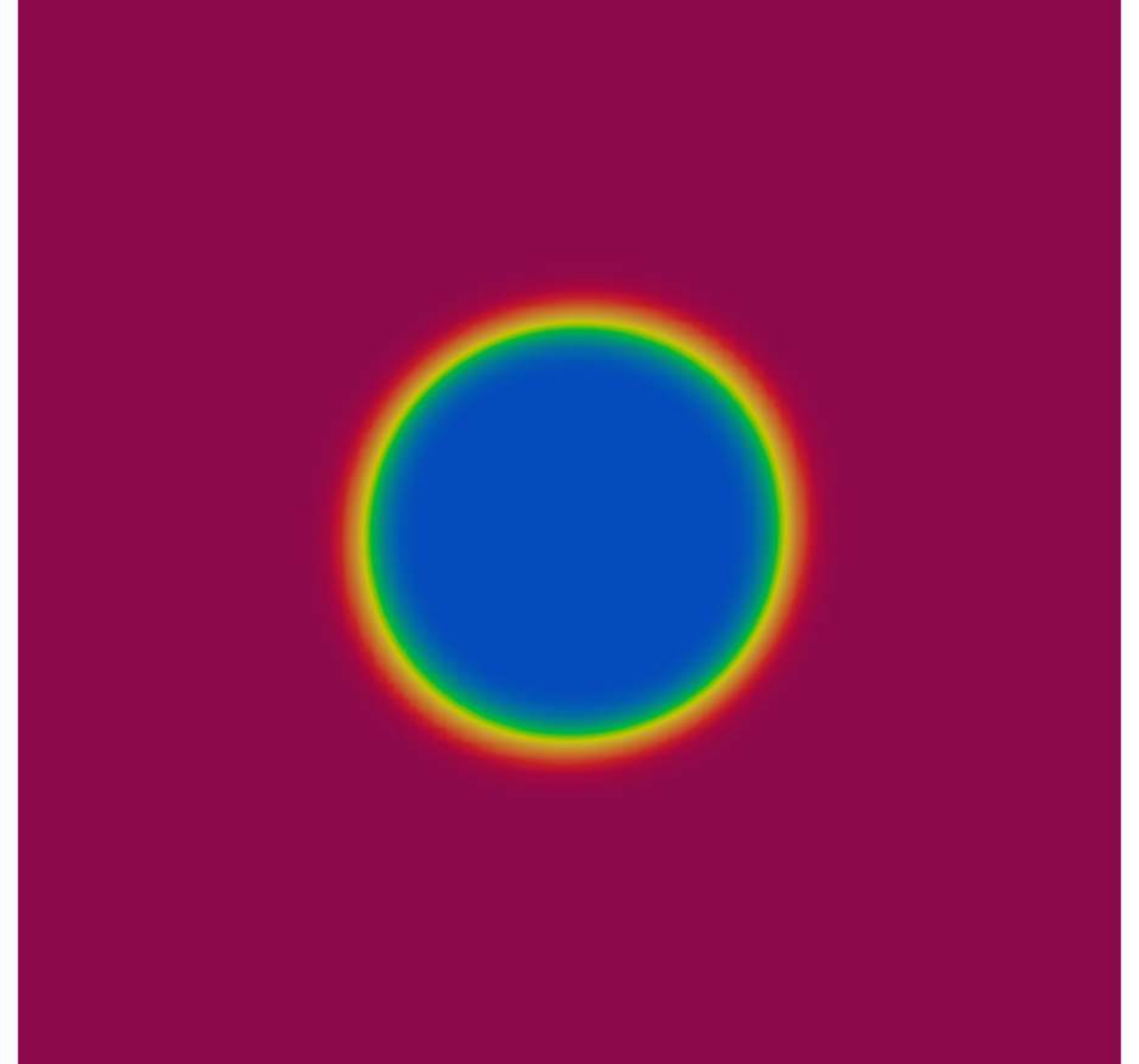}
\end{center}
\caption{
Dynamics of $\phi_1$ (top), $\phi_2$ (center) and $\phi_3$ (bottom) computed using scheme NTD1 at times $t=0.01, 0.05, 0.1, 0.15$ and $0.5$ (from left to right) with spreading coefficients $(\Sigma_1, \Sigma_2 , \Sigma_3) = (3,3,-0.1)$ and no penalization of the restriction $\Sigma_{i=1}^3\phi_i - 1$. In all plots red color denotes value 1 and blue color denotes value 0.}
\label{fig:BallsDynamicsLambda0}
\end{figure}

{
Then in Figure~\ref{fig:BallsDynamicsLambda} we present the dynamics of the results for different values of $\lambda>0$ and we present the evolution in time of the energy, volume and norms ($L^2$ and $L^\infty$) of the restriction in Figure~\ref{fig:BallsPlotsLambda}.
We can observe that, as expected, lower values of $\lambda$ produce better approximations of the constraint with the same fixed values of mesh discretization in time and space, being $\lambda=$1e-4 the one that start to achieve a reasonable approximation, but there is a moment that if this parameter is too low, the quality of the obtained approximation can be affected, as its happening is our case when $\lambda=$1e-5, that although the equilibrium configuration is the same as the other cases, the dynamics seems a bit awkward (as the evolution of the energy). This fact is related with introducing a small parameter in the linear system, that if it is not reduced as the same time as the discretization parameters, it can make the linear systems difficult to solve.
}
\begin{figure}[h]
\begin{center}
\includegraphics[scale=0.09]{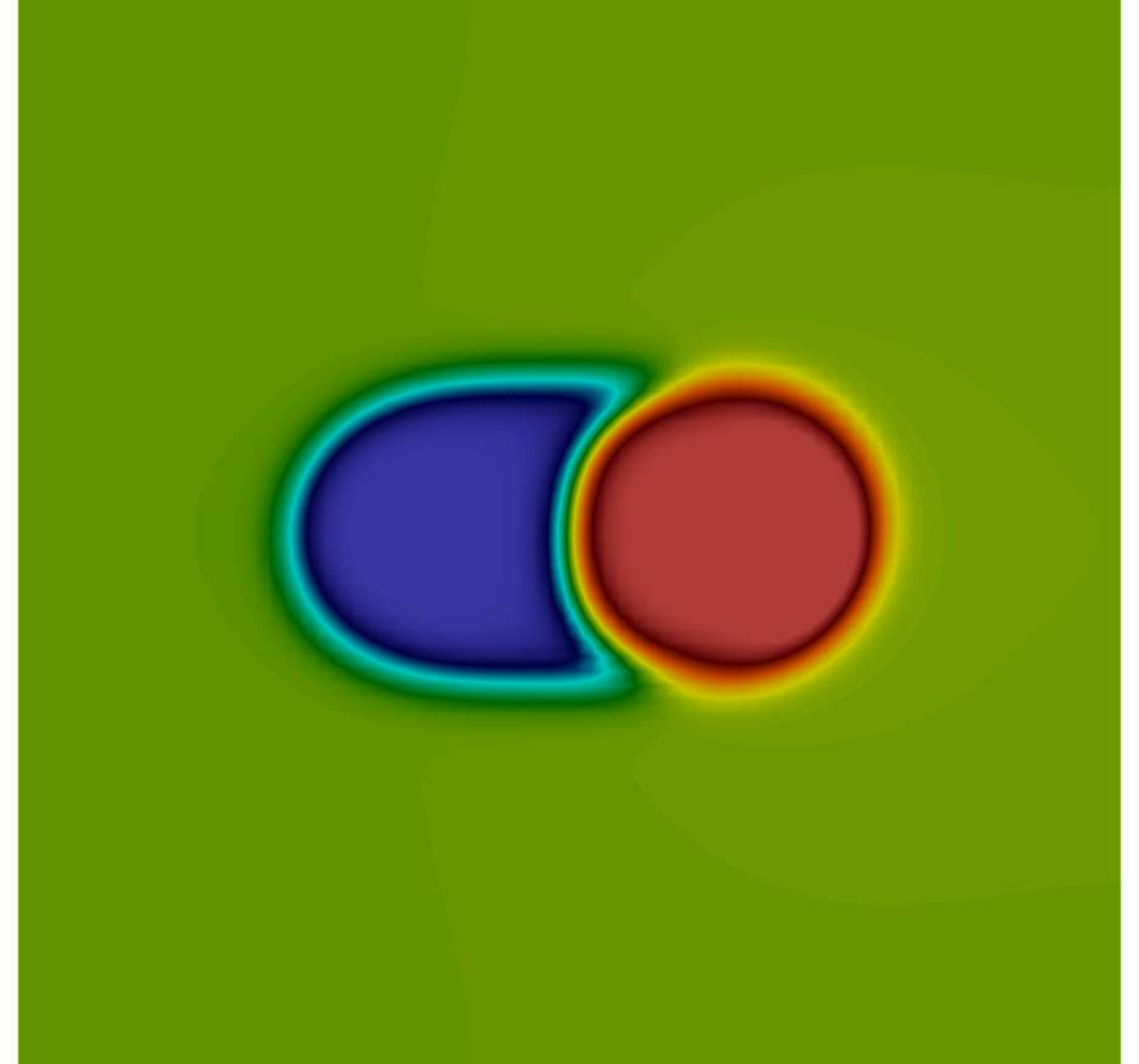}
\includegraphics[scale=0.09]{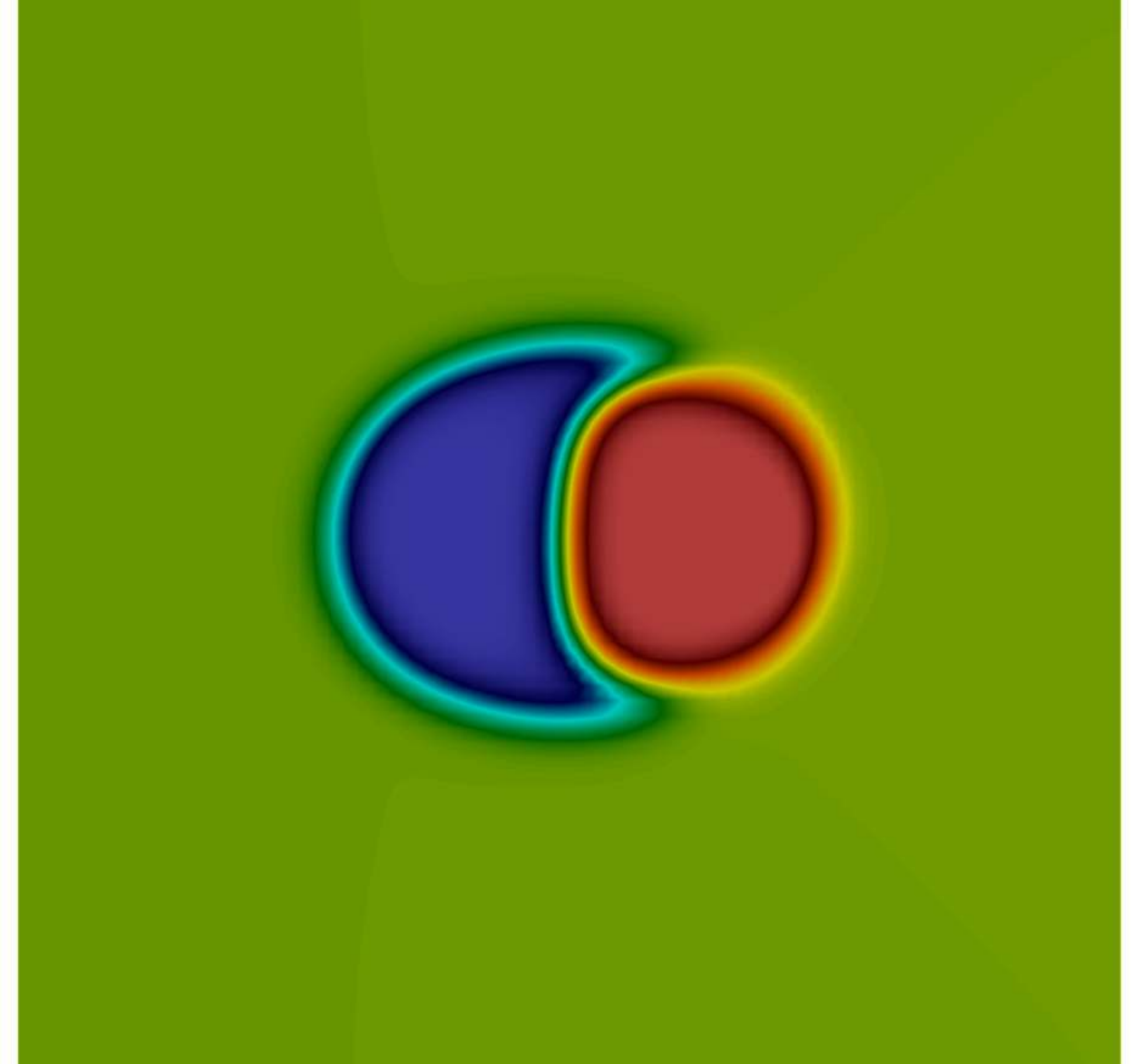}
\includegraphics[scale=0.09]{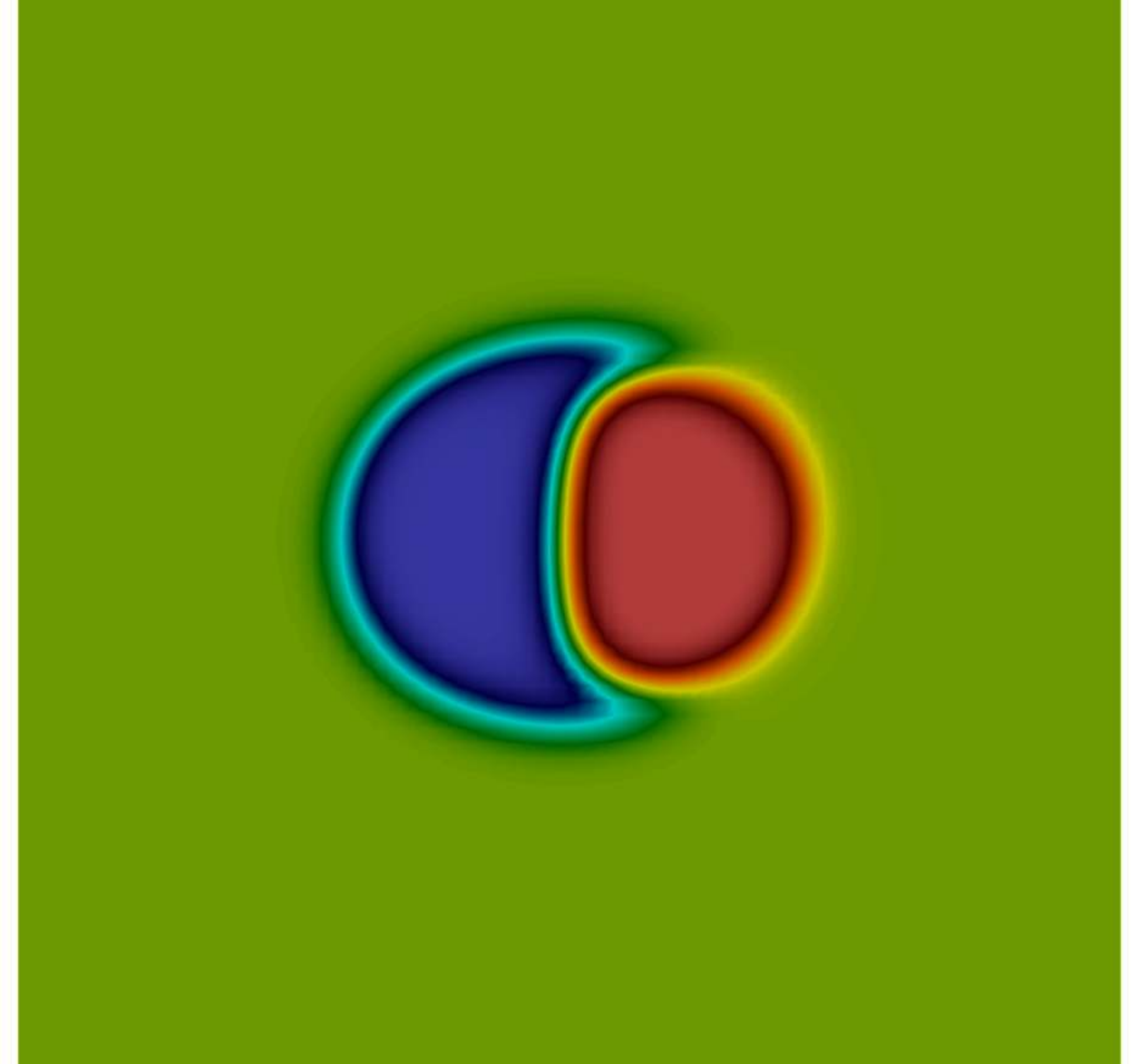}
\includegraphics[scale=0.09]{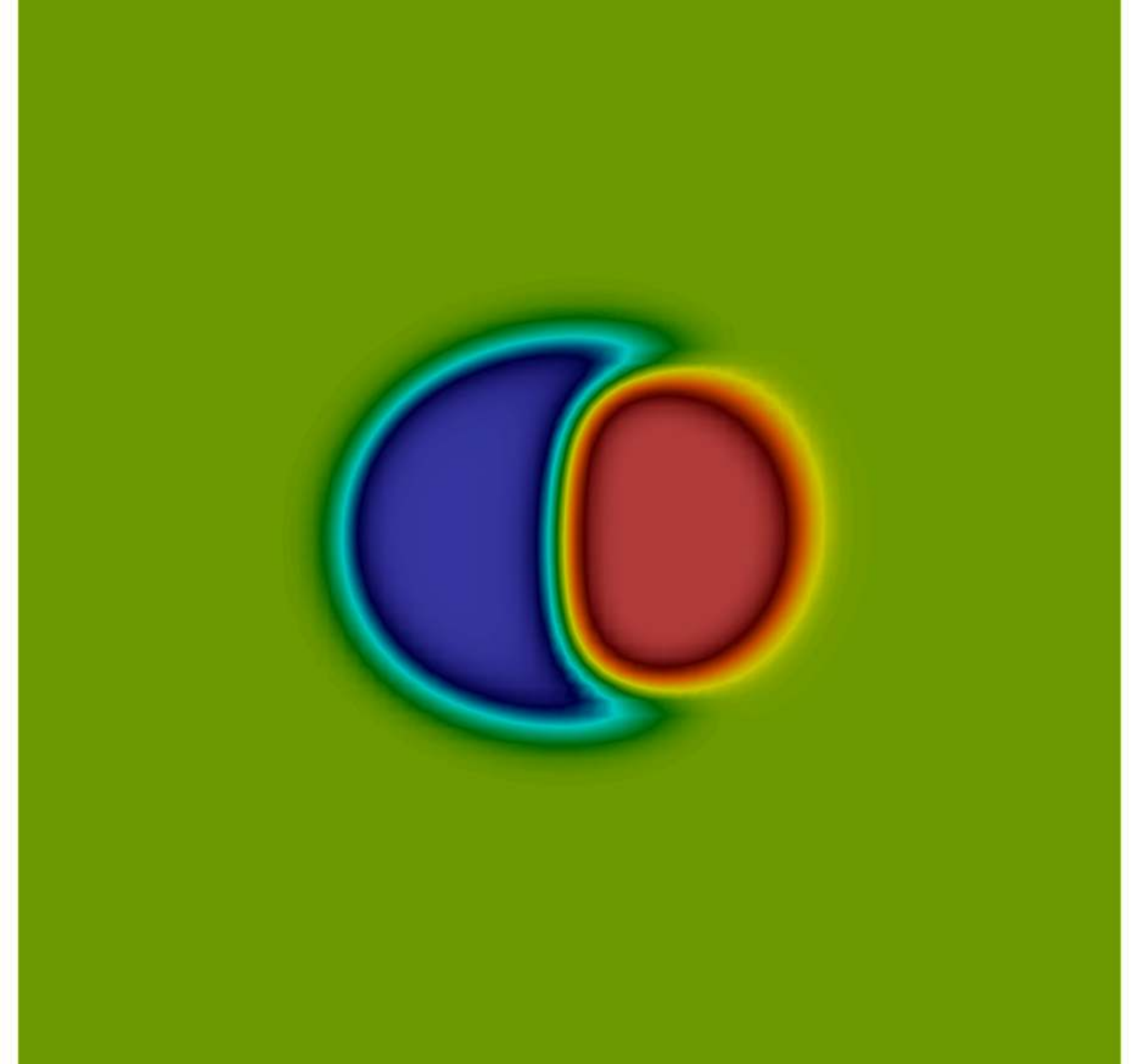}
\includegraphics[scale=0.09]{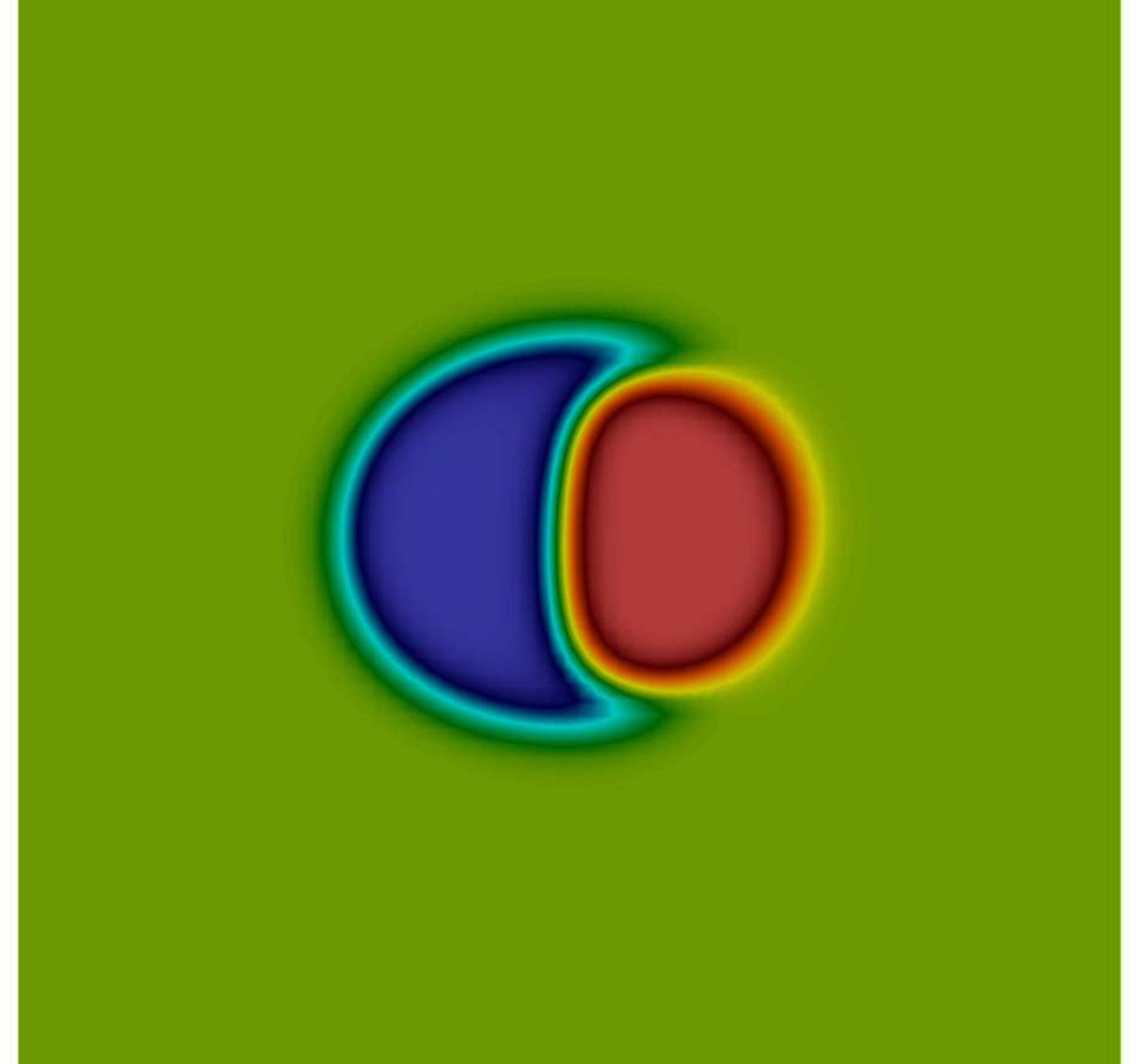}
\\ [1ex]
\includegraphics[scale=0.09]{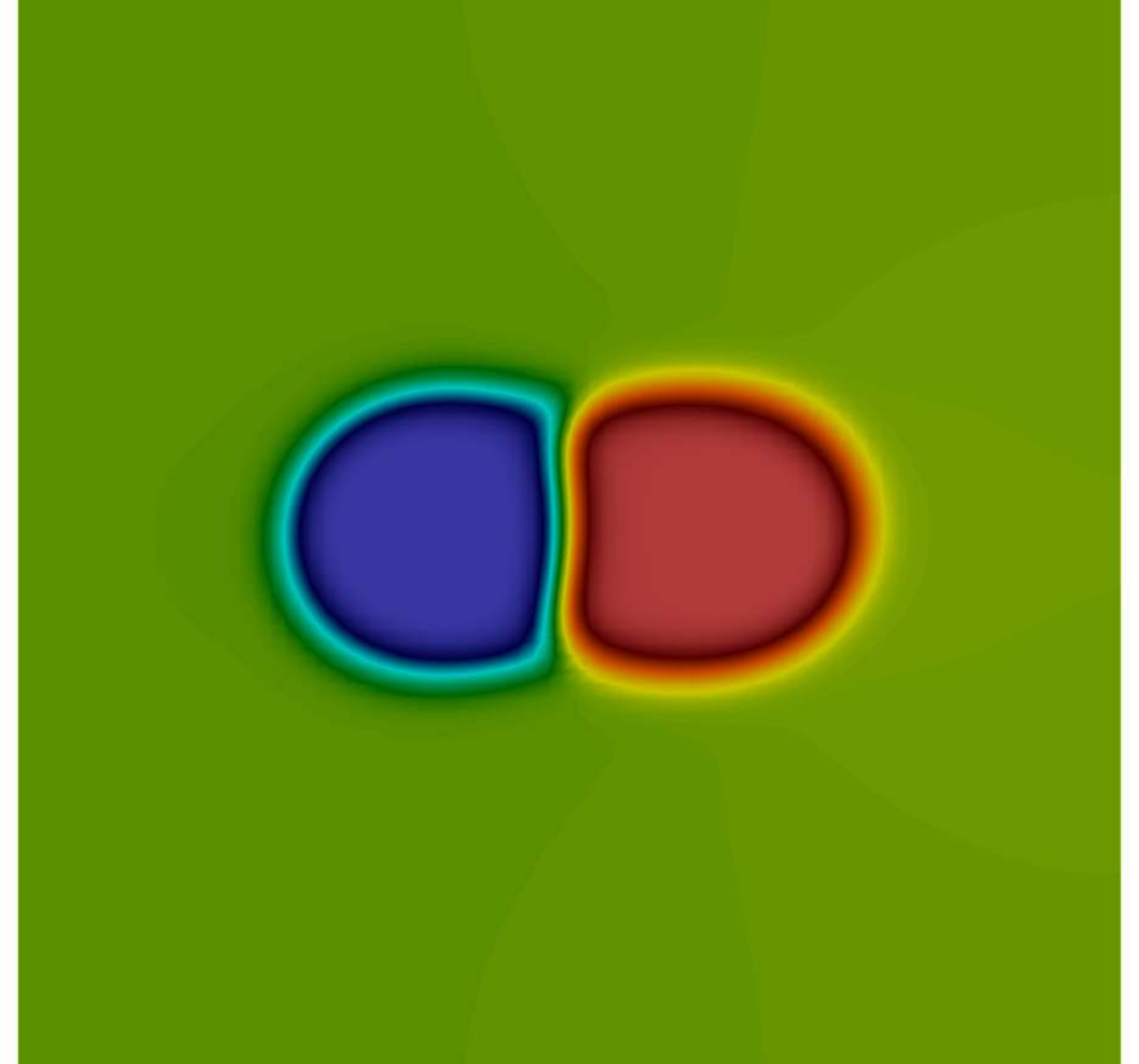}
\includegraphics[scale=0.09]{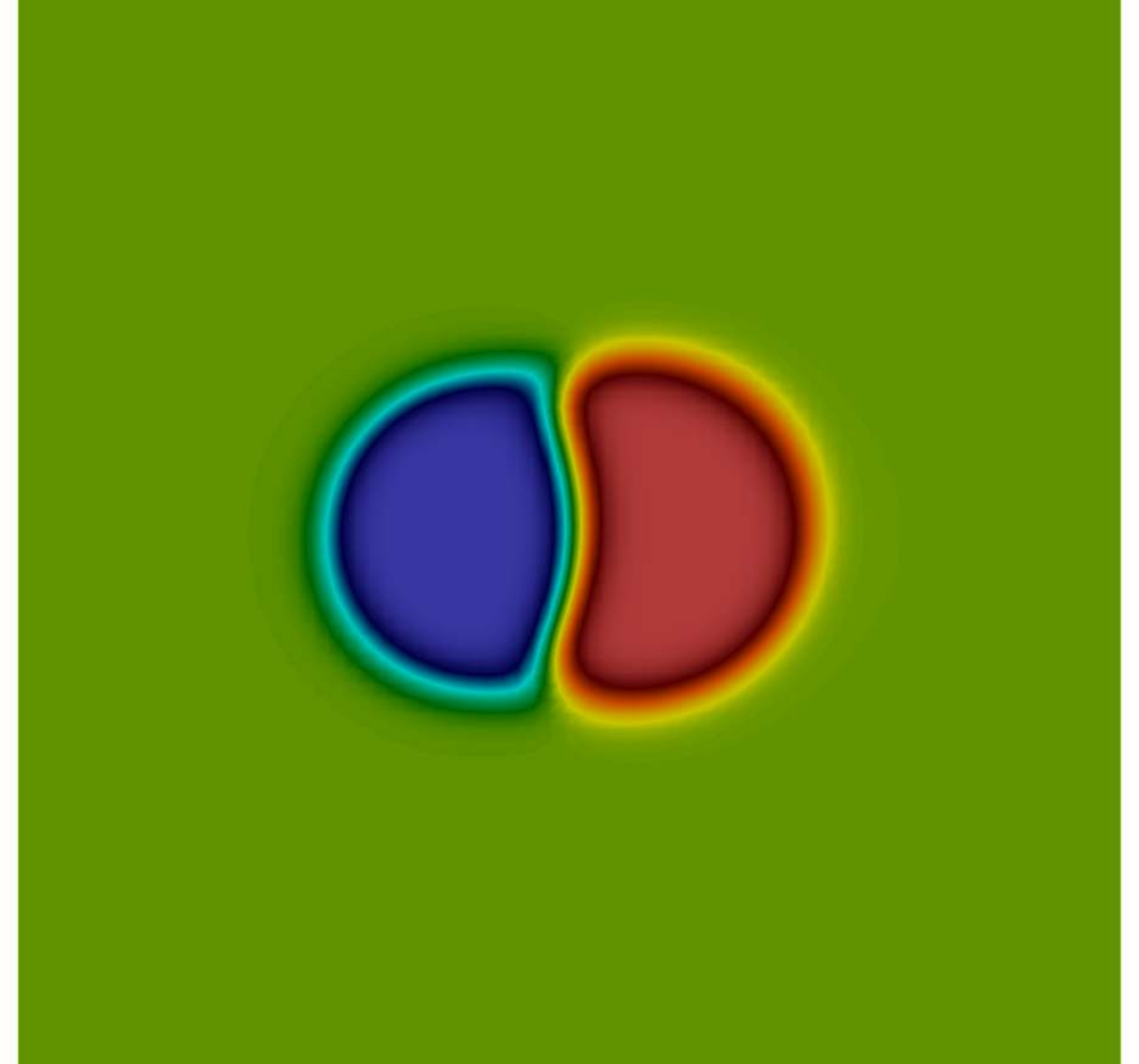}
\includegraphics[scale=0.09]{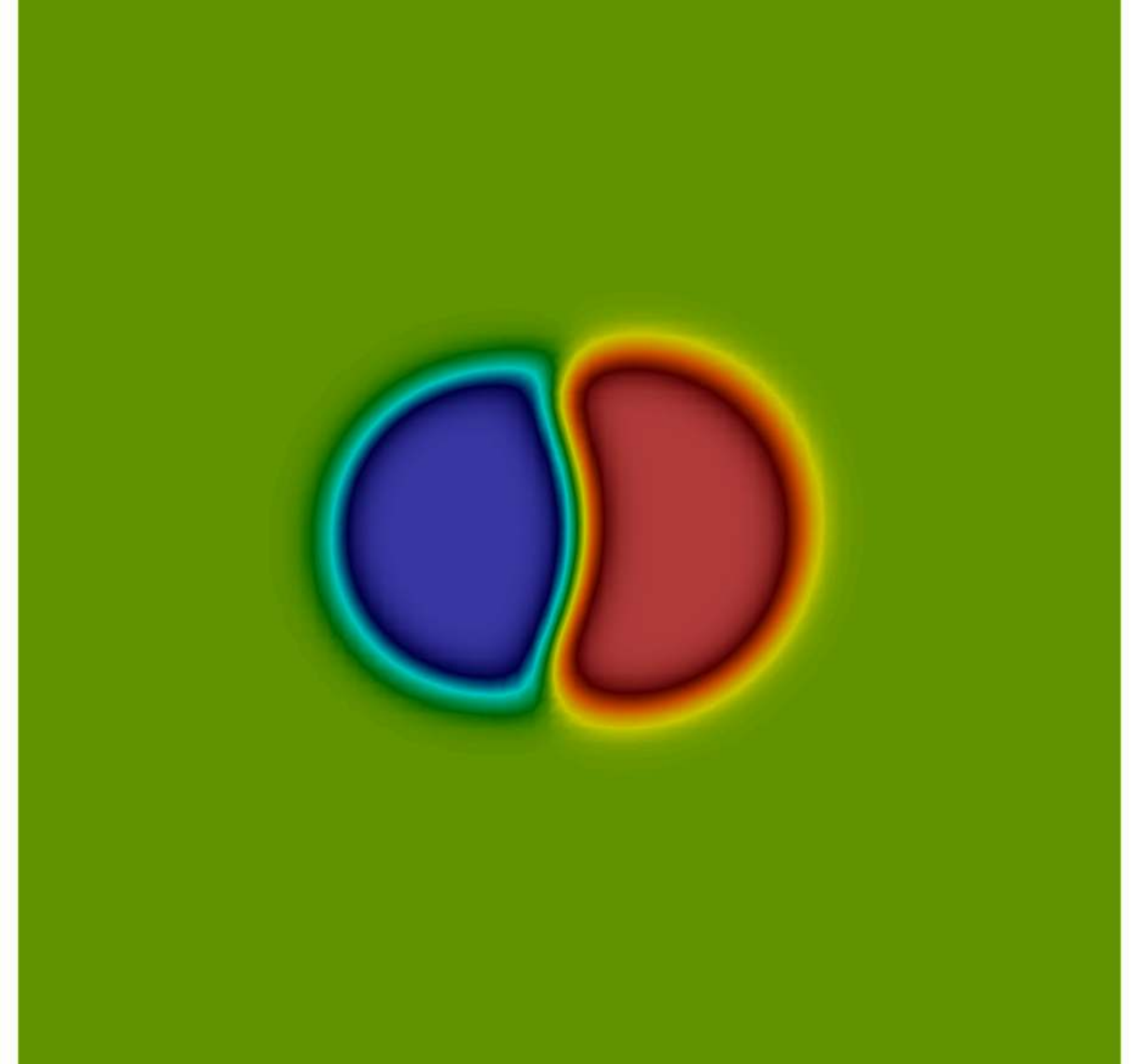}
\includegraphics[scale=0.09]{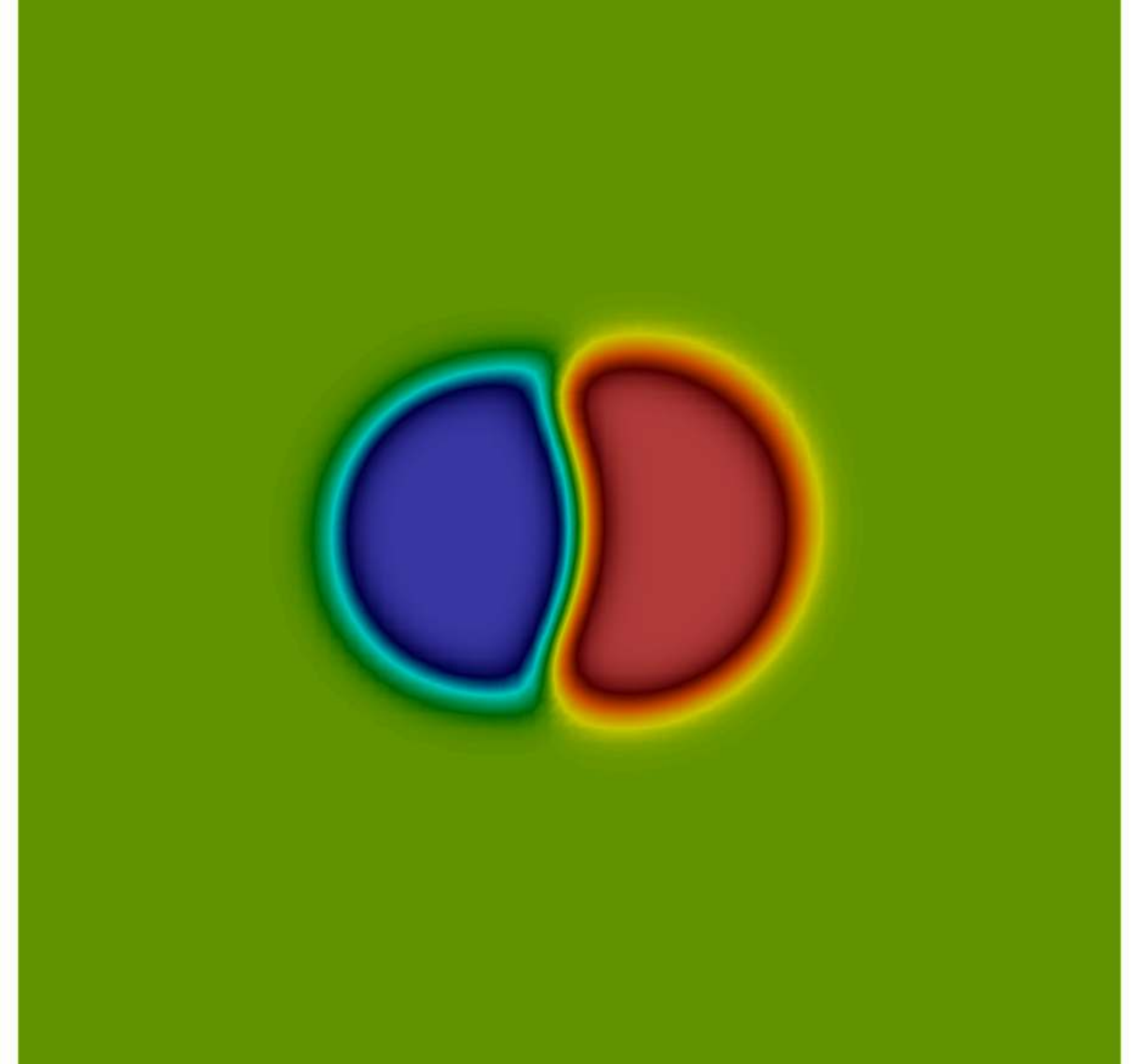}
\includegraphics[scale=0.09]{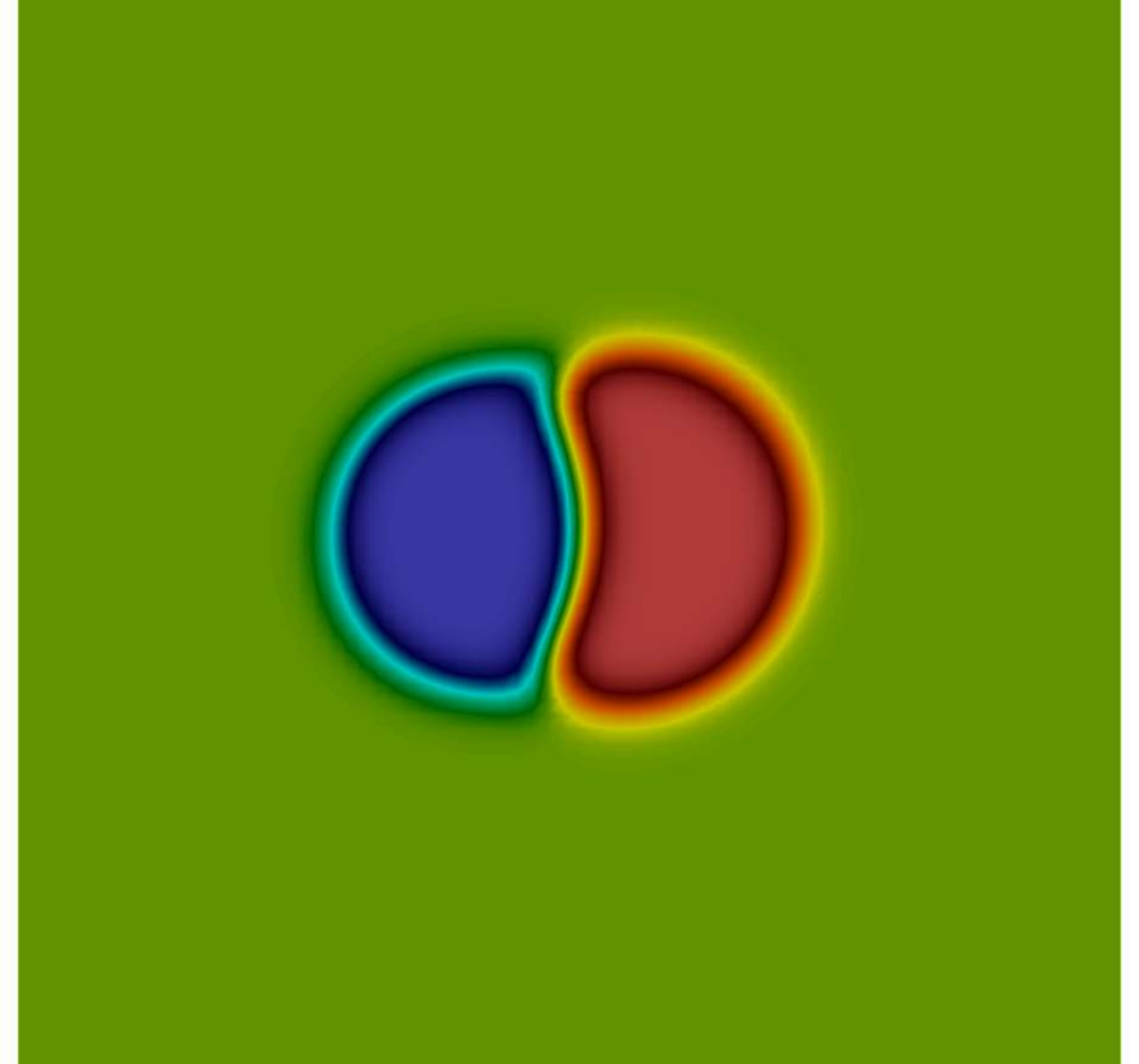}
\\ [1ex]
\includegraphics[scale=0.09]{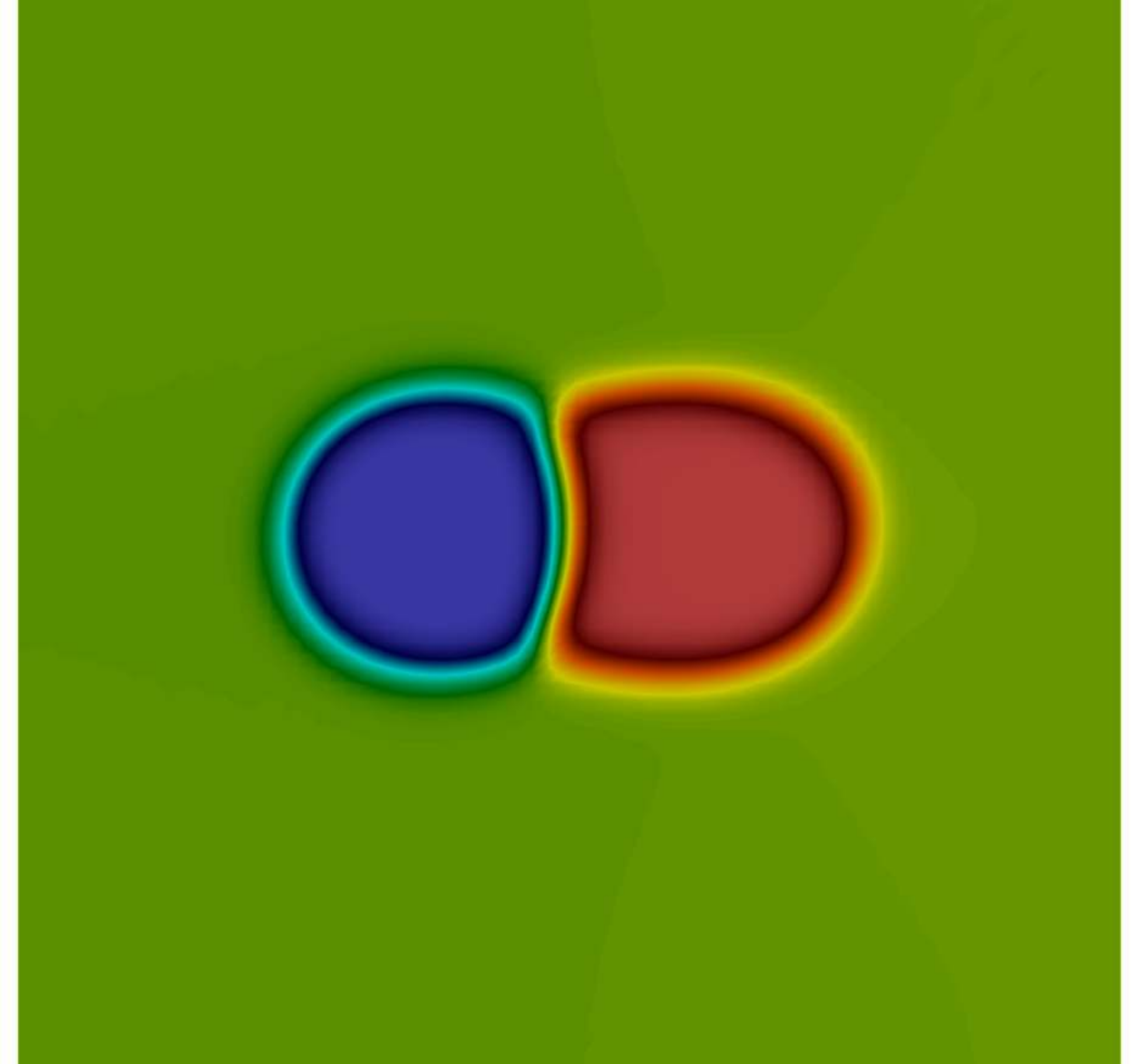}
\includegraphics[scale=0.09]{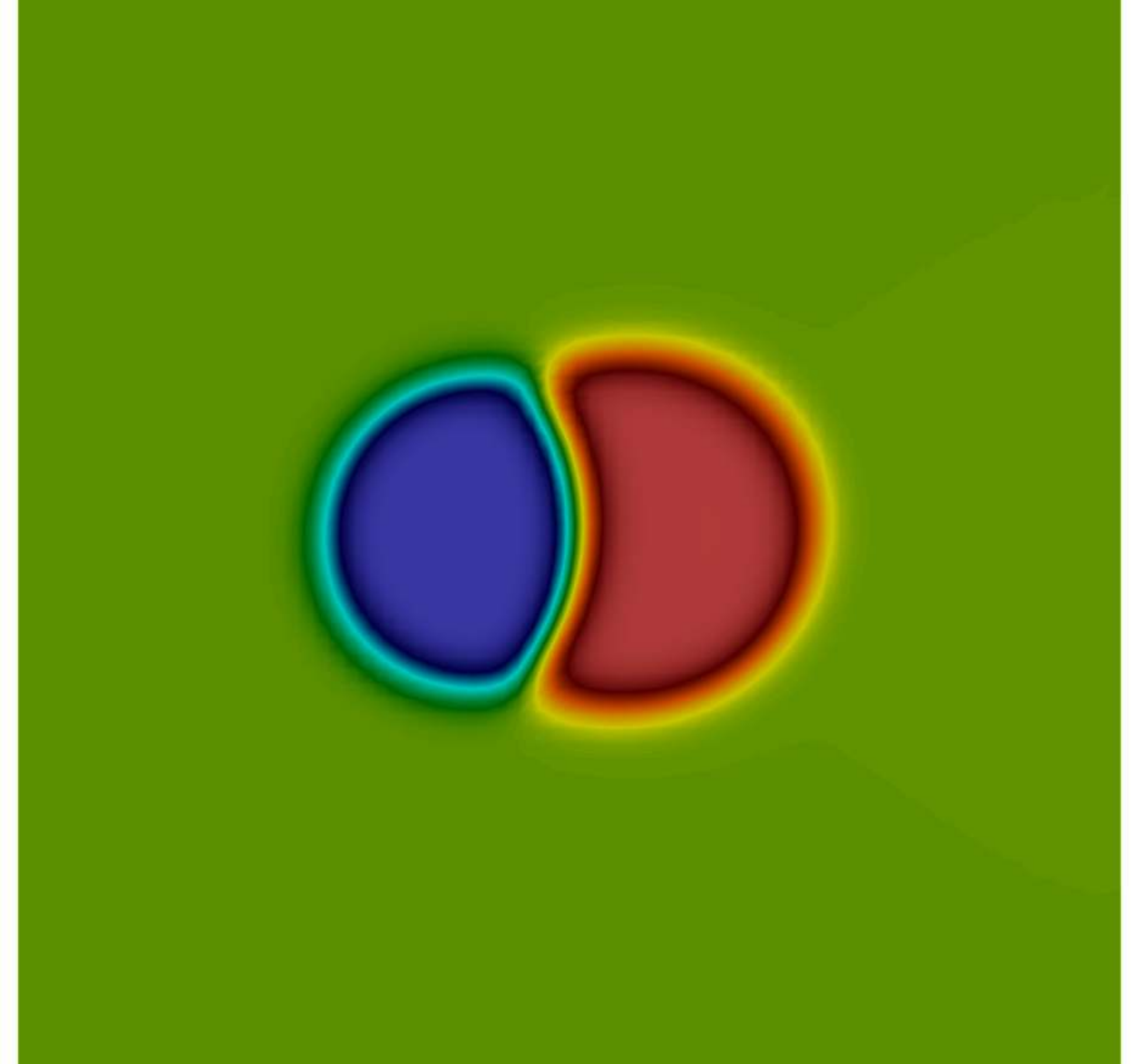}
\includegraphics[scale=0.09]{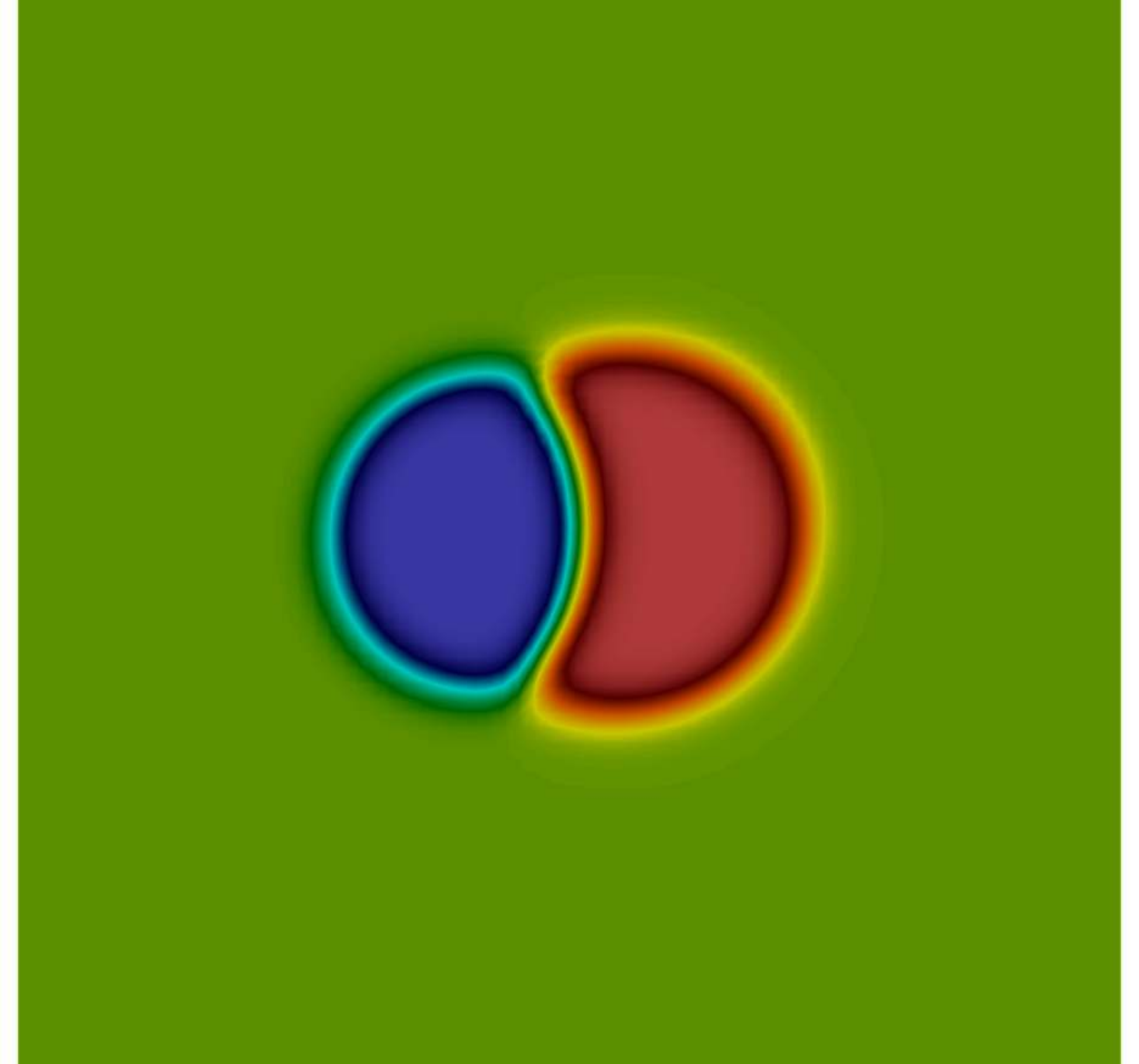}
\includegraphics[scale=0.09]{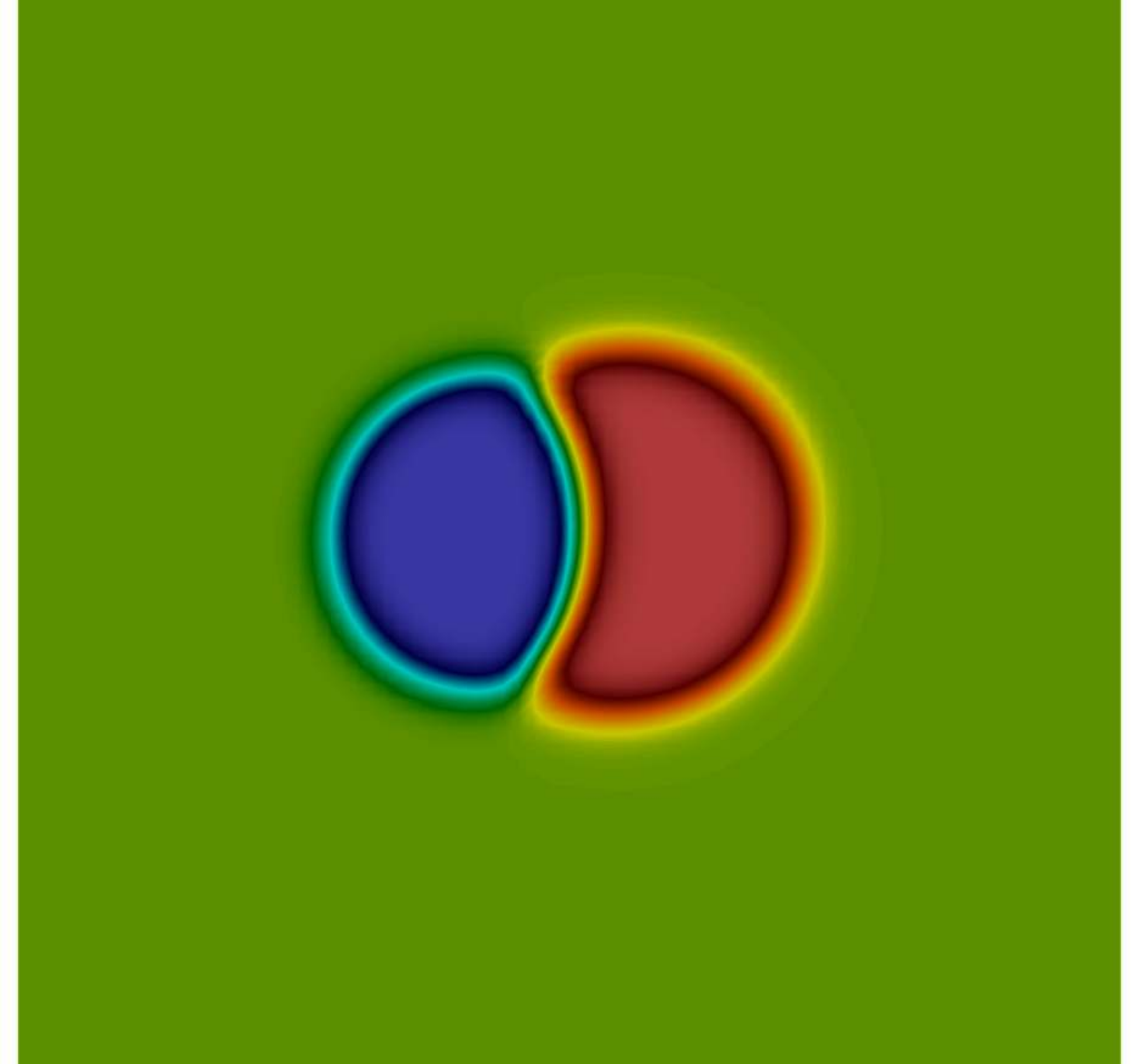}
\includegraphics[scale=0.09]{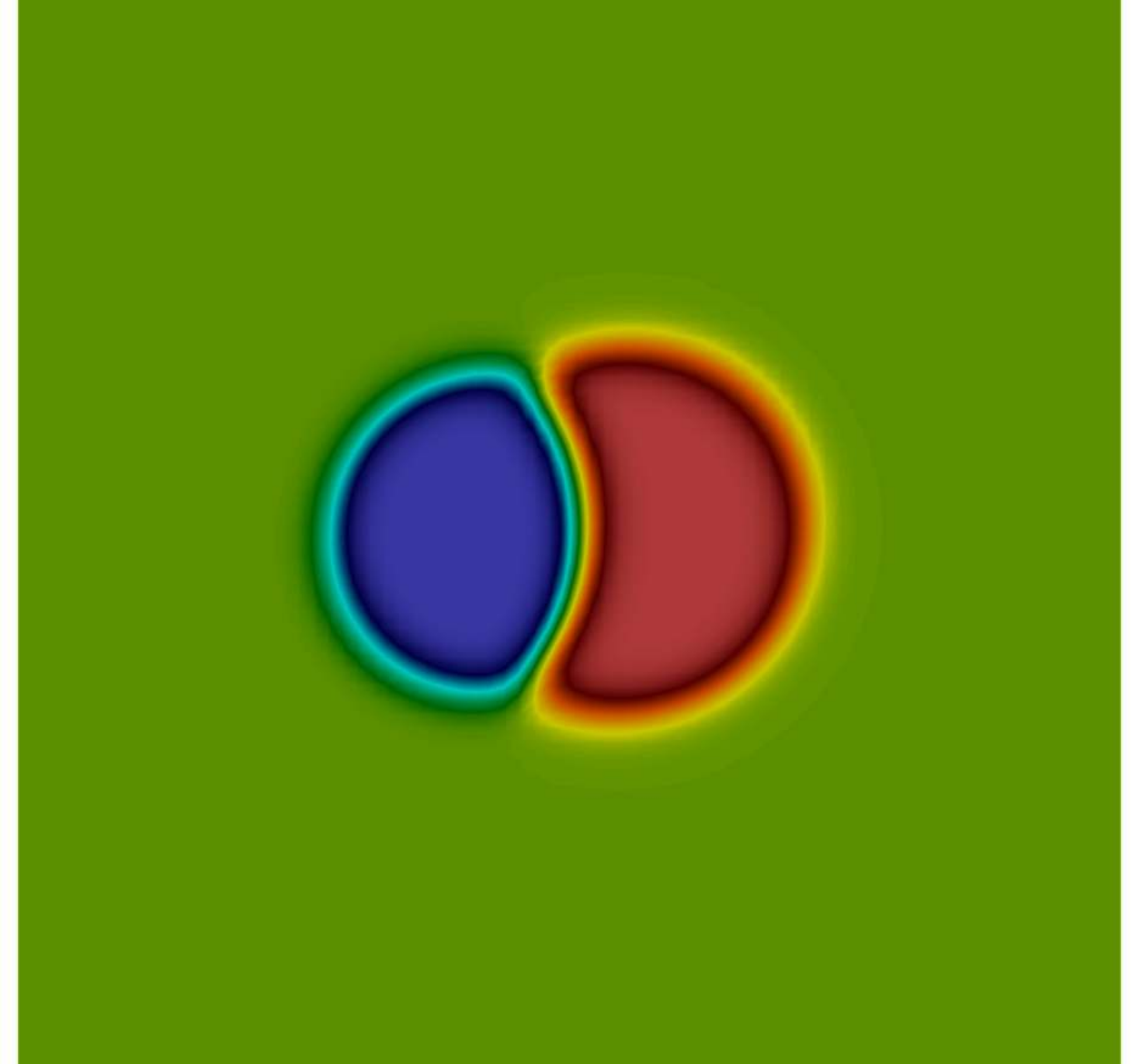}
\\ [1ex]
\includegraphics[scale=0.09]{new/2_Balls/2_1.pdf}
\includegraphics[scale=0.09]{new/2_Balls/2_5.pdf}
\includegraphics[scale=0.09]{new/2_Balls/2_10.pdf}
\includegraphics[scale=0.09]{new/2_Balls/2_15.pdf}
\includegraphics[scale=0.09]{new/2_Balls/2_50.pdf}
\\ [1ex]
\includegraphics[scale=0.09]{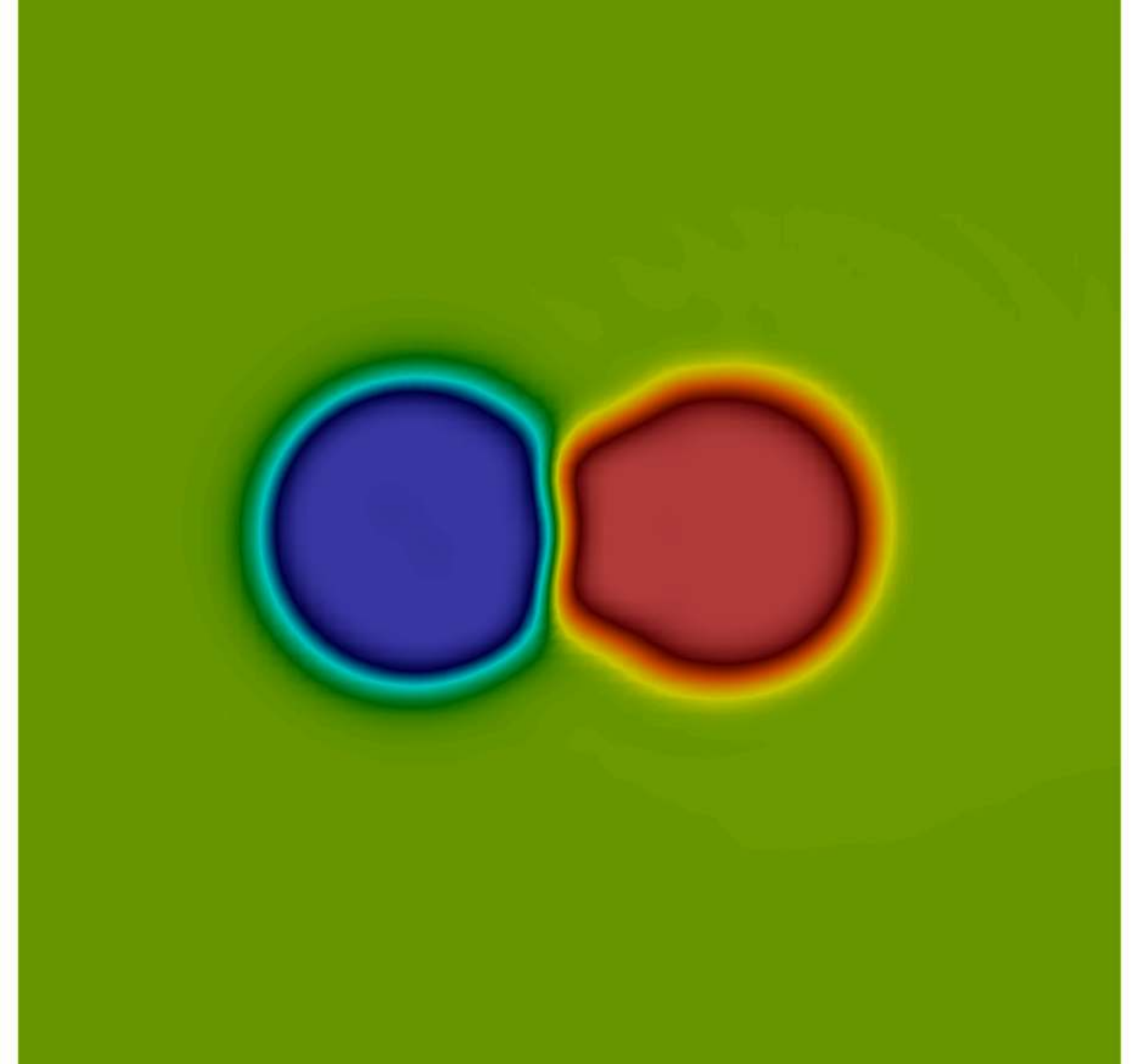}
\includegraphics[scale=0.09]{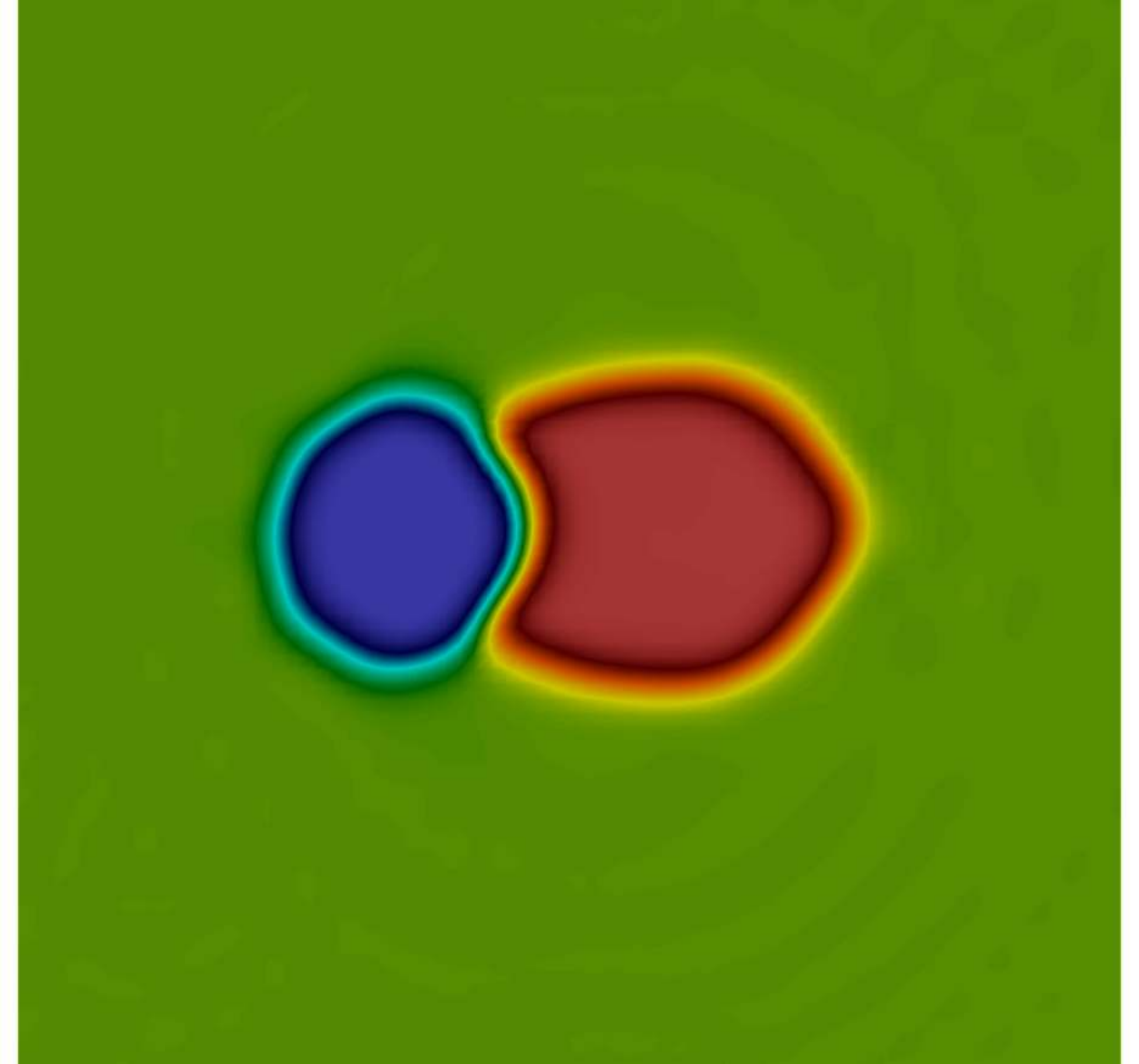}
\includegraphics[scale=0.09]{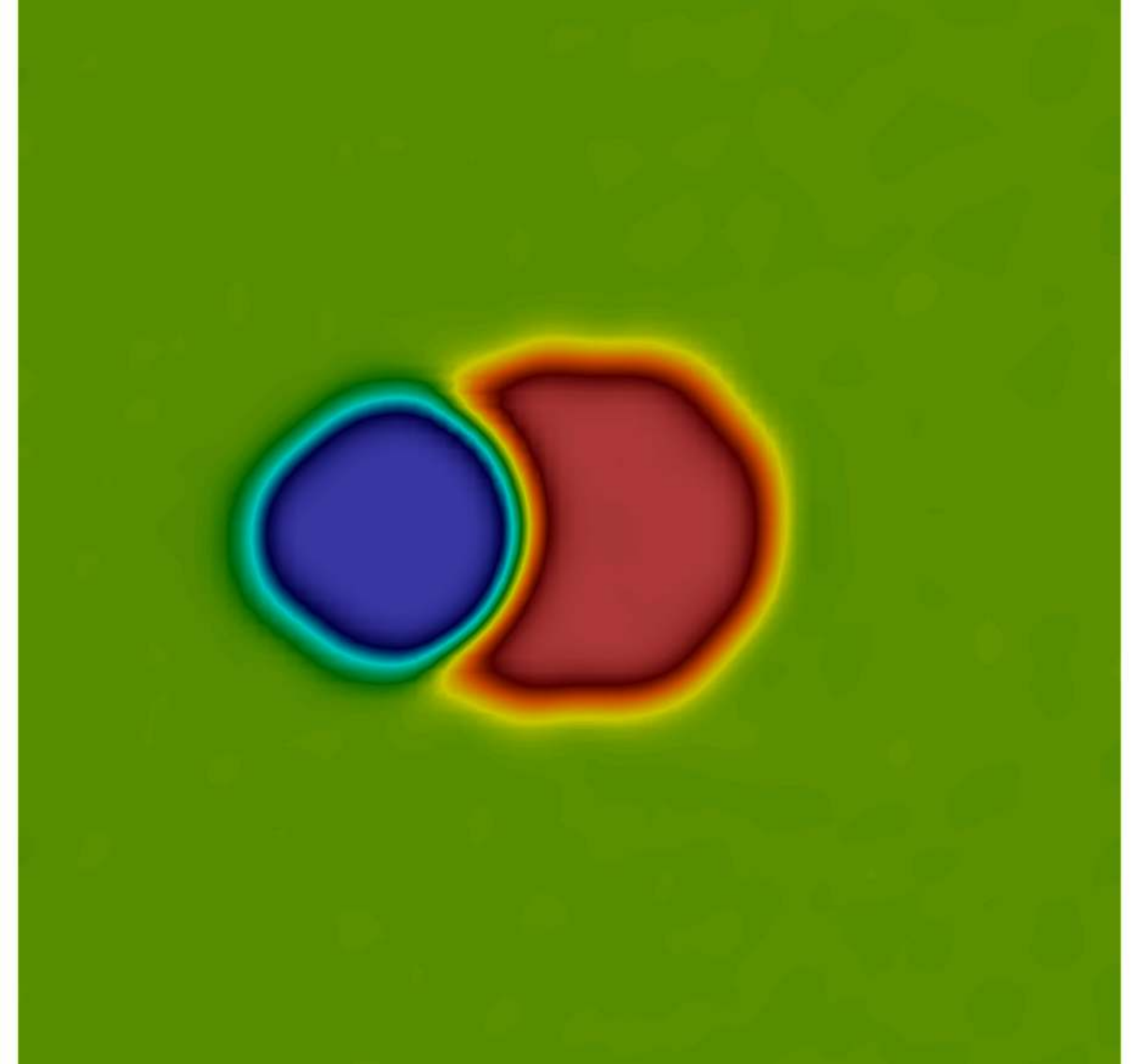}
\includegraphics[scale=0.09]{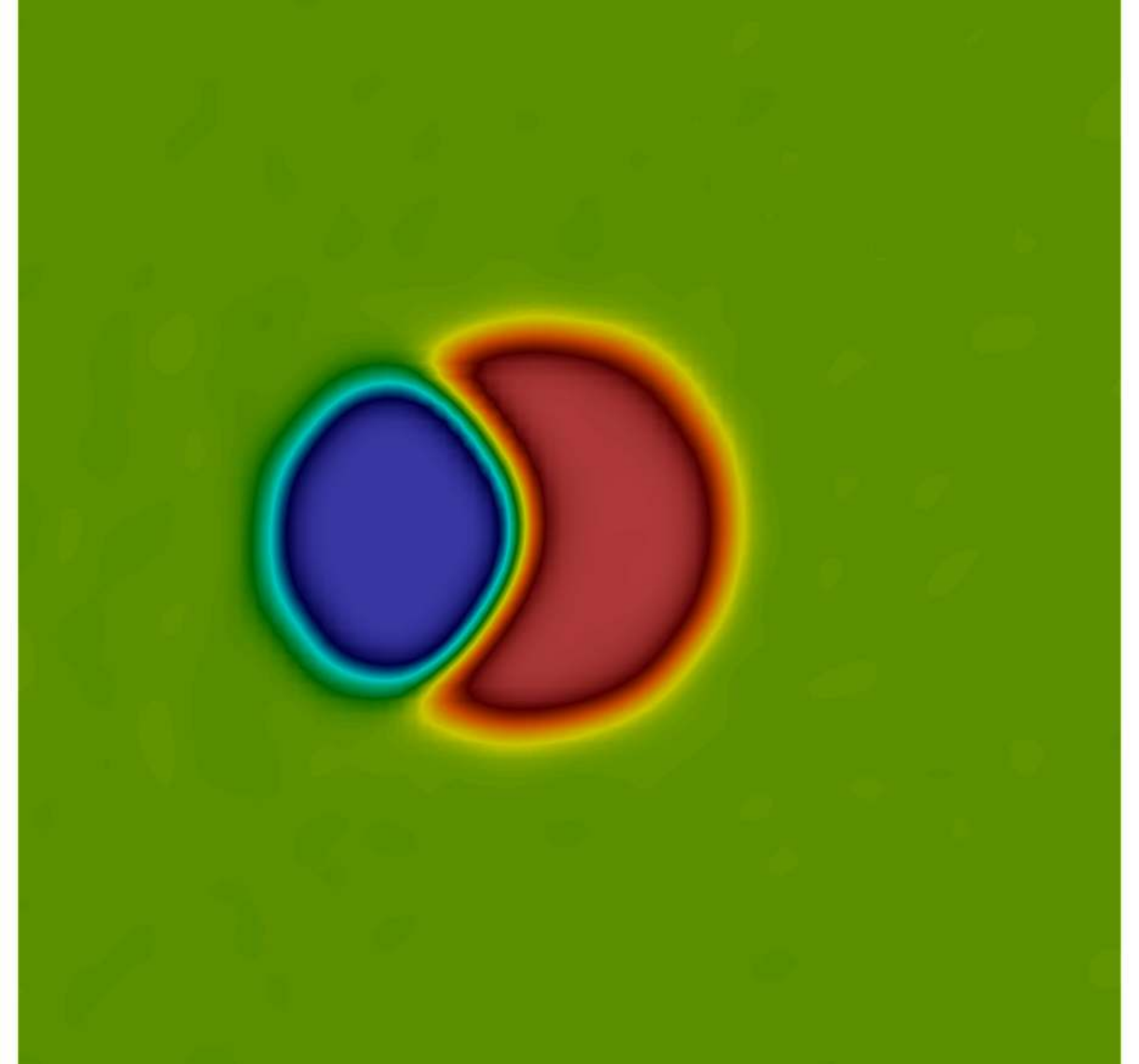}
\includegraphics[scale=0.09]{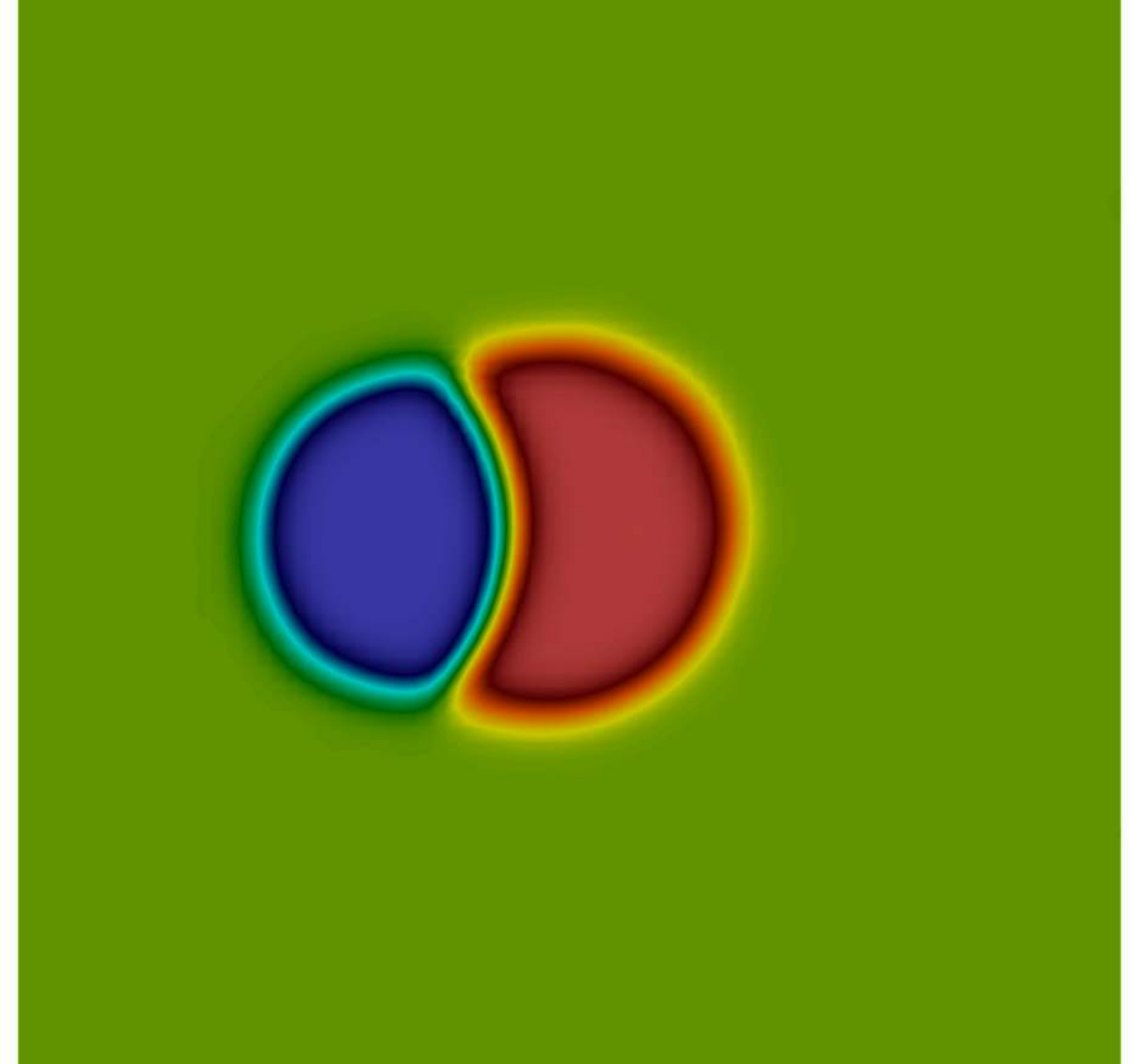}
\end{center}
\caption{Dynamics of scheme NTD1 at times $t=0.01, 0.05, 0.1, 0.15$ and $0.5$ (from left to right) with spreading coefficients $(\Sigma_1, \Sigma_2 , \Sigma_3) = (3,3,-0.1)$ and with penalization parameter
$\lambda=$1e-1 (top row),
$\lambda=$1e-2 (second row),
$\lambda=$1e-3 (third row),
$\lambda=$1e-4 (fourth row) and 
$\lambda=$1e-5 (last row).}
\label{fig:BallsDynamicsLambda}
\end{figure}

\begin{figure}[h]
\begin{center}
\includegraphics[scale=0.11]{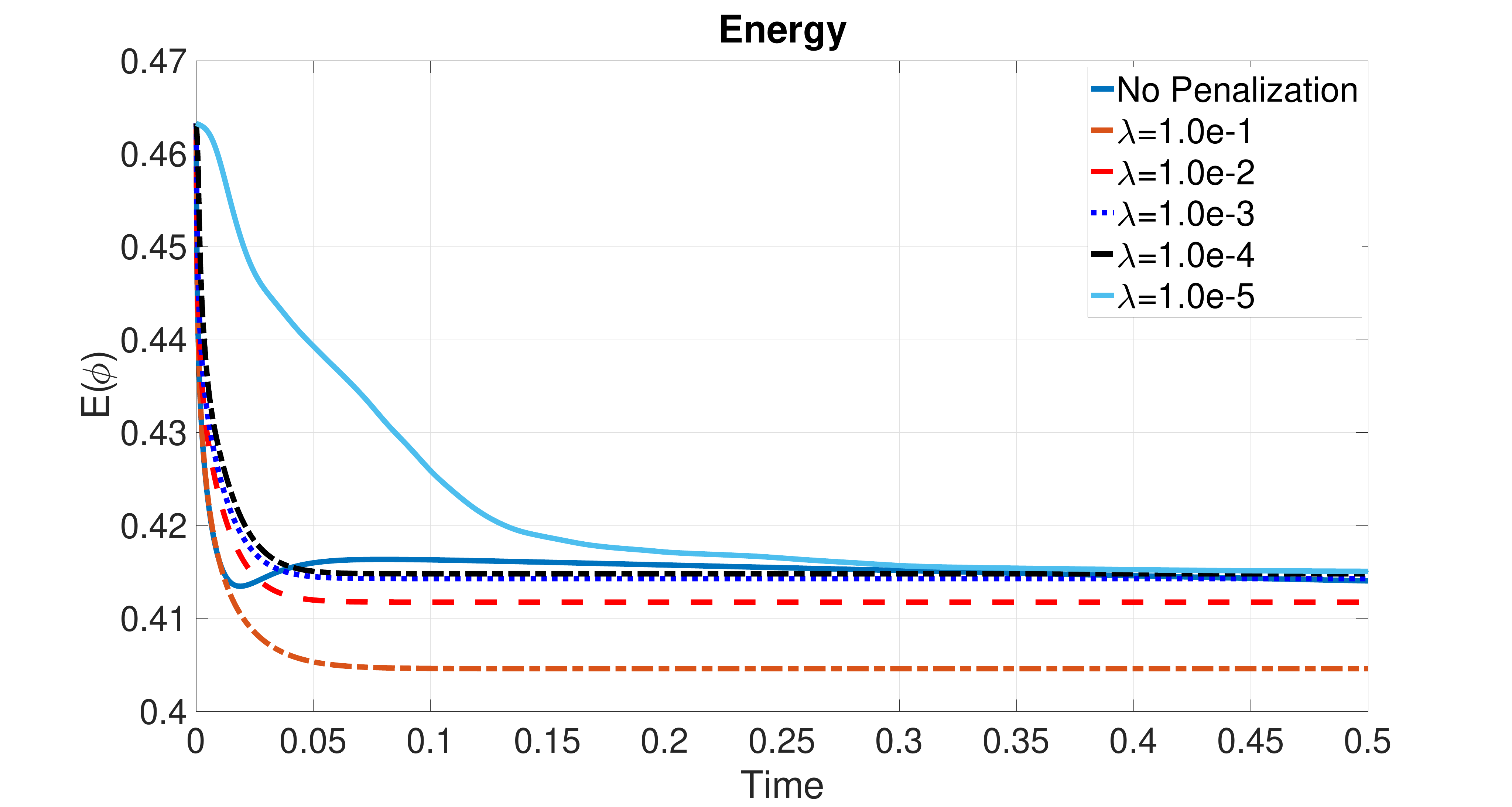}
\includegraphics[scale=0.11]{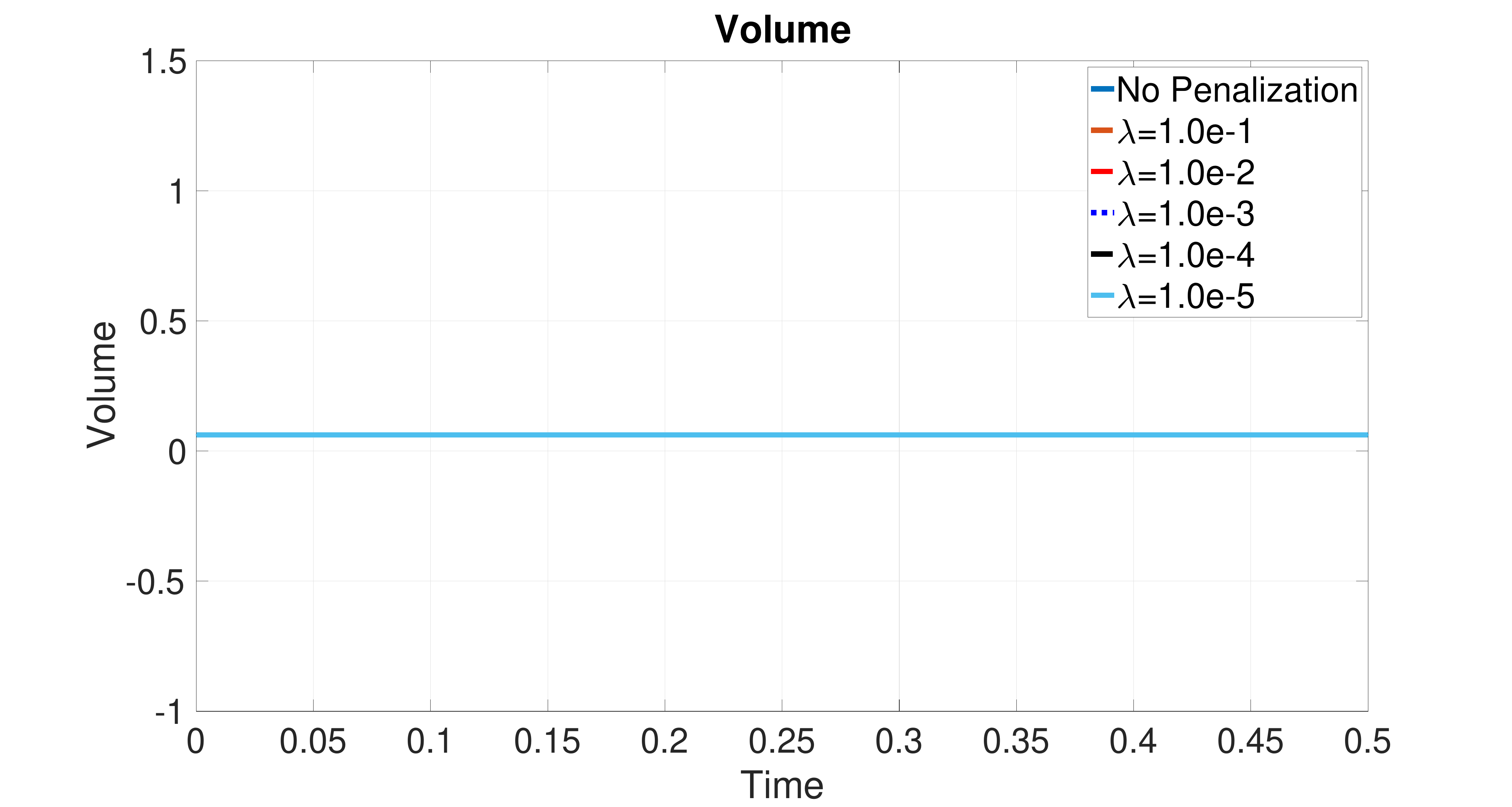}
\\ [1ex]
\includegraphics[scale=0.11]{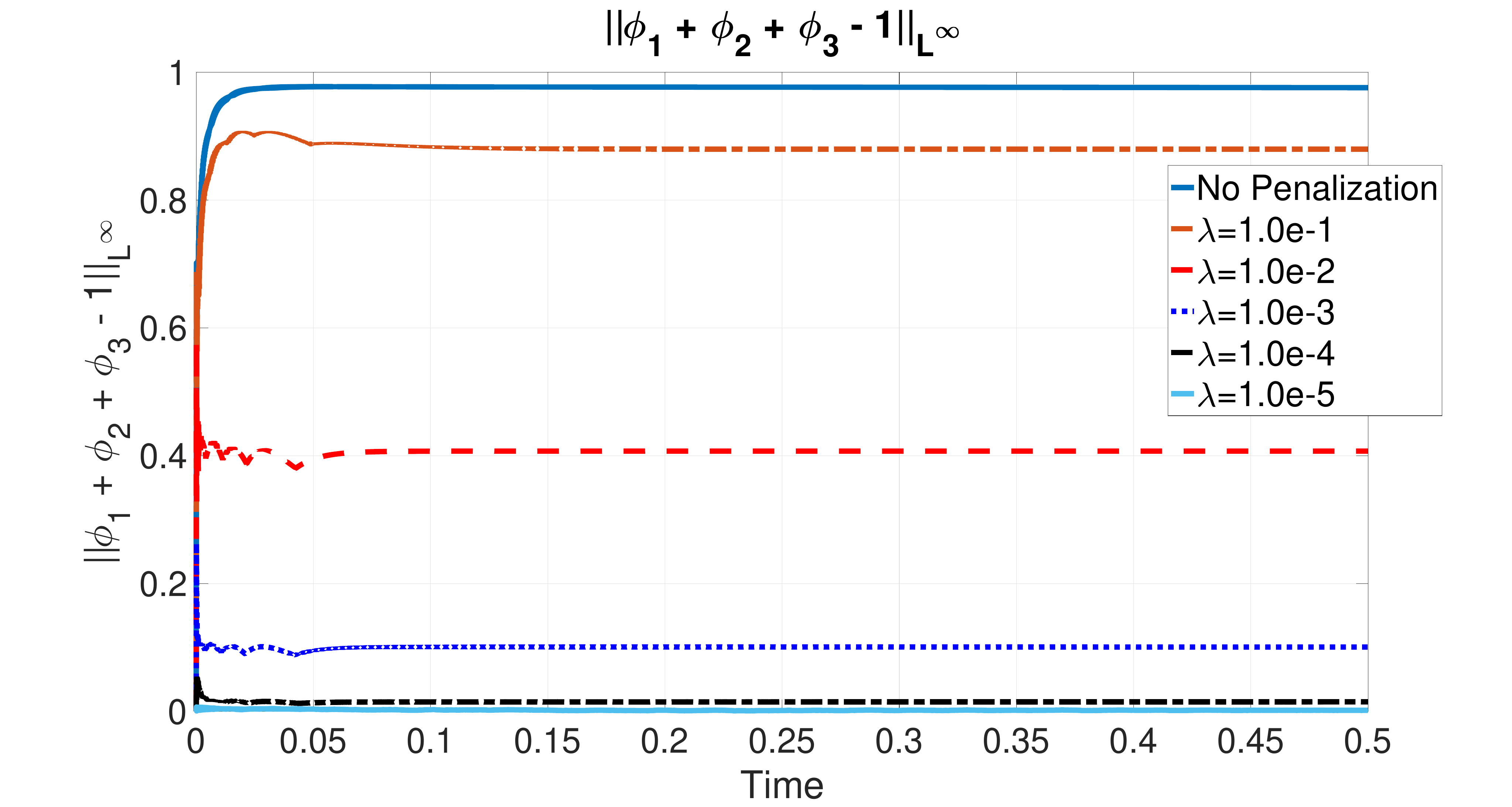}
\includegraphics[scale=0.11]{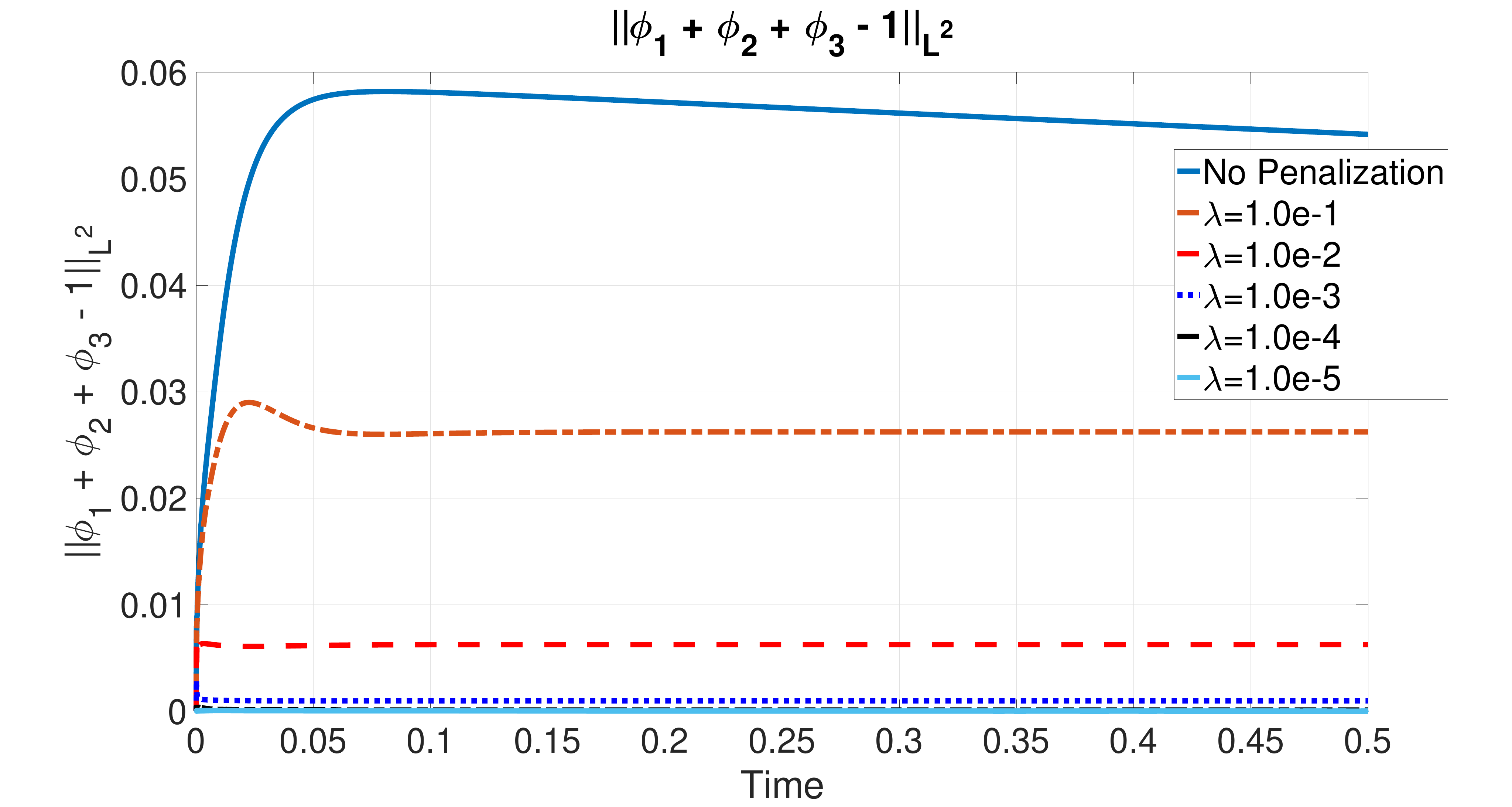}
\end{center}
\caption{Evolution in time of the energies (top left), the volume (top right), $\|\phi_1 + \phi_2 + \phi_3 -1\|_{L^\infty}$ (bottom left), $\|\phi_1 + \phi_2 + \phi_3 -1\|_{L^2}$ (bottom right) for the results presented in Figure~\ref{fig:BallsDynamicsLambda}.}
\label{fig:BallsPlotsLambda}
\end{figure}

%

\subsubsection{Particular case of considering only two components $(\phi_2=0)$}\label{sec:twoballsnofluid}
{
Now we study how consistent is scheme NTD1 with the two components systems when two balls are considered. To this end we modify the initial condition presented in \eqref{eq:twoBubblesInitial} by considering two balls of $\phi_3$ immersed in $\phi_1$. In Figure~\ref{fig:BallsCase0Dyn} we can see how the dynamics are the ones that we expected. But plotting function $\phi_2$ (see Figure~\ref{fig:BallsCase0NTD1phi2}) illustrates that, as in the lens experiments, some spurious creation of phase $\phi_2$  happens in the interface between the two phases. In Figure~\ref{fig:BallsCase0NTD1phi2maxmin} we show the evolution in time of the maximum and minimum of $\phi_2$. Again, it seems that after the first iteration (where the system is adapting to an initial condition that is not exactly a solution of the PDE), the spurious phase $\phi_2$ is disappearing and not affecting at all the dynamics of the system.
}
\begin{figure}[h]
\begin{center}
\includegraphics[scale=0.09]{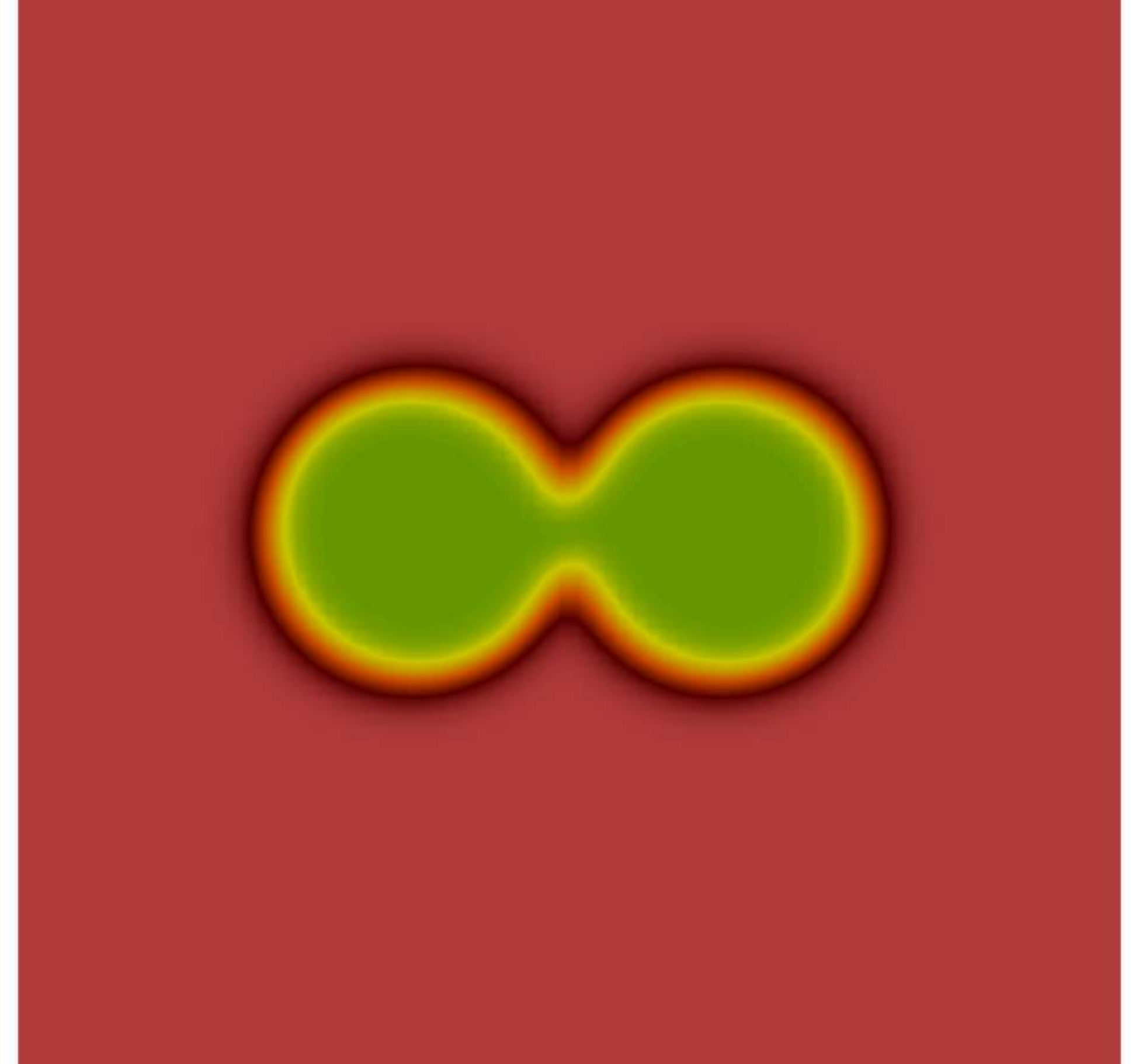}
\includegraphics[scale=0.09]{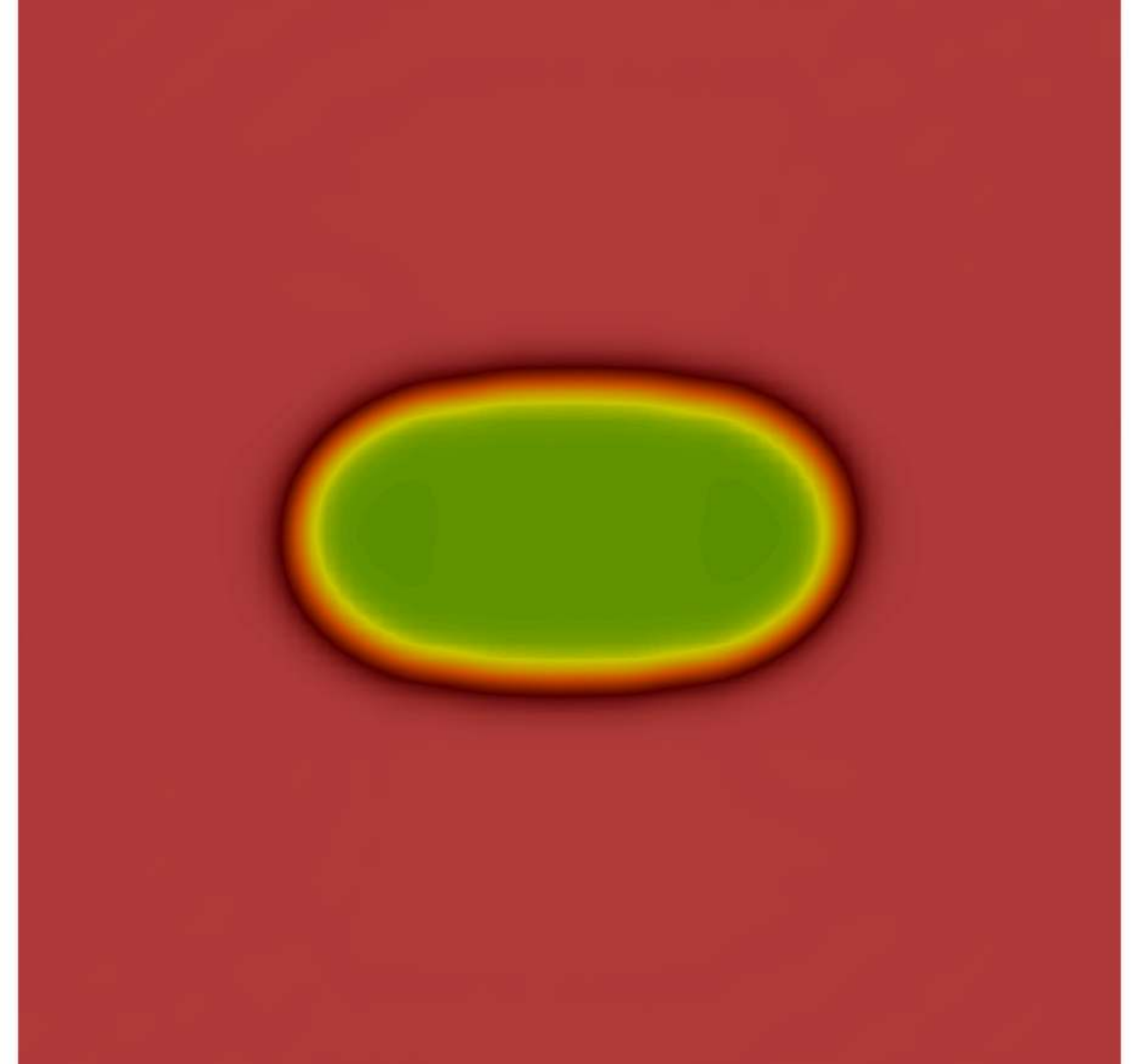}
\includegraphics[scale=0.09]{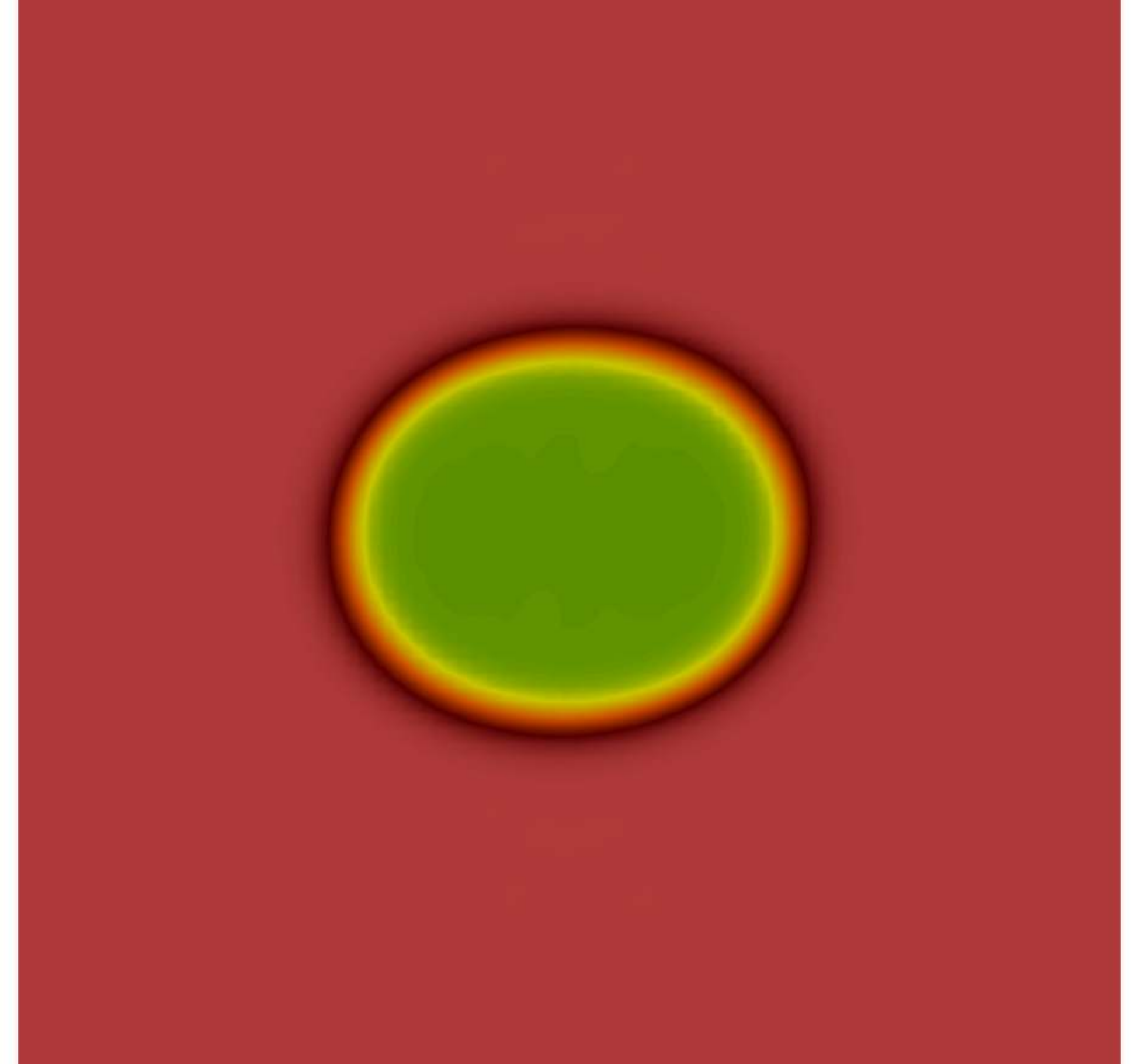}
\includegraphics[scale=0.09]{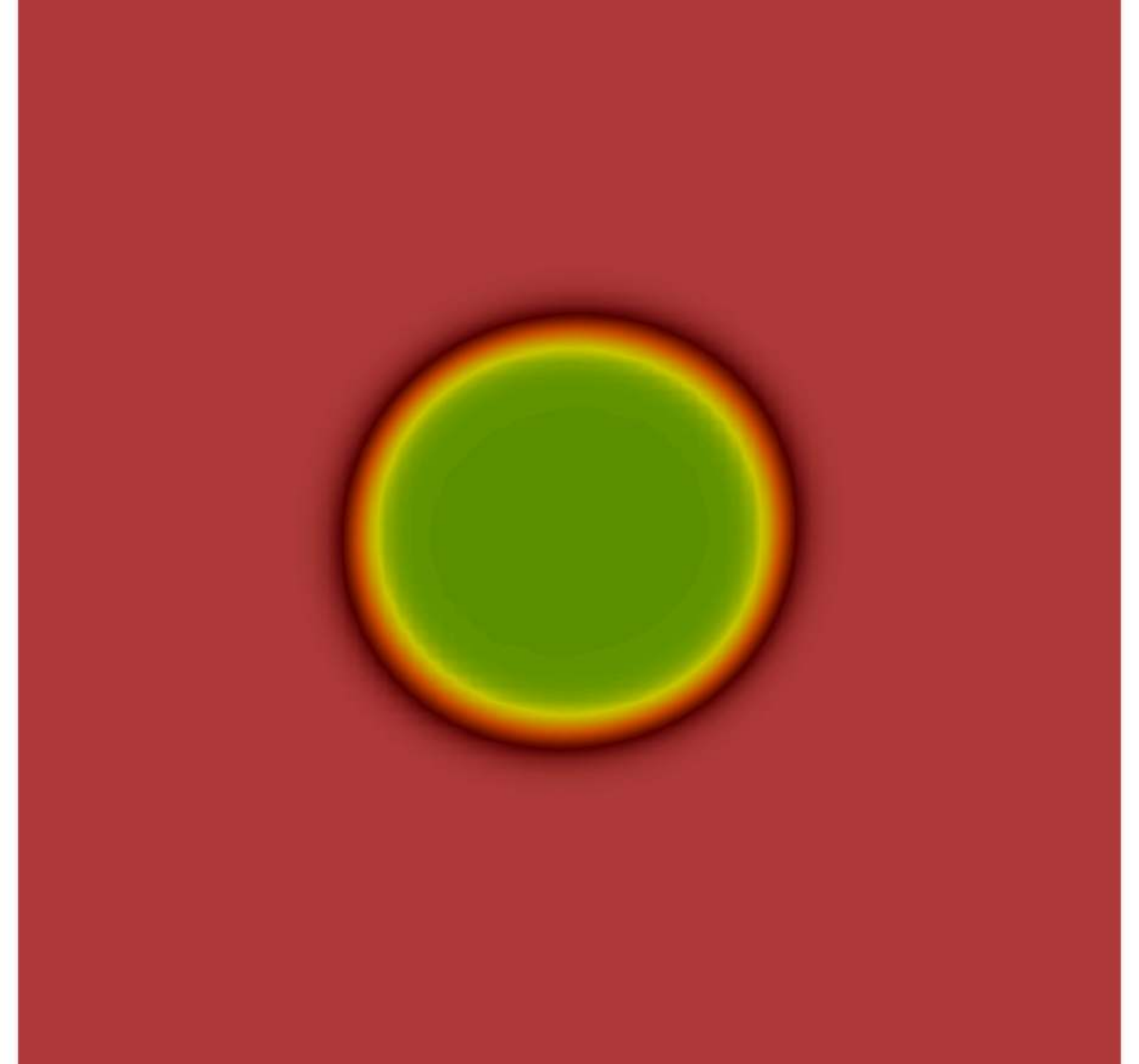}
\includegraphics[scale=0.09]{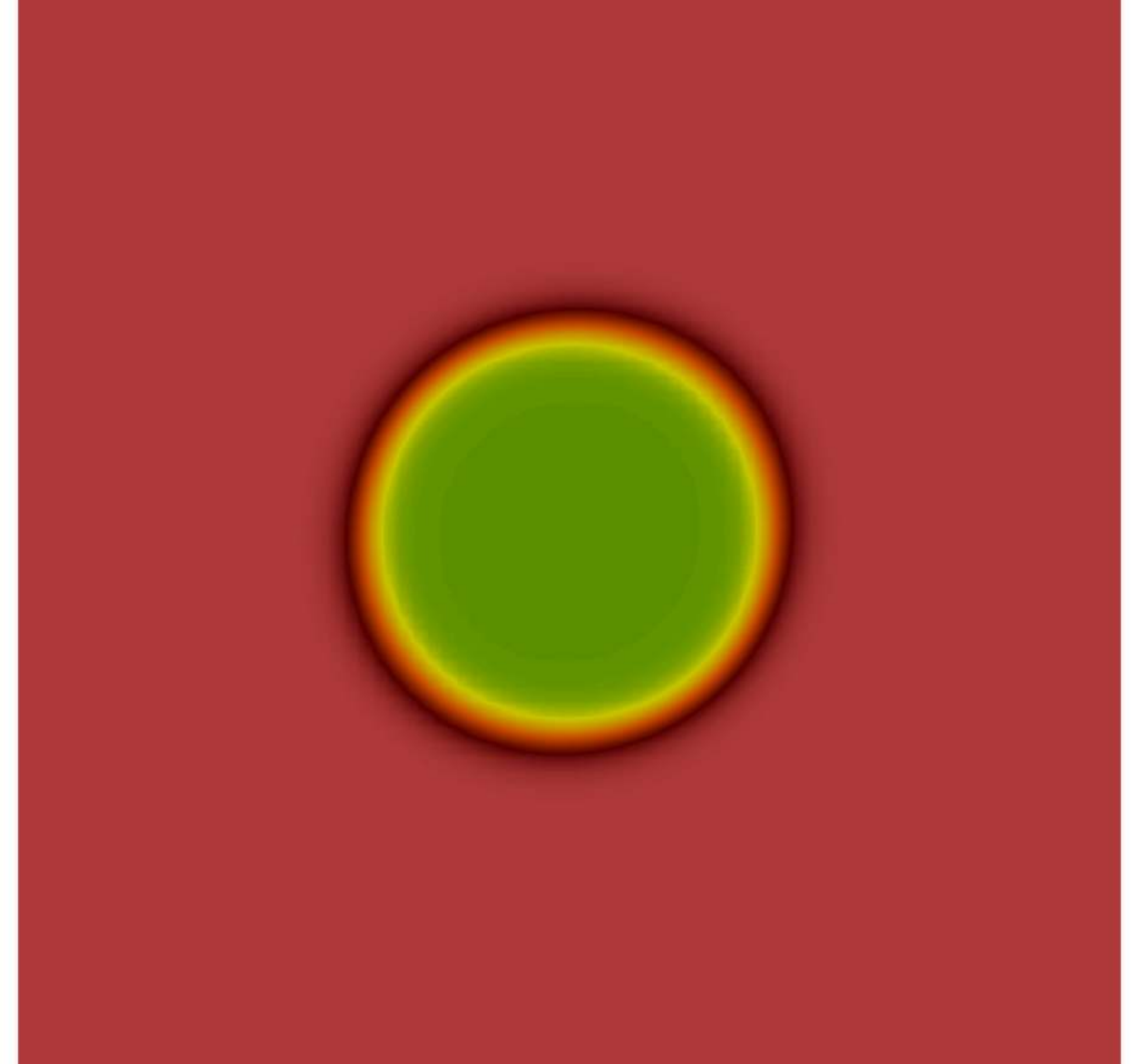}
\end{center}
\caption{Dynamics of scheme NTD1 at times $t=0, 0.01, 0.03, 0.05$ and $0.1$ (from left to right) with spreading coefficients $(\Sigma_1, \Sigma_2 , \Sigma_3) = (1,1,1)$. }
\label{fig:BallsCase0Dyn}
\end{figure}

\begin{figure}[h]
\begin{center}
\includegraphics[scale=0.125]{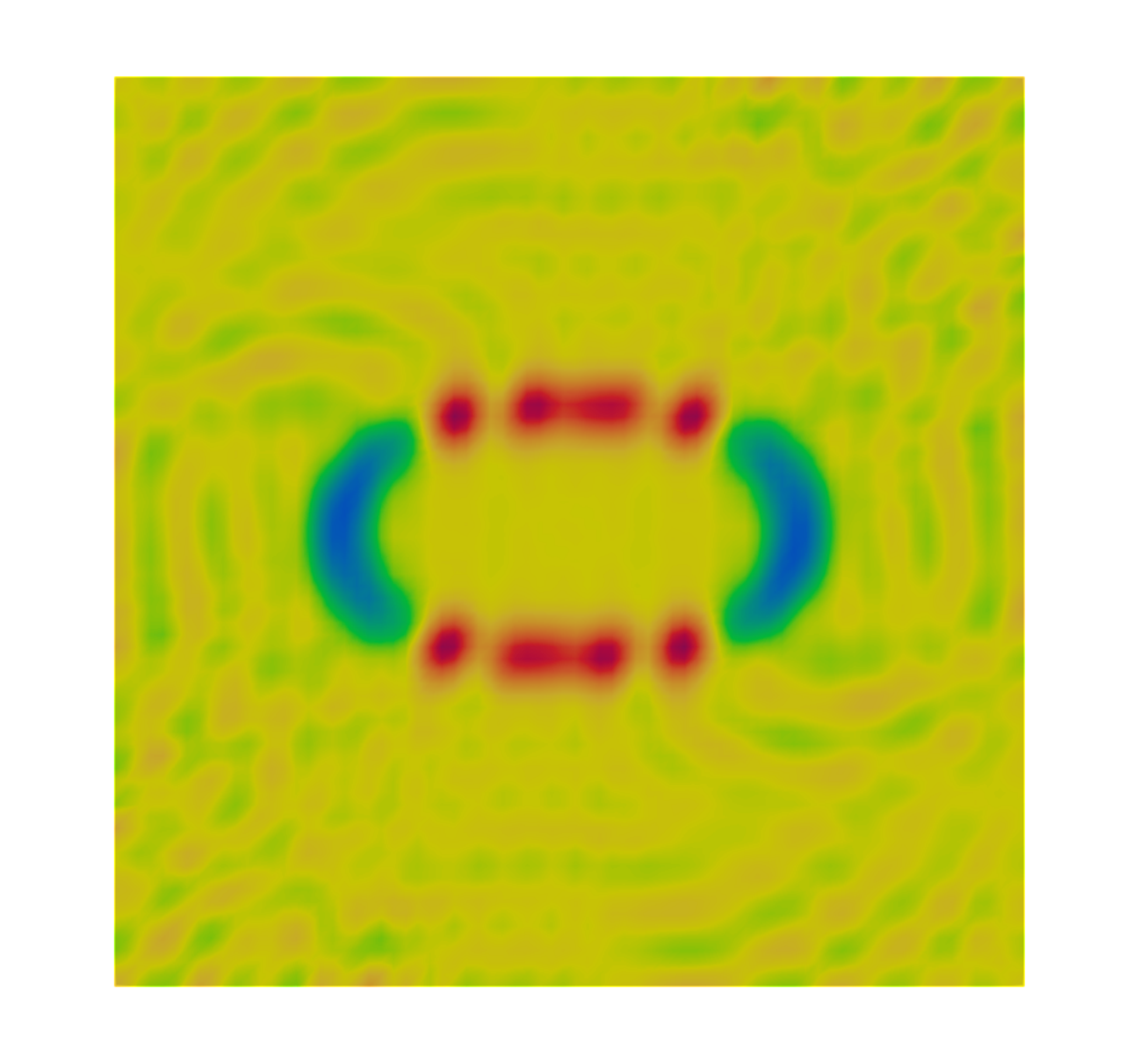}
\includegraphics[scale=0.125]{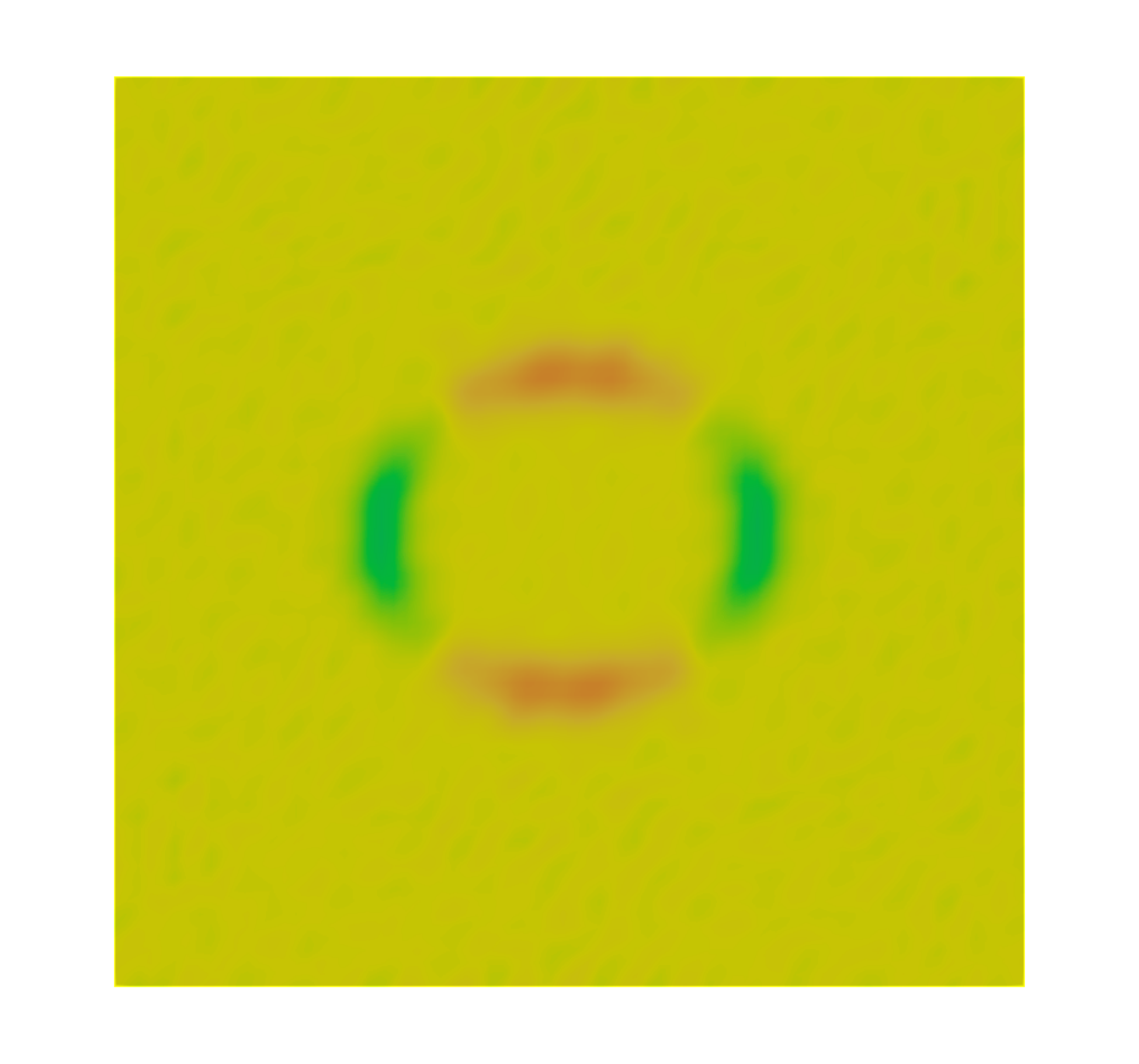}
\includegraphics[scale=0.125]{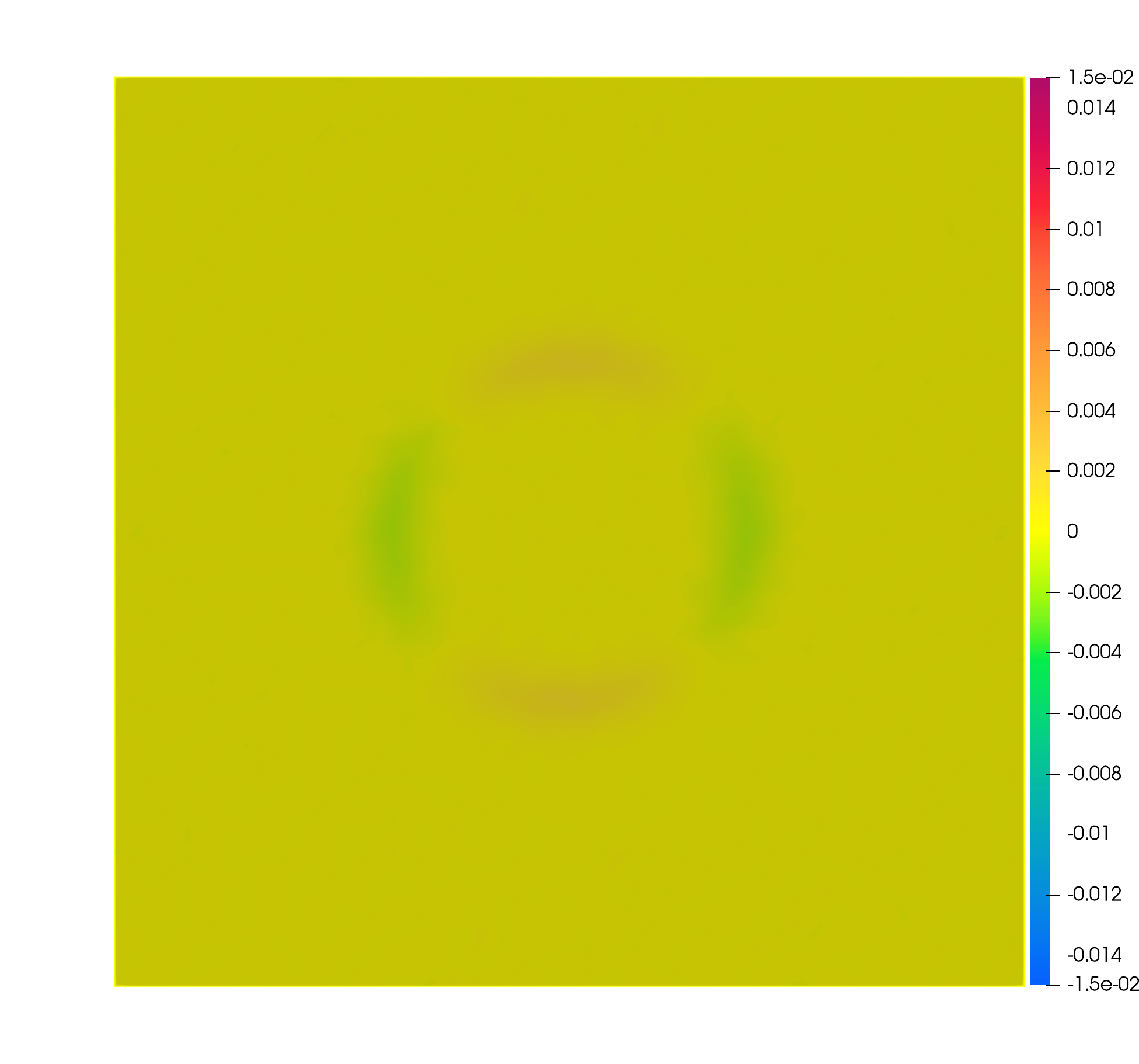}
\end{center}
\caption{Dynamics of $\phi_2$ for scheme NTD1 at times $t=0.01, 0.03$ and $0.05$ (from left to right) with spreading coefficients $(\Sigma_1, \Sigma_2 , \Sigma_3) = (1,1,1)$.}
\label{fig:BallsCase0NTD1phi2}
\end{figure}

\begin{figure}[h]
\begin{center}
\includegraphics[scale=0.11]{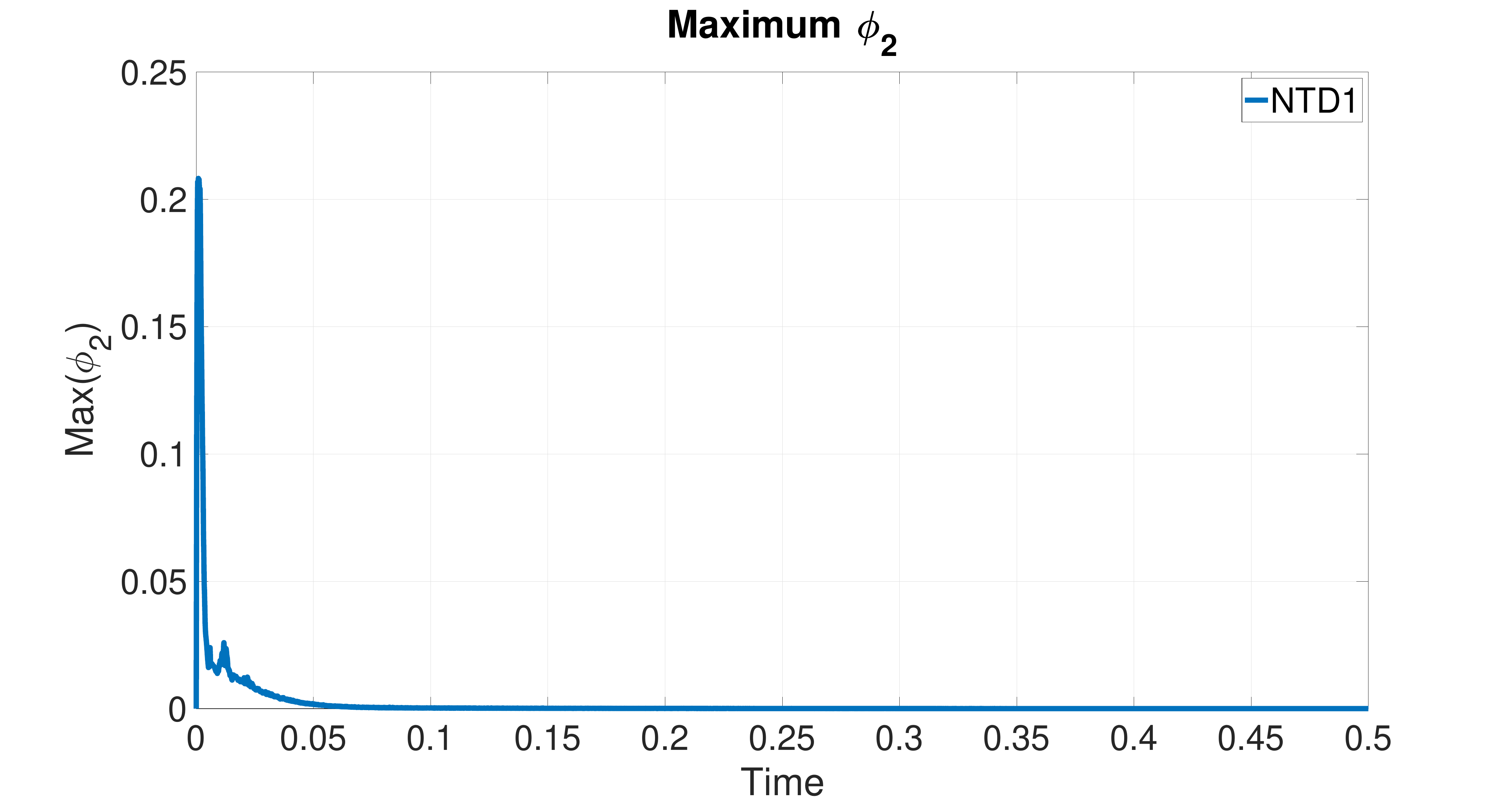}
\includegraphics[scale=0.11]{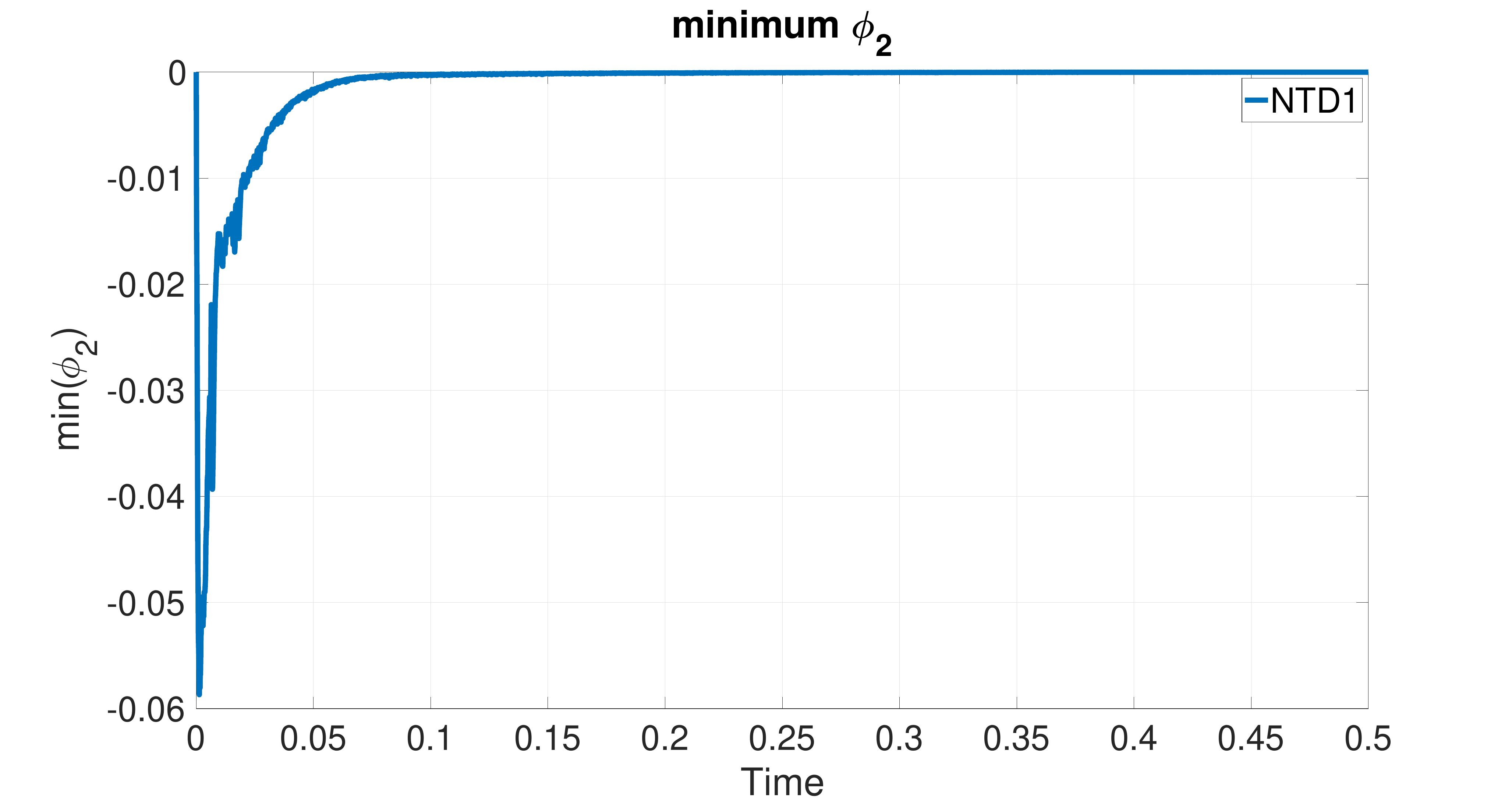}
\end{center}
\caption{Evolution in time of $\max(\phi_2)$ and $\min(\phi_2)$ for NTD1. }
\label{fig:BallsCase0NTD1phi2maxmin}
\end{figure}

\subsection{Spinodal Decomposition}

In this example, we show the dynamics of several simulations representing the spinodal decomposition case in $2D$ using different choices of $\Sigma_i$. The experimental parameters are given in Table~\ref{tab:SpinodalParameters}.
\begin{table}[h]
{
\begin{tabular}{c|c|c|c|c|c|c|c}
$\Omega$ 			& $h$ 		& $[0,T]$ 		& $\dt$ 	& $\eps$ & $\lambda$ 	& $M$ 	& $\Lambda$ \\
\hline
$[-0.125, 0.125]\times[-0.125,0.125]$ 	& 1/300	& $[0,2.5]$ 	& 1e-4 	& 1e-2  & 1e-4 	&  1e-3 	& 7 \\
\end{tabular}
}
\caption{\label{tab:SpinodalParameters} Parameters for the stratified lens experiments.}
\end{table}
The initial condition reads 
\beq\label{eq:spinodalInitial}
\left\{
\ba{rcl}
\phi_1(x,y) &=& 0.33 +  0.01 \mbox{rand}(x,y)\,,
\\ \hueco
\phi_2(x,y) &=& 0.33 +  0.01 \mbox{rand}(x,y)\,,
\\ \hueco
\phi_3(x,y) &=& 1 - \phi_1(x,y) - \phi_2(x,y)\,,
\ea
\right.
\eeq
where $\mbox{rand}(x,y)$ is randomly sampled from a uniform distribution on $[0,1]$.
{
The dynamics of the four cases are presented in Figure~\ref{fig:SpinodalDynamics}, and in all cases the obtained dynamics are what one would expect in this type of system.  
The top row presents the case $(\Sigma_1, \Sigma_2 , \Sigma_3) = (1,1,1)$ and we observe how again the equal values of the parameters $\Sigma_i$ produces that the boundaries between the phases tend to be flat. In the second row, the choice of  unequal surface tensions $(\Sigma_1, \Sigma_2 , \Sigma_3) = (0.4, 1.6, 1.2)$ is presented,  where the system try to develop asymmetric interfaces between the phases. The total spreading case with $(\Sigma_1, \Sigma_2 , \Sigma_3) = (3,3,-0.1)$ is presented in the third row where the negativity of $\Sigma_3$ induces the green phase ($\phi_3$) to be between the other two phases. Finally, another total spreading case is presented in the bottom row under the choice $(\Sigma_1, \Sigma_2 , \Sigma_3) = (-0.1,3,3)$, where the fact that $\Sigma_1<0$ make the red component ($\phi_1$) to spread between the other two phases.
\\
The evolution in time of the energy, the volume and the $L^2$ and $L^\infty$ norms of the restriction $\Sigma_{i=1}^3\phi_i - 1$ are presented in Figure~\ref{fig:SpinodalPlots} . In all the cases considered the energy decreases until the system reaches an almost equilibrium state and as expected the volume is conserved in all the simulations. Compared with previous examples, the $L^2$ and $L^\infty$ norms of the restriction are not as well approximated (this is a much challenging benchmark). To illustrate the validity of the scheme we have compared in Figure~\ref{fig:SpinodalPlotsdt} the results when the time step is lowered (and the time interval is only $[0,0.1]$ to save computational time) and we have seen how in this challenging situation the approximation of the constraint clearly improves when the time step is reduced. 
}

\begin{figure}[h]
\begin{center}
\includegraphics[scale=0.09]{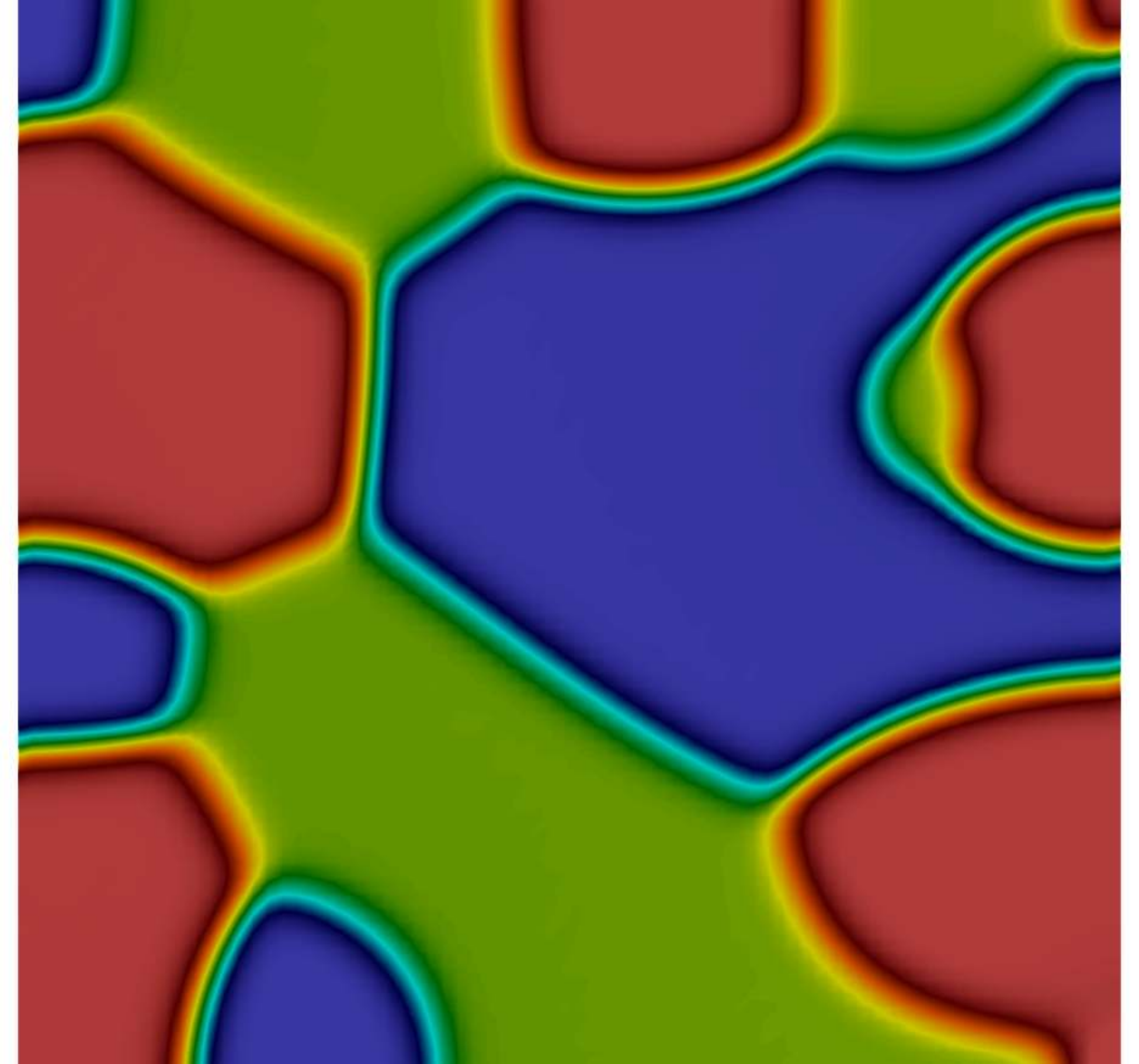}
\includegraphics[scale=0.09]{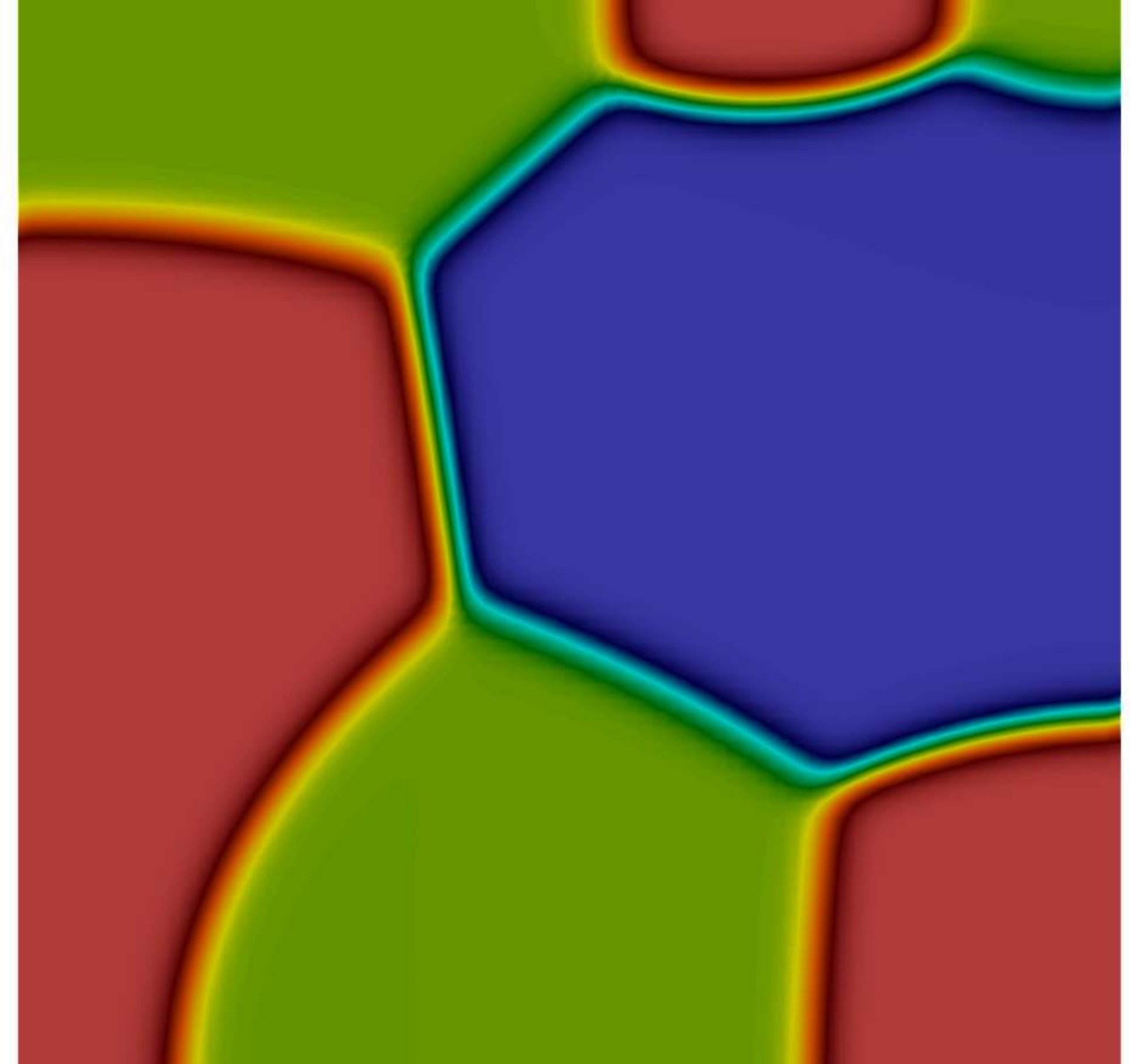}
\includegraphics[scale=0.09]{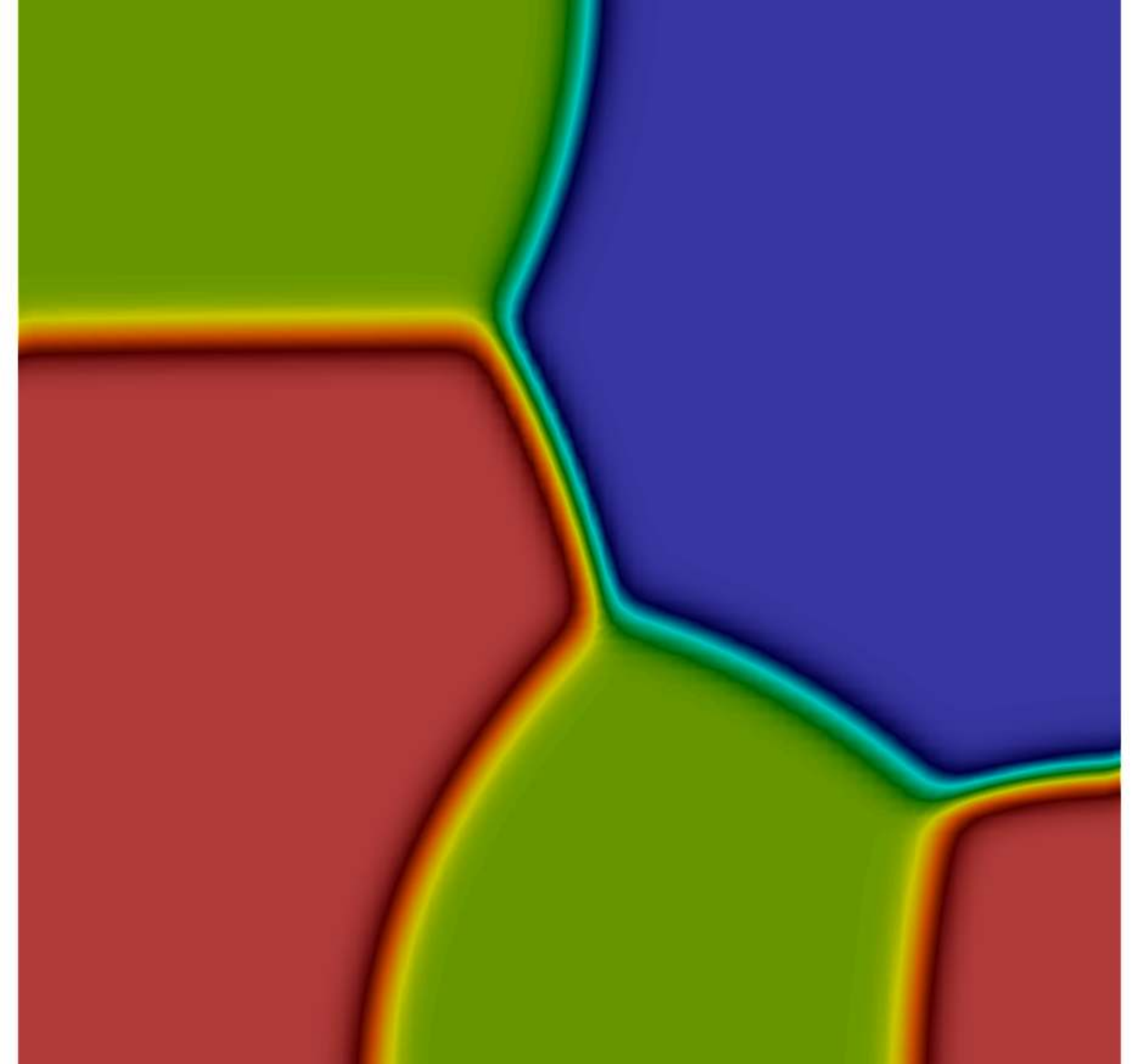}
\includegraphics[scale=0.09]{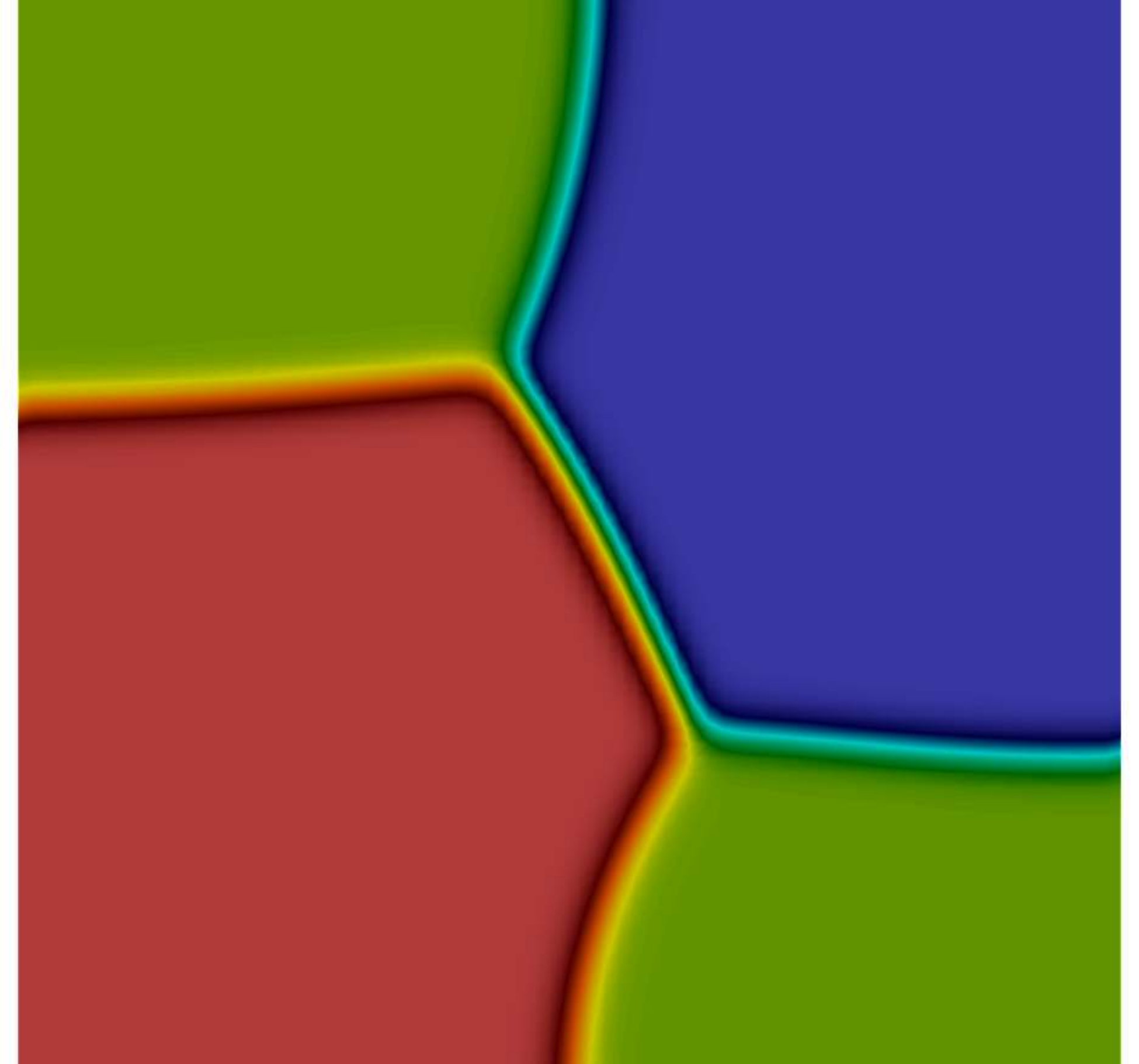}
\includegraphics[scale=0.09]{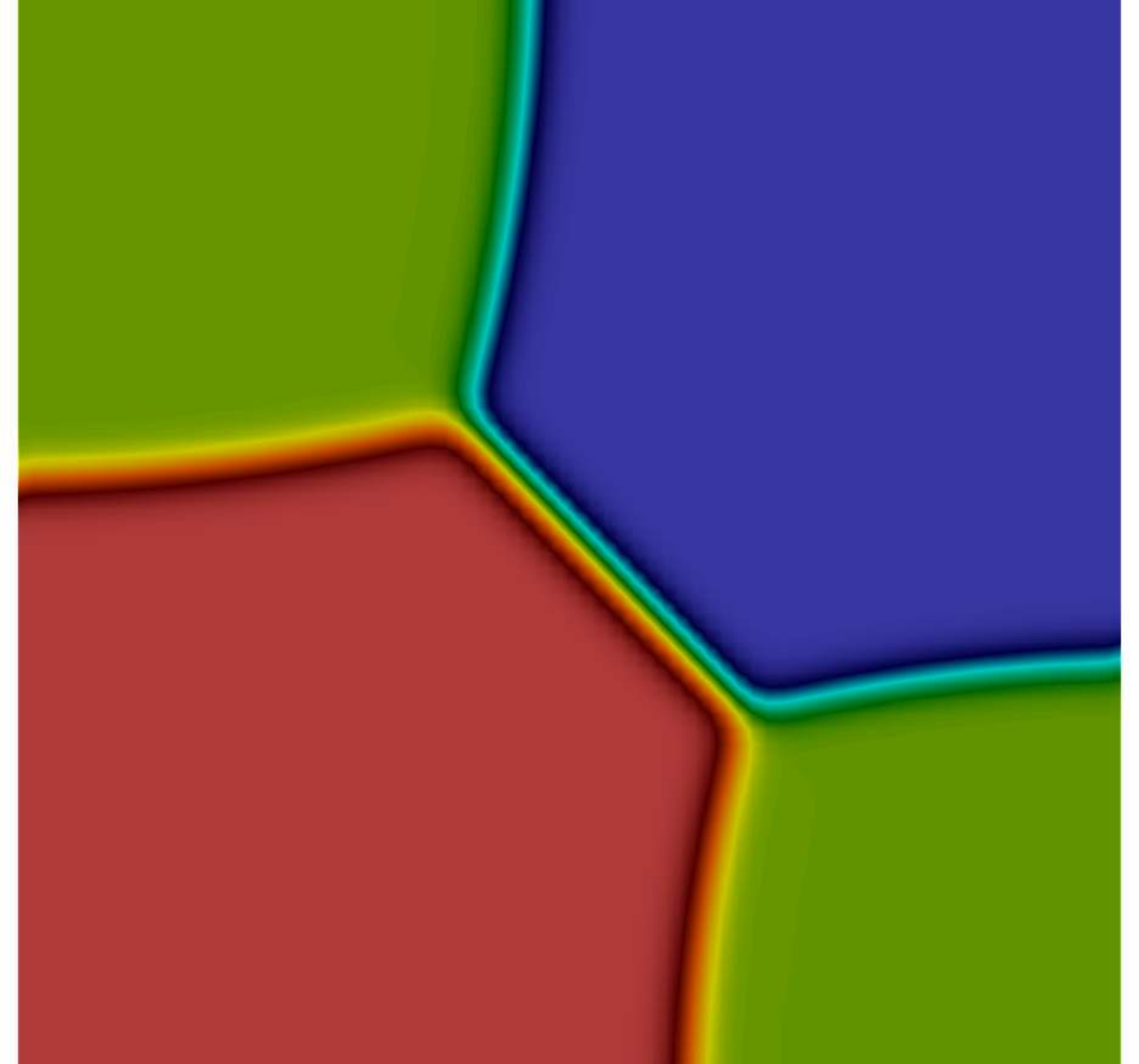}
\\ [1ex]
\includegraphics[scale=0.09]{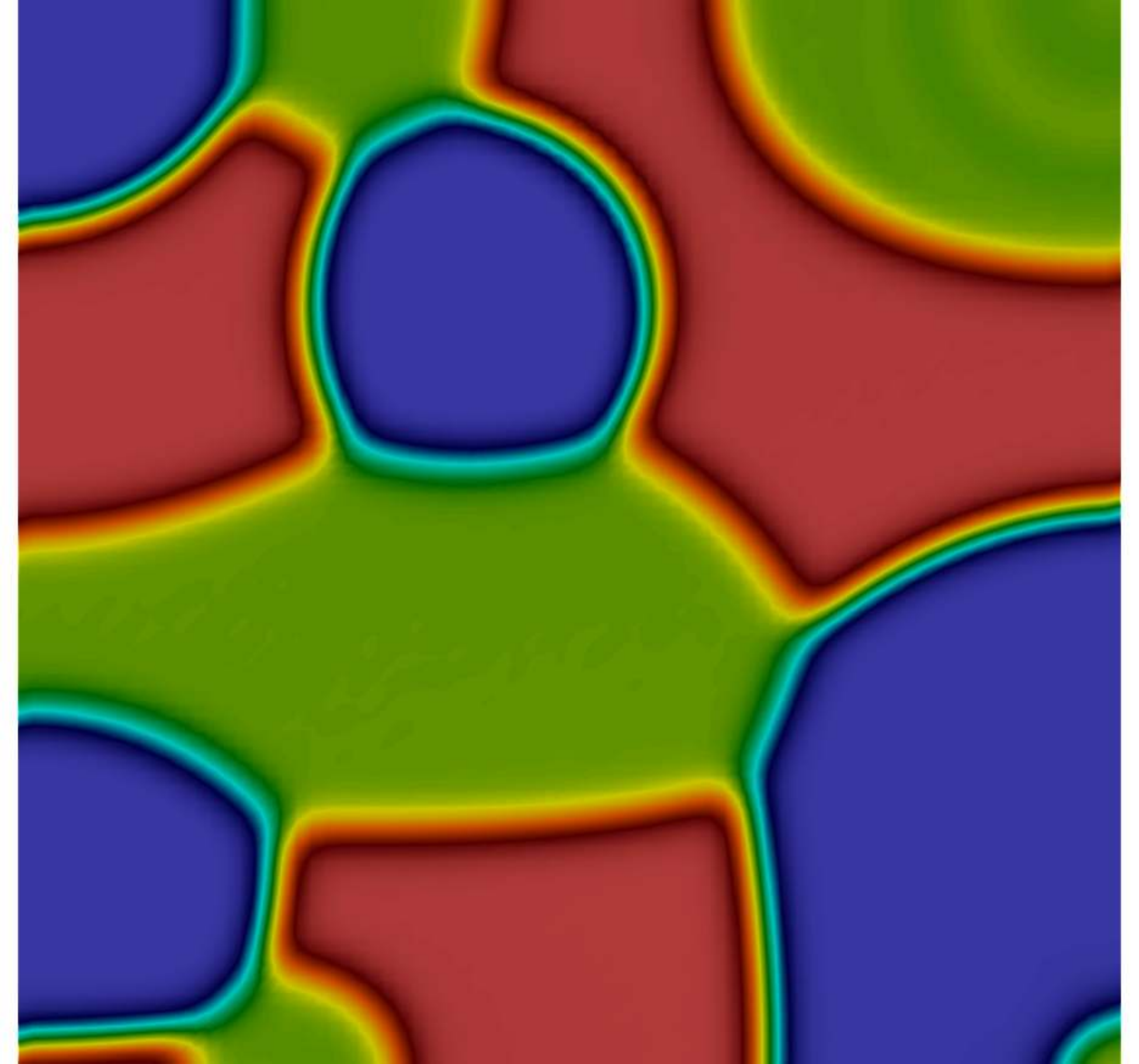}
\includegraphics[scale=0.09]{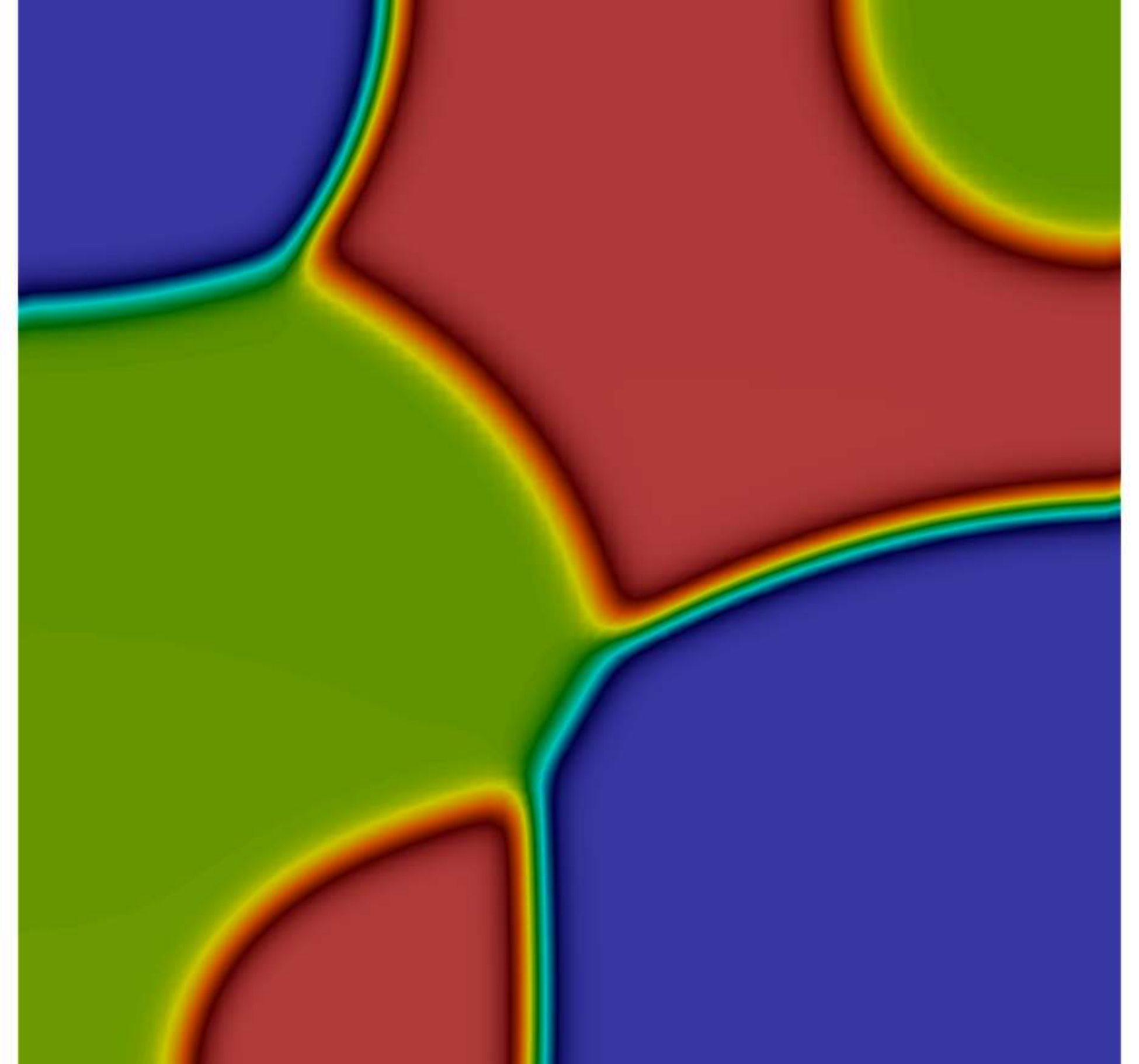}
\includegraphics[scale=0.09]{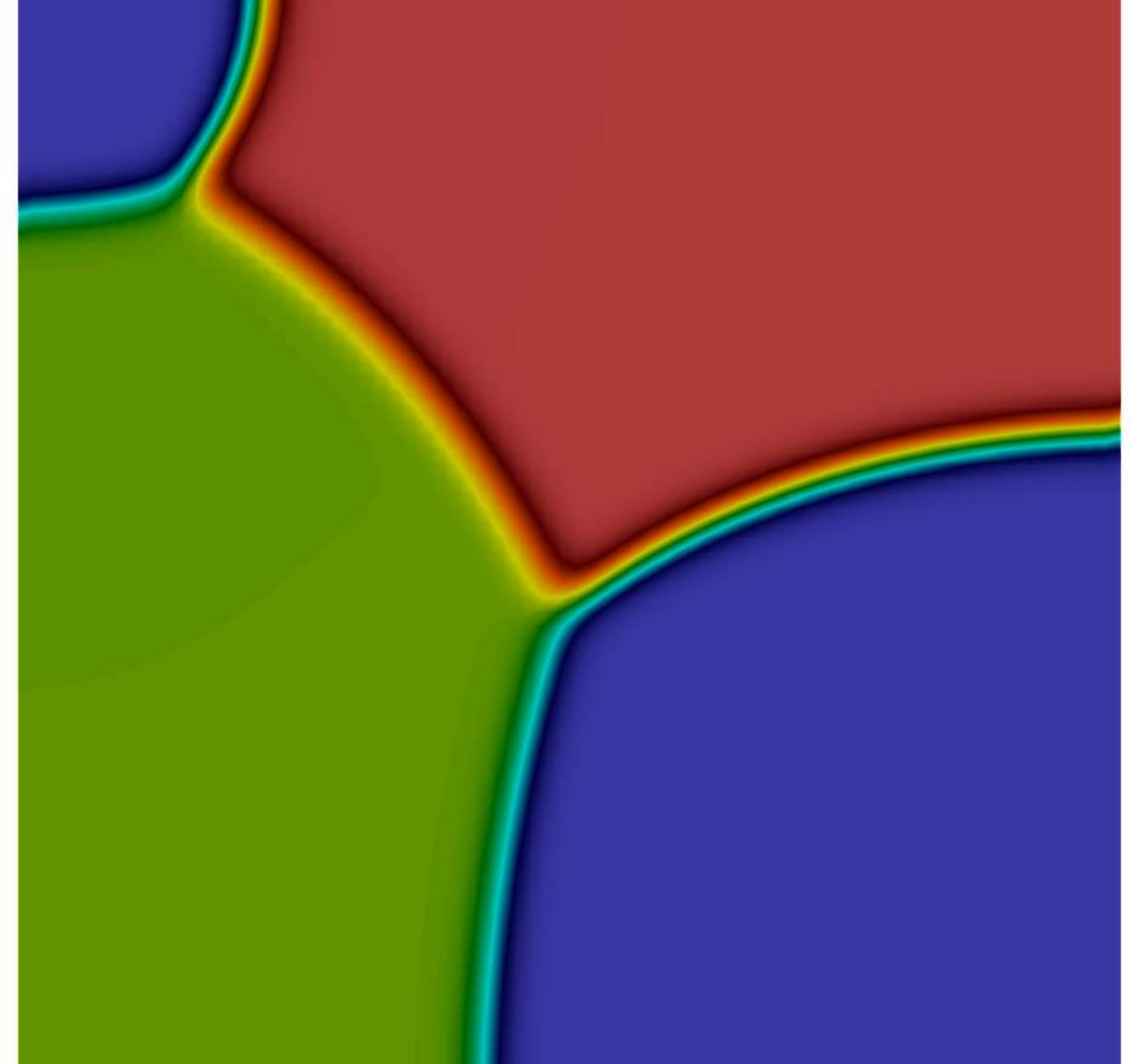}
\includegraphics[scale=0.09]{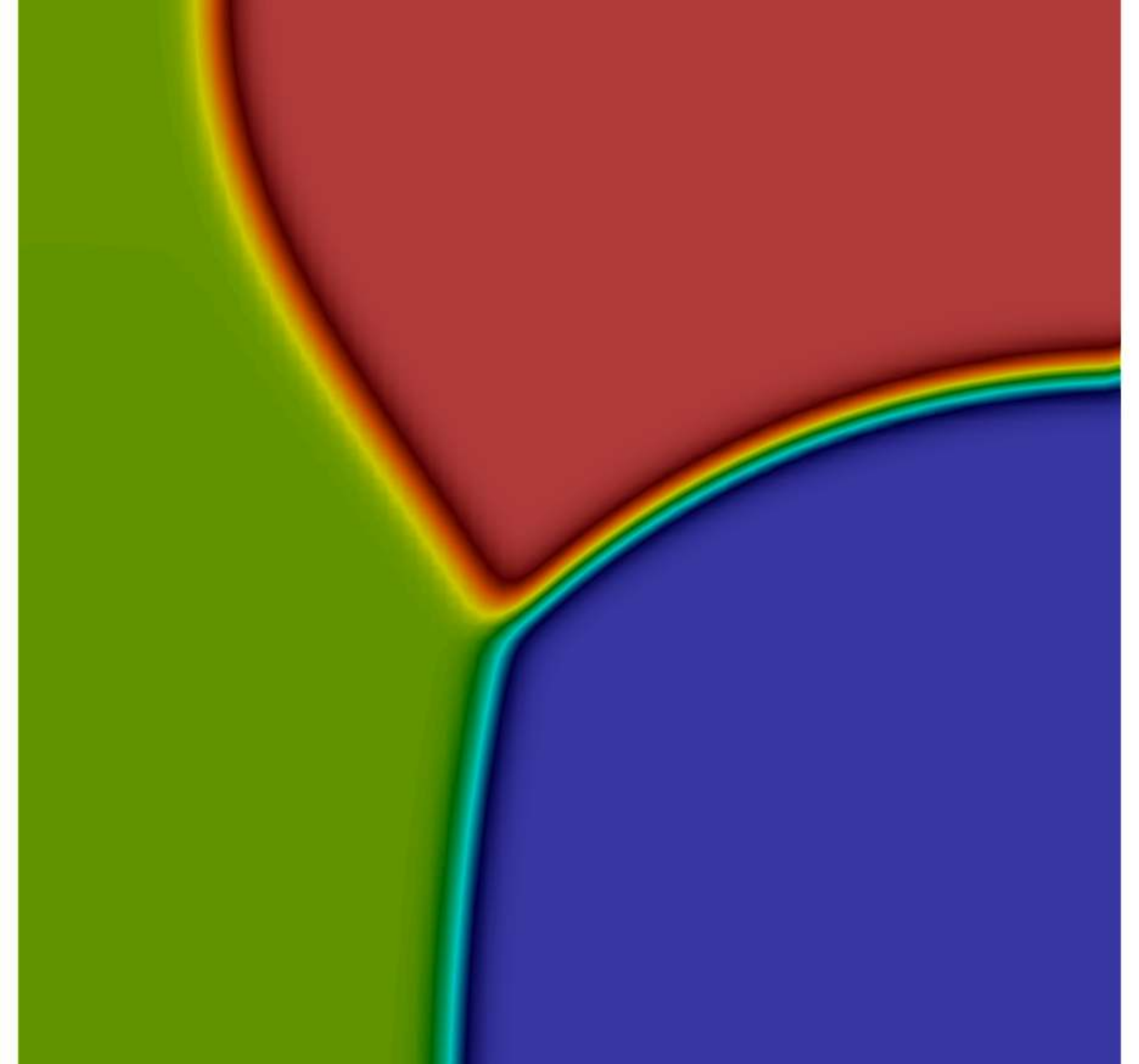}
\includegraphics[scale=0.09]{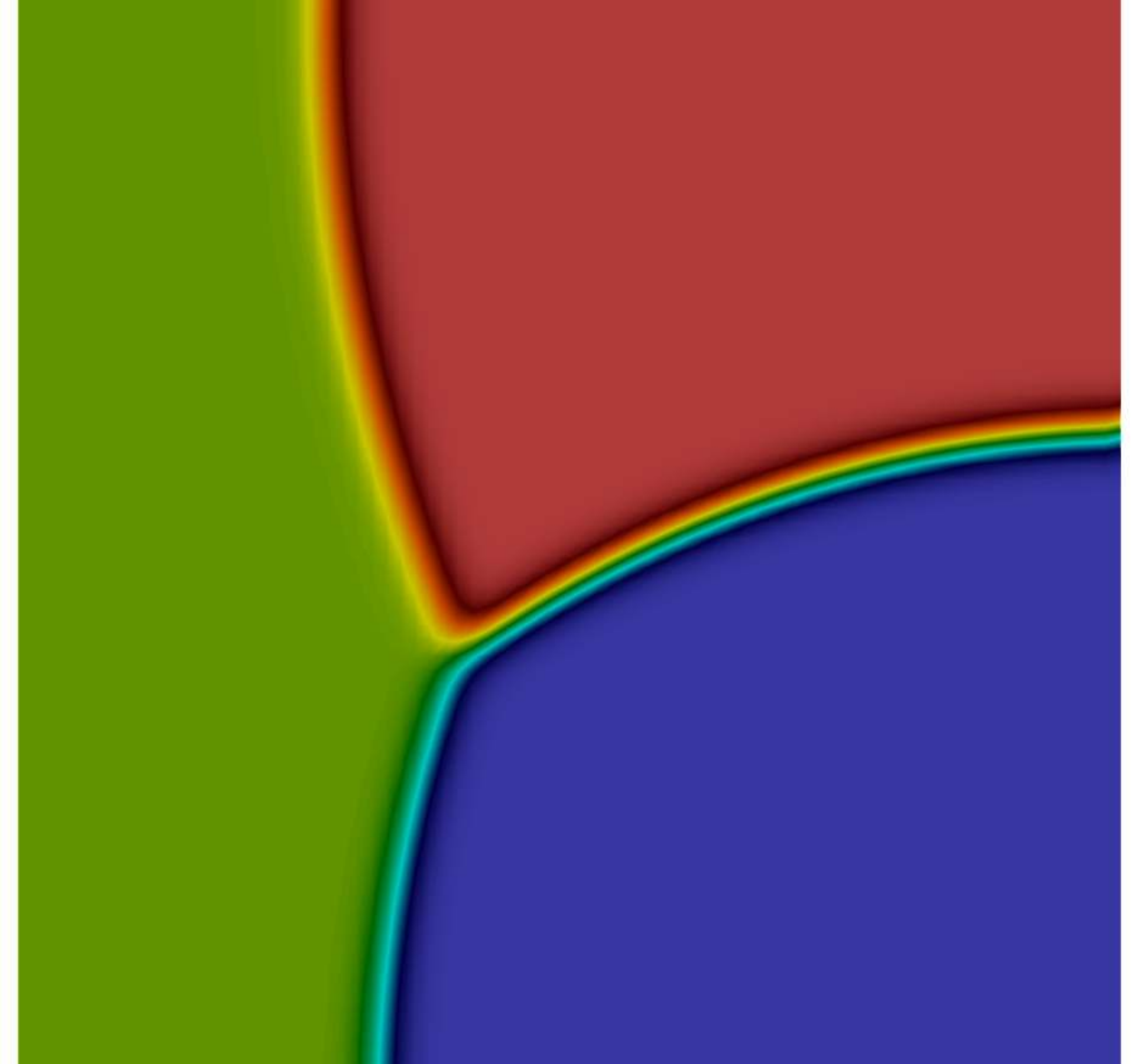}
\\ [1ex]
\includegraphics[scale=0.09]{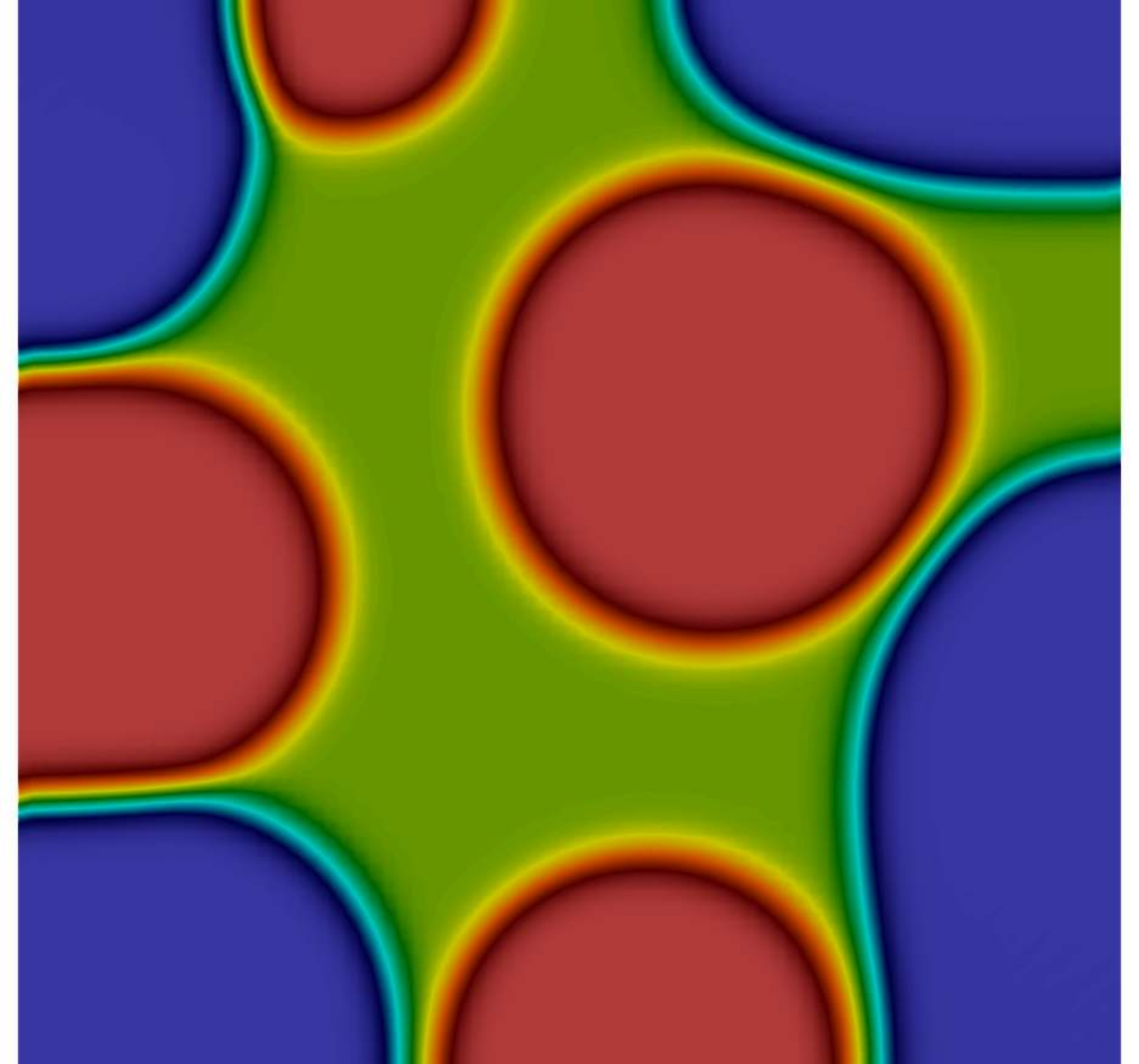}
\includegraphics[scale=0.09]{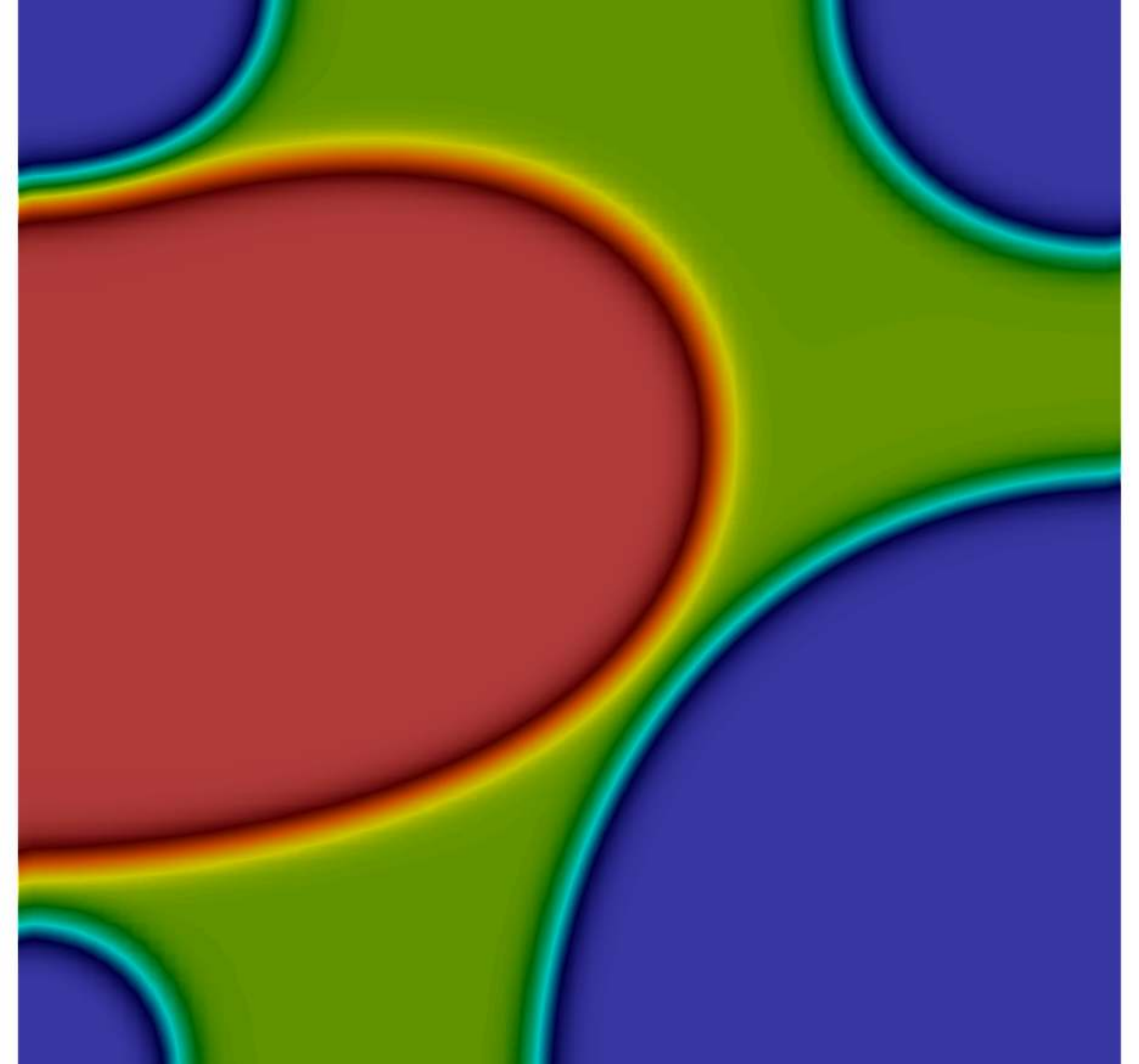}
\includegraphics[scale=0.09]{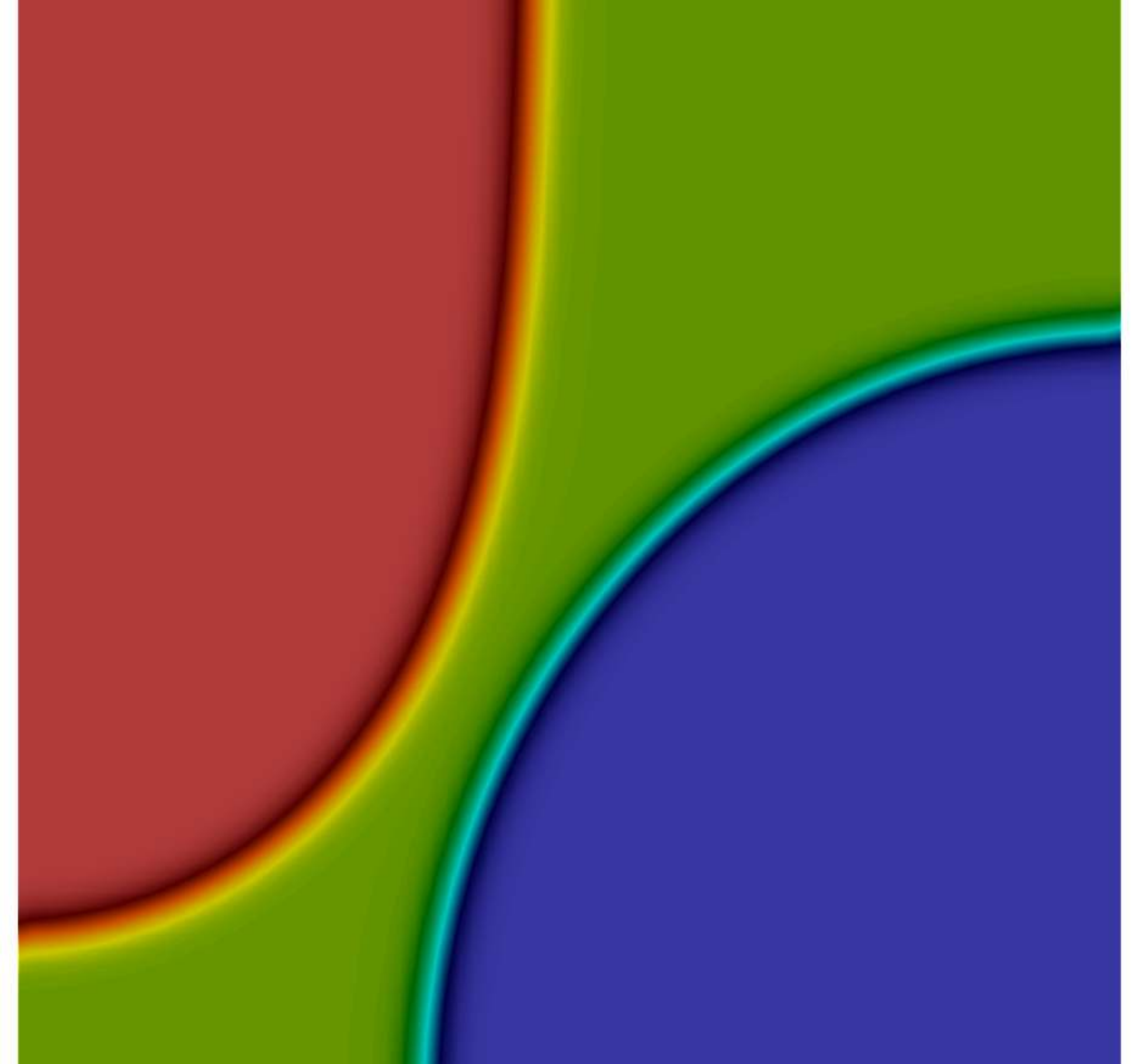}
\includegraphics[scale=0.09]{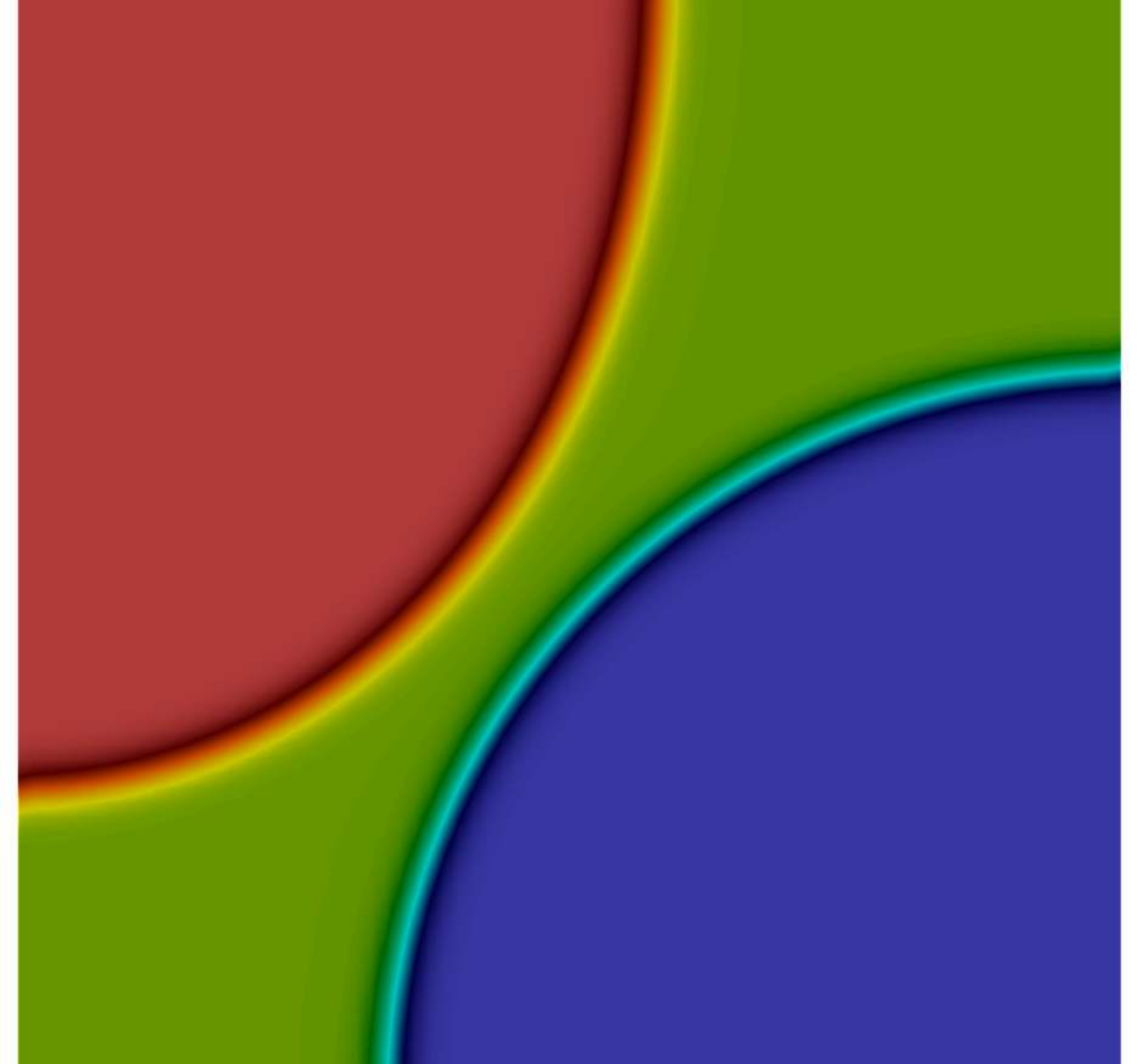}
\includegraphics[scale=0.09]{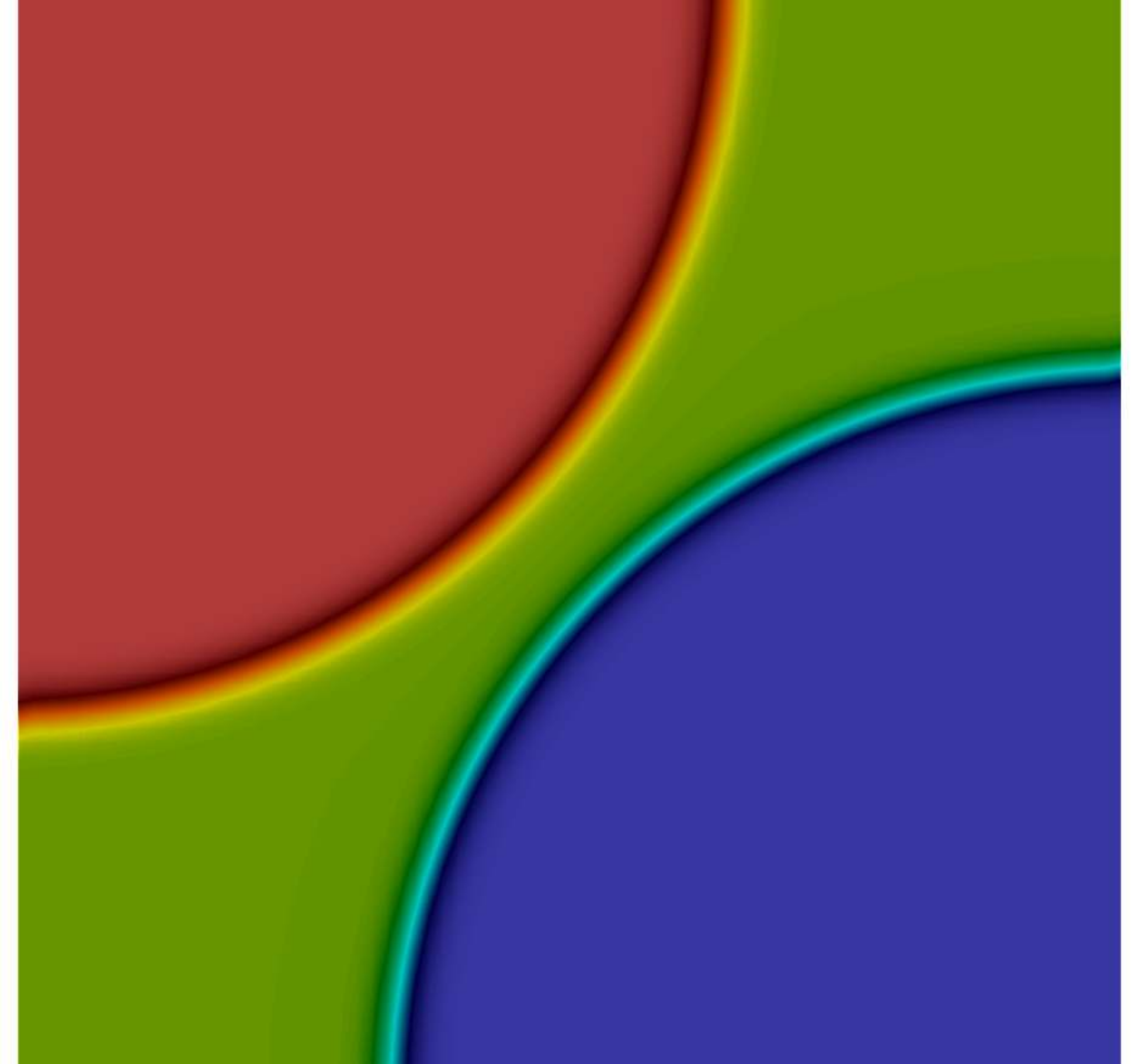}
\\ [1ex]
\includegraphics[scale=0.09]{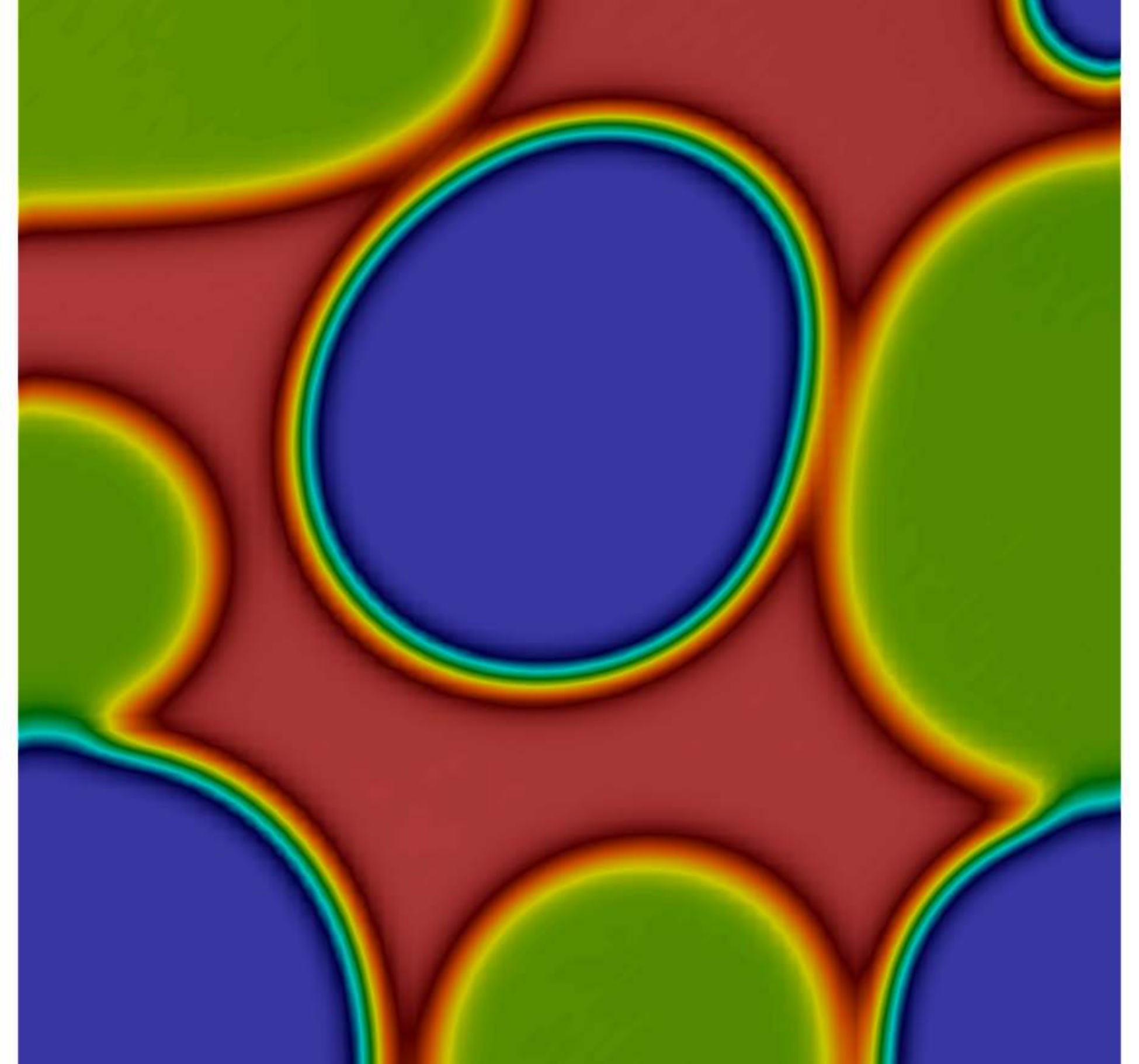}
\includegraphics[scale=0.09]{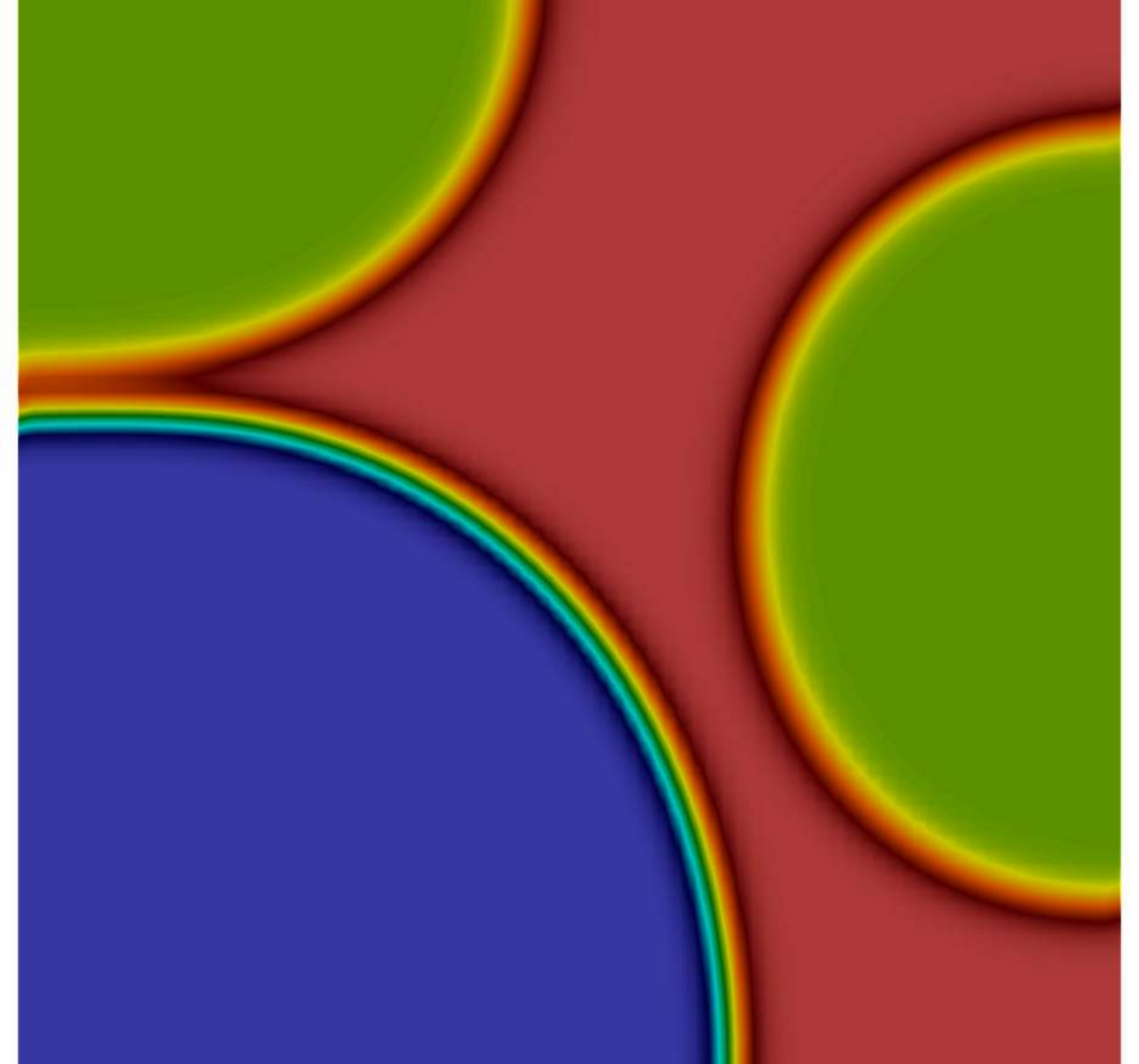}
\includegraphics[scale=0.09]{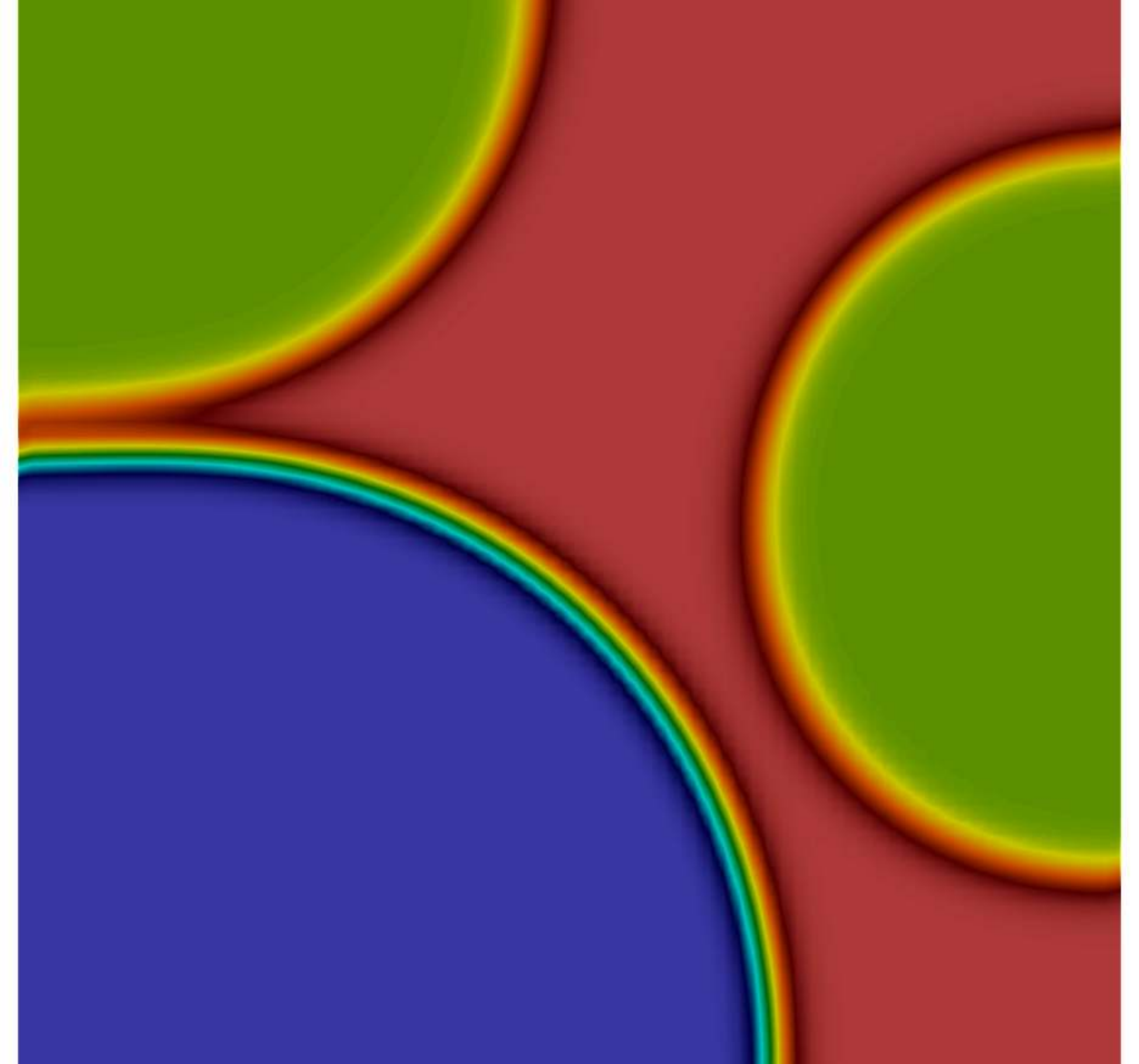}
\includegraphics[scale=0.09]{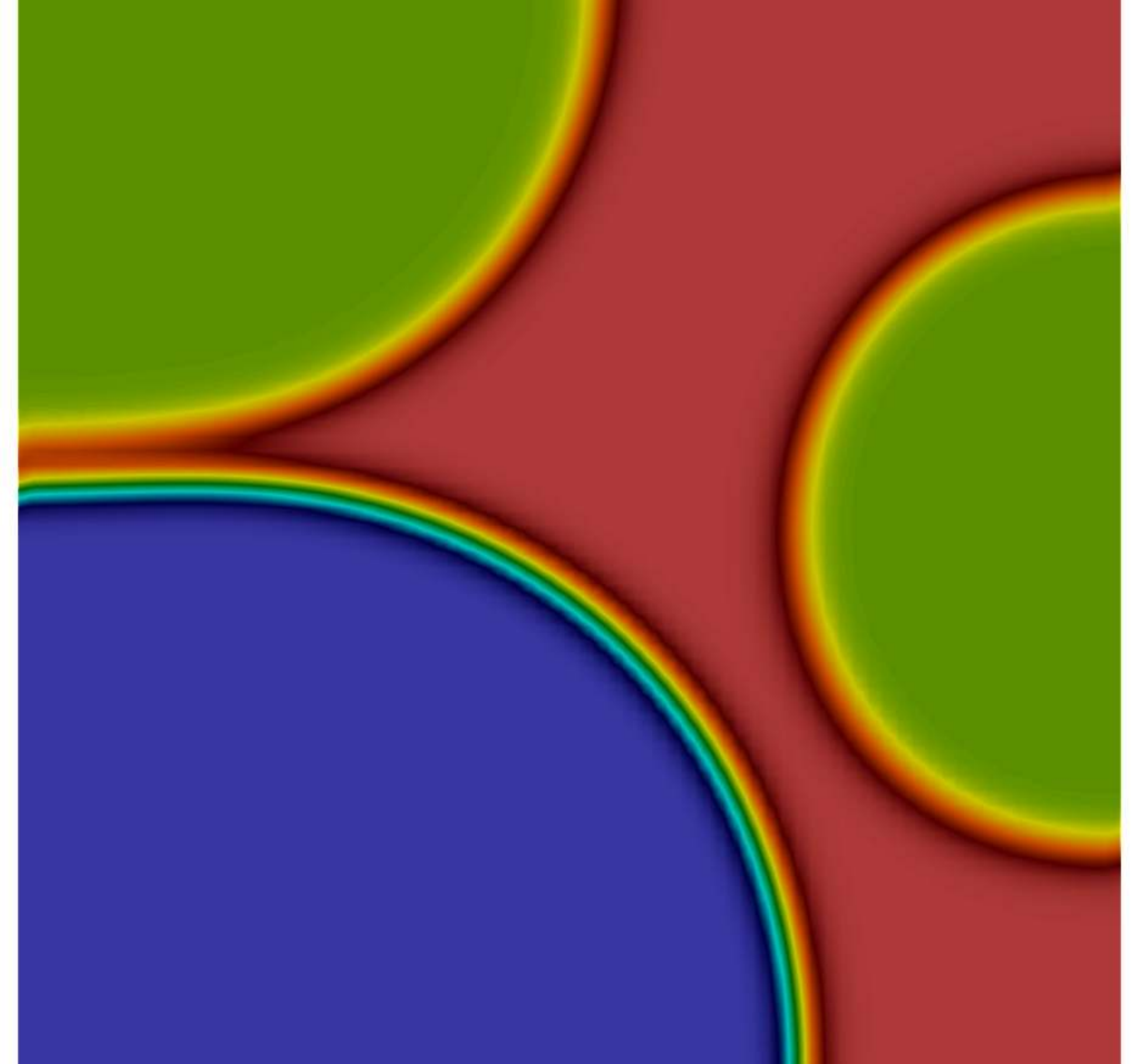}
\includegraphics[scale=0.09]{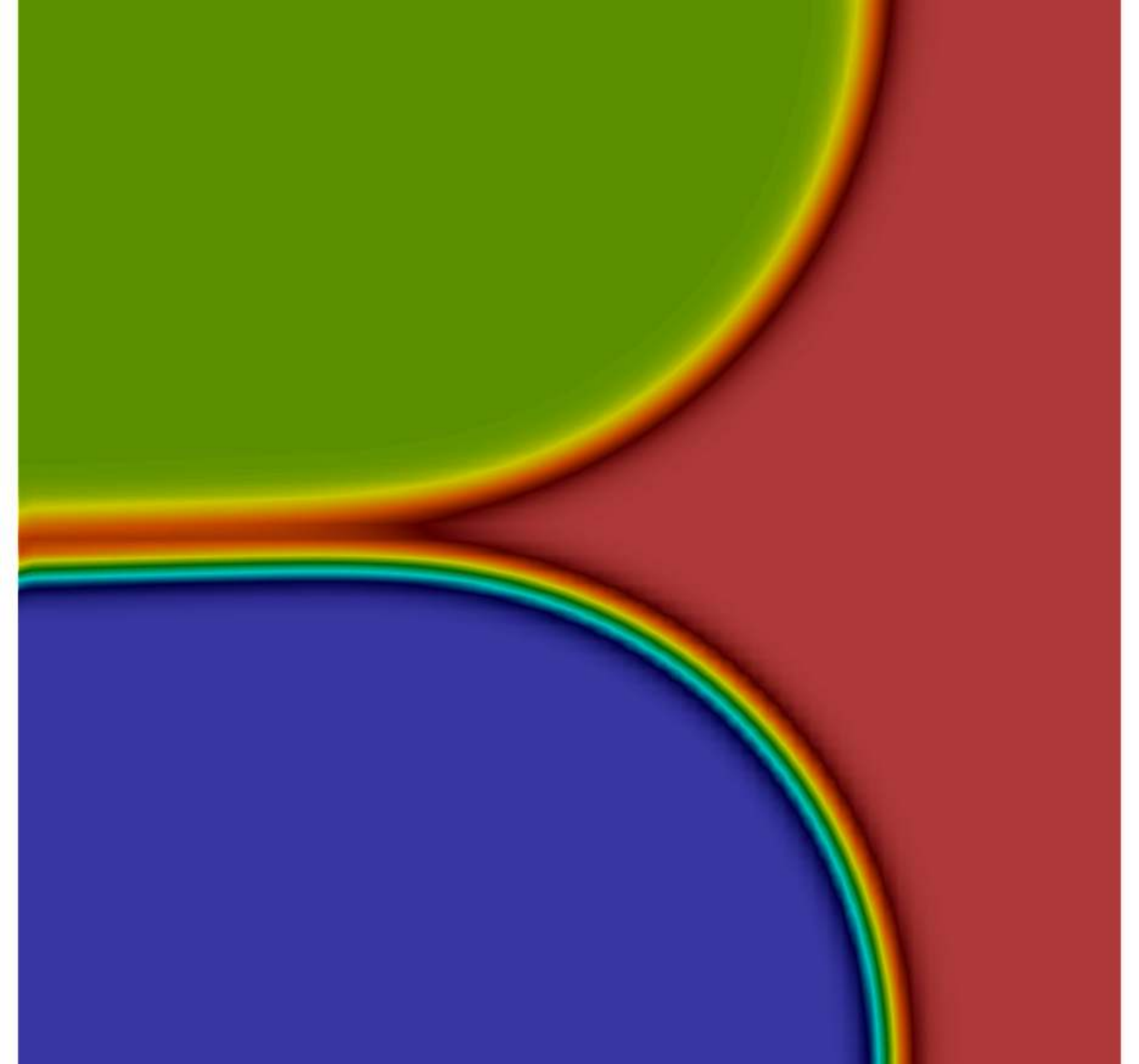}
\end{center}
\caption{Dynamics of scheme NTD1 at times $t=0.1, 0.5, 1, 1.5$ and $2.5$ (from left to right) with spreading coefficients 
$(\Sigma_1, \Sigma_2 , \Sigma_3) = (1,1,1)$ (top row)
$(\Sigma_1, \Sigma_2 , \Sigma_3) = (0.4, 1.6, 1.2)$ (second row)
$(\Sigma_1, \Sigma_2 , \Sigma_3) = (3,3,-0.1)$ (third row)
$(\Sigma_1, \Sigma_2 , \Sigma_3) = (-0.1,3,3)$ (bottom row).}\label{fig:SpinodalDynamics}
\end{figure}

\begin{figure}[h]
\begin{center}
\includegraphics[scale=0.11]{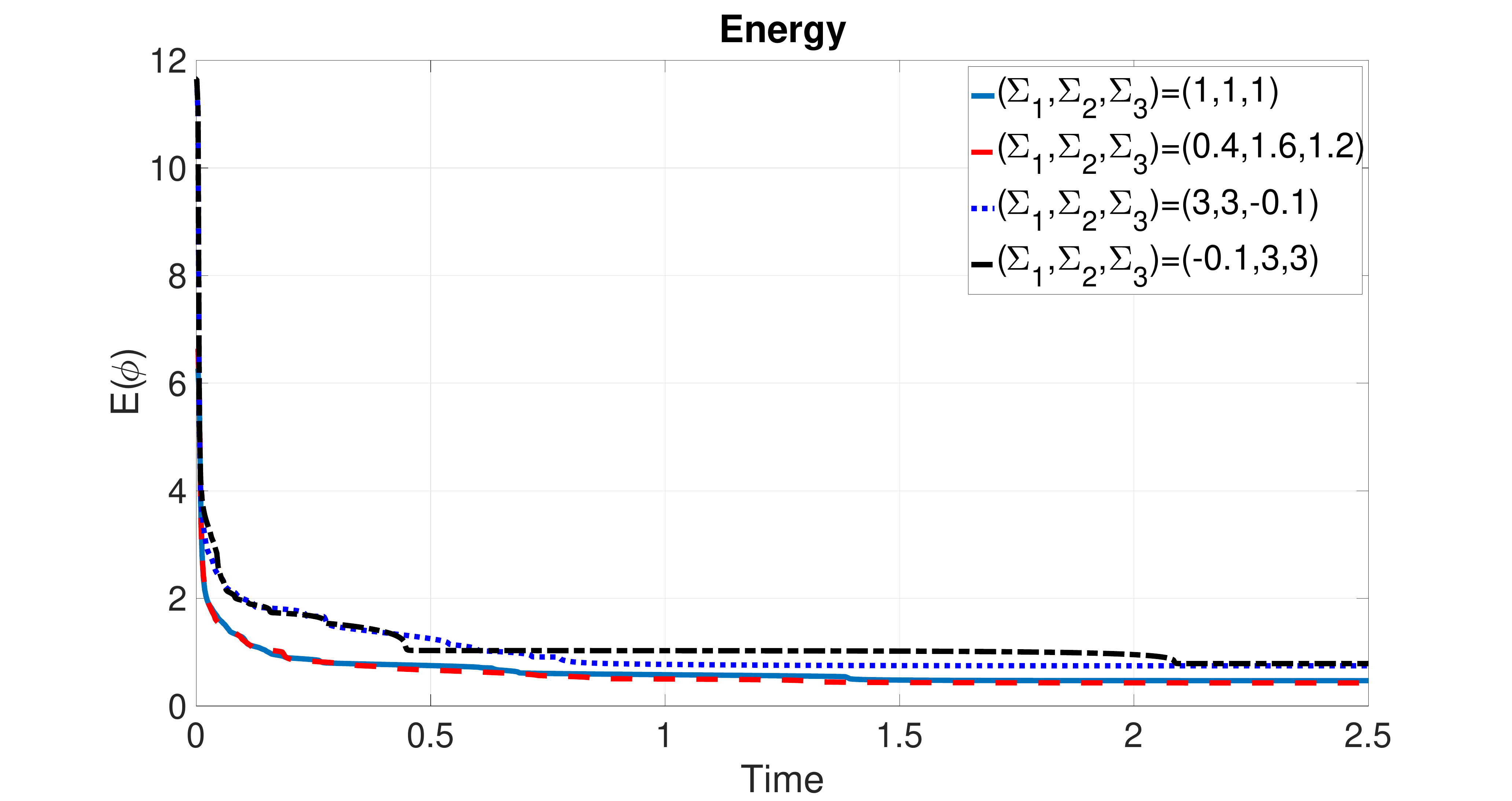}
\includegraphics[scale=0.11]{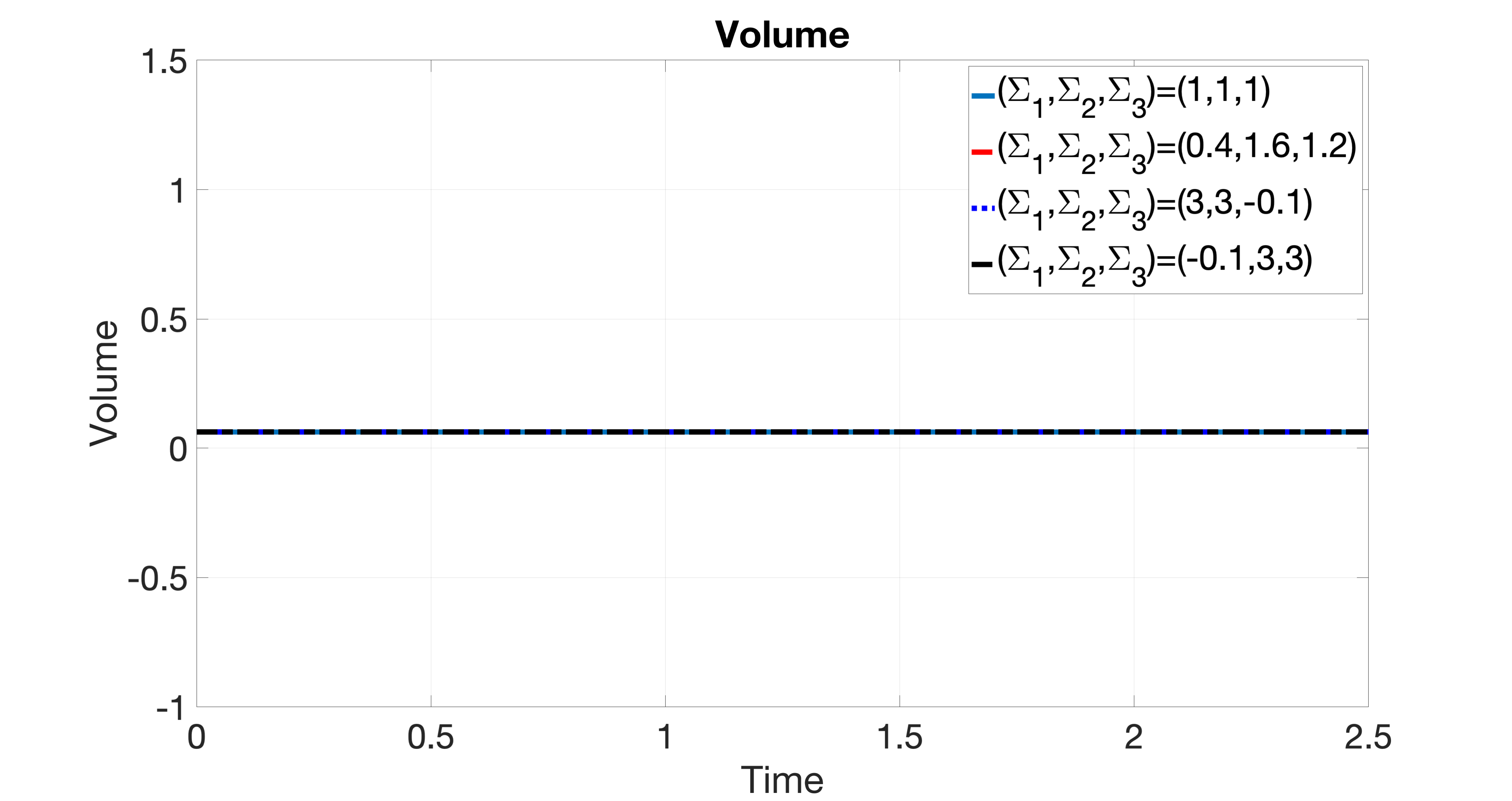}
\\ [1ex]
\includegraphics[scale=0.11]{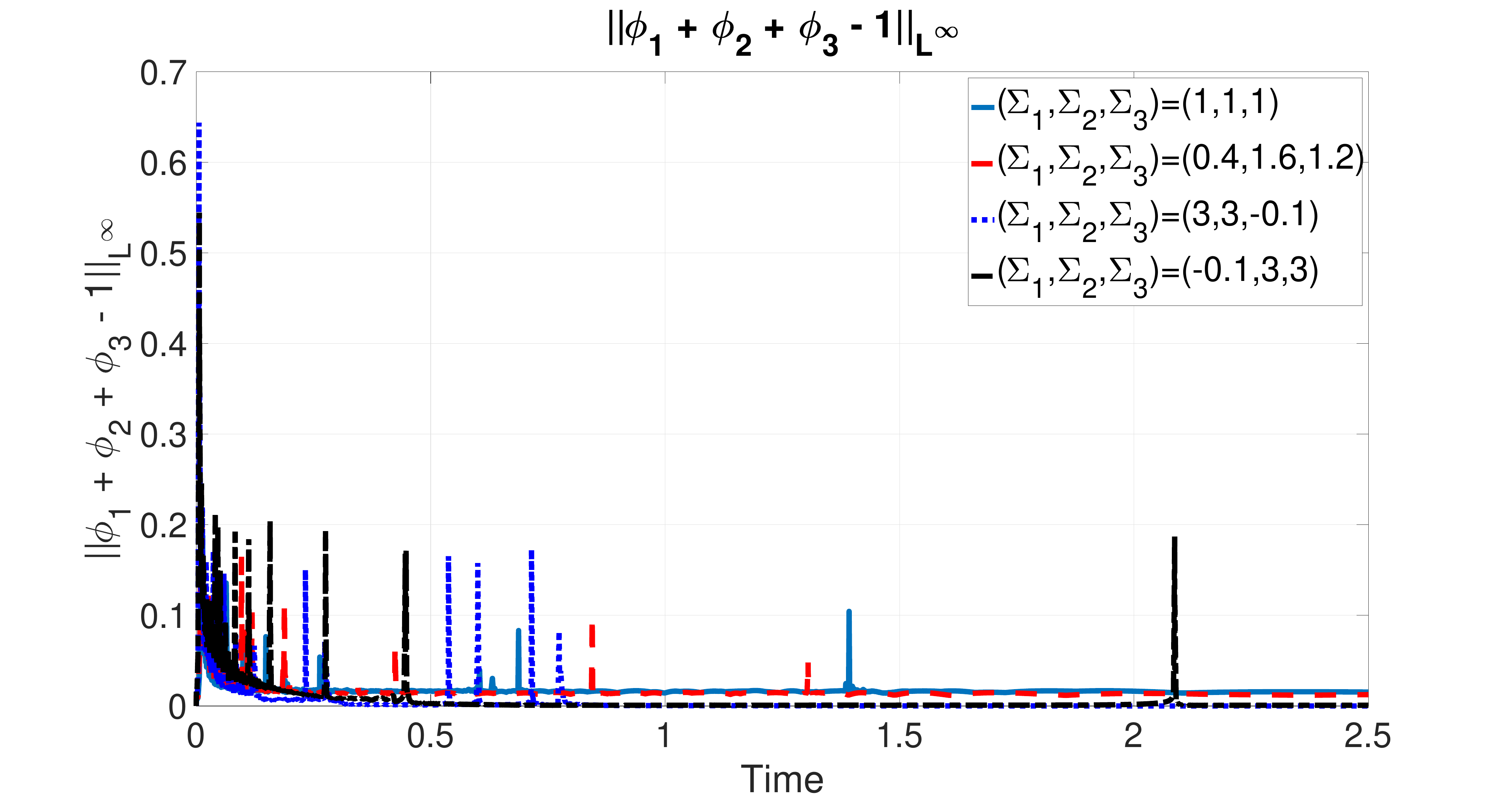}
\includegraphics[scale=0.11]{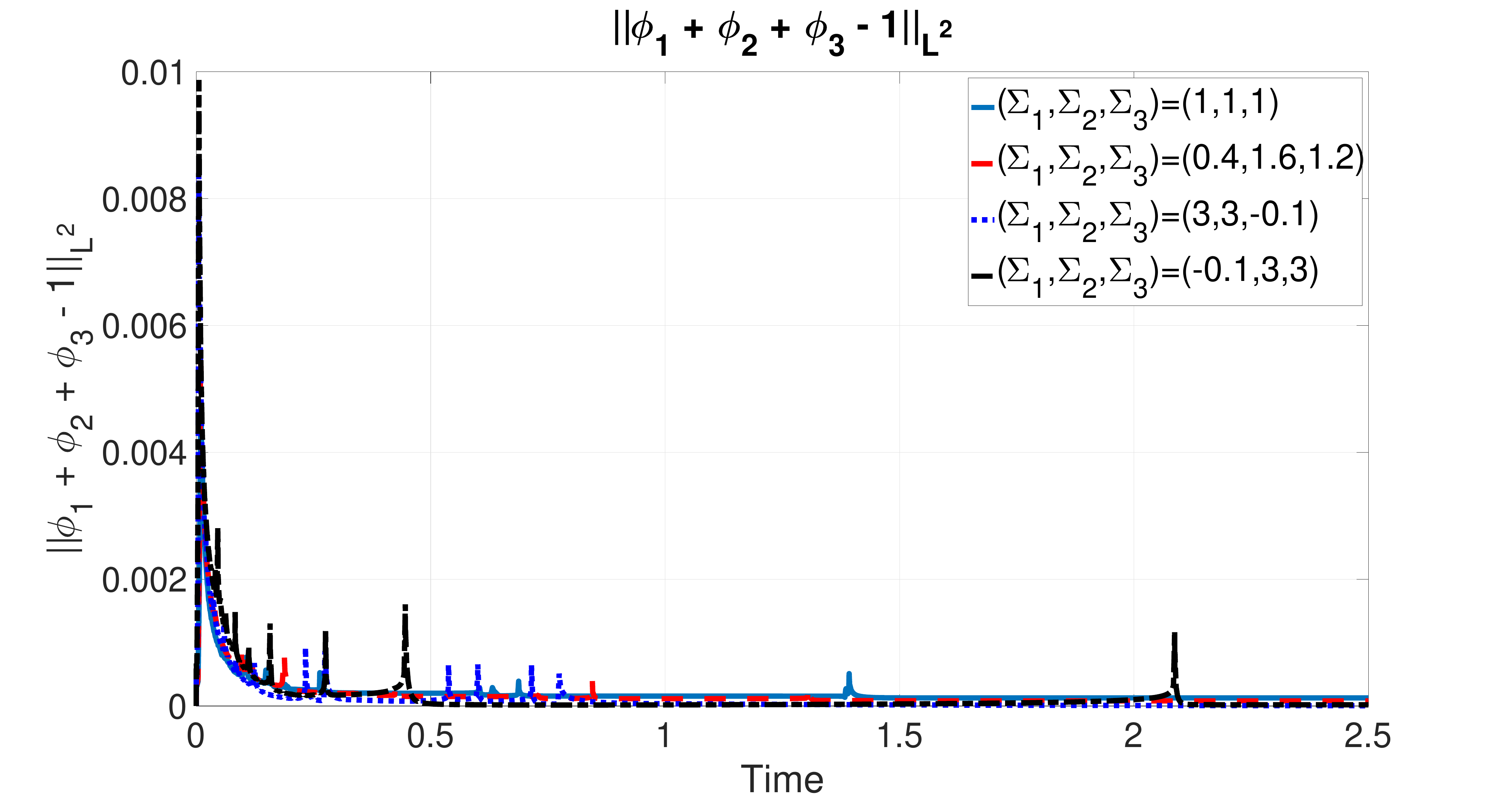}
\end{center}
\caption{Evolution in time of the energies (top left), the volume (top right), $\|\phi_1 + \phi_2 + \phi_3 -1\|_{L^\infty}$ (bottom left), $\|\phi_1 + \phi_2 + \phi_3 -1\|_{L^2}$ (bottom right) for the results presented in Figure~\ref{fig:SpinodalDynamics}.}
\label{fig:SpinodalPlots}
\end{figure}

\begin{figure}[h]
\begin{center}
\includegraphics[scale=0.11]{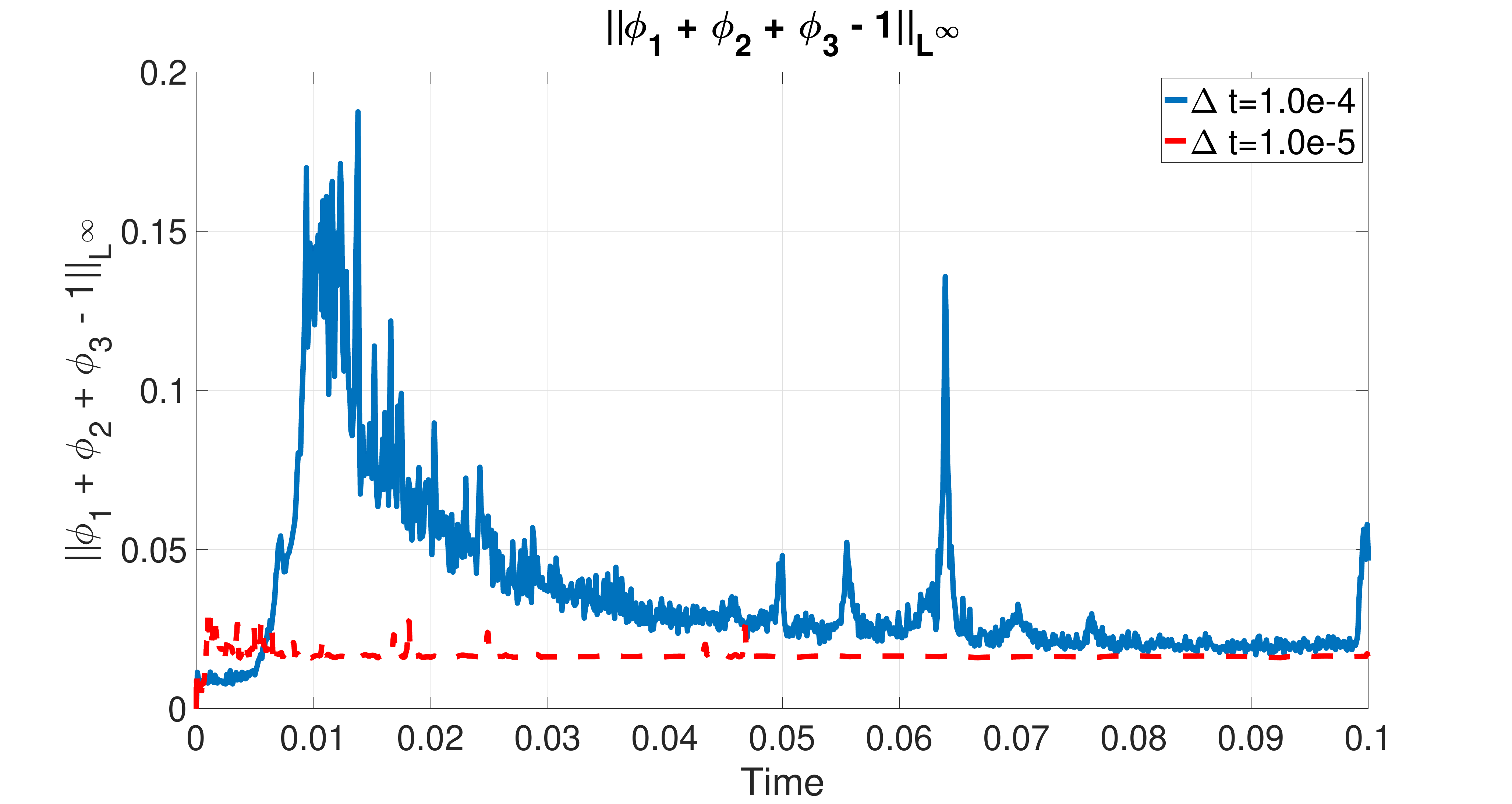}
\includegraphics[scale=0.11]{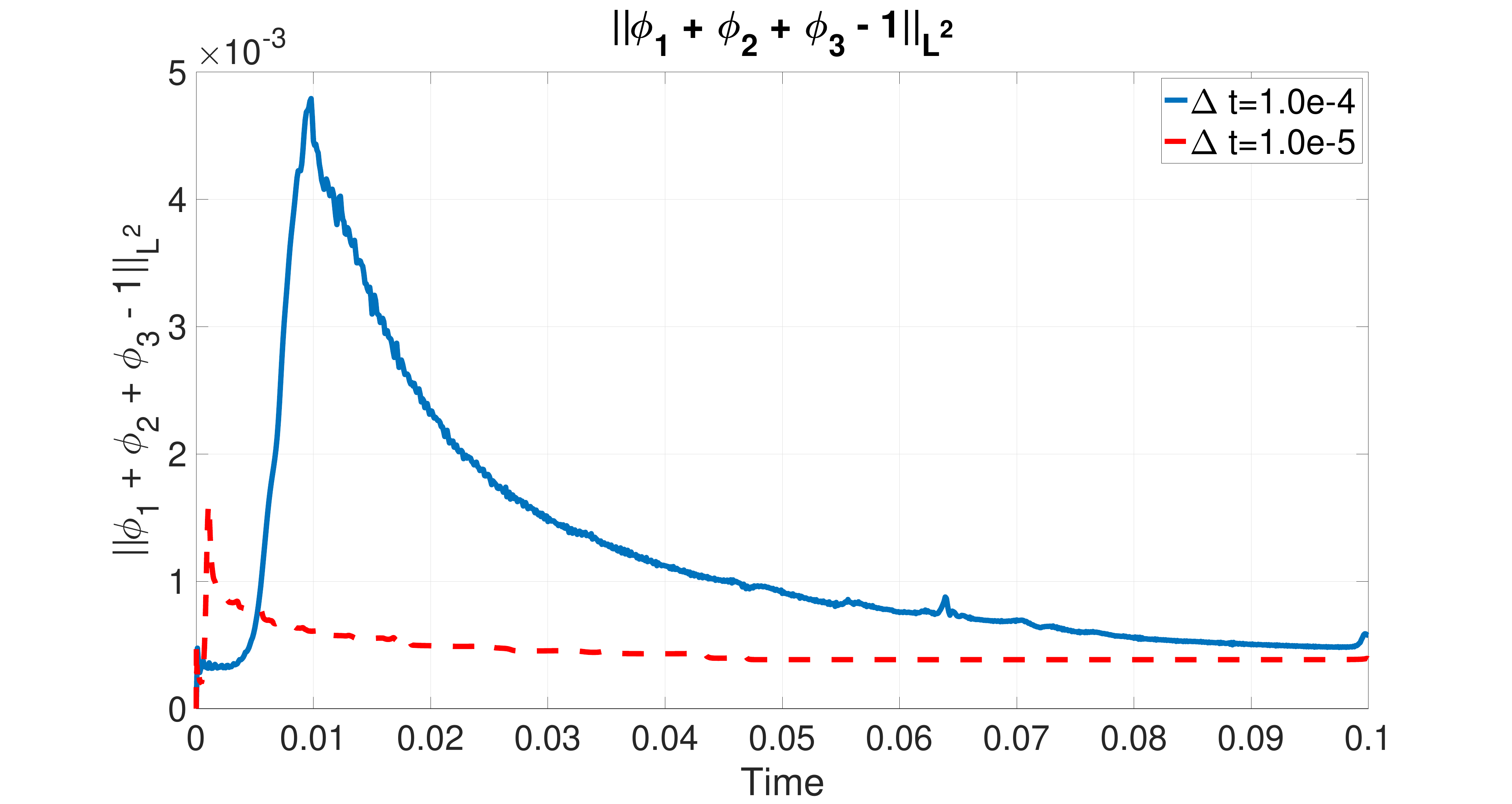}
\end{center}
\caption{Comparison of $\|\phi_1 + \phi_2 + \phi_3 -1\|_{L^\infty}$ (left), $\|\phi_1 + \phi_2 + \phi_3 -1\|_{L^2}$ (right) for time steps $\Delta t=$1e-4 and $\Delta t=$1e-5 with spreading coefficients 
$(\Sigma_1, \Sigma_2 , \Sigma_3) = (1,1,1)$.}
\label{fig:SpinodalPlotsdt}
\end{figure}

\subsubsection{3D simulations}

In order to showcase the efficiency of scheme NTD1 we perform simulations in three dimensions for both the partial and the total spreading situations. The experimental parameters are given in Table~\ref{tab:spinodal3DParameters} are the initial condition is presented in \eqref{eq:spinodal3DInitial}.
\begin{table}[h]
{
\begin{tabular}{c|c|c|c|c|c|c|c}
$\Omega$ 			& $h$ 		& $[0,T]$ 		& $\dt$ 	& $\eps$ & $\lambda$ 	& $M$ 	& $\Lambda$ \\
\hline
$[-0.125, 0.125]\times[-0.125, 0.125]\times[0,0.05]$	& 1/300	& $[0,0.4]$ 	& 1e-4 	& 1e-2  & 1e-4 	&  1e-3 	& 7 \\
\end{tabular}
}
\caption{\label{tab:spinodal3DParameters} Parameters of the three dimensional spinodal decomposition experiment.}
\end{table}
\beq\label{eq:spinodal3DInitial}
\left\{
\ba{rcl}
\phi_1(x,y,z) &=& 0.33 +  0.01 \mbox{rand}(x,y,z)\,,
\\ \hueco
\phi_2(x,y,z) &=& 0.33 +  0.01 \mbox{rand}(x,y,z)\,,
\\ \hueco
\phi_3(x,y,z) &=& 1 - \phi_1(x,y,z) - \phi_2(x,y,z)\,.
\ea
\right.
\eeq
{
We present the dynamics of two situations in Figures~\ref{fig:Spinodal3DDynamicsPartial} and \ref{fig:Spinodal3DDynamicsTotal} for partial spreading ($(\Sigma_1, \Sigma_2 , \Sigma_3) = (0.4, 1.6, 1.2)$) and total spreading ($(\Sigma_1, \Sigma_2 , \Sigma_3) = (-0.1,3,3)$), respectively. We can observe how the scheme is able to obtain the dynamics that we expect for this type of simulations in $3D$. In Figure~\ref{fig:Spinodal3DPlots} the evolution in time of the energy, the volume and the $L^2$ and $L^\infty$ norms of the restriction $\Sigma_{i=1}^3\phi_i - 1$ are shown. In both cases  the energy decreases (at the end of the simulations the system is not at equilibrium yet, the energy is decreasing but in a very slow way compared with the beginning of the simulations, as it usually happens in spinodal simulations). Moreover the volume is conserved in both situations. In this case, the $L^2$ norm of the restriction seems to get a reasonable approximation but the $L^\infty$ norm is not very well approximated, which has to do with the coarse choice of the time step that we have considered. To evidence this fact we present a comparison in Figure~\ref{fig:Spinodal3DPlotsdt} of the results when the time step is lowered to $\Delta t=$1e-5 and $\Delta t=$1e-6 (and the time interval is only $[0,0.001]$ to save computational time) in the total spreading case ($(\Sigma_1, \Sigma_2 , \Sigma_3) =  (-0.1,3,3))$ and we can see how even in this challenging situation the approximation of the constraint clearly improves when the time step is lowered. 
}

\begin{figure}[h]
\begin{center}
\includegraphics[scale=0.1225]{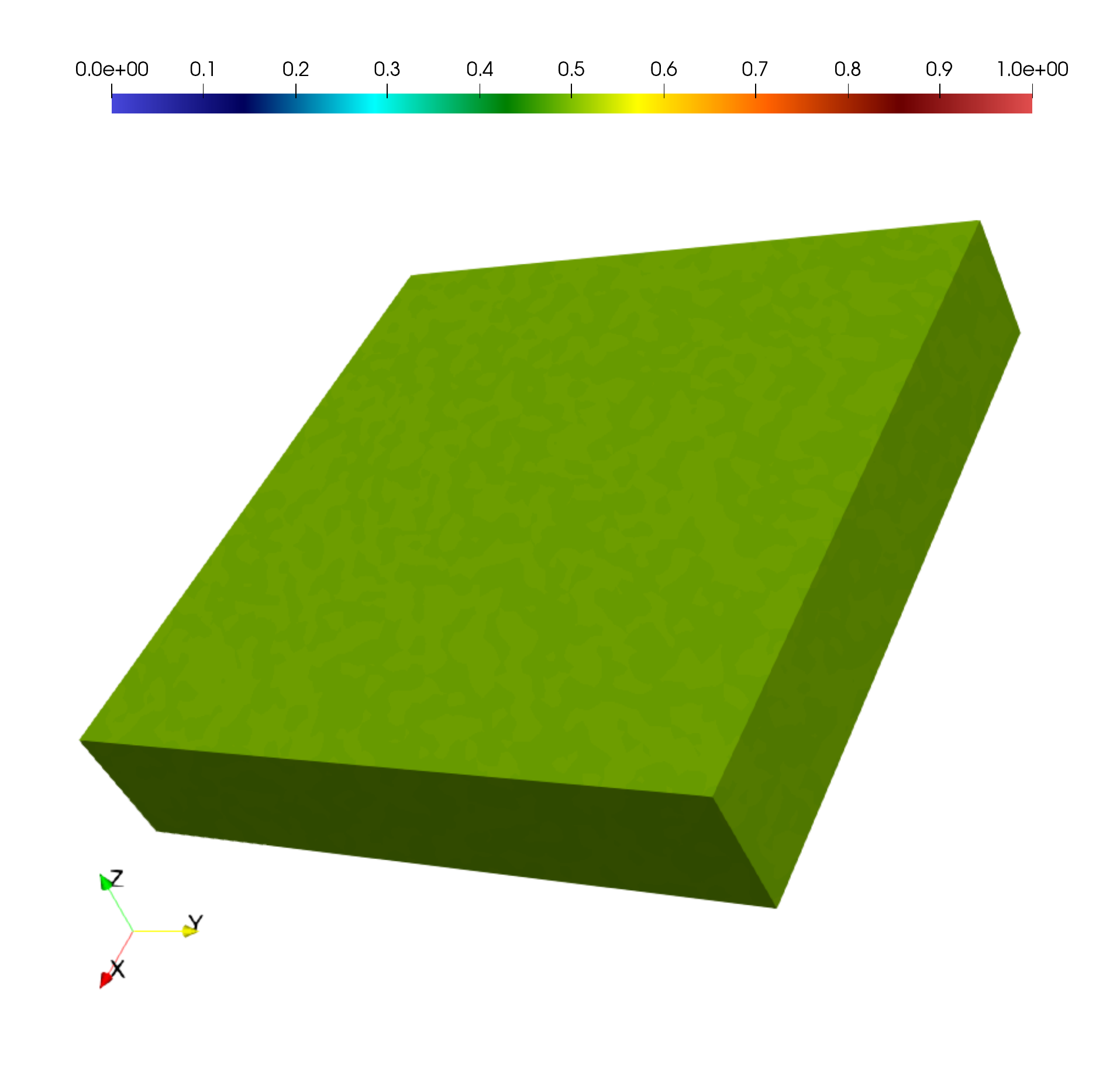}
\includegraphics[scale=0.1225]{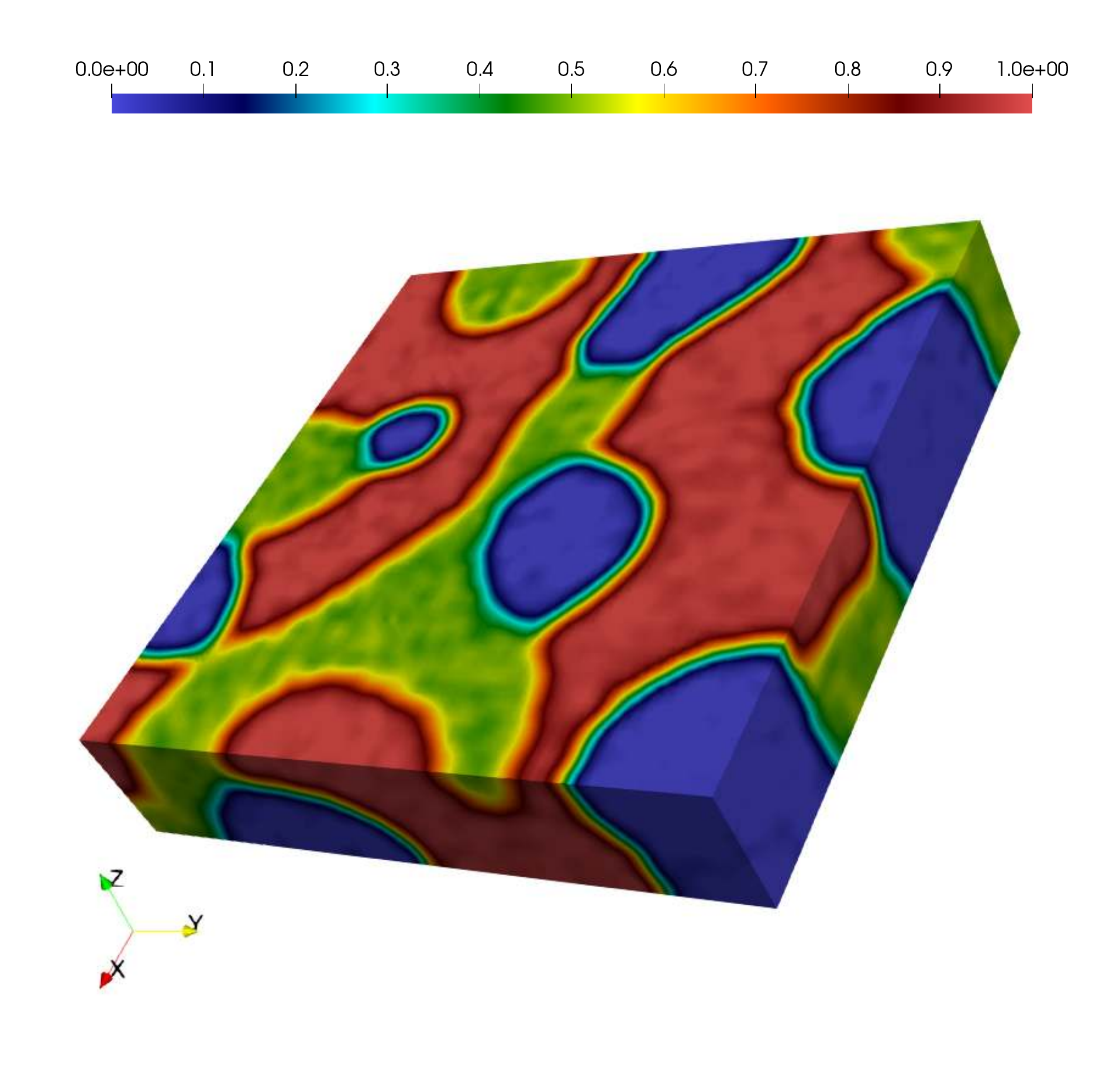}
\includegraphics[scale=0.1225]{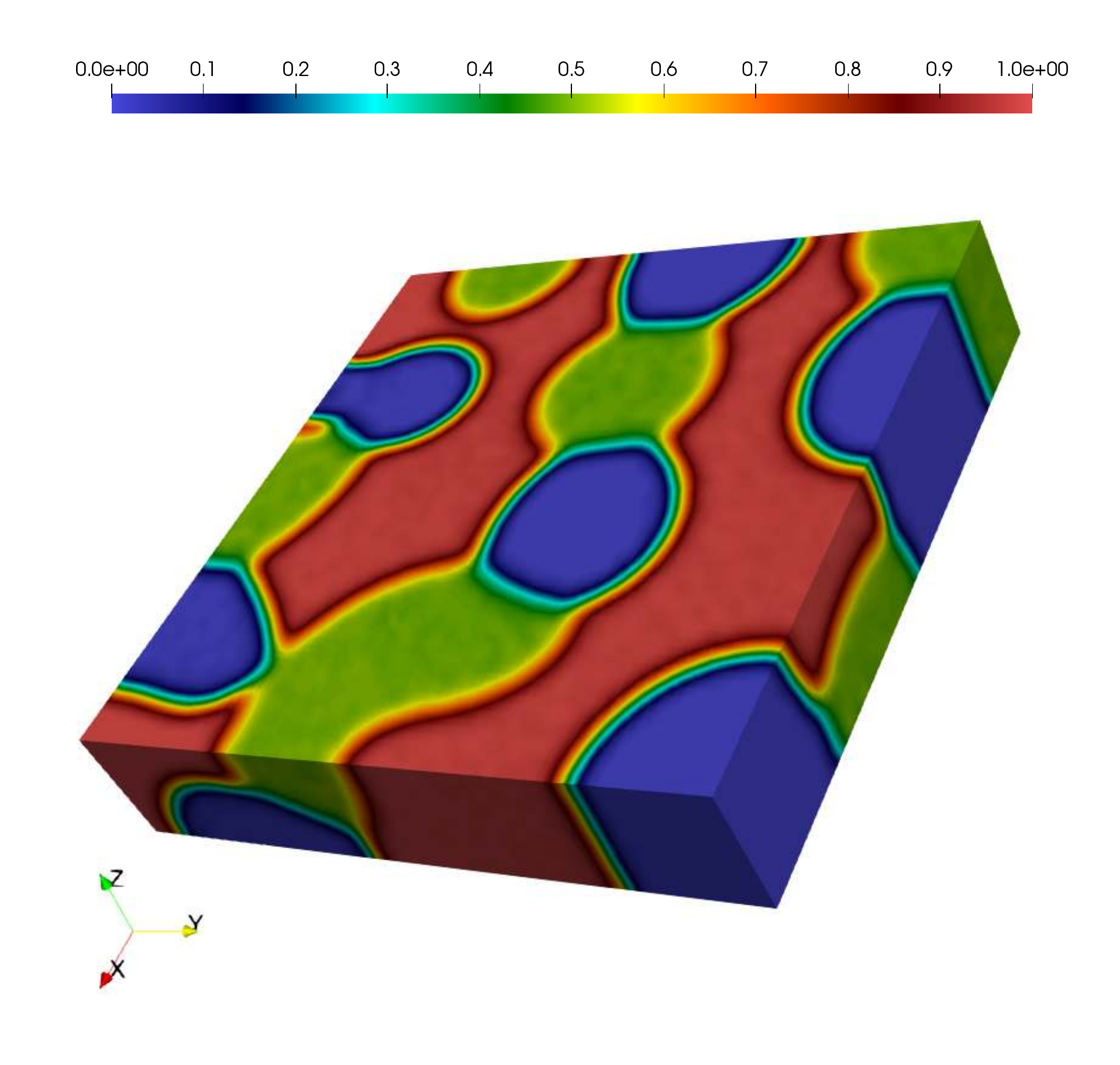}
\includegraphics[scale=0.1225]{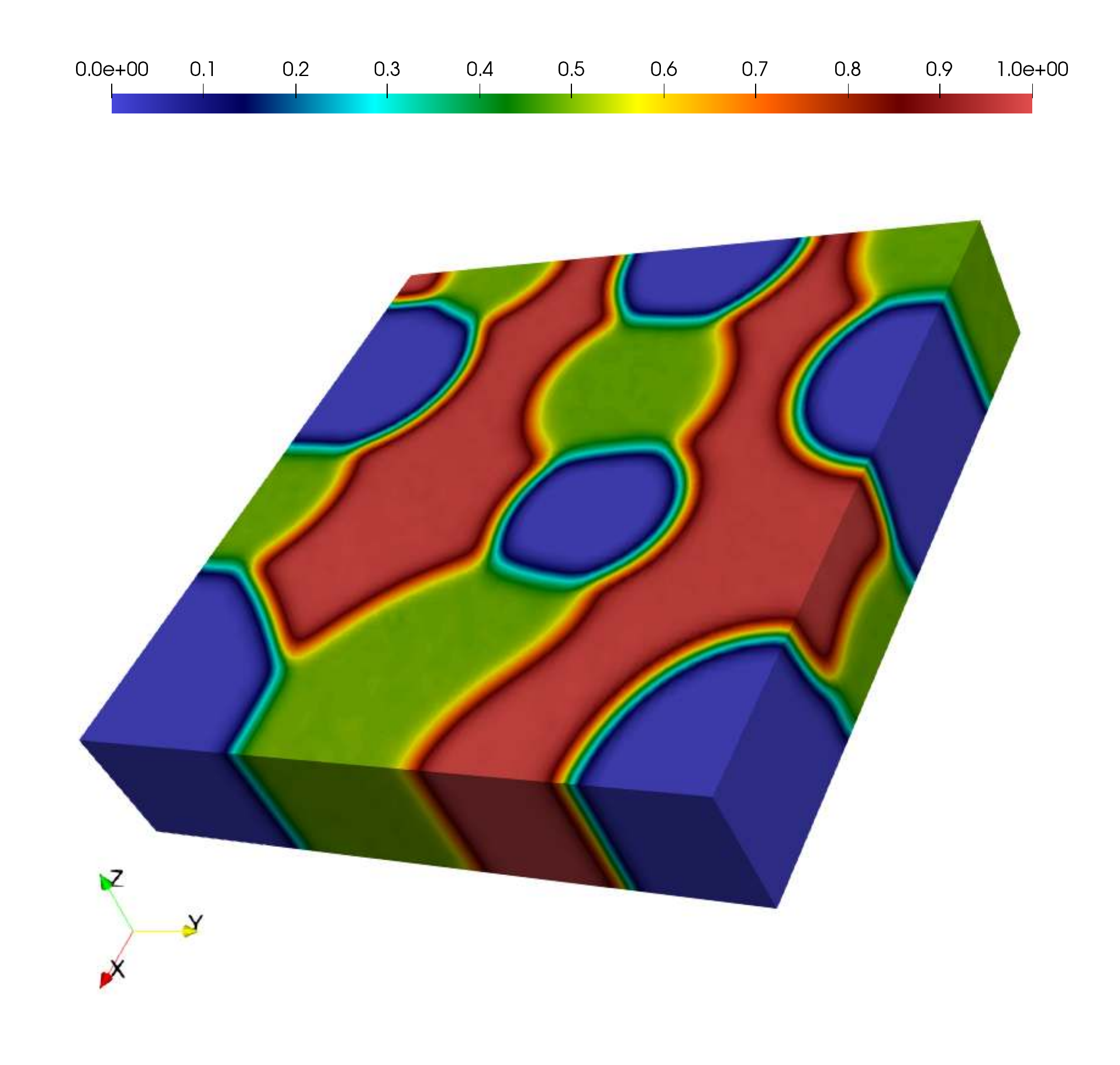}
\\ [1ex]
\includegraphics[scale=0.1225]{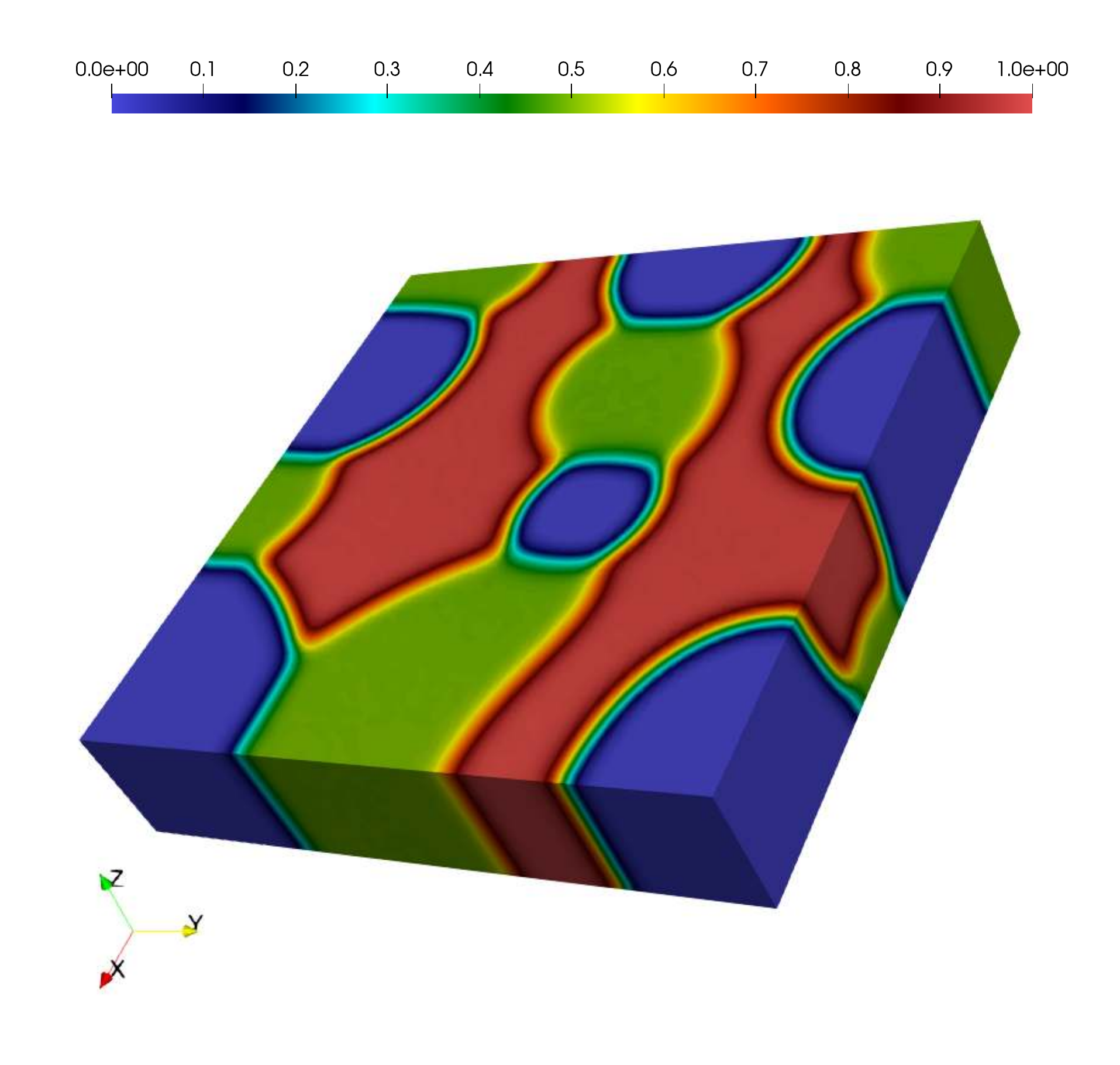}
\includegraphics[scale=0.1225]{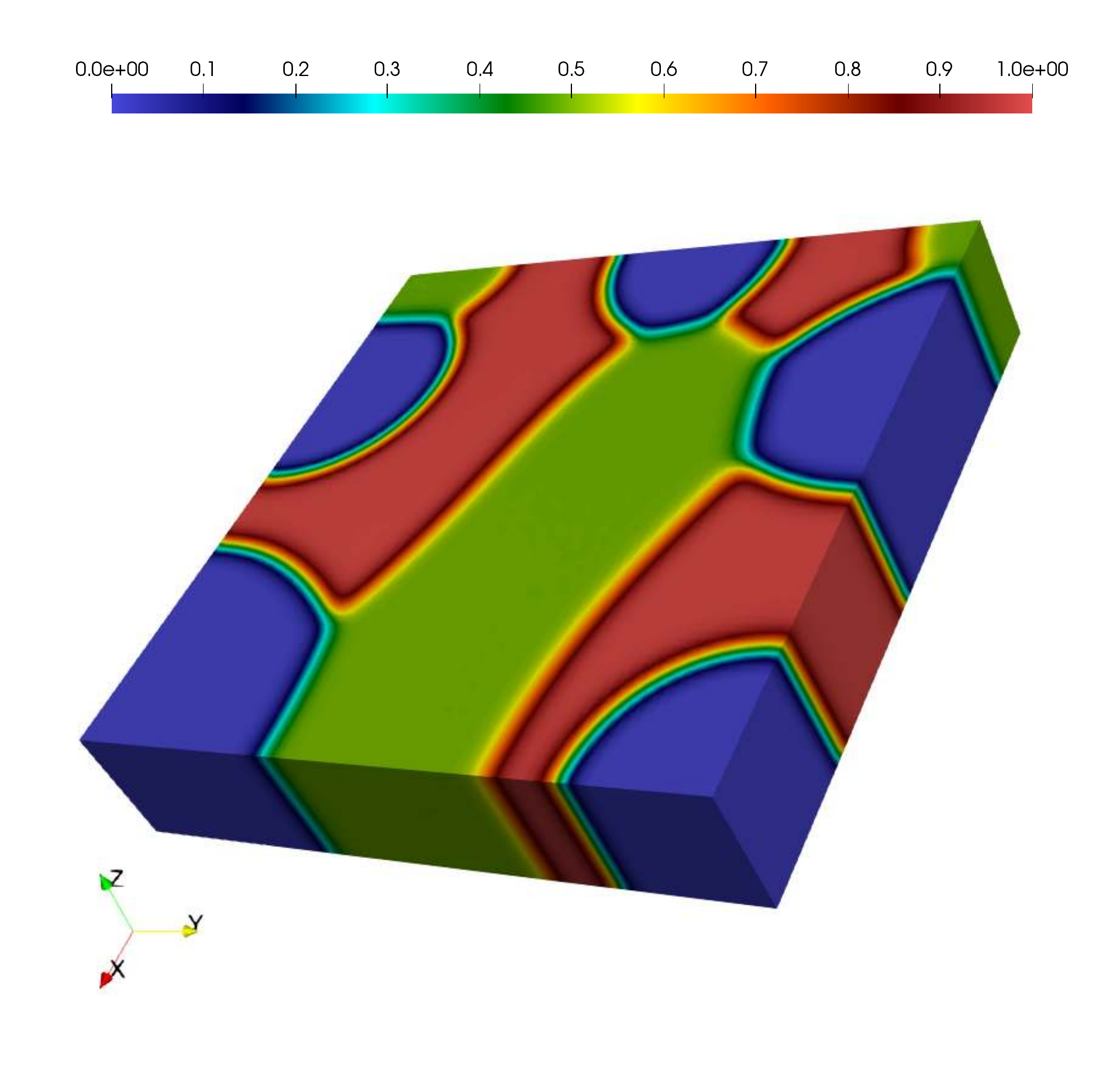}
\includegraphics[scale=0.1225]{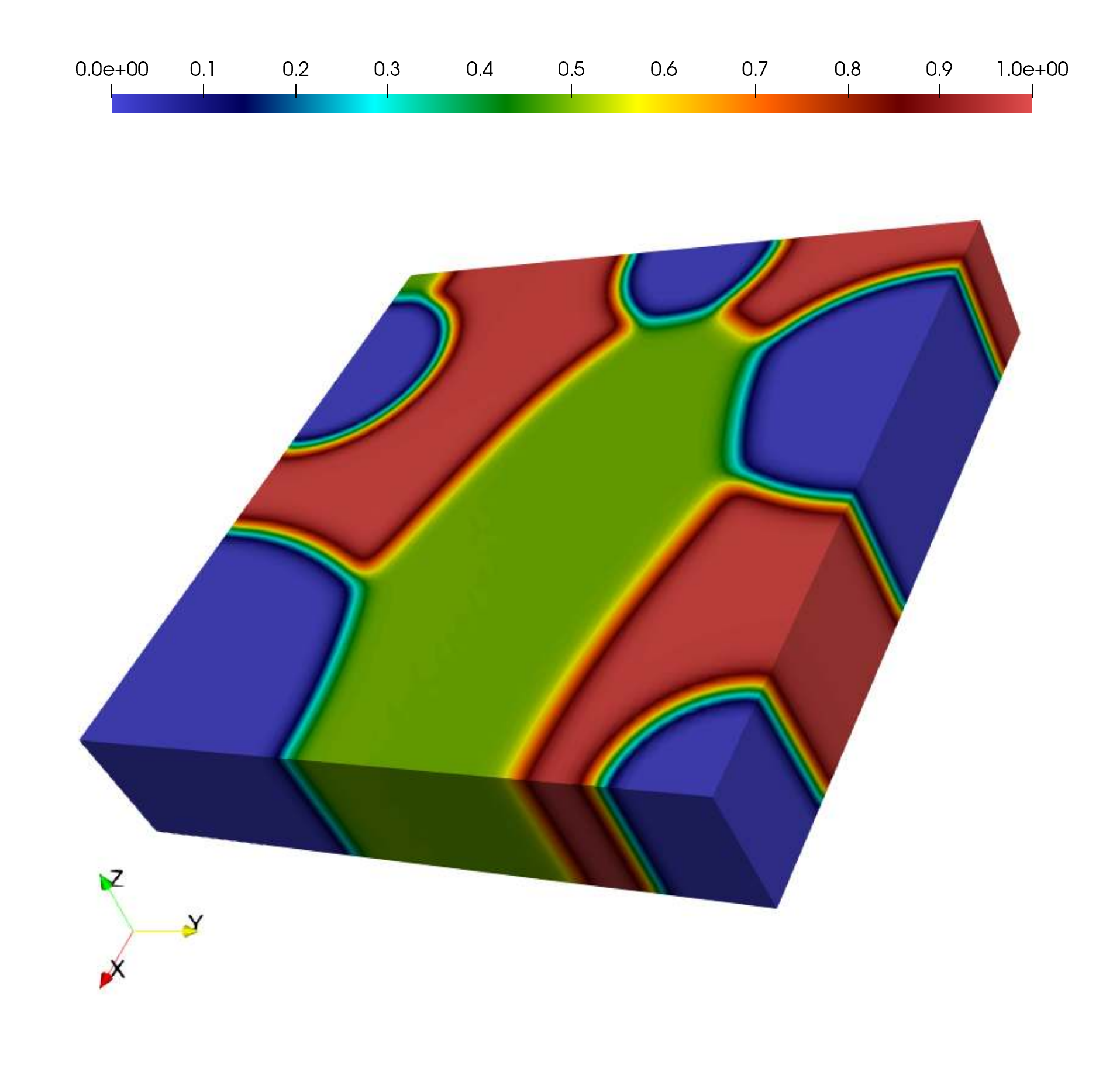}
\includegraphics[scale=0.1225]{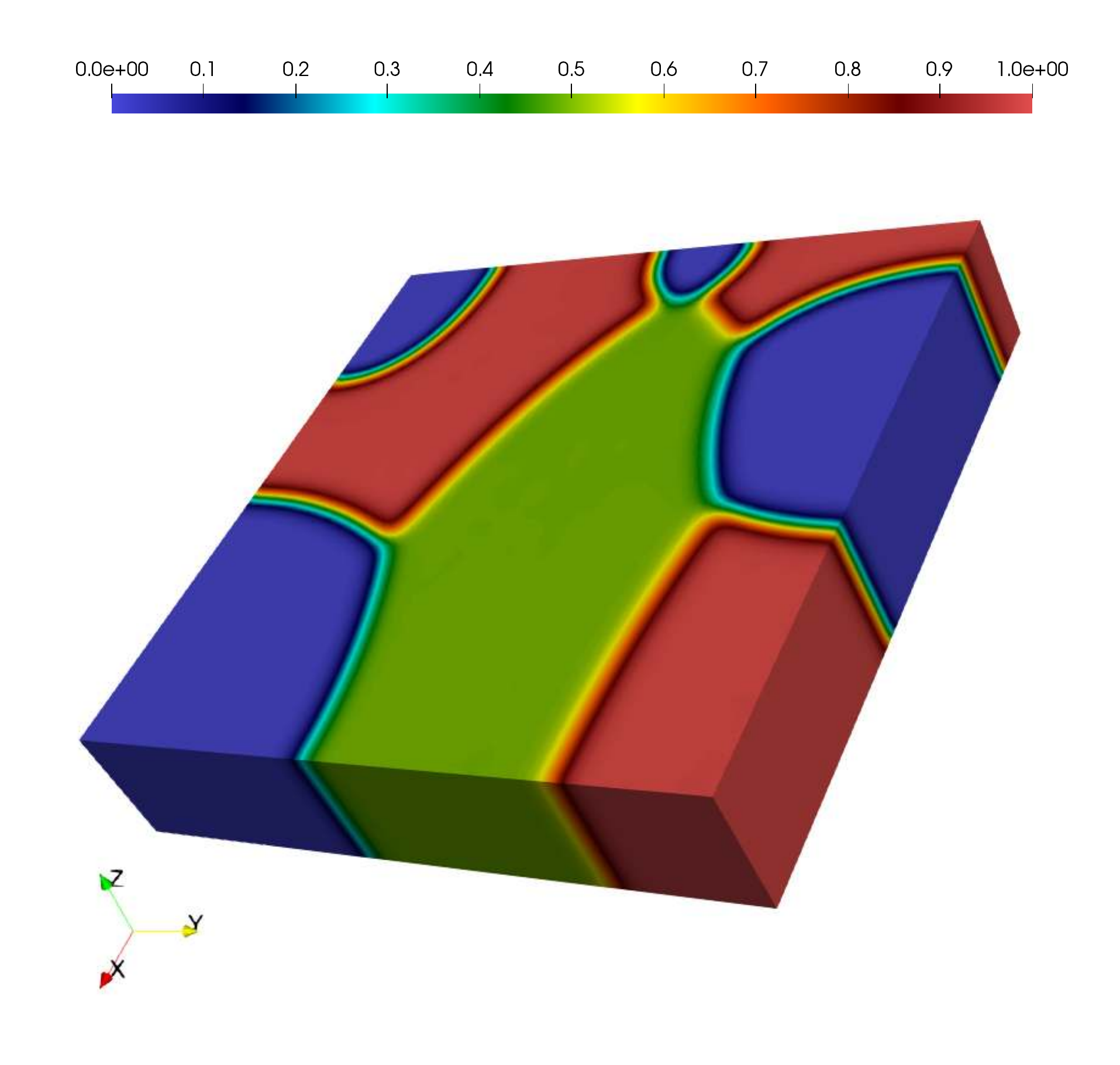}
\end{center}
\caption{Dynamics of scheme NTD1 at times $t=0, 0.025, 0.05, 0.075, 0.1, 0.2, 0.3$ and $0.4$ (from left to right and top to bottom) with spreading coefficients 
$(\Sigma_1, \Sigma_2 , \Sigma_3) = (0.4, 1.6, 1.2)$.}
\label{fig:Spinodal3DDynamicsPartial}
\end{figure}

\begin{figure}[h]
\begin{center}
\includegraphics[scale=0.1225]{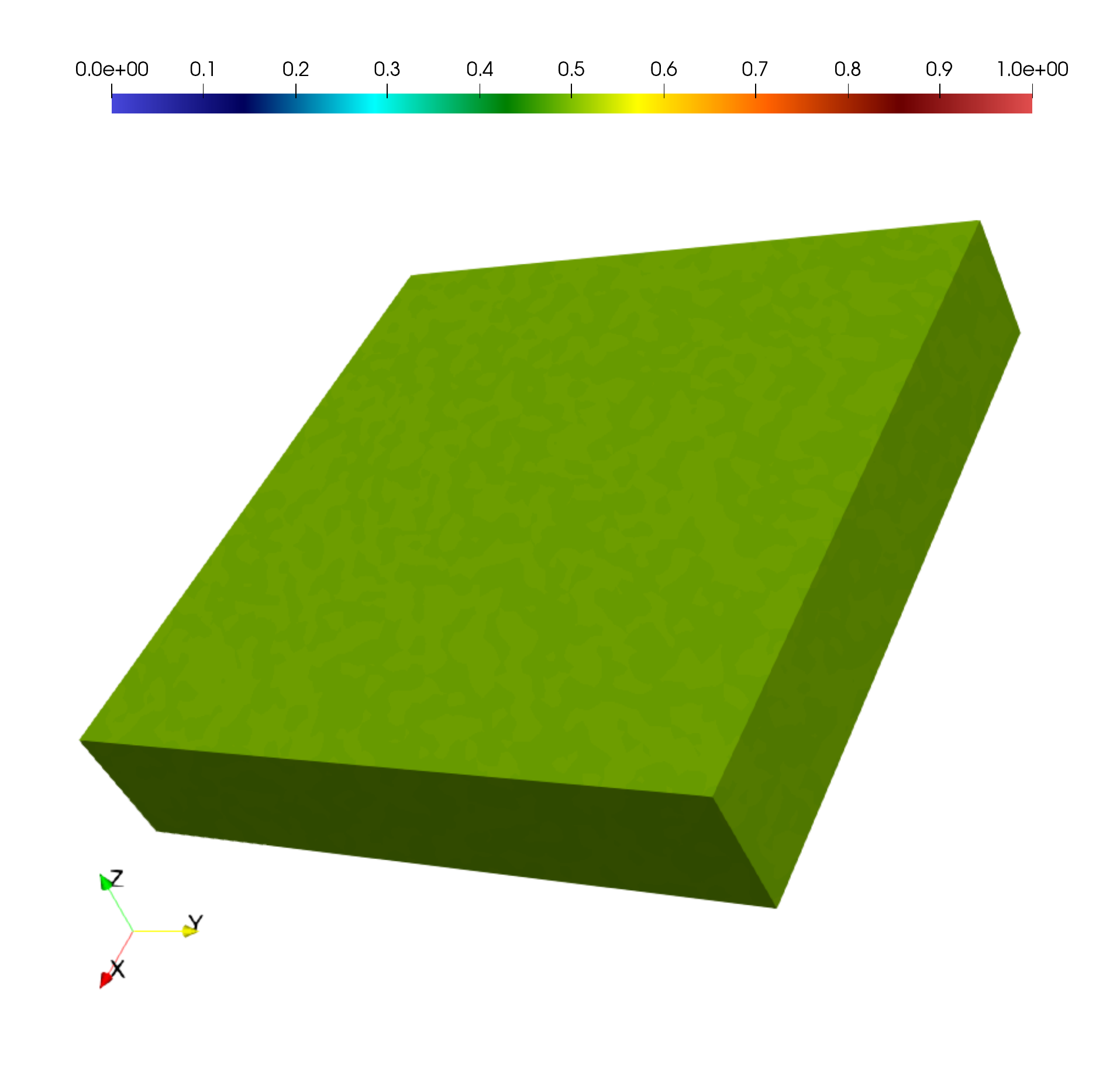}
\includegraphics[scale=0.1225]{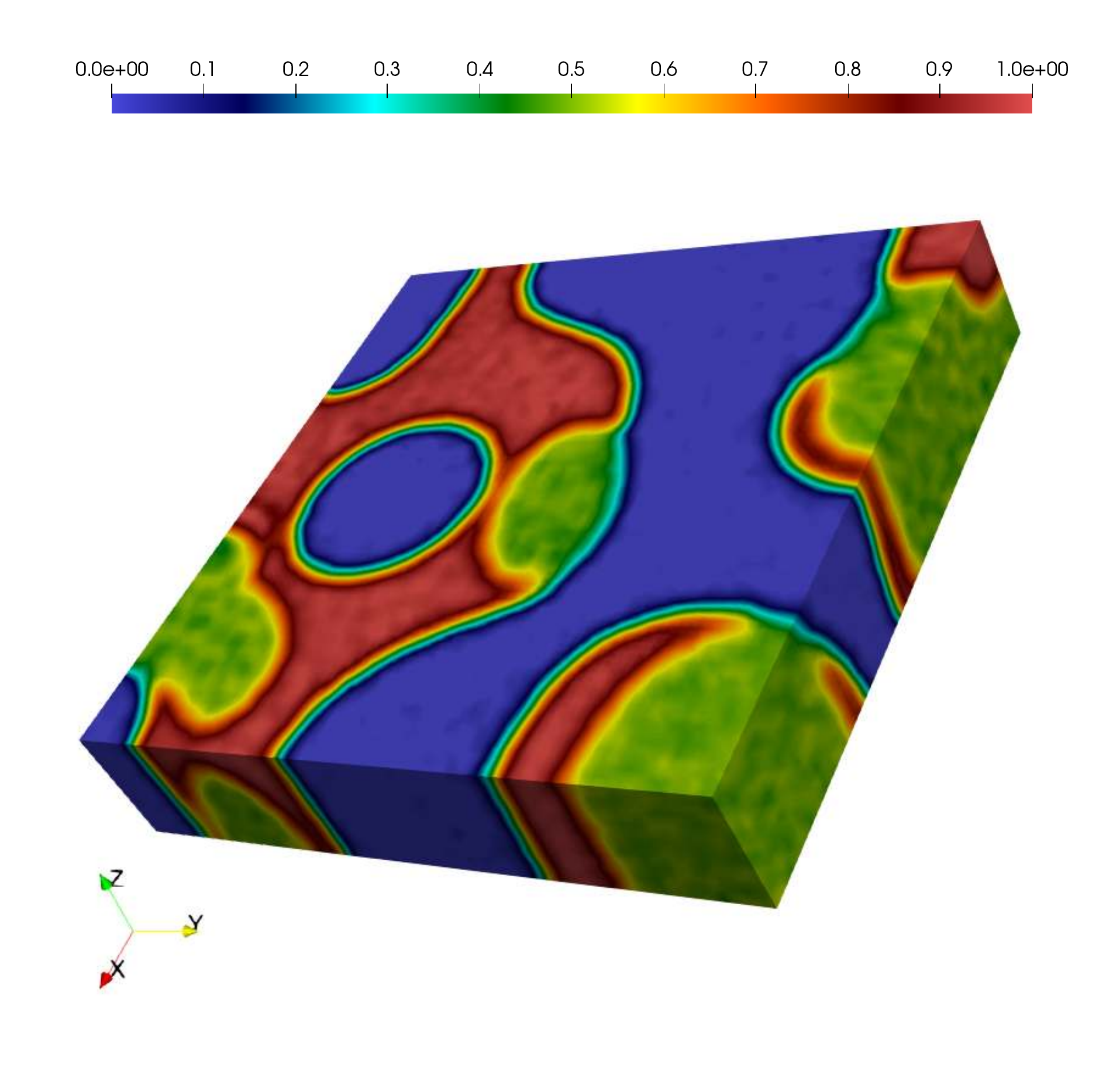}
\includegraphics[scale=0.1225]{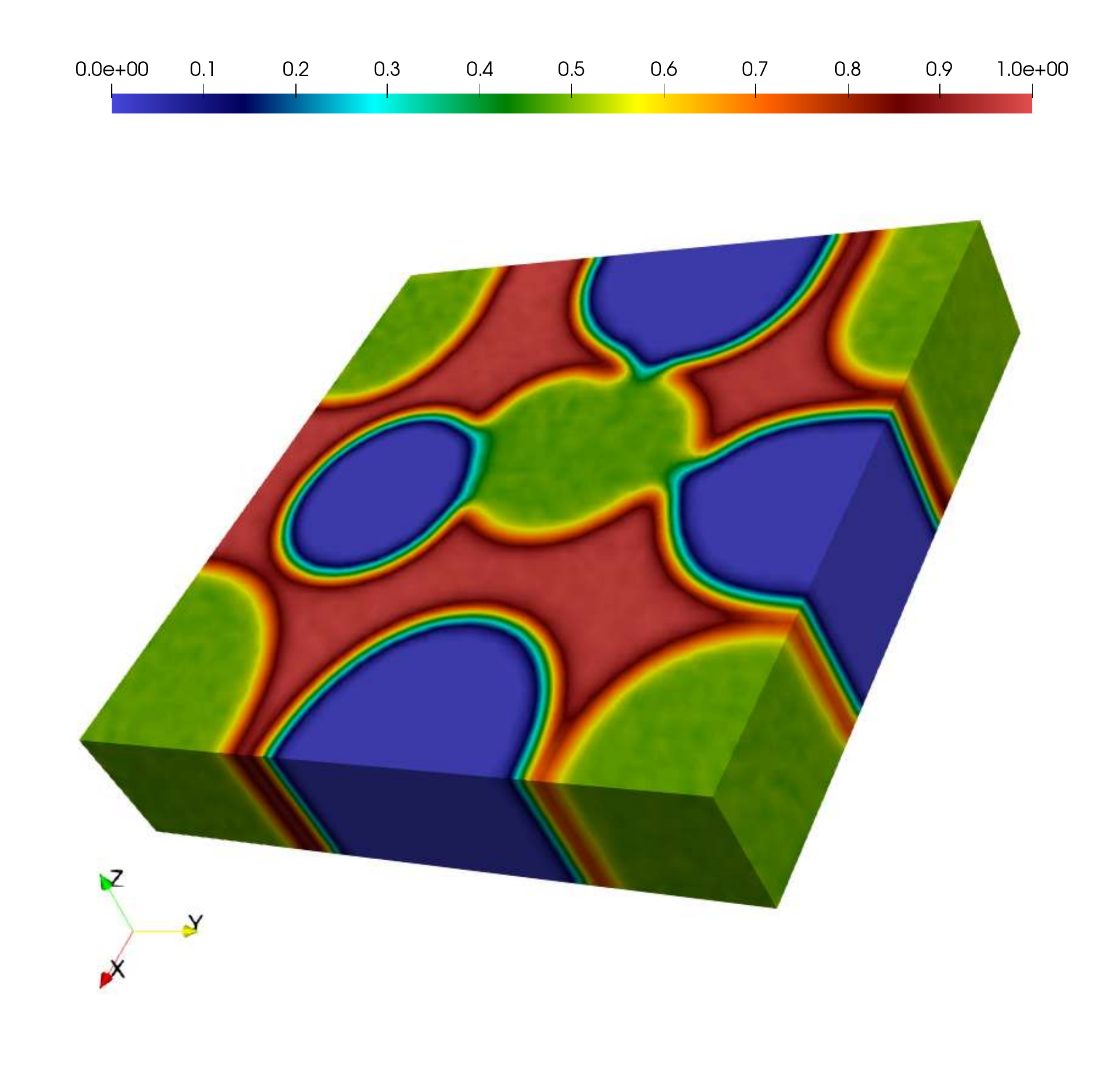}
\includegraphics[scale=0.1225]{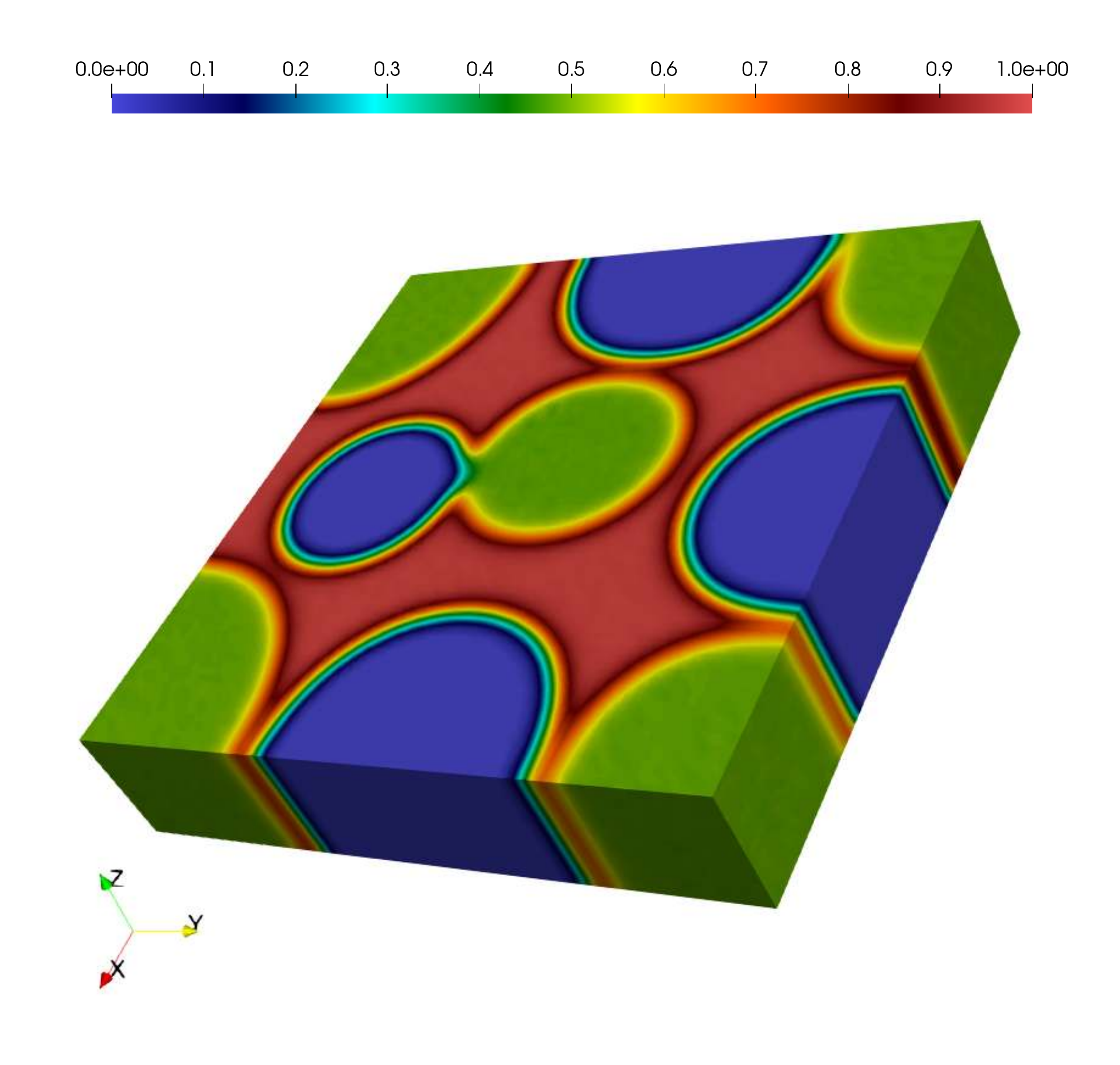}
\\ [1ex]
\includegraphics[scale=0.1225]{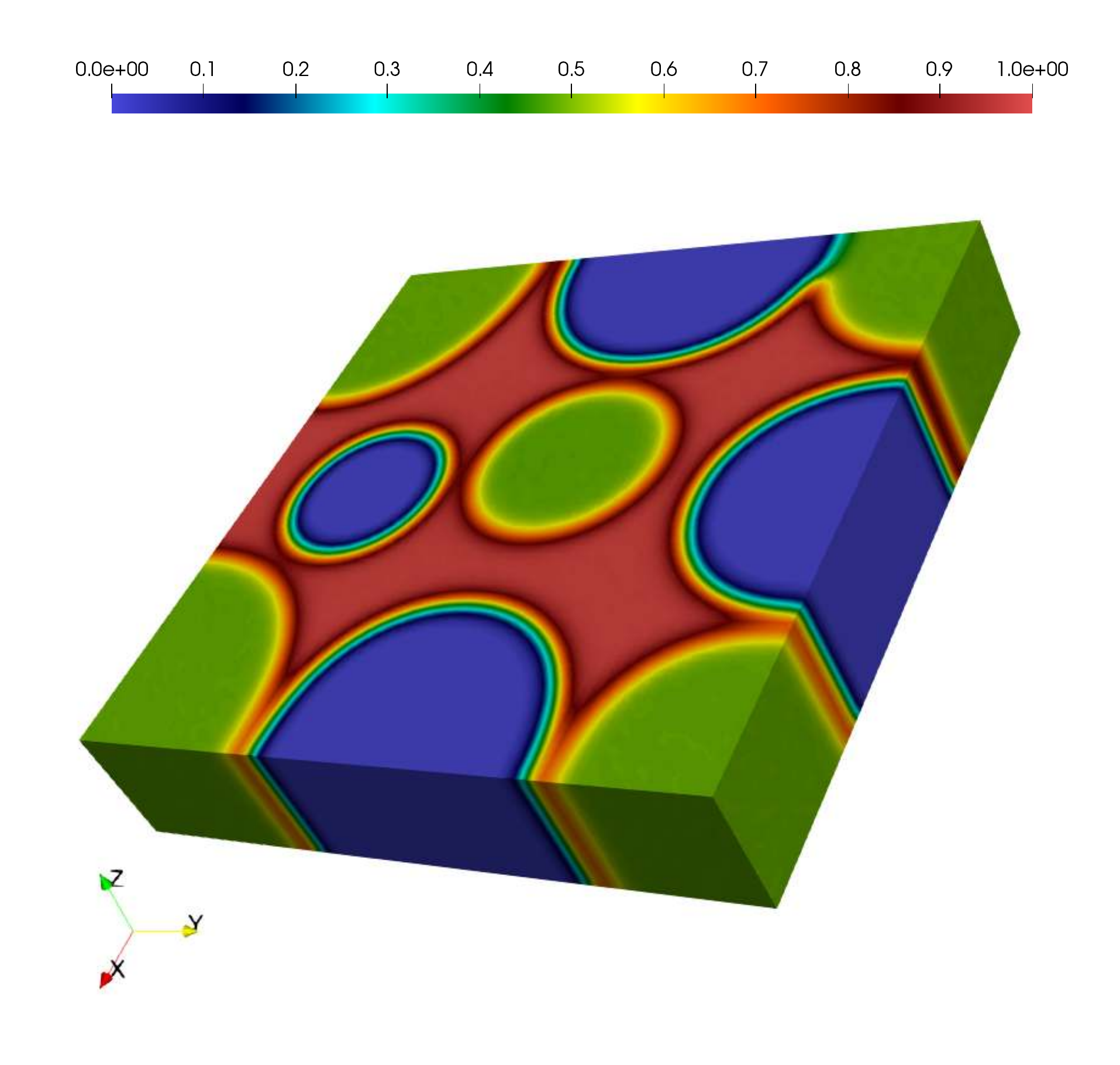}
\includegraphics[scale=0.1225]{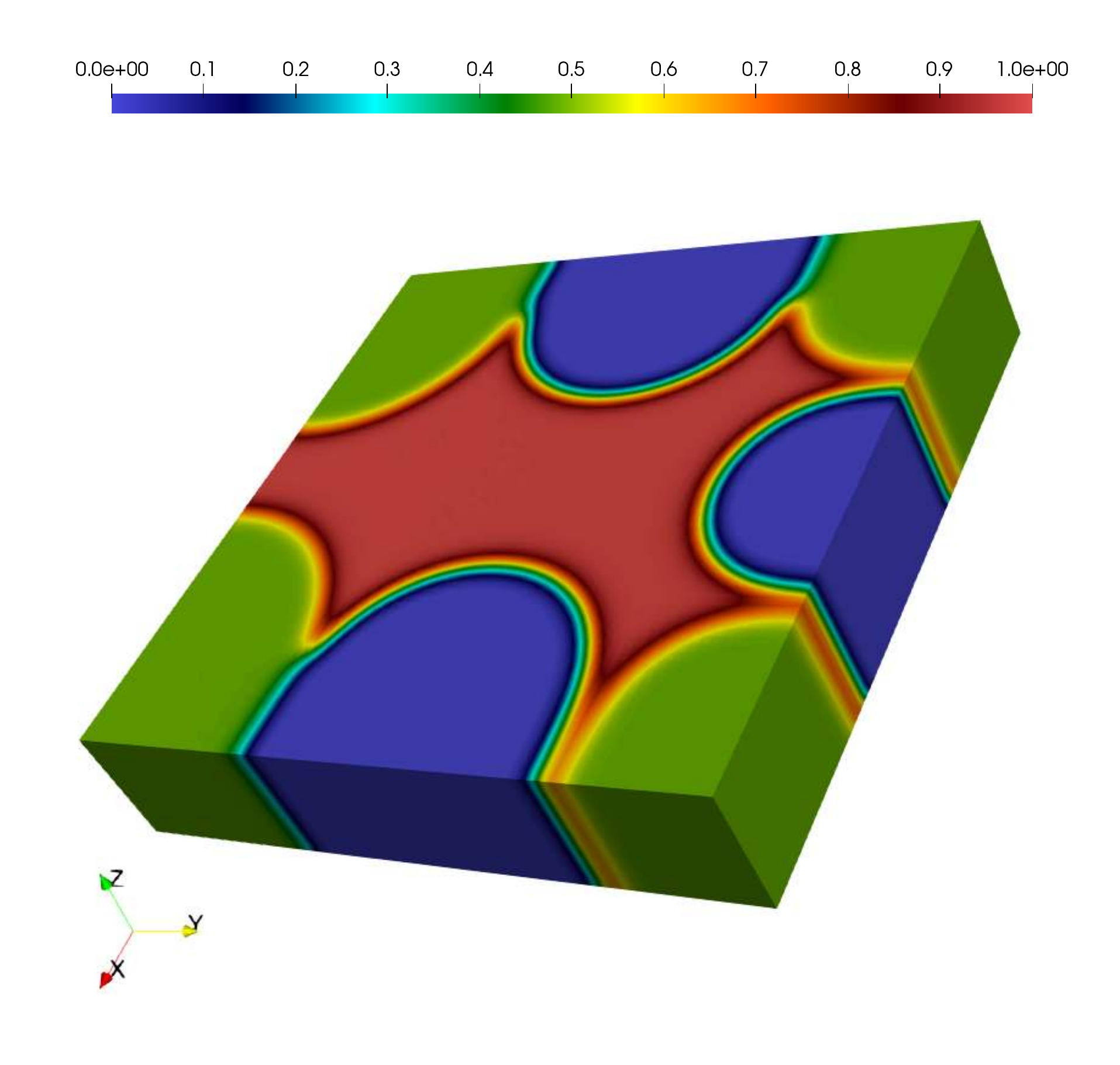}
\includegraphics[scale=0.1225]{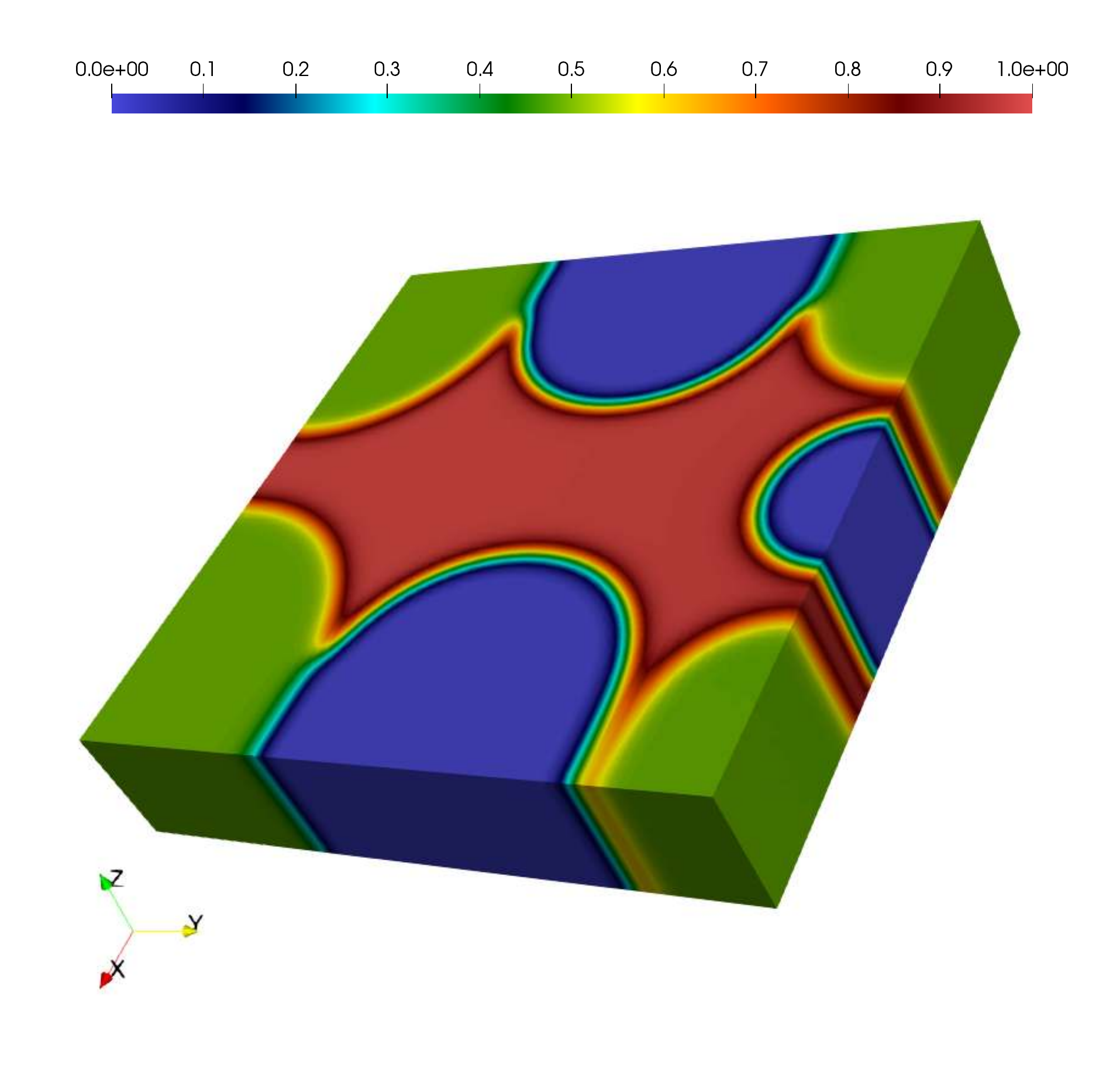}
\includegraphics[scale=0.1225]{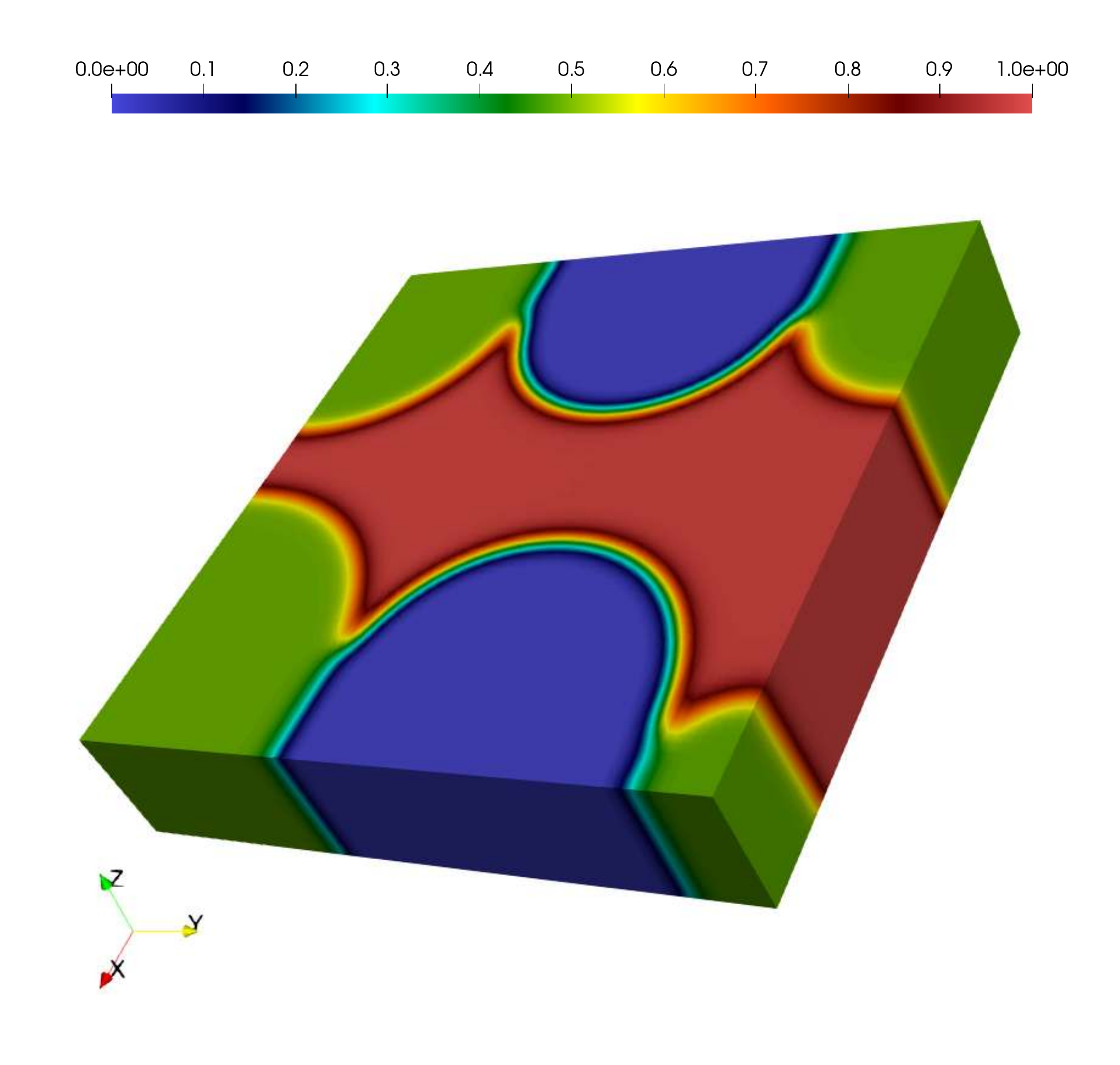}
\end{center}
\caption{Dynamics of scheme NTD1 at times $t=0, 0.025, 0.05, 0.075, 0.1, 0.2, 0.3$ and $0.4$ (from left to right and top to bottom) with spreading coefficients 
$(\Sigma_1, \Sigma_2 , \Sigma_3) = (-0.1, 3, 3)$.}\label{fig:Spinodal3DDynamicsTotal}
\end{figure}

\begin{figure}[h]
\begin{center}
\includegraphics[scale=0.11]{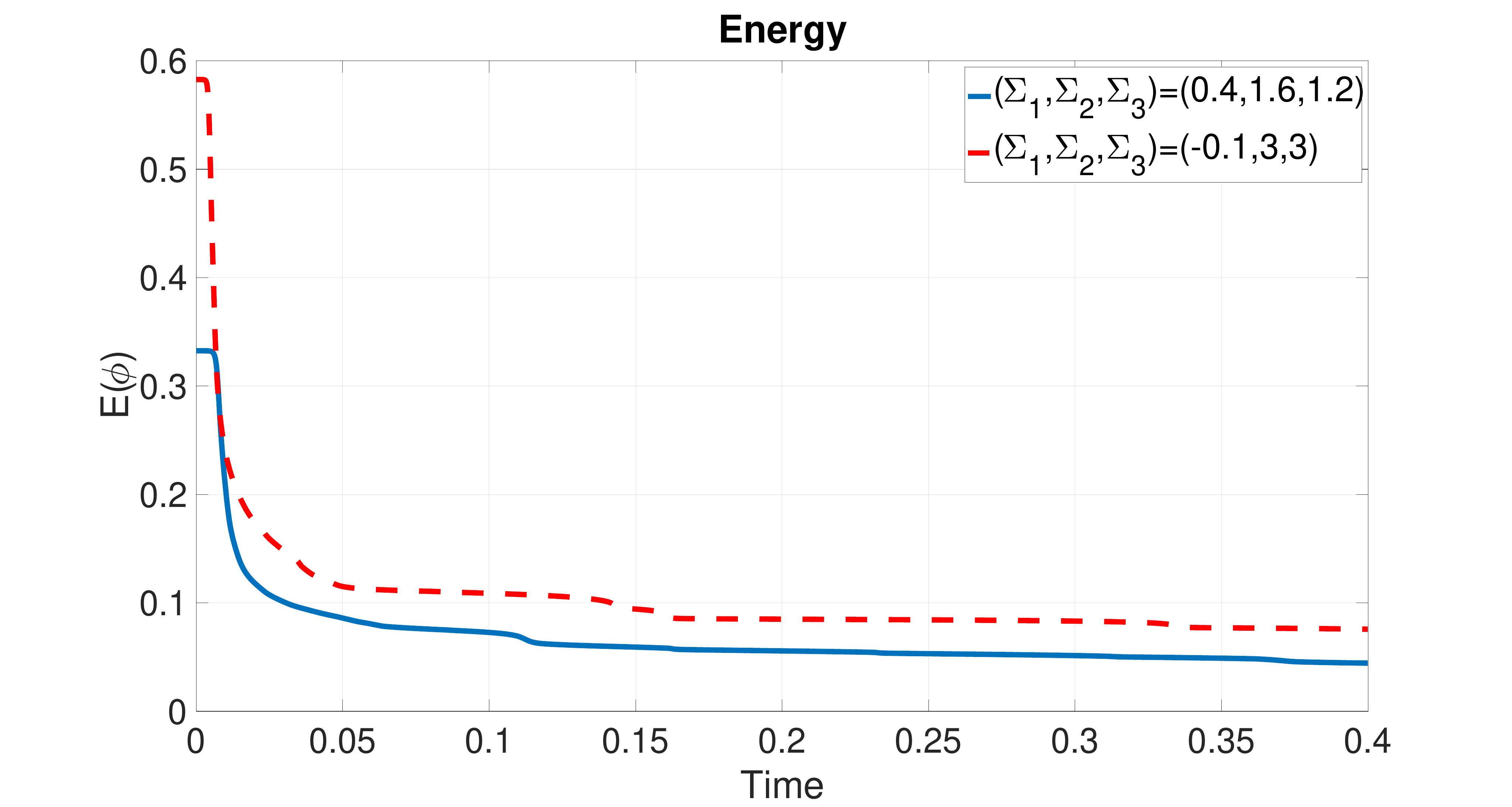}
\includegraphics[scale=0.11]{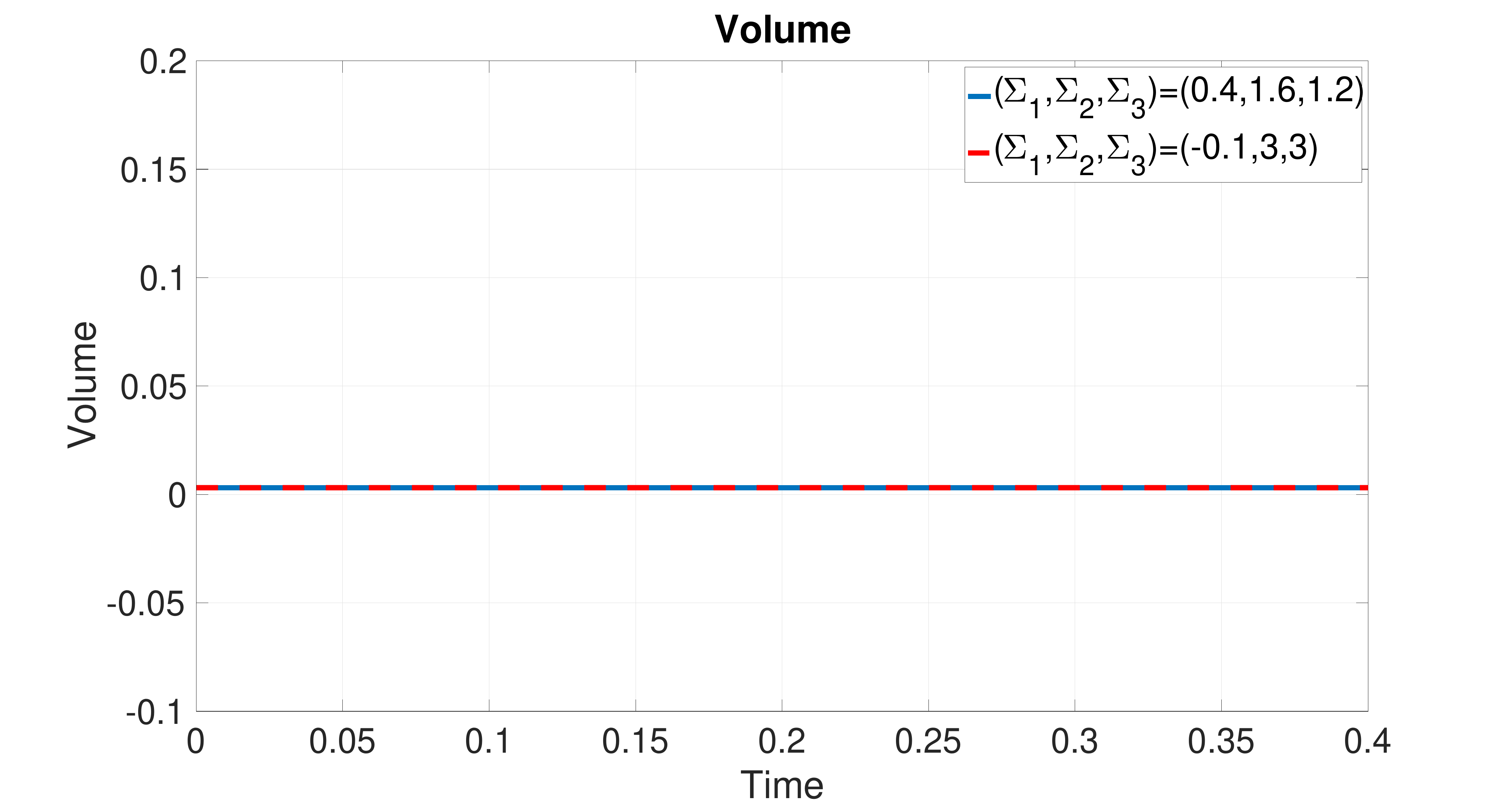}
\\ [1ex]
\includegraphics[scale=0.11]{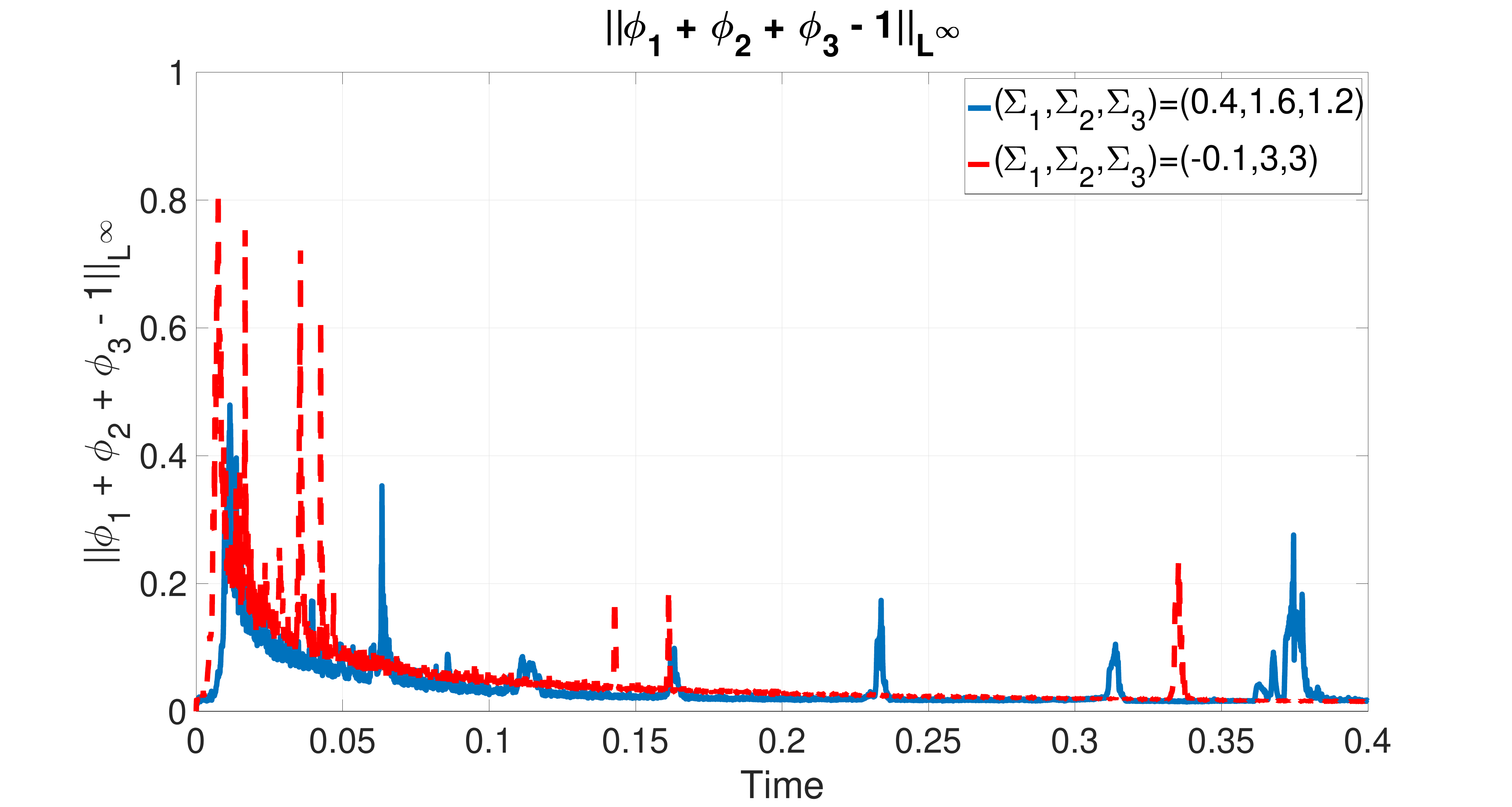}
\includegraphics[scale=0.11]{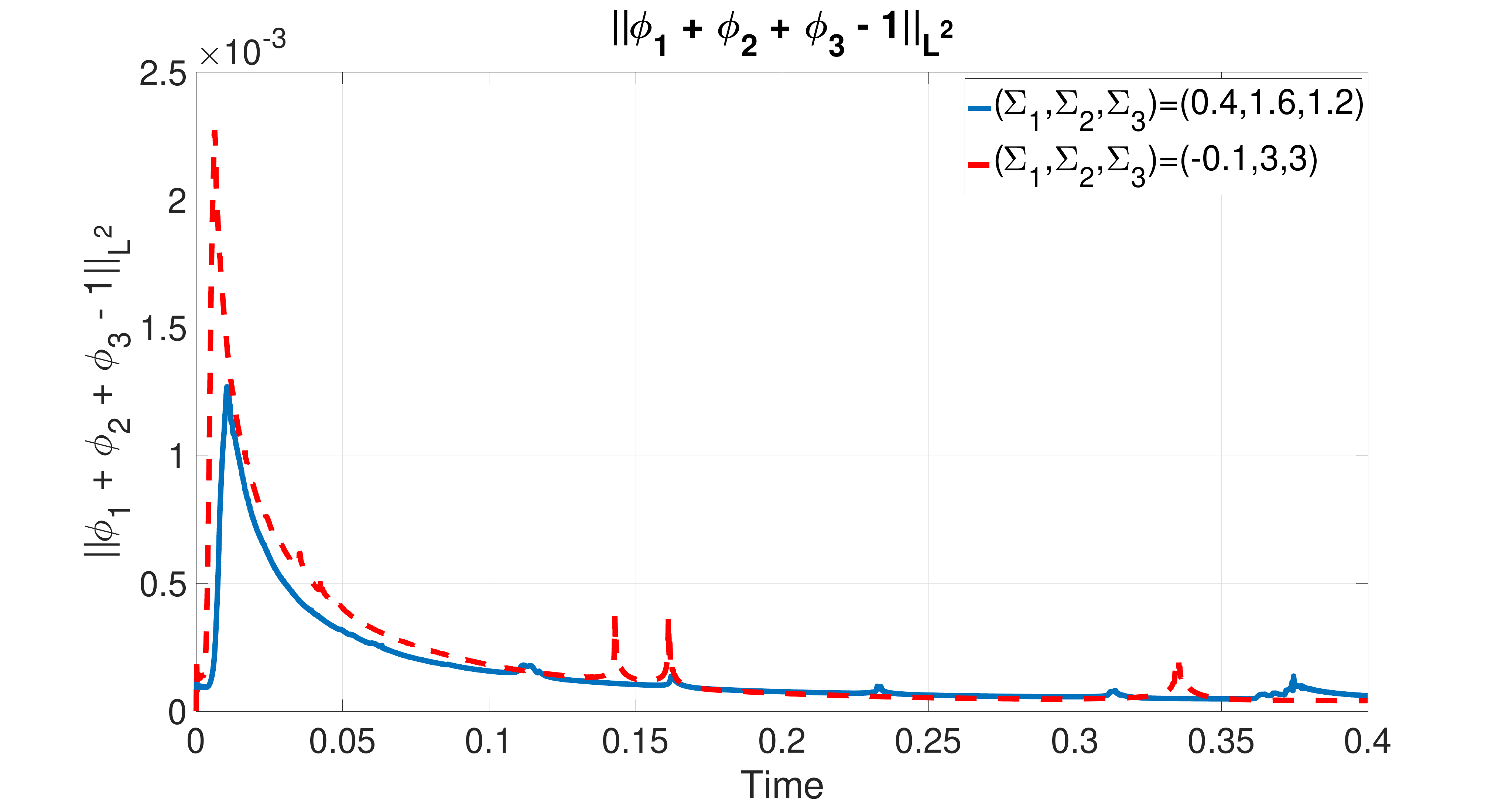}
\end{center}
\caption{Evolution in time of the energies (top left), the volume (top right), $\|\phi_1 + \phi_2 + \phi_3 -1\|_{L^\infty}$ (bottom left), $\|\phi_1 + \phi_2 + \phi_3 -1\|_{L^2}$ (bottom right) for the results presented in Figures~\ref{fig:Spinodal3DDynamicsPartial} and \ref{fig:Spinodal3DDynamicsTotal}.}
\label{fig:Spinodal3DPlots}
\end{figure}

\begin{figure}[h]
\begin{center}
\includegraphics[scale=0.11]{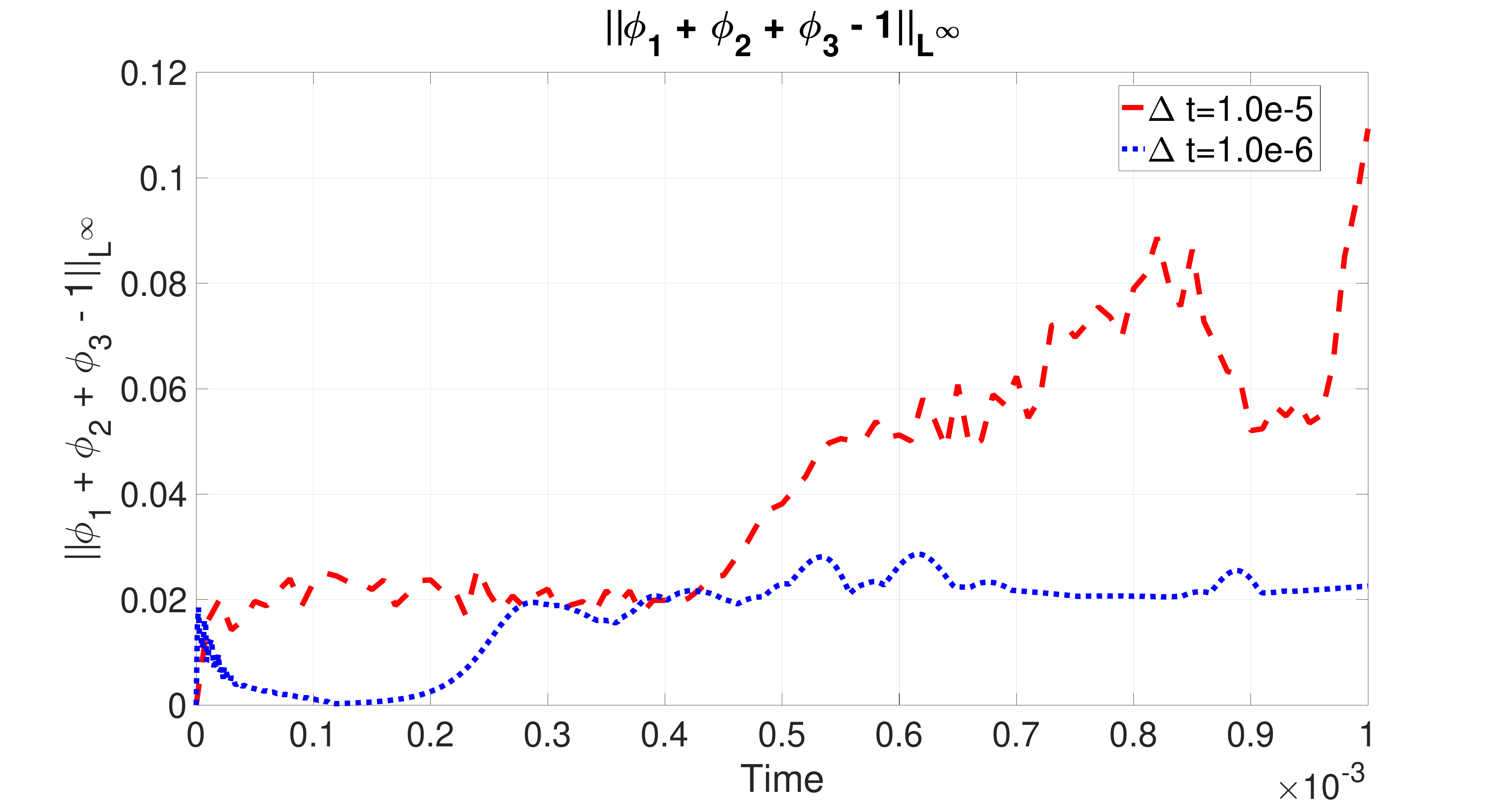}
\includegraphics[scale=0.11]{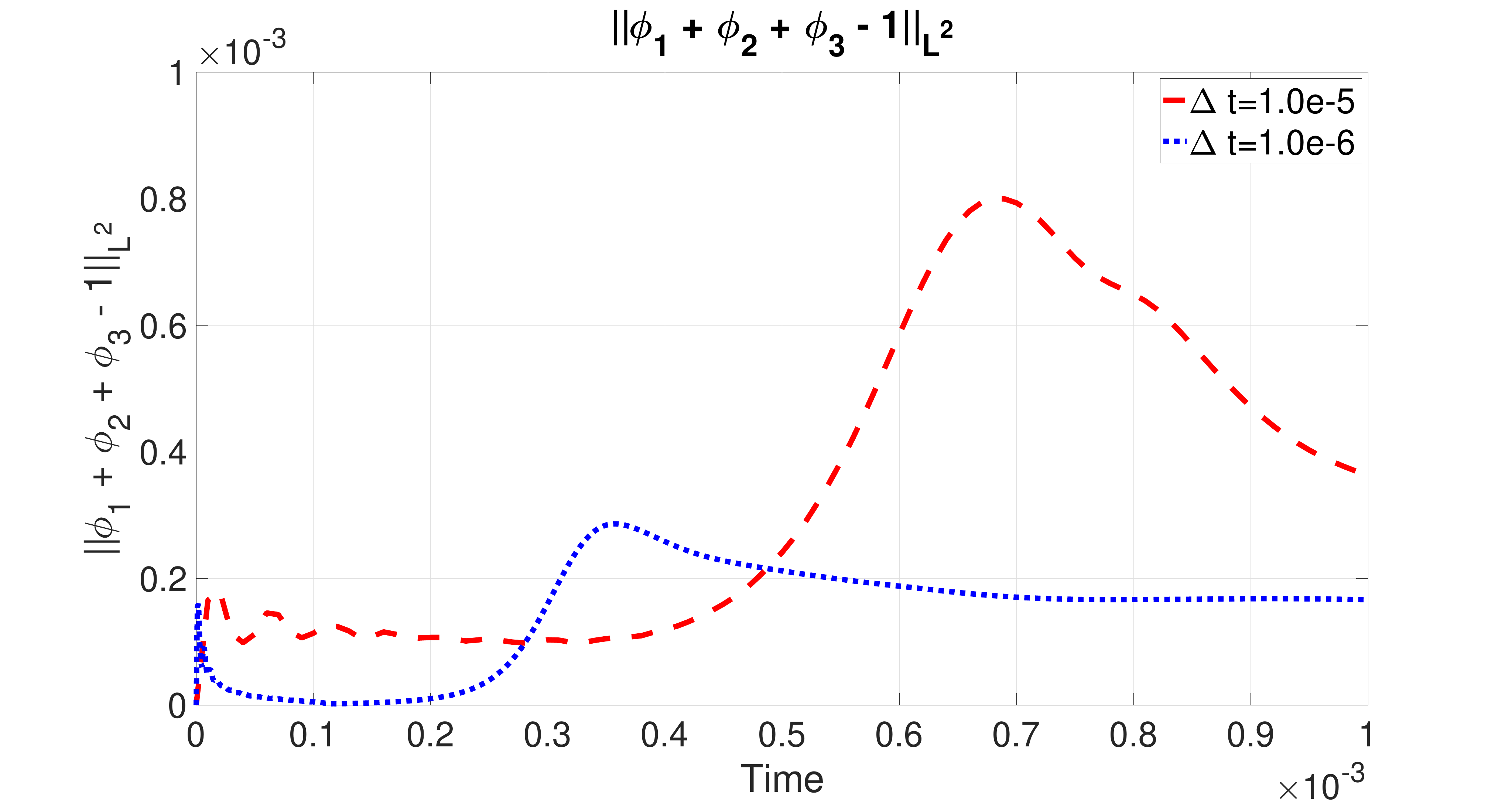}
\end{center}
\caption{Comparison of $\|\phi_1 + \phi_2 + \phi_3 -1\|_{L^\infty}$ (left), $\|\phi_1 + \phi_2 + \phi_3 -1\|_{L^2}$ (right) for time steps $\Delta t=$1e-5 and $\Delta t=$1e-6 with spreading coefficients 
$(\Sigma_1, \Sigma_2 , \Sigma_3) = (-0.1,3,3)$.}
\label{fig:Spinodal3DPlotsdt}
\end{figure}

\subsection{Extension to Navier-Stokes-Cahn-Hilliard. Two Bubbles Suspended in a Third Phase}
{
In this section, we use the model mentioned Section~\ref{sec:NSCH3} and the corresponding extension of the scheme presented in Section~\ref{sec:extensiontoNSCH3} to NTD1, in order to simulate a mixture of three different newtonian fluids with equal constant viscosity $\nu_1=\nu_2=\nu_3=1$. The parameters considered are the same that we used in Section~\ref{sec:2bubbles} (Table~\ref{tab:ballsParameters}) as well as the initial condition presented in \eqref{eq:twoBubblesInitial}, but now the balls representing two droplets of two liquids immersed into a third one. The initial velocity and the boundary conditions are designed to have a rotating effect such that $\u(x,y,0)=\u(x,y,t)|_{\partial\Omega}=(\widehat{u}_1(x,y),\widehat{u}_2(x,y))$ with
\beq\label{eq:intialu}
\left\{\ba{rcl}
\widehat{u}_1(x,y)|_{\partial\Omega}
&=&
2\pi\sin^2\big(4\pi(x-0.125)\big)\cos\big(4\pi(y-0.125)\big)\,,
\\ \hueco
\widehat{u}_2(x,y)|_{\partial\Omega}
&=&
-4\pi\sin\big(4\pi(x-0.125)\big)\cos\big(4\pi(x-0.125)\big)\sin\big(4\pi(y-0.125)\big)\,.
\ea\right.
\eeq
\\
We present in Figure~\ref{fig:BallsFluidsDynamics} the dynamics of the four cases. 
In the top row we present the case $(\Sigma_1, \Sigma_2 , \Sigma_3) = (1,1,1)$  where we observe how the boundary between the red phase ($\phi_1$) and blue phase ($\phi_2$) keeps flat while the droplets are rotating.  In the second row the choice $(\Sigma_1, \Sigma_2 , \Sigma_3) = (0.4, 1.6, 1.2)$ is presented where the expected asymmetric interface is also rotated. In the third row we take $(\Sigma_1, \Sigma_2 , \Sigma_3) = (3,3,-0.1)$ and we impose Dirichlet boundary condition $\phi_3|_{\partial\Omega}=1$ to prevent the droplets to attach to the boundary. This case is interesting because the negativity of $\Sigma_3$ prevents the droplets to touch each other while rotating. Finally, the choice $(\Sigma_1, \Sigma_2 , \Sigma_3) = (-0.1,3,3)$ is presented in the bottom row, where we have also considered the Dirichlet boundary condition $\phi_3|_{\partial\Omega}=1$ and we moved the droplets in the initial condition to the left, to be sure that the dynamics just happen away from the center of the domain. We can see how the dynamics resembles the case without fluid (the red component tends to 'engulf' the blue component) but now the droplets rotates at the same time.
\\
In Figure~\ref{fig:BallsFluidsPlots} we present only the evolution in time of the $L^2$ and $L^\infty$ norms of the restriction $\Sigma_{i=1}^3\phi_i - 1$ (due to the forcing in the boundary the energy is not expected to be decreasing anymore and the volume is conserved as in previous examples). As before, the $L^2$ norm of the restriction is of order $10^{-3}$ and the $L^\infty$ norm is of order $10^{-2}$ while there are points of the domain where the three components interact. Moreover there are some peaks for the total spreading case with $(\Sigma_1, \Sigma_2 , \Sigma_3) = (3,3,-0.1)$ which corresponds with the droplets touching the boundaries where we have set $\phi_3=1$.
\\
All the obtained dynamics seems reasonable and not affected by the non-exact conservation of the restriction even taking into account fluid effects, inducing to conclude that the presented model and schemes can be considered for developing numerical approximations of mixtures of fluids.
}

%

\begin{figure}[h]
\begin{center}
\includegraphics[scale=0.09]{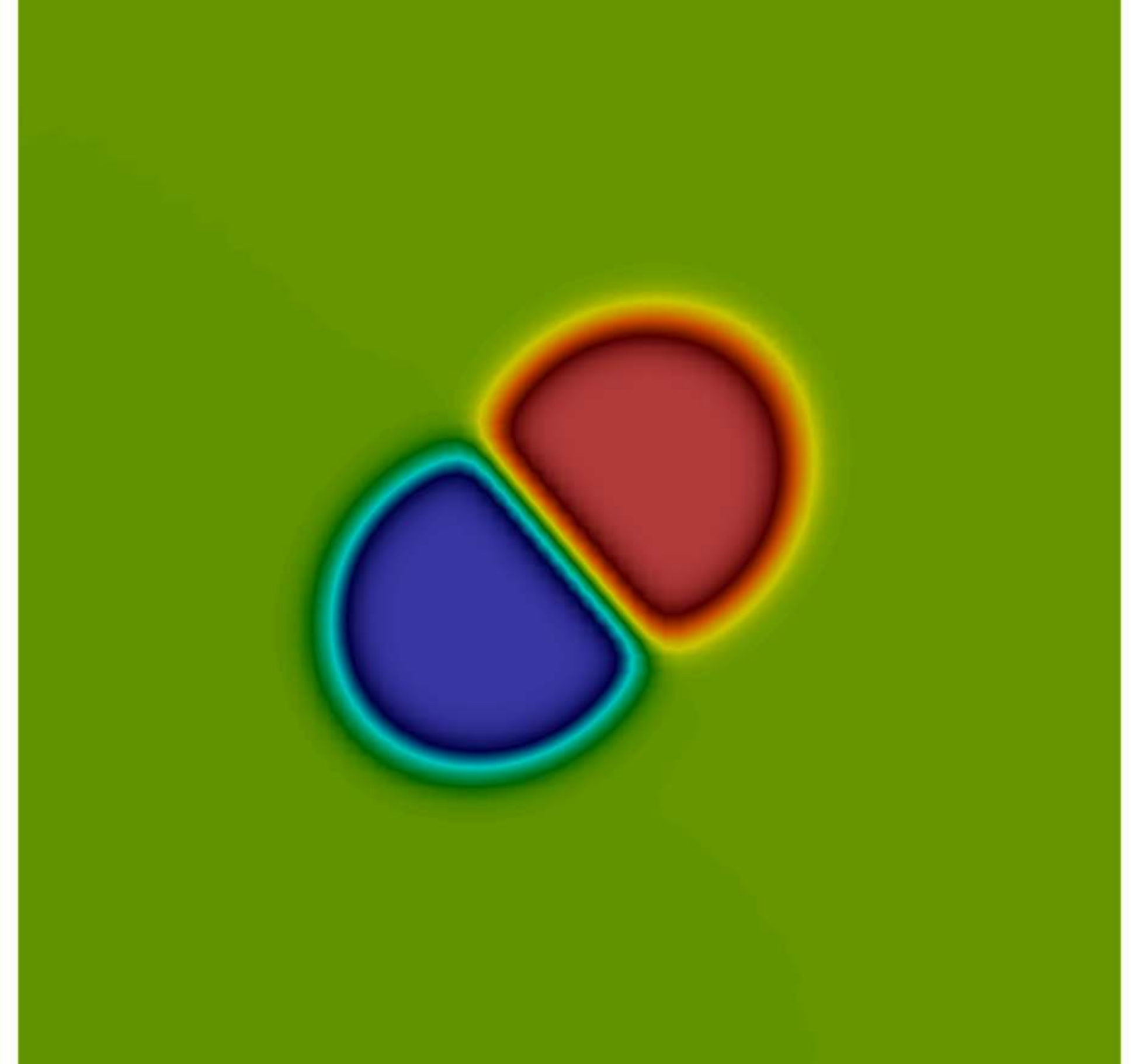}
\includegraphics[scale=0.09]{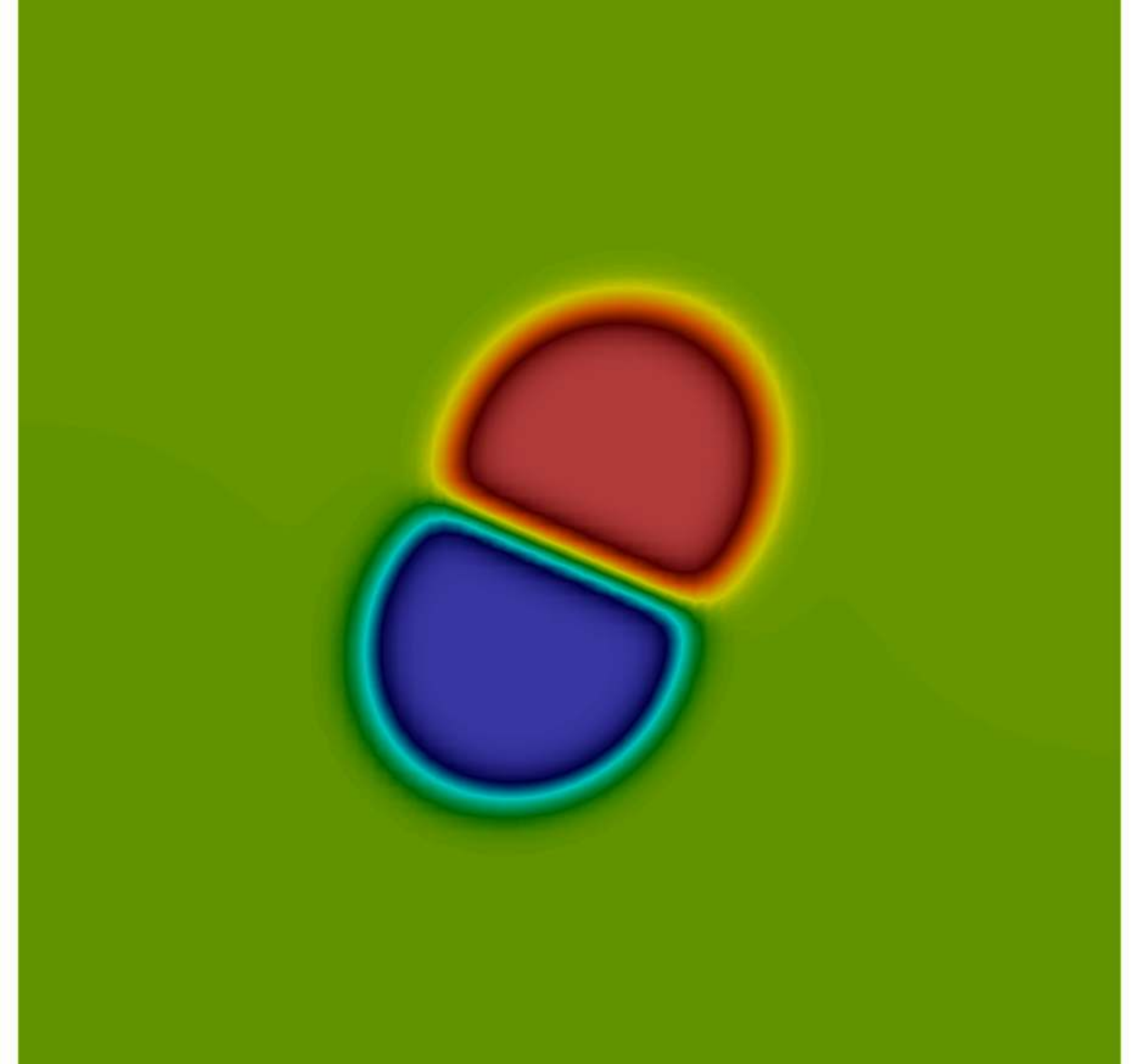}
\includegraphics[scale=0.09]{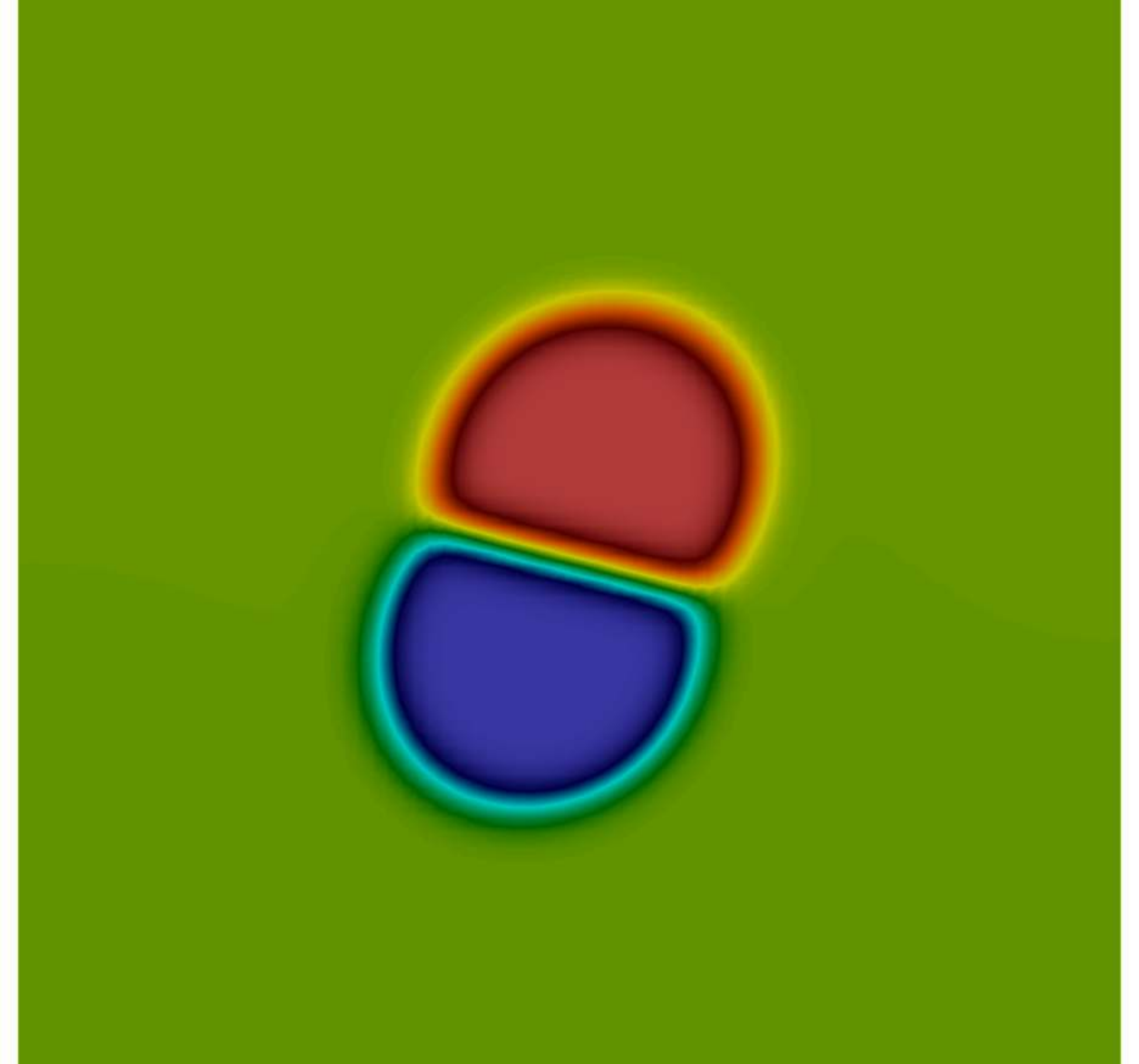}
\includegraphics[scale=0.09]{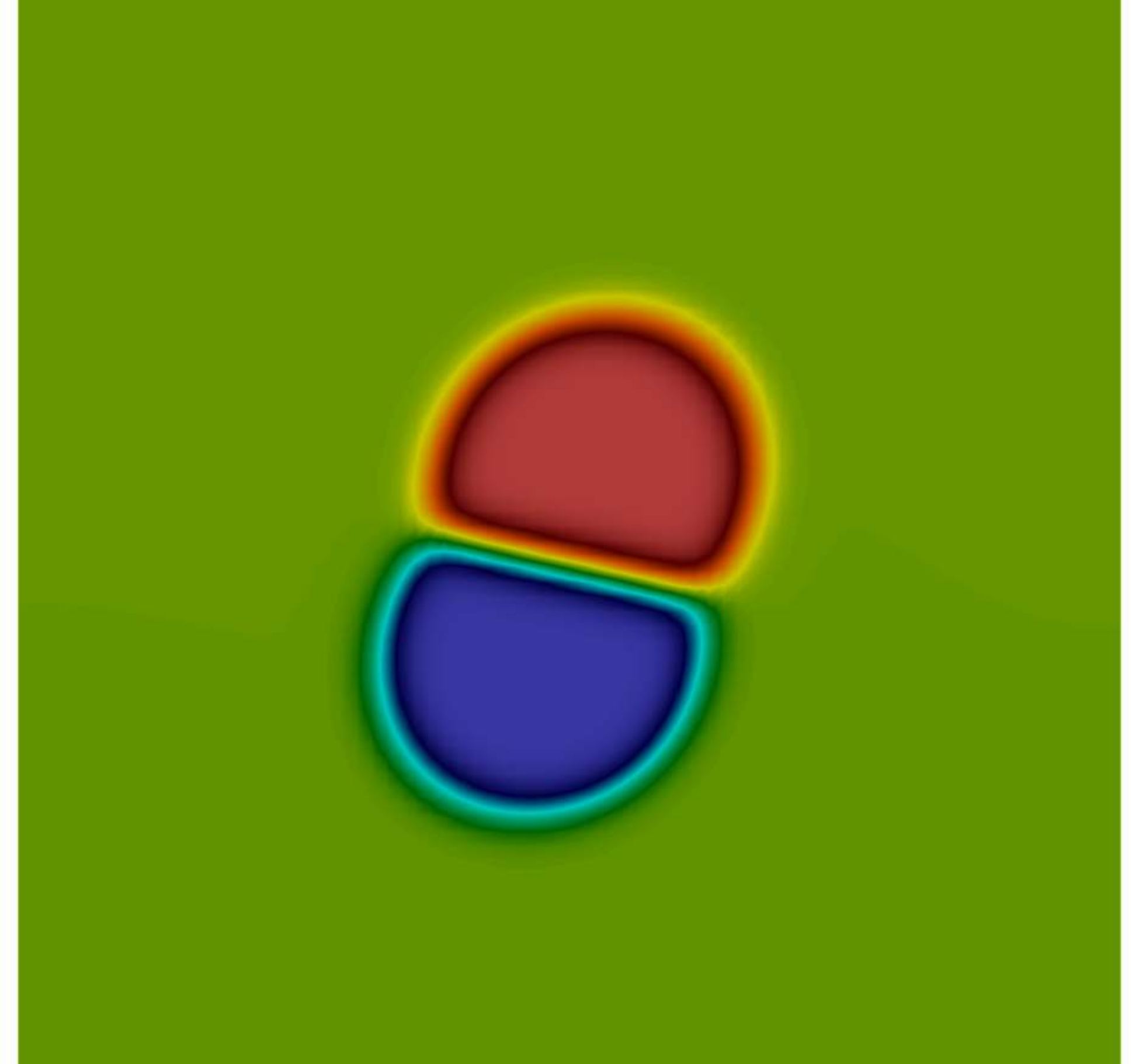}
\includegraphics[scale=0.09]{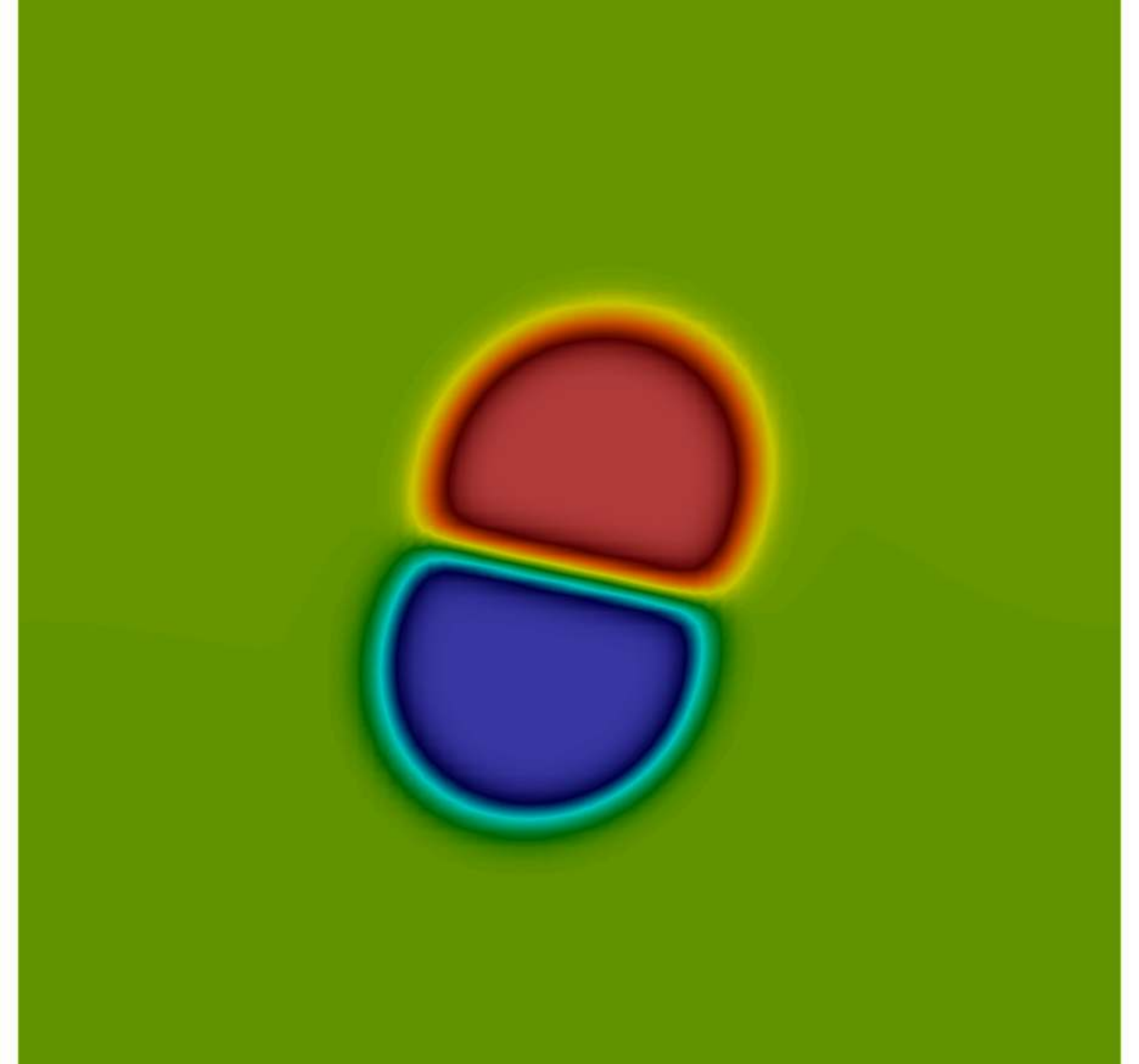}
\\ [1ex]
\includegraphics[scale=0.09]{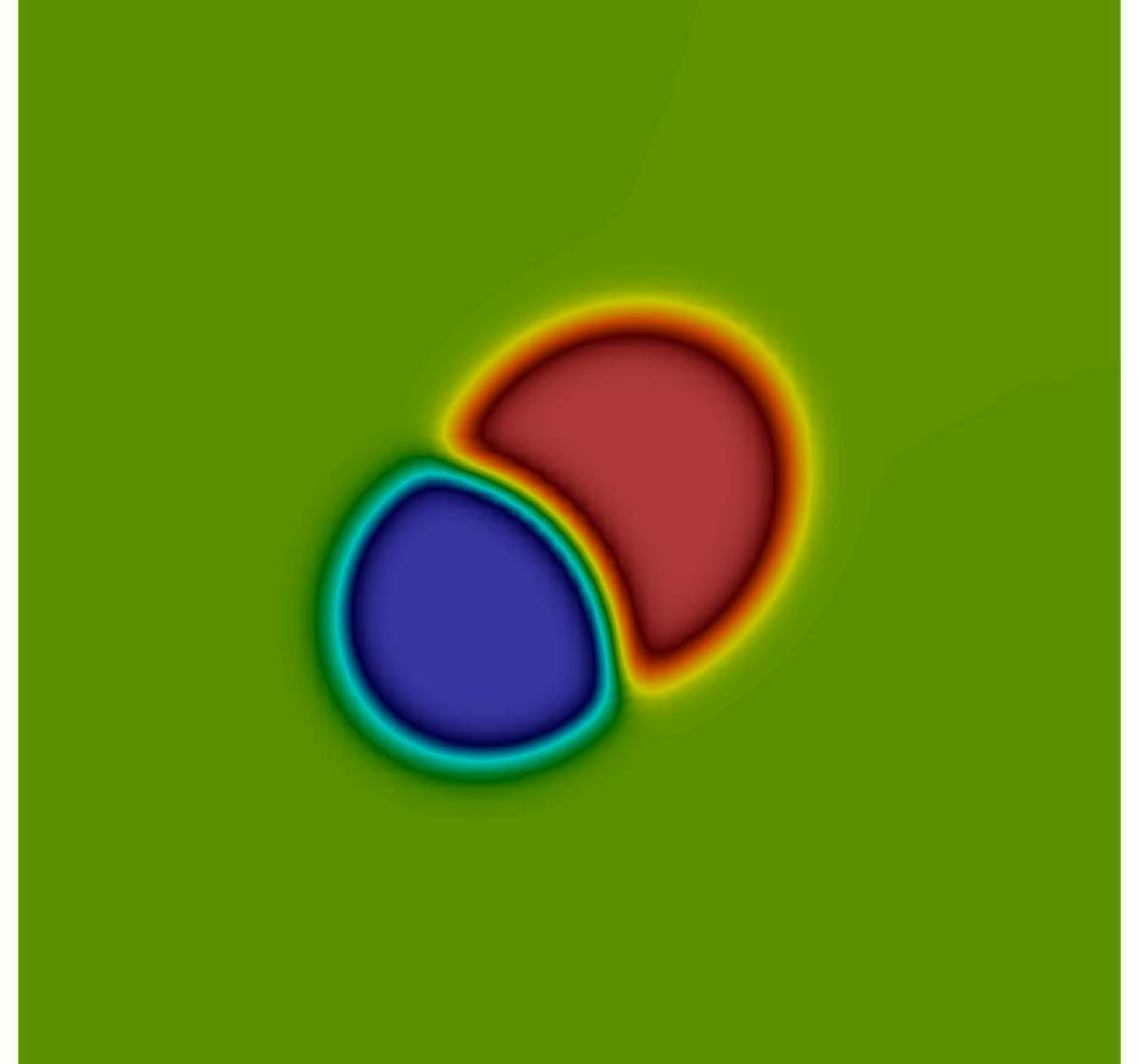}
\includegraphics[scale=0.09]{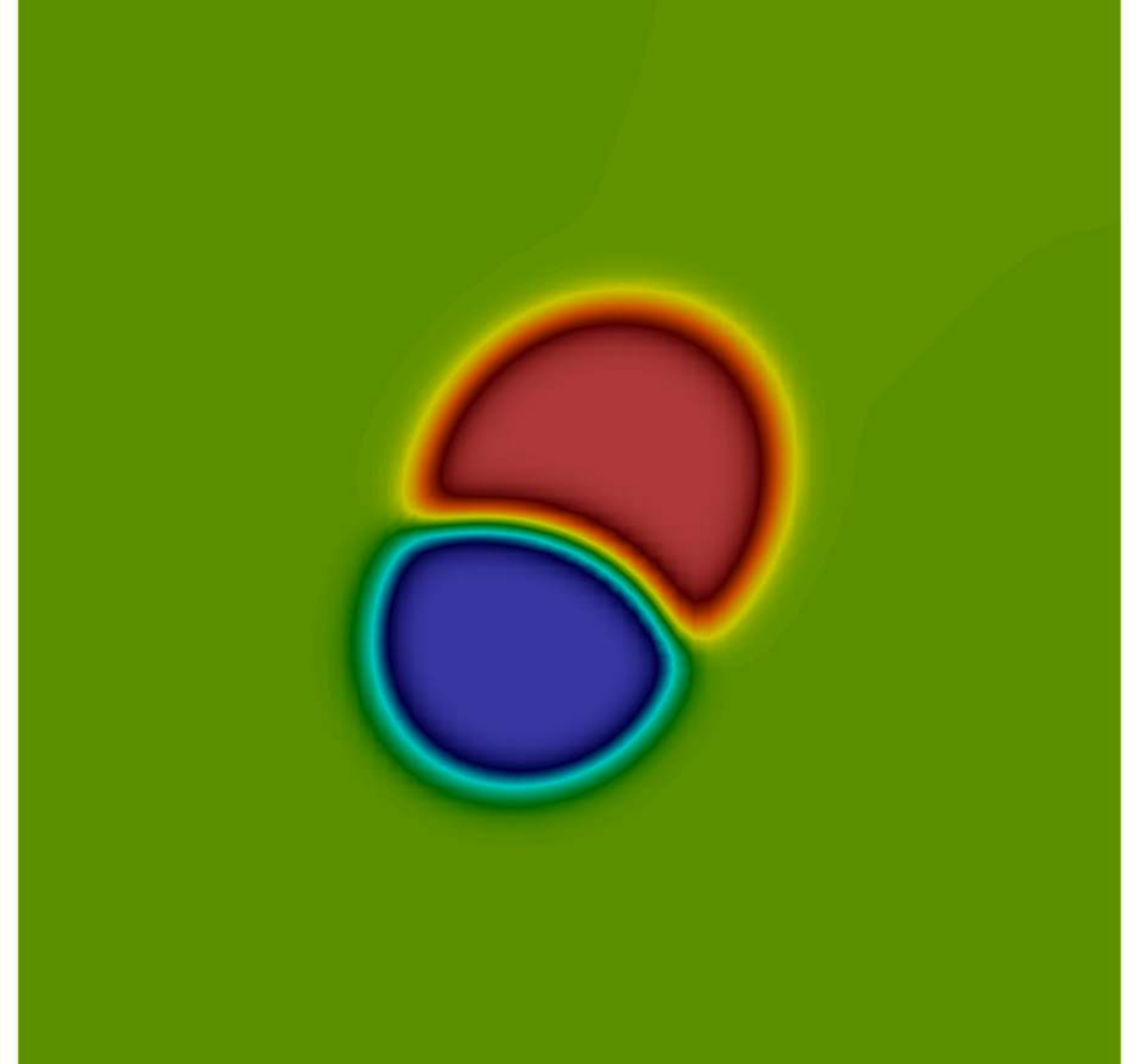}
\includegraphics[scale=0.09]{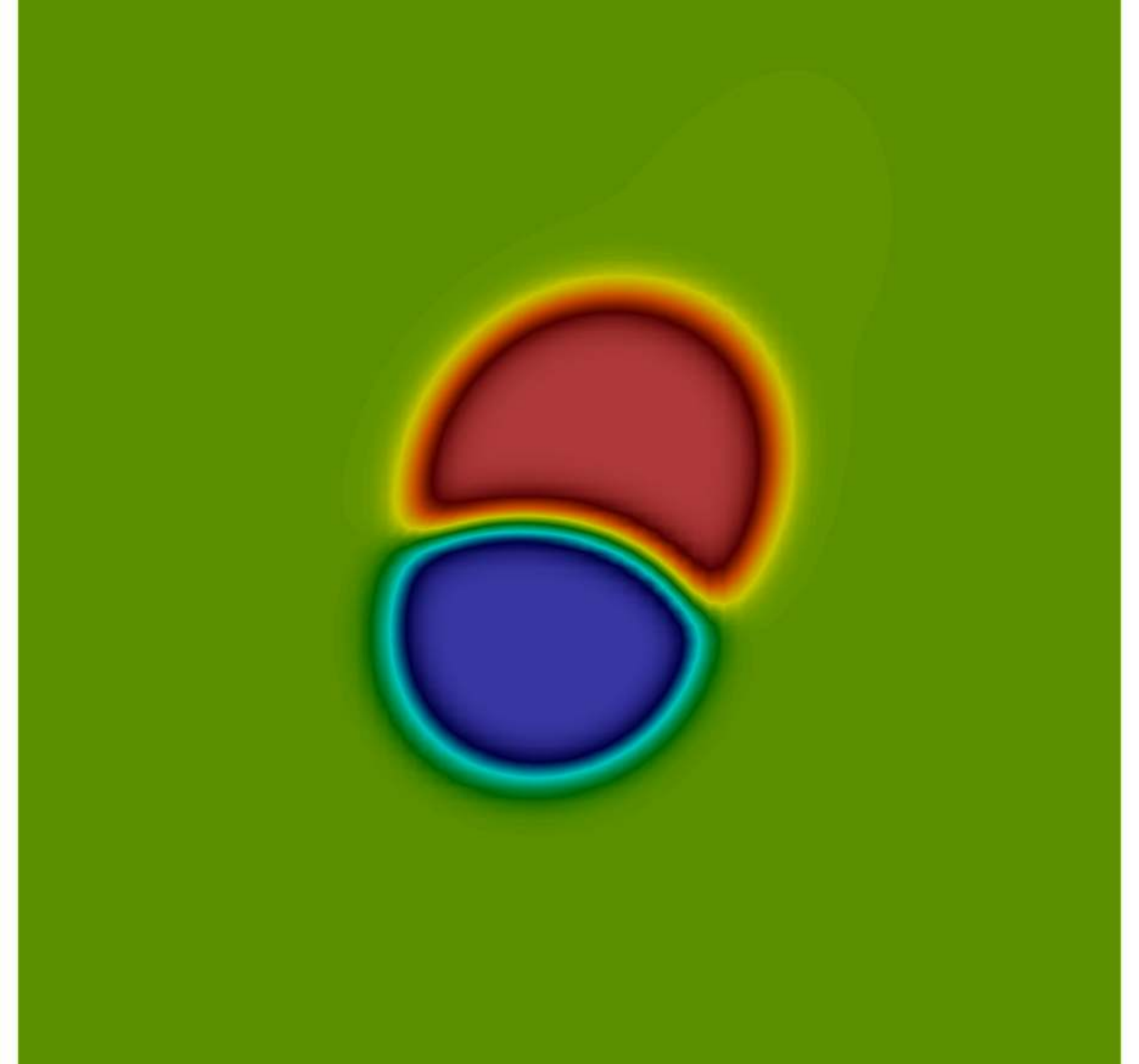}
\includegraphics[scale=0.09]{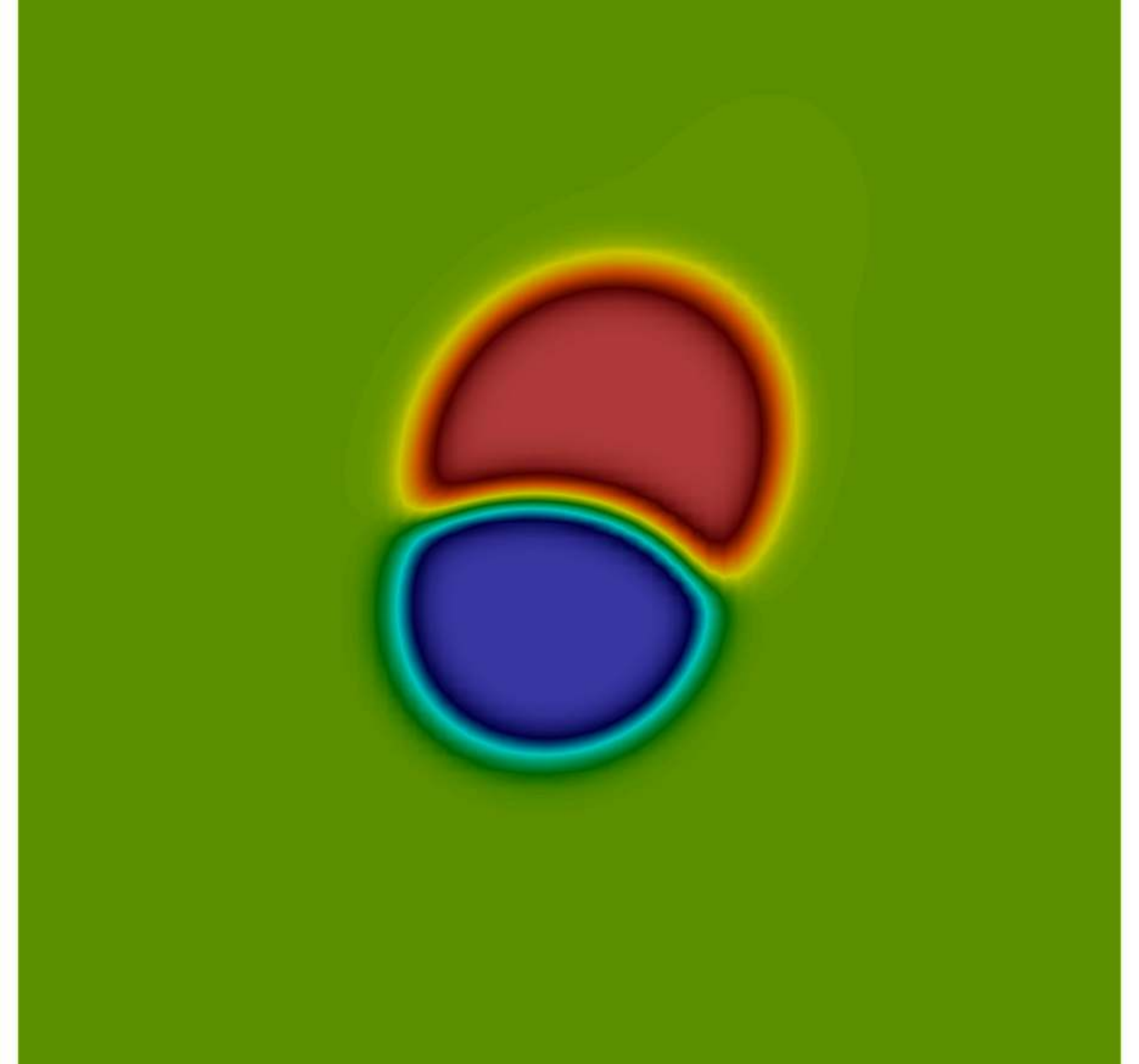}
\includegraphics[scale=0.09]{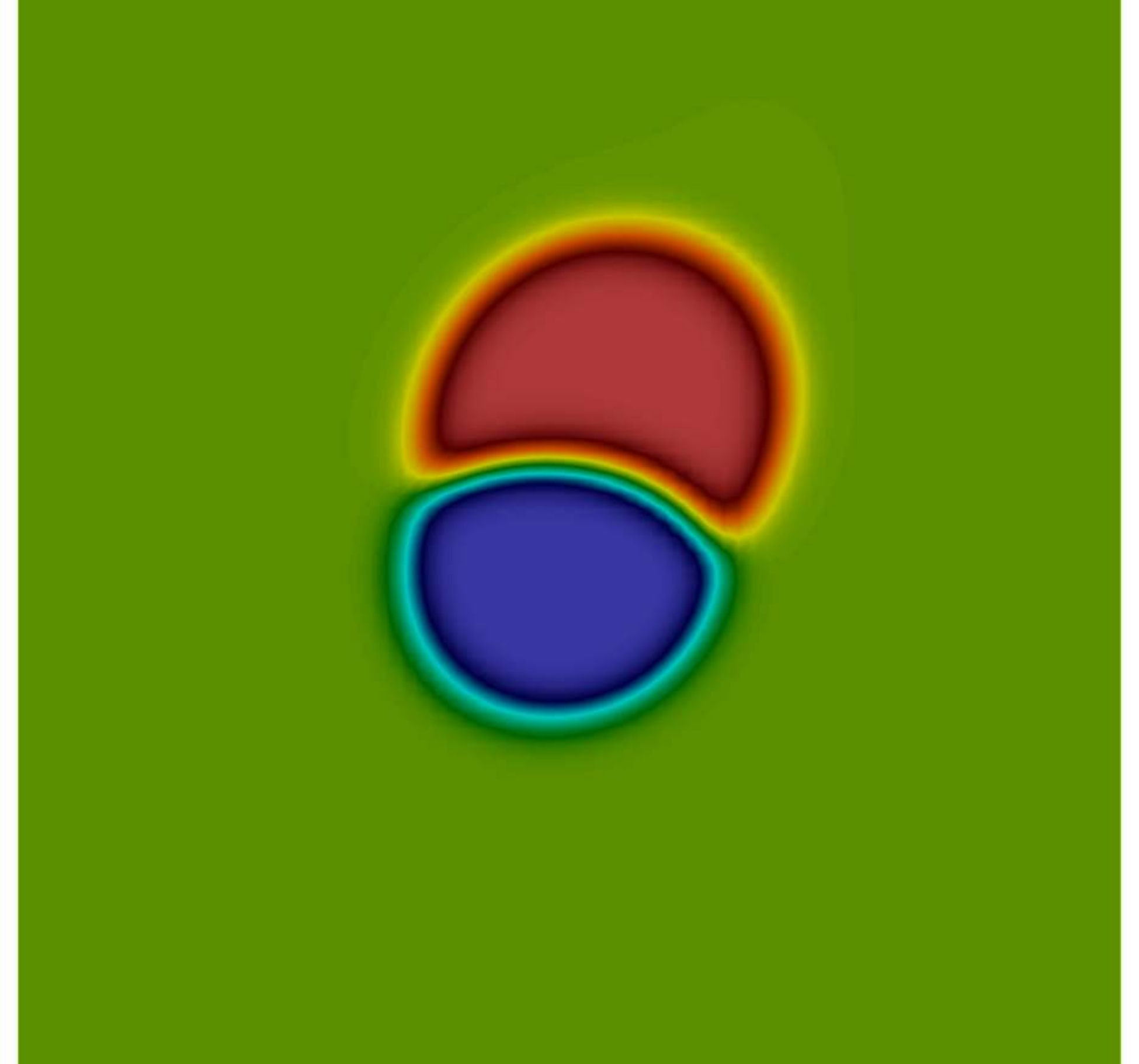}
\\ [1ex]
\includegraphics[scale=0.09]{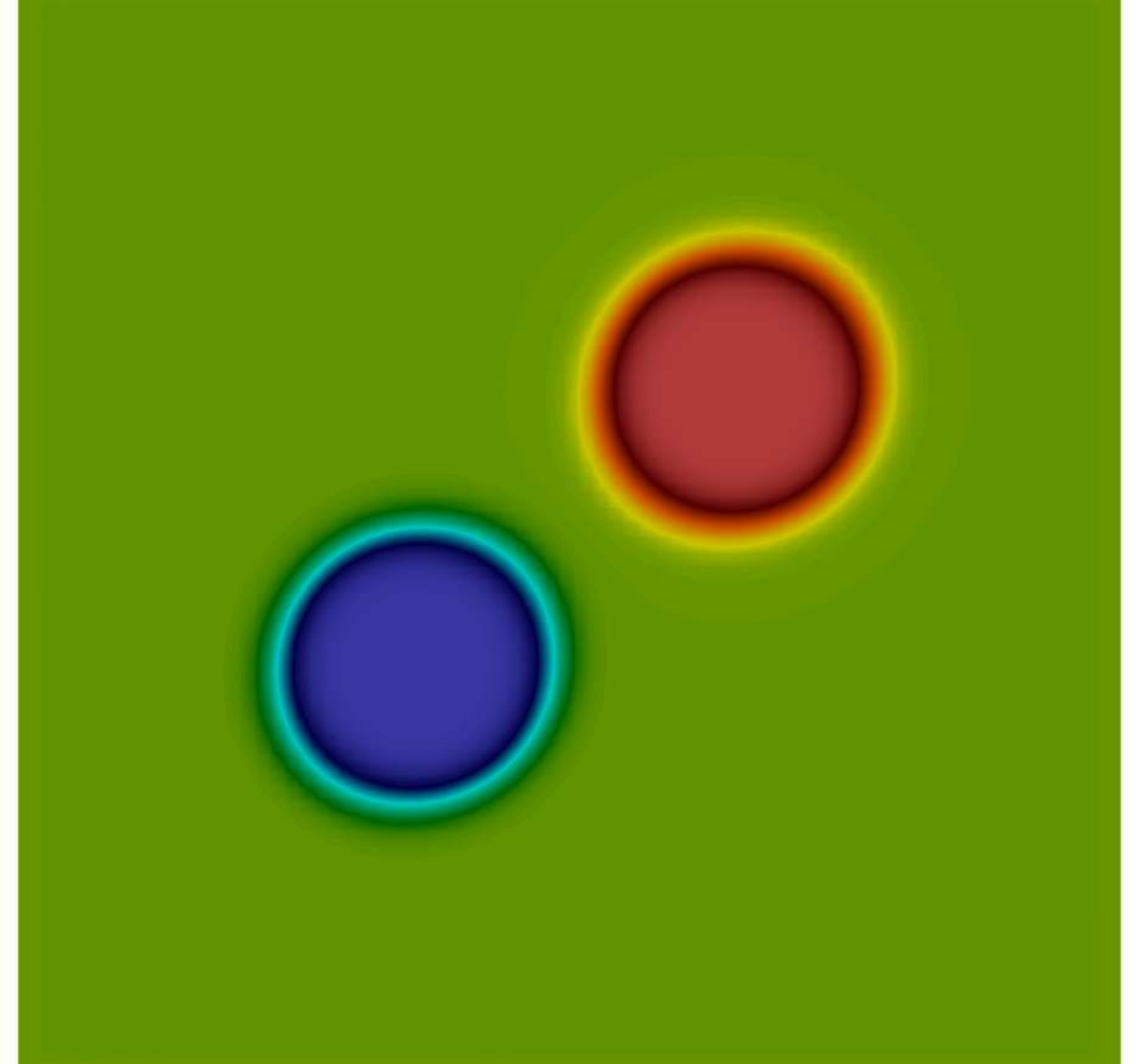}
\includegraphics[scale=0.09]{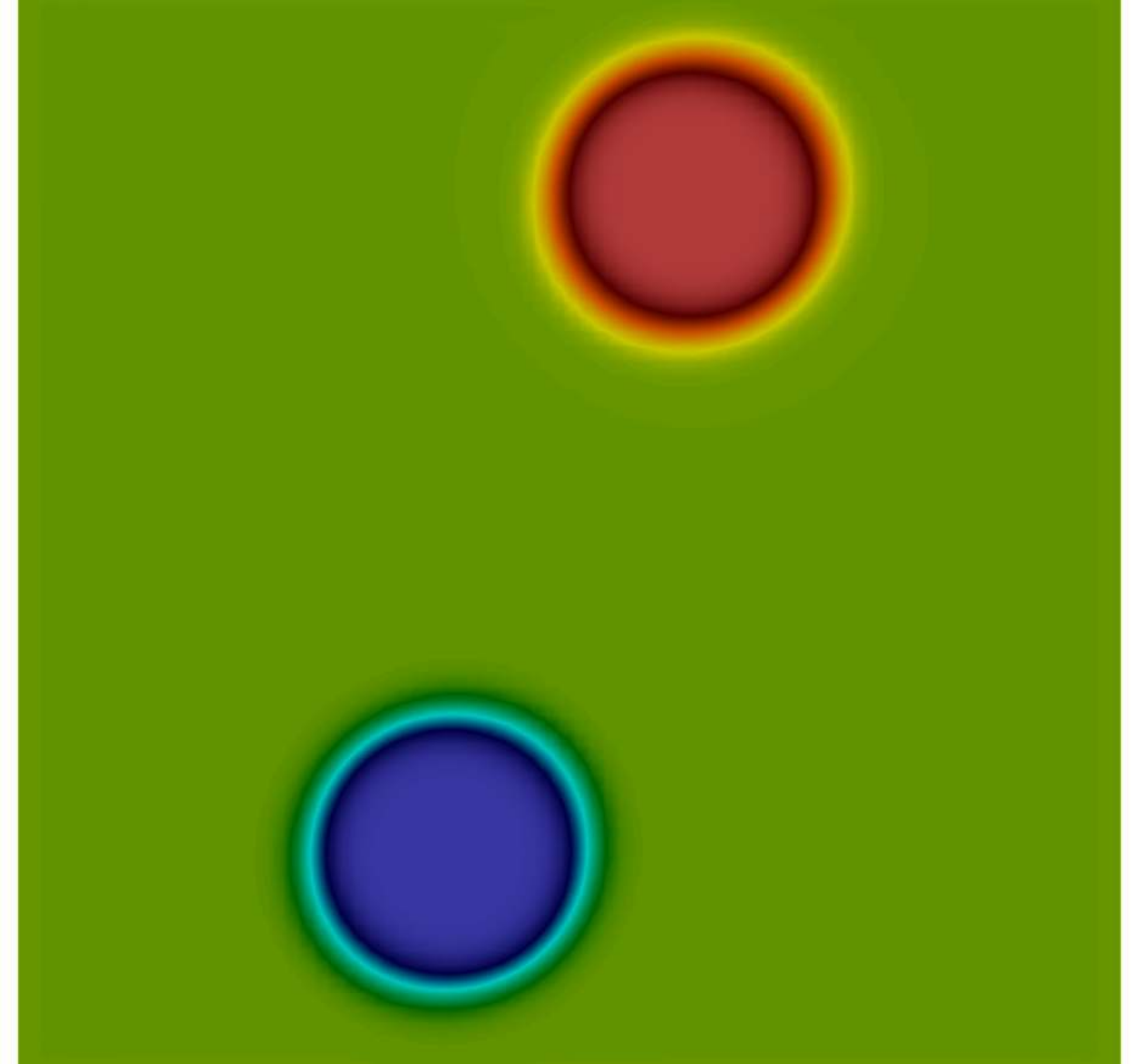}
\includegraphics[scale=0.09]{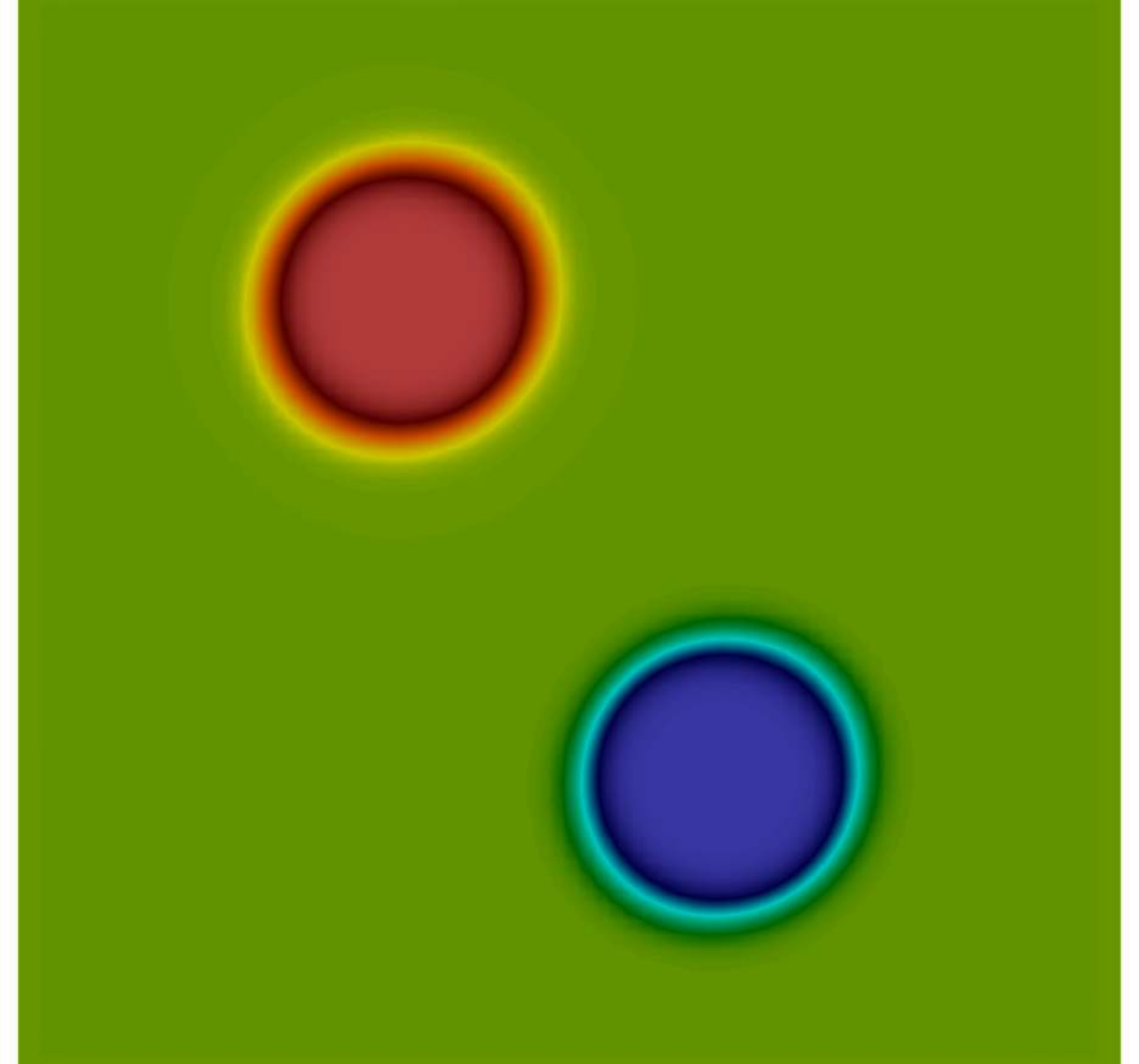}
\includegraphics[scale=0.09]{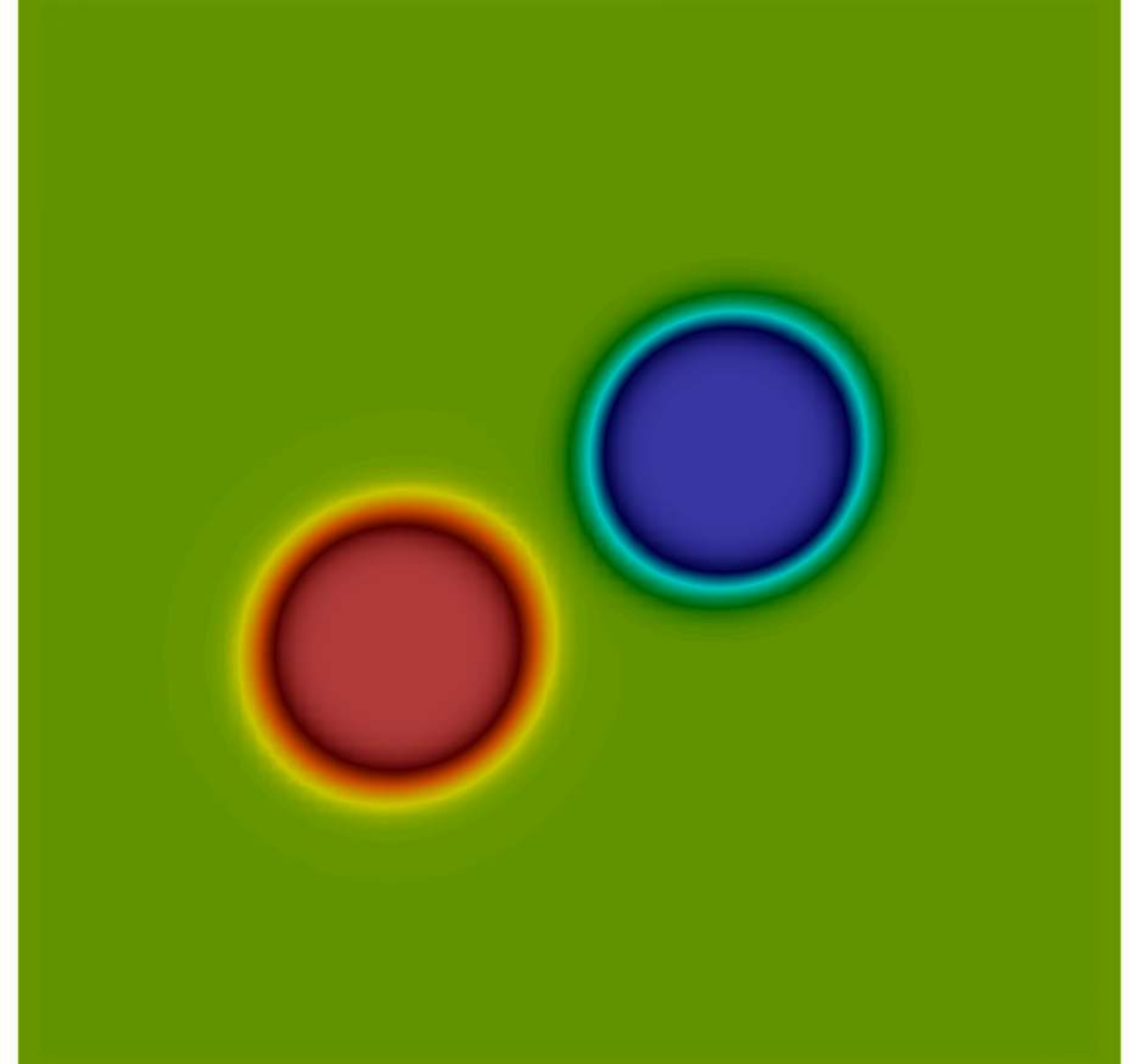}
\includegraphics[scale=0.09]{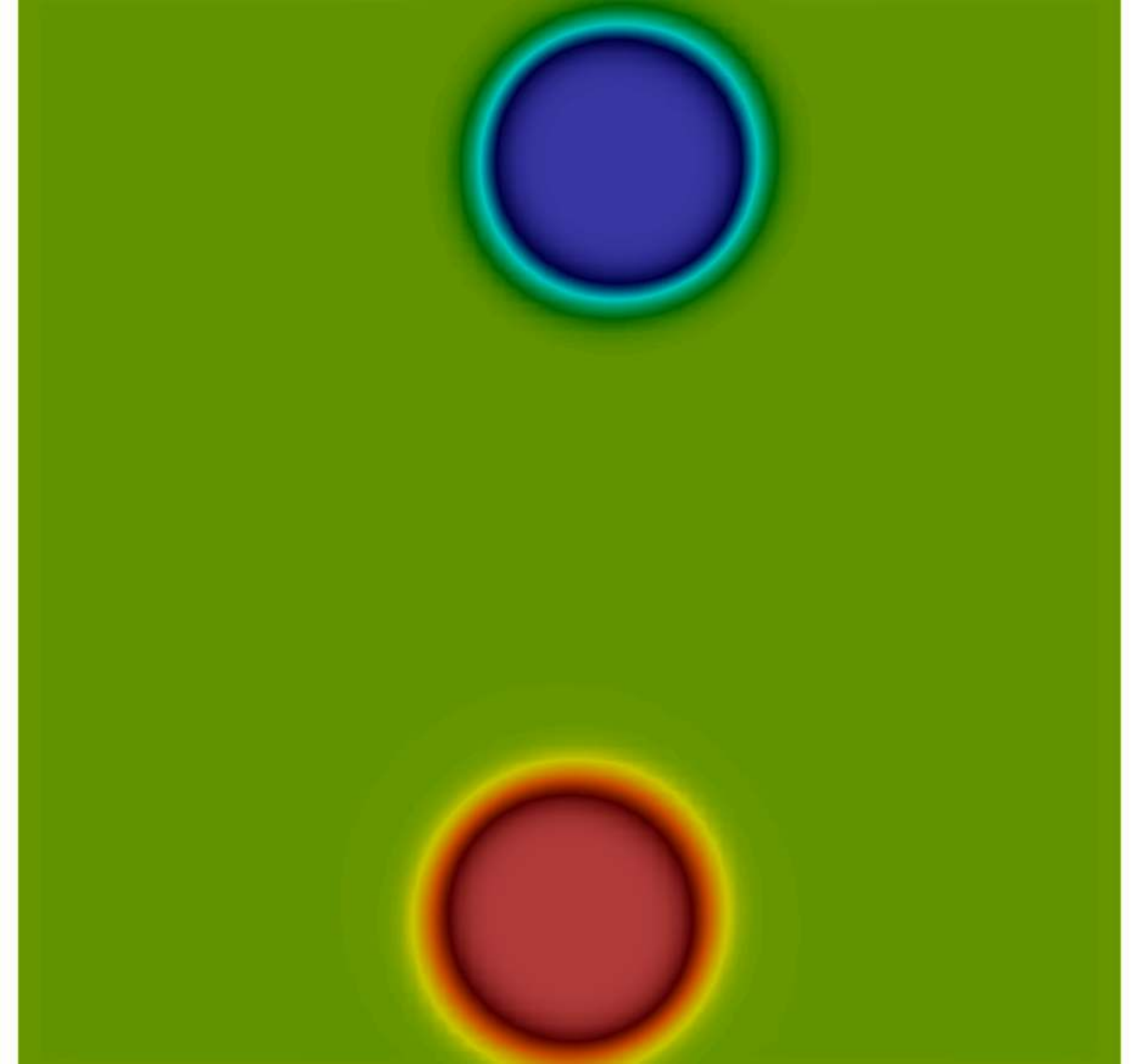}
\\ [1ex]
\includegraphics[scale=0.09]{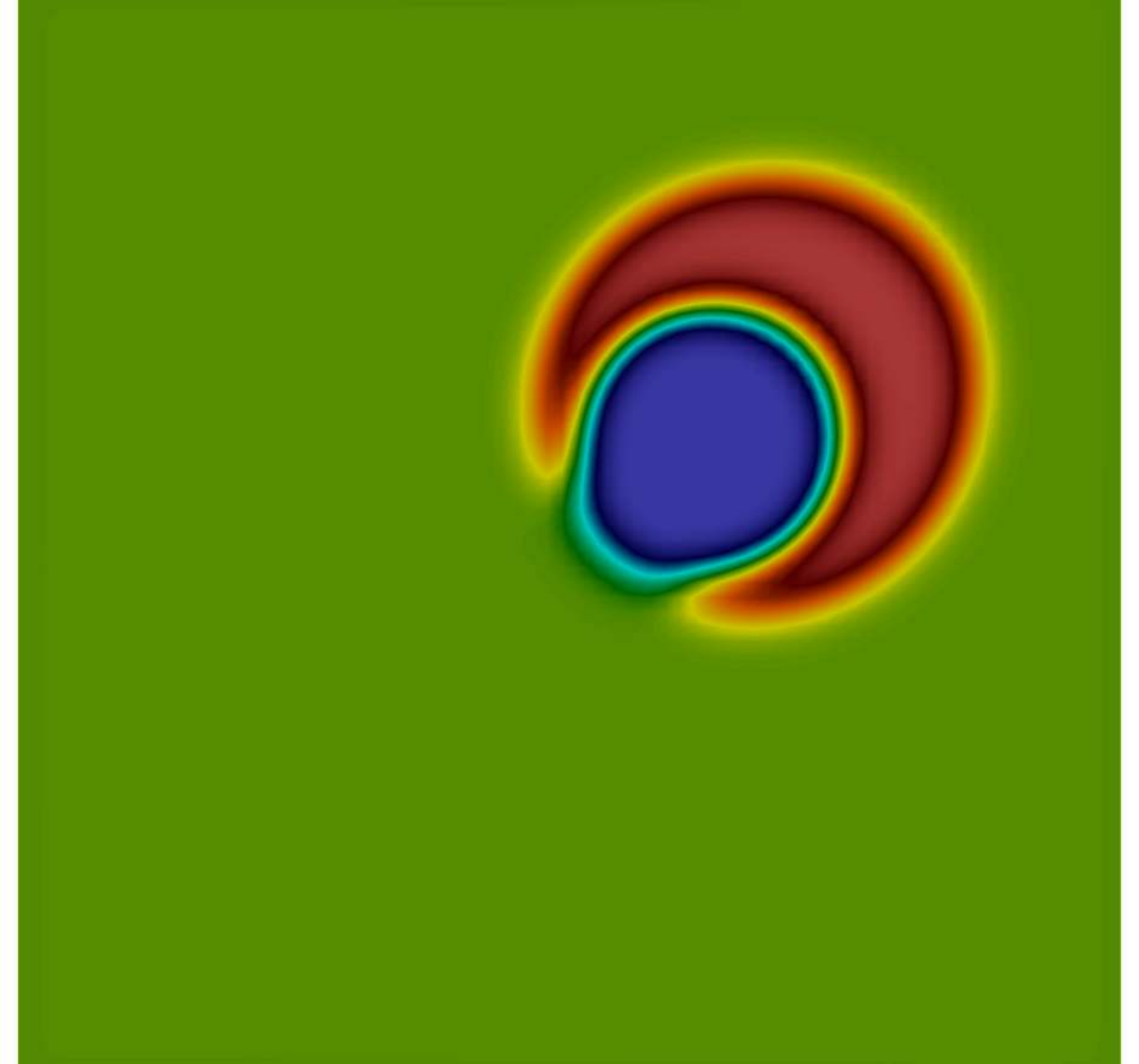}
\includegraphics[scale=0.09]{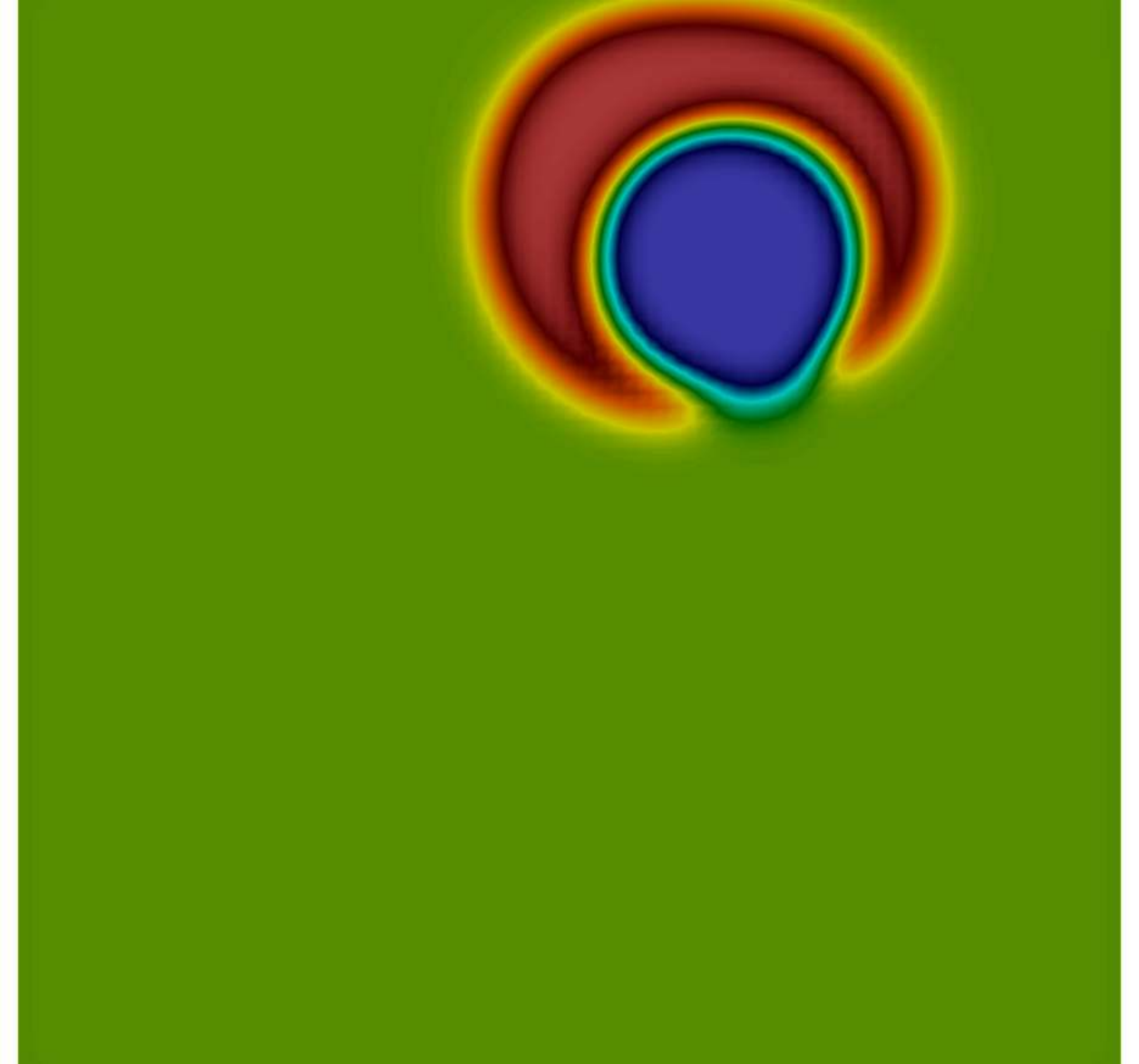}
\includegraphics[scale=0.09]{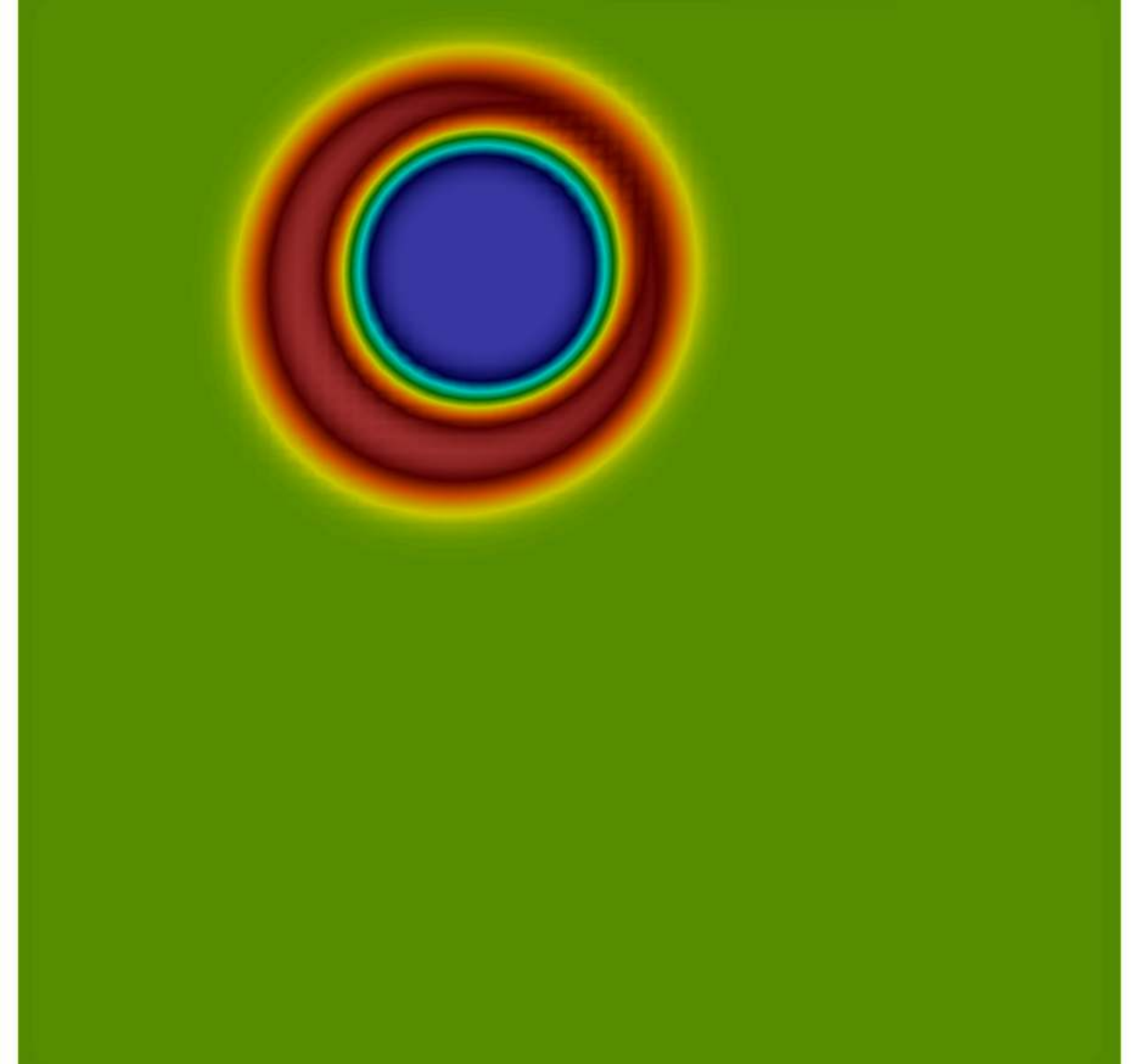}
\includegraphics[scale=0.09]{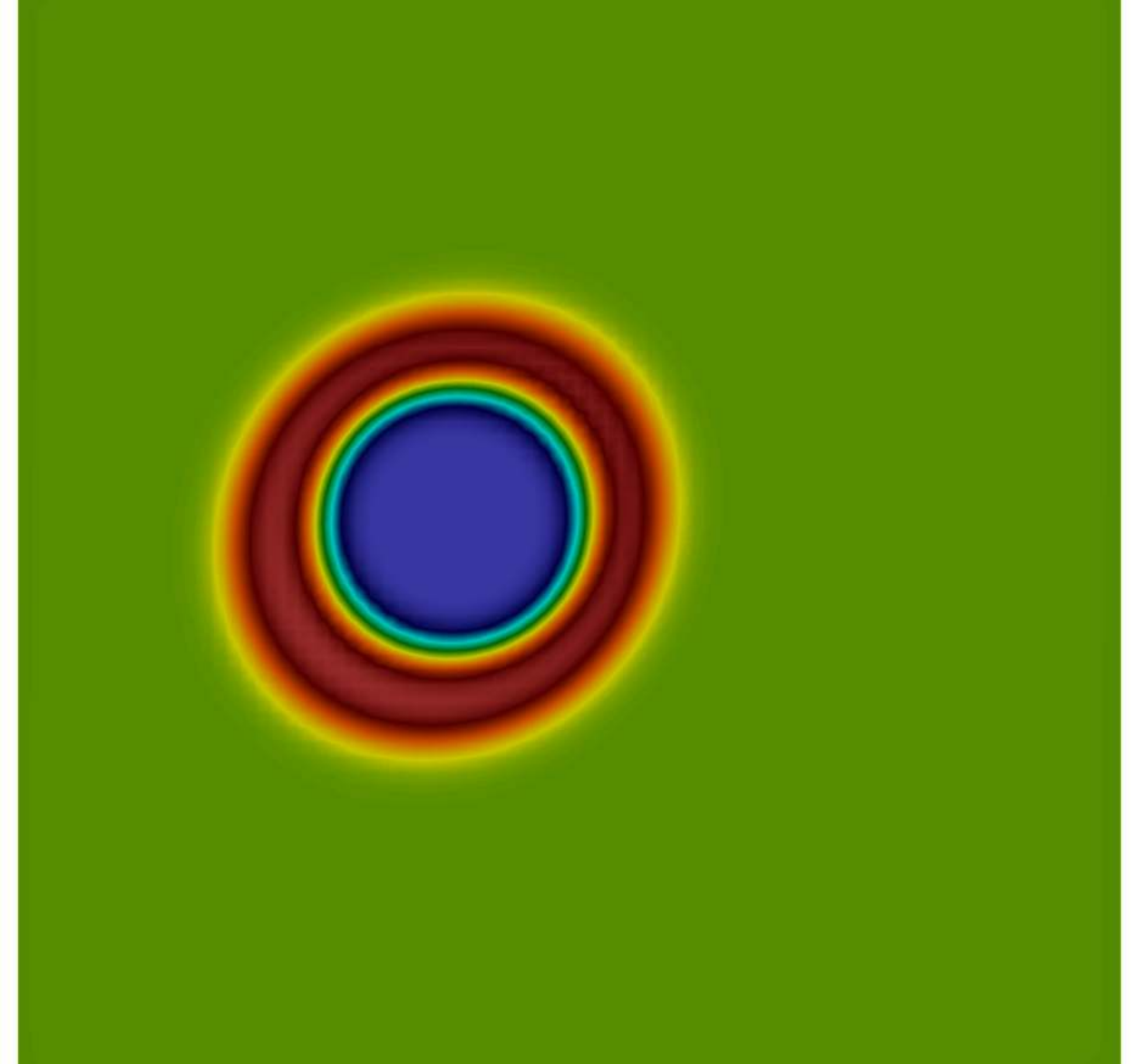}
\includegraphics[scale=0.09]{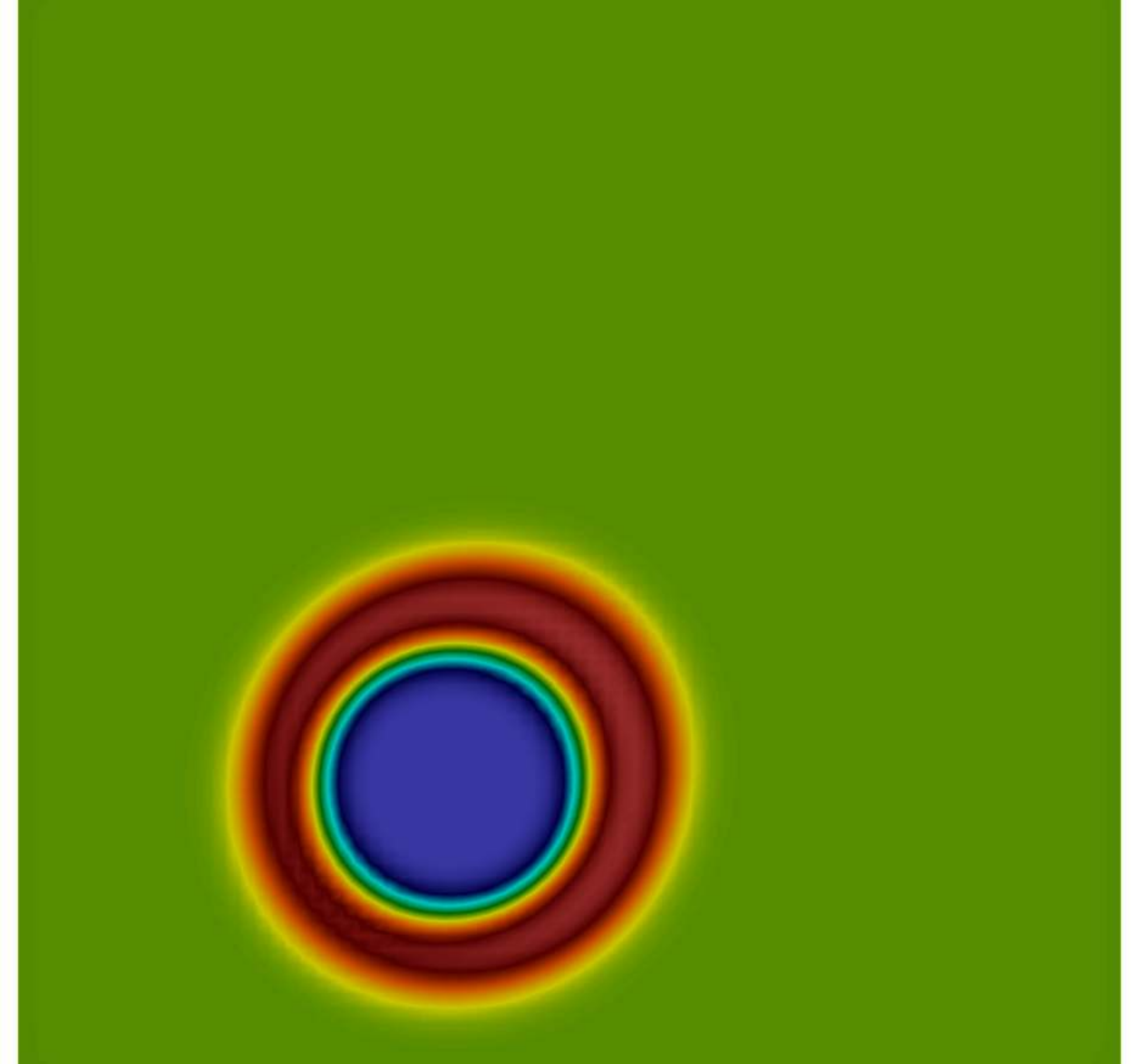}
\end{center}
\caption{Dynamics of scheme NTD1 at times $t=0.1, 0.25, 0.5 0.75$ and $1$ (from left to right) with spreading coefficients 
$(\Sigma_1, \Sigma_2 , \Sigma_3) = (1,1,1)$ (top row)
$(\Sigma_1, \Sigma_2 , \Sigma_3) = (0.4, 1.6, 1.2)$ (second row)
$(\Sigma_1, \Sigma_2 , \Sigma_3) = (3,3,-0.1)$ (third row)
$(\Sigma_1, \Sigma_2 , \Sigma_3) = (-0.1,3,3)$ (bottom row).}\label{fig:BallsFluidsDynamics}
\end{figure}

\begin{figure}[h]
\begin{center}
\includegraphics[scale=0.11]{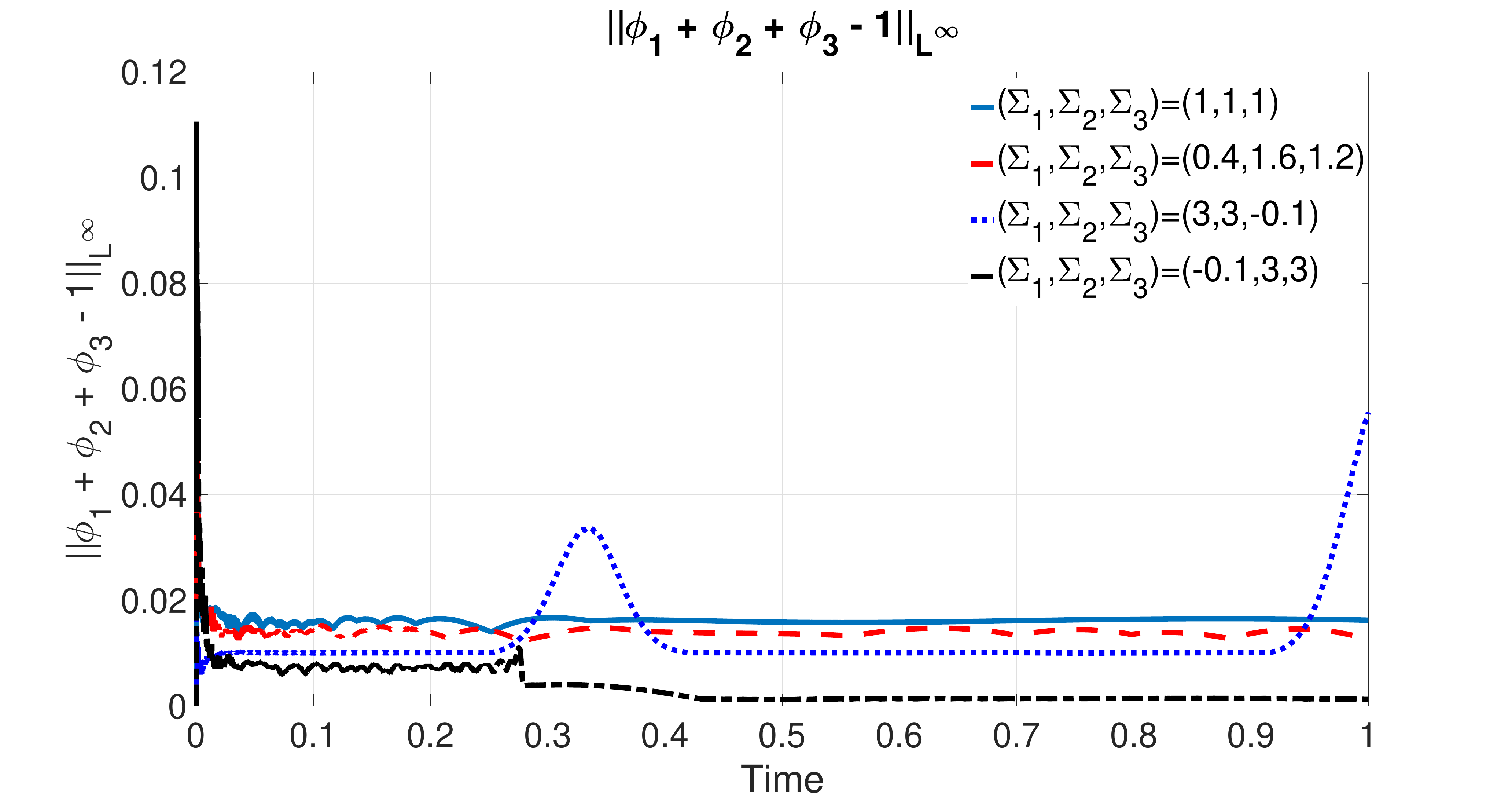}
\includegraphics[scale=0.11]{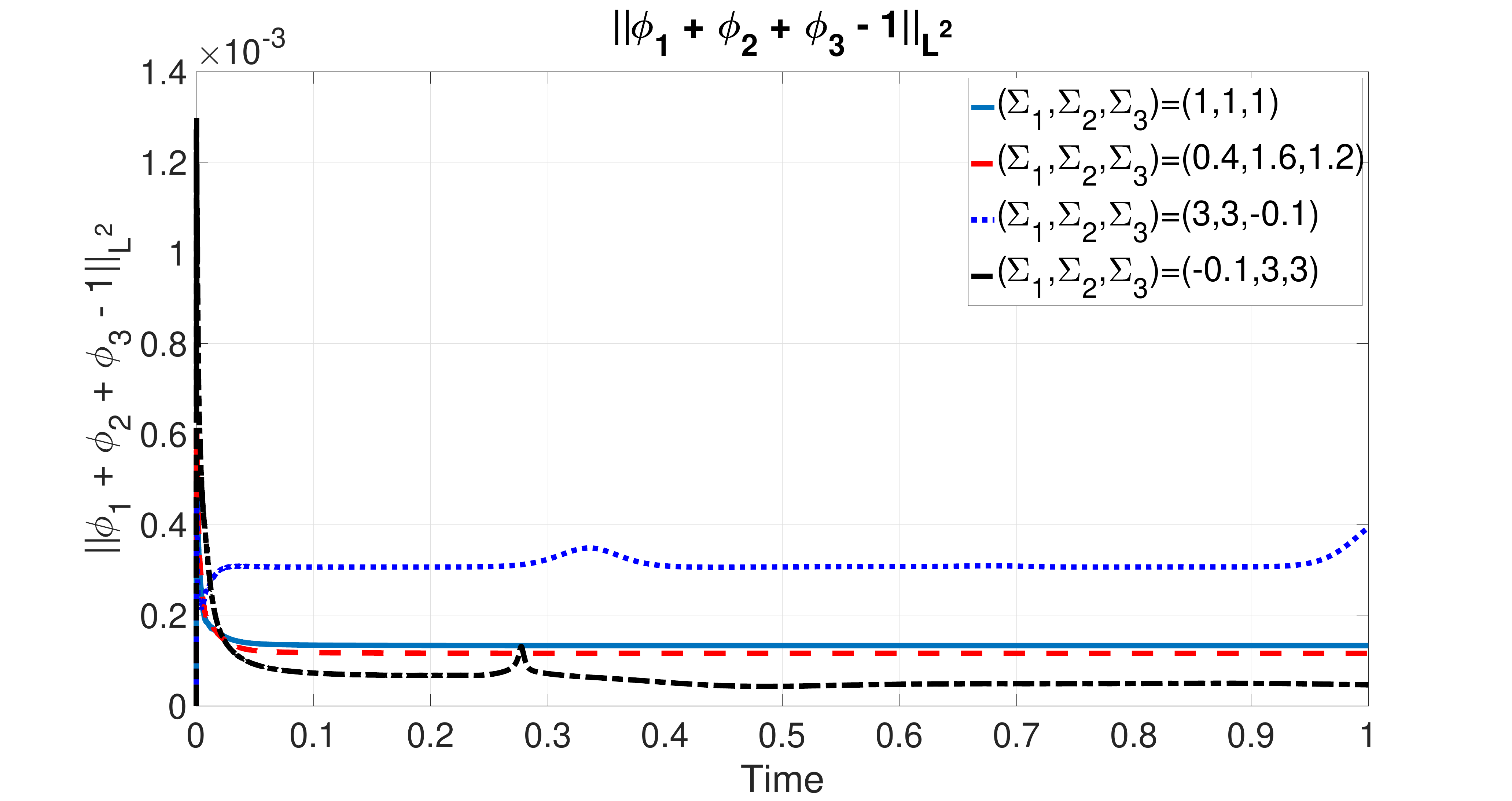}
\end{center}
\caption{Evolution in time of $\|\phi_1 + \phi_2 + \phi_3 -1\|_{L^\infty}$ (left) and $\|\phi_1 + \phi_2 + \phi_3 -1\|_{L^2}$ (right) for the results presented in Figure~\ref{fig:BallsFluidsDynamics}.}
\label{fig:BallsFluidsPlots}
\end{figure}

\subsubsection{Particular case of considering only two components $(\phi_2=0)$}
{
Now we study how consistent is scheme NTD1 with the two components systems when we add hydrodynamics effects to the two balls example presented in section~\ref{sec:twoballsnofluid}. To this end we consider the same initial condition and we present the resulting dynamics in Figure~\ref{fig:BallsFluidsCase0Dyn}. Again some spurious creation of $\phi_2$ occurs in the interface (see Figures~\ref{fig:BallsFluidsCase0NTD1phi2} and In Figure~\ref{fig:BallsFluidsCase0NTD1phi2maxmin} ), but again this fact does not seem to prevent the system to achieve the expected dynamics
}

\begin{figure}[h]
\begin{center}
\includegraphics[scale=0.09]{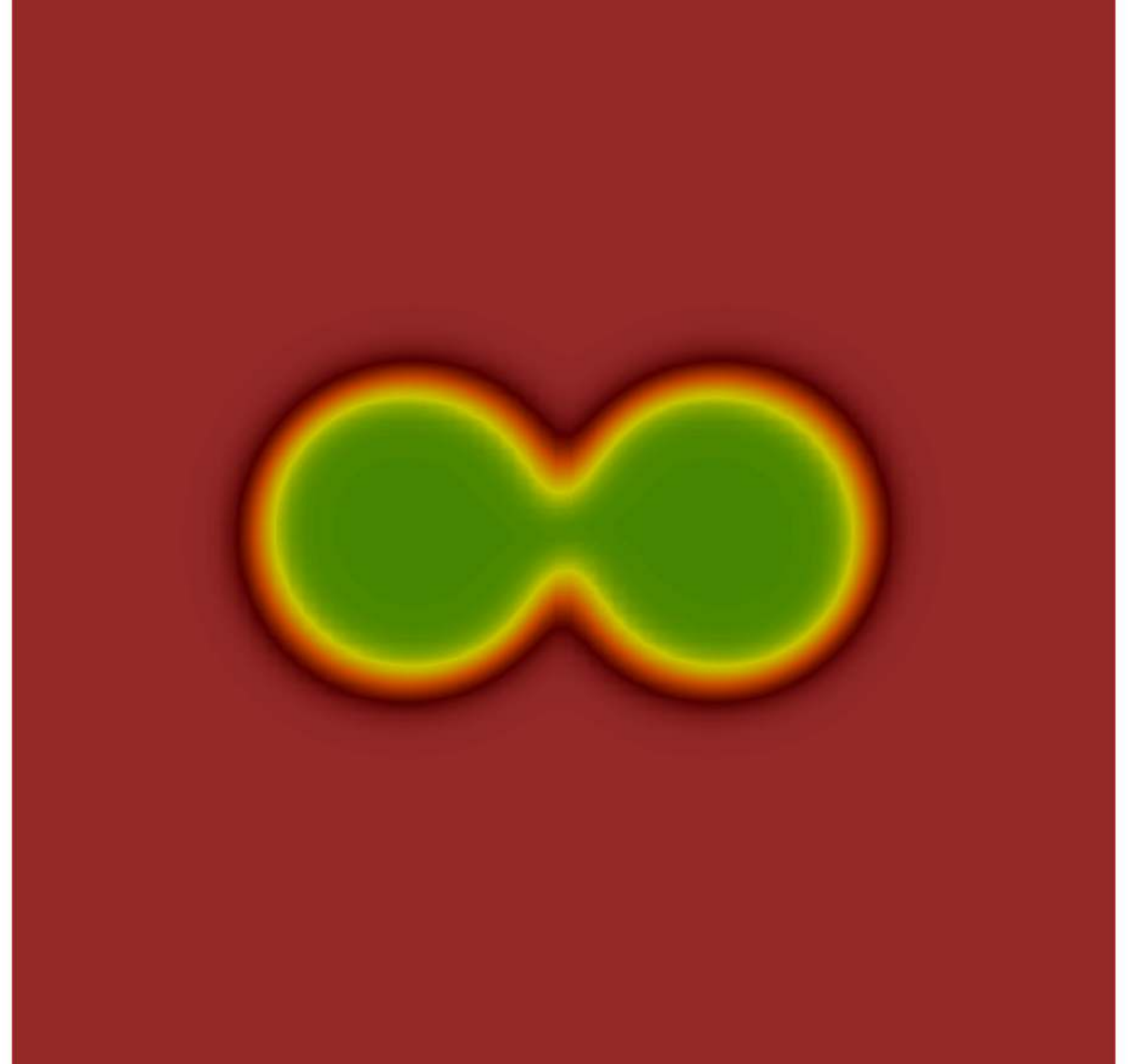}
\includegraphics[scale=0.09]{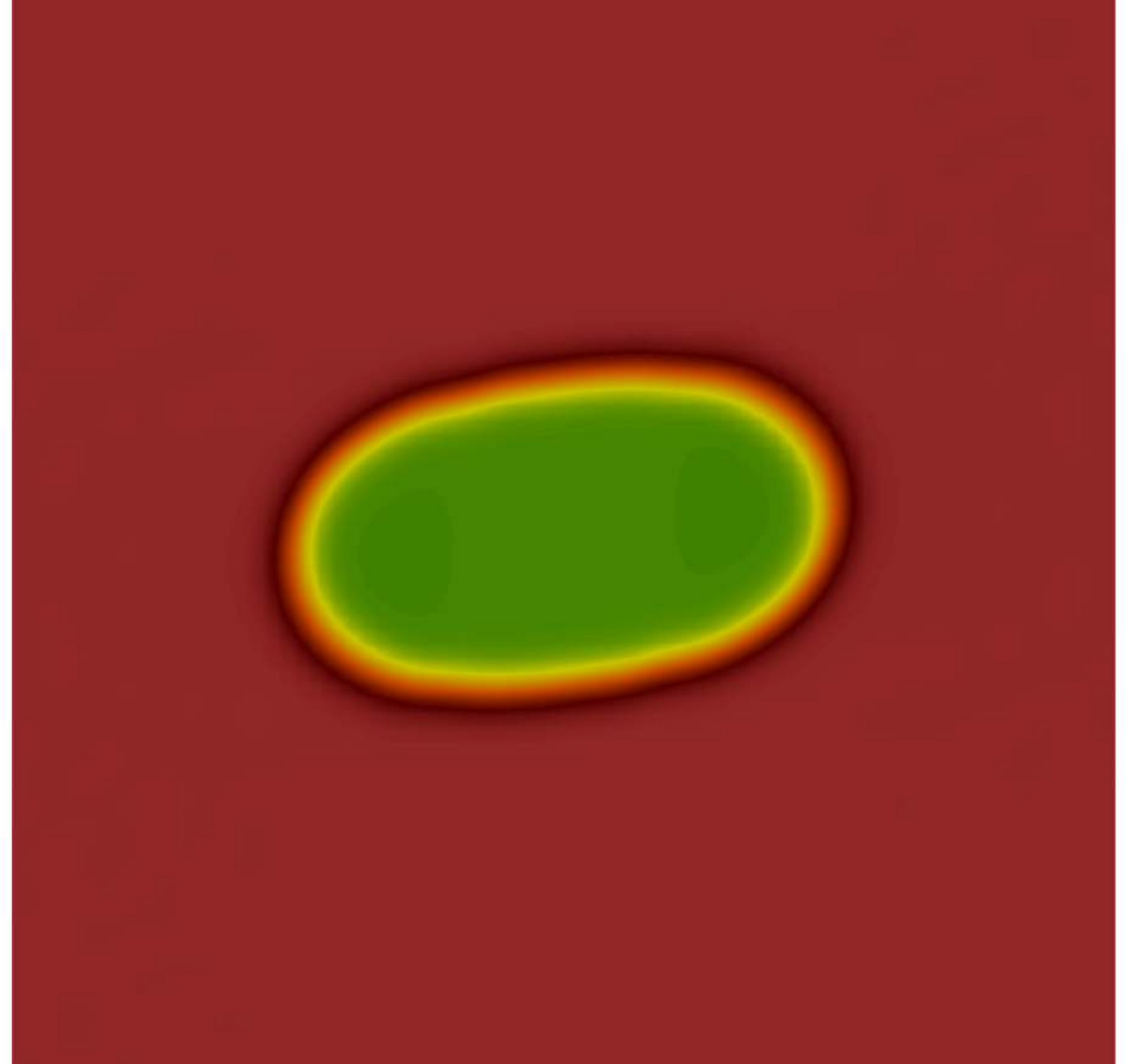}
\includegraphics[scale=0.09]{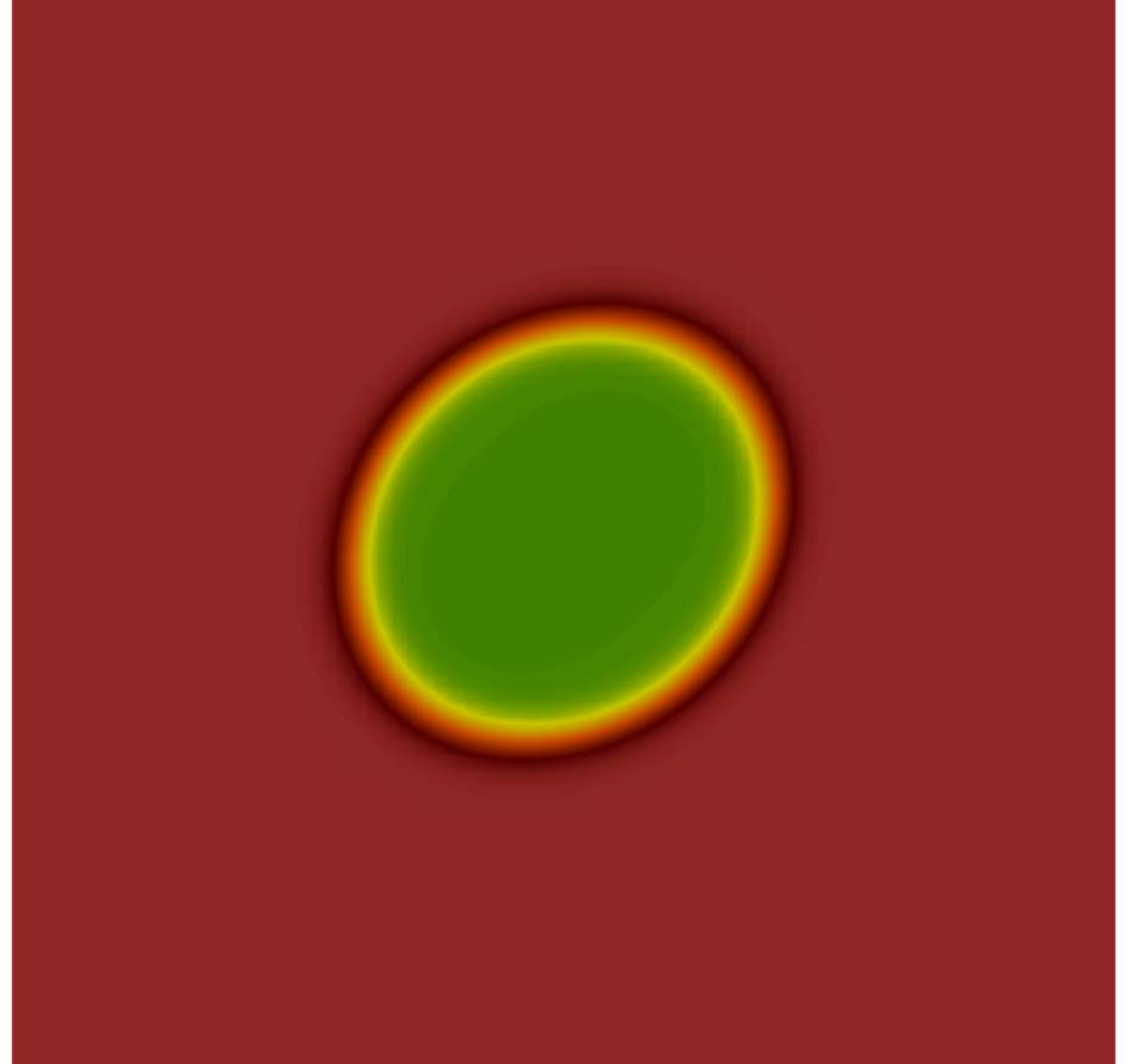}
\includegraphics[scale=0.09]{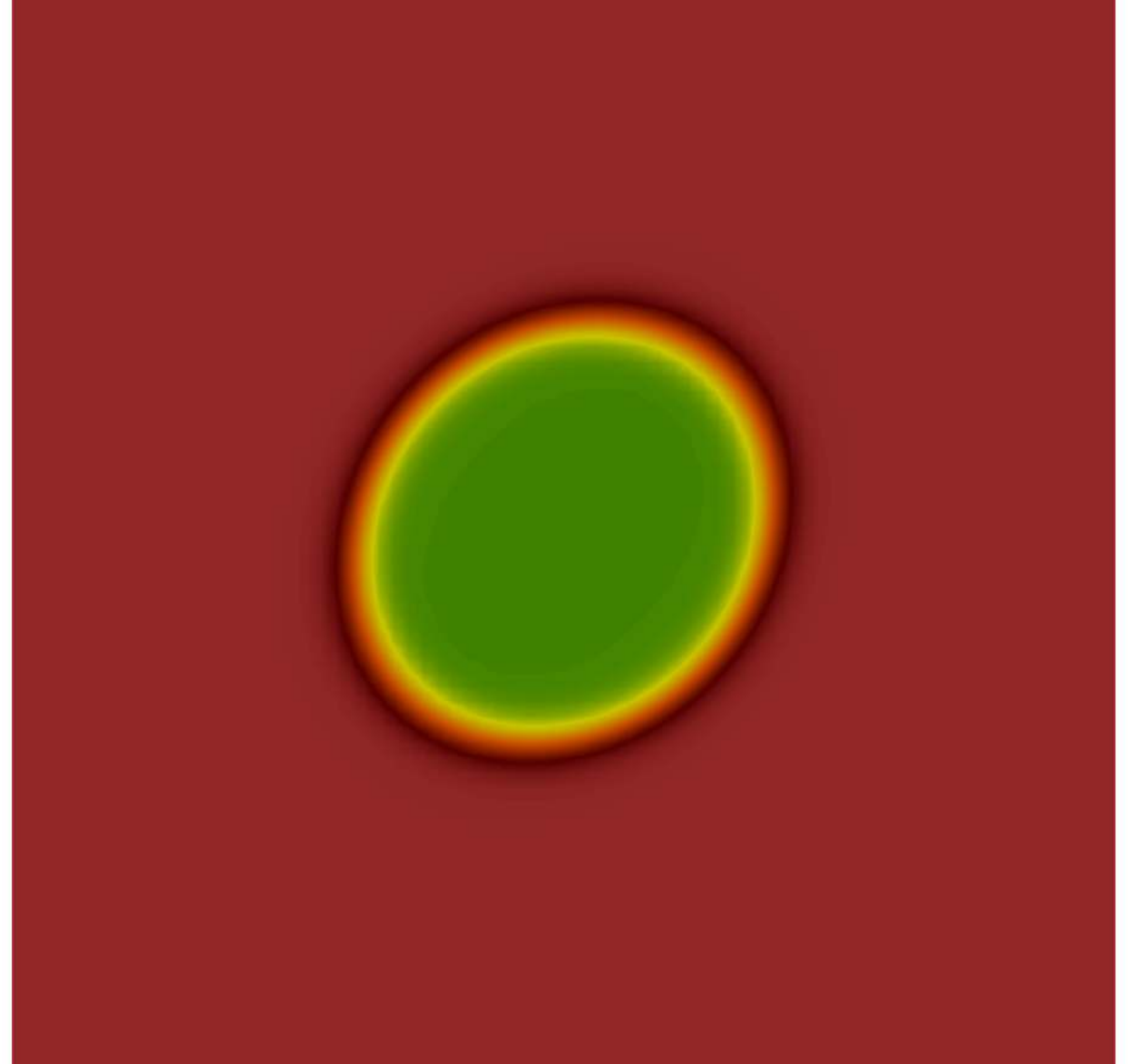}
\includegraphics[scale=0.09]{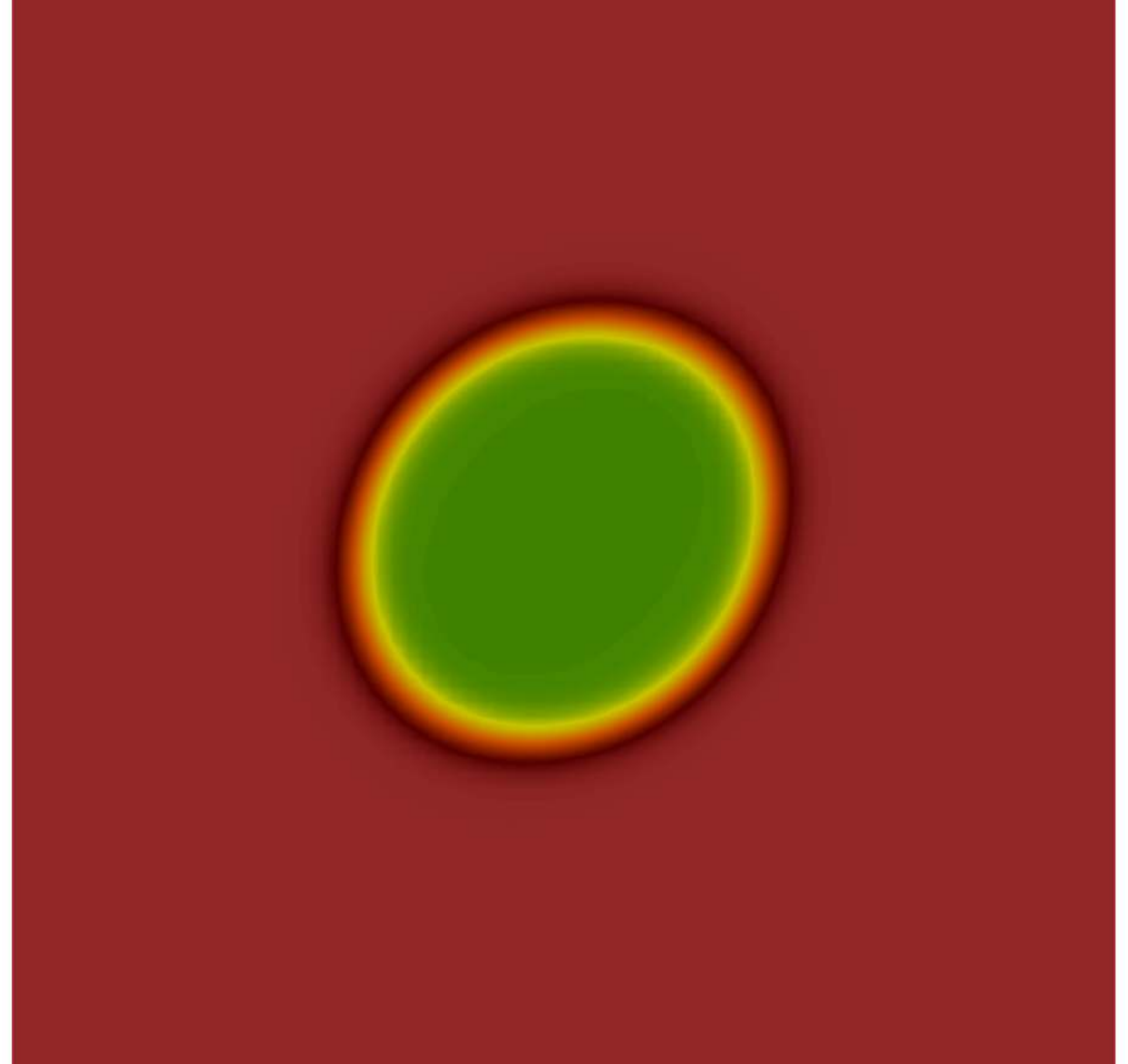}
\end{center}
\caption{Dynamics of scheme NTD1 at times $t=0, 0.01, 0.05, 0.1$ and $1$ (from left to right) with spreading coefficients $(\Sigma_1, \Sigma_2 , \Sigma_3) = (1,1,1)$. }
\label{fig:BallsFluidsCase0Dyn}
\end{figure}

\begin{figure}[h]
\begin{center}
\includegraphics[scale=0.15]{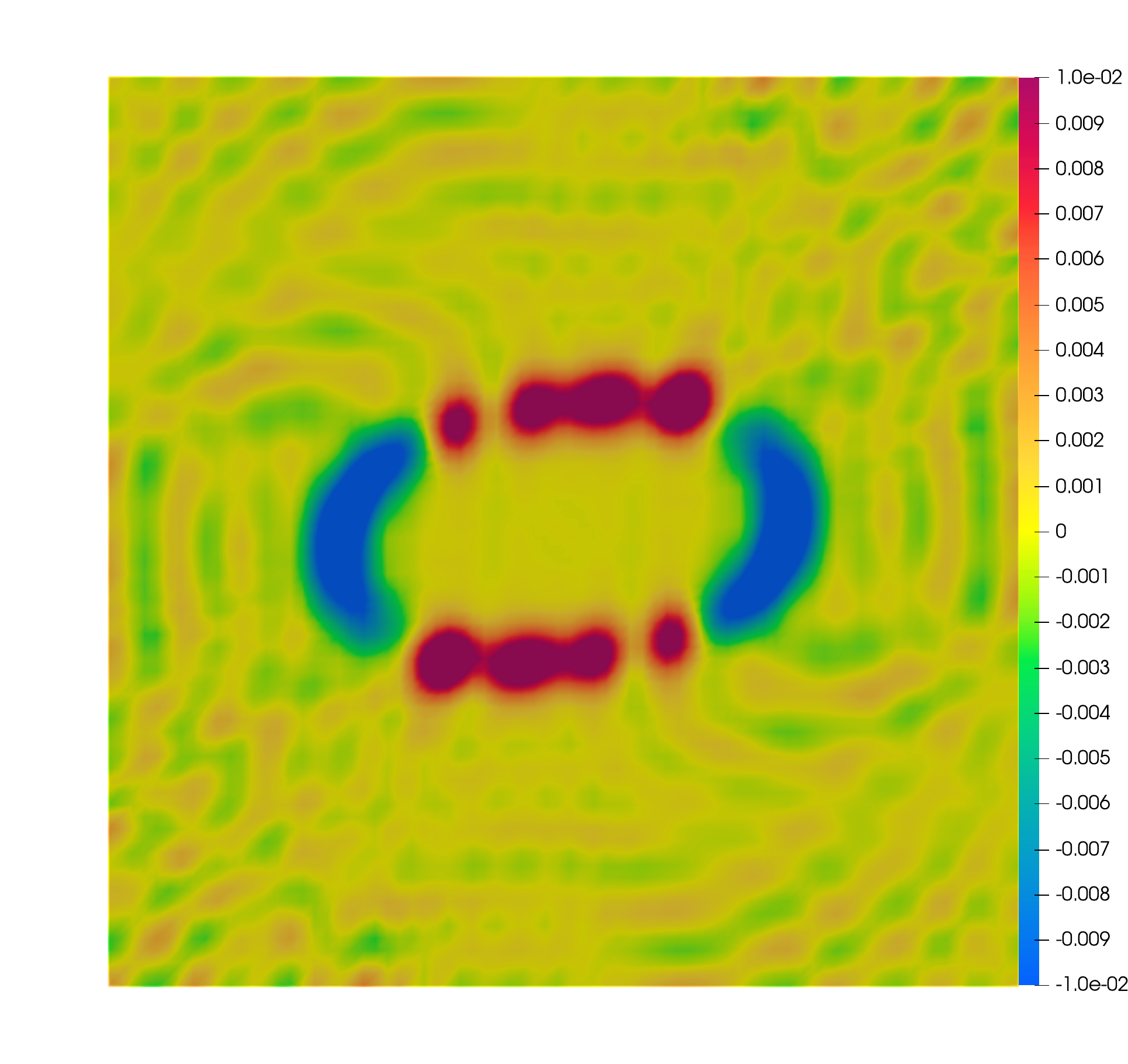}
\includegraphics[scale=0.15]{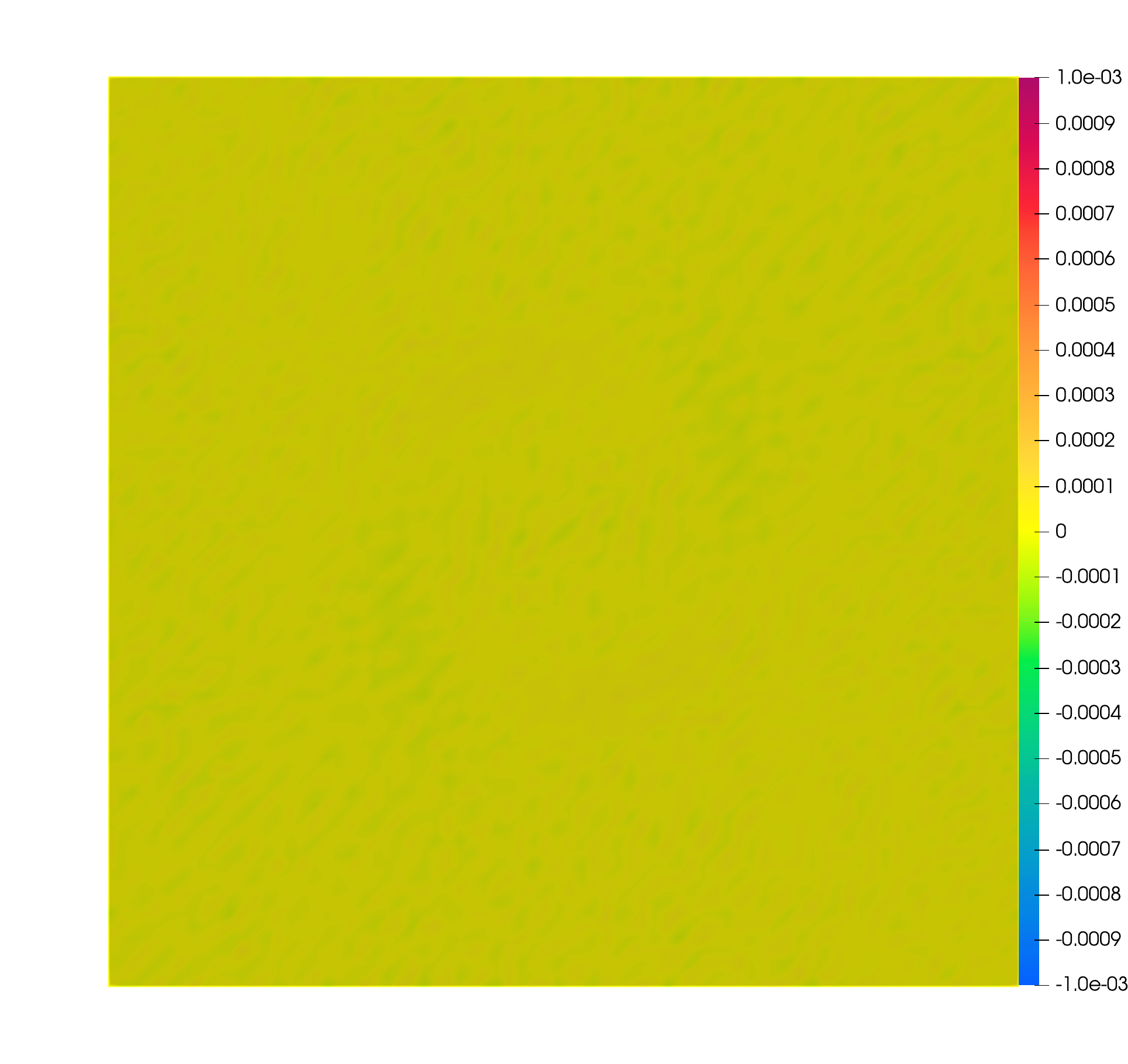}
\includegraphics[scale=0.15]{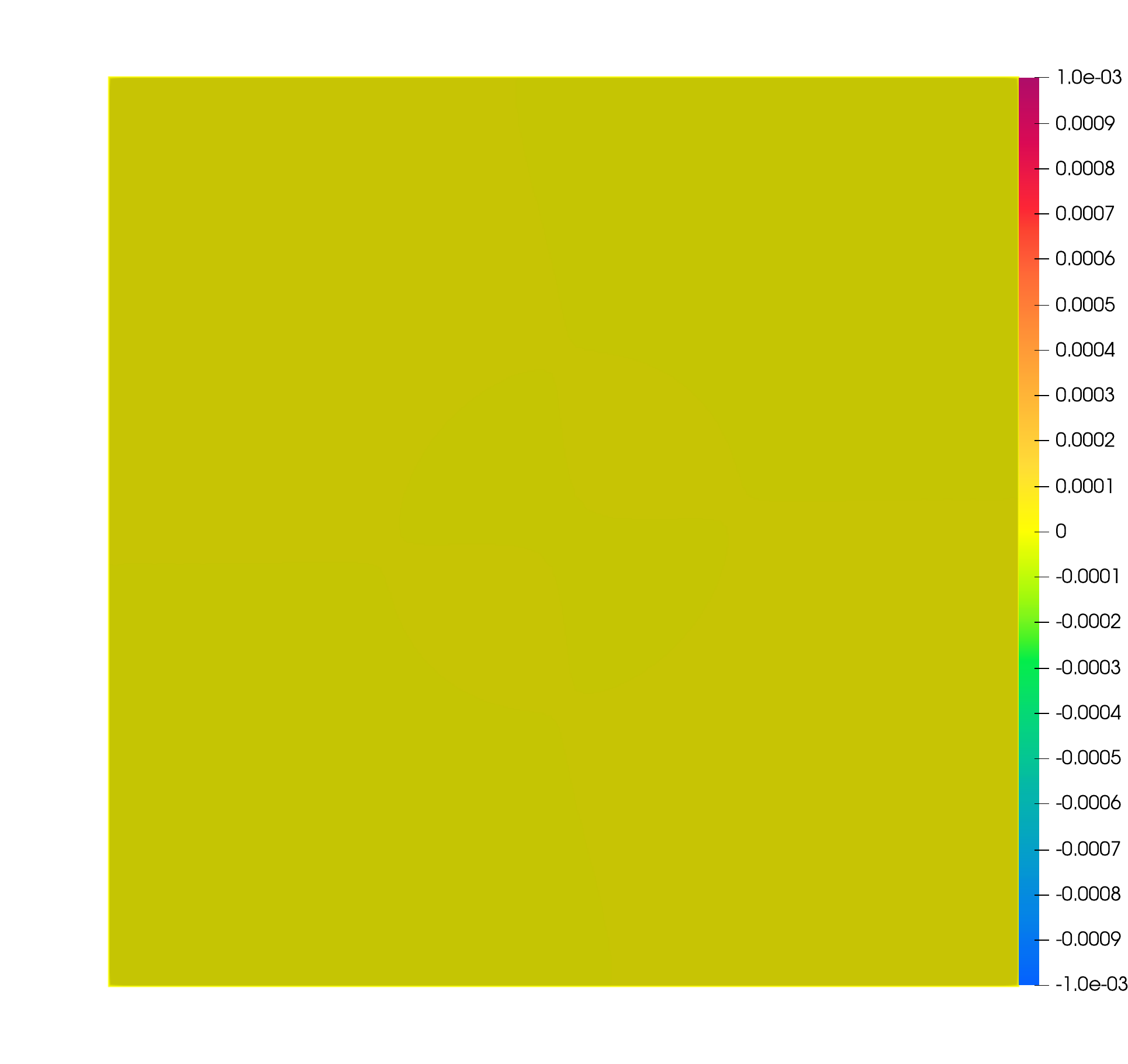}
\end{center}
\caption{Dynamics of $\phi_2$ for scheme NTD1 at times $t=0.01, 0.1$ and $1$ (from left to right) with spreading coefficients $(\Sigma_1, \Sigma_2 , \Sigma_3) = (1,1,1)$.}
\label{fig:BallsFluidsCase0NTD1phi2}
\end{figure}

\begin{figure}[h]
\begin{center}
\includegraphics[scale=0.11]{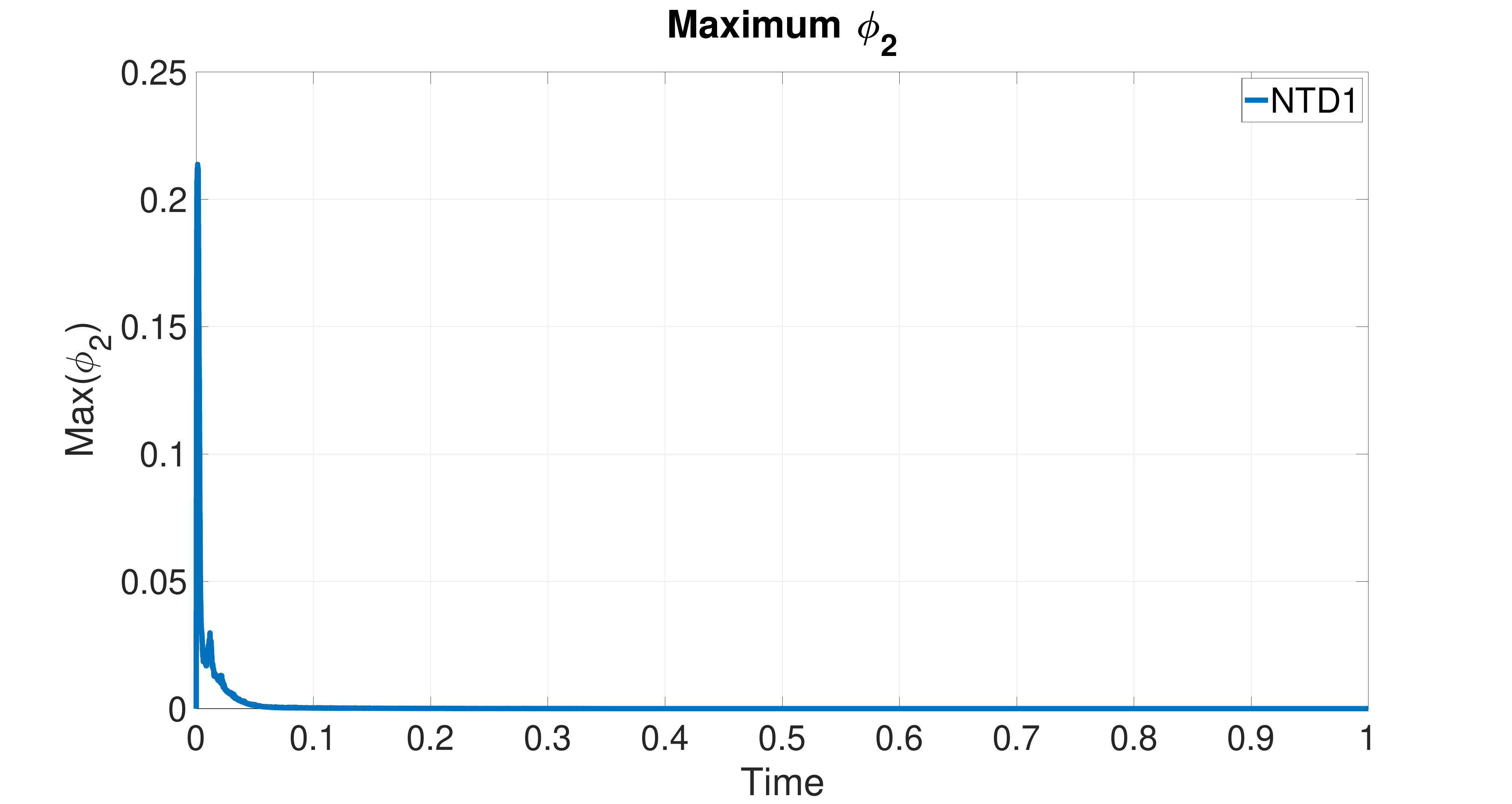}
\includegraphics[scale=0.11]{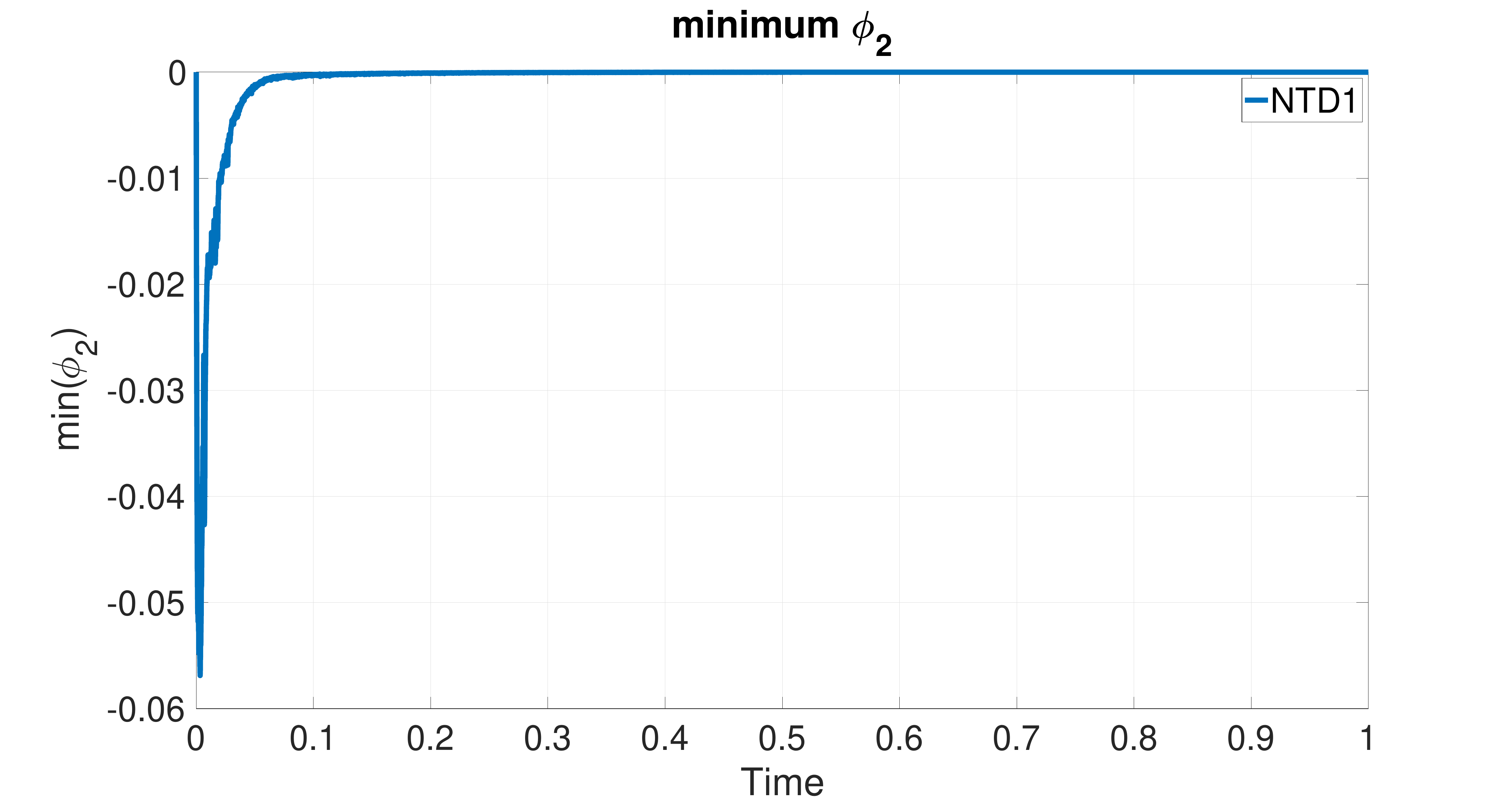}
\end{center}
\caption{Evolution in time of $\max(\phi_2)$ and $\min(\phi_2)$ for NTD1. }
\label{fig:BallsFluidsCase0NTD1phi2maxmin}
\end{figure}

\subsection{Extension to four component-type problems. Spinodal decomposition.}
{
In this final section we extend our model to a four component-type Cahn-Hilliard problem by recreating the spinodal decomposition experiment with four phases. The experimental parameters are given in Table~\ref{tab:SpinodalParameters}.
The initial condition is 
\beq\label{eq:spinodal4Initial}
\left\{
\ba{rcl}
\phi_1(x,y) &=& 0.25 +  0.01 \mbox{rand}(x,y)\,,
\\ \hueco
\phi_2(x,y) &=& 0.25 +  0.01 \mbox{rand}(x,y)\,,
\\ \hueco
\phi_3(x,y) &=& 0.25 +  0.01 \mbox{rand}(x,y)\,,
\\ \hueco
\phi_4(x,y) &=& 1 - \phi_1(x,y) - \phi_2(x,y) - \phi_3(x,y)\,,
\ea
\right.
\eeq
where $\mbox{rand}(x,y)$ is randomly sampled from a uniform distribution on $[0,1]$. 
\\
The results of the four component spinodal decomposition are shown in Figure~\ref{fig:Spinodal4Dynamics}. In the first row we plot the results for $(\Sigma_1, \Sigma_2, \Sigma_3, \Sigma_4) = (1,1,1,4)$ and in the second row for $(\Sigma_1, \Sigma_2, \Sigma_3, \Sigma_4) = (2.5,0.75,1.25,0.5)$ by plotting the function $\phi_1 +\frac13\phi_2+ \frac23\phi_3$.
In both cases the four components are initially mixed and eventually separate into 4 distinct regions. In Figure~\ref{fig:Spinodal4Plots} we observe how the energy decreases throughout the entire simulation. 
\\
As in previous spinodal simulations, the $L^2$ norm of the restriction seems reasonable but the $L^\infty$ norm might not be optimal at some times. In Figure~\ref{fig:Spinodal4Plotsdt} we compare the results when the time step is lowered (and the time interval is only $[0,0.5]$ to save computational time) and again we observe how reducing the time step clearly helps to improve the approximation of the constraint. 
}

\begin{figure}[h]
\begin{center}
\includegraphics[scale=0.09]{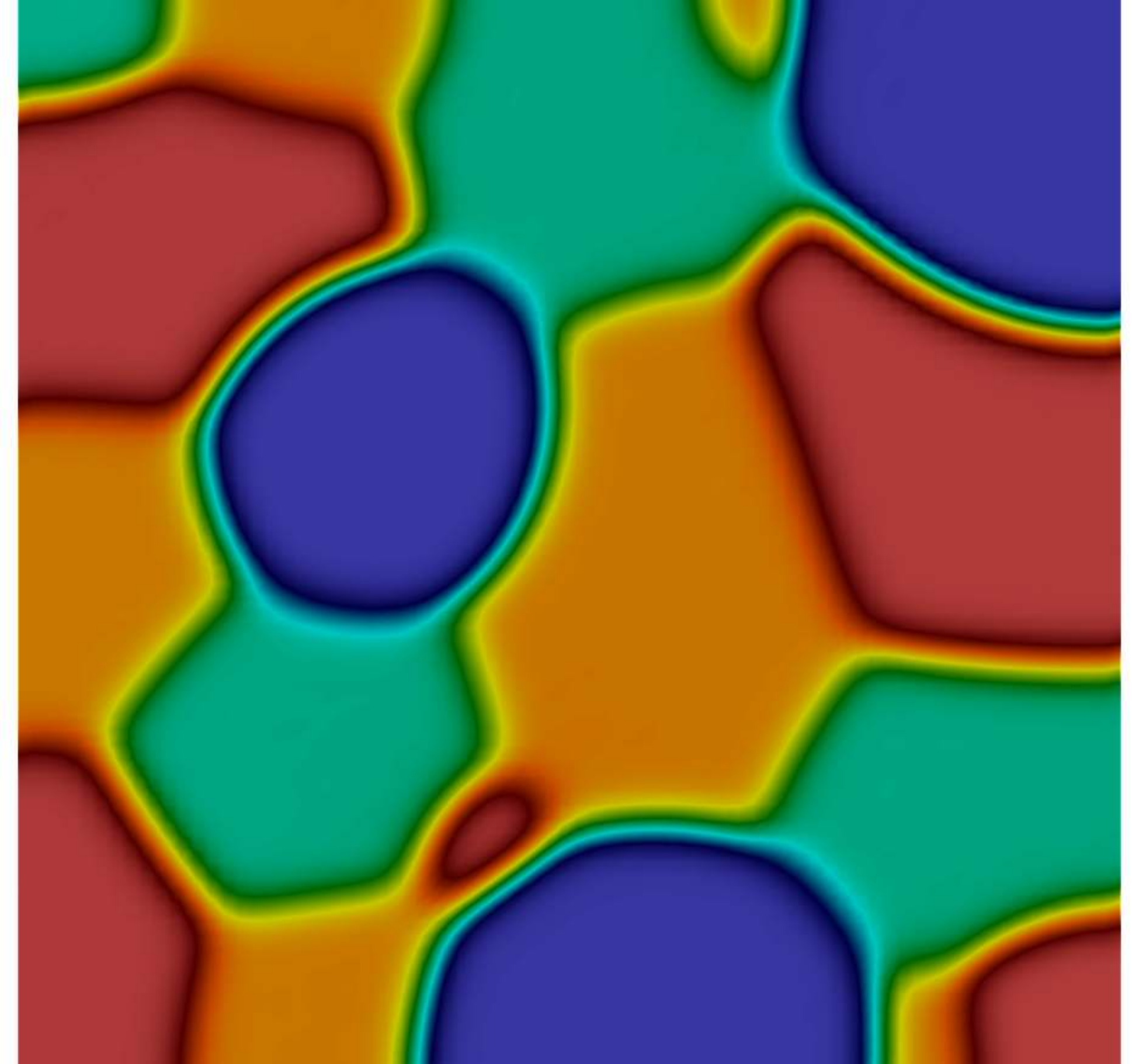}
\includegraphics[scale=0.09]{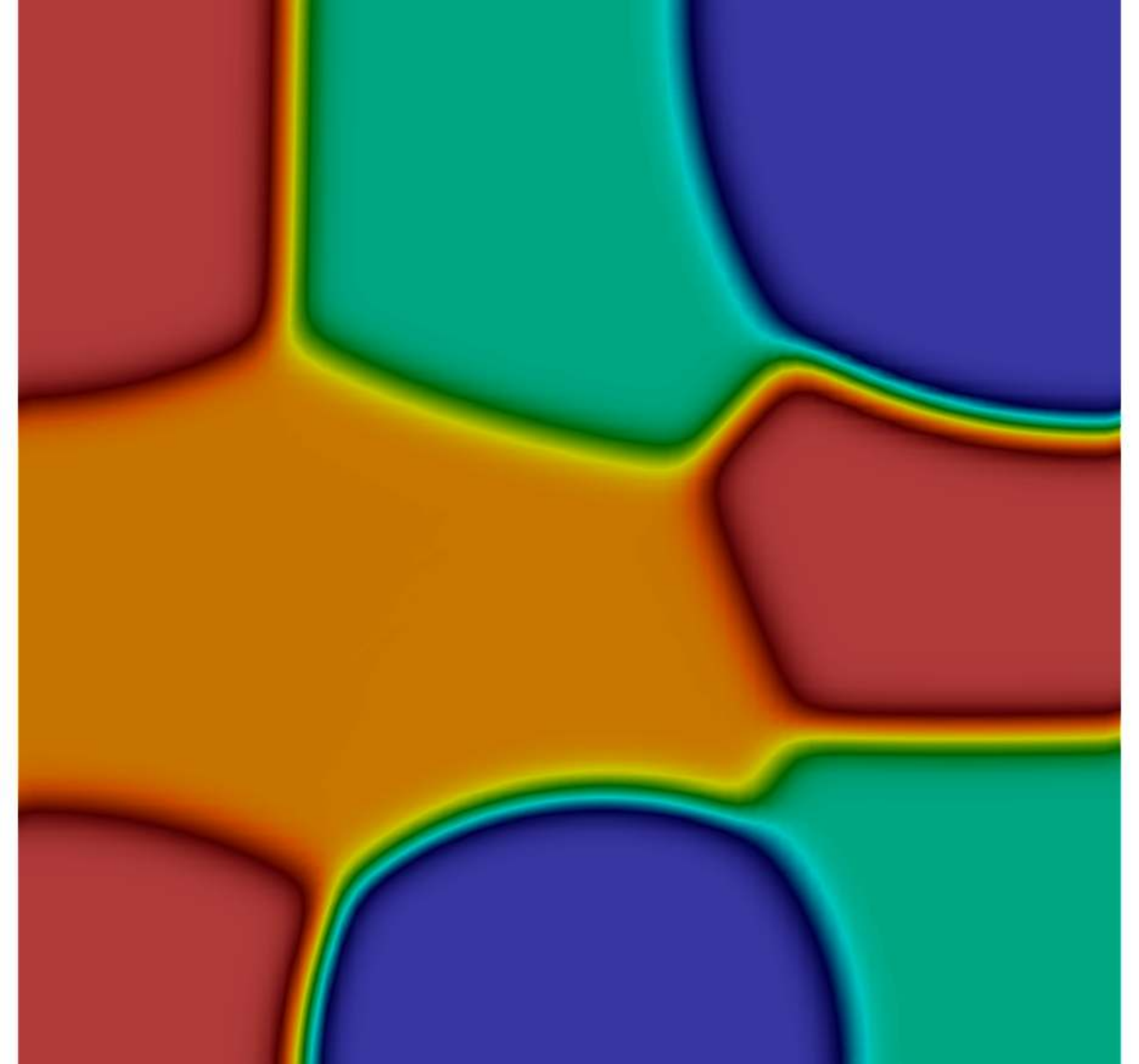}
\includegraphics[scale=0.09]{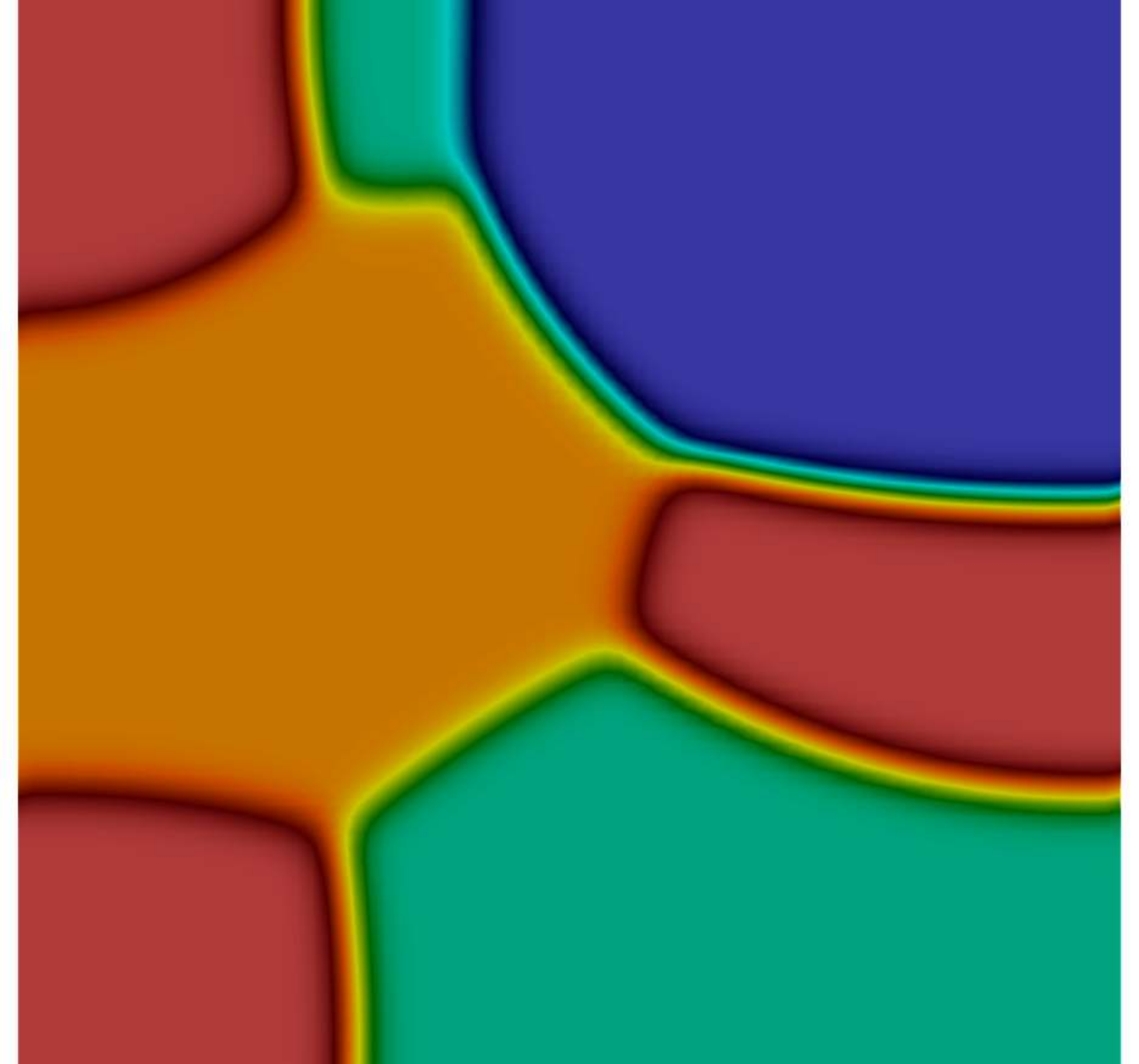}
\includegraphics[scale=0.09]{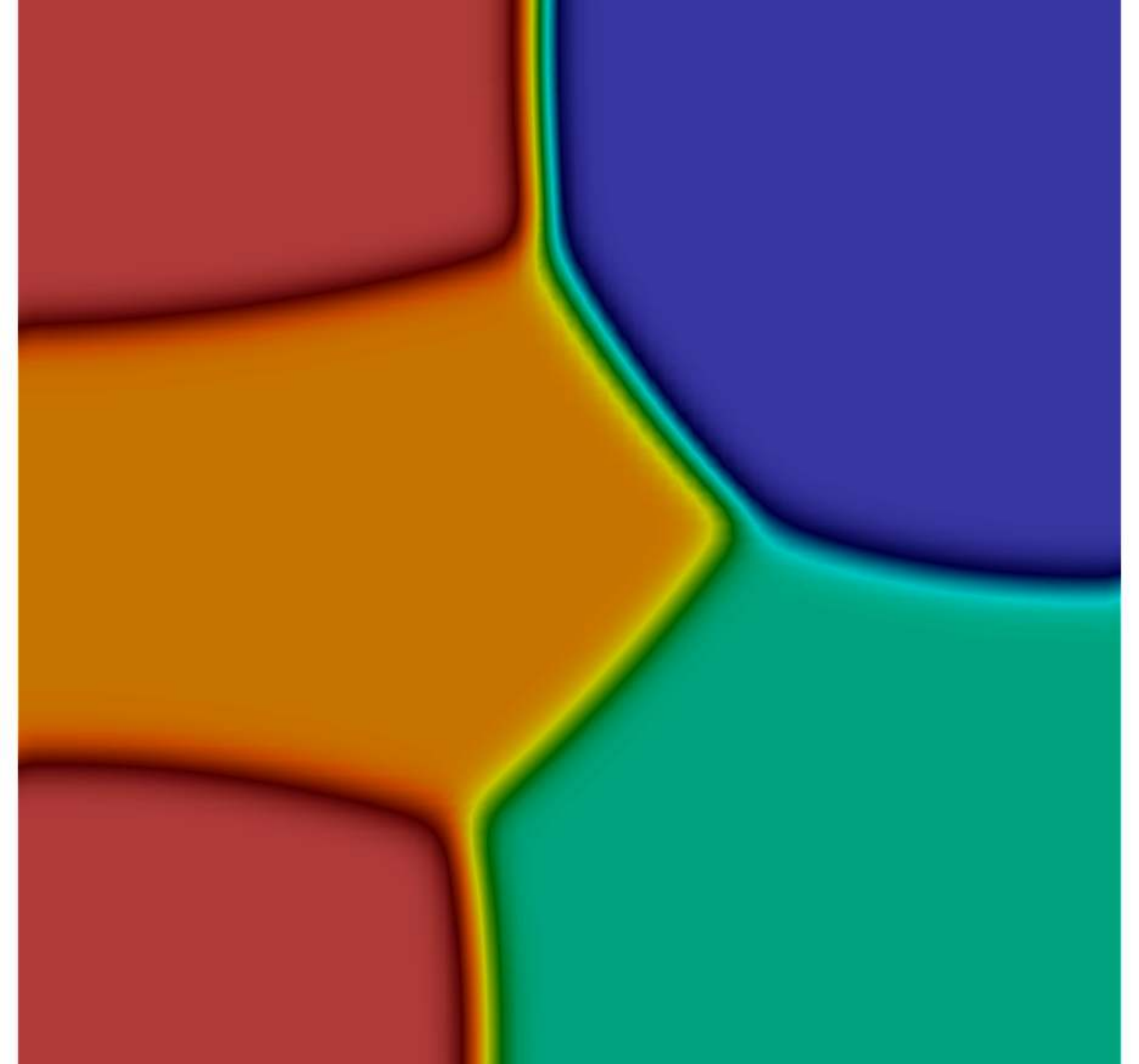}
\includegraphics[scale=0.09]{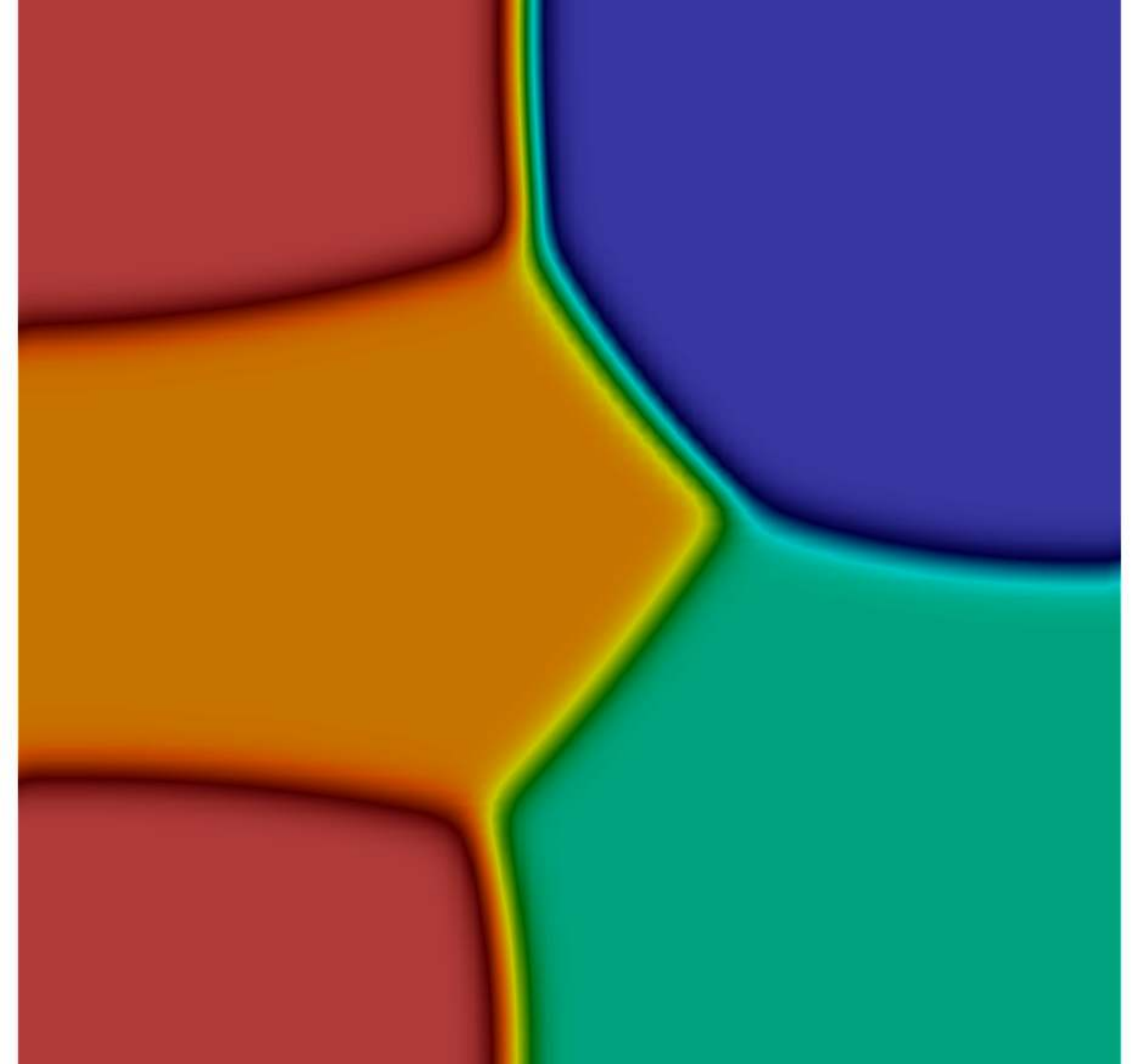}
\\ [1ex]
\includegraphics[scale=0.09]{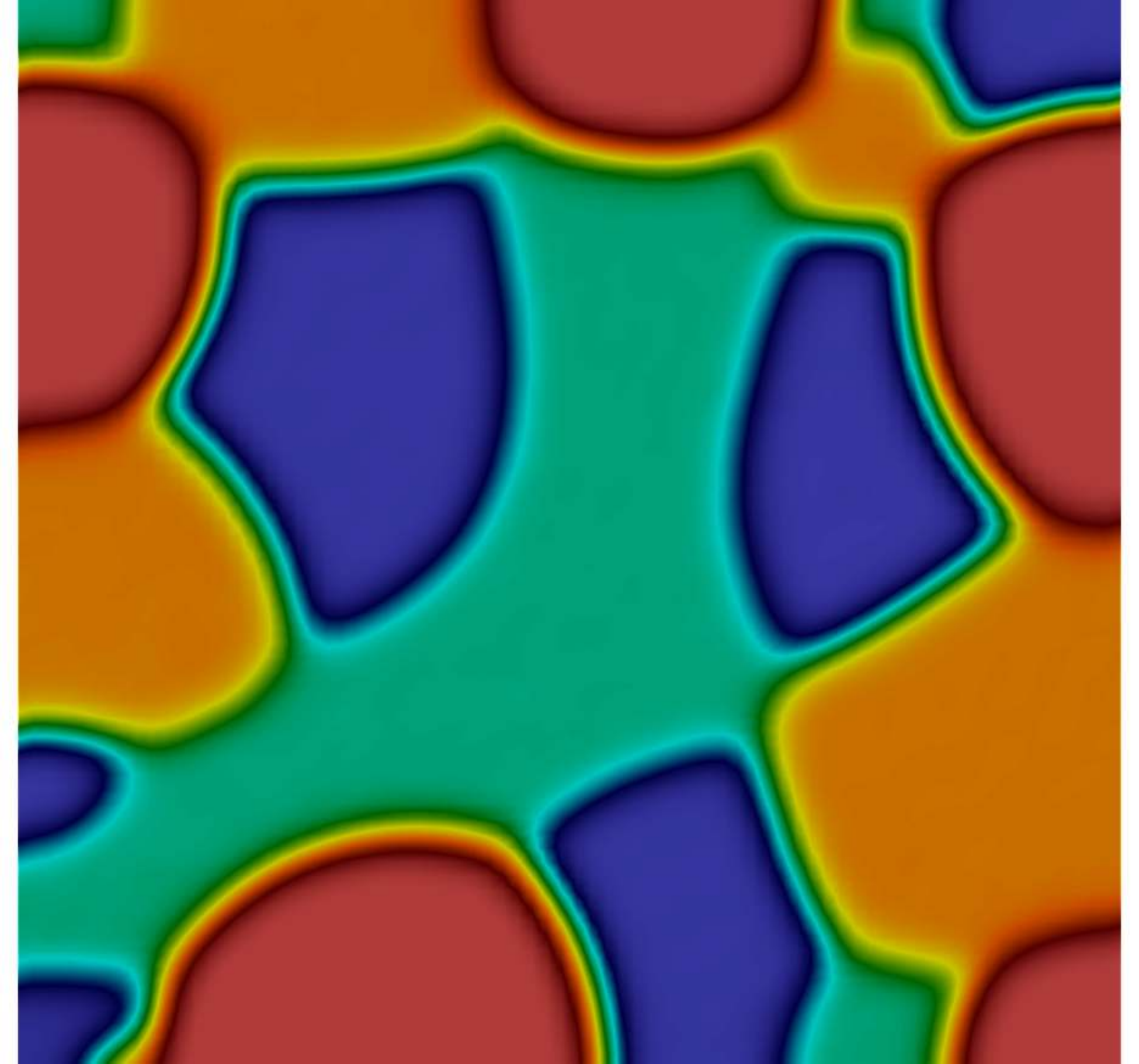}
\includegraphics[scale=0.09]{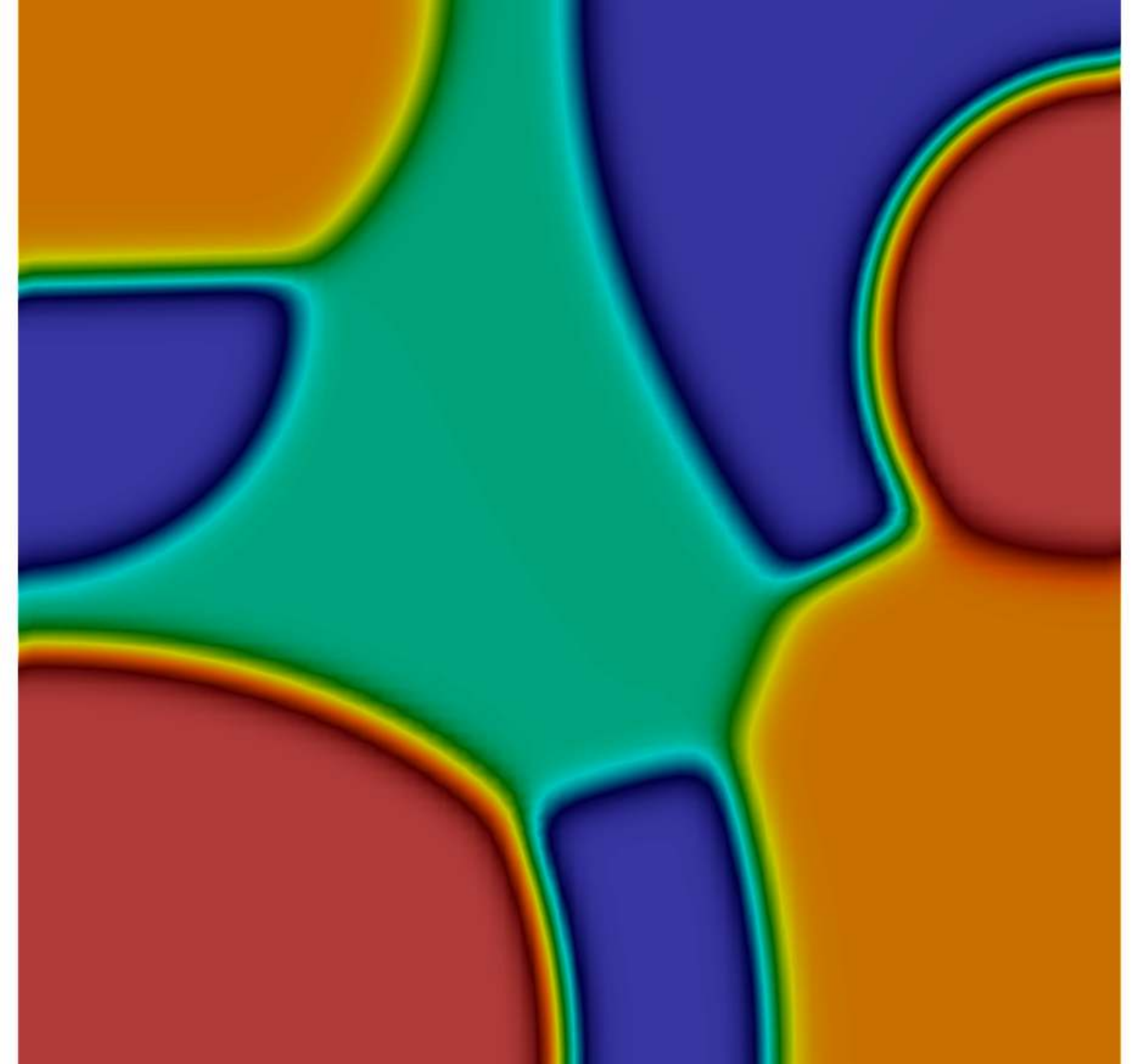}
\includegraphics[scale=0.09]{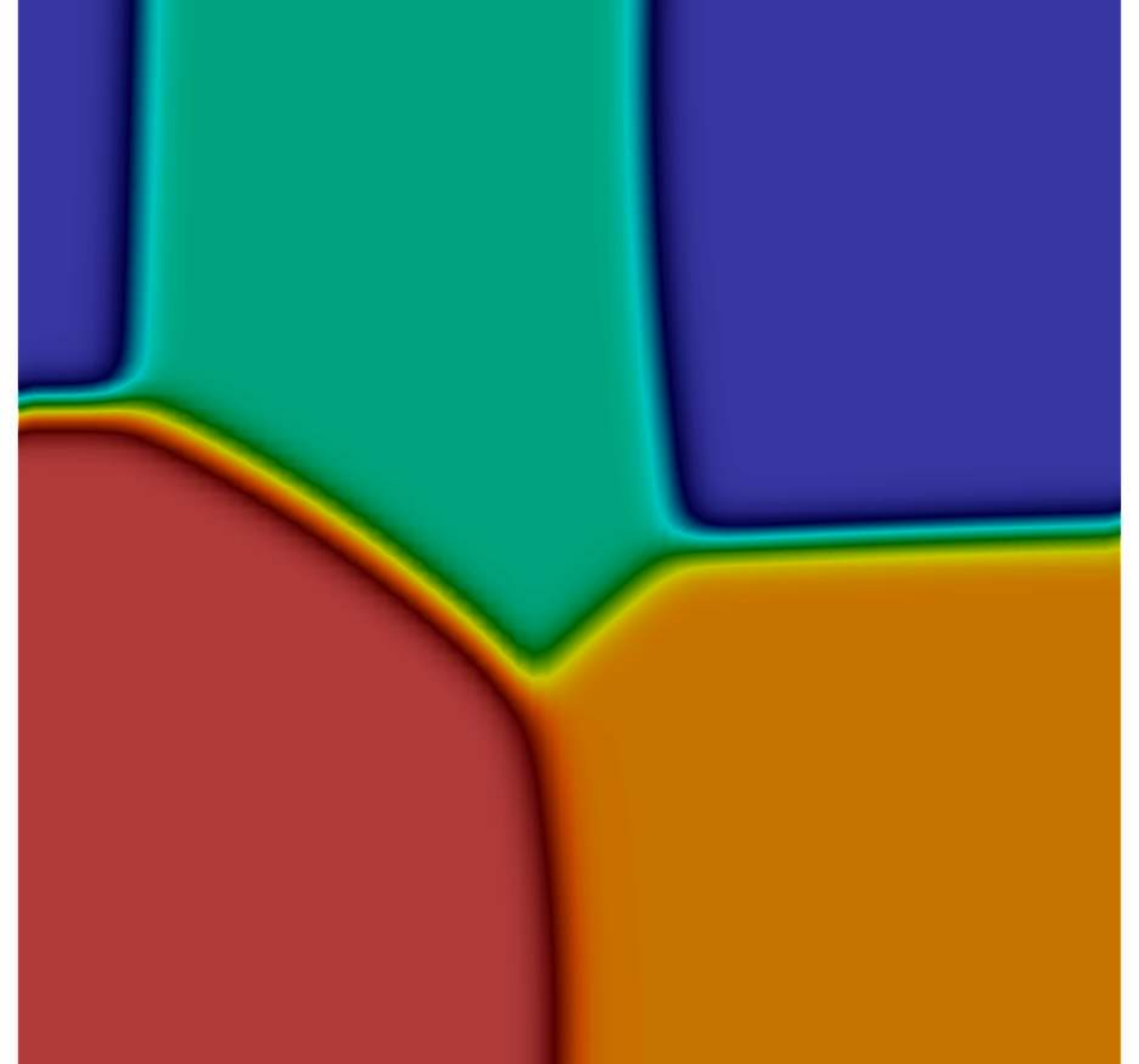}
\includegraphics[scale=0.09]{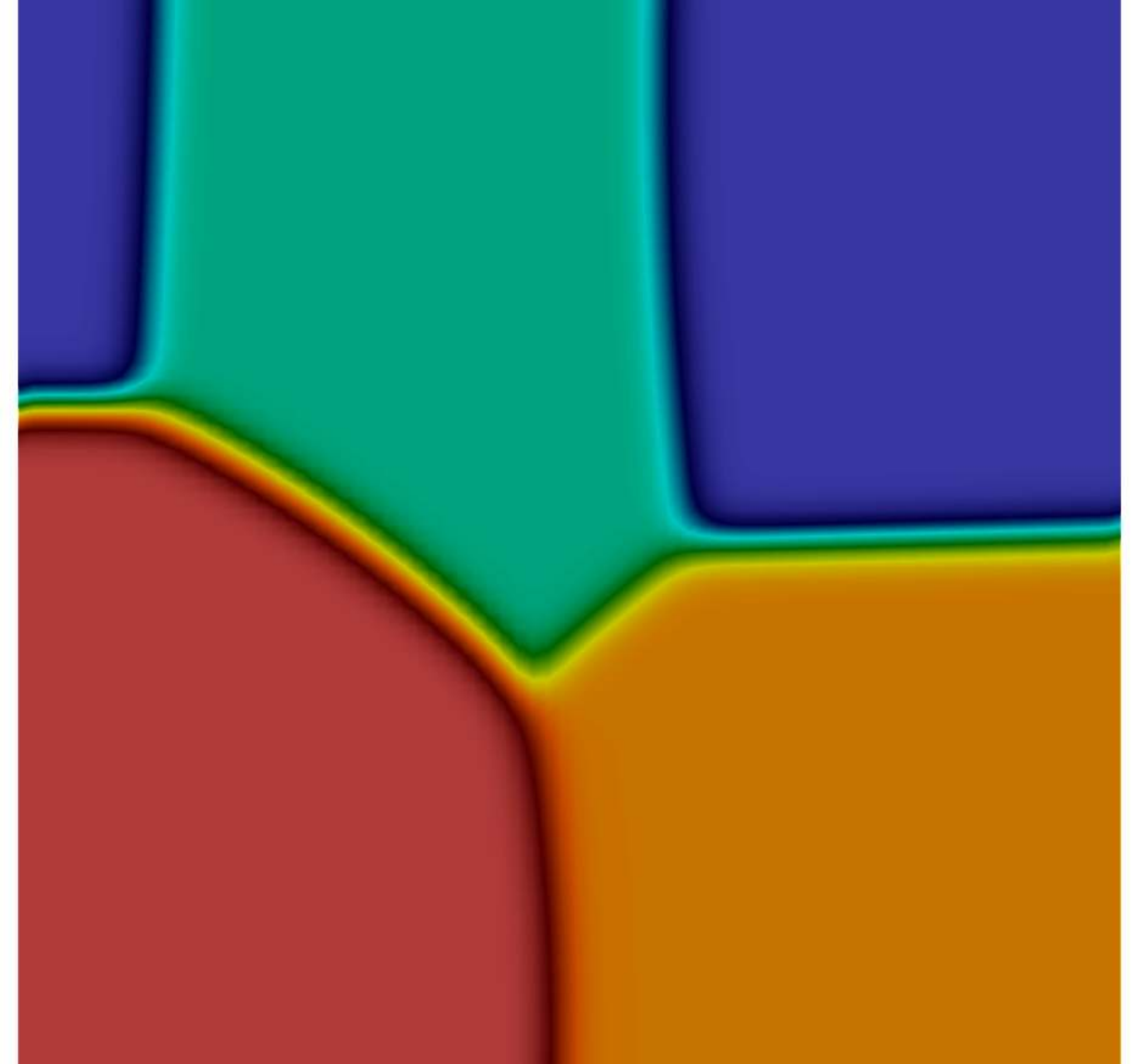}
\includegraphics[scale=0.09]{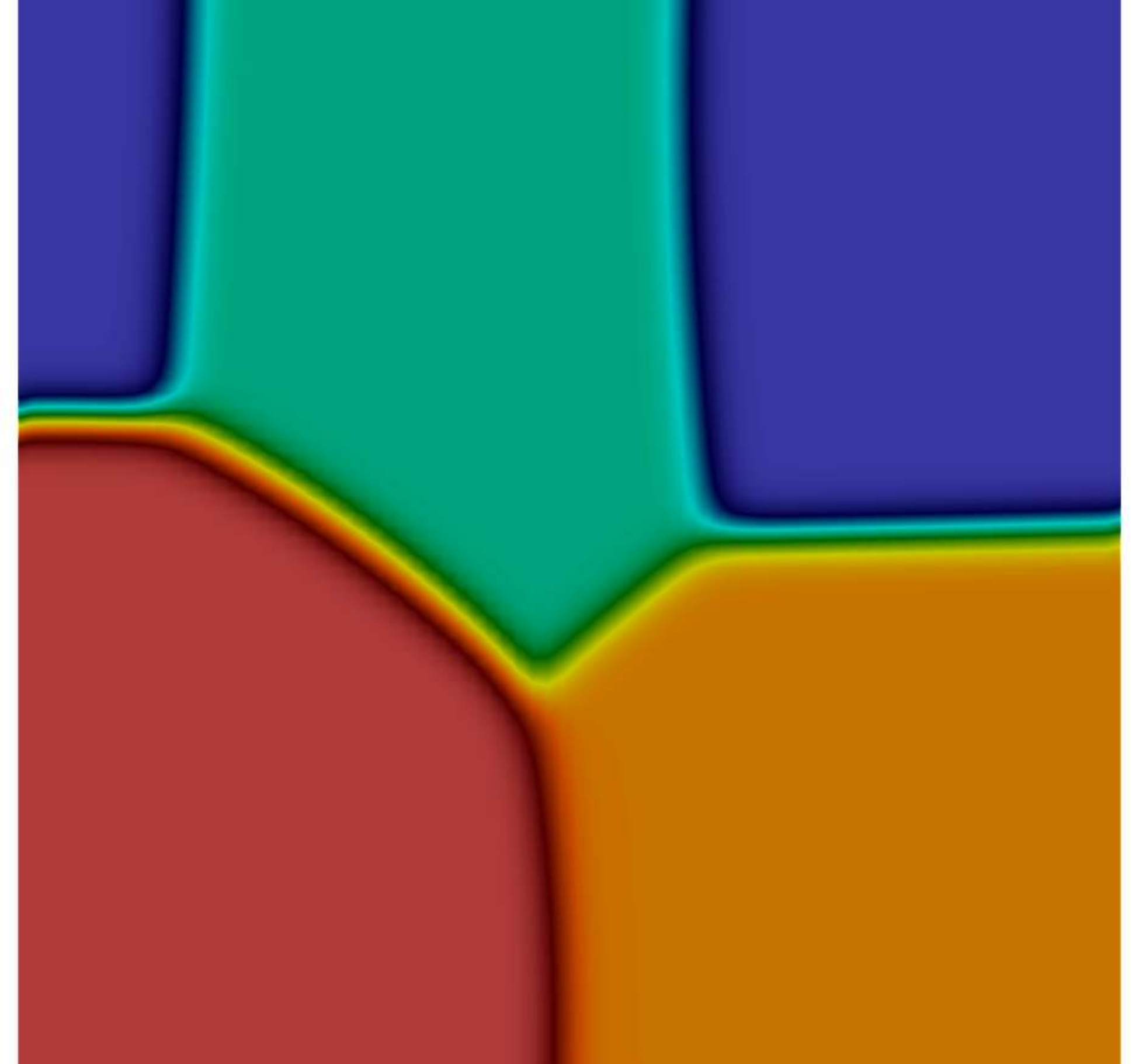}
\end{center}
\caption{Dynamics of scheme NTD1 at times $t=0.1, 0.5, 1.5, 2.5$ and $5$ (from left to right) with spreading coefficients 
$(\Sigma_1, \Sigma_2 , \Sigma_3, \Sigma_4) = (1,1,1,4)$ (top row) and 
$(\Sigma_1, \Sigma_2 , \Sigma_3, \Sigma_4) = (2.5,0.75,1.25,0.5)$ (bottom row).}\label{fig:Spinodal4Dynamics}
\end{figure}

\begin{figure}[h]
\begin{center}
\includegraphics[scale=0.11]{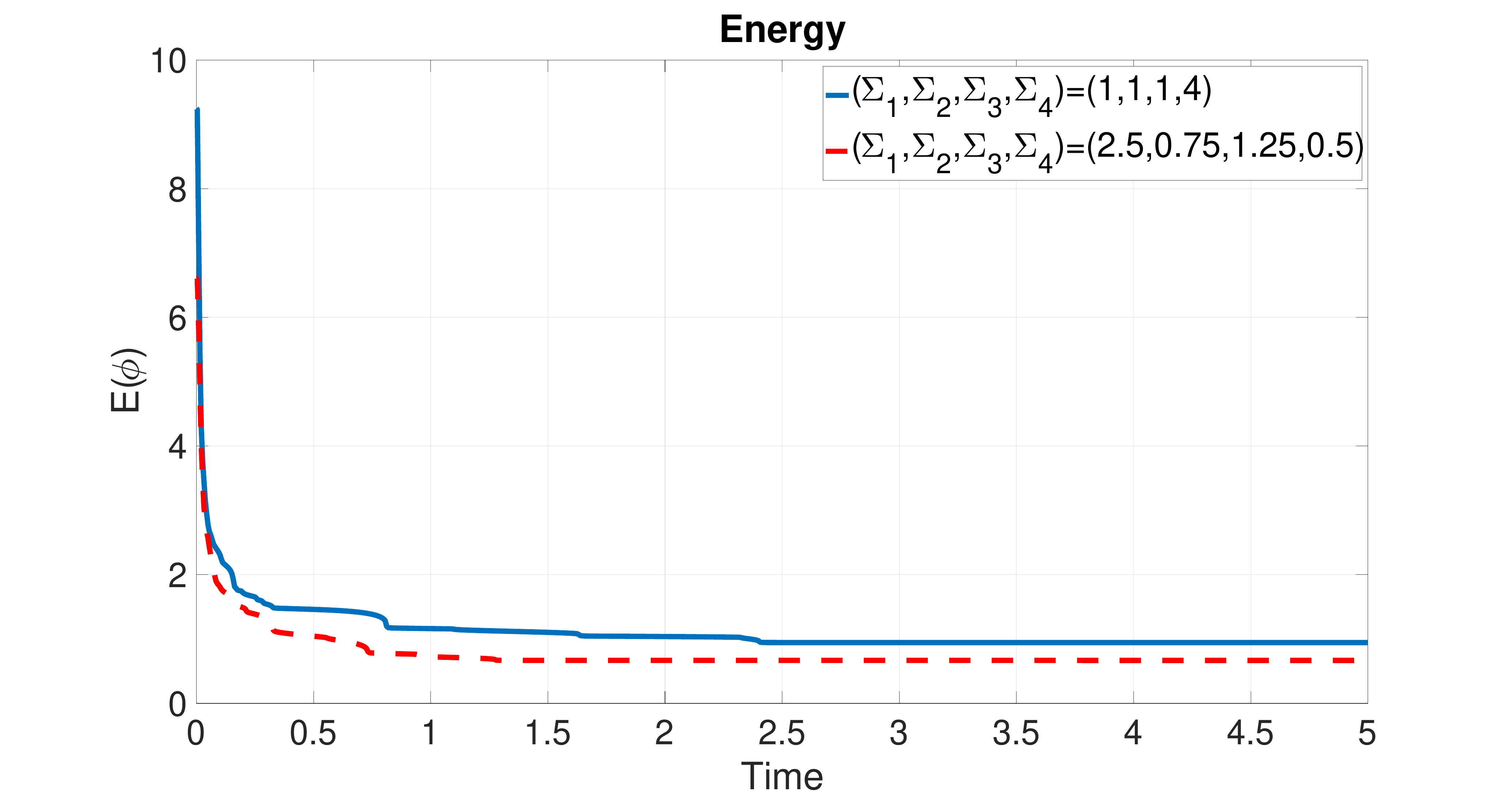}
\includegraphics[scale=0.11]{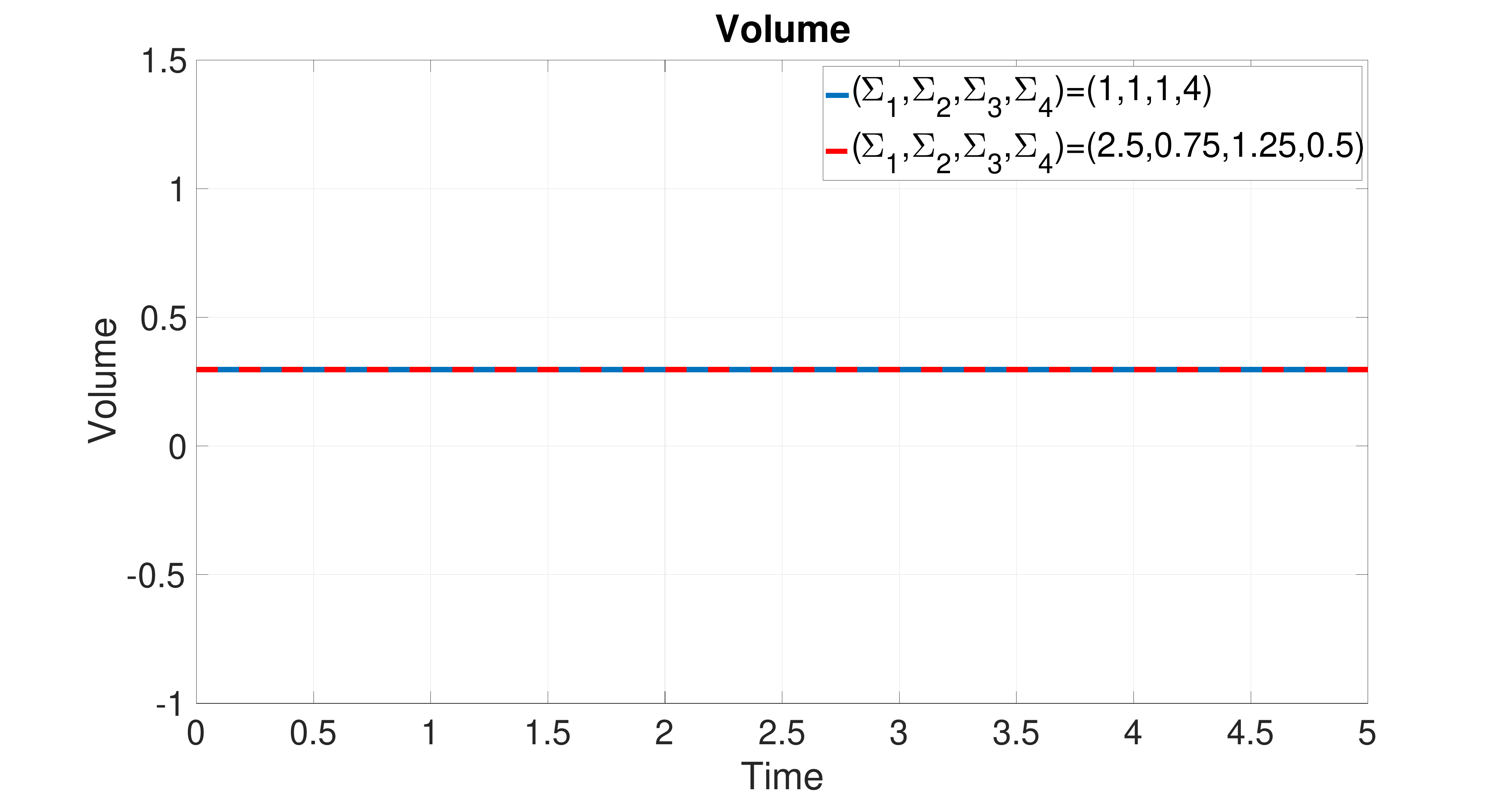}
\\ [1ex]
\includegraphics[scale=0.11]{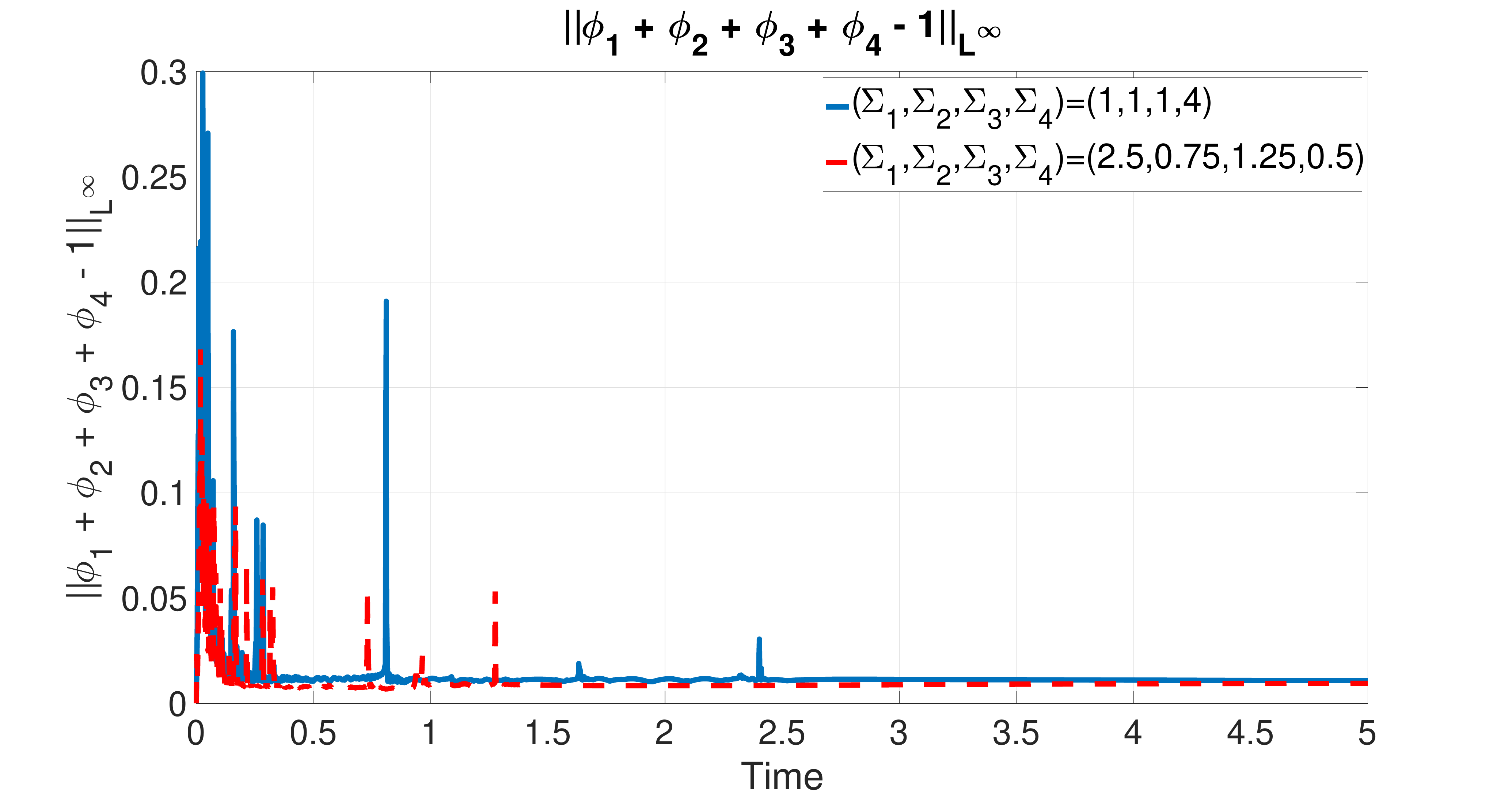}
\includegraphics[scale=0.11]{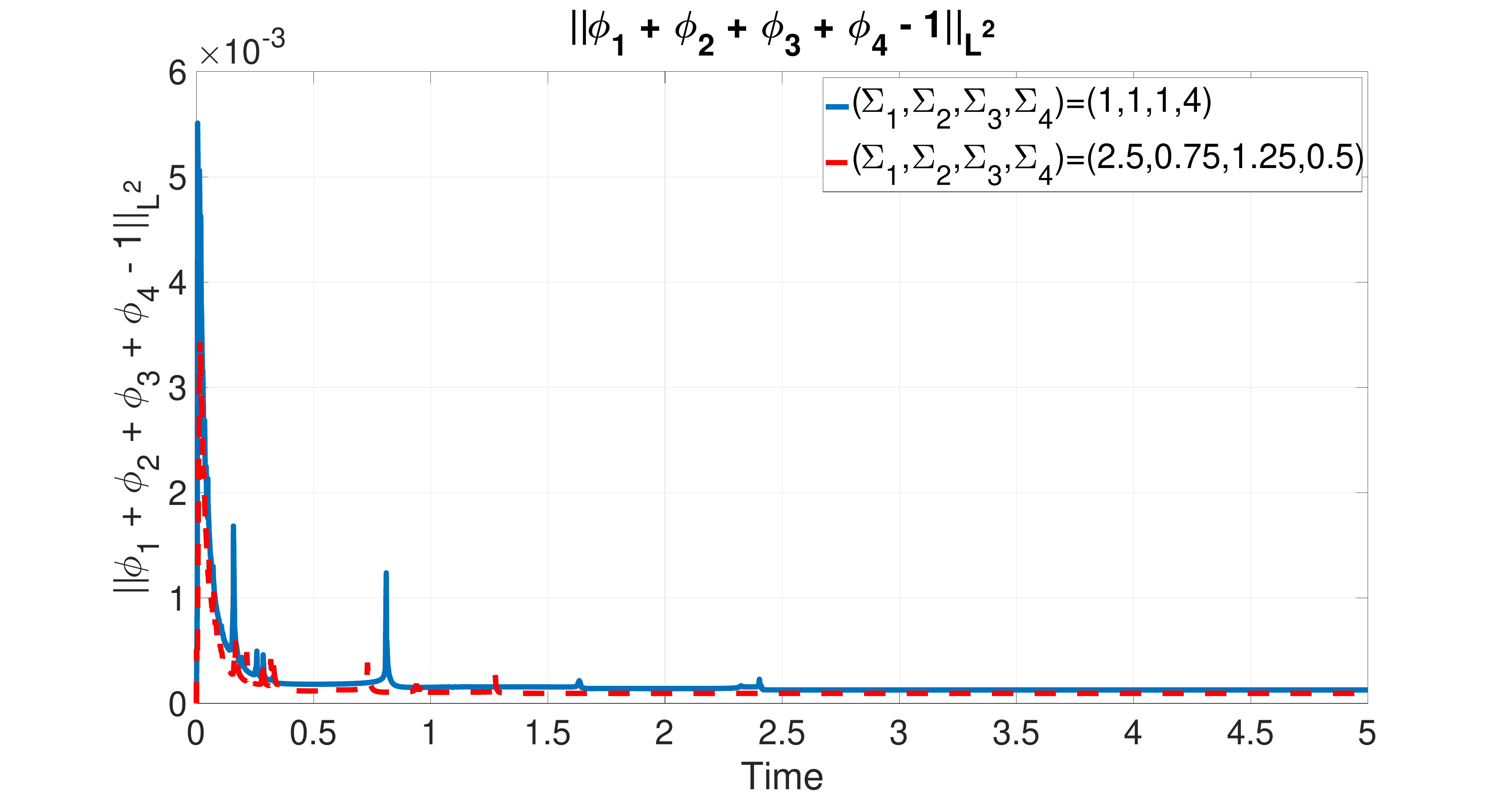}
\end{center}
\caption{Evolution in time of the energies (top left), the volume (top right), $\|\phi_1 + \phi_2 + \phi_3 + \phi_4 -1\|_{L^\infty}$ (bottom left), $\|\phi_1 + \phi_2 + \phi_3 + \phi_4 -1\|_{L^2}$ (bottom right) for the results presented in Figures~\ref{fig:Spinodal4Dynamics}.}
\label{fig:Spinodal4Plots}
\end{figure}

\begin{figure}[h]
\begin{center}
\includegraphics[scale=0.11]{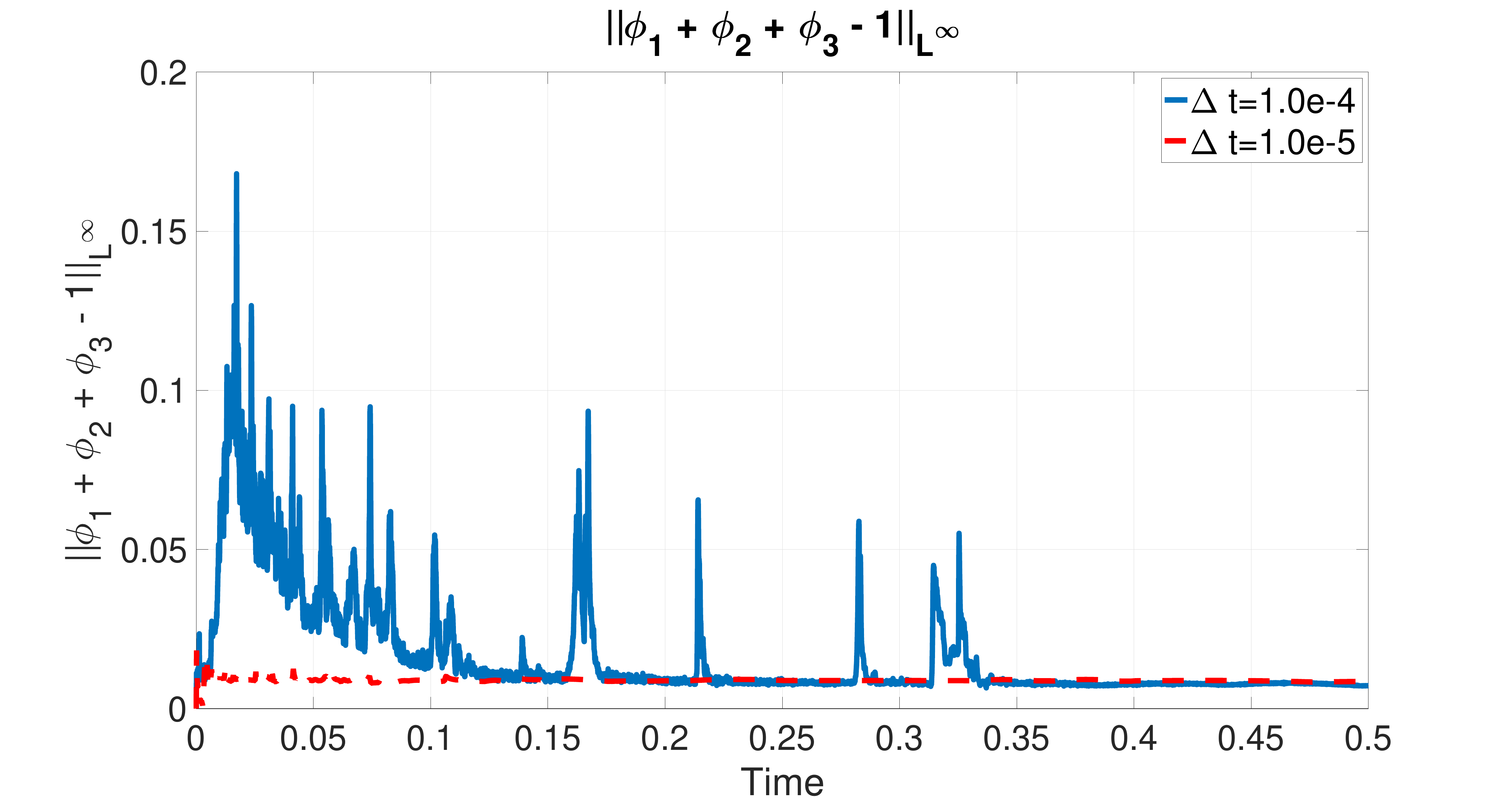}
\includegraphics[scale=0.11]{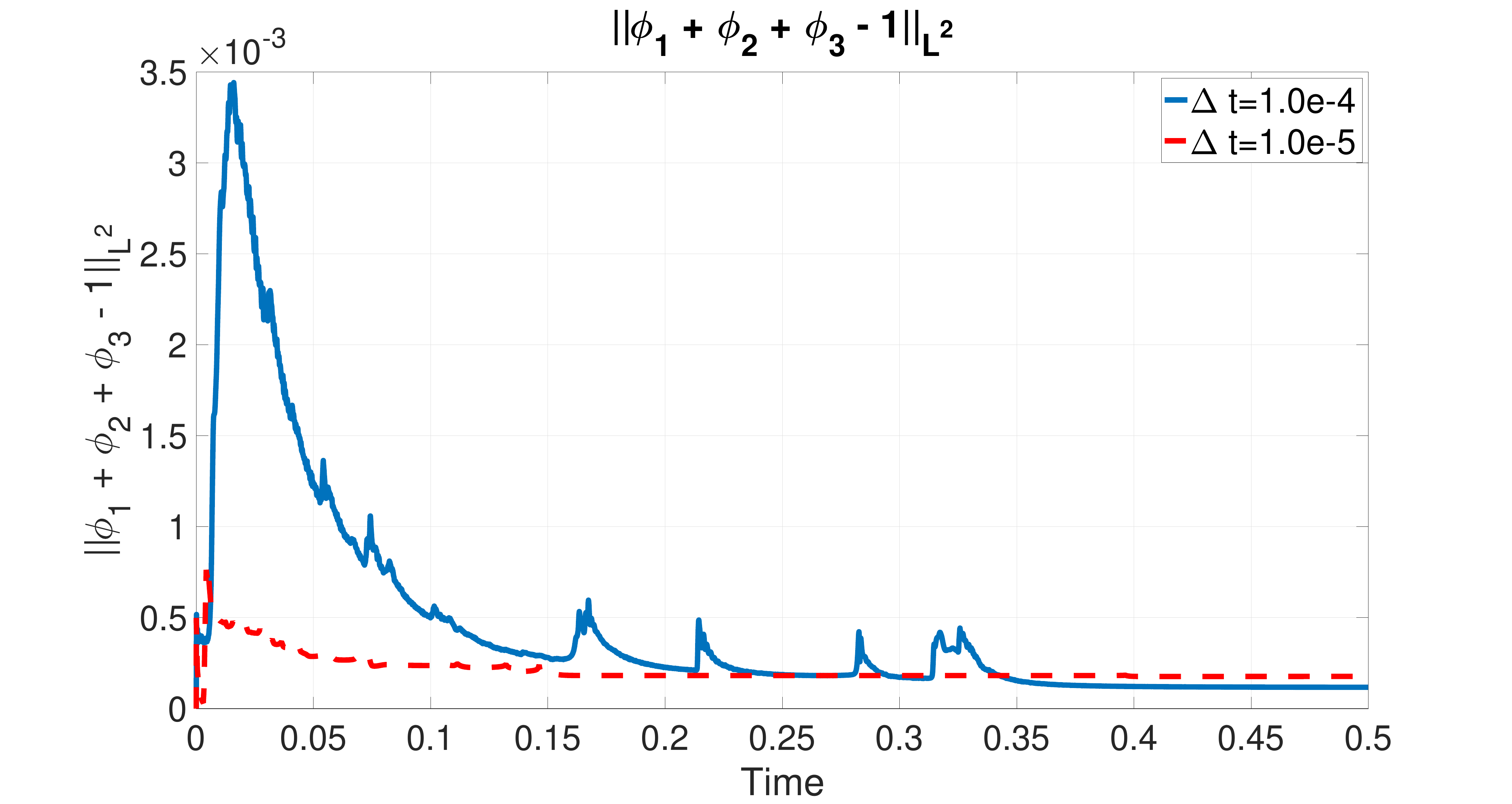}
\end{center}
\caption{Comparison of $\|\phi_1 + \phi_2 + \phi_3 + \phi_4 -1\|_{L^\infty}$ (left), $\|\phi_1 + \phi_2 + \phi_3 + \phi_4-1\|_{L^2}$ (right) for time steps $\Delta t=$1e-4 and $\Delta t=$1e-5 with spreading coefficients 
$(\Sigma_1, \Sigma_2 , \Sigma_3, \Sigma_4) = (1,1,1,4)$.}
\label{fig:Spinodal4Plotsdt}\end{figure}

\section{Conclusion}\label{sec:conclusion}
We have introduced a new formulation of the ternary Cahn-Hilliard model where the total volume constraint is enforced by adding a penalization term in the total energy of the system. Then we presented three numerical schemes for this system, which balance energy stability, accuracy, and efficiency. These schemes are powerful tools for studying interfacial dynamics and phase separation. In summary, scheme TD1 is an unconditionally energy stable scheme which is linear, first order accurate, and decouples the unknowns in the system into three sub-problems. Scheme NTD1 improves the first scheme in terms of efficiency, but is only conditionally energy stable. However, our numerical results have shown that this scheme is {more} reliable {than} TD1 for practical use, and is efficient enough to be reasonably used for three dimensional simulations, and extensions to four or more component systems. Finally, we have scheme NTC2, which is linear, coupled, conditionally energy stable, and second order accurate. In all three schemes, we have shown that the numerical dissipation introduced by the approximations of nonlinear terms is rather small, which indicates that the dynamics of the discrete solution is close to the true solution.

By providing a collection of schemes each with different properties, researchers looking to perform numerical simulations will be able to choose a method based on their individual goals. For practical purposes, it is scheme NTD1 that is most versatile due to its significant savings in computational cost. This makes it a good candidate for use with more complicated problems such as mixtures of four or more components, or coupling with Navier-Stokes to observe hydrodynamic effects. In all of the examples that we have presented, the scheme behaves as if it is unconditionally energy stable, and combined with the relatively low amount of artificial numerical dissipation we have shown that this scheme is sufficient for delivering accurate simulation results with reasonable computing time.

    %

\bibliographystyle{acm}
\bibliography{references}

\end{document}